\documentclass[11pt]{amsart}
\usepackage{amsmath,amsbsy,amssymb,amscd,amsfonts,latexsym,amstext,delarray, amsmath,color,caption}
\usepackage{tikz,lscape}
\usetikzlibrary{matrix,arrows}
\usepackage{graphicx}
\graphicspath{ {Home/Users/Fw18twelve/Desktop/Vanderbilt_Summer_2020/Research/Quadratic/Images} }
\usepackage{multicol}
\usepackage{scrextend}
\usepackage{mathtools}
\usepackage{amssymb} 
\usepackage{calligra}
\usepackage{calrsfs}
\usepackage{eucal}
\usepackage{calrsfs}
\usepackage{epsfig}
\usepackage{xcolor,import}
\usepackage{transparent}
\usepackage{tikz-cd}
\usepackage{epstopdf}
\usepackage{subcaption}
\usepackage{enumitem}
\usepackage{float}
\numberwithin{equation}{section}
\numberwithin{figure}{section}

\DeclareMathAlphabet{\pazocal}{OMS}{zplm}{m}{n}

\input xy

\xyoption{all}
    \advance \topmargin by -\headheight
    \advance \topmargin by -\headsep

    \evensidemargin \oddsidemargin
\makeatletter
\newcommand{\leqnomode}{\tagsleft@true\let\veqno\@@leqno}
\newcommand{\reqnomode}{\tagsleft@false\let\veqno\@@eqno}
\makeatother

\newtheorem {theorem}    {Theorem}[section]

\theoremstyle{definition}
\newtheorem {lemma}      [theorem]    {Lemma}
\newtheorem {corollary}  [theorem]    {Corollary}

\theoremstyle{definition}

\theoremstyle{definition}
\newtheorem{remark}[theorem]{Remark}

\newcommand{\defeq}{\vcentcolon=}
\newcommand{\gen}[1]{\langle #1 \rangle}
\newcommand\numberthis{\addtocounter{equation}{1}\tag{\theequation}}

\def\a{\alpha}                
\def\b{\beta}

\def\eps{\varepsilon}

\def\a{\alpha}
\def\b{\beta}

\def\N{\mathbb{N}}     % natural numbers
\def\lab{{\text{Lab}}}

\def\vertexradius{.1}
\def\vertex(#1){\fill (#1) circle (\vertexradius)}

\begin{document}

%%%%%%%%%%%%%%%%%%%%%%%%%%%%%%%%%%%%%%%%%%%%%%%%%%%%%%%%%%%%%%%%%
%%%%%%%%%%%%%%%%%%%%%%%%%%%%%%%%%%%%%%%%%%%%%%%%%%%%%%%%%%%%%%%%%
\title{\bf Torsion Subgroups of Groups with Quadratic Dehn Function}
\maketitle
\begin{center}

Francis Wagner

\end{center}

\bigskip

\begin{center}

\textbf{Abstract}

\end{center}

We construct the first examples of finitely presented groups with quadratic Dehn function containing a finitely generated infinite torsion subgroup. These examples are ``optimal'' in the sense that the Dehn function of any such finitely presented group must be at least quadratic. Moreover, we show that for any $n\geq2^{48}$ such that $n$ is either odd or divisible by $2^9$, any infinite free Burnside group with exponent $n$ is a quasi-isometrically embedded subgroup of a finitely presented group with quadratic Dehn function satisfying the Congruence Extension Property.

\bigskip

%%%%%%%%%%%%%%%%%%%%%%%%%%%%%%%%%%%%%%%%%%%%%%%%%%%%%%%%%%%%%%%%%

\section{Introduction}

Let $\pazocal{A}$ be an alphabet and $\pazocal{R}$ be a set of reduced words in the alphabet $\pazocal{A}\cup\pazocal{A}^{-1}$. Letting $F(\pazocal{A})$ be the free group with basis $\pazocal{A}$, define the \textit{normal closure} of $\pazocal{R}$ in $F(\pazocal{A})$, denoted $\gen{\gen{\pazocal{R}}}$, to be the smallest normal subgroup of $F(\pazocal{A})$ containing $\pazocal{R}$. One can verify that this subgroup exists and is generated by the set of reduced words of the form $fRf^{-1}$, where $f\in F(\pazocal{A})$ and $R\in\pazocal{R}$. The group $F(\pazocal{A})/\gen{\gen{\pazocal{R}}}$ is then denoted by $\gen{\pazocal{A}\mid\pazocal{R}}$.

Given a group $G$ isomorphic to $F(\pazocal{A})/\gen{\gen{\pazocal{R}}}$, it is convenient to view $G$ as being generated by $\pazocal{A}$, so that elements of $G$ can be represented by reduced words over $\pazocal{A}$. With this interpretation, $G$ is said to have \textit{presentation} $\pazocal{P}=\gen{\pazocal{A}\mid\pazocal{R}}$. It follows immediately that a reduced word $W$ in the alphabet $\pazocal{A}\cup\pazocal{A}^{-1}$ represents the identity in $G$ if and only if there exist $k\in\N$, $f_1,\dots,f_k\in F(\pazocal{A})$, $R_1,\dots,R_k\in\pazocal{R}$, and $\eps_1,\dots,\eps_k\in\{\pm1\}$ such that $W=\prod\limits_{i=1}^k f_iR_i^{\eps_i}f_i^{-1}$ in $F(\pazocal{A})$. If $W=1$ in $G$, then its \textit{area} with respect to $\pazocal{P}$, $\text{Area}_\pazocal{P}(W)$, is the minimal value of $k$ so that there exists such a representation of $W$.

Alternatively, given a group $G$ with presentation $\pazocal{P}$, the area of a word $W$ representing the identity in $G$ can be defined as the minimal area of a van Kampen diagram $\Delta$ over $\pazocal{P}$ (see Section 2.1) such that $\lab(\partial\Delta)\equiv W$, where $\equiv$ represents (here and throughout the rest of this paper) letter-for-letter equality.

If both $\pazocal{A}$ and $\pazocal{R}$ are finite, then the group $G$ is called \textit{finitely presented}. First introduced in [7], the \textit{Dehn function} of the group $G$ with respect to its finite presentation $\pazocal{P}=\gen{\pazocal{A}\mid\pazocal{R}}$ is the function $\delta_{\pazocal{P}}:\N\to\N$ defined by $\delta_{\pazocal{P}}(n)=\max\{\text{Area}(W):|W|_{\pazocal{A}}\leq n\}$.

Dehn functions are defined up to an asymptotic equivalence $\sim$ taken on functions $\N\to\N$ defined by $f\sim g$ if and only if $f\preccurlyeq g$ and $g\preccurlyeq f$, where $f\preccurlyeq g$ if and only if there exists $C>0$ such that $$f(n)\leq Cg(Cn)+Cn+C$$ for all $n\in\N$. Given a finitely presented group $G$ with finite presentations $\pazocal{P}$ and $\pazocal{S}$, one can verify that $\delta_{\pazocal{P}}\sim\delta_{\pazocal{S}}$. So, given a finitely presented group $G$, we define the \textit{Dehn function} of $G$, $\delta_G$, as the Dehn function of any of its finite presentations.

The Dehn function is a useful invariant for studying finitely presented groups. Just two of the numerous examples of this are:

\begin{enumerate}[label=({\arabic*})]

\item The Dehn function is closely related to the solvability of the word problem in the group, with smaller Dehn functions corresponding to groups with more tractable word problems [4], [24]. 

\item If $G$ is the fundamental group of a compact Riemannian manifold $M$, then $\delta_G$ is equivalent to the smallest isoperimetric function of the universal cover $\tilde{M}$.

\end{enumerate}

Note that under the equivalence relation $\sim$, all polynomial functions of degree $d$ are equivalent to one another. Because of this, it makes sense to consider groups of linear Dehn function, groups of quadratic Dehn function, etc. A finitely presented group is word hyperbolic in the sense of Gromov if and only if its Dehn function is linear [7]. Moreover, any finitely presented group $G$ satisfying $\delta_G\prec n^2$ is word hyperbolic [7], [3], [13]. This `gap' in possible Dehn functions leads naturally to the following question: 

\begin{center}

What properties satisfied by hyperbolic groups are satisfied by groups with quadratic Dehn function?

\end{center}

For example, hyperbolic groups are known to have solvable conjugacy problem ([7], [4]), leading Rips to pose the question of the solvability of the conjugacy problem in groups with quadratic Dehn function in the early 1990s. In 2020, Ol'shanskii and Sapir [23] answered this problem in the negative, exhibiting groups with quadratic Dehn function and unsolvable conjugacy problem. A problem arising in a similar manner (using methods similar to those used in [23]) is what is addressed in this paper.

The Burnside problem, first posed in 1902, asked whether or not there exists a finitely generated infinite torsion group. Although the problem was solved in the affirmative by Golod and Shaferevich in 1964 [6], the constructed examples did not have finite exponent. This led to the Bounded Burnside problem, asking whether there exists a finitely generated infinite group of exponent $n$.

For any $n>1$ and any set $\pazocal{A}$, let $F(\pazocal{A})^n$ be the normal subgroup of $F(\pazocal{A})$ generated by all words of the form $w^n$. Then the group $F(\pazocal{A})/F(\pazocal{A})^n$ is the \textit{free group relative the class of groups of exponent $n$} (this class is also known as the Burnside variety $\pazocal{B}_n$). This terminology is justified by the universal property of relatively free groups: If $G$ is a group such that $g^n=1$ for all $g\in G$ and $G$ is generated by $\{g_i\}_{i\in I}$, then for $\pazocal{A}=\{a_i\}_{i\in I}$, there exists an epimorphism $\phi:F(\pazocal{A})/F(\pazocal{A})^n\to G$ such that $\phi(a_iF(\pazocal{A})^n)=g_i$ for all $i\in I$. For convenience, the group $F(\pazocal{A})/F(\pazocal{A})^n$ is called a \textit{free Burnside group} and is denoted $B(\pazocal{A},n)$, or simply $B(m,n)$ if $|\pazocal{A}|=m$.

Hence, the bounded Burnside problem essentially asks whether there exists $m,n\in\N$ such that $B(m,n)$ is infinite (and, if so, for which choices of $m,n$). Novikov and Adian were the first to give examples of $m,n$ such that $B(m,n)$ is infinite, specifically for all $m>1$ and $n\geq4381$ odd [11]. Adian later improved the bound on $n$ to $n\geq665$ in 1978 [1]. In 1982, Ol'shanskii provided a simpler geometric proof that $B(m,n)$ is infinite for $m>1$ and sufficiently large odd $n$ ($n>10^{10}$), as well as proving the existence of the so-called Tarski monster groups [12]. Ivanov then proved in 1994 an analogous result for $n$ even, divisible by $2^9$, and sufficiently large ($n\geq2^{48}$) [8].

Though the infinite torsion groups constructed in each of these papers have solvable word problems, it is established they cannot be finitely presented, i.e they cannot be presented by a finite number of relations. As such, one cannot speak of the Dehn function of $B(m,n)$ for sufficiently large $n$. However, an infinite free Burnside group may be isomorphic to a proper subgroup of a finitely presented group, begging the following question:

\begin{center}

If $B(m,n)$ is infinite and $G$ is a finitely presented group such that $B(m,n)$ embeds into $G$, then what can we say about $\delta_G$?

\end{center}

Ghys and de la Harpe proved in 1991 that no hyperbolic group contains an infinite torsion subgroup [5]. In particular, this means that the we must have $n^2\preccurlyeq\delta_G$. 

On the other hand, in 2000, Ol'shanskii and Sapir exhibited an embedding of $B(m,n)$ for sufficiently large odd $n$ into a finitely presented group $G$ satisfying $\delta_G\preccurlyeq n^{10}$ [17].

Let $\N^*$ be the subset of the natural numbers defined by $n\in\N^*$ if and only if $n\geq2^{48}$ and is either odd or divisible by $2^9$. Using similar methods to those used in [16] and [23] and the geometric methods of [8] and [12] we prove the following here.

\begin{theorem} \label{main theorem}

For $m>1$ and $n\in\N^*$, there exists a finitely presented group $G_{m,n}$ with quadratic Dehn function into which the free Burnside group $B(m,n)$ embeds. In particular, there exists a finitely presented group $G$ with quadratic Dehn function containing a finitely generated infinite torsion subgroup.

\end{theorem}

If $\pazocal{A}=\{a_i\}_{i=1}^\infty$, then denote the free Burnside group $B(\pazocal{A},n)$ simply by $B(\infty,n)$. For $n$ sufficiently large and odd, Shirvanyan exhibited an embedding of $B(\infty,n)$ in $B(2,n)$ [27], while Ivanov and Ol'shanskii exhibited such an embedding for $n\geq2^{48}$ and divisible by $2^9$ [9]. 

Thus, taking $G_n=G_{2,n}$, Theorem \ref{main theorem} immediately implies the following corollary:

\begin{corollary}

For all $n\in\N^*$, there exists a finitely presented group $G_n$ with quadratic Dehn function into which the free Burnside group $B(\infty,n)$ embeds. In particular, for $m\geq2$, $G_n$ contains a subgroup isomorphic to $B(m,n)$.

\end{corollary}

A map $f:(X,d_X)\to(Y,d_Y)$ between two metric spaces is called a \textit{quasi-isometric embedding} if there exist $C\geq1$ and $K\geq0$ such that for all $x_1,x_2\in X$, 
$$\frac{1}{C}d_X(x_1,x_2)-K\leq d_Y(f(x_1),f(x_2))\leq Cd_X(x_1,x_2)+K$$
If $K=0$, then $f$ is called a \textit{bi-Lipschitz embedding}. Note that, unlike a quasi-isometric embedding, a bi-Lipschitz embedding is necessarily an embedding (as quasi-isometric embeddings need not be injective).

Let $G$ be a finitely generated group with finite generating set $X$. Then, $G$ can be viewed as a metric space with metric $d_X$ induced by the word norm $|\cdot|_X$. In other words, $d_X(g_1,g_2)=|g_1^{-1}g_2|_X$, i.e the word length of a shortest word in the alphabet $X\cup X^{-1}$ whose value in $G$ is $g_1^{-1}g_2$.

%Suppose $X$ and $Y$ are finite generating sets of $G$. Letting $C=\max\{|x|_Y,|y|_X: x\in X, y\in Y\}$, it follows that the identity map $id_G:(G,d_X)\to(G,d_Y)$ is a bi-Lipschitz embedding.

Now, suppose $G$ and $H$ are two finitely generated groups with finite generating sets $X$ and $Y$, respectively. Further, suppose there exists a monomorphism $\varphi:G\to H$. Then, it is clear that $\varphi$ is a bi-Lipschitz embedding if and only if there exists a $C\geq1$ such that for all $g\in G$,
$$\frac{1}{C}|g|_X\leq|\varphi(g)|_Y\leq C|g|_X$$
Letting $C_1=\max\{|\varphi(x)|_Y:x\in X\}$, it follows that for any $g\in G$, $|\varphi(g)|_Y\leq C_1|g|_X$. Hence, $\varphi$ is a bi-Lipschitz embedding if and only if there exists $C_2\geq1$ such that for any $g\in G$, $|g|_X\leq C_2|\varphi(g)|_Y$.

%Let $G$ and $H$ be finitely generated groups with finite generating set $X$ and $Y$, respectively, and such that $G\leq H$. Defining the word norms $|\cdot|_X$ and $|\cdot|_Y$, respectively, in the natural way, then the \textit{distortion} of $G$ in $H$ is the function $\text{disto}_G^H:\N\to\N$ defined by
%$$\text{disto}_G^H(n)=\max\{|g|_X : g\in G, \ |g|_Y\leq n\}$$
%The distortion is taken up to the equivalence $\simeq$ defined on functions $\N\to\N$ arising from the preorder $\preceq$ given by $f_1\preceq f_2$ if and only if there exists $C>0$ such that for all natural numbers $n$,
%$$f_1(n)\leq Cf_2(Cn)$$
%This justifies the notation given to the distortion function, as it follows that the function is well-defined (up to $\simeq$) for any choice of finite generating sets.
%
%Setting $C=\max\{|x|_Y:x\in X\}$, for any $g\in G$, $|g|_Y\leq C|g|_X$. So, for any $n$, $\text{disto}_G^H(Cn)\geq n$. 
%
%If $G$ and $H$ are both finitely generated groups, then the monomorphism $\iota:G\xhookrightarrow{}H$ is called a \textit{quasi-isometric embedding} of $G$ into $H$ if and only if $\text{disto}_G^H\simeq n$ (where $G$ is identified with $\iota(G)$).
%
%Hence, as it is clear that $\text{disto}_G^H$ is a nondecreasing function, $G$ is quasi-isometrically embedded in $H$ if and only if there exists $M>0$ such that $\text{disto}_G^H(\ell)\leq M\ell$ for all $\ell$.

\begin{theorem} \label{distortion}

The embedding given in Theorem \ref{main theorem} is a bi-Lipschitz embedding (and so a quasi-isometric embedding) of the free Burnside group $B(m,n)$ into the finitely presented group $G_{m,n}$.

\end{theorem}

%For $H\leq G$, define $\Delta_m(G;H)$ to be the amalgamated free product of $m$ copies of $G$ along their subgroups $H$. 

A subgroup $G$ of a group $H$ satisfies the \textit{Congruence Extension Property} (CEP) if for any epimorphism $\eps:G\to G_1$, there exists an epimorphism $\bar{\eps}:H\to H_1$ for some group $H_1$ containing $G_1$ as a subgroup and such that the restriction of $\bar{\eps}$ to $G$ is $\eps$. In this case, we write $G\leq_{CEP}H$ and say that $G$ is a CEP-subgroup of $H$ or that $G$ is CEP-embedded in $H$.

There are two convenient reformulations of the definition of CEP:

\begin{enumerate}

\item $G$ is a CEP-subgroup of $H$ if and only if for any normal subgroup $N\triangleleft G$, there exists a normal subgroup $M\triangleleft H$ such that $M\cap G=N$

\item $G$ is a CEP-subgroup of $H$ if and only if for any subset $S\subseteq G$, $G\cap\gen{\gen{S}}^G=\gen{\gen{S}}^H$ (where the normal closure of a subset $T$ in a group $K$ is denoted $\gen{\gen{T}}^K$).

\end{enumerate}

It is clear from (1) that any retract of a group is a CEP-subgroup and that $\leq_{CEP}$ is a transitive relation. However, some examples are less obvious. For example, Sonkin proved that for sufficiently large odd $n$, there exists a CEP-embedding of $B(\infty,n)$ into the group $B(2,n)$ [28].

\begin{theorem} \label{CEP}

The embedding given in Theorem \ref{main theorem} is a CEP-embedding of the free Burnside group $B(m,n)$ into the finitely presented group $G_{m,n}$.

\end{theorem}

As $B(m,n)$ is a retract of $B(\infty,n)$, Theorem \ref{CEP} immediately implies the following corollary:

\begin{corollary}

For all odd $n\in\N^*$, there exists a finitely presented group $G_n$ with quadratic Dehn function such that for any $m\geq2$, there is a CEP-embedding of $B(m,n)$ into $G_n$. Moreover, $G_n$ contains a CEP-subgroup isomorphic to $B(\infty,n)$.

\end{corollary}

As in [16] and [23], the construction of the groups of interest is through $S$-machines. $S$-machines were first introduced by Sapir in [24]; for a formal definition, see Section 3.1 below. Similar to the construction in those previous papers, we first create several auxiliary machines $\textbf{M}_1-\textbf{M}_4$ satisfying some desirable properties (see Section 4). The chief properties on which we base this construction are the following:

\begin{enumerate}[label=({\alph*})]

\item The language of accepted inputs is a set of relators for a presentation of the free Burnside group $B(m,n)$ (see Lemma \ref{M_3 language});

\item Any accepting computation of a word $u^n$ is linearly bounded by $\|u\|$ (see Lemma \ref{M_3 start to end});

\item The majority of an accepting computation is spent on one particular step (see Lemma \ref{M_3 input controlled}); and

\item The length of a computation in a specific class of bases is bounded by the length of the initial or terminal admissible words (see Lemma \ref{M_4 faulty}).

\end{enumerate}

Many copies of the machine $\textbf{M}_4$ are then `concatenated' to form our main machine $\textbf{M}$, a process that resembles the consideration of groups of interest in [16] and [23]. However, unlike in those sources, one copy of $\textbf{M}_4$ is deemed `special' and is operated upon in a different manner as the other copies, causing a distinct non-uniformity.

The purpose of this lack of symmetry is to allow our machine to accept two configurations which differ only in the insertion/deletion of an accepted input. As all accepted configurations are trivial in the group $G(\textbf{M})$ associated to the machine, this implies the relation $w=1$ for any word $w$ over the alphabet $\pazocal{A}$ of the input sector that represents the trivial element of $B(\pazocal{A},n)$.

%This permits the addition of extra relations, called $a$-relations, to the presentations of the groups associated to $\textbf{M}$. Particularly, these relations take the form $w=1$ where $w$ is a word that is trivial in $B(m,n)$. As such, there exists a natural homomorphism $\varphi$ from $B(m,n)$ to the groups associated to $\textbf{M}$. 

Conversely, this asymmetry is also the source of several new obstacles not faced in [16] or [23]. For example, many statements in Section 5 are devoted to understanding the relationship between computations of one copy of $\textbf{M}_4$ and computations of the standard base of $\textbf{M}$ (for example, see Lemmas \ref{extending computations} and \ref{projection admissible configuration not}), a relationship that would be trivial had the rules operated with symmetry.

In Sections 6-11, we study diagrams over the groups associated to the $S$-machine $\textbf{M}$, culminating in the proof of Theorem \ref{main theorem}. The general method of study follows a similar path to those followed in [16] and [23], but with one major change: The consideration of the groups $M_\Omega(\textbf{M})$ and $G_\Omega(\textbf{M})$ constructed by the addition of extra relations, called $a$-relations, to the groups $M(\textbf{M})$ and $G(\textbf{M})$, respectively. The set of relators $\Omega$ corresponding to the $a$-relations consist of words over the alphabet of the input sector and contains the set $\pazocal{S}$ of all words that represent the trivial element of $B(\pazocal{A},n)$.

The cells of a diagram over $M_\Omega(\textbf{M})$ or $G_\Omega(\textbf{M})$ corresponding to elements of $\Omega$, referred to as $a$-cells, are invaluable to the proof that $\varphi$ is an embedding (see Lemma \ref{embedding}) but cause a new obstacle in virtually every diagrammatic consideration. For example, the consideration of rim $\theta$-bands of a diagram must be replaced with the consideration of quasi-rim $\theta$-bands, i.e a band that may have $a$-cells between it and the boundary (see Lemma \ref{6.18}).

The proof of Theorem \ref{distortion} is presented Section 12. Its makeup is similar to the diagrammatic arguments presented in Section 10, but is unique to this setting in that it has no analogue in [16] or [23].

We conclude with the proof of Theorem \ref{CEP} in Section 13. The proof is a consequence of the arguments pertaining to minimal diagrams introduced in Sections 6-9.

Finally, we mention here the importance of the group $B(m,n)$ to this construction and proof. In the context of the proof of Theorem \ref{main theorem}, it is clear that the following two properties of the embedded group were necessary for the proof to follow: The existence of a presentation of the group whose relators satisfy some linear bound as in (b) above and the existence of another presentation of the group such that any van Kampen diagram over this presentation satisfies some quadratic bound as in Lemma \ref{a-cells are quadratic B(m,n)}. However, there is a third, more subtle requirement: In the proof of Lemma \ref{impeding G-weight}, it is essential that the relators are periodic. Due to this demand, that we are studying a group in the Burnside variety is crucial to our construction.

\bigskip

\section{Maps and diagrams}

A vital tool for many of the arguments to come is the concept of van Kampen diagrams over group presentations, a notion introduced by its namesake in 1933 [29]. It is assumed that the reader is intimately acquainted with this concept. The following subsection functions to recall the most important definitions; for further reference, see [14], [10], and [26].

\subsection{van Kampen diagrams} \

Let $G$ be a group with presentation $\gen{\pazocal{A}\mid\pazocal{R}}$. Suppose $\Delta$ is an oriented 2-complex homeomorphic to a disk equipped with a \textit{labelling function}, i.e a function $\lab:E(\Delta)\to\pazocal{A}\cup\pazocal{A}^{-1}\cup\{1\}$ which satisfies $\lab(e^{-1})\equiv\lab(e)^{-1}$ for any edge $e\in E(\Delta)$ (with, of course, $1^{-1}\equiv1$). The label of a path in $\Delta$ is defined in the obvious way, that is $\lab(e_1\dots e_n)\equiv\lab(e_1)\dots\lab(e_n)$. For any edge $e$ in $\Delta$, $e$ is called a $0$-edge if $\lab(e)\equiv1$; otherwise, $e$ is called an $\pazocal{A}$-edge.

Suppose that for each cell $\Pi$ of $\Delta$, one of the following is true:

\begin{enumerate}[label=({\arabic*})]

\item omitting the label of any zero edges, $\lab(\partial\Pi)$ is visually equal to a cyclic permutation of $R^{\pm1}$ for some $R\in\pazocal{R}$
%(Figure \ref{R-cell})

\item $\partial\Pi$ consists of $0$-edges and exactly two $\pazocal{A}$-edges $e$ and $f$, with $\lab(e)\equiv\lab(f^{-1})$
%(Figure \ref{0-cell-1})

\item $\partial\Pi$ consists only of $0$-edges. 
%(Figure \ref{0-cell-2})

\end{enumerate}

%\begin{figure}[H]
%\centering
%\begin{subfigure}[c]{0.3\textwidth}
%\centering
%\includegraphics[width=.95\textwidth]{Rcell.eps}
%\caption{An $\pazocal{R}$-cell corresponding to $R=aba^{-1}b^{-1}$. }
%\end{subfigure}\hfill
%\begin{subfigure}[c]{0.3\textwidth}
%\centering
%\includegraphics[width=0.34\textwidth]{0cell1.eps}
%\caption{A $0$-cell of type (2), $a\in\pazocal{A}$.}
%\end{subfigure}\hfill
%\begin{subfigure}[c]{0.3\textwidth}
%\centering
%\includegraphics[width=.85\textwidth]{0cell2.eps}
%\caption{A $0$-cell of type (3).}
%\end{subfigure}
%\caption{Cells in a van Kampen diagram}
%\end{figure}

Then $\Delta$ is called a \textit{(disk) van Kampen diagram} (or simply a \textit{disk diagram}) over the presentation $\gen{\pazocal{A}\mid\pazocal{R}}$. The cells satisfying condition (1) above are called $\pazocal{R}$-cells, while the others are called 0-cells.

It is easy to see that the contour, $\partial\Delta$, of a disk diagram $\Delta$ has label equal to the identity in $G$. Conversely, van Kampen's Lemma (Lemma 11.1 of [13]) states that a word $W$ over $\pazocal{A}$ represents the identity of $G$ if and only if there exists a disk diagram $\Delta$ over the presentation $\gen{\pazocal{A}\mid\pazocal{R}}$ with $\lab(\partial\Delta)\equiv W$.

The \textit{area} of a disk diagram $\Delta$, denoted $\text{Area}(\Delta)$, is the number of $\pazocal{R}$-cells it contains, while the area of a word $W$ satisfying $W=1$ in $G$ is the minimal area of a diagram $\Delta$ satisfying $\lab(\partial\Delta)\equiv W$.

\begin{figure}[H]
\centering
\subcaptionbox{$\pazocal{R}$-cell corresponding to the relator $R=aba^{-1}b^{-1}$.}[0.33\textwidth]{\includegraphics[height=1.9in]{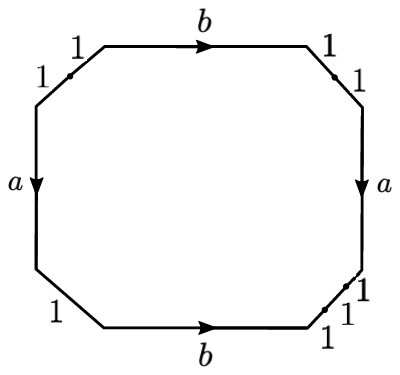}}\hfill
\subcaptionbox{$0$-cell of type (2), $a\in\pazocal{A}$.}[0.33\textwidth]{\includegraphics[height=1.9in]{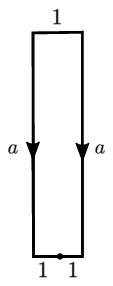}}\hfill
\subcaptionbox{$0$-cell of type (3).}[0.33\textwidth]{\includegraphics[height=1.9in]{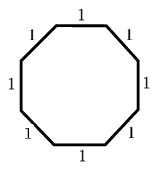}}\hfill
\caption{Cells in van Kampen diagrams}
\end{figure}

A \textit{0-refinement} of a disk diagram $\Delta$ is a disk diagram $\Delta'$ obtained from $\Delta$ by the insertion of 0-edges or 0-cells. Note that a 0-refinement has the same area as the diagram from which it arises.

Let $\Delta$ be a disk diagram and $\Pi_1$, $\Pi_2$ be two $\pazocal{R}$-cells in $\Delta$. Suppose there exists a simple path $t$ between the vertices $O_1,O_2$ of $\Pi_1,\Pi_2$, respectively, such that:

\begin{itemize}

\item $\lab(t)=1$ in $F(\pazocal{A})$ (that is, the free group with basis $\pazocal{A}$), and 

\item $\lab(\partial\Pi_1)$ read starting at $O_1$ is mutually inverse to $\lab(\partial\Pi_2)$ read starting at $O_2$ %(Figure \ref{cancellable}). 

\end{itemize}

Then $\Pi_1$ and $\Pi_2$ are called \textit{cancellable} in $\Delta$.

\begin{figure}[H]
\centering
\includegraphics[scale=1.25]{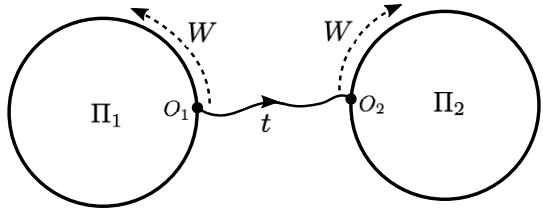}
\caption{Cancellable cells}
\end{figure}

This term is justified by the ability to `remove' the cells $\Pi_1$ and $\Pi_2$ from $\Delta$ without affecting the label of $\partial\Delta$, yielding a disk diagram $\Delta'$ satisfying $\lab(\partial\Delta')\equiv\lab(\partial\Delta)$ with $\text{Area}(\Delta')<\text{Area}(\Delta)$.

Naturally, a disk diagram is called \textit{reduced} if it has no pair of cancellable cells. By simply removing pairs of cancellable cells, any disk diagram over a presentation can be made reduced. This immediately implies a strengthened version of van Kampen's lemma: A word $W$ over $\pazocal{A}$ represents the identity in $G$ if and only if there exists a reduced disk diagram $\Delta$ over the presentation with $\lab(\partial\Delta)\equiv W$.

An annular (Schupp) diagram over the presentation $\gen{\pazocal{A}\mid\pazocal{R}}$ is defined in the analogous way. It is then an immediate consequence of van Kampen's lemma that two words $W$ and $V$ are conjugate in $G$ if and only if there exists a reduced annular diagram $\Delta$ with contour components $p$ and $q$ satsifying $\lab(p)\equiv W$ and $\lab(q)\equiv V^{-1}$. %(Figure \ref{annular}).

\begin{figure}[H]
\centering
\includegraphics[scale=0.9]{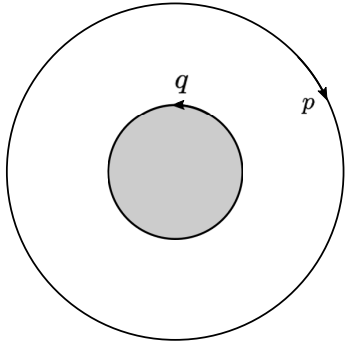}
\caption{Annular diagram}
\end{figure}

\smallskip

%%%%%%%%%%%%%%%%%%%%%%%%%%%%%%%%%%%%%%%%%%%%%%%%%%%%%%%%%%%%%%%%%

\subsection{Graded maps on a disk or annulus} \

The definitions and statements presented over the next several subsections can be found in [14] and [8]. Those relevant to the proof of Lemma \ref{a-cells are quadratic B(m,n)} are restated here for convenience, with reference given in place of proofs.

A \textit{map} $\Delta$ is a finite oriented planar graph on a disk which subdivides the surface into polygonal cells. In particular, by `forgetting' the labelling, one can interpret a van Kampen diagram as a map.

A map $\Delta$ is called \textit{graded} if each cell $\Pi$ in $\Delta$ is assigned a nonnegative integer $r(\Pi)$ called its \textit{rank}. The map $\Delta$ is called a \textit{map of rank at most k} if all its cells have rank $\leq k$. The minimal $k$ for which $\Delta$ is a map of rank at most $k$ is called the \textit{rank} of $\Delta$ and denoted $r(\Delta)$.

For $r(\Delta)=k$, the \textit{type} of $\Delta$, $\tau(\Delta)$, is the $(k+2)$-vector $(r(\Delta),\tau_0,\dots,\tau_k)$, where $\tau_i$ is the number of cells of rank $k-i$ in $\Delta$. The types of maps are ordered lexicographically, i.e for two maps $\Delta$ and $\Gamma$ with $\tau(\Delta)=(r(\Delta),\tau_0,\dots,\tau_k)$ and $\tau(\Gamma)=(r(\Gamma),\sigma_0,\dots,\sigma_\ell)$, $\tau(\Delta)\leq\tau(\Gamma)$ if the following three conditions hold:

\begin{itemize}

\item $r(\Delta)\leq r(\Gamma)$;

\item if $r(\Delta)=r(\Gamma)$, then $\tau_0\leq\sigma_0$;

\item for $1\leq i\leq r(\Delta)$, if $r(\Delta)=r(\Gamma)$ and $\tau_j=\sigma_j$ for all $j<i$, then $\tau_i\leq\sigma_i$.

\end{itemize}

For simplicity, the cells of rank 0 in a graded map are called \textit{0-cells}. All other cells are called \textit{$\pazocal{R}$-cells} (even though an alphabet $\pazocal{R}$ is not specified). 

The edges of the graph are divided into two disjoint sets, called the \textit{0-edges} and the \textit{$\pazocal{A}$-edges}. The \textit{length} of a path $p$ in a graded map $\Delta$, denoted $|p|$, is the number of $\pazocal{A}$-edges that comprise it. In particular, for $\partial\Pi$ the contour of a cell, $|\partial\Pi|$ is called the \textit{perimeter} of $\Pi$.

Motivated by the definition of van Kampen diagrams, the following three facts are assumed about graded maps:

\newpage

\begin{enumerate}[label=({\arabic*})]

\item the inverse edge of a 0-edge is also a 0-edge

\item the contour of a 0-cell either consists entirely of 0-edges or of exactly two $\pazocal{A}$-edges in addition to a number of 0-edges

\item if $\Pi$ is an $\pazocal{R}$-cell, then $|\partial\Pi|>0$

\end{enumerate}

If $\Delta$ is a graded map and $\Gamma$ is a subspace homeomorphic to a disk bounded by some edgepath of $\Delta$, then $\Gamma$ is called a \textit{submap} of $\Delta$.

It is assumed that the contour of a graded map has a fixed decomposition into at most eight distinct parts. In particular, if $\Delta$ is a graded map, then there is a standard factorization $p_1\dots p_k$ of $\partial\Delta$, with $k\leq8$ and each $p_i$ called a \textit{section} of the contour.

\smallskip

%%%%%%%%%%%%%%%%%%%%%%%%%%%%%%%%%%%%%%%%%%%%%%%%%%%%%%%%%%%%%%%%%

\subsection{0-Bonds and 0-contiguity submaps} \

Let $\Delta$ be a graded map and $\Pi$ be a 0-cell whose contour contains exactly two $\pazocal{A}$-edges, $e_1$ and $e_2$. Then the pair of edges $e_1,e_2^{-1}$ are called \textit{immediately adjacent} (as is the pair $e_1^{-1},e_2$). Two edges $e$ and $f$ of $\Delta$ are then said to be \textit{adjacent} if there exists a sequence of edges $e=e_1,e_2,\dots,e_{k+1}=f$ such that $e_i$ and $e_{i+1}$ are immediately adjacent for $i=1,\dots,k$.

Let $\Delta$ be a graded map with adjacent edges $e$ and $f$. Suppose $e$ belongs to the contour of the $\pazocal{R}$-cell $\Pi_1$ and $f^{-1}$ to the contour of some $\pazocal{R}$-cell $\Pi_2$. Per the definition, set $e=e_1,\dots,e_{k+1}=f$ with 0-cells $\pi_1,\dots,\pi_k$ such that the only two $\pazocal{A}$-edges of $\partial\pi_i$ are $e_i^{-1}$ and $e_{i+1}$.

We can then write $\partial\pi_i=e_i^{-1}p_ie_{i+1}s_i$ for $i=1,\dots,k$ such that $|p_i|=|s_i|=0$. With the aid of 0-refinement, we can assume that $p=p_1\dots p_k$ and $s=s_k\dots s_1$ are simple paths such that each intersects $\Pi_1,\Pi_2$ only on its endpoints.

Then, the submap $\Gamma$ with contour $p^{-1}es^{-1}f^{-1}$ consisting of the cells $\pi_1,\dots,\pi_k$ is called a \textit{0-bond} between $\Pi_1$ and $\Pi_2$. The edges $e$ and $f^{-1}$ are called the \textit{contiguity arcs} of the 0-bond $\Gamma$ and $p$ and $s$ the \textit{side arcs}.

\begin{figure}[H]
\centering
\includegraphics[scale=1]{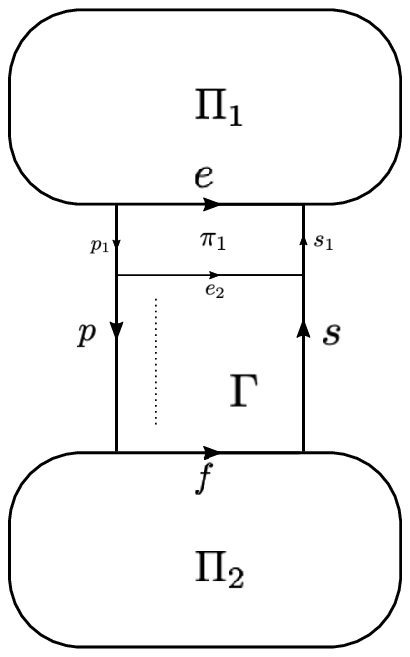}
\caption{A 0-bond between two $\pazocal{R}$-cells}
\end{figure}

Similarly, if $e$ and $f$ are adjacent edges with $e$ belonging to the contour of some $\pazocal{R}$-cell $\Pi$ and $f^{-1}$ belonging to some section $q$ of the contour, then a 0-bond between $\Pi$ and $q$ is defined. A 0-bond between two sections of the contour is defined analogously.

Now suppose $e_1,f_1$ and $e_2,f_2$ are two pairs of adjacent edges such that $e_1$ and $e_2$ belong to the contour of some $\pazocal{R}$-cell $\Pi_1$ and $f_1^{-1},f_2^{-1}$ to some $\pazocal{R}$-cell $\Pi_2$. Then, construct two 0-bonds, $\Gamma_1$ and $\Gamma_2$, between the two pairs, with $\partial\Gamma_i=z_ie_iw_if_i^{-1}$. If $\Gamma_1=\Gamma_2$, set $\Gamma=\Gamma_1$. Otherwise, there exist subpaths $y_1$ and $y_2$ of $\partial\Pi_1$ and $\partial\Pi_2$, respectively, such that $y_1=e_1pe_2$ and $y_2=f_2^{-1}uf_1^{-1}$ (or $y_1=e_2pe_1$ and $y_2=f_1^{-1}uf_2^{-1}$). Then let $\Gamma$ be the submap with contour $z_1y_1w_2y_2$ (or $z_2y_1w_1y_2$). If $\Gamma$ does not contain $\Pi_1$ or $\Pi_2$, then $\Gamma$ is called a \textit{0-contiguity submap} of $\Pi_1$ to $\Pi_2$. In this case, $y_1$ and $y_2$ are called the \textit{contiguity arcs} of $\Gamma$, denoted $y_i=\Gamma_{  ^\wedge}\Pi_i$. The paths $z_1$ and $w_2$ (or $z_2$ and $w_1$) are called the \textit{side arcs} of $\Gamma$. Note that both side arcs have zero length. The ratio $|y_1|/|\partial\Pi_1|$ is called the \textit{degree of contiguity} of $\Pi_1$ to $\Pi_2$ with respect to $\Gamma$ and is denoted $(\Pi_1,\Gamma,\Pi_2)$. Similarly, $(\Pi_2,\Gamma,\Pi_1)=|y_2|/|\partial\Pi_2|$ is the degree of contiguity of $\Pi_2$ to $\Pi_1$. 

\begin{figure}[H]
\centering
\includegraphics[scale=1.25]{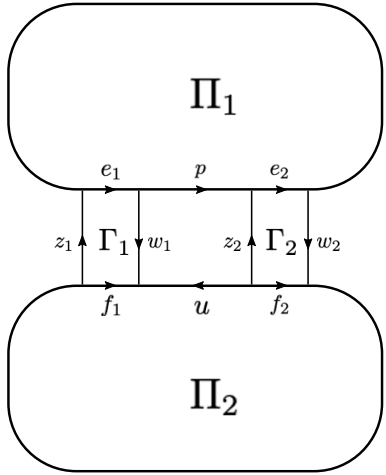}
\caption{A 0-contiguity submap between two $\pazocal{R}$-cells}
\end{figure}

Note, however, that if $\Pi_1=\Pi_2$, then $\Gamma_{ ^\wedge}\Pi_1$ represents two distinct arcs of $\partial\Pi_1$ and $(\Pi_1,\Gamma,\Pi_1)$ a pair of numbers.

As with 0-bonds, a 0-contiguity submap between an $\pazocal{R}$-cell and a section of $\partial\Delta$ is similarly defined, as is a 0-contiguity submap between two sections of $\partial\Delta$. The contiguity arcs, side arcs, and degree of contiguity of such 0-contiguity submaps are defined in the same way as as above; for example, if $\Gamma$ is a 0-contiguity submap between an $\pazocal{R}$-cell $\Pi$ and a section $q$ of the contour of $\partial\Delta$, then the degree of contiguity of $q$ to $\Pi$ is $(q,\Gamma,\Pi)=|\Gamma_{ ^\wedge}q|/|q|$.

Two 0-contiguity submaps $\Gamma_1$ and $\Gamma_2$ are \textit{disjoint} if they have no common cells, their contiguity arcs have no common points, and their side arcs have no common points.

\smallskip

%%%%%%%%%%%%%%%%%%%%%%%%%%%%%%%%%%%%%%%%%%%%%%%%%%%%%%%%%%%%%%%%%

\subsection{Bonds and contiguity submaps} \

In this subsection, $\eps\in(0,1)$ is taken to be a fixed constant. For the moment, one can think of this number as `sufficiently small', with this interpretation made precise in the next section.

Set $k>0$ and suppose the terms $j$-bond and $j$-contiguity submap have been defined for all $0\leq j<k$. Assume further that contiguity arcs, side arcs, and degrees of contiguity are defined for $j$-contiguity submaps in a way similar to how they were defined for 0-contiguity submaps. 

Two submaps $\Gamma_1,\Gamma_2$ such that $\Gamma_i$ is a $j_i$-contiguity submap for $j_i<k$ are called \textit{disjoint} if they have no common cells, their contiguity arcs have no common points, and their side arcs have no common points. Note that this definition agrees with that given for the case $j_1=j_2=0$.

Let $\pi$, $\Pi_1$, and $\Pi_2$ be cells of a graded map $\Delta$, $\Pi_1\neq\Pi_2$, satisfying the following:

\newpage

\begin{enumerate}[label=({\arabic*})]

\item $r(\pi)=k$, $r(\Pi_i)>k$ for $i=1,2$,

\item there are disjoint submaps $\Gamma_1,\Gamma_2$ such that $\Gamma_i$ is a $j_i$-contiguity submap of $\pi$ to $\Pi_i$ for $j_i<k$, $\Pi_1$ is not contained in $\Gamma_2$, and $\Pi_2$ is not contained in $\Gamma_1$,

\item $(\pi,\Gamma_i,\Pi_i)\geq\eps$ for $i=1,2$.

\end{enumerate}

For $i=1,2$, let $\partial\Gamma_i=v_is_i$ for $v_i={\Gamma_i}_{ ^\wedge}\pi$ and $\partial\pi=u_1v_1u_2v_2$. Letting $\Gamma$ be the submap with contour $s_1u_1^{-1}s_2u_2^{-1}$, $\Gamma$ is the \textit{$k$-bond} between $\Pi_1$ and $\Pi_2$ defined by the contiguity submaps $\Gamma_1$ and $\Gamma_2$ with \textit{principal cell} $\pi$. The \textit{contiguity arc} of $\Gamma$ to $\Pi_i$ is defined to be ${\Gamma_i}_{ ^\wedge}\Pi_i$ and denoted $\Gamma_{ ^\wedge}\Pi_i$. The \textit{side arcs} of $\Gamma$ are defined in the obvious way.

\begin{figure}[H]
\centering
\includegraphics[scale=1.25]{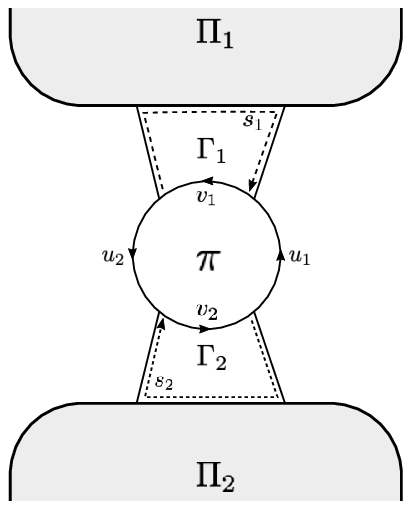}
\caption{A $k$-bond between two $\pazocal{R}$-cells}
\end{figure}

A $k$-bond between an $\pazocal{R}$-cell and a section of the contour or between two distinct sections of the contour is defined similarly.

Suppose $\Gamma_1$ is a $k$-bond between two cells $\Pi_1$ and $\Pi_2$ and $\Gamma_2$ is a $j$-bond between $\Pi_1$ and $\Pi_2$ for $j\leq k$. If $\Gamma_1=\Gamma_2$, then set $\Gamma=\Gamma_1$. Otherwise, if $\Gamma_1$ and $\Gamma_2$ are disjoint, then set $\partial\Gamma_i=z_iv_iw_is_i$ for $v_i={\Gamma_i}_{ ^\wedge}\Pi_1$ and $s_i={\Gamma_i}_{ ^\wedge}\Pi_2$. Then set $y_1$ as a subpath of $\partial\Pi_1$ of the form $v_1vv_2$ (or $v_2vv_1$) and $y_2$ as a subpath of $\partial\Pi_2$ of the form $s_2ss_1$ (or $s_1ss_2$). Setting $\Gamma$ as the submap with contour $z_1y_1w_2y_2$ (or $z_2y_1w_1y_2$), if $\Gamma$ does not contain $\Pi_1$ or $\Pi_2$, then it is called the \textit{$k$-contiguity submap} of $\Pi_1$ to $\Pi_2$ defined by the bonds $\Gamma_1$ and $\Gamma_2$. As with previous definitions, $y_i=\Gamma_{ ^\wedge}\Pi_i$ is called the \textit{contiguity arc} of $\Gamma$ to $\Pi_i$, $z_1$ and $w_2$ (or $w_1$ and $z_2$) are called the \textit{side arcs} of $\Gamma$, $(\Pi_1,\Gamma,\Pi_2)=|y_1|/|\Pi_1|$ is called the \textit{degree of contiguity} of $\Pi_1$ to $\Pi_2$ with respect to $\Gamma$.

A $k$-contiguity submap between an $\pazocal{R}$-cell and a section of the contour is defined similarly, as is a $k$-contiguity submap between two sections of the contour.

The number $k$ is often omitted when referring to $k$-contiguity submaps, so that there will be reference merely to a contiguity submap. Further, if $\Gamma$ is a contiguity submap between $\Pi_1$ and $\Pi_2$ and $\partial\Gamma=p_1q_1p_2q_2$ with $q_i=\Gamma_{ ^\wedge}\Pi_i$, then $\partial(\Pi_1,\Gamma,\Pi_2)$ denotes the \textit{standard decomposition} $p_1q_1p_2q_2$.

\smallskip

%%%%%%%%%%%%%%%%%%%%%%%%%%%%%%%%%%%%%%%%%%%%%%%%%%%%%%%%%%%%%%%%%

\subsection{Graded Presentations} \

Given an alphabet $\pazocal{A}$, let $\{\pazocal{S}_i\}_{i=1}^\infty$ be a collection of subsets of $F(\pazocal{A})$ such that if $W\in\pazocal{S}_i$ and $V$ is a cyclic permutation of $W$ or $W^{-1}$, then $V\notin\pazocal{S}_j$ for any $j\neq i$. 

Set $\pazocal{R}_j=\cup_{i=1}^j\pazocal{S}_i$ for $j\geq1$, $\pazocal{R}_0=\emptyset$, and $\pazocal{R}=\cup_{i=1}^\infty\pazocal{S}_i$. Further, define $G(j)=\gen{\pazocal{A}\mid\pazocal{R}_j}$ for all $j\geq0$. Note that $G(0)\cong F(\pazocal{A})$.

Then $\gen{\pazocal{A}\mid\pazocal{R}}$ is called a \textit{graded presentation} for the group $G=G(\infty)$. 

The words in $\pazocal{S}_i$ are called the \textit{relators of rank $i$}. For words $X,Y$ over $\pazocal{A}$, if $X=Y$ in $G(i)$, then $X$ and $Y$ are said to be \textit{equal in rank $i$}, with this relation denoted $X\stackrel{i}{=}Y$.

Given a disk diagram $\Delta$ over the presentation $\gen{\pazocal{A}\mid\pazocal{R}}$, let $\Pi$ be an $\pazocal{R}$-cell such that $\lab(\partial\Pi)$ is a cyclic permutation of a relator of rank $i$ (or the inverse of such a relator). Then $\Pi$ is called a \textit{cell of rank $i$}, denoted by the representative notation $r(\Pi)=i$. Naturally, the 0-cells of $\Delta$ are called cells of rank 0.

Note that if one forgets the labelling function of a disk diagram $\Delta$ over a graded presentation, then $\Delta$ is a graded map (with the ranks of cells assigned in the same way). A diagram satisfying this property is called a \textit{graded disk diagram}. It is then natural to define the \textit{rank} and \textit{type} of a graded disk diagram as the rank and type of the underlying map.

Let $\Delta$ be a graded disk diagram over $\gen{\pazocal{A}\mid\pazocal{R}}$ containing $\pazocal{R}$-cells $\Pi_1,\Pi_2$ with $r(\Pi_1)=r(\Pi_2)=j$. Suppose there exists a 0-refinement $\Delta'$ of $\Delta$ with copies $\Pi_1',\Pi_2'$ of $\Pi_1,\Pi_2$, respectively, and a simple path $t$ in $\Delta'$ between vertices $O_1,O_2$ of $\Pi_1',\Pi_2'$, respectively, such that:

\begin{itemize}

\item$\lab(t)\stackrel{j-1}{=}1$ and 

\item $\lab(\partial\Pi_1')$ read starting at $O_1$ is mutually inverse to $\lab(\partial\Pi_2')$ read starting at $O_2$.

\end{itemize} 

Then $\Pi_1$ and $\Pi_2$ are called a \textit{j-pair} in $\Delta$. 

This generalizes the concept of cancellable cells in a disk diagram over a presentation: If $\Delta$ is a graded disk diagram over a graded presentation with a $j$-pair $\Pi_1,\Pi_2$, then one can `remove' $\Pi_1$ and $\Pi_2$ from $\Delta$ at the cost of cells of rank $\leq j-1$, producing a graded disk diagram $\Delta''$ over the same presentation with $\lab(\partial\Delta'')\equiv\lab(\partial\Delta)$ and $\tau(\Delta'')<\tau(\Delta)$.

A graded disk diagram $\Delta$ over $\gen{\pazocal{A}\mid\pazocal{R}}$ is called \textit{reduced} if for any graded disk diagram $\Gamma$ over $\gen{\pazocal{A}\mid\pazocal{R}}$ satisfying $\lab(\partial\Delta)\equiv\lab(\partial\Gamma)$, the inequality $\tau(\Delta)\leq\tau(\Gamma)$ is satisfied. Similar to reduced disk diagrams over general presentations, one can make any graded disk diagram reduced simply by the removal of $j$-pairs (for varying $j$). As a result, van Kampen's Lemma can again be strengthened: Given a graded presentation $G=\gen{\pazocal{A}\mid\pazocal{R}}$, a word $W$ over $\pazocal{A}$ represents the identity in $G$ if and only if there exists a reduced graded disk diagram $\Delta$ over the presentation with $\lab(\partial\Delta)\equiv W$.

Graded annular diagrams are defined similarly.

%%%%%%%%%%%%%%%%%%%%%%%%%%%%%%%%%%%%%%%%%%%%%%%%%%%%%%%%%%%%%%%%%

\subsection{Auxiliary parameters} \

The arguments presented through the rest of this section rely on the \textit{lowest parameter principle} introduced in [14]. For this, we introduce the relation $>>$ on parameters defined as follows.

If $\a_1,\a_2,\dots,\a_k$ are (positive) parameters with $\a_1>>\a_2>>\dots>>\a_k$, then for $2\leq i\leq k$, it is understood that $\a_1,\dots,\a_{i-1}$ are assigned prior to the assignment of $\a_i$ and that the assignment of $\a_i$ is dependent on the assignment of its predecessors. The resulting inequalities are then understood as `$\a_i\leq$(any positive-valued expression involving $\a_1,\dots,\a_{i-1}$)'.

The principle makes the sequence of inequalities used throughout the rest of the section consistent without muddling the matter with the arithmetic of particular infinitesimals.

Specifically, the assignment of parameters used in this section is:
$$\b>>\gamma>>\delta>>\eps>>\zeta>>\iota$$
Note that $\eps$ is the parameter used to define contiguity submaps in the previous subsection. 

Further, one more restriction is imposed on the assignment of $\iota$, specifically that its inverse $n=1/\iota$ is an integer that is either odd or divisible by $2^9$ (and so, for small enough $\iota$, $n$ satisfies $(*)$).

In [8], these constants are labelled differently. Indeed, exact values are even given there, with $\b=0.05$, $\gamma=0.01$, $\delta=0.005$, $\eps=2^{-14}$, $\zeta=2^{-33}$, $\iota\leq2^{-48}$.

\smallskip

%%%%%%%%%%%%%%%%%%%%%%%%%%%%%%%%%%%%%%%%%%%%%%%%%%%%%%%%%%%%%%%%%

\subsection{The graded presentation of $B(\pazocal{A},n)$} \

Letting $\pazocal{A}$ be a finite alphabet, fix a total order $\prec$ on the set of words over $\pazocal{A}$ such that if $|X|<|Y|$, then $X\prec Y$. Also, set $\pazocal{R}_0=\emptyset$ and $\textbf{B}(0)=F(\pazocal{A})$.

With this terminology, inductively define $\pazocal{R}_i=\pazocal{R}_{i-1}\cup\{A_i^n\}$ where $A_i$ is the first (relative to $\prec$) with infinite order in $\textbf{B}(i-1)=\gen{\pazocal{A}\mid\pazocal{R}_{i-1}}$. 

The following is a main result of [8]:

\begin{lemma} \label{Theorem B}

\textit{(Theorem B of [8])}. For each $i$, $A_i$ exists, with $|A_i|\leq|A_{i-1}|+1$. Moreover, $\pazocal{R}=\cup\pazocal{R}_i$ can be taken as an independent set of defining relations of the free Burnside group $B(\pazocal{A},n)$ (and so defines a graded presentation $\textbf{B}(\infty)$ of $B(\pazocal{A},n)$).

\end{lemma}

Let $A$ be a freely cyclically reduced word over $\pazocal{A}$. Then a word $W$ over $\pazocal{A}$ is $A$-periodic if $W$ is a subword of a power $A^k$ for $k>0$, i.e $A^k\equiv Z_1WZ_2$ for some (perhaps empty) words $Z_1$ and $Z_2$. A decomposition $W\equiv W_1W_2$ is called \textit{phase} if there exist positive integers $k_1,k_2$ with $k_1+k_2=k$ such that $A^{k_1}\equiv Z_1W_1$ and $A^{k_2}\equiv W_2Z_2$.

If $\Delta$ is a graded diagram over $\textbf{B}(i)$, a section $q$ of $\partial\Delta$ is called $A$-periodic if $\lab(q)$ is an $A$-periodic word. In this case, a vertex of $q$ is called phase if the natural decomposition of $\lab(q)$ it defines is phase. Similarly, if $\Pi$ is an $\pazocal{R}$-cell with $r(\Pi)=j$, then a vertex $O$ of $\partial\Pi$ is called phase if $\lab(\partial\Pi)$ when read starting at $O$ is visually $A_j^{\pm n}$.

Denote $\pazocal{F}(A_i)$ as a finite subgroup of $\textbf{B}(i-1)$ that is maximal with respect to the property that $A_i$ normalizes $\pazocal{F}(A_i)$. Clearly, such a subgroup must exist as $A_i$ normalizes the trivial group. The following statement establishes the well-defined nature of this construction:

\begin{lemma}

\textit{(Lemma 18.5(a) of [8])}. The subgroup $\pazocal{F}(A_i)$ is uniquely defined and is a 2-group.

\end{lemma}

A word $J$ is called an \textit{$\pazocal{F}(A_i)$-involution} provided $J$ normalizes $\pazocal{F}(A_i)$ in $\textbf{B}(i-1)$, $J^2\in\pazocal{F}(A_i)$ in rank $i-1$, and $J^{-1}A_iJ\stackrel{i-1}= A_i^{-1}F$ with a word $F\in\pazocal{F}(A_i)$.

Let $\Delta$ be a diagram over $\textbf{B}(i)$ and $q_1,q_2$ be sections of $\partial\Delta$ or the contours of a cell in $\Delta$. Suppose $\lab(q_k)$ is $A_j^{\eps_k}$-periodic for $k\leq i$ and $\eps_k\in\{\pm1\}$. Then $q_1$ and $q_2$ are called \textit{$j$-compatible} in $\Delta$ if either:

\begin{enumerate}[label=(A{\arabic*})]

\item If $\eps_1\eps_2=-1$, then there are phase vertices $O_k\in q_k$ and a simple path $t$ between $O_1$ and $O_2$ such that $|t|<(1+\delta)|A_j|$ and $\lab(t)\stackrel{j-1}=T$ for some $T\in\pazocal{F}(A_j)$.

\item If $\eps_1\eps_2=1$, then there are phase vertices $O_k\in q_k$ and a simple path $t$ between $O_1$ and $O_2$ such that $|t|<(1+\delta)|A_j|$ and $\lab(t)\stackrel{j-1}=T$ for some $\pazocal{F}(A_j)$-involution $T$.

\end{enumerate}

A disk diagram $\Delta$ over $\textbf{B}(i)$ is called \textit{strictly reduced} if $\partial\Pi_1$ and $\partial\Pi_2$ are not $j$-compatible for cells $\Pi_1$ and $\Pi_2$ of rank $j$ (perhaps with $\Pi_1=\Pi_2$). (Note that in [8], such a diagram is simply called reduced).

This terminology is justified by Theorem C of [8], which essentially assures that cells whose contours are $j$-compatible can be removed from $\Delta$ and replaced with cells of rank $\leq j-1$, reducing the type of $\Delta$ (this is done in much the same way as it is done for $j$-pairs). As such, we may again strengthen van Kampen's Lemma, so that a word $W$ over $\pazocal{A}$ is trivial over $B(\pazocal{A},n)$ if and only if there exists $i\geq0$ and a strictly reduced diagram $\Delta$ over $\textbf{B}(i)$ with $\lab(\partial\Delta)\equiv W$.

\smallskip

%%%%%%%%%%%%%%%%%%%%%%%%%%%%%%%%%%%%%%%%%%%%%%%%%%%%%%%%%%%%%%%%%

\subsection{Tame diagrams} \

A word $X$ is \textit{cyclically reduced in rank $i$} if for any word $Y$ such that $X\stackrel{i}=ZYZ^{-1}$ (i.e $X$ and $Y$ are conjugate in rank $i$), then $|X|\leq|Y|$. The word $A$ is called \textit{simple in rank $i$} if $A$ is not conjugate in rank $i$ to $A_k^\ell F$ for any $k\leq i$, any integer $\ell$, and any $F\in\pazocal{F}(A_k)$.

For $\Delta$ a diagram over $\textbf{B}(i)$, an $A$-periodic section $q$ of $\partial\Delta$ is called \textit{smooth} if either:

\begin{enumerate}[label=(S{\arabic*})]

\item $A\equiv A_j^{\pm1}$ and there is no cell $\Pi$ in $\Delta$ such that $r(\Pi)=j$ and $\partial\Pi$ is $j$-compatible with $q$

\item $A$ is simple in rank $i$.

\end{enumerate}

If $q$ satisfies (S1), then the rank of $q$ is defined to be $r(q)=j$. Otherwise, we write $r(q)=\infty$.

Note that if $\Delta$ is a strictly reduced diagram containing an $\pazocal{R}$-cell $\Pi$, $\Gamma$ is a subdiagram of $\Delta$ such that a section $q$ of $\partial\Gamma$ is a subpath of $\partial\Pi$ in $\Delta$, and $\Pi$ is not contained in $\Gamma$, then $\Gamma$ is strictly reduced and $q$ is a smooth section of $\partial\Gamma$ with $r(q)=r(\Pi)$.

Finally, a strictly reduced diagram $\Delta$ over $\textbf{B}(i)$ is called \textit{tame} if it satisfies both:

\begin{enumerate}[label=(T{\arabic*})]

\item Let $p=\partial\Pi_1$ for some cell $\Pi_1$ in $\Delta$ and $q$ be a smooth section of $\partial\Delta$ or $q=\partial\Pi_2$ for some cell $\Pi_2$ in $\Delta$. If $\Gamma$ is a contiguity subdiagram between $p$ and $q$, then $r(\Gamma)<\min(r(p),r(q))$.

\item For any $\pazocal{R}$-cell $\Pi$ in $\Delta$, there is no $0$-bond in $\Delta$ from $\partial\Pi$ to itself.

\end{enumerate}

%The following statement is clear from the definitions above.
%
%\begin{lemma} \label{Lemma 15.1} \
%
%\begin{enumerate}[label=({\arabic*})]
%
%\item A submap of an A-map is an A-map.
%
%\item If a subpath $p$ of a smooth section $q$ of rank $k$ in an A-map $\Delta$ is a subpath of the contour of a submap $\Gamma$, then $p$ can be regarded as a smooth section of rank $k$ in $\partial\Gamma$.
%
%\item Let $\Delta$ be an A-map containing a cell $\Pi$ of rank $k$ and $q$ be a subpath of $\partial\Pi$. Suppose $q$ is a section of the contour of a submap $\Gamma$ not containing $\Pi$. Then $q$ is a smooth section of rank $k$ in $\partial\Gamma$.
%
%\end{enumerate}
%
%\end{lemma}

The following statements are proved in [8] and listed here for reference.

%\begin{lemma} \label{Theorem 16.1}
%
%\textit{(Theorem 16.1 of [13])} Let $\Delta$ be an A$^0$ map with $r(\Delta)>0$. Then there is in $\Delta$ a $\pazocal{R}$-cell $\pi$ and at most 10 disjoint contiguity submaps $\Gamma_1,\dots,\Gamma_k$ of $\pi$ to sections of the contour of $\Delta$ such that the sum of the contiguity degrees of $\pi$ across $\Gamma_1,\dots,\Gamma_k$ is greater than $\bar{\gamma}\equiv 1-\gamma$.
%
%\end{lemma}

\begin{lemma} \label{tame}

\textit{(Lemma 9.2 of [8])}. Every strictly reduced diagram $\Delta$ over $\textbf{B}(i)$ is tame.

\end{lemma} 

\begin{lemma} \label{Lemma 5.7}

\textit{(Lemma 5.7 of [8])} Let $\Delta$ be a tame disk diagram over $\textbf{B}(i)$ whose contour is decomposed into the subsections $q_1,\dots,q_m$ with $m\leq 8$. Then, in $\Delta$, there exists an $\pazocal{R}$-cell $\pi$ and disjoint contiguity submaps $\{\Gamma_j\}_{j=1}^k$ of $\pi$ to these sections such that $$\sum_{j=1}^k|{\Gamma_j} _{ ^\wedge}\pi|>(1-\gamma)|\partial\pi|$$

\end{lemma}

The cell $\pi$ guaranteed by Lemma \ref{Lemma 5.7} is called a \textit{$\gamma$-cell}.

\begin{lemma} \label{Lemma 6.1}

\textit{(Lemma 6.1 of [8])} Let $\Delta$ be a tame disk diagram over $\textbf{B}(i)$ with contour $qt$. If $q$ is a smooth section, then $(1-\b)|q|\leq|t|$.

\end{lemma}

\begin{lemma} \label{Lemma 6.2}

\textit{(Lemma 6.2 of [8])} If $\Delta$ is a tame disk diagram over $\textbf{B}(i)$ such that $|\partial\Delta|\leq(1-\b)n|A_k|$ for some $k\leq i$, then $r(\Delta)<k$.

\end{lemma}

\begin{lemma} \label{Lemma 3.1}

\textit{(Lemma 3.1 of [8])} Let $\Delta$ be a strictly reduced disk diagram over $\textbf{B}(i)$ and $\Gamma$ be a contiguity submap of a cell $\Pi$ to a section $q$ of $\partial\Delta$. If $\Gamma$ is a tame diagram with $\partial\Gamma=d_1p_1d_2q_1$ where $p_1=\Gamma_{ ^\wedge}\partial\Pi$ and $q_1=\Gamma_{ ^\wedge}q$. Then for $j=r(\Pi)$, $$\max(|d_1|,|d_2|)<2\eps^{-1}|A_j|\leq\zeta n|A_j|$$

\end{lemma}

%\begin{lemma} \label{Lemma 15.4}
%
%\textit{(Lemma 15.4 of [13])} If the degree of $\Gamma$-contiguity of a cell $\Pi$ to a section $q$ of a contour in an A-map $\Delta$ is equal to $\psi$ and $p_1q_1p_2q_2=\partial(\Pi,\Gamma,q)$, then $|q_2|>(\psi-2\b)|\partial\Pi|$.
%
%\end{lemma}

\smallskip

\smallskip

\subsection{Mass of a diagram} \

We now introduce a weighting on diagrams over the presentation $\textbf{B}(\infty)=\gen{\pazocal{A}\mid\pazocal{R}}$ of $B(\pazocal{A},n)$, generalizing the concept of the area of such a diagram.

For $\Pi$ an $\pazocal{R}$-cell in a reduced graded disk diagram $\Delta$ over the presentation $\textbf{B}(\infty)$, let $A^n$ be the relator associated to $\lab(\partial\Pi)$. Then, define the \textit{mass} of $\Pi$ as $\rho(\Pi)=|A|^2$. This definition is extended naturally to the mass of the entire diagram, taking $\rho(\Delta)$ to be the sum of the masses of its $\pazocal{R}$-cells.

\begin{lemma} \label{a-cells are quadratic B(m,n)}

If $\Delta$ is a strictly reduced graded disk diagram over the presentation $\textbf{B}(\infty)$ of $B(m,n)$, then $\rho(\Delta)\leq|\partial\Delta|^2$.

\end{lemma}

\begin{proof}

The proof inducts on $|\partial\Delta|$, with the base case $|\partial\Delta|\leq(1-\b)n$. In this case, Lemma \ref{Lemma 6.2} implies that $r(\Delta)=0$, i.e $\Delta$ contains no $\pazocal{R}$-cells. But then $\rho(\Delta)=0$. Hence, we may assume that $|\partial\Delta|>(1-\b)n$ and $\Delta$ is a `minimal counterexample' to the lemma.

Partition $\partial\Delta$ into 8 sections, $q_1\dots q_8$, any two of which differ in length by at most 1. By Lemma \ref{tame}, $\Delta$ is a tame diagram. Applying Lemma \ref{Lemma 5.7}, there exists a $\gamma$-cell $\pi$ in $\Delta$ together with contiguity submaps $\Gamma_j$ of $\pi$ to $q_{\ell(j)}$ for $j=1,\dots,k$.

As $|\partial\Delta|>(1-\b)n$ and $\b<\frac{1}{2}$, $(\frac{1}{8}-\frac{2}{n})|\partial\Delta|<|q_j|<(\frac{1}{8}+\frac{2}{n})|\partial\Delta|$ for all $1\leq j\leq 8$.

We now proceed in two cases.

\ \textbf{1.} Suppose there exists $1\leq m\leq 8$ such that $\ell(j)\neq m$ for all $1\leq j\leq k$. Without loss of generality, say $m=1$, i.e no $\Gamma_j$ is a contiguity submap between $\pi$ and $q_1$.

For any $1\leq j\leq k$, write $\partial(\pi,\Gamma_j,q_{\ell(j)})=d_1^jp_1^jd_2^jq_1^j$. By Lemma \ref{Lemma 3.1}, $\max(|d_1^j|,|d_2^j|)<\zeta n|A_r|=\zeta|\partial\pi|$ for $r=r(\pi)$.

Now, let $\Gamma$ be the smallest subdiagram of $\Delta$ containing each $\Gamma_j$. Then, there exists a decomposition of the contour $\partial\Gamma=s_1t_1s_2t_2$ where $t_1$ is a subpath of $\partial\pi$, $t_2$ is a subpath of $\partial\Delta$, and each $s_\ell$ is a side arc of some $\Gamma_j$. Note that each $p_1^j$ is a subpath of $t_1$, so that $|t_1|\geq\sum|p_1^j|>(1-\gamma)|\partial\pi|$.

As $\Gamma$ is a tame disk diagram and $t_1$ is a smooth section of $\partial\Gamma$, Lemma \ref{Lemma 6.1} implies that 
$$(1-\b)|t_1|\leq|t_2|+|s_1|+|s_2|\leq|t_2|+2\zeta|\partial\pi|$$
It follows that $|t_2|\geq((1-\b)(1-\gamma)-2\zeta)|\partial\pi|$.

%Then, let $p=q_2\dots q_{10}$. Define $i,j\in\{2,\dots,10\}$ such that $i$ is the minimal index for which $\Gamma_i$ exists and $j$ the maximal such index. Using the defining bonds of $\Gamma_i$ and $\Gamma_j$, one can then form a contiguity submap $\Gamma$ of $\pi$ to $p$ containing all present contiguity submaps $\Gamma_2,\dots,\Gamma_{10}$ and such that $(\pi,\Gamma,p)>1-\gamma$. Letting $s_1t_1s_2t_2=\partial(\pi,\Gamma,p)$, Lemma \ref{Lemma 3.1} and \ref{Lemma 15.4} then imply that $|t_2|>(1-\gamma-2\b)|\partial\pi|$ and $|s_1|+|s_2|<2\zeta|\partial\pi|$.

\begin{figure}[H]
\centering
\includegraphics[scale=1.25]{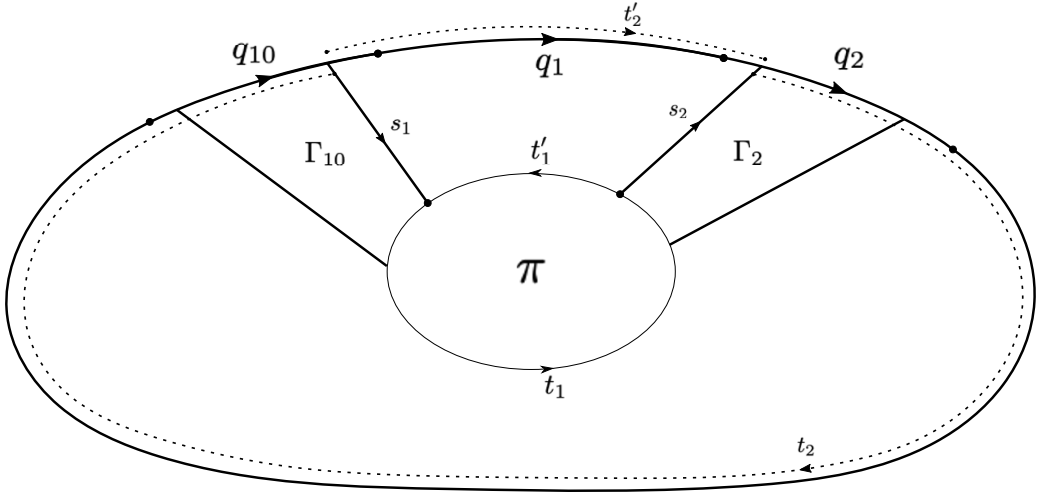}
\caption{$\Gamma_1$ is absent, $\Gamma_2$ and $\Gamma_{10}$ are present}
\end{figure}

Let $t_1'$ be the complement of $t_1$ in $\partial\pi$ so that $\partial\pi=t_1t_1'$. Further, let $t_2'$ be the complement of $t_2$ in $\partial\Delta$ so that $\partial\Delta=t_2t_2'$. Letting $u=s_2^{-1}t_1's_1^{-1}$, then $|t_1'|<\gamma|\partial\pi|$ and $|u|<(\gamma+2\zeta)|\partial\pi|$. 

Cutting $\Delta$ along $u$ yields two tame subdiagrams $\Delta_1$ and $\Delta_2$ with contours $u^{-1}t_2$ and $ut_2'$, respectively.

Then, by the parameter assignments, we have 
$$|u|<(\gamma+2\zeta)|\partial\pi|<2\gamma((1-\b)(1-\gamma)-2\zeta)^{-1}|t_2|<3\gamma|t_2|<3\gamma|\partial\Delta|$$ and $|t_2|<7(\frac{1}{8}+\frac{2}{n})|\partial\Delta|<\frac{9}{10}|\partial\Delta|$.

Hence, by the parameter assignment for $\gamma$,
$$|\partial\Delta_1|=|t_2|+|u|<(1+3\gamma)|t_2|<\frac{9}{10}(1+3\gamma)|\partial\Delta|<|\partial\Delta|$$
$$|\partial\Delta_2|=|u|+|t_2'|=|u|+|\partial\Delta|-|t_2|<|\partial\Delta|-(1-3\gamma)|t_2|<|\partial\Delta|$$

Applying the inductive hypothesis on both diagrams then yields 
$$\rho(\Delta_1)<(1+3\gamma)^2|t_2|^2$$
$$\rho(\Delta_2)<(|\partial\Delta|-(1-3\gamma)|t_2|)^2$$

As $\gamma$ is sufficiently small, note that $\frac{20}{9}(1-3\gamma)\geq(1+3\gamma)^2+(1-3\gamma)^2$. So,
$$|t_2|((1+3\gamma)^2+(1-3\gamma)^2)\leq\frac{20}{9}|t_2|(1-3\gamma)\leq2|\partial\Delta|(1-3\gamma)$$

Hence, $|t_2|^2(1+3\gamma)^2+|t_2|^2(1-3\gamma)^2-2|\partial\Delta||t_2|(1-3\gamma)\leq0$, and so
$$(|\partial\Delta|-(1-3\gamma)|t_2|)^2+(1+3\gamma)^2|t_2|^2\leq|\partial\Delta|^2$$

This final inequality yields $$\rho(\Delta)=\rho(\Delta_1)+\rho(\Delta_2)<|\partial\Delta|^2$$

\ \textbf{2.} Suppose that for every $1\leq m\leq 8$, there is a $j$ such that $\Gamma_j$ is a contiguity submap between $\pi$ and $q_m$.

For any $m\in\{1,\dots,8\}$, let $\Lambda_1^m,\dots,\Lambda_k^m$ be the collection of $\Gamma_j$ that are contiguity submaps between $\pi$ and $q_m$. Then, let $\Lambda_m$ be the smallest subdiagram of $\Delta$ containing each $\Lambda_j^m$.

\begin{figure}[H]
\centering
\includegraphics[scale=1.25]{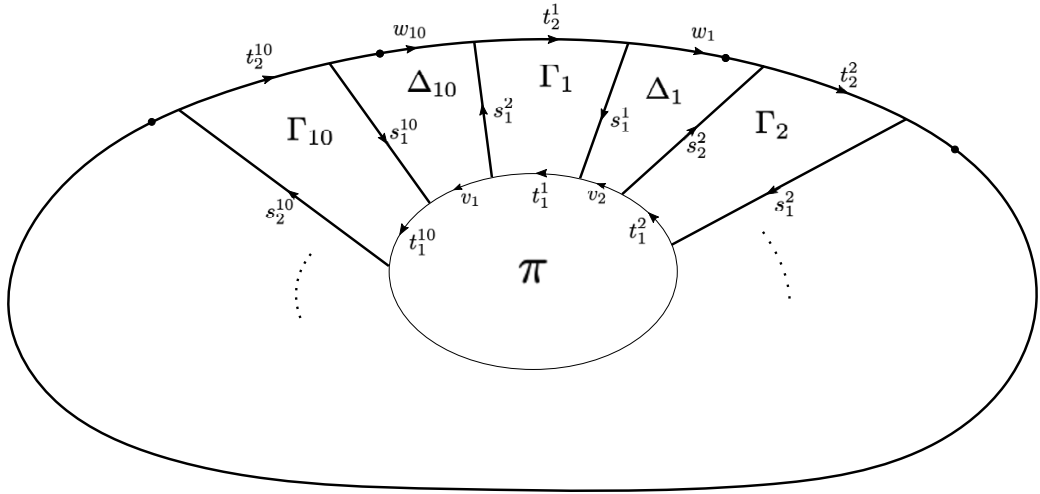}
\caption{All $\Gamma_i$ are are present}
\end{figure}

It follows that we may set $\partial\Lambda_m=s_1^m t_1^m s_2^m t_2^m$ for all $m$, $\partial\pi=t_1^8v_8t_1^7v_7\dots t_1^1v_1$, and $\partial\Delta=t_2^1w_1t_2^2w_2\dots t_2^8w_8$. Further, for $m=1,\dots,8$, let $\Delta_m$ be the subdiagram with contour $$w_m(s_2^{m+1})^{-1}v_{m+1}(s_1^m)^{-1}$$ (with indices $m$ counted mod 8).

As in the previous case, Lemma \ref{Lemma 3.1} implies that $|s_j^m|<\zeta|\partial\pi|$. Also, as $\b<\frac{1}{2}$, Lemma \ref{Lemma 6.2} implies $|\partial\pi|<2|\partial\Delta|$.

As $\Lambda_m$ is a tame subdiagram and $t_1^m$ is a smooth section of $\partial\Lambda_m$, Lemma \ref{Lemma 6.1} implies $$(1-\b)|t_1^m|<|s_1^m|+|s_2^m|+|t_2^m|$$ for all $m$. Further, since we also have $|t_2^m|\leq|q_m|<(\frac{1}{8}+\frac{2}{n})|\partial\Delta|$, it follows from the parameter choices that $$|s_1^m|+|s_2^m|+|t_2^m|<(\frac{1}{8}+\frac{2}{n}+4\zeta)|\partial\Delta|<\frac{1}{7}|\partial\Delta|$$ for all $m$. Hence, for all $m$, 
$$|\partial\Lambda_m|=|t_1^m|+|s_1^m|+|s_2^m|+|t_2^m|<\frac{1}{7}\left(1+\frac{1}{1-\b}\right)|\partial\Delta|<\frac{1}{6}|\partial\Delta|$$
So, applying the inductive hypothesis, $\rho(\Lambda_m)<\frac{1}{36}|\partial\Delta|^2$ for all $m$.

Further, $|w_m|<|q_m|+|q_{m+1}|<(\frac{1}{4}+\frac{4}{n})|\partial\Delta|$ and $|v_m|<\gamma|\partial\pi|<2\gamma|\partial\Delta|$, so that the parameter assignments yield
$$|\partial\Delta_m|<\bigg(\frac{1}{4}+\frac{4}{n}+4\zeta+2\gamma\bigg)|\partial\Delta|<\frac{2}{7}|\partial\Delta|$$ for all $m$. So, applying the inductive hypothesis yields $\rho(\Delta_m)<\frac{4}{49}|\partial\Delta|^2$ for all $m$.

Finally, note that since $\pi$ is an $\pazocal{R}$-cell, $\rho(\pi)=(\iota|\partial\pi|)^2<4\iota^2|\partial\Delta|^2<\frac{1}{9}|\partial\Delta|^2$. Thus,
$$\rho(\Delta)=\sum_{m=1}^{8}\rho(\Gamma_m)+\sum_{m=1}^{8}\rho(\Delta_m)+\rho(\pi)<\frac{2}{9}|\partial\Delta|^2+\frac{32}{49}|\partial\Delta|^2+\frac{1}{9}|\partial\Delta|^2<|\partial\Delta|^2$$

\end{proof} 

\medskip

%%%%%%%%%%%%%%%%%%%%%%%%%%%%%%%%%%%%%%%%%%%%%%%%%%%%%%%%%%%%%%%%%

\section{$S$-Machines}

\subsection{Definition of $S$-machine as a Rewriting System} \

There are many equivalent interpretations of $S$-machines [25]. Following the conventions of [2], [16], [18], [19], [21], [23], [24], and others, we approach them here as a rewriting system for words over group alphabets. 
%As such, the definitions in this section are identical to those found in these references.

Let $(Y,Q)$ be a pair of sets with $Q=\sqcup_{i=0}^N Q_i$ and $Y=\sqcup_{i=1}^N Y_i$ for some positive integer $N$. For convenience of notation, set $Y_0=Y_{N+1}=\emptyset$ in this setting.

The elements of $Q\cup Q^{-1}$ are called \textit{state letters} or \textit{$q$-letters}, while those of $Y\cup Y^{-1}$ are \textit{tape letters} or \textit{$a$-letters}. The sets $Q_i$ and $Y_i$ are called the \textit{parts} of $Q$ and $Y$, respectively. Note that the parts of the state letters are typically represented by capital letters, while their elements are represented by lowercase.

\newpage

The \textit{language of admissible words} for $(Y,Q)$ is the collection of reduced words of the form $q_0^{\eps_0}u_1q_1^{\eps_1}\dots u_kq_k^{\eps_k}$ where $\eps_i\in\{\pm1\}$ and each subword $q_{i-1}^{\eps_{i-1}}u_iq_i^{\eps_i}$ either:

\begin{enumerate}[label=({\arabic*})]

\item belongs to $(Q_{j-1}F(Y_j)Q_j)^{\pm1}$;

\item has the form $quq^{-1}$ for $q\in Q_j$ and $u\in F(Y_{j+1})$; or

\item has the form $q^{-1}uq$ for $q\in Q_j$ and $u\in F(Y_j)$

\end{enumerate}

For a reduced word $W\in F(Y\cup Q)$, define its \textit{$a$-length} $|W|_a$ as the number of $a$-letters that comprise it. The $q$-length of $W$ is defined similarly and is denoted $|W|_q$.

Let $W\equiv q_1u_1q_2u_2q_3\dots q_s$ be an admissible word with $q_i\in Q_{j(i)}^{\eps_i}$ for $\eps_i\in\{\pm1\}$  and $u_i\in F(Y)$. Then the \textit{base} of $W$ is $\text{base}(W)\equiv Q_{j(1)}^{\eps_1}Q_{j(2)}^{\eps_2}\dots Q_{j(s)}^{\eps_s}$, where these letters are merely representatives of their corresponding parts, and $u_i$ is called the $Q_{j(i)}^{\eps_i}Q_{j(i+1)}^{\eps_{i+1}}$-sector of $W$. Note that the base of an admissible word $W$ need not be a reduced word and that $W$ is permitted to have many sectors of the same name (for example, $W$ may contain many $Q_0Q_1$-sectors). 

The base $Q_0\dots Q_N$ is called the \textit{standard base}. An admissible word with the standard base is called a \textit{configuration}. 

Now, set $U_0,\dots,U_m$ and $V_0,\dots,V_m$ as a collection of reduced words over $Y\cup Q$ satisfying:

\begin{enumerate}[label=({\arabic*})]

\item $U_i$ and $V_i$ have base $Q_{\ell(i)}Q_{\ell(i)+1}\dots Q_{r(i)}$ with $\ell(i)\leq r(i)$ and such that both are subwords of admissible words

\item $\ell(i+1)=r(i)+1$ for all $i$

\item $U_0$ and $V_0$ start with letters from $Q_0$, while $U_m$ and $V_m$ end with letters from $Q_N$

\end{enumerate}

Define $Q(\theta)$ as the set of state letters appearing in some $U_i$. Note that $Q(\theta)$ contains exactly one state letter from each part.

Also, let $Y(\theta)=\cup Y_j(\theta)$ with $Y_j(\theta)\subseteq Y_j$ be some subset of the tape alphabet with the requirement that the set of tape letters appearing in $U_i$ or $V_i$ is a subset of $Y(\theta)^{\pm1}$. Each $Y_j(\theta)$ is called the \textit{domain} of $\theta$ in the corresponding sector of the standard base.

If $W$ is an admissible word with all its state letters contained in $Q(\theta)\cup Q(\theta)^{-1}$ and all its tape letters contained in $Y(\theta)\cup Y(\theta)^{-1}$, then define $W\cdot\theta$ as the result of simultaneously replacing every subword $U_i^{\pm1}\equiv(u_{\ell(i)}q_{\ell(i)}u_{\ell(i)+1}q_{\ell(i)+1}\dots q_{r(i)}u_{r(i)+1})^{\pm1}$ of $W$ by the subword $V_i^{\pm1}\equiv(v_{\ell(i)}q_{\ell(i)}'v_{\ell(i)+1}\dots q_{r(i)}'v_{r(i)+1})^{\pm1}$, followed by the necessary reduction to make the resulting word again admissible.

In this case, $\theta$ is called an \textit{$S$-rule} of $(Y,Q)$ and is denoted $\theta=[U_0\to V_0,\dots,U_m\to V_m]$. This notation fully describes the rule $\theta$ except for the corresponding sets $Y_j(\theta)$. Henceforth, $Y_j(\theta)$ is assumed to be either $Y_j$ or $\emptyset$ unless otherwise stated, with context making it clear which is chosen.

For any $S$-rule $\theta$, if $\theta$ is applicable to an admissible word $W$, then $W$ is called \textit{$\theta$-admissible}.

An important note to stress is that the application of an $S$-rule results in a reduced word, i.e reduction is not a separate step in the application of the $S$-rule.

If the $i$-th part of the $S$-rule $\theta$ is $U_i\to V_i$, $U_i$ and $V_i$ have base $Q_{\ell(i)}\dots Q_{r(i)}$, and $Y_{r(i)+1}(\theta)=\emptyset$, then this part of the rule is denoted $U_i\xrightarrow{\ell}V_i$ and $\theta$ is said to \textit{lock} the $Q_{r(i)}Q_{r(i)+1}$-sector.

Note that every $S$-rule $\theta$ has a natural inverse, namely $\theta^{-1}=[V_0\to U_0,\dots,V_m\to U_m]$ with $Y_j(\theta^{-1})=Y_j(\theta)$ for all $j$. 

An \textit{$S$-machine} $\textbf{S}$ with hardware $(Y,Q)$ is the defined to be the rewriting system whose software is a symmetric set of $S$-rules $\Theta(\textbf{S})=\Theta$, i.e $\theta\in\Theta$ if and only if $\theta^{-1}\in\Theta$. 

It is convenient to partition $\Theta$ into two disjoint sets, $\Theta^+$ and $\Theta^-$, such that $\theta\in\Theta^+$ if and only if $\theta^{-1}\in\Theta^-$. The elements of $\Theta^+$ are called the \textit{positive rules} and those of $\Theta^-$ the \textit{negative rules}.

For $t\geq0$, suppose $W_0,\dots,W_t$ are admissible words with the same base such that there exist $\theta_1,\dots,\theta_t\in\Theta$ satisfying $W_{i-1}\cdot\theta_i\equiv W_i$ for all $1\leq i\leq t$. Then the sequence of applications of rules $\pazocal{C}:W_0\to\dots\to W_t$ is called a \textit{computation} of \textit{length} or \textit{time} $t\geq0$ of $\textbf{S}$. The word $H=\theta_1\dots\theta_t$ is called the \textit{history} of $\pazocal{C}$ and the notation $W_t\equiv W_0\cdot H$ is used to represent the computation.

A computation is called \textit{reduced} if its history is a reduced word in $F(\Theta^+)$. Every computation can be made reduced without changing the initial and final admissible words of the computation simply by removing consecutive mutually inverse rules.

Typically, it is assumed that each part of the state letters contains two (perhaps the same) fixed elements, called the \textit{start} and \textit{end} state letters. A configuration is called a \textit{start} (or \textit{end}) configuration if all of its state letters are start (or end) letters.

A \textit{recognizing} $S$-machine is one with specified sectors called the \textit{input sectors}. If a start configuration has all sectors empty except for the input sectors, then it is called an \textit{input configuration} and its projection onto $Y\cup Y^{-1}$ is called its \textit{input}. The end configuration with every sector empty is called the \textit{accept configuration}.

A configuration $W$ is \textit{accepted} by a recognizing $S$-machine if there is an \textit{accepting computation}, i.e a computation whose initial configuration is $W$ and whose final configuration is the accept configuration. If $W$ is an accepted input configuration with input $u$, then $u$ is also said to be \textit{accepted}.

If the configuration $W$ is accepted by the $S$-machine $\textbf{S}$, then $T(W)$ is the minimal time of its accepting computations. For a recognizing $S$-machine \textbf{S}, its \textit{time function} is $$T_{\textbf{S}}(n)=\max\{T(W): W\text{ is an accepted input configuration of } \textbf{S}, \ |W|_a\leq n\}$$

\medskip

%%%%%%%%%%%%%%%%%%%%%%%%%%%%%%%%%%%%%%%%%%%%%%%%%%%%%%%%%%%%%%%%%

%\subsection{Simplifying the rules} \

If two recognizing $S$-machines have the same language of accepted words and $\Theta$-equivalent time functions, then they are said to be \textit{equivalent}.

The following simplifies how one approaches the rules of a recognizing $S$-machine.

\medskip

\begin{lemma} \label{simplify rules}

\textit{(Lemma 2.1 of [16])} Every recognizing $S$-machine $\textbf{S}$ is equivalent to a recognizing $S$-machine that satisfies:

\begin{enumerate}[label=({\arabic*})]

\item Every part of every rule has a 1-letter base (i.e if $U_i\to V_i$ is a part of a rule $\theta$, then $U_i\equiv u_iq_iu_{i+1}$ and $V_i\equiv v_iq_i'v_{i+1}$ for $q_i,q_i'$ state letters in $Q_i$)

\item In every part $u_iq_iu_{i+1}\to v_iq_i'v_{i+1}$ of every rule, $\|u_i\|+\|v_i\|\leq1$ and $\|u_{i+1}\|+\|v_{i+1}\|\leq1$.

\item Moreover, with the terminology of (2), $\|u_i\|+\|v_i\|+\|u_{i+1}\|+\|v_{i+1}\|\leq1$.

\end{enumerate}

\end{lemma}

As a result of Lemma \ref{simplify rules}, we may assume that each part of every rule of an $S$-machine is of the form $q_i\to aq_i'b$ with $\|a\|+\|b\|\leq1$. However, it will be convenient to allow $\|a\|=\|b\|=1$ in the defining rules of some of the $S$-machines we construct.

\smallskip

%%%%%%%%%%%%%%%%%%%%%%%%%%%%%%%%%%%%%%%%%%%%%%%%%%%%%%%%%%%%%%%%%

\subsection{Some elementary properties of $S$-machines} \

The following is an immediate consequence of the definition of admissible words.

\begin{lemma} \label{locked sectors}

If the rule $\theta$ locks the $Q_iQ_{i+1}$-sector, i.e it has a part $q_i\xrightarrow{\ell}aq_i'$ for some $q_i,q_i'\in Q_i$, then the base of any $\theta$-admissible word has no subword of the form $Q_iQ_i^{-1}$ or $Q_{i+1}^{-1}Q_{i+1}$.

\end{lemma}

Through the rest of our discussion of $S$-machines, we will often use copies of words over disjoint alphabets. To be precise, let $A$ and $B$ be disjoint alphabets, $W\equiv a_1^{\eps_1}\dots a_k^{\eps_k}$ with $a_i\in A$ and $\eps_i\in\{\pm1\}$, and $\varphi:\{a_1,\dots,a_k\}\to B$ be an injection. Then the \textit{copy} of $W$ over the alphabet $B$ formed by $\varphi$ is the word $W'\equiv\varphi(a_1)^{\eps_1}\dots\varphi(a_k)^{\eps_k}$. Typically, the injection defining the copy will be contextually clear.

Alternatively, a copy of an alphabet $A$ is a disjoint alphabet $A'$ which is in one-to-one correspondence with $A$. For a word over $A$, its copy over $A'$ is defined by the correspondence between the alphabets.

The following are properties of some simple computations in $S$-machines that are fundamental to the proofs presented in the next two sections. They are stated here without proof, with a reference provided for their proofs in previous literature.

\medskip

\begin{lemma} \label{multiply one letter}

\textit{(Lemma 2.7 of [16])} Let $\pazocal{C}:W_0\to\dots\to W_t$ be a reduced computation, where $W_0$ is an admissible word with the two-letter base $Q_iQ_{i+1}$. Denote the tape word of $W_j$ as $u_j$ for each $0\leq j\leq t$. Suppose that each rule of $\pazocal{C}$ multiplies the $Q_iQ_{i+1}$-sector by a letter on the left (respectively right). Suppose further that different rules multiply this sector by different letters. Then:

\begin{enumerate}[label=({\alph*})]

\item the history $H$ of $\pazocal{C}$ is a copy of the reduced form of $u_tu_0^{-1}$ read from right to left (respectively $u_0^{-1}u_t$ read left to right). In particular, if $u_0\equiv u_t$, then $\pazocal{C}$ is empty.

\item $\|H\|\leq\|u_0\|+\|u_t\|$

\item if $\|u_{j-1}\|<\|u_j\|$ for some $1\leq j\leq t-1$, then $\|u_j\|<\|u_{j+1}\|$

\item $\|u_j\|\leq\max(\|u_0\|,\|u_t\|)$

\end{enumerate}

\end{lemma}

%\begin{proof}
%
%$(a)$ Note that $u_j$ is freely equal to $\a_ju_{j-1}$ (or $u_{j-1}\a_j$), where $\a_j$ is the letter corresponding to the $j$-th letter of $H$. So, $u_t=\bar{H}'u_0$ (or $u_0H'$), where $\bar{H}'$ (respectively $H'$) is a copy of $H$ read right to left (respectively left to right). So, since $H$ is reduced, the statement follows.
%
%$(b)$ This follows immediately from $(a)$.
%
%$(c)$ Suppose the rules multiply the sector on the left. If $\|\a_ju_{j-1}\|=\|u_j\|>\|u_{j-1}\|$, then $\a_ju_{j-1}$ must be its reduced form. So, $u_{j+1}=\a_{j+1}\a_ju_{j-1}$ can only be unreduced if $\a_{j+1}$ and $\a_j$ cancel; but this would imply that the corresponding rules are mutually inverse, contradicting the assumption that $\pazocal{C}$ is reduced.
%
%The analogous argument applies if the rules multiply the sector on the right.
%
%$(d)$ Fix $0\leq r\leq t$ such that $\|u_r\|$ is minimal. Then either $\|u_r\|=\dots=\|u_t\|$ or there exists $r\leq s\leq t-1$ such that $\|u_r\|=\dots=\|u_s\|<\|u_{s+1}\|$. In the latter case, applying part $(c)$ to the subcomputation $W_s\to\dots\to W_t$ then gives $\|u_s\|<\|u_{s+1}\|<\dots<\|u_t\|$. So, $\|u_j\|\leq\|u_t\|$ for all $r\leq j\leq t$. 
%
%The analogous argument applied to the inverse computation $W_r\to\dots\to W_0$ then shows that $\|u_j\|\leq\|u_0\|$ for all $0\leq j\leq r$.
%
%\end{proof}

\medskip

\begin{lemma} \label{unreduced base}

\textit{(Lemma 3.6 of [22])} Suppose $\pazocal{C}:W_0\to\dots\to W_t$ is a reduced computation of an $S$-machine with base $Q_iQ_i^{-1}$ (respectively $Q_i^{-1}Q_i$). For $0\leq j\leq t$, let $u_j$ be the tape word of $W_j$. Suppose each rule of $\pazocal{C}$ multiplies the $Q_iQ_{i+1}$-sector (respectively the $Q_{i-1}Q_i$-sector) by a letter from the left (respectively from the right), with different rules corresponding to different letters. Then $\|u_j\|\leq\max(\|u_0\|,\|u_t\|)$ for all $j$ and the history of $\pazocal{C}$ has the form $H_1H_2^\ell H_3$, where $\ell\geq0$, $\|H_2\|\leq\min(\|u_0\|,\|u_t\|)$, $\|H_1\|\leq\|u_0\|/2$, and $\|H_3\|\leq\|u_t\|/2$.

\end{lemma}

\bigskip

%%%%%%%%%%%%%%%%%%%%%%%%%%%%%%%%%%%%%%%%%%%%%%%%%%%%%%%%%%%%%%%%%

\subsection{Parameters} \

The arguments spanning the rest of this paper are reliant on the \textit{highest parameter principle}, the obvious dual to the lowest parameter principle described in Section 2.5. In particular, we introduce the relation $<<$ on parameters defined as follows.

If $\a_1,\a_2,\dots,\a_n$ are parameters with $\a_1<<\a_2<<\dots<<\a_n$, then for all $2\leq i\leq n$, it is understood that $\a_1,\dots,\a_{i-1}$ are assigned prior to the assignment of $\a_i$ and that the assignment of $\a_i$ is dependent on the assignment of its predecessors. The resulting inequalities are then understood as `$\a_i\geq$(any expression involving $\a_1,\dots,\a_{i-1}$)'

Specifically, the assignment of parameters we use here is:
\begin{align*}
n&<<\lambda^{-1}<<c_0<<k<<c_1<<c_2<<c_3<<c_4<<c_5<<L_0<<L<<K_0 \\
&<<K<<J<<\delta^{-1}<<C_1<<C_2<<C_3<<N_1<<N_2<<N_3<<N_4<<N_5
\end{align*}

\newpage

%%%%%%%%%%%%%%%%%%%%%%%%%%%%%%%%%%%%%%%%%%%%%%%%%%%%%%%%%%%%%%%%%

\section{Auxiliary Machines}

\subsection{The machine $\textbf{M}_1$} \

Let $n$ be a positive integer and $\pazocal{A}$ be a finite set. Define the \textit{language of defining relations of $B(\pazocal{A},n)$} as the set $\pazocal{L}=\{u^n:u\in F(\pazocal{A})\}$.

For $0\leq i\leq4$, let $Q_i=\{q_i(j):j=1,\dots,2n\}$. Further, for $1\leq i\leq4$, let $Y_i=\{a_i:a\in\pazocal{A}\}$ be a copy of $\pazocal{A}$.

The recognizing $S$-machine $\textbf{M}_1$ has hardware $(\sqcup_{i=1}^4 Y_i,\sqcup_{i=0}^4 Q_i)$ and software the set of rules $\Phi$ defined below. The input sector is taken to be the $Q_0Q_1$-sector, while the letters $q_i(1)$ (respectively $q_i(2n)$) are the start (respectively end) letters.

The idea of the function of $\textbf{M}_1$ is the following. Consider an input configuration with input $u^n\in\pazocal{L}$. The machine removes one copy of $u$ and replaces it in the $Q_2Q_3$-sector. Next, this copy is moved to the $Q_1Q_2$-sector. It is then moved back to the $Q_2Q_3$-sector while another copy of $u$ is simultaneously erased from the input sector. The second and third steps are then repeated $n-2$ more times until the input sector is empty. In the final step of this iteration, though, the natural copy of $u^{-1}$ is written in the $Q_3Q_4$-sector. Finally, the copies of $u$ and $u^{-1}$ are erased from the $Q_2Q_3$- and $Q_3Q_4$-sectors, respectively.

The set of positive rules $\Phi^+$ is decomposed into $2n+1$ subsets, which are denoted $\Phi_1^+,\dots,\Phi_{2n}^+$, and $\{\sigma(i,i+1):i=1,\dots,2n-1\}$. 

For each $i$, the rules of $\Phi_i^+$ are in correspondence with $\pazocal{A}$, with the rule corresponding to $a\in\pazocal{A}$ denoted $\tau_i(a)$.

For simplicity, if a rule $\theta$ of $\textbf{M}_1$ does not lock the $Q_{i-1}Q_i$-sector, then we take $Y_i(\theta)=Y_i$.

\bigskip

$\bullet$ The rules of $\Phi_1^+$ are of the form

$\tau_1(a)=\begin{bmatrix}
	q_0(1)\to q_0(1), \  q_1(1)\xrightarrow{\ell}a_1^{-1}q_1(1), \ q_2(1)\to q_2(1)a_3, \ q_3(1)\xrightarrow{\ell} q_3(1), \ q_4(1)\to q_4(1)
\end{bmatrix}$

\textit{Comment:} The state letter $q_1(1)$ moves left, removing the copy of $a$ from the $Q_0Q_1$-sector and replacing its copy in the $Q_2Q_3$-sector.

\bigskip

$\bullet$ $\sigma(12)=\begin{bmatrix*}[l]
	q_0(1)\to q_0(2), \ q_1(1)\xrightarrow{\ell}q_1(2), \ q_2(1)\to q_2(2), \ q_3(1)\xrightarrow{\ell} q_3(2), \ q_4(1)\to q_4(2)
\end{bmatrix*}$

\textit{Comment:} The state letters are changed while the $Q_1Q_2$- and $Q_3Q_4$-sectors are locked.

\bigskip

$\bullet$ For $1\leq i\leq n-1$, the rules of $\Phi_{2i}^+$ are of the form

$\tau_{2i}(a)=\begin{bmatrix*}[l]
	&q_0(2i)\to q_0(2i), \ q_1(2i)\to q_1(2i), \ q_2(2i)\to a_2q_2(2i)a_3^{-1}, \\ 
	&q_3(2i)\xrightarrow{\ell}q_3(2i), \ q_4(2i)\to q_4(2i)
\end{bmatrix*}$

\textit{Comment:} The state letter $q_2(2i)$ moves right, removing the copy of $a$ from the $Q_2Q_3$-sector and replacing its copy in the $Q_1Q_2$-sector.

\bigskip

$\bullet$ For $1\leq i\leq n-1$,

$\sigma(2i,2i+1)=\begin{bmatrix*}[l]
	&q_0(2i)\to q_0(2i+1), \ q_1(2i)\to q_1(2i+1), \ q_2(2i)\xrightarrow{\ell} q_2(2i+1), \\
	&q_3(2i)\xrightarrow{\ell}q_3(2i+1), \ q_4(2i)\to q_4(2i+1)
\end{bmatrix*}$

\textit{Comment:} When $q_2(2i)$ reaches $q_3(2i)$, the state letters are changed.

\bigskip

$\bullet$ For $1\leq i\leq n-2$, the rules of $\Phi_{2i+1}^+$ are of the form

$\tau_{2i+1}(a)=\begin{bmatrix*}[l]
	&q_0(2i+1)\to q_0(2i+1), \ q_1(2i+1)\to a_1^{-1}q_1(2i+1), \\
	&q_2(2i+1)\to a_2^{-1}q_2(2i+1)a_3, \ q_3(2i+1)\xrightarrow{\ell} q_3(2i+1), \\ 
	&q_4(2i+1)\to q_4(2i+1)
\end{bmatrix*}$

\textit{Comment:} The state letter $q_2(2i+1)$ moves left, removing the copy of $a$ from the $Q_1Q_2$-sector and replacing its copy in the $Q_2Q_3$-sector. Simultaneously, the copy of $a$ is removed from the right of the $Q_0Q_1$-sector.

\bigskip

$\bullet$ For $1\leq i\leq n-2$,

$\sigma(2i+1,2i+2)=\begin{bmatrix*}[l]
	&q_0(2i+1)\to q_0(2i+2), \ q_1(2i+1)\xrightarrow{\ell} q_1(2i+2), \\ 
	&q_2(2i+1)\to q_2(2i+2), \ q_3(2i+1)\xrightarrow{\ell}q_3(2i+2), \\
	&q_4(2i+1)\to q_4(2i+2)
\end{bmatrix*}$

\textit{Comment:} When $q_2(2i+1)$ reaches $q_1(2i+1)$, the state letters are changed.

\bigskip

$\bullet$ The rules of $\Phi_{2n-1}^+$ are of the form

$\tau_{2n-1}(a)=\begin{bmatrix*}[l]
	&q_0(2i+1)\to q_0(2i+1), \ q_1(2i+1)\to a_1^{-1}q_1(2i+1), \\
	&q_2(2i+1)\to a_2^{-1}q_2(2i+1)a_3, \ q_3(2i+1)\to q_3(2i+1), \\ 
	&q_4(2i+1)\to a_4^{-1}q_4(2i+1)
\end{bmatrix*}$

\textit{Comment:} This rule functions similar to $\tau_{2i+1}(a)$ for $1\leq i\leq n-2$, but also inserts a copy of $a^{-1}$ in the $Q_3Q_4$-sector.

\bigskip

$\bullet$ $\sigma(2n-1,2n)=\begin{bmatrix*}[l]
	&q_0(2n-1)\xrightarrow{\ell}q_0(2n), \ q_1(2n-1)\xrightarrow{\ell} q_1(2n), \\
	&q_2(2n-1)\to q_2(2n), \ q_3(2n-1)\to q_3(2n), \\
	&q_4(2n-1)\to q_4(2n)
\end{bmatrix*}$

\textit{Comment:} When both $q_2(2n-1)$ reaches $q_1(2n-1)$ and $q_1(2n-1)$ reaches $q_0(2n-1)$, the state letters are changed.

\bigskip

$\bullet$ The rules of $\Phi_{2n}^+$ are of the form

$\tau_{2n}(a)=\begin{bmatrix*}[l]
	&q_0(2n)\xrightarrow{\ell}q_0(2n), \ q_1(2n)\xrightarrow{\ell} q_1(2n), \ q_2(2n)\to q_2(2n)a_3^{-1}, \\ 
	&q_3(2n)\to q_3(2n), \  q_4(2n)\to a_4q_4(2n)
\end{bmatrix*}$

\textit{Comment:} The letters in the $Q_2Q_3$- and $Q_3Q_4$-sectors are removed.

\bigskip

%%%%%%%%%%%%%%%%%%%%%%%%%%%%%%%%%%%%%%%%%%%%%%%%%%%%%%%%%%%%%%%%%

\subsection{Standard computations of $\textbf{M}_1$} \

The machine $\textbf{M}_1$ can be viewed as the \textit{composition} of $2n$ submachines, denoted $\textbf{M}_1(1),\dots,\textbf{M}_1(2n)$, which are concatenated by the rules $\sigma(i,i+1)^{\pm1}$. The set of positive rules of the machine $\textbf{M}_1(i)$ is $\Phi_i^+$ and each machine has a disjoint set of state letters. So, each $Q_j=Q_{j,1}\sqcup\dots\sqcup Q_{j,2n}$ where each $Q_{j,i}$ is the corresponding part of the hardware of $\textbf{M}_1(i)$ (in this machine, each such subset is a singleton).

Accordingly, the rules of the form $\sigma(i,i+1)^{\pm1}$ are called \textit{transition rules}, as their function is to force the steps to be carried out in the correct order. For clarity in later sections, these rules are henceforth referred to as \textit{$\sigma$-rules}. 

For simplicity of notation, denote the inverse of each $\sigma$-rule by switching the indices, so that $\sigma(i,i+1)^{-1}\equiv\sigma(i+1,i)$.

The history of a reduced computation of $\textbf{M}_1$ can be factored so that each factor is either a $\sigma$-rule or the history of a maximal subcomputation of $\textbf{M}_1(i)$ for some $i$. The \textit{step history} of a reduced computation is then defined so as to capture the order of the types of these factors. To do this, we denote the $\sigma$-rule $\sigma(i,j)$ by the pair $(ij)$ and a factor that is an element of $F(\Phi_i^+)$ simply by $(i)$.

For example, if $H\equiv H'H''H'''$ where $H'\in F(\Phi_2^+)$, $H''\equiv\sigma(23)$, and $H'''\in F(\Phi_3^+)$, then the step history of a computation with history $H$ is $(2)(23)(3)$. So, the step history of a reduced computation is some concatenation of the letters
$$\{(1), \ (2), \ \dots, \ (2n), \ (12), \ (23), \ \dots, \ (2n-1,2n), \ (21), \ (32), \ \dots, \ (2n,2n-1)\}$$
It is convenient to allow the omission of a letter representing a $\sigma$-rule in a step history when its existence is clear from its necessity. For example, given a reduced computation with step history (2)(23)(3), one can instead write the step history as $(2)(3)$, as the rule $\sigma(23)$ must occur for the maximal subcomputation with step history $(3)$ to be possible.

If the step history of a reduced computation is $(i-1,i)(i,i+1)$, it is also permitted for the step history to be written as $(i-1,i)(i)(i,i+1)$ even though the ‘maximal subcomputation’ with step history $(i)$ is empty.

A \textit{one-step computation} is a reduced computation of $\textbf{M}_1$ whose step history has exactly one factor corresponding to a maximal subcomputation of a submachine $\textbf{M}_1(i)$.

Certain subwords cannot appear in the step history of a reduced computation of $\textbf{M}_1$. For example, it is clear that it is impossible for the step history of a reduced computation to contain the subword $(1)(3)$. The next statement displays the impossibility of some less obvious potential subwords.

\begin{lemma} \label{M_1 step history}

Let $\pazocal{C}$ be a reduced computation with base $B$.

\renewcommand{\labelenumi}{(\alph{enumi})}

\begin{enumerate}

\item If $B$ contains a subword $B'$ of the form $(Q_2Q_3)^{\pm1}$, then the step history of $\pazocal{C}$ cannot be $(2i,2i+1)(2i+1)(2i+1,2i)$ or $(2i+1,2i)(2i)(2i,2i+1)$ for $1\leq i\leq n-1$.

\item If $B$ contains a subword $B'$ of the form $(Q_1Q_2)^{\pm1}$, then the step history of $\pazocal{C}$ cannot be $(2i-1,2i)(2i)(2i,2i-1)$ or $(2i+2,2i+1)(2i+1)(2i+1,2i+2)$ for $1\leq i\leq n-1$.

\item If $B$ contains a subword $B'$ of the form $(Q_3Q_4)^{\pm1}$, then the step history of $\pazocal{C}$ cannot be $(2n-2,2n-1)(2n-1)(2n-1,2n-2)$.

\end{enumerate}

\end{lemma}

\begin{proof}

Assuming to the contrary, let $\pazocal{C}':W_0'\to\dots\to W_t'$ be the restriction of $\pazocal{C}$ to the subword $B'$. In each case, $|W_0'|_a=|W_t'|_a=0$ and the subcomputation $W_1'\to\dots\to W_{t-1}'$ satisfies the hypotheses of Lemma \ref{multiply one letter}. But then this subcomputation must be empty, contradicting the assumption that $\pazocal{C}$ is reduced.

\end{proof}

%\begin{lemma} \label{free group}
%
%For any positive integer $\ell$ and $u,v\in F(\pazocal{A})$, $u^\ell=v^\ell$ if and only if $u=v$.
%
%\end{lemma}
%
%\begin{proof}
%
%By the Nielsen-Schreier theorem, $\gen{u,v}\leq F(\pazocal{A})$ is free. Let $B$ be a basis of $\gen{u,v}$.
%
%The universal property of free groups then yields an epimorphism $\varphi:\gen{u,v}\to\bigoplus_{i\in B}\Z_i$.
%
%As $\ell\varphi(u)=\varphi(u^\ell)=\varphi(v^\ell)=\ell\varphi(v)$, we have $\bigoplus_{i\in B}\Z_i=\varphi(\gen{u,v})=\Z\varphi(u)$ so that $|B|=1$.
%
%Hence, there exists $w\in F(\pazocal{A})$ and integers $p,q$ such that $w^p=u$ and $w^q=v$.
%
%If $w=1$, then $u=v=1$. Otherwise, $|w^m|\geq|w|+m-1$ for all positive integers $m$.
%
%Assuming $p\leq q$, $w^{p\ell}=w^{q\ell}$ so that $w^{\ell(q-p)}=1$. But then $q-p=0$, so that $u=v$.
%
%\end{proof}

For an admissible word $W$, there is a natural projection of $W$ onto $F(\pazocal{A})$ given by sending each tape letter to its natural copy and each state letter to the identity. 

Note that any application of a rule of $\Phi_1$ preserves the projection of a configuration. Similarly, for $i=2,\dots,2n-1$, any application of a rule of $\Phi_i$ preserves the projection of an admissible word with base $Q_1Q_2Q_3$.

An application of these useful facts (or those similar in nature) is referred to as a \textit{projection argument}. 

\begin{lemma} \label{M_1 controlled}

For $w\in F(\pazocal{A})$ and $i\in\{2\dots,2n-1\}$, there exists a unique reduced computation $\pazocal{C}:W_0\to\dots\to W_t$ with base $Q_1Q_2Q_3$ such that

\renewcommand{\labelenumi}{(\alph{enumi})}

\begin{enumerate}

\item the step history of $\pazocal{C}$ is $(i-1,i)(i)(i,i+1)$, and

\item the projection of $W_0$ onto $F(\pazocal{A})$ is $w$.

\end{enumerate}

Moreover, for $H_i$ the history of the maximal subcomputation of $\pazocal{C}$ with step history $(i)$, $H_i$ is a copy of $w$ read left to right (resp read right to left) if $i$ is even (resp odd) and $|W_j|_a=\|w\|$ for all $0\leq j\leq t$.

\end{lemma}

\begin{proof}

Let $\pazocal{C}$ be such a computation and suppose $i$ is even.

As $W_0$ is $\sigma(i-1,i)$-admissible, its $Q_1Q_2$-sector must be empty. So, since $w$ is reduced, $W_0$ must have the natural copy of $w$ written in its $Q_2Q_3$-sector, i.e $W_0\equiv q_1(i-1)q_2(i-1)w_3q_3(i-1)$ for $w_3$ the natural copy of $w$ in $F(Y_3)$.

Further, since $W_t$ is $\sigma(i+1,i)$-admissible, its $Q_2Q_3$-sector must be empty. 

But the restriction of the subcomputation $W_1\to\dots\to W_{t-1}$ to the $Q_2Q_3$-sector satisfies the hypotheses of Lemma \ref{multiply one letter}, so that its history must be the natural copy of $w$.

If $i$ is odd, then an analogous argument applies.

\end{proof}

\begin{lemma} \label{M_1 start to end} \

\renewcommand{\labelenumi}{(\alph{enumi})}

\begin{enumerate}

\item Let $\pazocal{C}:W_0\to\dots\to W_t$ be a reduced computation of $\textbf{M}_1$. Suppose $W_0$ is a start configuration and $W_t$ is an end configuration. Then there exists $u\in F(\pazocal{A})$ such that the projection of $W_0$ onto $F(\pazocal{A})$ is $u^n\in\pazocal{L}$.

\item For all $u\in F(\pazocal{A})$, there exists a unique reduced computation $\pazocal{D}_1(u):W_0\to\dots\to W_t$ of $\textbf{M}_1$ with step history $(12)(2)(3)\dots(2n-1)(2n-1,2n)$ and such that the projection of $W_0$ onto $F(\pazocal{A})$ is $u^n$. This computation has length $(2n-2)\|u\|+(2n-1)$ and $W_t$ has the natural copy of $u$ (respectively $u^{-1}$) written in its $Q_2Q_3$-sector (respectively $Q_3Q_4$-sector).

\end{enumerate}

\end{lemma}

\begin{proof}

(a) By a projection argument, it suffices to assume that the first letter of the step history is $(12)$. Lemma \ref{M_1 step history} then implies that the step history of $\pazocal{C}$ must have prefix $$(12)(2)(3)\dots(2n-1)(2n-1,2n)$$
Let $\pazocal{D}:W_0\to\dots\to W_s$ be the subcomputation with this step history. Further, let $u,v\in F(\pazocal{A})$ be the reduced words such that $W_0$ has the natural copy of $v$ written in its input sector and the natural copy of $u$ written in its $Q_2Q_3$-sector.

After restricting $\pazocal{D}$ to the subword $Q_1Q_2Q_3$ of the standard base, Lemma \ref{M_1 controlled} implies that the history of $\pazocal{D}$ must be
$$\sigma(12)H_2\sigma(23)H_3\dots H_{2n-1}\sigma(2n-1,2n)$$
where $H_i\in F(\Phi_i^+)$ is the natural copy of $u$ read left to right (resp right to left) for $i$ even (resp $i$ odd).

Then $W_{s-1}$ has the copy of $vu^{-(n-1)}$ written in its input sector. But $W_{s-1}$ is $\sigma(2n-1,2n)$-admissible, so that $vu^{-(n-1)}$ is freely trivial.

Hence, the projection of $W_0$ onto $F(\pazocal{A})$ is $vu=u^n\in\pazocal{L}$.

(b) Let $u\in F(\pazocal{A})$ and $\pazocal{D}:W_0\to\dots\to W_s$ be the computation described in (a).

Let $\pazocal{C}':W_0'\to\dots\to W_t'$ be a computation with step history $(12)(2)(3)\dots(2n-1)(2n-1,2n)$ such that the projection of $W_0'$ onto $F(\pazocal{A})$ is $u^n$.

Set $x,w\in F(\pazocal{A})$ as the reduced words such that $W_0'$ has the natural copy of $x$ written in its input sector and the natural copy of $w$ written in its $Q_2Q_3$-sector.

As in (a), applications of Lemma \ref{M_1 controlled} then imply that the history of $\pazocal{C}'$ must be 
$$\sigma(12)H_2'\sigma(23)H_3'\dots H_{2n-1}'\sigma(2n-1,2n)$$
where $H_i'\in F(\Phi_i^+)$ is the natural copy of $w$ read left to right (resp right to left) for $i$ even (resp $i$ odd).

Then, $W_{t-1}'$ has the natural copy of $xw^{-(n-1)}$ written in its input sector and is $\sigma(2n-1,2n)$-admissible, so that $x=w^{n-1}$. But the projection of $W_0'$ onto $F(\pazocal{A})$ is $xw=w^n$, so that $w\equiv u$.

Hence, $H_i'\equiv H_i$ for all $i$ and $W_0\equiv W_0'$, so that $\pazocal{C}'=\pazocal{D}$.

As $H_i$ is a copy of $u$ for each $i$, the length of $\pazocal{D}$ is $(2n-2)\|u\|+(2n-1)$.

\end{proof}

\begin{lemma} \label{M_1 no turn}

Let $\pazocal{C}:W_0\to\dots\to W_t$ be a reduced computation of $\textbf{M}_1$ such that $W_0$ is an end configuration. If the history $H$ of $\pazocal{C}$ contains a $\sigma$-rule, then $W_t$ is not an end configuration.

\end{lemma}

\begin{proof}

Assuming to the contrary, Lemma \ref{M_1 step history} implies that $H$ can be factored as $H'H''$ such that
$$H'\equiv H_{2n}'\sigma(2n,2n-1)H_{2n-1}'\dots\sigma(32)H_2'\sigma(21)$$
where $H_i'\in F(\Phi_i^+)$ for each $i$.

Then, for $W'\equiv W_0\cdot H'$, Lemma \ref{M_1 start to end}(a) implies there exists $u\in F(\pazocal{A})$ such that the projection of $W'$ onto $F(\pazocal{A})$ is $u^n$.

As we assume that $W_t$ is an end configuration, $H''$ must contain a $\sigma$-rule. By Lemma \ref{M_1 step history}, $H''$ must then have a prefix of the form
$$H_1''\sigma(12)H_2''\sigma(23)\dots H_{2n-1}''\sigma(2n-1,2n)$$
where $H_i''\in F(\Phi_i^+)$ for each $i$.

By a projection argument, $W'$ and $W'\cdot H_1''$ have the same projection onto $F(\pazocal{A})$ so that Lemma \ref{M_1 start to end}(b) implies $W'\cdot H_1''\equiv W'$. But then $H_1''$ must be empty by Lemma \ref{multiply one letter}, contradicting the assumption that $\pazocal{C}$ is reduced.

\end{proof}

\begin{lemma} \label{M_1 language}

The language of accepted inputs of $\textbf{M}_1$ is $\pazocal{L}$. Moreover, for any $u^n\in\pazocal{L}$, there exists a unique accepting computation $\pazocal{C}_1(u)$.

\end{lemma}

\begin{proof}

Suppose $\pazocal{C}$ is an accepting computation of some input configuration $W$ with input $w$.

Lemmas \ref{M_1 step history} and \ref{M_1 no turn} then imply that the history $H$ of $\pazocal{C}$ is of the form
$$H_1\sigma(12)H_2\sigma(23)\dots H_{2n-1}\sigma(2n-1,2n)H_{2n}$$
where $H_i\in F(\Phi_i^+)$ for all $i$.

As $W\cdot H_1$ is $\sigma(12)H_2\sigma(23)\dots H_{2n-1}\sigma(2n-1,2n)$-admissible, Lemma \ref{M_1 start to end}(a) implies that its projection onto $F(\pazocal{A})$ is $u^n$ for some $u\in F(\pazocal{A})$. A projection argument then implies $w\equiv u^n\in\pazocal{L}$.

Conversely, for any $u\in F(\pazocal{A})$, let $H_0(u)$ be the history of $\pazocal{D}_1(u)$ (see Lemma \ref{M_1 start to end}(b)). Further, let $H_1(u)$ be the natural copy of $u$ read right to left in $F(\Phi_1^+)$ and $H_{2n}(u)$ be the natural copy of $u$ read left to right in $F(\Phi_{2n}^+)$.

Then, for $W$ the input configuration with input $u^n$ and $H(u)\equiv H_1(u)H_0(u)H_{2n}(u)$, $W$ is $H(u)$-admissible with $W\cdot H(u)$ the accept configuration. Let $\pazocal{C}_1(u)$ be the reduced computation with history $H(u)$ accepting $u^n$.

Suppose $\pazocal{C}'$ is an arbitrary accepting computation of $u^n$. Again, Lemmas \ref{M_1 step history} and \ref{M_1 no turn} imply that the history $H'$ of $\pazocal{C}'$ can be factored as
$$H'\equiv H_1'\sigma(12)H_2'\sigma(23)\dots H_{2n-1}'\sigma(2n-1,2n)H_{2n}'$$
where $H_i'\in F(\Phi_i^+)$ for all $i$.

Then the projection of $W\cdot H_1'$ onto $F(\pazocal{A})$ is $u^n$, so that Lemma \ref{M_1 start to end}(b) yields

\renewcommand{\labelenumi}{(\roman{enumi})}

\begin{enumerate}

\item $W\cdot H_1'\equiv W\cdot H_1=q_0(1)u_1^{n-1}q_1(1)q_2(1)u_3q_3(1)q_4(1)$, 

\item $\sigma(12)H_2'\sigma(23)\dots H_{2n-1}'\sigma(2n-1,2n)\equiv H_0(u)$, and 

\item $W\cdot (H_1'H_0(u))\equiv q_0(2n)q_1(2n)q_2(2n)u_3q_3(2n)u_4^{-1}q_4(2n)$.

\end{enumerate}

where $u_i$ is the natural copy of $u$ in $F(Y_i)$.

Applications of Lemma \ref{multiply one letter} to the restriction of the subcomputations with history $H_1'$ and $H_{2n}'$ to the $Q_2Q_3$-sector then imply that $H_1'\equiv H_1(u)$ and $H_{2n}'\equiv H_{2n}(u)$.

Thus, $H'\equiv H(u)$, and so $\pazocal{C}'=\pazocal{C}_1(u)$.

\end{proof}

As $H_1(u)$ and $H_{2n}(u)$ are copies of $u$ (read in different directions), Lemma \ref{M_1 start to end}(b) implies that the length of $\pazocal{C}_1(u)$ is $2n\|u\|+2n-1$.

\begin{lemma} \label{M_1 no start or end}

Let $\pazocal{C}:W_0\to\dots\to W_t$ be a reduced computation with base $Q_1Q_2Q_3$. Suppose the step history of $\pazocal{C}$ does not contain the letter $(1)$ or $(2n)$. Then for $m=\max(|W_0|_a,|W_t|_a)$, $t\leq2n(m+1)$.

\end{lemma}

\begin{proof}

By Lemma \ref{M_1 step history}, the step history of $\pazocal{C}$ (or its inverse) is a subword of
$$(12)(2)(3)\dots(2n-1)(2n-1,2n)$$
Suppose the history of $\pazocal{C}$ contains no $\sigma$-rule. Then the restriction of $\pazocal{C}$ to the $Q_1Q_2$-sector satisfies the hypotheses of Lemma \ref{multiply one letter}, so that $t\leq|W_0|_a+|W_t|_a\leq2m$.

So, we may factor the history $H$ of $\pazocal{C}$ as $H_1H_2H_3$, where $H_1$ and $H_3$ contain no $\sigma$-rules and $H_2$ starts and ends with a $\sigma$-rule. Note that we may have $\|H_2\|=1$ or $\|H_i\|=0$ for $i=1,3$.

Let $w$ be the projection of $W_0$ onto $F(\pazocal{A})$. Then a projection argument implies that the projection of $W_i$ onto $F(\pazocal{A})$ is $w$ for all $0\leq i\leq t$. Hence, $|W_i|_a\geq\|w\|$ for all $i$.

Lemma \ref{M_1 controlled} applies to any subcomputation whose step history is of the form $(j-1,j)(j)(j,j+1)$. So, for the subcomputation $W_r\to\dots\to W_s$ with history $H_2$, we have $s-r\leq(2n-2)\|w\|+(2n-1)$ and $|W_i|_a=\|w\|$ for all $r\leq i\leq s$. 

Since $W_s$ is $\sigma$-admissible for some $\sigma$-rule, one of its sectors must be empty. The restriction of the subcomputation $W_s\to\dots\to W_t$ to this sector then satisfies the hypotheses of Lemma \ref{multiply one letter}, so that $t-s\leq|W_t|_a$. An analogous argument implies $r\leq|W_0|_a$.

Hence, $t\leq|W_0|_a+|W_t|_a+(2n-2)\|w\|+(2n-1)\leq2n(m+1)$.

\end{proof}

\begin{lemma} \label{M_1 (i)}

Let $\pazocal{C}:W_0\to\dots\to W_t$ be a reduced computation of $\textbf{M}_1$ in the standard base. Suppose the step history of $\pazocal{C}$ is $(i)$ for some $i\in\{2,\dots,2n-1\}$ and $W_0$ is $\sigma$-admissible for some $\sigma$-rule. Then $|W_0|_a\leq3|W_t|_a$.

\end{lemma}

\begin{proof}

Let $\pazocal{C}':W_0'\to\dots\to W_t'$ be the restriction of $\pazocal{C}$ to the base $Q_1Q_2Q_3$. 

Let $w$ be the projection of $W_0'$ onto $F(\pazocal{A})$. As $W_0$ is $\sigma$-admissible, one sector of $W_0'$ must be empty. The restriction of $\pazocal{C}'$ to this sector then satisfies the hypotheses of Lemma \ref{multiply one letter}, so that $t\leq|W_t'|_a$.

Further, a projection argument implies $|W_t'|_a\geq\|w\|=|W_0'|_a$.

Now let $\pazocal{C}'':W_0''\to\dots\to W_t''$ be the restriction of $\pazocal{C}$ to the input sector. As the application of any rule inserts/deletes at most one letter from the input sector, $|W_0''|_a\leq|W_t''|_a+t$.

Similarly, for $\pazocal{C}''':W_0'''\to\dots\to W_t'''$ the restriction to the $Q_3Q_4$-sector, $|W_0'''|_a\leq|W_t'''|_a+t$.

Hence, $|W_0|_a=|W_0'|_a+|W_0''|_a+|W_0'''|_a\leq|W_t'|_a+|W_t''|_a+|W_t'''|_a+2t\leq3|W_t|_a$.

\end{proof}

\begin{lemma} \label{M_1 input length}

Let $\pazocal{C}:W_0\to\dots\to W_t$ be a reduced computation of $\textbf{M}_1$ in the standard base. Suppose $W_0$ is an input configuration and the step history of $\pazocal{C}$ does not contain the letter $(2n)$. Then $|W_0|_a\leq9n|W_t|_a$.

\end{lemma}

\begin{proof}

Let $W_0\to\dots\to W_r$ be the maximal subcomputation with step history $(1)$. A projection argument implies $|W_0|_a\leq|W_r|_a$. So, it suffices to assume that $r<t$ and prove $|W_r|_a\leq6n|W_t|_a$.

By Lemma \ref{M_1 step history}, the step history of the subcomputation $\pazocal{C}':W_r\to\dots\to W_t$ must be a subword of $(12)(2)(3)\dots(2n-1)(2n-1,2n)$. 

Let $j\leq2n-1$ be the maximal index for which $\pazocal{C}'$ has a nonempty maximal subcomputation with step history $(j)$. As $\sigma$-rules do not alter the $a$-length of an admissible word, we may then assume that the step history of $\pazocal{C}'$ is $(12)(2)\dots(j)$.

Let $H'$ be the history of the subcomputation of $\pazocal{C}'$ with step history $(12)\dots(j-1,j)$.

Let $u,v\in F(\pazocal{A})$ be the reduced words such that $W_r$ has the natural copy of $u$ written in its $Q_2Q_3$-sector and the natural copy of $v$ in its input sector.

By Lemma \ref{M_1 controlled}, $W_s\equiv W_r\cdot H'$ has the natural copy of $u$ written in either its $Q_1Q_2$-sector (if $j$ is odd) or its $Q_2Q_3$-sector (if $j$ is even) and the natural copy of (the reduced form of) $vu^{-\ell}$ written in its input sector for some $\ell\leq n$.

If $2n\|u\|\geq\|v\|$, then $|W_r|_a=\|u\|+\|v\|\leq(2n+1)\|u\|\leq3n|W_s|_a$.

Otherwise, $\|vu^{-\ell}\|\geq\|v\|-\ell\|u\|\geq\|v\|-n\|u\|\geq\frac{1}{2}\|v\|$. So, $$|W_r|_a=\|u\|+\|v\|\leq2(\|u\|+\|vu^{-\ell}\|)\leq2|W_s|_a$$
As Lemma \ref{M_1 (i)} implies $|W_s|_a\leq3|W_t|_a$, wee have $|W_r|_a\leq9n|W_t|_a$.

\end{proof}

\begin{lemma} \label{M_1 end length}

Let $\pazocal{C}:W_0\to\dots\to W_t$ be a reduced computation of $\textbf{M}_1$ in the standard base. Suppose the first letter of the step history of $\pazocal{C}$ is $(2n,2n-1)$. Then $t\leq15n^2(|W_t|_a+1)$ and $|W_0|_a\leq12n|W_t|_a$.

\end{lemma}

\begin{proof}

Let $u\in F(\pazocal{A})$ be the reduced word such that $W_0$ has the natural copy of $u$ written in its $Q_2Q_3$-sector.

Suppose the step history of $\pazocal{C}$ is $(2n,2n-1)(2n-1)$. Then for $\pazocal{C}':W_0'\to\dots\to W_t'$ the restriction of $\pazocal{C}$ to the $Q_0Q_1$- or $Q_1Q_2$-sector, Lemma \ref{multiply one letter} implies $|W_t'|_a=|W_0'|_a+t-1$ and $t\leq|W_t'|_a+1$. As each of the rules of $\Phi_{2n-1}$ can decrease the length of an admissible word with base $Q_2Q_3$ or $Q_3Q_4$ by at most one, it then follows that $|W_0|_a\leq|W_t|_a$. 

So, Lemma \ref{M_1 step history} allows us to assume that $(2n,2n-1)(2n-1)(2n-1,2n-2)$ is a prefix of the step history of $\pazocal{C}$. As a result, $W_0$ has the natural copy of $u^{-1}$ written in its $Q_3Q_4$-sector, so that $|W_0|_a=2\|u\|$.

Next, suppose the step history of $\pazocal{C}$ is a subword of $(2n,2n-1)(2n-1)(2n-2)\dots(2)(21)$ and let $\pazocal{C}':W_0'\to\dots\to W_t'$ be the restriction to the base $Q_1Q_2Q_3$. Then the projection of $W_i'$ onto $F(\pazocal{A})$ is $u$ for all $i$, so that $\|u\|\leq|W_t'|_a$. Hence, $|W_0|_a\leq2|W_t'|_a$ and, by Lemma \ref{M_1 no start or end},  $t\leq2n(|W_t'|_a+1)$.

So, by Lemma \ref{M_1 step history}, we may assume that the step history of $\pazocal{C}$ has prefix
$$(2n,2n-1)(2n-1)(2n-2)\dots(2)(21)(1)$$
Let $W_0\to\dots\to W_s$ be the maximal subcomputation with this step history and $W_r\to\dots\to W_s$ be the maximal subcomputation with step history $(1)$.

By Lemma \ref{M_1 start to end}, the projection of $W_r$ onto $F(\pazocal{A})$ is $u^n$ and $r=(2n-2)\|u\|+(2n-1)$.

Let $v\in F(\pazocal{A})$ be the reduced word such that $W_s$ has the natural copy of $v$ written in its $Q_2Q_3$-sector. The restriction of $W_r\to\dots\to W_s$ to the $Q_2Q_3$-sector satisfies the hypotheses of Lemma \ref{multiply one letter}, so that $s-r\leq\|u\|+\|v\|$.

By Lemmas \ref{M_1 step history} and \ref{M_1 no turn}, we may apply Lemma \ref{M_1 no start or end} to the restriction of the subcomputation $W_s\to\dots\to W_t$ to the subword $Q_1Q_2Q_3$. So, $t-s\leq2n(|W_t|_a+1)$ and $\|v\|\leq|W_t|_a$.

Hence, $t\leq2n\|u\|+3n|W_t|_a+4n$.

Suppose $\|u\|\leq2n\|v\|$. Then $|W_0|_a\leq4n|W_t|_a$ and $t\leq(4n^2+3n)|W_t|_a+4n\leq7n^2(|W_t|_a+1)$.

Otherwise, let $W_x\to\dots\to W_t$ be the maximal suffix whose history contains no $\sigma$-rule. Lemma \ref{M_1 controlled} implies that for some $0\leq\ell\leq n$, $W_x$ has the natural copy of (the reduced form of) $u^nv^{-\ell}$ written in its input sector. Note that $\|u^nv^{-\ell}\|\geq\|u^n\|-\ell\|v\|\geq\|u\|-n\|v\|$, so that $|W_x|_a\geq\frac{1}{2}\|u\|$.

Lemma \ref{M_1 (i)} then implies that $\frac{1}{2}|W_0|_a=\|u\|\leq6|W_t|_a$, so that $t\leq15n|W_t|_a+4n\leq15n(|W_t|_a+1)$.

\end{proof}

\begin{lemma} \label{M_1 length}

For any reduced computation $\pazocal{C}:W_0\to\dots\to W_t$ of $\textbf{M}_1$ in the standard base, $t\leq c_0\max(\|W_0\|,\|W_t\|)$.

\end{lemma}

\begin{proof}

First, suppose the step history of $\pazocal{C}$ has no occurrence of $(2n)$.

By the parameter choice $c_0>>n$ and Lemma \ref{M_1 no start or end}, we may assume that the step history has an occurrence of $(1)$. Lemma \ref{M_1 step history} then implies that there is exactly one occurrence of $(1)$. Let $\pazocal{C}_1:W_r\to\dots\to W_s$ be the maximal subcomputation with step history $(1)$.

Further, let $\pazocal{C}_1':W_r'\to\dots\to W_s'$ be the restriction of $\pazocal{C}_1$ to the $Q_2Q_3$-sector and let $u$ and $v$ be the projections of $W_r'$ and $W_s'$, respectively, onto $F(\pazocal{A})$. Lemma \ref{multiply one letter} implies that $s-r\leq\|u\|+\|v\|$.

Next, let $\pazocal{C}_0'':W_0''\to\dots\to W_r''$ and $\pazocal{C}_t'':W_s''\to\dots\to W_t''$ be the restrictions of the corresponding subcomputations to the base $Q_1Q_2Q_3$. 
%Then, the projection of $W_i''$ onto $F(\pazocal{A})$ is $u$ for all $0\leq i\leq r$ and $v$ for all $s\leq i\leq t$.

If $\pazocal{C}_0''$ is nonempty, then $W_r''$ must be $\sigma(12)$-admissible, so that $|W_r''|_a=\|u\|$. A projection argument then implies $|W_0''|_a\geq\|u\|=|W_r''|_a$, so that Lemma \ref{M_1 no start or end} yields $r\leq 2n(|W_0''|_a+1)$.

Similarly, if $\pazocal{C}_t''$ is nonempty, then $|W_t''|_a\geq\|v\|=|W_s''|_a$ and $t-s\leq2n(|W_t''|_a+1)$.

Hence, $t\leq(2n+1)(|W_0''|_a+|W_t''|_a)+4n\leq(4n+2)\max(|W_0|_a,|W_t|_a)+4n\leq6n\max(\|W_0\|,\|W_t\|)$.

Thus, we may assume that the step history of $\pazocal{C}$ has an occurrence of $(2n)$. Then, Lemma \ref{M_1 no turn} implies that there is exactly one occurrence of $(2n)$. Let $\pazocal{C}_{2n}:W_x\to\dots\to W_y$ be the maximal subcomputation with step history $(2n)$.

Similar to above, applying Lemma \ref{multiply one letter} to the restriction of $\pazocal{C}_{2n}$ to the $Q_2Q_3$-sector implies $s-r\leq|W_x|_a+|W_y|_a$. But then Lemma \ref{M_1 end length} implies that $r\leq 15n^2(|W_0|_a+1)$, $|W_r|_a\leq12n|W_0|_a$, $t-s\leq 15n^2(|W_t|_a+1)$, and $|W_s|_a\leq12n|W_t|_a$.

Hence, $t\leq 15n^2(|W_0|_a+|W_t|_a+2)+|W_r|_a+|W_s|_a\leq(30n^2+24n)\max(\|W_0\|,\|W_t\|)$.

\end{proof}

\begin{lemma} \label{M_1 width}

For any reduced computation $\pazocal{C}:W_0\to\dots\to W_t$ of $\textbf{M}_1$ in the standard base, $\|W_i\|\leq 3c_0\max(\|W_0\|,\|W_t\|)$ for all $i=0,\dots,t$.

\end{lemma}

\begin{proof}

Note that the application of any rule of $\textbf{M}_1$ changes the length of a configuration by at most four.

For $i\leq t/2$, Lemma \ref{M_1 length} implies that the subcomputation $W_0\to\dots\to W_i$ has length at most $\frac{1}{2}c_0\max(\|W_0\|,\|W_t\|)$, so that $\|W_i\|\leq\|W_0\|+2c_0\max(\|W_0\|,\|W_t\|)$. 

For $i\geq t/2$, the analogous argument applies to the subcomputation $W_i\to\dots\to W_t$, so that $\|W_i\|\leq\|W_t\|+2c_0\max(\|W_0\|,\|W_t\|)$.

Hence, for any $i$, $\|W_i\|\leq(2c_0+1)\max(\|W_0\|,\|W_t\|)$.

\end{proof}

For $1\leq i\leq 2n$ and $1\leq j\leq 4$, suppose there exists an admissible word $W$ with base $Q_{j-1}Q_j$ and $\theta\in\Phi_i$ such that the tape word of $W\cdot\theta$ differs from that of $W$. Then the application of any rule of $\textbf{M}_1(i)$ to any admissible word with base $Q_{j-1}Q_j$ inserts/deletes one tape letter. Moreover, this insertion/deletion occurs on the same side of the tape word for fixed $i$ and $j$. 

If the insertion/deletion occurs on the left (resp right) of the tape word, then the subword $Q_{j-1}Q_j$ of the standard base of $\textbf{M}_1$ is called \textit{left-active} (resp \textit{right-active}) for $\textbf{M}_1(i)$.

\begin{lemma} \label{M_1 one-step}

For $i\in\{1,\dots,2n\}$, let $\pazocal{C}:W_0\to\dots\to W_t$ be a reduced computation of $\textbf{M}_1(i)$ in the standard base. Assume that for some index $j$, $|W_j|_a>4|W_0|_a$. Then there are $\ell,r\in\{1,2,3,4\}$ such that $Q_{\ell-1}Q_\ell$ is left-active, $Q_{r-1}Q_r$ is right-active, and for the restriction $W_0'\to\dots\to W_t'$ to either sector, $|W_j'|_a<|W_{j+1}'|_a<\dots<|W_t'|_a$.

\end{lemma}

\begin{proof}

Let $\pazocal{C}':W_0'\to\dots\to W_t'$ be the restriction to any sector. Then by the definition of the rules, $\left| |W_k'|_a-|W_{k-1}'|_a \right|=1$ for all $1\leq k\leq t$. Moreover, if $|W_k'|_a-|W_{k-1}'|_a=1$, then Lemma \ref{multiply one letter} implies that $|W_m'|_a-|W_{m-1}'|_a=1$ for all $m\geq k$.

Note that for each $i$, there exists some sector of the standard base that is left-active and another that is right-active. So, assuming the statement is false, there must exist a sector for which the restriction $\pazocal{C}'$ satisfies $|W_j'|_a=|W_{j+1}'|_a+1$.

Then, we must have $|W_{k-1}'|_a=|W_k'|_a+1$ for all $k\leq j$, so that $j\leq|W_0'|_a$.

For any other sector, the restriction $\pazocal{C}'':W_0''\to\dots\to W_t''$ satisfies the hypotheses of Lemma \ref{multiply one letter}, so that $|W_j''|_a\leq|W_0''|_a+j$.

But as there are three such sectors, we have $|W_j|_a\leq|W_0|_a+3j\leq4|W_0|_a$, yielding a contradiction.

\end{proof}

\smallskip

%%%%%%%%%%%%%%%%%%%%%%%%%%%%%%%%%%%%%%%%%%%%%%%%%%%%%%%%%%%%%%%%%

\subsection{Primitive Machines} \

As in the constructions of [16] and [23], we introduce two machines, $\textbf{LR}(Y)$ and $\textbf{RL}(Y)$ for an alphabet $Y$, that will be used to alter $\textbf{M}_1$. These machines are called \textit{primitive machines}.

The standard base of $\textbf{LR}(Y)$ is $Q^{(1)}PQ^{(2)}$ with $Q^{(1)}=\{q^{(1)}\}$, $P=\{p^{(1)},p^{(2)}\}$, and $Q^{(2)}=\{q^{(2)}\}$. The letter $p^{(1)}$ is the start letter of $P$, while $p^{(2)}$ is the end letter. 

The tape alphabets are two disjoint copies of $Y$, denoted $Y^{(1)}$ and $Y^{(2)}$ and assigned in the natural way.

The positive rules of $\textbf{LR}(Y)$ come in the following three forms:

\ $\bullet$ $\zeta^{(1)}(a)=[q^{(1)}\to q^{(1)}, \ p^{(1)}\to a_1^{-1}p^{(1)}a_2, \ q^{(2)}\to q^{(2)}]$ for all $a\in Y$, where $a_i$ is its copy in $Y^{(i)}$.

\textit{Comment.} The state letter $p^{(1)}$ moves left, replacing a letter from the $Q^{(1)}P$-sector with its copy in the $PQ^{(2)}$-sector.

\medskip

\ $\bullet$ $\zeta^{(12)}=[q^{(1)}\xrightarrow{\ell} q^{(1)}, \ p^{(1)}\to p^{(2)}, \ q^{(2)}\to q^{(2)}]$

\textit{Comment.} When $p^{(1)}$ meets $q^{(1)}$, it switches to $p^{(2)}$. This is called the \textit{connecting rule} of the machine.

\medskip

\ $\bullet$ $\zeta^{(2)}(a)=[q^{(1)}\to q^{(1)}, \ p^{(2)}\to a_1p^{(2)}a_2^{-1}, \ q^{(2)}\to q^{(2)}]$ for all $a\in Y$, where $a_i$ is its copy in $Y^{(i)}$.

\textit{Comment.} The state letter $p^{(2)}$ moves right towards $q^{(2)}$ and replaces a letter in the $PQ^{(2)}$-sector with its copy in the $Q^{(1)}P$-sector.

\smallskip

The state letters of $P$ are called \textit{running state letters}. In practice, they `run' left to the adjacent state letter and then right to the other, as is indicated by the name of the machine.

%It is useful to note the following two points:
%
%\begin{addmargin}[1em]{0em}
%
%$\bullet$ In a computation of the standard base, each of the rules $(\zeta^{(j)}(a))^{\pm1}$ changes the $a$-length of a configuration by at most two. In particular, it changes the length by two or leaves it the same.
%
%$\bullet$ Consider the projection of configurations onto $F(Y)$ given by sending state letters to the identity and letters from $Y^{(j)}$ to their copies in $Y$. No rule of $\textbf{LR}(Y)$ changes the value of the configuration under this projection. An application of this useful fact is referred to as a \textit{projection argument}.
%
%\end{addmargin}

\begin{lemma} \label{primitive computations}

\textit{(Lemma 3.1 of [16])} Let $\pazocal{C}:W_0\to\cdots\to W_t$ be a reduced computation of $\textbf{LR}(Y)$ in the standard base. Then:

\begin{enumerate}[label=({\arabic*})]

\item if $|W_{i-1}|_a<|W_i|_a$ for some $1\leq i\leq t-1$, then $|W_i|_a<|W_{i+1}|_a$

\item $|W_i|_a\leq\max(|W_0|_a,|W_t|_a)$ for each $i$

\item if $W_0\equiv q^{(1)}up^{(1)}q^{(2)}$ and $W_t\equiv q^{(1)}vp^{(2)}q^{(2)}$ for some $u,v\in F(Y^{(1)})$, then $u\equiv v$, $|W_i|_a=\|u\|\defeq\ell$ for each $i$, $t=2\ell+1$, and the $Q^{(1)}P$-sector is locked in the rule $W_\ell\to W_{\ell+1}$. Moreover, letting $\bar{u}$ be $u$ read right to left, the history $H$ of $\pazocal{C}$ is a copy of $\bar{u}\zeta^{(12)}u$

\item if $W_0\equiv q^{(1)}up^{(j)}q^{(2)}$ and $W_t\equiv q^{(1)}vp^{(j)}q^{(2)}$ for some $u,v$ and $j\in\{1,2\}$, then $u\equiv v$ and the computation is empty (i.e $t=0$)

\item if $W_0$ is of the form $q^{(1)}up^{(1)}q^{(2)}$, $q^{(1)}p^{(1)}uq^{(2)}$, $q^{(1)}up^{(2)}q^{(2)}$, or $q^{(1)}p^{(2)}uq^{(2)}$ for some word $u$, then $|W_i|_a\geq|W_0|_a$ for every $i$.

\end{enumerate}

\end{lemma}

%\begin{proof}
%
%(1) In this case, the $i$-th rule of $\pazocal{C}$ must be $(\zeta^{(j)}(a))^{\pm1}$ for some $a\in Y$ and increase the $a$-length in each of the sectors. This implies that both sectors of $W_i$ are nonempty, so that the subsequent rule of $\pazocal{C}$ must be of the form $(\zeta^{(j)}(b))^{\pm1}$ for some $b\in Y$. But then the restriction of the subcomputation $W_{i-1}\to W_i\to W_{i+1}$ to one of the sectors satisfies the hypotheses of Lemma \ref{multiply one letter}, so that part $(c)$ of that lemma implies that $W_i\to W_{i+1}$ increases the $a$-length of each sector.
%
%(2) Using part (1), this follows from a similar argument as the one presented as the proof of Lemma \ref{multiply one letter}(d).
%
%(3) A projection argument immediately implies $u\equiv v$. Applying Lemma \ref{multiply one letter} to the restriction of the computation to the $Q^{(1)}P$-sector proves the rest of the statement.
%
%(4) and (5) follow from similar projection arguments.
%
%\end{proof}

\begin{lemma} \label{primitive unreduced} 

\textit{(Lemma 3.4 of [23])} Suppose $W_0\to\dots\to W_t$ is a reduced computation of $\textbf{LR}(Y)$ with base $Q^{(1)}PP^{-1}(Q^{(1)})^{-1}$ (or $(Q^{(2)})^{-1}P^{-1}PQ^{(2)}$) such that $W_0\equiv q^{(1)}p^{(i)}u(p^{(i)})^{-1}(q^{(1)})^{-1}$ (or $W_0\equiv (q^{(2)})^{-1}(p^{(i)})^{-1}vp^{(i)}q^{(2)}$) for $i=1,2$ and some word $u$ (or $v$). Then $|W_0|_a\leq\dots\leq|W_t|_a$.

\end{lemma}

\begin{lemma} \label{primitive length}

Let $\pazocal{C}:W_0\to\dots\to W_t$ be a reduced computuation of $\textbf{LR}(Y)$ in the standard base. Then $t\leq|W_0|_a+|W_t|_a+1$.

\end{lemma}

\begin{proof}

If the history $H$ of $\pazocal{C}$ contains no connecting rule, then the restriction of $\pazocal{C}$ to the $Q^{(1)}P$-sector satisfies the hypotheses of Lemma \ref{multiply one letter}. So, $t\leq|W_0|_a+|W_t|_a$.

By Lemma \ref{primitive computations}(4), we then assume that $H$ contains exactly one connecting rule. Let $W_r\to W_{r+1}$ be the transition corresponding to this connecting rule. Then both $W_r$ and $W_{r+1}$ have empty $Q^{(1)}P$-sector, so that  Lemma \ref{multiply one letter} yields $r\leq|W_0|_a$ and $t-r-1\leq|W_t|_a$.

\end{proof}

The machine $\textbf{RL}(Y)$ is the right analogue of $\textbf{LR}(Y)$. To be precise, the standard base of $\textbf{RL}(Y)$ is $Q^{(1)}RQ^{(2)}$ with $R=\{r^{(1)},r^{(2)}\}$, the tape alphabets are again two copies of $Y$ denoted $Y^{(1)}$ and $Y^{(2)}$, and the positive rules are:

\ $\bullet$ $\xi^{(1)}(a)=[q^{(1)}\to q^{(1)}, \ r^{(1)}\to a_1r^{(1)}a_2^{-1}, \ q^{(2)}\to q^{(2)}]$ for all $a\in Y$, where $a_i$ is its copy in $Y^{(i)}$

\ $\bullet$ $\xi^{(12)}=[q^{(1)}\to q^{(1)}, \ r^{(1)}\xrightarrow{\ell}r^{(2)}, \ q^{(2)}\to q^{(2)}]$

\ $\bullet$ $\xi^{(2)}(a)=[q^{(1)}\to q^{(1)}, \ r^{(2)}\to a_1^{-1}r^{(2)}a_2, \ q^{(2)}\to q^{(2)}]$ for all $a\in Y$, where $a_i$ is its copy in $Y^{(i)}$.

\smallskip

There are obvious analogues of Lemmas \ref{primitive computations}-\ref{primitive length} in the setting of $\textbf{RL}(Y)$, which can be verified in much the same ways.

When the alphabet $Y$ is contextually clear, it is convenient to omit it from the names of these machines. So, there will be reference in subsequent constructions to the machines $\textbf{LR}$ and $\textbf{RL}$.

\smallskip

%%%%%%%%%%%%%%%%%%%%%%%%%%%%%%%%%%%%%%%%%%%%%%%%%%%%%%%%%%%%%%%%%

\subsection{The machine $\textbf{M}_2$} \

The next machine in our construction, $\textbf{M}_2$, is the composition of copies of the submachines $\textbf{M}_1(i)$ with copies of the primitive machines $\textbf{LR}$ and $\textbf{RL}$.

Four new parts are added to the standard base of $\textbf{M}_1$, producing the standard base $$Q_0P_1Q_1R_1Q_2R_2Q_3P_4Q_4$$ 
However, the parts of the form $Q_i$ have more letters than the corresponding parts of the hardware of $\textbf{M}_1$. The makeup of each part is contextually clear from the definition of the positive rules below.

The tape alphabets of: 

\begin{itemize}

\item the $Q_0P_1$- and $P_1Q_1$-sectors are copies of $Y_1$, 

\item the $Q_1R_1$- and $R_1Q_2$-sectors are copies of $Y_2$, 

\item the $Q_2R_2$- and $R_2Q_3$-sectors are copies of $Y_3$, and 

\item the $Q_3P_4$- and $P_4Q_4$-sectors are copies of $Y_4$. 

\end{itemize}

The $Q_0P_1$-sector functions as the machine's input sector.

The idea of the function of $\textbf{M}_2$ is the following. The $Q_0P_1$-, $R_1Q_2$-, $R_2Q_3$-, and $Q_3P_4$-sectors are identified with the sectors of the standard base of $\textbf{M}_1$, so that a computation of $\textbf{M}_1$ may be carried out while all other sectors are locked. However, before a transition between two steps of such a computation can take place, every unlocked sector must take part in at least one copy of a standard computation of a primitive machine.

To be precise, $\textbf{M}_2$ is the concatenation of $4n-1$ submachines, which are denoted $\textbf{M}_2(2),\dots,\textbf{M}_2(4n)$. Consequently, each part of the state letters is the disjoint union of $4n-1$ sets corresponding to the hardware of these submachines. 

The submachines are concatenated in the natural order. As such, the state letters of $\textbf{M}_2(2)$ and $\textbf{M}_2(4n)$ function as the start and end letters, respectively. 

To force the correct order of this concatenation, we introduce transition rules $\theta(i,i+1)^{\pm1}$ for $2\leq i\leq 4n-1$. The rule $\theta(i,i+1)$ changes the state letters from the end letters of $\textbf{M}_2(i)$ to the start letters of $\textbf{M}_2(i+1)$. Further, a sector of the standard base is locked by $\theta(i,i+1)$ if and only if it is locked by every rule of $\textbf{M}_2(i)$ or every rule of $\textbf{M}_2(i+1)$. The domain in a sector that is unlocked is the entire corresponding tape alphabet.

For $1\leq i\leq 2n$, the submachine $\textbf{M}_2(2i)$ corresponds to the submachine $\textbf{M}_1(i)$, with each part of the state letters consisting of a singleton. For any $\theta\in\Phi_i^+$, the corresponding positive rule of $\textbf{M}_2(2i)$ locks the $P_1Q_1$-, $Q_1R_1$-, $Q_2R_2$-, and $P_4Q_4$-sectors and operates in the remaining sectors as $\theta$, identifying these sectors with the standard base of $\textbf{M}_1$ in the obvious way. As such, the positive rules of $\textbf{M}_2(2i)$ are identified with $\Phi_i^+$.

For $3\leq i\leq4n-3$ odd, the submachine $\textbf{M}_2(i)$ is the concatenation of two submachines, which are denoted $\textbf{M}_2(i^-)$ and $\textbf{M}_2(i^+)$ and concatenated in this order. To achieve this concatentation, we introduce more transition rules, denoted $\chi_i^{\pm1}$. To differentiate these transition rules for clarity, we henceforth refer to them as \textit{$\chi$-rules} rather than transition rules. The rule $\chi_i$ changes the state letters from the end letters of $\textbf{M}_2(i^-)$ to the start letters of $\textbf{M}_2(i^+)$ and has the same domain as the rule $\theta(i-1,i)$. 

The submachine $\textbf{M}_2(i^-)$ operates as the machine $\textbf{LR}=\textbf{LR}(Y_1)$. The subword $Q_0P_1Q_1$ is identified with the standard base of $\textbf{LR}$, with each rule operating on this subword as its corresponding rule. Each other part of the standard base consists of a singleton. Additionally, the $R_2Q_3$-sector (respectively $R_1Q_2$-sector) remains unlocked by every rule if $i$ is of the form $4\ell-1$ (respectively $4\ell+1$). Every remaining sector of the standard base is locked by every rule.

If $i$ is of the form $4\ell-1$ (with $\ell<n$), then the submachine $\textbf{M}_2(i^+)$ operates as the machine $\textbf{RL}=\textbf{RL}(Y_3)$. The subword $Q_2R_2Q_3$ is identified with the standard base of $\textbf{RL}$, with each rule operating on this subword as its corresponding rule. Each other part of the standard base consists of a singleton. Additionally, the $Q_0P_1$-sector is unlocked by every rule, while all remaining sectors of the standard base are locked by every rule.

If $i$ is of the form $4\ell+1$, then the submachine $\textbf{M}_2(i^+)$ operates as the machine $\textbf{RL}=\textbf{RL}(Y_2)$. The subword $Q_1R_1Q_2$ is identified with the standard base of $\textbf{RL}$, while each other part of the standard base consists of a singleton. Again, the $Q_0P_1$-sector is unlocked by every rule, while all remaining sectors of the standard base are locked by every rule.

Finally, the submachine $\textbf{M}_2(4n-1)$ is the concatenation of $k$ submachines, where $k$ is the parameter specified in Section 3.3. These submachines are denoted $\textbf{M}_2((4n-1)_1),\dots,\textbf{M}_2((4n-1)_k)$ and are concatenated in the natural way. So, the start letters of $\textbf{M}_2((4n-1)_1)$ and the end letters of $\textbf{M}_2((4n-1)_k)$ function as the start and end letters of $\textbf{M}_2(4n-1)$, respectively. 

To force the correct order of this concatenation, we introduce more transition rules, denoted $\chi(j,j+1)^{\pm1}$. The rule $\chi(j,j+1)$ changes the state letters from the end letters of $\textbf{M}_2((4n-1)_j)$ to the start letters of $\textbf{M}_2((4n-1)_{j+1})$, locking all sectors of the standard base except for the $R_2Q_3$- and $Q_3P_4$-sectors. 

As with the transition rules within other submachines, the rules $\chi(j,j+1)^{\pm1}$ are called $\chi$-rules. As a result, forthcoming references to `transition rules' of $\textbf{M}_2$ are implicitly restricted to those of the form $\theta(i,i+1)^{\pm1}$.

Each submachine $\textbf{M}_2((4n-1)_j)$ operates in parallel as a copy of $\textbf{RL}=\textbf{RL}(Y_3)$ on the subword $Q_2R_2Q_3$ and a copy of $\textbf{LR}=\textbf{LR}(Y_4^{-1})$ on the subword $Q_3P_4Q_4$. As such, for every $a\in\pazocal{A}$ and $i\in\{1,2\}$, there exists a positive rule of this submachine that simultaneously acts as $\xi^{(i)}(a)$ on the subword $Q_2R_2Q_3$ and as $\zeta^{(i)}(a^{-1})$ on the subword $Q_3P_4Q_4$. The subsets of $Q_0$, $P_1$, $Q_1$, and $R_1$ corresponding to this submachine are singletons, while the remaining sectors are locked by every rule.

Note that we may interpret $\textbf{M}_2$ as the concatenation of $6n-4+k$ submachines, which are concatenated in the following order:
\begin{align*}
&\textbf{M}_2(2), \ \textbf{M}_2(3^-), \ \textbf{M}_2(3^+), \ \dots, \ \textbf{M}_2(4n-4), \ \textbf{M}_2((4n-3)^-), \ \textbf{M}_2((4n-3)^+), \\
&\textbf{M}_2(4n-2), \ \textbf{M}_2((4n-1)_1), \ \dots, \ \textbf{M}_2((4n-1)_k), \ \textbf{M}_2(4n)
\end{align*}

\smallskip

%%%%%%%%%%%%%%%%%%%%%%%%%%%%%%%%%%%%%%%%%%%%%%%%%%%%%%%%%%%%%%%%%

\subsection{Standard computations of $\textbf{M}_2$} \

%For $i$ odd, the step history of a reduced computation of $\textbf{M}_2(i_j)$ is defined in a manner similar to how it is defined for a reduced computation of $\textbf{M}_1$. In particular, the history of such a computation is factored so that each factor is either the history of a maximal subcomputation of one of the defining submachines $\textbf{M}_2(i_j^{\pm})$ or a $\chi$-rule $\chi(i_j)^{\pm1}$. Then, the step history captures the order of the types of these factors, with a factor corresponding to $\textbf{M}_2(i_j^{\pm})$ denoted by $(i_j^{\pm})$, the rule $\chi(i_j)$ by $(i_j^-,i_j^+)$, and the rule $\chi(i_j)^{-1}$ by $(i_j^+,i_j^-)$.
%
%This is then used to define the step history of a reduced computation of $\textbf{M}_2(i)$ for $i$ odd. Specifically, the history of such a computation is factored so that each factor is either the history of a maximal subcomputation of one of the defining submachines $\textbf{M}_2(i_j)$ or a $\chi$-rule $\chi_i(j,j+1)^{\pm1}$. Then, a maximal subcomputation of $\textbf{M}_2(i_j)$ is represented by its step history while a $\chi$-rule $\chi_i(j,\ell)$ is denoted $(i_j,i_\ell)$ (where, as with $\sigma$-rules, we take $\chi_i(j,j+1)^{-1}\equiv\chi_i(j+1,j)$).

The step history of a reduced computation of $\textbf{M}_2$ is defined in much the same way as it is defined for a reduced computation of $\textbf{M}_1$. As such, we first factor the computation's history so that each factor is either the history of a maximal subcomputation of one of the defining submachines $\textbf{M}_2(i)$ or a transition rule $\theta(i,i+1)^{\pm1}$. Then, a factor corresponding to a computation of $\textbf{M}_2(i)$ is represented by $(i)$ and a factor corresponding to a transition rule $\theta(i,j)$ is represented by $(ij)$, where we take $\theta(i,i+1)^{-1}\equiv \theta(i+1,i)$.

%Note that if $\pazocal{C}$ is a reduced computation of $\textbf{M}_2$ and $\pazocal{C}'$ is a maximal subcomputation of $\textbf{M}_2(i)$ for $i$ odd, then the step history of $\pazocal{C}'$ does not appear as a subword of the step history of $\pazocal{C}$. Rather, $\pazocal{C}'$ corresponds simply to an occurrence of the letter $(i)$.

The notational conventions described in Section 4.2 are used for step histories of this machine. For example, we may write the step history of a reduced computation of $\textbf{M}_2$ as $(2)(3)$, omitting reference to the rule $\theta(23)$ as its presence is clear from its necessity. 
%Similarly, the step history of a reduced computation of $\textbf{M}_2(5)$ may be $(5_k^-)(5_{k-1}^+)$, with reference to the $\chi$-rule $\chi(5_k,5_{k-1})$ omitted.

Similarly, a one-step computation of $\textbf{M}_2$ is a reduced computation whose step history has exactly one letter corresponding to a computation of $\textbf{M}_2(i)$. 

The following statement is an immediate consequence of Lemma \ref{M_1 step history}.

\begin{lemma} \label{M_2 step history 1}

Let $\pazocal{C}$ be a reduced computation of $\textbf{M}_2$ with base $B$.

\renewcommand{\labelenumi}{(\alph{enumi})}

\begin{enumerate}

\item If $B$ contains a subword of the form $(R_2Q_3)^{\pm1}$, then the step history of $\pazocal{C}$ cannot be $(4i+1,4i+2)(4i+2)(4i+2,4i+1)$ or $(4i+1,4i)(4i)(4i,4i+1)$ for $1\leq i\leq n-1$.

\item If $B$ contains a subword of the form $(R_1Q_2)^{\pm1}$, then the step history of $\pazocal{C}$ cannot be $(4i-1,4i)(4i)(4i,4i-1)$ or $(4i+3,4i+2)(4i+2)(4i+2,4i+3)$ for $1\leq i\leq n-1$.

\item If $B$ contains a subword of the form $(Q_3P_4)^{\pm1}$, then the step history of $\pazocal{C}$ cannot be $(4n-3,4n-2)(4n-2)(4n-2,4n-3)$.

\end{enumerate}

\end{lemma}

Further, the following statement is an immediate consequence of Lemma \ref{primitive computations}(4).

%\begin{lemma} \label{M_2 step history 2}
%
%For $i$ odd, let $\pazocal{C}$ be a reduced computation of $\textbf{M}_2(i)$ with base $B$.
%
%\renewcommand{\labelenumi}{(\alph{enumi})}
%
%\begin{enumerate}
%
%\item If $B$ contains a subword of the form $(Q_0P_1Q_1)^{\pm1}$, then the step history of $\pazocal{C}$ cannot be 
%
%\renewcommand{\labelenumii}{(\roman{enumii})}
%
%\begin{enumerate}
%
%\item $(i_{j-1},i_j)(i_j^-)(i_j,i_{j-1})$ for $2\leq j\leq k$, or
%
%\item $(i_j^+,i_j^-)(i_j^-)(i_j^-,i_j^+)$ for $1\leq j\leq k$.
%
%\end{enumerate}
%
%\item If $B$ contains a subword of the form $(Q_2R_2Q_3)^{\pm1}$ and $i=4\ell-1$ for some integer $\ell$, then the step history of $\pazocal{C}$ cannot be 
%
%\begin{enumerate}
%
%\item $(i_{j+1},i_j)(i_j^+)(i_j,i_{j+1})$ for $1\leq j\leq k-1$, or
%
%\item $(i_j^-,i_j^+)(i_j^+)(i_j^+,i_j^-)$ for $1\leq j\leq k$.
%
%\end{enumerate}
%
%\item If $B$ contains a subword of the form $(Q_1R_1Q_2)^{\pm1}$ and $i=4\ell+1$ for some integer $\ell$, then the step history of $\pazocal{C}$ cannot be
%
%\begin{enumerate}
%
%\item $(i_{j+1},i_j)(i_j^+)(i_j,i_{j+1})$ for $1\leq j\leq k-1$, or
%
%\item $(i_j^-,i_j^+)(i_j^+)(i_j^+,i_j^-)$ for $1\leq j\leq k$.
%
%\end{enumerate}
%
%\end{enumerate}
%
%\end{lemma}
%
%Lemmas \ref{primitive computations}(4) and \ref{M_2 step history 2} immediately imply the following statement.

\begin{lemma} \label{M_2 step history 2}

Let $\pazocal{C}$ be a reduced computation of $\textbf{M}_2$ in the standard base. Then the step history of $\pazocal{C}$ cannot be $(2i,2i+1)(2i+1)(2i+1,2i)$ or $(2i+2,2i+1)(2i+1)(2i+1,2i+2)$ for $1\leq i\leq 2n-1$.

\end{lemma}

Suppose $\pazocal{C}$ is a one-step computation of $\textbf{M}_2$ with step $(4n-1)$. Then the history of $\pazocal{C}$ is called \textit{controlled} if it (or its inverse) is of the form
$$\chi(j-1,j)H'\chi(j,j+1)$$
for $1\leq j\leq k$, where $H'$ contains no $\chi$-rule, $\chi(0,1)$ is taken to be $\theta(4n-2,4n-1)$, and $\chi(k,k+1)$ is taken to be $\theta(4n-1,4n)$.

\begin{lemma} \label{M_2 controlled}

Let $\pazocal{C}:W_0\to\dots\to W_t$ be a reduced computation of $\textbf{M}_2$ with controlled history $H$. Then the base $B$ of the computation is reduced and all configurations are uniquely defined by $H$ and $B$.
\newline
Moreover, if $\pazocal{C}$ is a computation in the standard base, then $|W_i|_a=|W_0|_a$ for all $0\leq i\leq t$, $\|H\|=|W_0|_a+3$, and $W_0$ is accepted by $\textbf{M}_2$.

\end{lemma}

\begin{proof}

Without loss of generality, suppose $H\equiv\chi(j-1,j)H'\chi(j,j+1)$. Then, any sector of the standard base not locked by $\chi(j-1,j)$ is locked by the connecting rule of $\textbf{M}_2((4n-1)_j)$. So, Lemma \ref{locked sectors} implies that the base must be reduced. Lemma \ref{primitive computations}(3) then implies that $\pazocal{C}$ is uniquely defined by $H$ and $B$.

If $\pazocal{C}$ is a computation in the standard base, then the parallel nature of the rules implies that there exists $w\in F(\pazocal{A})$ such that $W_0$ has the natural copy of $w$ written in its $R_2Q_3$-sector and the natural copy of $w^{-1}$ written in its $Q_3P_4$-sector. By Lemma \ref{primitive computations}(3), we then have $|W_i|_a=|W_0|_a$ for all $i$ and $\|H\|=2\|w\|+3=|W_0|_a+3$.

Using Lemma \ref{primitive computations}(3), we may construct a reduced computation $W_0\to\dots\to W_h$ of $\textbf{M}_2(4n-1)$ such that $W_h$ is $\theta(4n-1,4n)$-admissible. Then, $W\equiv W_h\cdot\theta(4n-1,4n)$ is the end configuration with the natural copy of $w$ written in its $R_2Q_3$-sector and the natural copy of $w^{-1}$ written in its $Q_3P_4$-sector. Setting $H''$ as the natural copy of $w$ in $F(\Phi_{2n}^+)$, $W\cdot H''$ is the accept configuration. Hence, $W_0$ is accepted.

\end{proof}

The following statement is a similar consequence of Lemma \ref{primitive computations}(3).

\begin{lemma} \label{M_2 semi-controlled}

Let $\pazocal{C}:W_0\to\dots\to W_t$ be a reduced computation of $\textbf{M}_2$ with history $H$. Suppose the step history of $\pazocal{C}$ is $(2i,2i+1)(2i+1)(2i+1,2i+2)$ for some $1\leq i\leq 2n-2$. Then the base $B$ of the computation is reduced and all configurations are uniquely defined by $H$ and $B$.
\newline
Moreover, if $\pazocal{C}$ is a computation in the standard base, then $|W_i|_a=|W_0|_a$ for all $0\leq i\leq t$ and $\|H\|=2|W_0|_a+5$.

\end{lemma}

\smallskip

%\begin{lemma} \label{M_2 controlled length}
%
%Let $W_0\to\dots\to W_t$ be a reduced computation of $\textbf{M}_2$ in the standard base with step history of the form $(2i,2i+1)(2i+1)(2i+1,2i+2)$ for some $1\leq i\leq 2n-1$. Then $t=2k|W_0|_a+4k+1$ and $|W_j|_a=|W_0|_a$ for all $0\leq j\leq t$.
%
%\end{lemma}
%
%\begin{proof}
%
%By Lemma \ref{M_2 step history 2}, each unlocked sector of $W_0$ takes part in exactly $k$ copies of a primitive computation. The rest of the statement follows from Lemma \ref{primitive computations}(3).
%
%\end{proof}

A configuration of $\textbf{M}_2$ is called \textit{tame} if its $P_1Q_1$-, $Q_1R_1$-, $Q_2R_2$-, and $P_4Q_4$-sectors are all empty. Note that for $i=1,\dots,2n$, a tame configuration $W$ of $\textbf{M}_2(2i)$ corresponds naturally to a configuration $W'$ of $\textbf{M}_1(i)$. Moreover, if $W$ is admissible for a rule of $\textbf{M}_2(2i)$, then $W'$ is admissible for the corresponding rule of $\textbf{M}_1(i)$. Similarly, if $W$ is $\theta(2i,2i+1)$-admissible (respectively $\theta(2i,2i-1)$-admissible), then $W'$ is $\sigma(i,i+1)$-admissible (respectively $\sigma(i,i-1)$-admissible).

Suppose $\pazocal{C}:W_0\to\dots\to W_t$ is a reduced computation of $\textbf{M}_2$ in the standard base such that neither the first nor last letter of its step history is of the form $(2i+1)$. Then by Lemma \ref{M_2 step history 2}, any occurrence of $(2i+1)$ in the step history of $\pazocal{C}$ must be part of a subword that is of the form 

\renewcommand{\labelenumi}{(\alph{enumi})}

\begin{enumerate}

\item $(2i,2i+1)(2i+1)(2i+1,2i+2)$ or 

\item$(2i+2,2i+1)(2i+1)(2i+1,2i)$.

\end{enumerate}

Let $W_r\to\dots\to W_s$ be a subcomputation of $\pazocal{C}$ with such a step history. Then $W_r$ and $W_s$ are both tame configurations. Moreover, Lemma \ref{primitive computations}(3) implies that the configurations $W_r'$ and $W_s'$ of $\textbf{M}_1$ corresponding to $W_r$ and $W_s$, respectively, satisfy $W_r'\cdot \sigma(i,i+1)\equiv W_s'$ if the step history is of the form (a) or $W_r'\cdot\sigma(i+1,i)\equiv W_s'$ if the step history is of the form (b).

So, we can associate to $\pazocal{C}$ a reduced computation $\pazocal{C}'$ of $\textbf{M}_1$ by doing the following:

\begin{itemize}

\item replace each subcomputation $W_r\to\dots\to W_s$ whose step history is of the form (a) with the single transition $W_r'\to W_r'\cdot\sigma(i,i+1)$,

\item replace each subcomputation $W_r\to\dots\to W_s$ whose step history is of the form (b) with the single transition $W_r'\to W_r'\cdot\sigma(i+1,i)$,

\item if the first letter of the step history is of the form $(2i+1,2i)$ (respectively $(2i-1,2i)$), then replace the transition $W_0\to W_1$ with the transition $W_1'\cdot\sigma(i,i+1)\to W_1'$ (respectively $W_1'\cdot\sigma(i,i-1)\to W_1'$),

\item if the last letter of the step history is $(2i,2i+1)$ (respectively $(2i,2i-1)$), then replace the transition $W_{t-1}\to W_t$ with the transition $W_{t-1}'\to W_{t-1}'\cdot\sigma(i,i+1)$ (respectively $W_{t-1}'\to W_{t-1}'\cdot\sigma(i,i-1)$), and

\item replace all other transitions $W_{j-1}\to W_j$ with the corresponding transition $W_{j-1}'\to W_j'$.

\end{itemize}

In this case, the reduced computation $\pazocal{C}'$ is called the \textit{$\textbf{M}_1$ computation associated to $\pazocal{C}$}.

Note that any subcomputation of $\pazocal{C}$ that is `removed' to construct $\pazocal{C}'$ corresponds to an occurrence of a $\sigma$-rule in the history of $\pazocal{C}'$. Hence, Lemmas \ref{M_1 step history} and \ref{M_1 no turn} imply that at most $8n$ distinct subcomputations are removed.

\begin{lemma} \label{M_2 no turn}

Let $\pazocal{C}:W_0\to\dots\to W_t$ be a reduced computation such that $W_0$ is an end configuration. If the history $H$ of $\pazocal{C}$ contains a transition rule, then $W_t$ is not an end configuration.

\end{lemma}

\begin{proof}

Assuming toward contradiction, neither the first nor the last letter of the step history of $\pazocal{C}$ can be of the form $(2i+1)$. So, we may construct $\pazocal{C}'$ the $\textbf{M}_1$ computation associated to $\pazocal{C}$. 

Then $\pazocal{C}'$ is a reduced computation of $\textbf{M}_1$ starting and ending with an end configuration. However, the existence of a transition rule in $H$ necessitates the existence of a $\sigma$-rule in the history of $\pazocal{C}'$, so that $\pazocal{C}'$ contradicts Lemma \ref{M_1 no turn}.

\end{proof}

\begin{lemma} \label{M_2 input length}

Let $\pazocal{C}:W_0\to\dots\to W_t$ be a reduced computation of $\textbf{M}_2$ in the standard base. Suppose $W_0$ is an input configuration and the step history of $\pazocal{C}$ does not contain the letter $(4n)$. Then $|W_0|_a\leq9n|W_t|_a$.

\end{lemma}

\begin{proof}

Let $\pazocal{D}:W_0\to\dots\to W_s$ be the maximal subcomputation such that the last letter of the step history of $\pazocal{D}$ is not of the form $(2i+1)$. Then, we may construct $\pazocal{D}':W_0'\to\dots\to W_s'$ the $\textbf{M}_1$ computation associated to $\pazocal{D}$.

Lemma \ref{M_1 input length} then implies that $|W_0|_a=|W_0'|_a\leq9n|W_s'|_a=9n|W_s|_a$.

If the subcomputation $W_s\to\dots\to W_t$ is nonempty, then its step history is of the form $(2i+1)$. But then Lemma \ref{primitive computations} implies $|W_s|_a\leq|W_t|_a$, so that $|W_0|_a\leq9n|W_t|_a$.

\end{proof}

The next statement follows from an analogous proof, using Lemma \ref{M_1 end length} in place of \ref{M_1 input length}.

\begin{lemma} \label{M_2 end length}

Let $\pazocal{C}:W_0\to\dots\to W_t$ be a reduced computation of $\textbf{M}_2$ in the standard base. Suppose the first letter of the step history of $\pazocal{C}$ is $(4n,4n-1)$. Then $|W_0|_a\leq12n|W_t|_a$.

\end{lemma}

\smallskip

\begin{lemma} \label{M_2 start to end} \

\renewcommand{\labelenumi}{(\alph{enumi})}

\begin{enumerate}

\item Let $\pazocal{C}:W_0\to\dots\to W_t$ be a reduced computation of $\textbf{M}_2$. Suppose $W_0$ is a start configuration and $W_t$ is an end configuration. Then there exists $u\in F(\pazocal{A})$ such that the projection of $W_0$ onto $F(\pazocal{A})$ is $u^n\in\pazocal{L}$.

\item For all $u\in F(\pazocal{A})$, there exists a unique reduced computation $\pazocal{D}_2(u):W_0\to\dots\to W_t$ of $\textbf{M}_2$ in the standard base with step history $(23)(3)\dots(4n-1)(4n-1,4n)$ and such that the projection of $W_0$ onto $F(\pazocal{A})$ is $u^n$. 

\item Let $t$ be the length of $\pazocal{D}_2(u)$ and $\ell$ be the length of the subcomputation with step history $(4n-2,4n-1)(4n-1)(4n-1,4n)$. Then $\ell=2k\|u\|+2k+1$ and $t-\ell\leq c_0(\|u\|+1)$.

\end{enumerate}

\end{lemma}

\begin{proof}

(a) Note that neither the first nor the last letter of the step history of $\pazocal{C}$ can be of the form $(2i+1)$. So, we may construct $\pazocal{C}':W_0'\to\dots\to W_t'$ the $\textbf{M}_1$ computation associated to $\pazocal{C}$.

Then $W_0'$ is a start configuration and $W_t'$ is an end configuration, so that Lemma \ref{M_1 start to end} implies that there exists $u\in F(\pazocal{A})$ such that the projection of $W_0'$ onto $F(\pazocal{A})$ is $u^n$. As $W_0$ is tame, its projection must also be $u^n$.

(b) For any reduced computation $\pazocal{C}$ satisfying the statement, the $\textbf{M}_1$ computation associated to $\pazocal{C}$ must be the computation $\pazocal{D}_1(u)$ in Lemma \ref{M_1 start to end}(b).

The removed computations correspond to primitive computations at the $\sigma$-rules. As the tape words of the terminal configuration of $\pazocal{D}_1(u)$ are mutually inverse, the statement follows from Lemma \ref{primitive computations}(3).

(c) By Lemma \ref{M_1 start to end}(b), the subcomputation with step history $(4n-2,4n-1)(4n-1)(4n-1,4n)$ operates on the base $Q_2R_2Q_3$ as $k$ copies of the standard computation of $\textbf{RL}$ with tape word $u$. So, Lemma \ref{primitive computations}(3) implies $\ell=2k\|u\|+2k+1$.

Let $\pazocal{E}$ be the maximal subcomputation of $\pazocal{D}_2(u)$ with step history $(23)(3)\dots(4n-2)$. So, the length of $\pazocal{E}$ is $t-\ell$.

The $\textbf{M}_1$ computation associated to $\pazocal{E}$ is the maximal subcomputation of $\pazocal{D}_1(u)$ with step history $(12)(2)(3)\dots(2n-1)$. So, its length is $(2n-2)(\|u\|+1)$.

The subcomputations removed from $\pazocal{E}$ correspond to the subcomputations with step history of the form $(2i,2i+1)(2i+1)(2i+1,2i+2)$. Let $W_r\to\dots\to W_s$ be such a subcomputation. Then $W_r$ has the natural copy of $u$ written in its $R_2Q_3$- or $R_1Q_2$-sector, depending on the parity of $i$, and the natural copy of $u^{n-m}$ written in its $Q_0P_1$-sector for some $1\leq m\leq n-1$. Lemma \ref{M_2 semi-controlled} then implies that $s-r=2\|u\|+2\|u^{n-m}\|+5\leq2(n+1)\|u\|+5$.

As there are $2n-2$ such subcomputations of $\pazocal{E}$, $t-\ell\leq(2n-2)\left((2n+3)\|u\|+6\right)$. So, the statement follows by a parameter choice of $c_0$.

\end{proof}

\begin{lemma} \label{M_2 language} 

The language of accepted inputs of $\textbf{M}_2$ is $\pazocal{L}$. Moreover, for any $u^n\in\pazocal{L}$, there exists a unique accepting computation $\pazocal{C}_2(u)$.

\end{lemma}

\begin{proof}

By Lemma \ref{M_2 start to end}(a), any accepted input must be an element of $\pazocal{L}$.

%Let $w\in F(\pazocal{A})$ and let $\pazocal{C}$ be an accepting computation of $w$.
%
%As neither the first nor the last letter of the step history of $\pazocal{C}$ can be of the form $(i^+)$, $(i^-)$, $(i^+,(i+1)^-)$, or $((i+1)^-,i^+)$, we can construct $\pazocal{C}'$ the $\textbf{M}_1$ computation associated to $\pazocal{C}$. By construction, $\pazocal{C}'$ is an accepting computation for the input $w$, so that Lemma \ref{M_1 language} implies $w\in\pazocal{L}$.

Conversely, for any $u^n\in\pazocal{L}$, Lemma \ref{M_1 language} provides a unique accepting computation $\pazocal{C}_1(u)$. Concatenating the steps of $\pazocal{C}_1(u)$ with primitive computations corresponding to Lemma \ref{primitive computations}(3) then yields an accepting computation $\pazocal{C}_2(u)$.

%Let $\pazocal{D}$ be an arbitrary accepting computation of $u^n$. Then let $\pazocal{D}'$ be the $\textbf{M}_1$ computation associated to $\pazocal{D}$. As $\pazocal{D}'$ is an accepting computation of $u^n$ in $\textbf{M}_1$, $\pazocal{D}'=\pazocal{C}_1(u)$. Further, Lemma \ref{primitive computations}(3) implies that the computations removed from $\pazocal{D}$ to construct $\pazocal{D}'$ must correspond to the same primitive subcomputations as $\pazocal{C}_2(u)$. Thus, $\pazocal{D}=\pazocal{C}_2(u)$.

Uniqueness of this computation follows from Lemmas \ref{M_2 step history 1}, \ref{M_2 step history 2}, \ref{M_2 no turn}, \ref{M_2 start to end}(b), and \ref{multiply one letter}.

\end{proof}

\begin{lemma} \label{M_2(i) length}

For $i$ odd, any reduced computation $\pazocal{C}:W_0\to\dots\to W_t$ of $\textbf{M}_2(i)$ in the standard base satisfies $t\leq2k\max(\|W_0\|,\|W_t\|)$.

\end{lemma}

\begin{proof}

If the history contains no $\chi$-rules, then the statement follows from Lemma \ref{primitive length}. So, we may assume there exists a maximal subcomputation $W_r\to\dots\to W_s$ starting and ending with $\chi$-rules.

Suppose $i\neq4n-1$. Then Lemma \ref{primitive computations}(4) implies that there is exactly one $\chi$-rule in the history $H$ of $\pazocal{C}$, so that $s=r+1$. Lemmas \ref{primitive computations}(5) and \ref{primitive length} then imply that $t-s\leq2|W_t|_a+1$, $|W_s|_a\leq|W_t|_a$, $r\leq 2|W_0|_a+1$, and $|W_r|_a\leq|W_0|_a$. Hence, $t\leq4\max(\|W_0\|,\|W_t\|)$.

If $i=4n-1$, then there are at most $k$ copies of primitive computations in the subcomputation $W_r\to\dots\to W_s$. Lemma \ref{M_2 controlled} then implies that $s-r\leq k\|W_r\|=k\|W_s\|$. Again, Lemmas \ref{primitive computations}(5) and \ref{primitive length} imply that $t-s\leq 2\|W_t\|$, $\|W_s\|\leq\|W_t\|$, $r\leq2\|W_0\|$, and $\|W_r\|\leq\|W_0\|$. Hence, $t\leq(k+4)\max(\|W_0\|,\|W_t\|)$, so that the statement is implied by the parameter choice $k\geq4$.

\end{proof}

\begin{lemma} \label{M_2 length}

For any reduced computation $\pazocal{C}:W_0\to\dots\to W_t$ of $\textbf{M}_2$ in the standard base, $t\leq c_1\max(\|W_0\|,\|W_t\|)$.

\end{lemma}

\begin{proof}

By Lemma \ref{M_2(i) length} and the parameter choice $c_1>>k$, we may assume that $\pazocal{C}$ is not a one-step computation with step $(2i+1)$. 

Let $\pazocal{D}:W_x\to\dots\to W_y$ be the maximal subcomputation of $\pazocal{C}$ such that neither the first nor the last letter of the step history of $\pazocal{D}$ is of the form $(2i+1)$. 

Then, let $\pazocal{D}':W_x'\to\dots\to W_y'$ be the $\textbf{M}_1$ computation associated to $\pazocal{D}$ and $\ell$ be the length of $\pazocal{D}'$ (note that $\ell$ may be less than $y-x$ if subcomputations are removed).

By Lemma \ref{M_1 length}, $\ell\leq c_0\max(\|W_x'\|,\|W_y'\|)\leq c_0\max(\|W_x\|,\|W_y\|)$. Moreover, by Lemma \ref{M_1 width}, $\|W_i'\|\leq3c_0\max(\|W_x'\|,\|W_y'\|)\leq3c_0\max(\|W_x\|,\|W_y\|)$ for all $x\leq i\leq y$ for which $W_i'$ is part of the computation $\pazocal{D}'$.

The difference between $y-x$ and $\ell$ arises from removed subcomputations $W_r\to\dots\to W_s$. By Lemma \ref{M_2(i) length}, the length of such a subcomputation is at most $2k\|W_r\|$. So, this removed subcomputation has length at most $6c_0k\max(\|W_x\|,\|W_y\|)$.

As there are at most $8n$ such removed subcomputations, we have $y-x-\ell\leq k^2\max(\|W_x\|,\|W_y\|)$ as $k>>c_0$.

Hence, $y-x\leq(k^2+c_0)\max(\|W_x\|,\|W_y\|)$.

If the subcomputation $W_y\to\dots\to W_t$ is nonempty, then its step history is of the form $(2i+1)$. Lemma \ref{M_2(i) length} then implies that $t-y\leq 2k\max(\|W_y\|,\|W_t\|)$. But $W_y$ is tame, so that Lemma \ref{primitive computations} implies $|W_y|_a\leq|W_t|_a$. Hence, $t-y\leq 2k\|W_t\|$.

By the analogous arguments, $|W_x|_a\leq|W_0|_a$ and $x\leq 2k\|W_0\|$.

Thus, $t\leq(k^2+4k+c_0)\max(\|W_0\|,\|W_t\|)$, so that the statement is implied by the parameter choices $c_1>>k>>c_0$.

\end{proof}

For $2\leq i\leq4n$, let $UV$ be a two-letter subword of the standard base of $\textbf{M}_2$. Suppose that the application of any rule of $\textbf{M}_2(i)$ to an admissible word with base $UV$ either leaves the tape word fixed or inserts/deletes one letter on the left of the tape word. Then $UV$ is called \textit{left-active} in $\textbf{M}_2(i)$. A \textit{right-active} two-letter subword is defined analogously.

For example, the subword $Q_0P_1$ is right-active for the submachine $\textbf{M}_2(3)$ even though applications of the rules of $\textbf{M}_2(3^+)$ do not alter an admissible word with base $Q_0P_1$.

Note that a two-letter subword of the standard base of $\textbf{M}_2$ is left-active (resp right-active) in $\textbf{M}_2(2i)$ if and only if it is operated upon as a sector of $\textbf{M}_1(i)$ and the corresponding two-letter subword of the standard base of $\textbf{M}_1$ is left-active (resp right-active) in $\textbf{M}_1(i)$. 

Further, a two-letter subword that is neither left-active nor right-active in $\textbf{M}_2(i)$ is locked by every rule of the submachine.

\begin{lemma} \label{M_2 one-step}

Let $\pazocal{C}:W_0\to\dots\to W_t$ be a reduced computation in the standard base of $\textbf{M}_2(i)$ for some $i$. Assume that for some index $j$, $|W_j|_a>4|W_0|_a$. Then there are subwords $U_\ell V_\ell$ and $U_rV_r$ of the standard base such that $U_\ell V_\ell$ is left-active in $\textbf{M}_2(i)$, $U_rV_r$ is right-active in $\textbf{M}_2(i)$, and for $W_0'\to\dots\to W_t'$ the restriction of $\pazocal{C}$ to either sector, $|W_j'|_a<|W_{j+1}'|_a<\dots<|W_t'|_a$.

\end{lemma}

\begin{proof}

If $i$ is even, then the statement is an immediate consequence of Lemma \ref{M_1 one-step}. So, we may assume $i$ is odd. 

Further, inducting on $t$, we may assume that $|W_1|_a>|W_0|_a$. 

If $i\neq4n-1$, then Lemma \ref{primitive computations}(1) implies that $\pazocal{C}$ is a computation of either $\textbf{M}_2(i^-)$ or $\textbf{M}_2(i^+)$, i.e there exists a three-letter subword of the standard base on which $\pazocal{C}$ operates as a primitive machine. As one of the corresponding sectors is left-active while the other is right-active, the statement follows.

If $i=4n-1$, then the $R_2Q_3$- and $P_4Q_4$-sectors are left-active, the $Q_2R_2$- and $Q_3P_4$-sectors are right-active, and all other sectors are locked. As any rule changes the $a$-length of any of the four sectors above by one, at least three must have their length increased at the first transition. Two of these three must then be operated upon by $\pazocal{C}$ as a copy of a primitive machine.

Hence, Lemma \ref{primitive computations}(1) implies that there exists $1\leq m\leq k$ such that $\pazocal{C}$ is a computation of $\textbf{M}_2((4n-1)_m)$. The statement then follows as above.

\end{proof}

\smallskip

%%%%%%%%%%%%%%%%%%%%%%%%%%%%%%%%%%%%%%%%%%%%%%%%%%%%%%%%%%%%%%%%%

\subsection{The machine $\textbf{M}_3$} \

The recognizing $S$-machine $\textbf{M}_3$ is the composition of $\textbf{M}_2$ with one more simple machine.

The standard base of $\textbf{M}_3$ is taken to be 
$$P_0Q_0P_1Q_1R_1Q_2R_2Q_3P_4Q_4$$
As in the construction of $\textbf{M}_2$, any part of this standard base given the same name as a part of the standard base of $\textbf{M}_2$ contains more letters than its predecessor. The makeup of these parts is clear from the definition of the rules below.

The tape alphabet of the $P_0Q_0$- and $Q_0P_1$-sectors are copies of $Y_1$, while all other tape alphabets naturally correspond to those of $\textbf{M}_2$. The $P_0Q_0$-sector is the input sector of the machine.

The idea of the function of $\textbf{M}_3$ is the following. Given an input configuration, an accepting computation first moves all the letters to the right into the $Q_0P_1$-sector while all other sectors are locked. Then, the subword $Q_0P_1Q_1R_1Q_2R_2Q_3P_4Q_4$ is operated upon as the standard base of $\textbf{M}_2$ while the $P_0Q_0$-sector is locked.

To be precise, we view $\textbf{M}_3$ as the concatenation of $4n$ submachines: The $4n-1$ submachines $\textbf{M}_3(2),\dots,\textbf{M}_3(4n)$ corresponding to the submachines of $\textbf{M}_2$ and the new machine $\textbf{M}_3(1)$.

The submachines corresponding to those of $\textbf{M}_2$ differ only in that the newly introduced part of the standard base consists of a single letter and the newly introduced sector remains locked.

For each part of the standard base, the subset corresponding to the submachine $\textbf{M}_3(1)$ is a singleton. The positive rules of this submachine are in correspondence with $\pazocal{A}$. For $a\in\pazocal{A}$, the corresponding rule has the part $q_0(1)\to a_1^{-1}q_0(1)a_1'$, where $q_0(1)\in Q_0$ and $a_1$ (respectively $a_1'$) is the copy of $a$ in the tape alphabet of the $P_0Q_0$-sector (respectively the $Q_0P_1$-sector). All other sectors of the standard base are locked by this rule.

We also introduce more transition rules, defined in the same way as for previous machines to force the natural order of the concatenation of these submachines. The transition rules $\theta(i,i+1)^{\pm1}$ for $2\leq i\leq 4n-1$ correspond to the rules of the same name in $\textbf{M}_2$, operating on the natural copy of the standard base of $\textbf{M}_2$ in the same way and locking all other sectors. Meanwhile, the transition rules $\theta(12)^{\pm1}$ connect $\textbf{M}_3(1)$ and $\textbf{M}_3(2)$, locking all sectors of the standard base of $\textbf{M}_3$ except for the $Q_0P_1$-sector.

\smallskip

%%%%%%%%%%%%%%%%%%%%%%%%%%%%%%%%%%%%%%%%%%%%%%%%%%%%%%%%%%%%%%%%%

\subsection{Standard computations of $\textbf{M}_3$} \

The step history of a reduced computation of $\textbf{M}_3$ is defined in a manner similar to how it was for reduced computations of $\textbf{M}_2$. The only new letters corresponding to this machine are $(1)$, $(12)$, and $(21)$, which correspond to maximal subcomputations of $\textbf{M}_3(1)$ and transition rules in the obvious way.

Further, for $3\leq i\leq 4n-3$ odd, we define the submachines $\textbf{M}_3(i^-)$ and $\textbf{M}_3(i^+)$ as the submachines of $\textbf{M}_3(i)$ in the same way as the corresponding submachines of $\textbf{M}_2(i)$. Similarly, for $1\leq j\leq k$, we define the submachines $\textbf{M}_3((4n-1)_j)$.

As a result, Lemmas \ref{M_2 step history 1} and \ref{M_2 step history 2} have obvious analogues in $\textbf{M}_3$. The following statement is similar in nature to those, dealing with the newly added steps. Its proof is identical to that of Lemma \ref{M_1 step history}.

\begin{lemma} \label{M_3 step history}

Let $\pazocal{C}$ be a reduced computation of $\textbf{M}_3$ with base $B$. 

\renewcommand{\labelenumi}{(\alph{enumi})}

\begin{enumerate}

\item If $B$ contains a subword $B'$ of the form $(P_0Q_0)^{\pm1}$, then the step history of $\pazocal{C}$ cannot be $(21)(1)(12)$.

\item If $B$ contains a subword $B'$ of the form $(R_2Q_3)^{\pm1}$, then the step history of $\pazocal{C}$ cannot be $(12)(2)(21)$.

\end{enumerate}

\end{lemma}

%\begin{proof}
%
%(a) Assuming toward contradiction, let $W_0\to\dots\to W_t$ be the restriction of $\pazocal{C}$ to the subword $B'$. Then since $W_1$ and $W_{t-1}$ are both $\theta(12)$-admissible, they have empty tape word. Lemma \ref{multiply one letter} then implies that the maximum subcomputation with step history $(1)$ is empty, contradicting the assumption that $\pazocal{C}$ is reduced.
%
%(b) follows from the analogous argument.
%
%\end{proof}

Much of the same terminology regarding reduced computations is carried over from $\textbf{M}_2$. 

For example, the history $H$ of a reduced computation $\pazocal{C}$ of $\textbf{M}_3$ is controlled if $\pazocal{C}$ is the natural copy of a reduced computation of $\textbf{M}_2$ whose history is controlled.

However, a configuration of $\textbf{M}_3$ is \textit{tame} if, in addition to its $P_1Q_1$-, $Q_1R_1$-, $Q_2R_2$-, and $P_4Q_4$-sectors being empty, its $P_0Q_0$-sector is also empty.

\begin{lemma} \label{M_3 no turn}

Let $\pazocal{C}:W_0\to\dots\to W_t$ be a reduced computation such that $W_0$ is an end configuration. If the history $H$ of $\pazocal{C}$ contains a transition rule, then $W_t$ is not an end configuration. 

\end{lemma}

\begin{proof}

Assuming toward contradiction, Lemma \ref{M_3 step history}(a) implies that the step history of $\pazocal{C}$ has no occurrence of $(1)$, $(12)$, or $(21)$.

But then $\pazocal{C}$ can be viewed as a reduced computation of $\textbf{M}_2$, so that it contradicts Lemma \ref{M_2 no turn}.

\end{proof}

\begin{lemma} \label{M_3 input length}

Let $\pazocal{C}:W_0\to\dots\to W_t$ be a reduced computation of $\textbf{M}_3$ in the standard base. Suppose $W_0$ is an input configuration and the step history of $\pazocal{C}$ does not contain the letter $(4n)$. Then $|W_0|_a\leq9n|W_t|_a$.

\end{lemma}

\begin{proof}

If $\pazocal{C}$ is a one-step computation with step (1), then $|W_0|_a\leq|W_t|_a$ by a projection argument.

Otherwise, let $W_0\to\dots\to W_r$ be the subcomputation with step history $(1)(12)$. Then as above $|W_0|_a\leq|W_r|_a$.

By Lemmas \ref{M_2 step history 1}, \ref{M_2 step history 2}, and \ref{M_3 step history}, the subcomputation $W_r\to\dots\to W_t$ can be identified with a reduced computation of $\textbf{M}_2$. But then Lemma \ref{M_2 input length} implies $|W_r|_a\leq9n|W_t|_a$.

\end{proof}

An analogous proof immediately implies the following statement.

\begin{lemma} \label{M_3 end length}

Let $\pazocal{C}:W_0\to\dots\to W_t$ be a reduced computation of $\textbf{M}_3$ in the standard base. Suppose the first letter of the step history of $\pazocal{C}$ is $(4n,4n-1)$. Then $|W_0|_a\leq12n|W_t|_a$.

\end{lemma}

\begin{lemma} \label{M_3 start to end} \

\renewcommand{\labelenumi}{(\alph{enumi})}

\begin{enumerate}

\item Let $\pazocal{C}:W_0\to\dots\to W_t$ be a reduced computation of $\textbf{M}_3$. Suppose $W_0$ is a start configuration and $W_t$ is an end configuration. Then there exists $u\in F(\pazocal{A})$ such that the projection of $W_0$ onto $F(\pazocal{A})$ is $u^n\in\pazocal{L}$.

\item For all $u\in F(\pazocal{A})$, there exists a unique reduced computation $\pazocal{D}_3(u):W_0\to\dots\to W_t$ of $\textbf{M}_3$ in the standard base with step history $(12)(2)(3)\dots(4n-1)(4n-1,4n)$ and such that the projection of $W_0$ onto $F(\pazocal{A})$ is $u^n$.

\item Let $t$ be the length of $\pazocal{D}_3(u)$ and $\ell$ be the length of the subcomputation with step history $(4n-2,4n-1)(4n-1)(4n-1,4n)$. Then $\ell=2k\|u\|+2k+1$ and $t-\ell\leq (c_0+1)(\|u\|+1)$.

\end{enumerate}

\end{lemma}

\begin{proof}

(a) As an application of a rule of step history $(1)$ does not change the projection of a configuration onto $F(\pazocal{A})$, the statement follows from a projection argument and Lemma \ref{M_2 start to end}(a).

Statement (b) follows immediately from Lemmas \ref{M_2 start to end}(b),(c) and \ref{multiply one letter}(a).

(c) Let $\pazocal{E}$ be the maximal subcomputation of $\pazocal{D}_3(u)$ with step history $(12)(2)(3)\dots(4n-2)$ and $\pazocal{E}'$ be the maximal subcomputation with step history $(12)(2)$. Then the length of $\pazocal{E}$ is $t-\ell$ and, for $\ell'$ the length of $\pazocal{E}'$, Lemma \ref{M_2 start to end}(c) implies $t-\ell-\ell'\leq c_0(\|u\|+1)$.

But Lemma \ref{multiply one letter} implies that $\ell'=\|u\|+1$, so that the statement follows.

\end{proof}

\begin{lemma} \label{M_3 language}

The language of accepted inputs of $\textbf{M}_3$ is $\pazocal{L}$. Moreover, for any $u^n\in\pazocal{L}$, there exists a unique accepting computation $\pazocal{C}_3(u)$.

\end{lemma}

\begin{proof}

Lemma \ref{M_3 start to end}(a) implies that any accepted input must be an element of $\pazocal{L}$.

Conversely, for $u^n\in\pazocal{L}$, let $H_1$ be the natural copy of $u^n$ read right to left in the rules of step history $(1)$ and $H_2$ be the history of $\pazocal{C}_2(u)$. Then, for $H\equiv H_1\theta(12)H_2$, the input configuration with input $u^n$ is $H$-admissible with $W\cdot H$ the accept configuration.

The uniqueness of this computation follows immediately from Lemmas \ref{M_3 start to end}(b) and \ref{multiply one letter}(a).

\end{proof}

\begin{lemma} \label{M_3 input controlled}

Let $\pazocal{C}:W_0\to\dots\to W_t$ be a reduced computation of $\textbf{M}_3$. Suppose $W_0$ is an input configuration and $W_t$ is either an input or the accept configuration. Then the sum of the lengths of the subcomputations of $\pazocal{C}$ whose step histories are of the form $(4n-2,4n-1)(4n-1)(4n-1,4n)$ or $(4n,4n-1)(4n-1)(4n-1,4n-2)$ is at least $\left(1-\frac{1}{c_0}\right)t$.

\end{lemma}

\begin{proof}

By Lemmas \ref{multiply one letter} and \ref{M_3 step history}, the step history of $\pazocal{C}$ has prefix $(1)(2)\dots(4n-1)(4n-1,4n)$. Let $W_0\to\dots\to W_s$ be the subcomputation with this step history. Lemma \ref{M_3 start to end}(a) then implies that there exists $u\in F(\pazocal{A})$ such that the input of $W_0$ is $u^n$.

Let $W_0\to\dots\to W_r$ be the maximal subcomputation with step history $(1)$. Then Lemma \ref{multiply one letter} implies $r=\|u^n\|$.

By Lemma \ref{M_3 start to end}(b), the subcomputation $W_r\to\dots\to W_s$ must be $\pazocal{D}_3(u)$. Letting $\ell_1$ be the length of the subcomputation with step history $(4n-2,4n-1)(4n-1)(4n-1,4n)$, Lemma \ref{M_3 start to end}(c) then implies
$s-\ell_1\leq\|u^n\|+(c_0+1)(\|u\|+1)$ and $\ell_1=2k\|u\|+2k+1$.

If $W_t$ is the accept configuration, then Lemma \ref{multiply one letter} implies $t-s=\|u\|$, so that the parameter choice $c_0>>n$ yields
$$t-\ell_1\leq\|u^n\|+(c_0+2)(\|u\|+1)\leq2c_0(\|u\|+1)$$ while $\ell_1\geq2k(\|u\|+1)$. So, $\ell_1\geq\frac{k}{c_0}(t-\ell_1)$, implying $\ell_1\geq\frac{k}{k+c_0}=1-\frac{c_0}{k+c_0}$. The parameter choice $k\geq c_0^2$ then implies the statement.

Now suppose $W_t$ is an input configuration. Then we may apply the same arguments to the inverse subcomputation, so that:

\begin{itemize}

\item the input of $W_t$ is $v^n$ for some $v\in F(\pazocal{A})$, 

\item there exists a maximal subcomputation $\pazocal{E}:W_x\to\dots\to W_t$ of $\pazocal{C}$ whose step history is $(4n,4n-1)(4n-1)\dots(2)(1)$

\item for $\ell_2$ the length of the subcomputation with step history $(4n,4n-1)(4n-1)(4n-1,4n-2)$, $t-x-\ell_2\leq\|v^n\|+(c_0+n)(\|v\|+1)$ and $\ell_2=2k\|v\|+2k+1$.

\end{itemize}

By Lemma \ref{M_3 no turn}, the subcomputation $W_s\to\dots\to W_x$ has step history $(4n)$. So, Lemma \ref{multiply one letter} implies $x-s\leq\|u\|+\|v\|$. Combining these inequalities and taking $c_0>>n$ then yields
$$t-(\ell_1+\ell_2)\leq\|u^n\|+\|v^n\|+(c_0+1)\left(\|u\|+\|v\|+2\right)+\|u\|+\|v\|\leq2c_0(\|u\|+\|v\|+2)$$
while $\ell_1+\ell_2\geq2k(\|u\|+\|v\|+2)$. So, $\ell\geq\frac{k}{c_0}(t-\ell)$ for $\ell=\ell_1+\ell_2$, so that the statement follows as above.

\end{proof}

\begin{lemma} \label{M_3 length}

For any reduced computation $\pazocal{C}:W_0\to\dots\to W_t$ of $\textbf{M}_3$ in the standard base, $t\leq 2c_1\max(\|W_0\|,\|W_t\|)$.

\end{lemma}

\begin{proof}

If $\pazocal{C}$ is a one-step computation with step $(1)$, then the statement follows from Lemma \ref{multiply one letter}.

So, by Lemma \ref{M_3 step history}, we may assume there exists a maximal subcomputation $\pazocal{C}_2:W_r\to\dots\to W_s$ of $\pazocal{C}$ whose step history has no occurrence of the letters $(1)$, $(12)$, or $(21)$. 

Then, $\pazocal{C}_2$ can be viewed as a computation of $\textbf{M}_2$, so that we have $s-r\leq c_1\max(\|W_r\|,\|W_s\|)$ by Lemma \ref{M_2 length}.

If the subcomputation $W_s\to\dots\to W_t$ is nonempty, then it must be a one-step computation with step $(1)$. But then this implies $t-s\leq |W_t|_a+1$ and $|W_s|_a\leq|W_t|_a$.

The symmetric argument implies $r\leq|W_0|_a+1$ and $|W_r|_a\leq|W_0|_a$. 

Hence, $t\leq(c_1+2)\max(\|W_0\|,\|W_t\|)$, so that the statement follows by $c_1\geq2$.

\end{proof}

A two-letter subword of the standard base of $\textbf{M}_3$ is defined to be left-active (or right-active) in $\textbf{M}_3(i)$ in the same way as subwords of the standard base of $\textbf{M}_2$.

For example, the subwords $P_0Q_0$ and $Q_0P_1$ are right-active and left-active, respectively, in $\textbf{M}_3(1)$.

\begin{lemma} \label{M_3 one-step}

Let $\pazocal{C}:W_0\to\dots\to W_t$ be a reduced computation of $\textbf{M}_3(i)$ in the standard base for some $i$. Assume that for some index $j$, $|W_j|_a>4|W_0|_a$. Then there are subwords $U_\ell V_\ell$ and $U_rV_r$ of the standard base such that $U_\ell V_\ell$ is left-active in $\textbf{M}_3(i)$, $U_rV_r$ is right-active in $\textbf{M}_3(i)$, and for $W_0'\to\dots\to W_t'$ the restriction of $\pazocal{C}$ to either sector, $|W_j'|_a<|W_{j+1}'|_a<\dots<|W_t'|_a$.

\end{lemma}

\begin{proof}

By Lemma \ref{M_2 one-step}, it suffices to assume that $i=1$. But then the statement follows immediately from Lemma \ref{multiply one letter}.

\end{proof}

\smallskip

%%%%%%%%%%%%%%%%%%%%%%%%%%%%%%%%%%%%%%%%%%%%%%%%%%%%%%%%%%%%%%%%%

\subsection{The machine $\textbf{M}_4$} \

The machine $\textbf{M}_4$ is the `circular' analogue of a simple tweak to the machine $\textbf{M}_3$.

The standard base of $\textbf{M}_4$ adds just one part to that of $\textbf{M}_3$. In particular, setting $B_3$ as the standard base of $\textbf{M}_3$, the standard base of $\textbf{M}_4$ is $\{t\}B_3$, where $\{t\}$ consists of a single letter (which, clearly, acts as both the start and end letter of its part). The tape alphabet of the new sector in the standard base, i.e the $\{t\}P_0$-sector, is empty. All other tape alphabets are carried over from $\textbf{M}_3$.

A major difference between $\textbf{M}_4$ and the machines constructed in previous sections is that a tape alphabet is assigned to the space after the final letter $Q_4$ of $B_3$, which corresponds to the $Q_4\{t\}$-sector. As such, it is possible for an admissible word of $\textbf{M}_4$ to have base
$$Q_0^{-1}P_0^{-1}\{t\}^{-1}Q_4^{-1}P_4^{-1}P_4Q_4\{t\}P_0Q_0$$
i.e it essentially `wraps around' the standard base. An $S$-machine with this property is called a \textit{cyclic machine}, as one can think of the standard base as being written on a circle.

In this machine, the tape alphabet assigned to the $Q_4\{t\}$-sector is empty. The positive rules of $\textbf{M}_4$ correspond to those of $\textbf{M}_3$, operating on the copy of the hardware of $\textbf{M}_3$ in the same way and locking the new sectors. 

As such, we define the submachines $\textbf{M}_4(i)$ as in $\textbf{M}_3$ and define the step history of a reduced computation in the natural way.

The input sector of $\textbf{M}_4$ is the same as that of $\textbf{M}_3$, i.e the $P_0Q_0$-sector.

There are obvious analogues of the statements from previous sections. Instead of reformulating them here, we reference the previous statements even when in reference to this machine.
%For example, the analogues of Lemmas \ref{M_3 language} and \ref{M_3 length} are listed below.
%
%\begin{lemma} \label{M_4 language}
%
%The language of accepted inputs of $\textbf{M}_4$ is $\pazocal{L}$. Moreover, for any $u^n\in\pazocal{L}$, there exists a unique accepting computation $\pazocal{C}_4(u)$.
%
%\end{lemma}
%
%\begin{lemma} \label{M_4 length}
%
%For any reduced computation $\pazocal{C}:W_0\to\dots\to W_t$ of $\textbf{M}_4$ in the standard base, $t\leq c_1\max(\|W_0\|,\|W_t\|)$.
%
%\end{lemma}

The base of an admissible word of $\textbf{M}_4$ (or any cyclic $S$-machine) is called \textit{revolving} if:

\renewcommand{\labelenumi}{(\alph{enumi})}

\begin{enumerate}

\item it starts and ends with the same base letter, and

\item none of its proper subwords satisfy (a).

\end{enumerate}

An unreduced revolving base is called \textit{faulty}.

Suppose $W$ is an admissible word of a cyclic $S$-machine $\textbf{S}$ whose base $B\equiv xvx$ is revolving. If $v$ has the form $v_1yv_2$, for some letter $y$, then there exists a naturally formed admissible word $W'$ with revolving base $B'\equiv yv_2xv_1y$ and satisfying $|W'|_a=|W|_a$. In this case, $B'$ (respectively $W'$) is called a \textit{cyclic permutation} of $B$ (respectively of $W$).

Note that for any reduced computation $\pazocal{C}:W_0\to\dots\to W_t$ of $\textbf{S}$ with base $B$ and history $H$, there exists a reduced computation $\pazocal{C}':W_0'\to\dots\to W_t'$ with base $B'$ and history $H$ and so that $|W_j|_a=|W_j'|_a$ for all $0\leq j\leq t$.

\begin{lemma} \label{M_4 faulty}

For every reduced computation $\pazocal{C}:W_0\to\dots\to W_t$ of $\textbf{M}_4$ with faulty base $B$, $|W_j|_a\leq c_0\max(|W_0|_a,|W_t|_a)$ for all $0\leq j\leq t$.

\end{lemma}

\begin{proof}

Note that we may assume that $t>1$ and $|W_j|_a>\max(|W_0|_a,|W_t|_a)$ for all $0<j<t$, as otherwise the statement follows from an obvious inductive argument. In particular, since a transition rule (resp $\chi$-rule, connecting rule) does not alter the $a$-length of an admissible word, we may assume that neither the first nor the last letter of the history $H$ of $\pazocal{C}$ is a transition rule (resp $\chi$-rule, connecting rule.).

\ \textbf{1.} Suppose $\pazocal{C}$ is a computation of $\textbf{M}_4(i)$ for some $i\in\{1,2,4,\dots,4n\}$. Then the restriction $\pazocal{C}':W_0'\to\dots\to W_t'$ of $\pazocal{C}$ to any two-letter subword of $B$ has fixed $a$-length, satisfies the hypotheses of Lemma \ref{multiply one letter}, or satisfies the hypotheses of Lemma \ref{unreduced base}. In each case, $|W_j'|_a\leq\max(|W_0'|_a,|W_t'|_a)$. So, $|W_j|_a=\sum|W_j'|_a\leq\sum\max(|W_0'|_a,|W_t'|_a)\leq2\max(|W_0|_a,|W_t|_a)$ for all $0\leq j\leq t$.

\smallskip

\ \textbf{2.} Suppose $\pazocal{C}$ is a computation of $\textbf{M}_4(i^-)$ for some $i\in\{3,5,\dots,4n-3\}$.

If $H$ contains no connecting rule, then an identical argument to the one used in Step 1 applies. So, assume that $H$ contains such a connecting rule, locking the $Q_0P_1$-sector of the standard base.

If $B$ has no occurrence of the letters $P_1^{\pm1}$, then no rule of $H$ changes the $a$-length of an admissible word with base $B$. So, assuming that $B$ contains such a letter, Lemma \ref{locked sectors} and the definition of faulty imply that $B$ has a subword $UV$ of the form $(Q_0P_1)^{\pm1}$. 

If more than one connecting rule occurs in $H$, then $H$ must have a subword $\zeta H'\zeta^{-1}$, where $\zeta$ is a connecting rule and $H'$ contains no connecting rule. Letting $\pazocal{C}'$ be the subcomputation with history $H'$, the restriction of $\pazocal{C}'$ to the $UV$-sector satisfies the hypotheses of Lemma \ref{multiply one letter}. But then $H'$ must be empty, yielding a contradiction.

So, $H$ contains exactly one connecting rule. Let $W_{r-1}\to W_r$ be the subcomputation corresponding to this connecting rule.

By the definition of faulty, any subword of $B$ of the form $(P_1Q_1)^{\pm1}$ is contained in a subword of a cyclic permutation of $B$ of the form $(Q_0P_1Q_1)^{\pm1}$. Letting $\pazocal{C}':W_0'\to\dots\to W_t'$ be the restriction of (a cyclic permutation of) $\pazocal{C}$ to a subword of the form $(Q_0P_1Q_1)^{\pm1}$, Lemma \ref{primitive computations} implies $|W_r'|_a\leq\dots\leq|W_t'|_a$.

Further, any subword of $B$ of the form $P_1P_1^{-1}$ is contained in a subword of a cyclic permutation of $B$ of the form $Q_0P_1P_1^{-1}Q_0^{-1}$. Letting $\pazocal{C}':W_0'\to\dots\to W_t'$ be the restriction of (a cyclic permutation of) $\pazocal{C}$ to a subword of the form $Q_0P_1P_1^{-1}Q_0^{-1}$, Lemma \ref{primitive unreduced} implies the inequalities $|W_r'|_a\leq\dots\leq|W_t'|_a$.

The tape word of any other sector is fixed throughout $\pazocal{C}$, so that $|W_r|_a\leq|W_t|_a$. By assumption, we must then have $r=t$. But then the final letter of $H$ is a connecting rule, contradicting our assumption.

Analogous arguments yield the same inequalities if $\pazocal{C}$ is a computation of $\textbf{M}_4(i^+)$ for some $i$.

\smallskip

\ \textbf{3.} Suppose $\pazocal{C}$ is a computation of $\textbf{M}_4(i)$ for some $i\in\{3,5,\dots,4n-3\}$.

Suppose $H$ has a suffix $\chi H^-$ where $\chi$ is a $\chi$-rule and $H^-$ is the history of a maximal subcomputation of $\textbf{M}_4(i^-)$. Let $\pazocal{C}^-:W_r\to\dots\to W_t$ be the subcomputation with history $H^-$.

Note that if $i$ is of the form $4\ell-1$ (resp $4\ell+1$), then the only sectors of the standard base that $\chi$ does not lock are the $Q_0P_1$-sector and the $R_2Q_3$-sector (resp $R_1Q_2$-sector).

So, by the definition of faulty, any subword of $B$ of the form $(Q_0P_1)^{\pm1}$ is contained in a subword of a cyclic permutation of $B$ of the form $(Q_0P_1Q_1)^{\pm1}$. The restriction of (a cyclic permutation of) $\pazocal{C}^-$ to this subword then satisfies the hypotheses of Lemma \ref{primitive computations}.

Further, any subword of $B$ of the form $P_1^{-1}P_1$ is contained in a subword of a cyclic permutation of $B$ of the form $Q_1^{-1}P_1^{-1}P_1Q_1$. The restriction of (a cyclic permutation of) $\pazocal{C}^-$ to this subword then satisfies the hypotheses of Lemma \ref{primitive unreduced}.

All other sectors have fixed $a$-length throughout $\pazocal{C}^-$, so that $|W_r|_a\leq|W_t|_a$. But this contradicts our assumption.

So, $H$ has no suffix of the form $\chi H^-$. By Step 2, $H$ must contain a $\chi$-rule, so that it must have a suffix of the form $\chi H^+$ where $\chi$ is a $\chi$-rule and $H^+$ is the history of a maximal subcomputation of $\textbf{M}_4(i^+)$.

But then an analogous argument yields a similar contradiction.

\smallskip

\ \textbf{4.} Suppose $\pazocal{C}$ is a computation of $\textbf{M}_4((4n-1)_j)$ for some $1\leq j\leq k$.

As in Step 2, $H$ must contain a connecting rule, as otherwise we may apply the argument used in Step 1.

If $B$ has no occurrence of the letters $R_2^{\pm1}$ or $P_4^{\pm1}$, then no rule of $H$ changes the $a$-length of an admissible word with base $B$. So, we assume that $B$ has a subword of the form $(R_2Q_3)^{\pm1}$ or $(Q_3P_4)^{\pm1}$. 

As in Step 2, this implies that $H$ contains exactly one connecting rule. Let $W_{r-1}\to W_r$ be the subcomputation corresponding to this connecting rule.

Note the following consequences of the definition of faulty and Lemma \ref{locked sectors}:

\begin{itemize}

\item any subword of $B$ of the form $(Q_2R_2)^{\pm1}$ is contained in a cyclic permutation of $B$ of the form $(Q_2R_2Q_3)^{\pm1}$

\item any subword of $B$ of the form $(P_4Q_4)^{\pm1}$ is contained in a cyclic permutation of $B$ of the form $(Q_3P_4Q_4)^{\pm1}$

\item any subword of $B$ of the form $R_2^{-1}R_2$ is contained in a cyclic permutation of $B$ of the form $Q_3^{-1}R_2^{-1}R_2Q_3$

\item any subword of $B$ of the form $P_4P_4^{-1}$ is contained in a cyclic permutation of $B$ of the form $Q_3P_4P_4^{-1}Q_4^{-1}$

\end{itemize}

Then, the restriction of the subcomputation $W_r\to\dots\to W_t$ to any of the subwords above satisfies the hypotheses of Lemma \ref{primitive computations}(5) or Lemma \ref{primitive unreduced}. As any other sector has fixed tape word, this implies $|W_r|_a\leq\dots\leq|W_t|_a$. But this leads to a contradiction in the same way as in Step 2.

\smallskip

\ \textbf{5.} Suppose $\pazocal{C}$ is a computation of $\textbf{M}_4(4n-1)$.

By Step 4, $H$ must then have suffix $\chi H_j$ where $\chi$ is a $\chi$-rule and $H_j$ is the history of a maximal subcomputation of $\textbf{M}_4((4n-1)_j)$. Let $\pazocal{C}_j:W_r\to\dots\to W_t$ be the subcomputation with history $H_j$.

Note that $\chi$ locks every sector of the standard base except for the $R_2Q_3$- and $Q_3P_4$-sectors, while these sectors are locked by any connecting rule. So, $H_j$ must contain no connecting rule.

Then, any unreduced two-letter subword of $B$ must be of the form $R_2R_2^{-1}$, $Q_3^{-1}Q_3$, $Q_3Q_3^{-1}$, or $P_4^{-1}P_4$. We then have the following consequences of the definition of faulty and Lemma \ref{locked sectors}:

\begin{itemize}

\item any subword of $B$ of the form $(R_2Q_3)^{\pm1}$ is contained in a cyclic permutation of $B$ of the form $(Q_2R_2Q_3)^{\pm1}$

\item any subword of $B$ of the form $(Q_3P_4)^{\pm1}$ is contained in a cyclic permutation of $B$ of the form $(Q_3P_4Q_4)^{\pm1}$

\item any subword of $B$ of the form $R_2R_2^{-1}$ is contained in a cyclic permutation of $B$ of the form $Q_2R_2R_2^{-1}Q_2^{-1}$

\item any subword of $B$ of the form $P_4^{-1}P_4$ is contained in a cyclic permutation of $B$ of the form $Q_4^{-1}P_4^{-1}P_4Q_4$

\end{itemize}

As in Step 4, Lemmas \ref{primitive computations}(5) and \ref{primitive unreduced} imply $|W_r|_a\leq|W_t|_a$, yielding a contradiction.

Hence, it suffices to assume that $H$ contains some transition rule.

\smallskip

\ \textbf{6.} Suppose $H$ has a suffix of the form $\theta(i-1,i)H_i$ where $H_i$ is the history of a maximal subcomputation with step history $(i)$ for $3\leq i\leq 4n-3$ of the form $4\ell-1$. 

Let $\pazocal{C}_i:W_r\to\dots\to W_t$ be the subcomputation with history $H_i$. 

Then $\pazocal{C}_i$ is a reduced computation of $\textbf{M}_4(i)$, so that there exists a maximal subcomputation $W_r\to\dots\to W_s$ of $\pazocal{C}_i$ which operates as $\textbf{M}_4(i^-)$.

As $\theta(i-1,i)$ locks every sector of the standard base except for the $Q_0P_1$-sector and the $R_2Q_3$-sector, any unreduced two-letter subword of $B$ must be of the form $Q_0Q_0^{-1}$, $P_1^{-1}P_1$, $R_2R_2^{-1}$, or $Q_3^{-1}Q_3$.

By the definition of faulty, any subword of $B$ of the form $(Q_0P_1)^{\pm1}$ is contained in a subword of a cyclic permutation of $B$ of the form $(Q_0P_1Q_1)^{\pm1}$. Similarly, any subword of $B$ of the form $P_1P_1^{-1}$ is contained in a subword of a cyclic permutation of $B$ of the form $Q_1^{-1}P_1^{-1}P_1Q_1$. As in previous steps, Lemmas \ref{primitive computations} and \ref{primitive unreduced} then imply $|W_r|_a\leq\dots\leq|W_s|_a$. 

As a result, we may assume that $t>s$, so that $H_i$ contains the letter $\chi_i$. As $H_i$ must also contain the connecting rule of $\textbf{M}_4(i^-)$, every unreduced two-letter subword of $B$ must be of the form $R_2R_2^{-1}$ or $Q_3^{-1}Q_3$. In particular, $B$ must be a cyclic permutation of
$$Q_2R_2R_2^{-1}Q_2^{-1}R_1^{-1}Q_1^{-1}P_1^{-1}Q_0^{-1}P_0^{-1}\{t\}^{-1}Q_4^{-1}P_4^{-1}Q_3^{-1}Q_3P_4Q_4\{t\}P_0Q_0P_1Q_1R_1Q_2$$
Let $\pazocal{C}_i^+:W_{s+1}\to\dots\to W_x$ be the maximal subcomputation of $\pazocal{C}_i$ which operates as $\textbf{M}_4(i^+)$.  As the connecting rule of $\textbf{M}_4(i^+)$ locks the $R_2Q_3$-sector, the restriction of $\pazocal{C}_i^+$ to the $Q_2R_2$-sector satisfies the hypotheses of Lemma \ref{multiply one letter}. So, since $W_{s+1}$ is $\chi_i^{-1}$-admissible, $W_x$ cannot be. In particular, $x=t$.

Then, the restriction of $\pazocal{C}_i^+$ to the subword $Q_2R_2R_2^{-1}Q_2^{-1}$ satisfies the hypotheses of Lemma \ref{primitive unreduced}. Since the tape word of any other sector is fixed throughout $\pazocal{C}_i^+$, this implies $|W_s|_a\leq|W_t|_a$, yielding a contradiction.

%An analogous argument leads to a contradiction if $H$ has a suffix $\theta(4n-2,4n-1)H_{4n-1}$ where $H_{4n-1}$ is the history of a maximal subcomputation with step history $(4n-1)$.

\smallskip

\ \textbf{7.} Suppose $H$ has a suffix of the form $\theta(i-1,i)H_i$ where $H_i$ is the history of a maximal subcomputation with step history $(i)$ for $3\leq i\leq 4n-3$ of the form form $4\ell+1$.

Let $\pazocal{C}_i:W_r\to\dots\to W_t$ be the subcomputation with history $H_i$.

As in Step 6, $\pazocal{C}_i$ must have a maximal subcomputation $W_r\to\dots\to W_s$ which operates as $\textbf{M}_4(i^-)$ such that $s<t$. So, since $H_i$ must contain the connecting rule of $\textbf{M}_4(i^-)$, every unreduced two-letter subword of $B$ must be of the form $R_1R_1^{-1}$ or $Q_2^{-1}Q_2$. As a result, $B$ must be a cyclic permutation of
$$Q_1R_1R_1^{-1}Q_1^{-1}P_1^{-1}Q_0^{-1}P_0^{-1}\{t\}^{-1}Q_4^{-1}P_4^{-1}Q_3^{-1}R_2^{-1}Q_2^{-1}Q_2R_2Q_3P_4Q_4\{t\}P_0Q_0P_1Q_1$$
Let $\pazocal{C}_i^+:W_{s+1}\to\dots\to W_x$ be the maximal subcomputation of $\pazocal{C}_i$ which operates as $\textbf{M}_4(i^+)$. As the connecting rule of $\textbf{M}_4(i^+)$ locks the $R_1Q_2$-sector, we again have $x=t$. Applying Lemma \ref{primitive unreduced} to the restriction of $\pazocal{C}_i^+$ to the subword $Q_1R_1R_1^{-1}Q_1^{-1}$ then implies $|W_s|_a\leq|W_t|_a$.

\smallskip

\ \textbf{8.} Suppose $H$ has a suffix of the form $\theta(i+1,i)H_i$ where $H_i$ is the history of a maximal subcomputation with step history $(i)$ for $3\leq i\leq 4n-3$ odd.

Let $\pazocal{C}_i:W_r\to\dots\to W_t$ be the subcomputation with history $H_i$.

Letting $W_r\to\dots\to W_s$ be the maximal subcomputation with step history $(i^+)$, as in Steps 4 and 5 we have $|W_r|_a\leq\dots\leq|W_s|_a$. As a result, it suffices to assume that $s<t$.

So, any unreduced two-letter subword of $B$ must be of the form $Q_0Q_0^{-1}$ or $P_1^{-1}P_1$. In particular, $B$ must be a cyclic permutation of
$$Q_1^{-1}P_1^{-1}P_1Q_1R_1Q_2R_2Q_3P_4Q_4\{t\}P_0Q_0Q_0^{-1}P_0^{-1}\{t\}^{-1}Q_4^{-1}P_4^{-1}Q_3^{-1}R_2^{-1}Q_2^{-1}R_1^{-1}Q_1^{-1}$$
Let $\pazocal{C}_i^-:W_{s+1}\to\dots\to W_x$ be the maximal subcomputation of $\pazocal{C}_i$ which operates as $\textbf{M}_4(i^-)$. As in Steps 6 and 7, the presence of the subword $P_0Q_0$ implies that we must have $x=t$. But then an application of Lemma \ref{primitive unreduced} to the restriction of $\pazocal{C}_i^-$ to the subword $Q_1^{-1}P_1^{-1}P_1Q_1$ implies $|W_s|_a\leq|W_t|_a$.

Hence, we may assume that the last letter of the step history of $\pazocal{C}$ is not of the form $(i)$ for some $i\in\{3,5,\dots,4n-3\}$. Moreover, the symmetric argument allows the same assumption to be made about the first letter of the step history.

\smallskip

\ \textbf{9.} Suppose $H$ has a subword $\theta H_{4n-1}$ where $\theta$ is a transition rule and $H_{4n-1}$ is the history of a maximal subcomputation with step history $(4n-1)$.

Note that both $\theta(4n-2,4n-1)$ and $\theta(4n,4n-1)$ lock every sector of the standard base except for the $R_2Q_3$- and $Q_3P_4$-sectors. Meanwhile, for any $j$, the connecting rule of $\textbf{M}_4((4n-1)_j)$ locks these two sectors. So, by Lemma \ref{locked sectors}, no connecting rule can appear in $H_{4n-1}$.

As a result, $H_{4n-1}$ is the history of a computation of $\textbf{M}_4((4n-1)_1)$ if $\theta=\theta(4n-2,4n-1)$ or $\textbf{M}_4((4n-1)_k)$ if $\theta=\theta(4n,4n-1)$.

Note the following consequences of the definition of faulty and Lemma \ref{locked sectors}:

\begin{itemize}

\item any subword of $B$ of the form $(R_2Q_3)^{\pm1}$ is contained in a cyclic permutation of $B$ of the form $(Q_2R_2Q_3)^{\pm1}$

\item any subword of $B$ of the form $(Q_3P_4)^{\pm1}$ is contained in a cyclic permutation of $B$ of the form $(Q_3P_4Q_4)^{\pm1}$

\item any subword of $B$ of the form $R_2R_2^{-1}$ is contained in a cyclic permutation of $B$ of the form $Q_2R_2R_2^{-1}Q_2^{-1}$

\item any subword of $B$ of the form $P_4^{-1}P_4$ is contained in a cyclic permutation of $B$ of the form $Q_4^{-1}P_4^{-1}P_4Q_4$

\end{itemize}

Let $\pazocal{C}':W_r\to\dots\to W_s$ be the subcomputation with history $H_{4n-1}$. Then, as in previous steps, Lemmas \ref{primitive computations} and \ref{primitive unreduced} imply $|W_r|_a\leq\dots\leq |W_s|_a$. So, we must have $s<t$.

As no connecting rule can occur in $H_{4n-1}$, the subsequent rule of $H$ must be $\theta^{-1}$, i.e $H$ has a subword $\theta H_{4n-1}\theta^{-1}$.

If $B$ contains a subword of the form $(Q_2R_2)^{\pm1}$ or $(P_4Q_4)^{\pm1}$, then the restriction of $\pazocal{C}'$ to this sector satisfies the hypotheses of Lemma \ref{multiply one letter}. But then $H_{4n-1}$ must be empty, yielding a contradiction.

So, $B$ cannot contain such a subword. By the definition of faulty, it follows that $B$ cannot contain the letters $R_2^{\pm1}$ or $P_4^{\pm1}$. In particular, $B$ must be a cyclic permutation of $Q_3Q_3^{-1}Q_3$.

However, no rule alters the $a$-length of an admissible with such a base, so that $|W_j|_a=|W_0|_a$ for all $j$.

Hence, $H$ has no such subword. What's more, by the symmetric argument, $H$ has no subword of the form $H_{4n-1}\theta$. So, we may assume that the step history of $\pazocal{C}$ has no occurrence of $(4n-1)$, $(4n)$, $(4n-1,4n)$, $(4n,4n-1)$, $(4n-2,4n-1)$, or $(4n-1,4n-2)$.

\smallskip

\ \textbf{10.} Suppose the step history of $\pazocal{C}$ contains the letter $(12)$. 

Then $H$ has a subword $H_1\theta(12)H_2$ where $H_1$ and $H_2$ are histories of maximal subcomputations of step history $(1)$ and $(2)$, respectively.

As $\theta(12)$ locks all sectors of the standard base except for the $Q_0P_1$-sector, $B$ must be a cyclic permutation of 
$$P_0Q_0Q_0^{-1}P_0^{-1}\{t\}^{-1}Q_4^{-1}P_4^{-1}Q_3^{-1}R_2^{-1}Q_2^{-1}R_1^{-1}Q_1^{-1}P_1^{-1}P_1Q_1R_1Q_2R_2Q_3P_4Q_4\{t\}P_0$$ 
Let $\pazocal{C}_1:W_r\to\dots\to W_s$ be the subcomputation with history $H_1$. Then the restriction of $\pazocal{C}_1$ to the $P_0Q_0$-sector satisfies the hypotheses of Lemma \ref{multiply one letter}. So, since $W_s$ is $\theta(12)$-admissible, $W_r$ cannot be. In particular, $r=0$.

Let $\pazocal{C}_1':W_0'\to\dots\to W_s'$ be the restriction of $\pazocal{C}_1$ to the subword $P_0Q_0Q_0^{-1}P_0^{-1}$. Then we may apply Lemma \ref{primitive unreduced} to $\pazocal{C}_1'$, so that $|W_s'|_a\leq|W_0'|_a$.

But every other sector must have fixed $a$-length throughout $\pazocal{C}_1$, so that $|W_s|_a\leq|W_0|_a$.

Hence, we may assume that the step history of $\pazocal{C}$ has no occurrence of the letter $(12)$ and, by the symmetric argument, no occurrence of the letter $(21)$. In particular, $\pazocal{C}$ has no subcomputation with step history $(1)$ and every rule of $H$ locks the $P_0Q_0$-sector.

\smallskip

\ \textbf{11.} Suppose $H$ contains a connecting rule $\zeta$ of $\textbf{M}_4(i^-)$ for some $3\leq i\leq 4n-3$ of the form $4\ell-1$.

Note that every sector of the standard base is locked by at least one of $\zeta$, $\chi_i$, or the connecting rule of $\textbf{M}_4(i^+)$. So, by Steps 6-8, $H$ must contain a subword $\theta(i-1,i)H_i\theta(i,i-1)$ where $H_i$ contains $\zeta$ and is the history of a subcomputation $\pazocal{C}_i$ of $\textbf{M}_4(i)$.

The only sector of the standard base not locked by at least one of $\theta(i-1,i)$ or $\zeta$ is the $R_2Q_3$-sector. So, $B$ must be a cyclic permutation of
$$Q_2R_2R_2^{-1}Q_2^{-1}R_1^{-1}Q_1^{-1}P_1^{-1}Q_0^{-1}P_0^{-1}\{t\}^{-1}Q_4^{-1}P_4^{-1}Q_3^{-1}Q_3P_4Q_4\{t\}P_0Q_0P_1Q_1R_1Q_2$$
Suppose $H_i$ contains the letter $\chi_i$. As the connecting rule of $\textbf{M}_4(i^+)$ locks the $R_2Q_3$-sector, it cannot occur in $H_i$. So, $H_i$ must have a subword $\chi_iH_i'\chi_i^{-1}$, where $H_i'$ is the history of a subcomputation $\pazocal{C}_i'$ of $\textbf{M}_4(i^+)$ not containing a connecting rule. But then applying Lemma \ref{multiply one letter} to the restriction of $\pazocal{C}_i'$ to the $Q_2R_2$-sector implies that $H_i'$ is empty.

So, $\pazocal{C}_i$ must be a computation of $\textbf{M}_4(i^-)$. But then applying Lemma \ref{primitive computations}(4) to the restriction of $\pazocal{C}_i$ to the subword $Q_0P_1Q_1$ implies that $H_i$ is empty, yielding a contradiction.

\smallskip

\ \textbf{12.} Suppose $H$ contains a connecting rule $\zeta$ of $\textbf{M}_4(i^-)$ for some $3\leq i\leq 4n-3$ of the form $4\ell+1$.

As in Step 11, $H$ must contain a subword $\theta(i-1,i)H_i\theta(i,i-1)$ where $H_i$ contains $\zeta$ and is the history of a computation $\pazocal{C}_i$ of $\textbf{M}_4(i)$.

The only sector of the standard base not locked by at least one of $\theta(i-1,i)$ or $\zeta$ is the $R_1Q_2$-sector. So, $B$ must be a cyclic permutation of
$$Q_1R_1R_1^{-1}Q_1^{-1}P_1^{-1}Q_0^{-1}P_0^{-1}\{t\}^{-1}Q_4^{-1}P_4^{-1}Q_3^{-1}R_2^{-1}Q_2^{-1}Q_2R_2Q_3P_4Q_4\{t\}P_0Q_0P_1Q_1$$
Suppose $H_i$ contains the letter $\chi_i$. As in Step 11, $H_i$ must then have a subword $\chi_iH_i'\chi_i^{-1}$, where $H_i'$ is the history of a subcomputation $\pazocal{C}_i'$ of $\textbf{M}_4(i^+)$ not containing a connecting rule. But then applying Lemma \ref{multiply one letter} to the restriction of $\pazocal{C}_i'$ to the $Q_1R_1$-sector implies that $H_i'$ is empty.

So, $\pazocal{C}_i$ must be a computation of $\textbf{M}_4(i^-)$. But then applying Lemma \ref{primitive computations}(4) to the restriction of $\pazocal{C}_i$ to the subword $Q_0P_1Q_1$ implies that $H_i$ is empty, yielding a contradiction.

\smallskip

\ \textbf{13.} Suppose $H$ contains a connecting rule $\zeta$ of $\textbf{M}_4(i^+)$ for some $3\leq i\leq 4n-3$ odd.

Similar to the arguments in Steps 11 and 12, $H$ must contain a subword $\theta(i+1,i)H_i\theta(i,i+1)$, where $H_i$ contains $\zeta$ and is the history of a computation $\pazocal{C}_i$ of $\textbf{M}_4(i)$.

The only sector of the standard base not locked by at least one of $\theta(i+1,i)$ or $\zeta$ is the $Q_0P_1$-sector. So, $B$ must be a cyclic permutation of
$$\{t\}P_0Q_0Q_0^{-1}P_0^{-1}\{t\}^{-1}Q_4^{-1}P_4^{-1}Q_3^{-1}R_2^{-1}Q_2^{-1}R_1^{-1}Q_1^{-1}P_1^{-1}P_1Q_1R_1Q_2R_2Q_3P_4Q_4\{t\}$$
Suppose $H_i$ contains the letter $\chi_i^{-1}$. As the connecting rule of $\textbf{M}_4(i^-)$ locks the $Q_0P_1$-sector, it cannot occur in $H_i$. So, $H_i$ must have a subword $\chi_i^{-1}H_i'\chi_i$ where $H_i'$ is the history of a subcomputation $\pazocal{C}_i'$ of $\textbf{M}_4(i^-)$ not containing a connecting rule. Applying Lemma \ref{multiply one letter} to the restriction of $\pazocal{C}_i'$ to the $P_1Q_1$-sector then implies that $H_i'$ is empty, yielding a contradiction.

So, $\pazocal{C}_i$ must be a computation of $\textbf{M}_4(i^+)$. But then we may apply Lemma \ref{primitive computations}(4) to the restriction of $\pazocal{C}_i$ to the subword $Q_2R_2Q_3$ (respectively $Q_1R_1Q_2$) if $i$ is of the from $4\ell-1$ (respectively $4\ell+1)$, so that that $H_i$ must be empty.

Hence, by Steps 11-13, we may assume that $H$ contains no connecting rule.

\smallskip

\ \textbf{14.} Suppose the step history of $\pazocal{C}$ contains the letter $(4n-3,4n-2)$.

Then $H$ has a subword $H_{4n-3}\theta(4n-3,4n-2)H_{4n-2}$ where each $H_i$ is the history of a maximal subcomputation of step history $(i)$.

Note that the rule $\theta(4n-3,4n-2)$ locks every sector of the standard base except for the $Q_0P_1$- and $R_1Q_2$-sectors. So, every unreduced two-letter subword of $B$ must be of the form $Q_0Q_0^{-1}$, $P_1^{-1}P_1$, $R_1R_1^{-1}$, or $Q_2^{-1}Q_2$.

By Steps 6-8, $H_{4n-3}$ cannot be a prefix of $H$. So, since $H_{4n-3}$ contains no connecting rule, $\theta(4n-2,4n-3)H_{4n-3}\theta(4n-3,4n-2)$ must be a subword of $H$. 

Further, the subcomputation $\pazocal{C}_{4n-3}$ of history $H_{4n-3}$ must be a computation of $\textbf{M}_4((4n-3)^+)$. If $B$ contains a subword of the form $(Q_1R_1)^{\pm1}$, then the restriction of $\pazocal{C}_{4n-3}$ to this sector satisfies the hypotheses of Lemma \ref{multiply one letter}. But then $H_{4n-3}$ must be empty, so that $H$ is not reduced.

So, $B$ cannot contain the letters $R_1^{\pm1}$. This implies that $B$ must be a cyclic permutation of
$$P_0Q_0Q_0^{-1}P_0^{-1}\{t\}^{-1}Q_4^{-1}P_4^{-1}Q_3^{-1}R_2^{-1}Q_2^{-1}Q_2R_2Q_3P_4Q_4\{t\}P_0$$
Let $\pazocal{C}_{4n-2}:W_r\to\dots\to W_s$ be the subcomputation with history $H_{4n-2}$. The restriction of $\pazocal{C}_{4n-2}$ to the $R_2Q_3$-sector satisfies the hypotheses of Lemma \ref{multiply one letter}. So, since $W_r$ is $\theta(4n-2,4n-3)$-admissible, $W_s$ cannot be. By Step 9, this implies that $s=t$ and $H_{4n-2}$ is a suffix of $H$.

Let $\pazocal{C}_{4n-2}':W_r'\to\dots\to W_t'$ be the restriction of $\pazocal{C}_{4n-2}$ to the subword $Q_3^{-1}R_2^{-1}Q_2^{-1}Q_2R_2Q_3$. As every rule with step history $(4n-2)$ locks the $Q_2R_2$-sector, we may view the subwords with base $(Q_2R_2)^{\pm1}$ as a single state letter. With this view, $\pazocal{C}_{4n-2}'$ satisfies the hypotheses of Lemma \ref{primitive unreduced}, so that $|W_r'|_a\leq\dots\leq|W_t'|_a$.

For $\pazocal{C}_{4n-2}'':W_r''\to\dots\to W_t''$ the restriction of $\pazocal{C}_{4n-2}$ to a subword $(Q_3P_4)^{\pm1}$, Lemma \ref{multiply one letter} implies that $|W_r''|_a\leq\dots\leq|W_t''|_a$.

Any other sector must have fixed tape word throughout $\pazocal{C}_{4n-2}$. But then $|W_r|_a\leq|W_t|_a$, contradicting our assumption.

Hence, we may assume that the step history of $\pazocal{C}$ has no occurrence of $(4n-2)$, so that the $Q_3P_4$-sector is locked by every rule of $H$.

\smallskip

\ \textbf{15.} Suppose the step history of $\pazocal{C}$ contains the letter $(2i,2i+1)$ for some $1\leq i\leq 2n-2$ odd.

Note that $\theta(2i,2i+1)$ locks all sectors of the standard base except for the $Q_0P_1$-sector and the $R_2Q_3$-sector. So, any unreduced two-letter subword of $B$ must be of the form $Q_0Q_0^{-1}$, $P_1^{-1}P_1$, $R_2R_2^{-1}$, or $Q_3^{-1}Q_3$.

By Steps 6-8, the step history of $\pazocal{C}$ must then have a subword $(2i,2i+1)(2i+1)(2i+1,2i)$. Let $\pazocal{C}_{2i+1}$ be the maximal subcomputation with step history $(2i+1)$ in this subword. By Steps 11-13, the history $H_{2i+1}$ of $\pazocal{C}_{2i+1}$ contains no connecting rule. So, $\pazocal{C}_{2i+1}$ is a computation of $\textbf{M}_4((2i+1)^-)$.

If $B$ contains a subword of the form $(P_1Q_1)^{\pm1}$, then the restriction of $\pazocal{C}_{2i+1}$ to this subword satisfies the hypotheses of Lemma \ref{multiply one letter}, so that $H_{2i+1}$ must be empty. So, $B$ cannot contain such a subword and, by the definition of faulty, cannot contain the letters $P_1^{\pm1}$. In particular, $B$ must be a cyclic permutation of 
$$P_0Q_0Q_0^{-1}P_0^{-1}\{t\}^{-1}Q_4^{-1}P_4^{-1}Q_3^{-1}Q_3P_4Q_4\{t\}P_0$$
But then Steps 9, 10, and 14 imply that the application of any rule of $H$ does not alter the tape word of an admissible word with such a base, so that $|W_j|_a=|W_0|_a$ for all $j$.

\smallskip

\ \textbf{16.} Suppose the step history of $\pazocal{C}$ contains the letter $(2i,2i+1)$ for some $1\leq i\leq 2n-2$ even. 

Then $H$ has a subword $H_{2i}\theta(2i,2i+1)H_{2i+1}$ where $H_{2i}$ and $H_{2i+1}$ are histories of maximal subcomputations of step history $(2i)$ and $(2i+1)$, respectively.

Note that $\theta(2i,2i+1)$ locks all sectors of the standard base except for the $Q_0P_1$- and $R_1Q_2$-sectors. So, any unreduced two-letter subword of $B$ must be of the form $Q_0Q_0^{-1}$, $P_1^{-1}P_1$, $R_1R_1^{-1}$, or $Q_2^{-1}Q_2$.

As in Step 15, we then have that $H$ also contains a subword $\theta(2i,2i+1)H_{2i+1}\theta(2i+1,2i)$. Again, this implies that $B$ cannot contain the letters $P_1^{\pm1}$. So, $B$ must be a cyclic permutation of
$$P_0Q_0Q_0^{-1}P_0^{-1}\{t\}^{-1}Q_4^{-1}P_4^{-1}Q_3^{-1}R_2^{-1}Q_2^{-1}Q_2R_2Q_3P_4Q_4\{t\}P_0$$
Let $\pazocal{C}_{2i}:W_r\to\dots\to W_s$ be the subcomputation with history $H_{2i}$. The restriction of $\pazocal{C}_{2i}$ to the $R_2Q_3$-sector satisfies the hypotheses of Lemma \ref{multiply one letter}, so that $r=0$.

Let $\pazocal{C}_{2i}':W_0'\to\dots\to W_s'$ be the restriction of $\pazocal{C}_{2i}$ to the subword $Q_3^{-1}R_2^{-1}Q_2^{-1}Q_2R_2Q_3$. As in Step 14, Lemma \ref{primitive unreduced} then implies that $|W_s'|_a\leq|W_0'|_a$.

But all other sectors have fixed tape word throughout $\pazocal{C}_{2i}$, so that $|W_s|_a\leq|W_0|_a$.

Hence, Steps 15 and 16 imply that the step history $\pazocal{C}$ contains no letter of the form $(2i,2i+1)$. The symmetric argument further implies that it contains no letter of the form $(2i+1,2i)$.

\smallskip

\ \textbf{17.} Suppose the step history of $\pazocal{C}$ contains the letter $(2i,2i-1)$ for some $2\leq i\leq 2n-2$ even.

As $\theta(2i,2i-1)$ locks every sector of the standard base except for the $Q_0P_1$- and $R_2Q_3$-sectors, any unreduced two-letter subword of $B$ must be of the form $Q_0Q_0^{-1}$, $P_1^{-1}P_1$, $R_2R_2^{-1}$, or $Q_3^{-1}Q_3$.

By Steps 6-8, the step history of $\pazocal{C}$ must contain the subword $(2i,2i-1)(2i-1)(2i-1,2i)$. Let $\pazocal{C}_{2i-1}$ be the maximal subcomputation with step history $(2i-1)$ in this subword. By Steps 11-13, the history $H_{2i-1}$ of $\pazocal{C}_{2i-1}$ contains no connecting rule. So, $\pazocal{C}_{2i-1}$ is a computation of $\textbf{M}_4((2i-1)^+)$.

If $B$ contains a subword of the form $(Q_2R_2)^{\pm1}$, then the restriction of $\pazocal{C}_{2i-1}$ to this subword satisfies the hypotheses of Lemma \ref{multiply one letter}, so that $H_{2i-1}$ must be empty. So, $B$ cannot contain the letters $R_2^{\pm1}$. In particular, $B$ must be a cyclic permutation of 
$$P_0Q_0Q_0^{-1}P_0^{-1}\{t\}^{-1}Q_4^{-1}P_4^{-1}Q_3^{-1}Q_3P_4Q_4\{t\}P_0$$
As in Step 15, this implies $|W_j|_a=|W_0|_a$ for all $j$.

The same argument implies that the step history of $\pazocal{C}$ cannot contain a letter $(2i-1,2i)$ for $1\leq i\leq 2n-2$ even.

So, every transition rule of $H$ must be of the form $\theta(2i-1,2i)^{\pm1}$ for $2\leq i\leq 2n-2$ odd.

\smallskip

\ \textbf{18.} Finally, assume that $\pazocal{C}$ contains the letter $(2i,2i-1)$ for some $2\leq i\leq2n-2$ odd. 

Then $H$ must contain a subword $H_{2i}\theta(2i,2i-1)H_{2i-1}$ where $H_{2i-1}$ and $H_{2i}$ are histories of maximal subcomputations of step history $(2i-1)$ and $(2i)$, respectively.

Since $\theta(2i,2i-1)$ locks every sector of the standard base except for the $Q_0P_1$- and $R_1Q_2$-sectors, any unreduced two-letter subword of $B$ must be of the form $Q_0Q_0^{-1}$, $P_1^{-1}P_1$, $R_1R_1^{-1}$, or $Q_2^{-1}Q_2$.

As in Step 17, the step history of $\pazocal{C}$ must contain the subword $(2i,2i-1)(2i-1)(2i-1,2i)$. Let $\pazocal{C}_{2i-1}$ be the maximal subcomputation with step history $(2i-1)$ in this subword. By Steps 11-13, the history $H_{2i-1}$ of $\pazocal{C}_{2i-1}$ contains no connecting rule, so that $\pazocal{C}_{2i-1}$ is a computation of $\textbf{M}_4((2i-1)^+)$.

If $B$ contains a subword of the form $(Q_1R_1)^{\pm1}$, then the restriction of $\pazocal{C}_{2i-1}$ to this subword satisfies the hypotheses of Lemma \ref{multiply one letter}, so that $H_{2i-1}$ must be empty. So, $B$ cannot contain $R_1^{\pm1}$, and so is a cyclic permutation of
$$P_0Q_0Q_0^{-1}P_0^{-1}\{t\}^{-1}Q_4^{-1}P_4^{-1}Q_3^{-1}R_2^{-1}Q_2^{-1}Q_2R_2Q_3P_4Q_4\{t\}P_0$$
As in Step 16, letting $\pazocal{C}_{2i}:W_r\to\dots\to W_s$ be the subcomputation with history $H_{2i}$, we must then have $r=0$ and $|W_s|_a\leq|W_0|_a$. Thus, as we can apply the symmetric argument, we reach a final contradiction.

\end{proof}

\medskip

%%%%%%%%%%%%%%%%%%%%%%%%%%%%%%%%%%%%%%%%%%%%%%%%%%%%%%%%%%%%%%%%%

\subsection{The machines $\textbf{M}_{5,1}$ and $\textbf{M}_{5,2}$} \

The cyclic machine $\textbf{M}_{5,1}$ functions as the `parallel' composition of the machine $\textbf{M}_4$ with itself a large number of times.

Letting $\{t(i)\}B_3(i)$ be a copy of the standard base of $\textbf{M}_4$ for $i\in\{1,\dots,L\}$, the standard base of $\textbf{M}_{5,1}$ is
$$\{t(1)\}B_3(1)\{t(2)\}B_3(2)\dots\{t(L)\}B_3(L)$$
For any letter of $\{t(i)\}B_3(i)$ (or its inverse), the index $i$ is called its \textit{coordinate}.

The tape alphabet of any sector formed by a one-letter part $\{t(i)\}$ of the standard base (including the $Q_4(L)\{t(1)\}$-sector) is defined to be empty. The tape alphabets of all other sectors arise from $\textbf{M}_4$ in the natural way.

The rules of $\textbf{M}_{5,1}$ are in correspondence with those of $\textbf{M}_4$, with each rule operating in parallel on each of the copies of the standard base of $\textbf{M}_4$ in the same way as its corresponding rule.

The copies of the input sector are taken as the input sectors of the machine.

Naturally, there arise submachines $\textbf{M}_{5,1}(i)$ corresponding to the submachines $\textbf{M}_4$. As such, the definition of step history and controlled history extend to reduced computations of $\textbf{M}_{5,1}$.

The statements of Section 4.7 have natural analogues in $\textbf{M}_{5,1}$. For example, letting $I_5(w)$ be the input configuration with the natural copy of $w$ in each $P_0(i)Q_0(i)$-sector, then the following is the analogue of Lemma \ref{M_3 language}.

\begin{lemma} \label{M_{5,1} language}

An input configuration $W$ is accepted by $\textbf{M}_{5,1}$ if and only if $W\equiv I_5(u^n)$ for some $u^n\in\pazocal{L}$. Moreover, for any $u^n\in\pazocal{L}$, there exists a unique accepting computation $\pazocal{C}_{5,1}(u)$ of the input configuration $I_5(u^n)$.

\end{lemma}

The cyclic machine $\textbf{M}_{5,2}$ is constructed in much the same way as $\textbf{M}_{5,1}$, but with one fundamental difference: Each rule locks the first input sector, i.e the $P_0(1)Q_0(1)$-sector.

The definitions of $\textbf{M}_{5,1}$ extend in an obvious way to $\textbf{M}_{5,2}$, and many of the statements of Section 4.7 again have natural analogues. For example, letting $J_5(w)$ be the input configuration that is obtained from emptying the $P_0(1)Q_0(1)$-sector of the natural copy of $I_5(w)$, the following is the analogue of Lemma \ref{M_3 language} (and Lemma \ref{M_{5,1} language}).

\begin{lemma} \label{M_{5,2} language}

An input configuration $W$ is accepted by $\textbf{M}_{5,2}$ if and only if $W\equiv J_5(u^n)$ for some $u^n\in\pazocal{L}$. Moreover, for any $u^n\in\pazocal{L}$, there exists a unique accepting computation $\pazocal{C}_{5,2}(u)$ of the input configuration $J_5(u^n)$.

\end{lemma}

\medskip

%%%%%%%%%%%%%%%%%%%%%%%%%%%%%%%%%%%%%%%%%%%%%%%%%%%%%%%%%%%%%%%%%

\section{The machine $\textbf{M}$}

\subsection{Definition of the machine} \

The final step of our construction is to combinine the machines $\textbf{M}_{5,1}$ and $\textbf{M}_{5,2}$ to create the cyclic machine $\textbf{M}$ that is sufficient for the proof of Theorem \ref{main theorem}.

Similar to $\textbf{M}_{5,1}$ and $\textbf{M}_{5,2}$, the standard base of $\textbf{M}$ is of the form $\{t(1)\}B_3(1)\dots\{t(L)\}B_3(L)$, with the sectors of the form $P_0(i)Q_0(i)$ taken to be the input sectors. However, each of the parts making up $B_3(i)$ consists of more state letters than its counterparts in $\textbf{M}_{5,1}$ and $\textbf{M}_{5,2}$. 

To be precise, any part of the standard base that is not a one-letter part $\{t(i)\}$ consists of a copy of the corresponding part of the standard base of $\textbf{M}_{5,1}$, a (disjoint) copy of the corresponding part of the standard base of $\textbf{M}_{5,2}$, and two new letters which function as the part's start and end letters. The accept configuration of $\textbf{M}$ is denoted $W_{ac}$.

The set of rules $\Theta$ of $\textbf{M}$ is partitioned into two symmetric sets, $\Theta_1$ and $\Theta_2$. The positive rules of each consist of a set of `working' rules and two more transition rules. Unlike in previous constructions, though, these two sets are not concatenated in order to force them to run sequentially, rather in order to force them to operate `one or the other'.

The rules of $\Theta_1^+$ are defined as follows:

\begin{itemize}

\item The transition rule $\theta(s)_1$ locks all sectors other than the input sectors. It switches the state letters from the start state of $\textbf{M}$ to the copy of the start state of $\textbf{M}_{5,1}$.

\item The positive `working' rules of $\Theta_1^+$ are copies of the positive rules of the machine $\textbf{M}_{5,1}$.

\item The transition rule $\theta(a)_1$ locks all sectors and switches the state letters from the copies of the end letters of $\textbf{M}_{5,1}$ to the end letters of $\textbf{M}$.

\end{itemize}

The rules of $\Theta_2^+$ are defined as follows:

\begin{itemize}

\item The transition rule $\theta(s)_2$ locks each of the sectors locked by $\theta(s)_1$, but also locks the $P_0(1)Q_0(1)$-sector. It switches the state letters from the start state of $\textbf{M}$ to the copy of the start state of $\textbf{M}_{5,2}$.

\item The positive `working' rules of $\Theta_2^+$ are copies of the positive rules of the machine $\textbf{M}_{5,2}$.

\item The transition rule $\theta(a)_2$ locks all sectors and switches the state letters from the copies of the end letters of $\textbf{M}_{5,2}$ to the end letters of $\textbf{M}$.

\end{itemize}

By the definition of the rules, one might infer that the first input sector $P_0(1)Q_0(1)$ is of particular significance. Hence, it is referred to as the \textit{`special' input sector}.

For $w\in F(\pazocal{A})$, the natural copy of $I_5(w)$ (respectively $J_5(w)$) in the hardware of this machine is $\theta(s)_1^{-1}$-admissible (respectively $\theta(s)_2^{-1}$-admissible). We denote $I(w)$ (respectively $J(w)$) as the input configuration satisfying $I(w)\equiv I_5(w)\cdot\theta(s)_1^{-1}$ (respectively $J(w)\equiv J_5(w)\cdot\theta(s)_2^{-1}$). Note that both $I(w)$ and $J(w)$ are $\theta(s)_1$-admissible, while $I(w)$ is not $\theta(s)_2$-admissible if $w\neq1$.

\smallskip

%%%%%%%%%%%%%%%%%%%%%%%%%%%%%%%%%%%%%%%%%%%%%%%%%%%%%%%%%%%%%%%%%

\subsection{Standard computations of $\textbf{M}$} \

Next, we adapt the definition of step history to computations of $\textbf{M}$. To this end, let the letters $(s)_j^{\pm1}$ and $(a)_j^{\pm1}$ represent the transition rules $\theta(s)_j^{\pm1}$ and $\theta(a)_j^{\pm1}$ of $\Theta_j$, respectively, and add the subscript $j$ to each letter of the step history of a maximal subcomputation whose history consists of working rules of $\Theta_j$.

So, an example of a step history of a reduced computation of $\textbf{M}$ is $(s)_1(1)_1(12)_1(2)_1$, while a general step history is some concatenation of the letters 
$$\begin{Bmatrix*}[l]
	&(1)_j, \ (2)_j, \ \dots, \ (4n)_j, \ (12)_j, \ (23)_j, \ \dots, \ (4n-1,4n)_j, \\ 
	&(21)_j, \ (32)_j, \ \dots, \ (4n,4n-1)_j, \ (s)_j^{\pm1}, \ (a)_j^{\pm1}; \ j=1,2
\end{Bmatrix*}$$
A one-step computation of $\textbf{M}$ is defined similar to how it was defined in previous machines. For example, reduced computations with step history $(s)_2^{-1}(s)_1(1)_1(12)_1$ or $(4n)_2(a)_2(a)_1^{-1}$ are one-step computations of $\textbf{M}$.

A reduced computation is called a \textit{one-machine computation} if every letter of its step history has the same index. If this index is $i$, then the computation is called a \textit{one-machine computation of the $i$-th machine}. 

For example, a reduced computation with step history $(s)_1(1)_1(12)_1(2)_1$ is a one-machine computation of the first machine, while a reduced computation with step history $(1)_1(s)_1^{-1}(s)_2(1)_2$ is not a one-machine computation, i.e it is a \textit{multi-machine} computation.

As with previous machines, some subwords clearly cannot appear in the step history of a reduced computation, while other impossibilities are less obvious. However, there are clear analogues of Lemmas \ref{M_2 step history 1}, \ref{M_2 step history 2}, and \ref{M_3 step history}(b) (after adding the same index to each letter of the step histories), as $\textbf{M}$ operates on the standard base as parallel copies of $\textbf{M}_4$ in any one-machine computation whose step history does not contain $(s)_i^{\pm1}$, $(a)_i^{\pm1}$, or $(1)_2$.

The following is the analogue of Lemma \ref{M_3 step history}(a) and is proved in exactly the same way.

\begin{lemma} \label{M step history 1}

Suppose the base $B$ of a reduced computation $\pazocal{C}$ of $\textbf{M}$ contains a subword $UV$ of the form $(P_0(i)Q_0(i))^{\pm1}$. Then the step history of $\pazocal{C}$ cannot be $(21)_1(1)_1(12)_1$. Moreover, if $i\neq 1$, then the step history of $\pazocal{C}$ cannot be $(21)_2(1)_2(12)_2$.

\end{lemma}

\begin{lemma} \label{M step history 2}

Let $\pazocal{C}$ be a reduced computation of $\textbf{M}$ with base $B$.

\renewcommand{\labelenumi}{(\alph{enumi})}

\begin{enumerate}

\item If $B$ contains a subword $UV$ of the form $(Q_0(i)P_1(i))^{\pm1}$, then the step history of $\pazocal{C}$ cannot be $(s)_j(1)_j(s)_j^{-1}$ for $j=1,2$.

\item If $B$ contains a subword $UV$ of the form $(R_2(i)Q_3(i))^{\pm1}$ or $(Q_3(i)P_4(i))^{\pm1}$, then the step history of $\pazocal{C}$ cannot be $(a)_j^{-1}(4n)_j(a)_j$ for $j=1,2$.

\end{enumerate}

\begin{proof}

Both statements follow from an application of Lemma \ref{multiply one letter}(a) to the restriction of $\pazocal{C}$ to the $UV$-sector.

\end{proof}

\end{lemma}

\begin{lemma} \label{long step history}

Let $\pazocal{C}$ be a reduced computation with base $\{t(i)\}B_4(i)$ for some $i\in\{2,\dots,L\}$. Suppose $\pazocal{C}$ contains at least $8n$ distinct maximal one-step computations. Then $\pazocal{C}$ contains a subword of the form $(4n-2,4n-1)_j(4n-1)_j(4n-1,4n)_j$ or $(4n,4n-1)_j(4n-1)_j(4n-1,4n-2)_j$.

\end{lemma}

\begin{proof}

Assuming the step history has no such subword, Lemmas \ref{M_2 step history 1}, \ref{M_2 step history 2}, \ref{M_3 step history}, \ref{M step history 1}, and \ref{M step history 2} imply that the step history is a subword of

\begin{itemize}

\item $(4n-1)_1(4n-2)_1\dots(1)_1(s)_1^{-1}(s)_2(1)_2\dots(4n-2)_2(4n-1)_2$, or

\item $(4n-1)_1(4n)_1(a)_1(a)_2^{-1}(4n)_2(4n-1)_2$

\end{itemize}

But then $\pazocal{C}$ has at most $8n-2$ distinct maximal one-step computations.

\end{proof}

%\begin{lemma} \label{long step history}
%
%Let $\pazocal{C}$ be a reduced computation with base $\{t(i)\}B_4(i)$ for some $i\in\{1,\dots,L\}$. Then the step history of $\pazocal{C}$ either:
%
%\begin{addmargin}[1em]{0em}
%
%(A) contains a subword of the form $(34)_j(4)_j(45)_j$, $(54)_j(4)_j(43)_j$, $(12)_j(2)_j(23)j$, or $(32)_j(2)_j(21)_j$
%
%(B) has length at most $8$
%
%\end{addmargin}
%
%\end{lemma}
%
%\begin{proof}
%
%Suppose (A) is not satisfied. By Lemmas \ref{(A) and (B)}, \ref{M third step history}, and \ref{M fourth step history}, any such one-machine computation has step history of length at most 5, while any such one-machine computation ending (or starting) with a subword of a start or end configuration has step history of length at most 4. What's more, the same lemmas imply that there are at most two maximal one-machine subcomputations of $\pazocal{C}$.
%
%\end{proof}

\begin{lemma} \label{turn}

Let $\pazocal{C}:W_0\to W_1\to W_2$ be a reduced computation with step history $((s)_1^{-1}(s)_2)^{\pm1}$ and base $(P_0(1)Q_0(1))^{\pm1}$. Then $|W_i|_a=0$ for $0\leq i\leq 2$. 

\end{lemma}

\begin{lemma} \label{return to start}

Let $\pazocal{C}:W_0\to\dots\to W_t$ of $\textbf{M}$ be a one-machine computation of the $i$-th machine in the standard base. Suppose the step history of $\pazocal{C}$ is of the form $(s)_ih_i(s)_i^{-1}$. Then there exist $u,v\in F(\pazocal{A})$ with $u\neq v$ such that

\begin{itemize}

\item $W_0\equiv I(u^n)$ and $W_t\equiv I(v^n)$ if $i=1$ or

\item $W_0\equiv J(u^n)$ and $W_t\equiv J(v^n)$ if $i=2$.

\end{itemize}

\end{lemma}

\begin{proof}

Since $W_0$ is $\theta(s)_i$-admissible, it is an input configuration.

Further, as $h_i$ cannot be empty, Lemmas \ref{M_2 step history 1}, \ref{M_2 step history 2}, \ref{M_3 step history}(b), and \ref{M step history 2}(a) imply that it has prefix $$(1)_i(2)_i\dots(4n-1)_i(4n-1,4n)_i$$
Let $W_1\to\dots\to W_s$ be the subcomputation with this step history and $W_1\to\dots\to W_r$ be the subcomputation with step history $(1)_i$. Then $W_r$ is $\theta(12)_i$-admissible, so that its input sectors are empty.

If $i=1$, then Lemma \ref{multiply one letter}(a) implies that $W_0$ must have a copy of the same word $w\in F(\pazocal{A})$ written in each input sector. So, $W_0\equiv I(w)$.

If $i=2$, then the only difference is that the `special' input sector must be empty, i.e $W_0\equiv J(w)$.

The restriction of $W_1\to\dots\to W_s$ to $B_3(2)$ can be identified with a reduced computation of $\textbf{M}_3$. Lemma \ref{M_3 start to end}(a) then implies that there exists $u\in F(\pazocal{A})$ such that $w=u^n$.

The same argument applied to the inverse computation implies that there exists $v\in F(\pazocal{A})$ such that $W_t\equiv I(v^n)$ if $i=1$ and $W_t\equiv J(v^n)$ if $i=2$. Let $W_x\to\dots\to W_t$ be the subcomputation with step history $(4n,4n-1)_i(4n-1)_i\dots(2)_i(1)_i(s)_i^{-1}$. Then applying Lemma \ref{M_3 no turn} to the restriction of $W_s\to\dots\to W_x$ to $B_3(2)$ implies that its step history is $(4n)_i$. As this subcomputation cannot be empty, Lemma \ref{multiply one letter} implies that $W_s\neq W_x$. Hence, by Lemma \ref{M_3 start to end}(b), $u\neq v$.

\end{proof}

\begin{lemma} \label{M language}

An input configuration $W$ is accepted by the machine $\textbf{M}$ if and only if $W\equiv I(u^n)$ or $W\equiv J(u^n)$ for some $u^n\in\pazocal{L}$. Moreover, for any $u\in F(\pazocal{A})$, there exists a unique one-machine computation of the first (respectively second) machine accepting $I(u^n)$ (respectively $J(u^n)$).

\end{lemma}

\begin{proof}

Let $\pazocal{C}$ be an accepting computation for $W$ and $\pazocal{C}'$ be the maximal one-machine computation serving as a prefix of $\pazocal{C}$.

Then the step history of $\pazocal{C}'$ must either be of the form $(s)_ih_i(s)_i^{-1}$ or $(s)_ih_i(a)_i$. By Lemma \ref{return to start}, it suffices to suppose the step history is of the form $(s)_ih_i(a)_i$.

The maximal subcomputation with step history $h_i$ must then be an accepting computation of the natural copy of $W\cdot\theta(s)_i$ in $\textbf{M}_{5,i}$. So, Lemma \ref{M_{5,1} language} implies that $W\equiv I(u^n)$ if $i=1$, while Lemma \ref{M_{5,2} language} implies $W\equiv J(u^n)$ if $i=2$.

The existence and uniqueness of an accepting one-machine computation similarly follow from Lemmas \ref{M_{5,1} language} and \ref{M_{5,2} language}.

\end{proof}

\bigskip

%%%%%%%%%%%%%%%%%%%%%%%%%%%%%%%%%%%%%%%%%%%%%%%%%%%%%%%%%%%%%%%%%

\subsection{Components of a configuration} \

For a configuration $W$ and $1\leq i\leq L$, the \textit{$i$-th component} of $W,$ $W(i)$, is defined to be the admissible subword of $W$ with base $\{t(i)\}B_3(i)$. So, since the tape alphabet of the $Q_4(i)\{t(i+1)\}$-sector is empty for each $i$, $W\equiv W(1)\dots W(L)$ for any configuration $W$. It is useful to note that if a rule $\theta$ is applicable to some configuration $W,$ then $\theta$ operates on each $W(j)$ in parallel for $j\geq2$ (but may not operate on $W(1)$ in the analogous way).

Particularly, for $1\leq i\leq L$, we denote the components $A(i)\equiv W_{ac}(i)$, $I(w,i)\equiv (I(w))(i)$, and $J(w,i)\equiv (J(w))(i)$ for all $w\in F(\pazocal{A})$.

The history $H$ of a reduced computation $\pazocal{C}$ of $\textbf{M}$ is called \textit{controlled} if $\pazocal{C}$ is a one-machine computation and $H$ corresponds to a controlled computation of $\textbf{M}_4$. As such, the next statement follows immediately from Lemma \ref{M_2 controlled}.

\begin{lemma} \label{M controlled}

Let $\pazocal{C}:W_0\to\dots\to W_t$ be a reduced computation of $\textbf{M}$ with controlled history $H$. Then the base of the computation is a reduced word and all configurations are uniquely defined by the history $H$ and the base of $\pazocal{C}$.\newline 
Moreover, if $\pazocal{C}$ is a computation in the standard base, then $|W_j|_a=|W_0|_a$ for all $0\leq j\leq t$, $\|H\|=|W_0(i)|_a+3$ for all $1\leq i\leq L$, and $W_0$ is accepted.

\end{lemma}

Let $V$ be an admissible word with base $B$ and suppose there exists $i\in\{1,\dots,L\}$ such that every letter of $B$ has coordinate $i$. Then, a \textit{coordinate shift} of $V$ is an admissible word $V'$ obtained by changing each of the state letters' coordinates from $i$ to $j$ for some $j\in\{1,\dots,L\}$ and taking the natural copies of the tape words. For example, if $W$ is an accepted configuration, then $W(i)$ and $W(j)$ are coordinate shifts of one another for $i,j\geq2$, while $J(w,1)$ is not a coordinate shift of $J(w,2)$ if $w\neq1$.

\begin{lemma} \label{input sectors locked}

For $i\in\{2,\dots,L\}$, let $\pazocal{C}:A(i)\to\dots\to A(i)$ be a nonempty reduced computation of $\textbf{M}$. Then $\pazocal{C}$ is not a one-machine computation.

\end{lemma}

\begin{proof}

Assume toward contradiction that $\pazocal{C}$ is a one-machine computation of the $j$-th machine. Then $H$ can be factored as $H\equiv \theta(a)_j^{-1}H'\theta(a)_j$ where $H'$ has no letters of the form $\theta(s)_j^{\pm1}$ or $\theta(a)_j^{\pm1}$. 

Let $\pazocal{C}'$ be the subcomputation with history $H'$. Then, we can identify $\pazocal{C}'$ with a reduced computation of $\textbf{M}_4$. This computation starts and ends with the accept configuration of $\textbf{M}_4$, so that Lemma \ref{M_3 no turn} implies that it cannot contain a transition rule.

But then $\pazocal{C}$ has step history $(a)_j^{-1}(4n)_j(a)_j$, so that it contradicts Lemma \ref{M step history 2}(b).

\end{proof}

\begin{lemma} \label{extending one-machine}

Let $V_0\to\dots\to V_t$ be a one-machine computation of the $j$-th machine with history $H$ and base $\{t(i)\}B_3(i)$ for some $i\in\{2,\dots,L\}$. Then there exists a one-machine computation $W_0\to\dots\to W_t$ in the standard base with history $H$ such that $W_\ell(i)\equiv V_\ell$ for all $0\leq \ell\leq t$.

\end{lemma}

\begin{proof}

For each $\ell\in\{0,\dots,t\}$ and $x\in\{2,\dots,L\}$, define $V_\ell(x)$ as the coordinate shift of $V_\ell$ with base $\{t(x)\}B_3(x)$.

If $j=1$, then similarly define $V_\ell(1)$ as the coordinate shift of $V_\ell$ with base $\{t(1)\}B_3(1)$. Conversely, if $j=2$, then define $V_\ell(1)$ as the admissible word obtained from emptying the `special' input sector of the coordinate shift of $V_\ell$.

Now define $W_\ell\equiv V_\ell(1)\dots V_\ell(L)$ for each $0\leq \ell\leq t$. Clearly, $W_\ell(i)\equiv V_\ell$ for all $\ell$.

Letting $H\equiv\theta_1\dots\theta_t$, it follows from construction that $W_{\ell-1}$ is $\theta_\ell$-admissible and $W_{\ell-1}\cdot\theta_\ell\equiv W_\ell$ for all $1\leq\ell\leq t$.

\end{proof}

Using Lemma \ref{extending one-machine}, the following statement is an immediate consequence of Lemma \ref{return to start}.

\begin{lemma} \label{subword return to start}

Let $\pazocal{C}:V_0\to\dots\to V_t$ be a one-machine computation of the $j$-th machine with base $\{t(i)\}B_3(i)$ for some $i\in\{2,\dots,L\}$. Suppose the step history of $\pazocal{C}$ is of the form $(s)_jh_j(s)_j^{-1}$. Then there exist $u,v\in F(\pazocal{A})$ with $u\neq v$ such that $V_0\equiv I(u^n,i)$ and $V_t\equiv I(v^n,i)$.

Moreover, for $H$ the history of $\pazocal{C}$, $I(u^n)\cdot H\equiv I(v^n)$ if $j=1$ and $J(u^n)\cdot H\equiv J(v^n)$ if $j=2$.

\end{lemma}

Similarly, the following is an immediate consequence of Lemmas \ref{M language} and \ref{extending one-machine}.

\begin{lemma} \label{subword M language}

If $W_0$ is an admissible subword of a start configuration with base $\{t(i)\}B_3(i)$ for some $i\in\{2,\dots,L\}$, then there exists a one-machine computation $W_0\to\dots\to A(i)$ of the first machine (respectively of the second machine) if and only if $W_0\equiv I(u^n,i)$ (respectively $W_0\equiv J(u^n,i)$) for some $u^n\in\pazocal{L}$.

\end{lemma}

Finally, the next statement is an immediate consequence of Lemmas \ref{input sectors locked}, \ref{subword return to start}, and \ref{subword M language}.

\begin{lemma} \label{extending computations}

For $i\in\{2,\dots,L\}$, suppose $\pazocal{C}:A(i)\to\dots\to A(i)$ is a reduced computation of $\textbf{M}$ with history $H$. Let $H_1\dots H_\ell$ be the factorization of $H$ such that for all $j\in\{1,\dots,\ell\}$, $H_j$ is the history of a maximal one-machine subcomputation $\pazocal{C}_j:U_j\to\dots\to V_j$ of $\pazocal{C}$. Then $\ell\geq2$ and for all $j$, either: 

\begin{enumerate}

\item $V_j\equiv A(i)$, or 

\item $V_j\equiv I(w_j,i)$ or $J(w_j,i)$ for some $w_j\in\pazocal{L}$.

\end{enumerate}

In case (a), set $W_j^{(1)}\equiv W_j^{(2)}\equiv W_{ac}$; in case (b), set $W_j^{(1)}\equiv I(w_j)$ and $W_j^{(2)}\equiv J(w_j)$. Further, set $W_0^{(1)}\equiv W_0^{(2)}\equiv W_{ac}$.

If $\pazocal{C}_j$ is a one-machine computation of the $z_j$-th machine, then for each $1\leq j\leq \ell$ there exists a reduced computation in the standard base $\pazocal{C}_j':W_{j-1}^{(z_j)}\to\dots\to W_j^{(z_j)}$ with history $H_j$.

\end{lemma}

\smallskip

In other words, Lemma \ref{extending computations} says that except for the insertion/deletion of elements of $\pazocal{L}$ in the `special' input sector between its maximal one-machine subcomputations, the computation $\pazocal{C}$ can be `almost-extended' to a reduced computation $\pazocal{C}':W_{ac}\to\dots\to W_{ac}$ (though such a computation need not exist).

\begin{lemma} \label{projection admissible configuration not}

Let $W$ be an accepted configuration and $\theta\in\Theta$. For $i\in\{2,\dots,L\}$, suppose $W(i)$ is $\theta$-admissible while $W$ is not. Then either:

\begin{addmargin}[1em]{0em}

(1) $\theta=\theta(s)_2$ and $W\equiv I(u^n)$ for some $u^n\in\pazocal{L}\setminus\{1\}$, or

(2) $\theta=\theta(12)_1$ and $W$ has $u^n$ written in the `special' input sector for some $u^n\in\pazocal{L}\setminus\{1\}$.

\end{addmargin}

In particular, the configuration obtained from $W$ by emptying the `special' input sector is $\theta$-admissible.

\end{lemma}

\begin{proof}

The symmetry of the rules implies that $W(j)$ is $\theta$-admissible for each $j\geq2$. So, $W(1)$ must not be $\theta$-admissible. By the definition of the rules, $\theta$ must lock the `special' input sector while that sector is not empty in $W$.

\ \textbf{1.} Suppose $\theta\in\Theta_2$.

As $W$ is accepted, it must be $\theta'$-admissible for some $\theta'\in\Theta$. But each rule of $\Theta_2$ locks the `special' input sector, so that $\theta'\in\Theta_1$. As a result, $W(i)$ is admissible for rules from both $\Theta_1$ and $\Theta_2$, which implies that $W$ must either be a start or an end configuration. 

But the only accepted end configuration is $W_{ac}$, which has empty `special' input sector. So, $W$ must be an accepted start configuration, $\theta'=\theta(s)_1$, and $\theta=\theta(s)_2$. Finally, since $W$ has nonempty `special' input sector, Lemma \ref{M language} yields $W\equiv I(u^n)$ for $u\neq1$.

\ \textbf{2.} Suppose $\theta\in\Theta_1$.

Let $\pazocal{C}':W_{ac}\equiv V_0\to\dots\to V_t\equiv W$ be the inverse of an accepting computation of $W$. 

As the rules of $\Theta_1$ operate in parallel as $\textbf{M}_4$, $W(1)$ cannot be a coordinate shift of $W(i)$. So, $\pazocal{C}'$ cannot be a one-machine computation of the first machine.

As a result, there exists a maximal one-machine subcomputation $\pazocal{D}':V_r\to\dots\to V_s$ of the second machine such that the subsequent subcomputation $\pazocal{E}':V_s\to\dots\to V_t\equiv W$ is a (perhaps empty) one-machine computation of the first machine. The parallel nature of the rules of $\Theta_1$ imply $V_s(1)$ is not a coordinate shift of $V_s(i)$.

Note that if $\pazocal{E}'$ is empty, then $V_s\equiv W$, so that $V_s(i)$ is $\theta$-admissible. Otherwise, $V_s(i)$ is $\theta'$-admissible for $\theta'\in\Theta_1$ the first rule in the history of $\pazocal{E}'$.

So, since $V_s(i)$ is $(\theta'')^{-1}$-admissible for $\theta''\in\Theta_2$ the final rule in the history of $\pazocal{D}'$, it follows that $V_s$ is either an accepted start or end configuration. Lemma \ref{M language} then implies that $V_s\equiv J(w)$ for some $w\in\pazocal{L}\setminus\{1\}$.

As no rule of a computation with step history $(s)_1(1)_1$ locks the special input sector, $\theta$ cannot be such a rule. So, the step history of $\pazocal{E}'$ has prefix $(s)_1(1)_1$. 

Lemma \ref{M step history 2}(a) then implies that any subsequent letter of the step history must be $(12)_1$. But since $V_s$ has empty `special' input sector while every other input sector is nonempty, $V_t$ cannot be $\theta(12)_1$-admissible. So, the entire step history of $\pazocal{C}'_1$ is $(s)_1(1)_1$ and $\theta=\theta(12)_1$.

Since $W(i)$ is $\theta$-admissible, Lemma \ref{multiply one letter} implies that the `special' input sector of $W$ contains the natural copy of the word $w^{-1}\in\pazocal{L}\setminus\{1\}$. Note that removing $w^{-1}$ from the `special' input sector of $W$ yields a configuration that is $\theta$-admissible.

\end{proof}

For $W$ an accepted configuration of $\textbf{M}$ such that $W\neq W_{ac}$, let $A(W)$ be the set of accepting computations of $W$. For $\pazocal{C}\in A(W)$, define $\ell(\pazocal{C})$ as the number of maximal one-machine subcomputations of $\pazocal{C}$. Then, define $\ell(W)=\min\{\ell(\pazocal{C})\mid\pazocal{C}\in A(W)\}$.

For simplicity, further define $\ell(W_{ac})=0$.

\begin{lemma} \label{two one-machine}

For any accepted configuration $W$ of $\textbf{M}$, $\ell(W)\leq2$.

\end{lemma}

\begin{proof}

Suppose $\ell=\ell(W)\geq3$ and set $\pazocal{C}\in A(W)$ such that $\ell(\pazocal{C})=\ell$.

Let $H$ be the history of the inverse computation of $\pazocal{C}$. Then, factor $H\equiv H_1\dots H_\ell$ such that each $H_i$ is the history of a maximal one-machine subcomputation.

For $j\in\{1,\dots,\ell-1\}$, let $V_j\equiv W_{ac}\cdot(H_1\dots H_j)$. If $V_j\equiv W_{ac}$ for some $j$, then $(H_{j+1}\dots H_\ell)^{-1}$ is the history of an accepting computation $\pazocal{C}'$ of $W$ with $\ell(\pazocal{C}')<\ell$, contradicting the definition of $\ell$. 

Lemma \ref{M language} then implies that for all $1\leq j\leq \ell-1$, there exists $w_j\in\pazocal{L}$ such that $V_j\equiv I(w_j)$ or $J(w_j)$. Lemma \ref{M language} then provides a one-machine computation $\pazocal{D}$ accepting $V_2$. 

Let $H'$ be the history of $\pazocal{D}$. Then $(H_3\dots H_\ell)^{-1}H'$ is the history of an accepting computation of $W$ whose number of maximal one-machine computations is less than $\ell$, again contradicting the definition of $\ell$.

\end{proof}

\begin{lemma} \label{two one-machine J}

Let $W$ be an accepted configuration with $\ell(W)=2$ and set $\pazocal{C}\in A(W)$ such that $\ell(\pazocal{C})=2$. Factor the history $H$ of $\pazocal{C}$ as $H\equiv H_1H_2$, where each $H_i$ is the history of a one-machine computation. Then $H_j$ is the history of a one-machine computation of the $j$-th machine and $W\cdot H_1\equiv W_{ac}\cdot H_2^{-1}\equiv J(w)$ for some $w\in\pazocal{L}\setminus\{1\}$.

\end{lemma}

\begin{proof}

As in the proof of Lemma \ref{two one-machine}, $W\cdot H_1\equiv W_{ac}\cdot H_2^{-1}$ must be an accepted input configuration.

Note that the final rule of $H_1$ is $\theta(s)_i^{-1}$ and the first rule of $H_2$ is $\theta(s)_j$ for $i\neq j$. Lemma \ref{turn} then implies that $W\cdot H_1$ has empty `special' input sector, so that $W_{ac}\cdot H_2^{-1}\equiv J(w)$ for some $w\in\pazocal{L}$ by Lemma \ref{M language}.

Suppose $H_2$ is the history of a one-machine computation of the first machine. Then, since every rule of the first machine operates in parallel on the input sectors, $W_{ac}\cdot H_2^{-1}\equiv I(w)$. This implies $I(w)\equiv J(w)$, so that $w=1$. Lemma \ref{M language} provides a one-machine computation $\pazocal{D}$ of the second machine accepting $J(1)$. Let $H'$ be the history of $\pazocal{D}$. Then, the reduced form of $H_1H'$ is the history of a one-machine computation accepting $W$, contradicting the hypothesis.

Hence, $H_j$ is the history of a one-machine computation of the $j$-th machine.

If $w=1$, then there exists a one-machine computation $\pazocal{E}$ of the first machine accepting $I(w)$ by Lemma \ref{M language}. But then for $H''$ the history of $\pazocal{E}$, the reduced form of $H_1H''$ is the history of a one-machine computation accepting $W$. Thus, $w\in\pazocal{L}\setminus\{1\}$.

\end{proof}

\begin{lemma} \label{accepted configuration a-length}

Let $W$ be an accepted configuration. Then $|W(1)|_a\leq2|W(j)|_a$ for all $2\leq j\leq L$.

\end{lemma}

\begin{proof}

The symmetry of the rules implies that $|W(j)|_a$ is constant for $j\geq2$.

Let $\pazocal{C}$ be an accepting computation of $W$ with $\ell(\pazocal{C})=\ell(W)$. As the statement is obvious for $W\equiv W_{ac}$, we may assume that $\ell=\ell(\pazocal{C})\geq1$.

If $\pazocal{C}$ is a one-machine computation of the first machine, then $W(1)$ and $W(j)$ are coordinate shifts of one another, so that $|W(1)|_a=|W(j)|_a$. 

If it is a one-machine computation of the second machine, then the `special' input sector is empty while any other sector of $W(1)$ is a coordinate shift of the corresponding admissible subword of $W(j)$. So, in this case $|W(1)|_a\leq|W(j)|_a$.

Hence, we may assume $\ell=2$. Let $H\equiv H_1H_2$ be the factorization of the history of $\pazocal{C}$ provided by Lemma \ref{two one-machine J}, so that $W\cdot H_1\equiv J(w)$ for some $w\in\pazocal{L}\setminus\{1\}$.

In particular, the `special' input sector of $W\cdot H_1$ is empty while each of its other input sectors is not. So, since each rule of the first machine operates in parallel on the input sectors, $H_1$ cannot contain the letter $\theta(12)_1^{\pm1}$. 

Let $\pazocal{C}_1$ be the subcomputation with history $H_1$. As the application of a transition rule does not alter the $a$-length of any sector, we may assume that the step history of $\pazocal{C}_1$ is $(1)_1(s)_1^{-1}$.

Factor $H_1\equiv H_1'\theta(s)_1^{-1}$ and let $v$ be the natural copy of $H_1'$ read right to left in $F(\pazocal{A})$. Letting $\pazocal{C}_1'$ be the subcomputation with history $H_1'$, we may apply Lemma \ref{multiply one letter} to the restriction of $\pazocal{C}_1'$ to the subwords $P_0(i)Q_0(i)$ and $Q_0(i)P_1(i)$. It then follows that $W$ has (the natural copy of):

\begin{itemize}

\item $v$ written in the `special' input sector,

\item $wv$ written in every other input sector, and

\item $v^{-1}$ written in every $Q_0(i)P_1(i)$-sector.

\end{itemize}

As all other sectors of $W$ are empty, $|W(1)|_a=\|v\|+\|v^{-1}\|=2\|v\|\leq2(\|v^{-1}\|+\|wv\|)=2|W(j)|_a$.

\end{proof}

\begin{lemma} \label{M accepted configurations}

Let $W$ be an accepted configuration of $\textbf{M}$ and $\pazocal{C}:W\equiv W_0\to\dots\to W_t\equiv W_{ac}$ be an accepting computation with $\ell(\pazocal{C})=\ell(W)$. Then $t\leq c_2\|W(i)\|$ for all $i\in\{2,\dots,L\}$.

\end{lemma}

\begin{proof}

The statement is clear for $W\equiv W_{ac}$, so we may assume $\ell(W)\geq1$.

Suppose $\ell(W)=1$, so that $\pazocal{C}$ is a one-machine computation of the $j$-th machine. The history of $\pazocal{C}$ can then be factored as $H''H'\theta(a)_j$, where: 

\begin{itemize}

\item $H''$ is either empty or $\theta(s)_j$, and

\item $H'$ does not contain the letters $\theta(a)_j^{\pm1}$ or $\theta(s)_j^{\pm1}$.

\end{itemize}

Let $\pazocal{C}':W_0'\to\dots\to W_s'$ be the subcomputation with history $H'$ and let $\pazocal{C}_i'$ be its restriction to the base $\{t(i)\}B_3(i)$ for some $i\geq2$. We can then identify $\pazocal{C}_i'$ with a reduced computation of $\textbf{M}_4$ in the standard base. 

Lemma \ref{M_3 length} implies that $s\leq2c_1\max(\|W_0'(i)\|,\|W_s'(i)\|)$. As $W_s'$ is $\theta(a)_j$-admissible, $|W_s'|_a=0$, so that $\|W_0'(i)\|\geq\|W_s'(i)\|$. Further, since transition rules do not change the tape word of any sector, $\|W_0'(i)\|=\|W(i)\|$.

Hence, $t\leq2c_1\|W(i)\|+2\leq3c_1\|W(i)\|$. The statement then follows from the parameter choice $c_2>>c_1$.

So, we may assume that $\ell(W)=2$. Then, factor the history $H\equiv H_1H_2$ of $\pazocal{C}$ as in Lemma \ref{two one-machine J}. Let $\pazocal{C}_1:W_0\to\dots\to W_r$ be the subcomputation with history $H_1$. Then, $W_r\equiv J(w)$ for $w\in \pazocal{L}\setminus\{1\}$. 

As $w\neq1$ and the rules with step history $(1)_1$ operate in parallel on all input sectors, the step history of $\pazocal{C}_1$ must be $(1)_1(s)_1^{-1}$ (with perhaps empty maximal subcomputation with step history $(1)_1$). For $i\geq2$, applying Lemma \ref{multiply one letter} to the restriction of $\pazocal{C}_1$ to the $Q_0(i)P_1(i)$-sector then implies that $r-1\leq|W(i)|_a$. Further, a projection argument yields $|W(i)|_a\geq|W_r(i)|_a$.

The subcomputation $W_r\to\dots\to W_t$ is a one-machine computation accepting $W_r$, so that as above $t-r\leq3c_1\|W_r(i)\|$ for all $i\geq2$.

Thus, $t\leq3c_1\|W(i)\|+\|W(i)\|\leq4c_1\|W(i)\|$, so that the statement again follows from the parameter choice $c_2>>c_1$.

\end{proof}

\begin{lemma} \label{M width}

For any reduced computation $\pazocal{C}:W_0\to\dots\to W_t$ of $\textbf{M}$ in the standard base, $\|W_i\|\leq c_2\max(\|W_0\|,\|W_t\|)$ for all $0\leq i\leq t$.

\end{lemma}

\begin{proof}

As the application of a transition rule does not change the length of a configuration, we may assume without loss of generality that neither the first nor the last rule of the history $H$ of $\pazocal{C}$ is a transition rule.

\ \textbf{1.} Suppose $\pazocal{C}$ is a one-machine computation of the first machine. Then $H$ cannot contain the letters $\theta(s)_j^{\pm1}$ or $\theta(a)_j^{\pm1}$. 

So, for each $1\leq j\leq L$, the restriction $\pazocal{C}(j):W_0(j)\to\dots\to W_t(j)$ to $\{t(j)\}B_3(j)$ can be identified with a reduced computation of $\textbf{M}_4$ in the standard base. Lemma \ref{M_3 length} then implies $t\leq2c_1\max(\|W_0(j)\|,\|W_t(j)\|)$. 

Note that the application of any rule alters the $a$-length of any component by at most four. So, applying the same argument as used in the proof of Lemma \ref{M_1 width}, we get $\|W_i(j)\|\leq5c_1\max(\|W_0(j)\|,\|W_t(j)\|)$ for all $0\leq i\leq t$. So, for all $i$, 
\begin{align*}
\|W_i\|&=\sum_{j=1}^L\|W_i(j)\|\leq5c_1\sum_{j=1}^L\max(\|W_0(j)\|,\|W_t(j)\|)\leq5c_1\left(\sum_{j=1}^L\|W_0(j)\|+\sum_{j=1}^L\|W_t(j)\|\right) \\
&=5c_1(\|W_0\|+\|W_t\|)\leq10c_1\max(\|W_0\|,\|W_t\|)
\end{align*}
\ \textbf{2.} Suppose $\pazocal{C}$ is a one-machine computation of the second machine. If the step history of $\pazocal{C}$ does not contain the letter $(1)_2$, then the same argument as above implies $\|W_i\|\leq10c_1\max(\|W_0\|,\|W_t\|)$ for all $i$.

If the step history of $\pazocal{C}$ is $(1)_2$, then the restriction of $\pazocal{C}$ to any unlocked sector satisfies the hypotheses of Lemma \ref{multiply one letter}. As a result, it follows that $|W_i|_a\leq2\max(|W_0|_a,|W_t|_a)$ for all $i$.

So, we assume that the step history contains $(1)_2$ as a proper subword. Lemma \ref{M step history 1} then implies that any occurrence of $(1)_2$ is as the first or last letter of the step history. Let $W_r\to\dots\to W_s$ be the maximal subcomputation of $\pazocal{C}$ such that its step history has no occurrence of $(1)_2$. Then $\|W_i\|\leq10c_1\max(\|W_r\|,\|W_s\|)$ for all $r\leq i\leq s$.

If $\pazocal{C}_1:W_s\to\dots\to W_t$ is nonempty, then it has step history $(1)_2$ and $W_s$ is $\theta(12)_2$-admissible. Lemma \ref{primitive computations} then implies that $|W_s(j)|_a\leq\dots\leq|W_t(j)|_a$ for $j\geq2$. Further, Lemma \ref{multiply one letter} applies to the restriction of $\pazocal{C}_1$ to the $P_0(j)Q_0(j)$-sector for any $j\geq2$, implying $t-s\leq|W_t(j)|_a$. Then, for all $s\leq i\leq t$, Lemma \ref{multiply one letter} implies that $|W_i(1)|_a\leq|W_t(1)|_a+t-i\leq|W_t(1)|_a+|W_t(j)|_a$ for any $j\geq2$. So, $|W_i|_a\leq2|W_t|_a$ for all $s\leq i\leq t$.

Similarly, if $W_0\to\dots\to W_r$ is nonempty, then $|W_i|_a\leq2|W_0|_a$ for all $0\leq i\leq r$. Combining these inequalities yields $\|W_i\|\leq20c_1\max(\|W_0\|,\|W_t\|)$ for all $0\leq i\leq t$.

Hence, by the parameter choice $c_2>>c_1$, we may assume that $\pazocal{C}$ is a multi-machine computation. 

Moreover, we may induct on the number of maximal one-machine subcomputations of $\pazocal{C}$.

\smallskip

\ \textbf{3.} Suppose $W_0$ is not an accepted configuration.

As $\pazocal{C}$ is multi-machine, there exists $0<s<t$ such that $W_s$ is either a start or an end configuration.

If $W_s$ is an end configuration, then it must be $\theta(a)_j^{-1}$-admissible, so that $W_s\equiv W_{ac}$. But then $W_0$ must be an accepted configuration, contradicting our assumption.

So, $W_s$ must be an input configuration. Lemma \ref{turn} then implies that it has empty `special' input sector. If all input sectors are empty, then $W_s\equiv I(1)$, so that $W_0$ is an accepted configuration. As a result, $W_s$ must have a nonempty input sector.

Perhaps taking the inverse computation, we may assume without loss of generality that there exists a maximal one-machine computation $\pazocal{C}_1:W_r\to\dots\to W_s$ of the first machine. Since $W_s$ has empty `special' input sector, Lemmas \ref{multiply one letter} and \ref{M step history 2}(a) imply that the step history of $\pazocal{C}_1$ is $(1)_1(s)_1^{-1}$, so that $r=0$. Lemmas \ref{multiply one letter} and \ref{primitive computations} then imply that $|W_i|_a\leq|W_0|_a$ for all $i\leq s$.

As $W_s\to\dots\to W_t$ consists of one less maximal one-machine subcomputation than $\pazocal{C}$, the inductive hypothesis implies $\|W_i\|\leq c_2\max(\|W_s\|,\|W_t\|)$ for all $s\leq i\leq t$. This yields $\|W_i\|\leq c_2\max(\|W_0\|,\|W_t\|)$ for all $i$.

Hence, we may assume that $W_0$ is an accepted configuration.

\smallskip

\ \textbf{4.} Suppose there exists $w\in\pazocal{L}\setminus\{1\}$ such that $W_s\equiv I(w)$ or $J(w)$ for some $s\in\{0,\dots,t\}$.

If $s=0$ (resp $s=t$), then the first (resp last) rule of $H$ must be a transition rule of the form $\theta(s)_j^{\pm1}$. But this contradicts our assumption. Lemma \ref{turn} then implies that the `special' input sector of $W_s$ must be empty, so that $W_s\equiv J(w)$.

As a result, we may assume without loss of generality that there exists a maximal one-machine subcomputation $\pazocal{C}_1:W_r\to\dots\to W_s$ of the first machine. As in Step 3, Lemmas \ref{multiply one letter} and \ref{M step history 2}(a) then imply that $r=0$ and $|W_i|_a\leq|W_0|_a$ for all $0\leq i\leq s$.

But then the inductive hypothesis again yields $\|W_i\|\leq c_2\max(\|W_0\|,\|W_t\|)$ for all $i$.

\smallskip

\ \textbf{5.} Finally, suppose that for any $s\in\{0,\dots,t\}$ such that $W_s$ is a start (resp end) configuration, $W_s\equiv I(1)$ (resp $W_s\equiv W_{ac}$).

As $\pazocal{C}$ is not a one-machine computation, there exists such an $s$. Further, by the same reasoning as used in Step 4, we may assume that $s\in\{1,\dots,t-1\}$.

Then, $W_0\to\dots\to W_s$ and $W_s\to\dots\to W_t$ each consist of less maximal one-machine subcomputations than does $\pazocal{C}$, so that the inductive hypothesis implies $\|W_i\|\leq c_2\max(\|W_0\|,\|W_s\|)$ for all $0\leq i\leq s$ and $\|W_i\|\leq c_2\max(\|W_s\|,\|W_t\|)$ for all $s\leq i\leq t$.

But $|W_s|_a=0$, so that $\max(\|W_0\|,\|W_s\|)=\|W_0\|$ and $\max(\|W_s\|,\|W_t\|)=\|W_t\|$. Thus, the statement is proved.

\end{proof}

\newpage

%%%%%%%%%%%%%%%%%%%%%%%%%%%%%%%%%%%%%%%%%%%%%%%%%%%%%%%%%%%%%%%%%

\subsection{Computations of $\textbf{M}$ with long history} \

\begin{lemma} \label{M projected long history}

Let $\pazocal{C}:V_0\to\dots\to V_t$ be a reduced computation of $\textbf{M}$ with base $\{t(i)\}B_3(i)$ for some $i\in\{2,\dots,L\}$. Suppose $t>c_3\max(\|V_0\|,\|V_t\|)$. Then:

\begin{enumerate}

\item There exist accepted configurations $W_0$ and $W_t$ such that $W_0(i)\equiv V_0$ and $W_t(i)\equiv V_t$,

\item Let $\pazocal{C}_0$ and $\pazocal{C}_t$ be accepting computations of $W_0$ and $W_t$, respectively, with $\ell(\pazocal{C}_j)=\ell(W_j)$. For $H_j$ the history  $\pazocal{C}_j$, $\|H_0\|+\|H_t\|\leq t/500$ 

\item The sum of the lengths of all subcomputations of $\pazocal{C}$ whose step histories are of the form $(4n-2,4n-1)_j(4n-1)_j(4n-1,4n)_j$ or $(4n,4n-1)_j(4n-1)_j(4n-1,4n-2)_j$ is at least $0.98t$.

\end{enumerate}

\end{lemma}

\begin{proof}

For $H'$ the history of $\pazocal{C}$, factor $H'\equiv H'_1\dots H'_m$ for $m\geq1$ so that each $H'_j$ is the history of a maximal one-machine subcomputation of $\pazocal{C}$.

Suppose $m=1$, i.e $\pazocal{C}$ is a one-machine computation. Then the letters $\theta(s)_j^{\pm1}$ or $\theta(a)_j^{\pm1}$ can only occur in $H'$ as the first or last letter. So, other than perhaps these two rules, $\pazocal{C}$ can be viewed as a reduced computation of $\textbf{M}_4$ in the standard base. But then Lemma \ref{M_3 length} implies that $t-2\leq 2c_1\max(\|V_0\|,\|V_t\|)$, so that the parameter choice $c_3>>c_1$ provides a contradiction. So, we may assume that $m\geq2$, i.e $\pazocal{C}$ is a multi-machine computation.

For each $1\leq j\leq m$, let $\pazocal{C}_j:V_{y(j)}\to\dots\to V_{z(j)}$ be the subcomputation with history $H_j'$. By Lemma \ref{extending one-machine}, there exists a one-machine computation $\pazocal{C}_j':W_{y(j)}'\to\dots\to W_{z(j)}'$ in the standard base with history $H_j'$ and such that $W_x'(i)\equiv V_x$ for $y(j)\leq x\leq z(j)$. Note that for $1\leq j\leq m-1$, $W_{z(j)}'$ and $W_{y(j+1)}'$ are not necessarily the same configuration; to differentiate them, they are represented with these indices.

Suppose $W_{z(1)}'$ is not an accepted configuration. Since $V_{z(1)}\equiv V_{y(2)}$ is admissible for the last rule of $H_1'$ and the first rule of $H_2'$, $W_{z(1)}'$ must be a start or an end configuration. 

Note that the only end configuration that is admissible for any rule is $W_{ac}$, which is accepted. So, since $W_{z(1)}'$ is admissible for the last rule of $H_1'$, it must be a start configuration which is admissible for $\theta(s)_j$ for some $j$.

By construction, there exists $w\in F(\pazocal{A})$ such that $W_{z(1)}'\equiv I(w)$ (or $W_{z(1)}'\equiv J(w)$) if $\pazocal{C}_1$ is a one-machine computation of the first (or second) machine. By Lemma \ref{M language}, we must then have $w\notin\pazocal{L}$. Lemmas \ref{subword return to start} and \ref{subword M language} then imply that $m=2$. 

Suppose $\pazocal{C}_2$ is a one-machine computation of the $j$-th machine. Then as above, we can view this as a reduced computation of $\textbf{M}_4$ after disregarding the first rule. Lemmas \ref{M_3 start to end} and \ref{M_3 step history} then imply that the step history is a subword of $(s)_j(1)_j\dots(4n-1)_j$. So, the length of $\pazocal{C}_2$ must be at most $c_2\|V_t\|$ by Lemmas \ref{M_3 length} and \ref{M_3 input length}. 

Similarly, the length of $\pazocal{C}_1$ must be at most $c_2\|V_0\|$. 

So, $t\leq c_2(\|V_0\|+\|V_t\|)$, so that the parameter choice $c_3>>c_2$ yields a contradiction.

Hence, $W_{z(1)}'$ must be an accepted configuration. Consequently, $W_0\equiv W_{y(1)}'$ is an accepted configuration with $W_0(i)\equiv V_0$. 

By the construction of the one-machine computations in the standard base outlined in the proof of Lemma \ref{extending one-machine}, that $W_{z(1)}'$ is accepted implies that $W_{y(2)}'$ is accepted. As a result, $W_{z(2)}'$ is accepted. Continuing, we have that $W_t\equiv W_{z(m)}'$ is an accepted configuration with $W_t(i)\equiv V_t$.

Thus, (a) is satisfied.

By Lemma \ref{M accepted configurations}, we then have $\|H_0\|+\|H_t\|\leq 2c_2\max(\|V_0\|,\|V_t\|)\leq t/500$ as $c_3>>c_2$.

For $2\leq j\leq m-1$, let $\ell_j$ be the sum of the lengths of the subcomputations of $\pazocal{C}_j$ whose step histories are of the form given in (c). 

As above, for such $j$, by neglecting the first and last rules, we may view $\pazocal{C}_j$ as a reduced computation of $\textbf{M}_4$ in the standard base. So, Lemma \ref{M_3 input controlled} implies that $\ell_j\geq\left(1-\frac{1}{c_0}\right)(\|H_j'\|-2)$. As $\|H_j'\|\geq k$ by Lemma \ref{M_3 start to end}(c), taking $c_0$ sufficiently large then yields $\ell_j\geq0.99\|H_j'\|$.

Let $y(m)=y$, so that $\pazocal{C}_m:V_y\to\dots\to V_t$ is the subcomputation of $\pazocal{C}$ with history $H_m'$. Then, as $\pazocal{C}_m$ is a one-machine computation, Lemma \ref{M_3 length} implies $\|H_m'\|=t-y\leq c_2\max(\|V_y\|,\|V_t\|)$. 

If $\|V_y\|\leq c_2\|V_t\|$, then $\|H_m'\|\leq c_2^2\|V_t\|\leq t/200$ by the parameter choice $c_3>>c_2$.

Otherwise, $\|V_y\|>c_2\|V_t\|$. If $V_y$ is the $i$-th component of an end configuration, then it must be $\theta(a)_j^{-1}$-admissible. But then $V_y\equiv A(i)$, so that $\|V_y\|\leq\|V_t\|$. 

So, $V_y$ must be the $i$-th component of a start configuration. As above, we may choose this start configuration to be accepted, so that $V_y\equiv I(u^n,i)$ for some $u^n\in \pazocal{L}$.

Identifying $\pazocal{C}_m$ with a reduced computation of $\textbf{M}_4$, Lemma \ref{M_3 input length} then implies that its step history must contain the letter $(4n)_j$. Let $\pazocal{C}_m'':V_y\to\dots\to V_z$ be the maximal subcomputation whose step history is $(s)_j(1)_j\dots(4n)_j$ and $V_x\to\dots\to V_z$ be the maximal subcomputation with step history $(4n)_j$. So, by Lemmas \ref{M_3 input length} and \ref{M_3 start to end}, $|V_y|_a\leq9n|V_x|_a$ and $|V_x|_a=2\|u\|$.

If $z\neq t$, then the first letter of the step history of $V_z\to\dots\to V_t$ is either $(a)_j$ or $(4n,4n-1)_j$. Lemma \ref{M_3 end length} then implies that $|V_z|_a\leq12n|V_t|_a$, so that $|V_y|_a>c_1|V_z|_a$ and $t-z\leq t/200$. So, taking $c_1$ sufficiently large, $|V_z|_a\leq|V_x|_a$. Hence, Lemma \ref{multiply one letter} implies $z-x\leq|V_x|_a=2\|u\|$.

Let $\ell_m$ be the sum of the lengths of the subcomputations of $\pazocal{C}_m''$ whose step histories are of the form described in the statement.  By Lemmas \ref{M_3 start to end}(c) and \ref{multiply one letter}, we then have $$z-y-\ell_m\leq(c_0+1)(\|u\|+1)+2\|u\|+\|u^n\|+1\leq2c_0(\|u\|+1)$$ and $\ell_m=2k\|u\|+2k+1\geq2k(\|u\|+1)$. As in the proof of Lemma \ref{M_3 input controlled}, this then implies $\ell_m\geq\left(1-\frac{1}{c_0}\right)(z-y)\geq0.99(z-y)$.

We can then do the same for the subcomputation $\pazocal{C}_1:V_0\to\dots\to V_s$ with history $H_1'$, finding $r\in\{0,\dots,s\}$ with $r\leq t/200$ such that for $\ell_1$ the sum of the lengths of the subcomputations of $V_r\to\dots\to V_s$ whose step histories are of the form described in (c), $\ell_1\geq0.99(s-r)$.

Let $\ell$ be the sum of the lengths of the subcomputations of $V_r\to\dots\to V_z$ whose step histories are of the form given in (c). Then $\ell=\sum_{i=1}^m\ell_i\geq0.99(z-r)$ while $z-r\geq0.99t$. Thus, $\ell\geq0.98t$.

\end{proof}

\begin{lemma} \label{M projected long history controlled}

Let $\pazocal{C}:V_0\to\dots\to V_t$ be a reduced computation of $\textbf{M}$ with base $\{t(i)\}B_3(i)$ for some $i\in\{2,\dots,L\}$. If $t>c_3\max(\|V_0\|,\|V_t\|)$, then the history of any subcomputation $\pazocal{D}:V_r\to\dots\to V_s$ of $\pazocal{C}$ (or the inverse of $\pazocal{D}$) of length at least $0.4t$ contains a controlled subword.

\end{lemma}

\begin{proof}

By Lemma \ref{M projected long history}(c), the sum of the lengths of all subcomputations of $\pazocal{C}$ with step histories of the form $(4n-2,4n-1)_j(4n-1)_j(4n-1,4n)_j$ or $(4n,4n-1)_j(4n-1)_j(4n-1,4n-2)_j$ is at least $0.98t$. So, there exists such a subcomputation $\pazocal{C}'$ such that $\pazocal{D}$ contains a subcomputation $\pazocal{D}'$ which is also a subcomputation of $\pazocal{C}'$. Moreover, for $H'$ and $K'$ the histories of $\pazocal{C}'$ and $\pazocal{D}'$, we may assume $\|K'\|\geq0.3\|H'\|$.

But $\pazocal{C}'$ repeats $k$ copies of a controlled history (with an overlap of one rule), so that taking $k$ sufficiently large implies that $K'$ must contain a controlled subword.

\end{proof}

A two-letter subword $UV$ of the standard base of $\textbf{M}$ is defined to be left-active (resp right-active) with respect to the step $(i)_j$ if any rule of step history $(i)_j$ that alters the tape word of an admissible word with base $UV$ inserts/deletes one letter on the left (resp right) of the tape word. 

Note that if $(i)_j$ is not $(1)_2$ or the subword does not correspond to the `special' input sector, then $UV$ is left-active (resp right-active) with respect to $(i)_j$ if and only if the corresponding subword of the standard base of $\textbf{M}_4$ is left-active (resp right-active) for $\textbf{M}_4(i)$. Hence, the following is an immediate consequence of Lemma \ref{M_3 one-step}.

\begin{lemma} \label{M one-step}

Let $\pazocal{C}:W_0\to\dots\to W_t$ be a reduced computation of $\textbf{M}$ with step history $(i)_m$ and base $B\equiv\{t(x)\}B_3(x)$ for some $2\leq x\leq L$. Assume that for some index $j$, $|W_j|_a>4|W_0|_a$. Then there are subwords $U_\ell V_\ell$ and $U_rV_r$ of $B$ such that $U_\ell V_\ell$ is left-active with respect to $(i)_m$, $U_rV_r$ is right-active with respect to the step $(i)_m$, and for $W_0'\to\dots\to W_t'$ the restriction of $\pazocal{C}$ to either sector, $|W_j'|_a<|W_{j+1}'|_a<\dots<|W_t'|_a$.

\end{lemma}

\smallskip

%%%%%%%%%%%%%%%%%%%%%%%%%%%%%%%%%%%%%%%%%%%%%%%%%%%%%%%%%%%%%%%%%

\subsection{Reverted Bases} \

Let $B$ be the base of an admissible word $W$ of $\textbf{M}$. Then the \textit{reversion} of $B$, denoted $\pi(B)$, is the word obtained from $B$ by `forgetting' the coordinates of its letters. In this case, $\pi(B)$ is called the \textit{reverted base} of $W$. 

For example, the reverted base of any configuration is the concatenation of $L$ copies of the standard base of $\textbf{M}_4$. Similarly, if $B=Q_4(1)\{t(2)\}P_0(2)Q_0(2)P_1(2)Q_1(2)Q_1(2)^{-1}P_1(2)^{-1}$, then $$\pi(B)=Q_4\{t\}P_0Q_0P_1Q_1Q_1^{-1}P_1^{-1}$$

\begin{lemma} \label{reverted base}

Let $B$ be the base of an admissible word $W$ of $\textbf{M}$. Then there exists an admissible word $W'$ of $\textbf{M}_4$ with base $\pi(B)$ and such that $|W'|_a=|W|_a$. \newline
Moreover, if none of the state letters of $W$ are start or end letters (or their inverses), then $W'$ can be chosen to be the natural copy of $W$ in the hardware of $\textbf{M}_4$.

\end{lemma}

\begin{proof}

Let $W\equiv q_0w_1q_1\dots w_rq_r$, $B\equiv V_0\dots V_r$, and $\pi(B)\equiv U_0\dots U_r$.

Suppose $V_0\equiv V_1^{-1}$, so that $U_0\equiv U_1^{-1}$. Then the tape alphabet corresponding to the $V_0V_1$-sector is a copy of that corresponding to the $U_0U_1$-sector, so that there exists a natural copy $w_1'$ of $w_1$ in this alphabet.

Further, if neither $q_0$ nor $q_1$ corresponds to a start or end letter in $V_0$ or $V_1$, then we can choose $q_0'$ and $q_1'$ as the natural copies of $q_0$ and $q_1$ in $U_0$ and $U_1$, respectively. Otherwsise, either $q_0$ and $q_1$ both correspond to start letters or both correspond to end letters. In this case, we can choose $q_0'$ as any state letter of $U_0$ and $q_1'$ as its inverse.

Now suppose $V_0\neq V_1^{-1}$. Then since the tape alphabet of the $Q_s\{t\}$-sector of $\textbf{M}_4$ is empty, we again have that the tape alphabet corresponding to the $V_0V_1$-sector is a copy of that corresponding to the $U_0U_1$-sector. This allows us to construct $w_1'$ as a copy of $w_1$ as above.

Further, if neither $q_0$ nor $q_1$ corresponds to a start or end letter in $V_0$ or $V_1$, then we can again choose $q_0'$ and $q_1'$ as their natural copies in $U_0$ and $U_1$. Meanwhile, if $q_0$ (or $q_1$) corresponds to a start or end letter, then we can choose $q_0'$ (or $q_1'$) as any letter from $U_0$ (respectively $U_1$).

With $q_1'$ now chosen, we can apply the same construction to obtain $q_2'$ and $w_2'$.

Iterating, we construct the admissible word $W'\equiv q_0'w_1'q_1'\dots w_r'q_r'$ satisfying the statement.

\end{proof}

\newpage

\begin{lemma} \label{lifted rule}

Suppose $\pazocal{C}:W_0\to W_1$ is a one-rule computation of $\textbf{M}$ with history $\theta\in\Theta$, where $\theta\notin\{\theta(s)_i,\theta(a)_i:i=1,2\}^{\pm1}$. Further, suppose that either:

\begin{enumerate}

\item the step history of $\pazocal{C}$ is not $(1)_2$, or

\item the base of $\pazocal{C}$ does not contain a subword of the form $(P_0(1)Q_0(1))^{\pm1}$.

\end{enumerate}

Then there exists a one-rule computation of $\textbf{M}_4$ $\pazocal{C}':W_0'\to W_1'$ with history $\theta'$, where $\theta'$ is the natural copy of $\theta$ in $\Theta(\textbf{M}_4)$ and $W_0'$ and $W_1'$ are the natural copies of $W_0$ and $W_1$, respectively, in the hardware of $\textbf{M}_4$.

\end{lemma}

\begin{proof}

As $\theta\notin\{\theta(s)_i,\theta(a)_i:i=1,2\}^{\pm1}$, none of the state letters of $W_0$ or $W_1$ are start or end letters (or their inverses). So, applying Lemma \ref{reverted base}, we can find admissible words $W_0'$ and $W_1'$ which are the natural copies of $W_0$ and $W_1$, respectively, in the hardware of $\textbf{M}_4$.

Let $\theta'$ be the natural copy of $\theta$ in $\Theta(\textbf{M}_4)$. If (a) holds, then $\theta$ operates on each sector of the standard base of $\textbf{M}$ in the same way as $\theta'$ operates on the copy of the corresponding sector of the standard base of $\textbf{M}_4$. 

Conversely, if $\theta$ is a rule of step history $(1)_2$, then all sectors of the standard base of $\textbf{M}$ other than the `special' input sector are again operated on by $\theta$ in the same way as $\theta'$ operates on their copy. As $\theta$ locks the `special' input sector, Lemma \ref{locked sectors} implies that this sector is not present in $W_0$ if the base of $\pazocal{C}$ satisfies (b).

\end{proof}

The base $B$ of an admissible word of $\textbf{M}$ is called \textit{hyperfaulty} (or \textit{pararevolving}) if its reversion $\pi(B)$ is faulty (or revolving) as the base of an admissible word of $\textbf{M}_4$. Note that a base is hyperfaulty if and only if it is pararevolving and unreduced.

A hyperfaulty base is necessarily faulty, while a faulty base need not be hyperfaulty. For example, if
$B\equiv Q_0(3)Q_0(3)^{-1}P_0(3)^{-1}\{t(3)\}^{-1}\dots Q_4(1)^{-1}Q_4(1)\dots \{t(3)\}P_0(3)Q_0(3)$, where gaps correspond to strings of letters that follow the order of the standard base of $\textbf{M}$ or its inverse, then $B$ is faulty but not hyperfaulty.

Conversely, a pararevolving base that is not hyperfaulty (for example, $\{t(1)\}\dots\{t(2)\}$) is not revolving, while a revolving base that is not faulty (for example, $\{t(1)\}\dots\{t(1)\}$) is not pararevolving.

As the standard base of $\textbf{M}_4$ has length $11$, a pararevolving base has length at most $23$ while a revolving base of $\textbf{M}$ has length at most $22L+1$.

\begin{lemma} \label{one-machine hyperfaulty}

Suppose $\pazocal{C}:W_0\to\dots\to W_t$ is a one-machine computation of $\textbf{M}$ with hyperfaulty base $B$. Then for all $0\leq j\leq t$, $|W_j|_a\leq c_0\max(|W_0|_a,|W_t|_a)$.

\end{lemma}

\begin{proof}

As in the proof of Lemma \ref{M_4 faulty}, we may assume that $|W_r|_a>\max(|W_0|_a,|W_t|_a)$ for all $0<r<t$. In particular, we assume that neither the first nor the last letter of the history $H$ of $\pazocal{C}$ is a transition rule. 

As $\pazocal{C}$ is a one-machine computation, any occurrence of a letter of the form $\theta(s)_i^{\pm1}$ or $\theta(a)_i^{\pm1}$ in $H$ would have to be either the first or the last letter. So, no such letter occurs in $H$.

Suppose the hypotheses of Lemma \ref{lifted rule} are satisfied by each rule of $\pazocal{C}$. Then, we obtain the reduced computation $\pazocal{C}':W_0'\to\dots\to W_t'$ of $\textbf{M}_4$ with base $\pi(B)$ such that $|W_j'|_a=|W_j|_a$ for all $0\leq j\leq t$. But $\pi(B)$ is faulty, so that Lemma \ref{M_4 faulty} implies that $|W_j|_a\leq c_0\max(|W_0|_a,|W_t|_a)$ for all $0\leq j\leq t$.

So, it suffices to assume that $\pazocal{C}$ is a one-machine computation of the second machine, that the step history of $\pazocal{C}$ contains the letter $(1)_2$, and that $B$ contains a subword of the form $(P_0(1)Q_0(1))^{\pm1}$.

Suppose that the step history of $\pazocal{C}$ is $(1)_2$. Then the restriction of $\pazocal{C}$ to any two-letter subword of $B$ has fixed tape word, satisfies the hypotheses of Lemma \ref{multiply one letter}, or satisfies the hypotheses of Lemma \ref{unreduced base}. As in Step 1 of the proof of Lemma \ref{M_4 faulty}, this implies the statement for $c_0\geq2$.

So, $H$ (or its inverse) must contain a subword of the form $H_1\theta(12)_2H_2$, where $H_i$ is the history of a maximal subcomputation with step history $(i)_2$. As $\theta(12)_2$ locks every sector of the standard base other than those of the form $(Q_0(i)P_1(i))^{\pm1}$, any unreduced two-letter subword of $B$ must be of the form $Q_0(i)Q_0(i)^{-1}$ or $P_1(i)^{-1}P_1(i)$. As a result, $B$ must be a cyclic permutation of
$$P_0(1)Q_0(1)Q_0(1)^{-1}P_0(1)^{-1}\{t(1)\}^{-1}Q_4(L)^{-1}\dots Q_1(L)^{-1}P_1(L)^{-1}P_1(L)Q_1(L)\dots Q_4(L)\{t(1)\}P_0(1)$$
As $B$ contains the subword $R_2(L)Q_3(L)$, Lemma \ref{M_3 step history}(b) implies that the step history of $\pazocal{C}$ cannot contain a subword of the form $(12)_2(2)_2(21)_2$.

Let $\pazocal{C}_2:W_r\to\dots\to W_s$ be the subcomputation with history $H_2$ and let $\pazocal{C}_2':W_r'\to\dots\to W_s'$ be its restriction to the subword 
$$Q_3(L)^{-1}R_2(L)^{-1}Q_2(L)^{-1}R_1(L)^{-1}Q_1(L)^{-1}P_1(L)^{-1}P_1(L)Q_1(L)R_1(L)Q_2(L)R_2(L)Q_3(L)$$ 
Note that every rule with step history $(2)_2$ locks each sector of an admissible subword with base $\left(P_1(L)Q_1(L)R_1(L)Q_2(L)R_2(L)\right)^{\pm1}$. So, we may view these subwords as a single state letter. With this view, we may apply Lemma \ref{primitive unreduced} to $\pazocal{C}_2'$, so that $|W_r'|_a\leq\dots\leq|W_s'|_a$.

As any other sector has fixed tape word throughout $\pazocal{C}_2$, this yields $|W_r|_a\leq|W_s|_a$. So, $H$ must contain a subword $H_2\theta(23)_2H_3$, where $H_3$ is the history of a maximal subcomputation with step history $(3)_2$.

Note that the connecting rule of $\textbf{M}_4(3^-)$ locks the $Q_0P_1$-sector, so that it cannot occur in $H_3$. So, letting $\pazocal{C}_3:W_{s+1}\to\dots\to W_x$ be the subcomputation with history $H_3$, the restriction of $\pazocal{C}_3$ to the $P_1(L)Q_1(L)$-sector satisfies the hypotheses of Lemma \ref{multiply one letter}. As a result, $W_x$ cannot be $\theta(32)_2$-admissible. Hence, $x=t$ and $\pazocal{C}_3$ is the copy of computation of $\textbf{M}_4(3^-)$.

Letting $\pazocal{C}_3':W_{s+1}'\to\dots\to W_t'$ be the restriction of $\pazocal{C}_3$ to the subword $Q_1(L)^{-1}P_1(L)^{-1}P_1(L)Q_1(L)$, Lemma \ref{primitive unreduced} implies $|W_{s+1}'|_a\leq\dots\leq |W_t'|_a$. But the tape word in each other sector remains unchanged throughout $\pazocal{C}_3$, so that $|W_s|_a\leq|W_t|_a$.

\end{proof}

\begin{lemma} \label{hyperfaulty}

Let $\pazocal{C}:W_0\to\dots\to W_t$ be a reduced computation of $\textbf{M}$ with hyperfaulty base $B$. Then for all $0\leq j\leq t$, $|W_j|_a\leq c_0\max(|W_0|_a,|W_t|_a)$.

\end{lemma}

\begin{proof}

As in the proofs of Lemmas \ref{M_4 faulty} and \ref{one-machine hyperfaulty}, we assume that $|W_r|_a>\max(|W_0|_a,|W_t|_a)$ for all $0<r<t$, so that neither the first nor the last letter of the history $H$ of $\pazocal{C}$ is a transition rule.

By Lemma \ref{one-machine hyperfaulty}, it suffices to assume that $\pazocal{C}$ is a multi-machine computation. Further, as $B$ must be unreduced and the rules $\theta(a)_i^{\pm1}$ lock each sector of the standard base, the step history of $\pazocal{C}$ must contain a subword of the form $((s)_1^{-1}(s)_2)^{\pm1}$ by Lemma \ref{locked sectors}.

The only sectors of the standard base not locked by $\theta(s)_1^{\pm1}$ or $\theta(s)_2^{\pm1}$ are those of the form $P_0(i)Q_0(i)$ for $i\geq2$. As a result, any unreduced two-letter subword of $B$ must be of the form $P_0(i)P_0(i)^{-1}$ or $Q_0(i)^{-1}Q_0(i)$ for $i\geq2$.

As both $\theta(12)_1$ and $\theta(12)_2$ lock all $P_0(i)Q_0(i)$-sectors, the step history of $\pazocal{C}$ cannot contain the letters $(12)_j$ and  $(21)_j$.

So, if $B$ does not contain the letters $Q_0(i)^{\pm1}$, then no rule of $\pazocal{C}$ alters the $a$-length of an admissible word with base $B$. As a result, we may assume that $B$ contains such a letter.

By the definition of hyperfaulty, $B$ must contain a subword of the form $(Q_0(i)P_1(i))^{\pm1}$. Lemma \ref{M step history 2}(a) then implies that the step history of $\pazocal{C}$ cannot contain a subword of the form $(s)_j(1)_j(s)_j^{-1}$.

Hence, the step history of $\pazocal{C}$ (or its inverse) is $(1)_2(s)_2^{-1}(s)_1(1)_1$. Let $\pazocal{C}_1:W_s\to\dots\to W_t$ be the maximal subcomputation with step history $(1)_1$. 

Note that any subword of $B$ of the form $(P_0(i)Q_0(i))^{\pm1}$ is contained in a subword of a cyclic permutation of $B$ of the form $(P_0(i)Q_0(i)P_1(i))^{\pm1}$. Further, any subword of $B$ of the form $Q_0(i)^{-1}Q_0(i)$ is contained in a subword of a cyclic permutation of $B$ of the form $P_1(i)^{-1}Q_0(i)^{-1}Q_0(i)P_1(i)$.

So, since $W_s$ is $\theta(s)_1^{-1}$-admissible, Lemmas \ref{primitive computations} and \ref{primitive unreduced} imply $|W_s|_a\leq|W_t|_s$, contradicting our assumption.

\end{proof}

\begin{lemma} \label{M revolving}

Let $\pazocal{C}:W_0\to\dots\to W_t$ be a reduced computation with revolving base $B$. Then $\|W_j\|\leq c_4\max(\|W_0\|,\|W_t\|)$ for all $0\leq j\leq t$.

\end{lemma}

\begin{proof}

If $B$ is reduced, then the statement follows from Lemma \ref{M width} and the parameter choice $c_4>>c_2$. Similarly, if $B$ is hyperfaulty, then the statement follows from Lemma \ref{hyperfaulty} and the parameter choice $c_4>>c_0$.

So, it suffices to assume that $B$ is faulty but not hyperfaulty. As a result, $B$ contains a reduced pararevolving subword $B'$. Fix $i$ such that $B'$ contains a subword $(P_0(i)Q_0(i))^{\pm1}$. Since all non-input sectors are operated on in parallel across coordinates, we may assume that $B'$ is of the form $\{t(i)\}B_3(i)\{t(i+1)\}$ (where we take $L+1$ to be 1).

Let $\pazocal{C}':W_0'\to\dots\to W_t'$ be the restriction to $B'$ and suppose $i\neq1$. Since $B$ is unreduced, Lemma \ref{M controlled} implies that the history $H$ of $\pazocal{C}$ cannot contain a controlled subword. So, Lemma \ref{M projected long history} yields $t\leq c_3\max(\|W_0'\|,\|W_t'\|)\leq c_3\max(\|W_0\|,\|W_t\|)$. Hence, the statement follows from the parameter choice $c_4>>c_3$.

Further, if $i=1$ and the step history of $\pazocal{C}$ does not contain the letter $(1)_2$, then we may construct a coordinate shift of $\pazocal{C}'$, implying the statement in the same way.

So, we may assume that the step history of $\pazocal{C}$ contains the letter $(1)_2$ and that any reduced pararevolving subword of $B$ contains a subword $(P_0(1)Q_0(1))^{\pm1}$.

If the step history of $\pazocal{C}$ is $(1)_2$, then the restriction of $\pazocal{C}$ to any two-letter subword has fixed tape word, satisfies the hypotheses of Lemma \ref{multiply one letter}, or satisfies the hypotheses of Lemma \ref{unreduced base}. As in Step 1 of the proof of Lemma \ref{M_4 faulty}, this implies $\|W_j\|\leq2\max(\|W_0\|,\|W_t\|)$. 

So, $H$ must contain a transition rule.

As in previous proofs, we may assume that $\|W_r\|>\max(\|W_0\|,\|W_t\|)$ for all $0<r<t$, so that neither the first nor the last letter of $H$ is a transition rule (or $\chi$-rule or connecting rule).

Suppose the step history of $\pazocal{C}$ contains the letter $(12)_2$. Then, $H$ must have a subword of the form $H_1\theta(12)_2H_2$, where each $H_1$ (respectively $H_2$) is the history of a maximal subcomputation with step history $(1)_2$ (respectively $(2)_2$).

As $\theta(12)_2$ locks every sector of the standard base except for those of the form $Q_0(i)P_1(i)$, any unreduced two-letter subword of $B$ must be of the form $Q_0(i)Q_0(i)^{-1}$ or $P_1(i)^{-1}P_1(i)$. As a result, $B$ must be a cyclic permutation of
$$Q_0(2)Q_0(2)^{-1}\dots \{t(1)\}^{-1}Q_4(L)^{-1}\dots Q_1(L)^{-1}P_1(L)^{-1}P_1(L)Q_1(L)\dots Q_4(L)\{t(1)\}\dots Q_0(2)$$
Let $\pazocal{C}_2:W_r\to\dots\to W_s$ be the subcomputation of $\pazocal{C}$ with history $H_2$ and $\pazocal{C}_2':W_r'\to\dots\to W_s'$ be its restriction to the subword $$Q_3(L)^{-1}R_2(L)^{-1}Q_2(L)^{-1}R_1(L)^{-1}Q_1(L)^{-1}P_1(L)^{-1}P_1(L)Q_1(L)R_1(L)Q_2(L)R_2(L)Q_3(L)$$ 
As in the proof of Lemma \ref{one-machine hyperfaulty}, we may view the admissible subwords whose bases are of the form $(P_1(L)Q_1(L)R_1(L)Q_2(L)R_2(L))^{\pm1}$ as a single state letter. With this view, Lemma \ref{primitive unreduced} implies $|W_r'|_a\leq\dots\leq|W_s'|_a$.

Let $\pazocal{C}_2'':W_r''\to\dots\to W_s''$ be the restriction of $\pazocal{C}_2$ to a subword of the form $(\{t(1)\}B_3(1))^{\pm1}$. Then a projection argument implies $|W_r''|_a\leq\dots\leq|W_s''|_a$.

As any other sector has fixed tape word throughout $\pazocal{C}_2$, this yields $|W_r|_a\leq|W_s|_a$. So, we may assume $t>s$.

Since $B$ contains the subword $R_2(1)Q_3(1)$, Lemma \ref{M_3 step history}(b) implies that $H$ must have a subword $H_2\theta(23)_2H_3$, where $H_3$ is the history of a maximal subcomputation with step history $(3)_2$.

The connecting rule of $\textbf{M}_4(3^-)$ locks the $Q_0P_1$-sector, so that $H_3$ cannot contain an occurrence of a copy of this rule. So, since $B$ contains the subword $P_1(1)Q_1(1)$, Lemma \ref{multiply one letter} implies that $H_3$ must be a suffix of $H$.

Let $\pazocal{C}_3:W_{s+1}\to\dots\to W_t$ be the subcomputation with history $H_3$ and $\pazocal{C}_3':W_{s+1}'\to\dots\to W_t'$ be its restriction to a subword of the form $(Q_0(1)P_1(1)Q_1(1))^{\pm1}$. Then, Lemma \ref{primitive computations} implies $|W_{s+1}'|_a\leq\dots\leq |W_t'|_a$.

Letting $\pazocal{C}_3'':W_{s+1}''\to\dots\to W_t''$ be the restriction of $\pazocal{C}_3$ to the subword $Q_1(L)^{-1}P_1(L)^{-1}P_1(L)Q_1(L)$, Lemma \ref{primitive unreduced} implies $|W_{s+1}''|_a\leq\dots\leq|W_t''|_a$.

As any other sector has fixed tape word throughout $\pazocal{C}_3$, it follows that $|W_s|_a\leq|W_t|_a$, contradicting our assumption.

Hence, we may assume that the step history of $\pazocal{C}$ does not contain the letters $(12)_2$ or $(21)_2$. So, since the step history contains the letter $(1)_2$ and a letter corresponding to a transition rule, it must contain a letter of the form $(s)_2^{\pm1}$.

As $\theta(s)_2$ locks every sector of the standard base except for those of the form $P_0(i)Q_0(i)$ for $i\geq2$, $B$ must be a cyclic permutation of
$$P_0(2)P_0(2)^{-1}\dots\{t(1)\}^{-1}Q_4(L)^{-1}\dots P_1(L)^{-1}Q_0(L)^{-1}Q_0(L)P_1(L)\dots Q_4(L)\{t(1)\}\dots P_0(2)$$
As $B$ contains the subword $Q_0(L)P_1(L)$, Lemma \ref{M step history 2}(a) implies that the step history of $\pazocal{C}$ cannot contain a subword of the form $(s)_j(1)_j(s)_j^{-1}$. So, the step history of $\pazocal{C}$ (or its inverse) must be $(1)_2(s)_2^{-1}(s)_1(1)_1$. 

Let $\pazocal{C}_1:W_s\to\dots\to W_t$ be the maximal subcomputation of $\pazocal{C}$ with step history $(1)_1$ and $\pazocal{C}_1':W_s'\to\dots\to W_t'$ be its restriction to the subword $P_1(L)^{-1}Q_0(L)^{-1}Q_0(L)P_1(L)$. Then, Lemma \ref{primitive unreduced} implies $|W_s'|_a\leq\dots\leq|W_t'|_a$.

Letting $\pazocal{C}_1'':W_s''\to\dots\to W_t''$ be the restriciton of $\pazocal{C}_1$ to a subword of the form $(P_0(1)Q_0(1)P_1(1))^{\pm1}$, Lemma \ref{primitive computations} (or a projection argument) implies $|W_s''|_a\leq\dots\leq|W_t''|_a$.

As any other sector has fixed tape word throughout $\pazocal{C}_1$, $|W_s|_a\leq|W_t|_a$. Thus, we reach a final contradiction.

\end{proof}

\medskip

%%%%%%%%%%%%%%%%%%%%%%%%%%%%%%%%%%%%%%%%%%%%%%%%%%%%%%%%%%%%%%%%%

\section{Groups Associated to an $S$-machine and their Diagrams}

\subsection{The groups} \

As in previous literature (for example [16], [20], [23]), we now associate two finitely presented groups to a cyclic $S$-machine $\textbf{S}$. These groups are denoted $M(\textbf{S})$ and $G(\textbf{S})$ and `simulate' the work of $\textbf{S}$ in the precise sense described in Section 6.3.

Let $\textbf{S}$ be a cyclic recognizing $S$-machine with hardware $(Y,Q)$, where $Q=\sqcup_{i=0}^s Q_i$ and $Y=\sqcup_{i=1}^{s+1} Y_i$, and software the set of rules $\Theta=\Theta^+\sqcup\Theta^-$. For notational purposes, set $Q_0=Q_{s+1}$ and denote the accept word of $\textbf{S}$ by $W_{ac}$.

For $\theta\in\Theta^+$, Lemma \ref{simplify rules} allows us to assume that $\theta$ takes the form $$\theta=[q_0\to v_{s+1}q_0'u_1, \ q_1\to v_1q_1'u_2, \ \dots, \ q_{s-1}\to v_{s-1}q_{s-1}'u_s, \ q_s\to v_sq_s'u_{s+1}]$$ where $q_i,q_i'\in Q_i$, $u_i$ and $v_i$ are either empty or single letters in $Y_i^{\pm1}$, and some of the arrows may take the form $\xrightarrow{\ell}$. Note that if $\theta$ locks the $i$-th sector, then both $u_i$ and $v_i$ are necessarily empty.

Define $R=\{\theta_i: \theta\in\Theta^+,0\leq i\leq s\}$. For notational convenience, set $\theta_{s+1}=\theta_0$ for all $\theta\in\Theta^+$.

The group $M(\textbf{S})$ is then defined by taking the (finite) generating set $\pazocal{X}=Q\cup Y\cup R$ and subjecting it to the (finite number of) relations:

\begin{addmargin}[1em]{0em}

$\bullet$ $q_i\theta_{i+1}=\theta_i v_iq_i'u_{i+1}$ for all $\theta\in\Theta^+$ and $0\leq i\leq s$,

$\bullet$ $\theta_ia=a\theta_i$ for all $0\leq i\leq s$ and $a\in Y_i(\theta)$.

\end{addmargin}

As in the language of computations of $S$-machines, letters from $Q^{\pm1}$ are called \textit{$q$-letters} and those from $Y^{\pm1}$ are called \textit{$a$-letters}. Additionally, those from $R^{\pm1}$ are called \textit{$\theta$-letters}. The relations of the form $q_i\theta_{i+1}=\theta_iv_iq_i'u_{i+1}$ are called \textit{$(\theta,q)$-relations}, while those of the form $\theta_ia=a\theta_i$ are called \textit{$(\theta,a)$-relations}.

Note that the number of $a$-letters in any part of $\theta$, and so in any defining relation of $M(\textbf{S})$, is at most two.

To simplify these relations, it is convenient to omit reference to the indices of the letters of $R$. This notational quirk may make it appear as though $\theta$ commutes with the letters of $Y_i(\theta)$ and conjugates $q_i$ to $v_iq_i'u_{i+1}$ for each $i$; it should be noted that these statements are not strictly true. Further, it is useful to note that if $\theta$ locks the $i$-th sector, then $Y_i(\theta)=\emptyset$ so that $\theta$ has no relation with the elements of $Y_i$.

However, this group evidently lacks any reference to the accept configuration. To amend this, the group $G(\textbf{S})$ is constructed by adding one more relation to the presentation of $M(\textbf{S})$, namely the \textit{hub-relation} $W_{ac}=1$. In other words, $G(\textbf{S})\cong M(\textbf{S})/\gen{\gen{W_{ac}}}$.

Moreover, it is useful for the purposes of our construction to consider extra relations, called \textit{$a$-relations}, within the language of tape letters. If $\Omega$ is the set of relators defining these $a$-relations, then we denote the groups arising from the addition of $a$-relations by $M_\Omega(\textbf{S})$ and $G_\Omega(\textbf{S})$. Note that $M_\Omega(\textbf{S})\cong M(\textbf{S})/\gen{\gen{\Omega}}$ and $G_\Omega(\textbf{S})\cong G(\textbf{S})/\gen{\gen{\Omega}}$.

It is henceforth taken as an assumption that any $a$-relation adjoined to the groups associated to the machine $\textbf{M}$ correspond to words over the alphabet of the `special' input sector. 

For the purposes of Section 11 and the proof of Theorem \ref{main theorem}, the set of $a$-relators $\pazocal{S}$ is taken to be exactly the words that represent the trivial element in $B(\pazocal{A},n)$, where the tape alphabet of the sector is identified with $\pazocal{A}$. However, in the proof of Theorem \ref{CEP} presented in Section 13, the $a$-relators are taken to be a larger set of words. So, for the sake of generality, until Section 11, the set of $a$-relators $\Omega$ is taken to be some set of words over $\pazocal{A}$ containing $\pazocal{S}$ as a subset. 

Note that though they remain finitely generated, $M_\Omega(\textbf{S})$ and $G_\Omega(\textbf{S})$ may no longer be finitely presented. In fact, in all situations encountered in what follows, $M_\Omega(\textbf{M})$ and $G_\Omega(\textbf{M})$ are not finitely presented.

\smallskip

%%%%%%%%%%%%%%%%%%%%%%%%%%%%%%%%%%%%%%%%%%%%%%%%%%%%%%%%%%%%%%%%%

\subsection{Bands and annuli} \

Many of the arguments presented in the forthcoming sections rely on van Kampen diagrams (see Section 2.1) over the presentations of the groups introduced in Section 6.1. To present these arguments efficiently, we first differentiate between the types of edges and cells that arise in such diagrams in a way similar to that employed in [16] and [23]. 

For simplicity, we will often disregard the presence of $0$-cells in these diagrams. For example, we do not differentiate between adjacent edges, so that any edge not on the boundary of a diagram is on the boundary of two $\pazocal{R}$-cells (for $\pazocal{R}$ the defining relators of the corresponding group). Additionally, we will adopt the convention that the contour of any diagram, subdiagram, or cell is traced in the counterclockwise direction.

An edge labelled by a $q$-letter is called a \textit{$q$-edge}. Similarly, an edge labelled by an $a$-letter is called an \textit{$a$-edge} and one labelled by a $\theta$-letter is a \textit{$\theta$-edge}. 

For a path \textbf{p} in $\Delta$, the (combinatorial) length of $\textbf{p}$ is denoted $\|\textbf{p}\|$. Further, the path's \textit{$a$-length} $|\textbf{p}|_a$ is the number of $a$-edges in the path. The path's \textit{$\theta$-length} and \textit{$q$-length}, denoted $|\textbf{p}|_{\theta}$ and $|\textbf{p}|_q$, respectively, are defined similarly.

A cell whose contour label corresponds to a $(\theta,q)$-relation is called a \textit{$(\theta,q)$-cell}. Similarly, there are \textit{$(\theta,a)$-cells}, \textit{$a$-cells}, and \textit{hubs}.

In the general setting of a reduced diagram $\Delta$ over a presentation $\gen{X\mid\pazocal{R}}$, let $\pazocal{Z}\subseteq X$. For $m\geq1$, a sequence of (distinct) cells $\pazocal{B}=(\Pi_1,\dots,\Pi_m)$ in $\Delta$ is called a \textit{$\pazocal{Z}$-band} of length $m$ if:

\begin{itemize}

\item every two consecutive cells $\Pi_i$ and $\Pi_{i+1}$ have a common boundary edge $\textbf{e}_i$ labeled by a letter from $\pazocal{Z}^{\pm1}$ and

\item for every $i$, $\partial\Pi_i$ has exactly two edges labelled by a letter from $\pazocal{Z}^{\pm1}$, $\textbf{e}_{i-1}^{-1}$ and $\textbf{e}_i$, and $\text{Lab}(\textbf{e}_{i-1})$ and $\text{Lab}(\textbf{e}_i)$ are either both positive or both negative.

\end{itemize}

\begin{figure}[H]
\centering
\begin{subfigure}[b]{0.48\textwidth}
\centering
\raisebox{0.5in}{\includegraphics[scale=1.35]{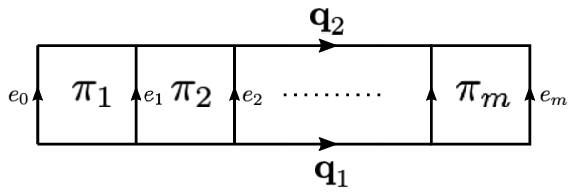}}
\caption{Non-annular $\pazocal{Z}$-band of length $m$}
\end{subfigure}\hfill
\begin{subfigure}[b]{0.48\textwidth}
\centering
\includegraphics[scale=1.35]{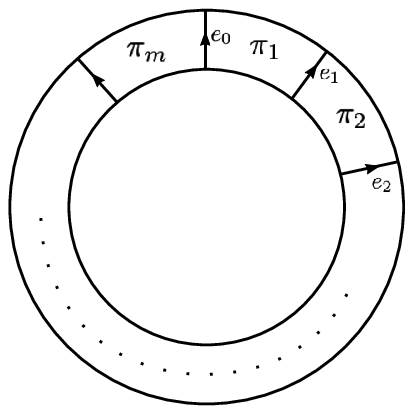}
\caption{Annular $\pazocal{Z}$-band of length $m$}
\end{subfigure}
\caption{ \ }
\end{figure}

For convenience, we extend this definition by saying that any edge labelled by a letter of $\pazocal{Z}^{\pm1}$ is a $\pazocal{Z}$-band of length zero.

A $\pazocal{Z}$-band is \textit{maximal} if it is not contained in any other $\pazocal{Z}$-band. Note that every edge labelled by a letter of  $\pazocal{Z}^{\pm1}$ is contained in a maximal $\pazocal{Z}$-band.

In a $\pazocal{Z}$-band $\pazocal{B}$ of length $m\geq1$ made up of the cells $(\Pi_1,\dots,\Pi_m)$, using only edges from the contours of $\Pi_1,\dots,\Pi_m$, there exists a closed path $\textbf{e}_0^{-1}\textbf{q}_1\textbf{e}_m\textbf{q}_2^{-1}$ such that $\textbf{q}_1$ and $\textbf{q}_2$ are simple (perhaps closed) paths. In this case, $\textbf{q}_1$ is called the \text{bottom} of $\pazocal{B}$, denoted $\textbf{bot}(\pazocal{B})$, while $\textbf{q}_2$ is called the \textit{top} of $\pazocal{B}$ and denoted $\textbf{top}(\pazocal{B})$. When $\textbf{q}_1$ and $\textbf{q}_2$ need not be distinguished, they are called the \textit{sides} of the band.

If $\textbf{e}_0=\textbf{e}_m$ in a $\pazocal{Z}$-band $\pazocal{B}$ of length $m\geq1$, then $\pazocal{B}$ is called a \textit{$\pazocal{Z}$-annulus}. If $\pazocal{B}$ is a non-annular $\pazocal{Z}$-band of length $m\geq1$, then $\textbf{e}_0^{-1}\textbf{q}_1\textbf{e}_m\textbf{q}_2^{-1}$ is called the \textit{standard factorization} of the contour of $\pazocal{B}$. If either $(\textbf{e}_0^{-1}\textbf{q}_1\textbf{e}_m)^{\pm1}$ or $(\textbf{e}_m\textbf{q}_2^{-1}\textbf{e}_0^{-1})^{\pm1}$ is a subpath of $\partial\Delta$, then $\pazocal{B}$ is called a \textit{rim $\pazocal{Z}$-band}.

A $\pazocal{Z}_1$-band and a $\pazocal{Z}_2$-band \textit{cross} if they have a common cell and $\pazocal{Z}_1\cap\pazocal{Z}_2=\emptyset$.

In particular, in a reduced diagram over the canonical presentations of the groups of interest, there exist \textit{$q$-bands} corresponding to bands arising from $\pazocal{Z}=Q_i^{\pm1}$ for some $i$, where every cell is a $(\theta,q)$-cell. Similarly, there exist \textit{$\theta$-bands} for $\theta\in\Theta^+$ and \textit{$a$-bands} for $a\in Y$. However, it is useful to restrict the definition of an $a$-band so that they consist only of $(\theta,a)$-cells.

Note that by definition, distinct maximal $q$-bands ($\theta$-bands, $a$-bands) cannot intersect.

Given an $a$-band $\pazocal{B}$, the makeup of the groups' relations dictates that the defining $a$-edges $\textbf{e}_0,\dots,\textbf{e}_m$ are labelled identically. Similarly, the $\theta$-edges of a $\theta$-band correspond to the same rule; however, the (suppressed) index of two such $\theta$-edges may differ.

If a maximal $a$-band contains a cell with an $a$-edge that is also on the contour of a $(\theta,q)$-cell, then the $a$-band is said to \textit{end} (or \textit{start}) on that $(\theta,q)$-cell and the corresponding $a$-edge is said to be the \textit{end} (or \textit{start}) of the band. This definition extends similarly, so that: 

\begin{itemize}

\item a maximal $a$-band can end on a $(\theta,q)$-cell, on an $a$-cell, or on the diagram's contour,

\item a maximal $\theta$-band can end only on the diagram's contour, and

\item a maximal $q$-band can end on a hub or on the diagram's contour.

\end{itemize}

Note that if a maximal $\theta$-band ($a$-band, $q$-band) ends as above in one part of the diagram, then it must also end in another part of the diagram as it cannot be a $\theta$-annulus ($a$-annulus, $q$-annulus).

\begin{figure}[H]
\centering
\includegraphics[scale=1.2]{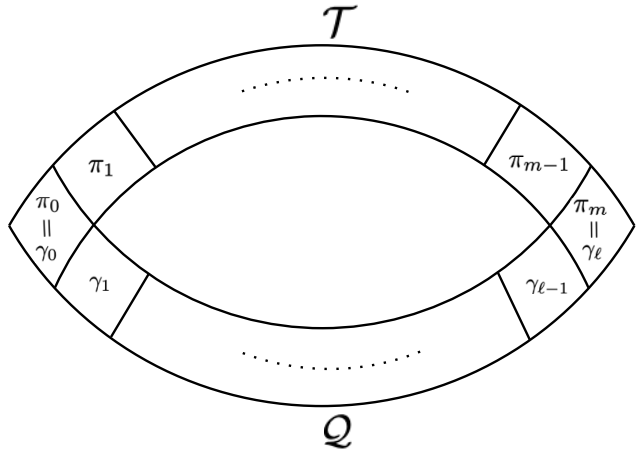}
\caption{$(\theta,q)$-annulus with defining $\theta$-band $\pazocal{T}$ and $q$-band $\pazocal{Q}$}
\end{figure}

The natural projection of the label of the top (or bottom) of a $q$-band onto $F(\Theta^+)$ is called the \textit{history} of the band; the \textit{step history} of the band is then defined in the obvious way. The natural projection (without reduction) of the top (or bottom) of a $\theta$-band onto the alphabet $\{Q_0,\dots,Q_s\}$ is called the \textit{base} of the band.

Let $\pazocal{T}$ be a maximal $\theta$-band in a reduced diagram $\Delta$ over $G_\Omega(\textbf{M})$ with two ends on $\partial\Delta$. Suppose that any cell between one side of $\pazocal{T}$ and $\partial\Delta$ is an $a$-cell. Then $\pazocal{T}$ is called a \textit{quasi-rim $\theta$-band}. Note that a rim $\theta$-band is a quasi-rim $\theta$-band.

Suppose the sequence of cells $(\pi_0,\pi_1,\dots,\pi_m)$ comprises a $\theta$-band and $(\gamma_0,\gamma_1,\dots,\gamma_\ell)$ a $q$-band such that $\pi_0=\gamma_0$, $\pi_m=\gamma_\ell$, and no other cells are shared. Suppose further that $\partial\pi_0$ and $\partial\pi_m$ both contain edges on the outer countour of the annulus bounded by the two bands. Then the union of these two bands is called a \textit{$(\theta,q)$-annulus} and $\pi_0$ and $\pi_m$ are called its \textit{corner} cells. A \textit{$(\theta,a)$-annulus} is defined similarly.

\smallskip

The following statement is proved in a more general setting in [15]:

\begin{lemma} \label{M(S) annuli}

\textit{(Lemma 6.1 of [15])} A reduced diagram over $M(\textbf{S})$ contains no:

\begin{enumerate}[label=({\arabic*})]

\item $(\theta,q)$-annuli

\item $(\theta,a)$-annuli

\item $a$-annuli

\item $q$-annuli

\item $\theta$-annuli

\end{enumerate}

\end{lemma}

As a result, in a reduced diagram $\Delta$ over $M(\textbf{S})$, if a maximal $\theta$-band and a maximal $q$-band (respectively $a$-band) cross, then their intersection is exactly one $(\theta,q)$-cell (respectively $(\theta,a)$-cell). Further, every maximal $\theta$-band and maximal $q$-band ends on $\partial\Delta$ in two places.

%\begin{lemma} \label{annuli lower bound}
%
%If $\Delta$ is a reduced diagram over $G_a(\textbf{S})$ and $S$ is a $(\theta,q)$-annulus (respectively a $(\theta,a)$-annulus) with boundary $q$-band (respectively $a$-band) $\pazocal{Q}$, then the length of $\pazocal{Q}$ is at least three.
%
%\end{lemma}
%
%\begin{proof}
%
%By the definition of the annulus, the history of the $q$-band must be of the form $\theta w\theta^{-1}$ for some $\theta\in\Theta$ and $w\in F(\Theta^+)$. If the length of $\pazocal{Q}$ is two, then this history is the unreduced word $\theta\theta^{-1}$, so that $\pazocal{Q}$ is a pair of cancellable $(\theta,q)$-cells.
%
%The argument is identical for a $(\theta,a)$-annulus by considering the label of the top (or bottom) of the corresponding $a$-band.
%
%\end{proof}

\smallskip

%%%%%%%%%%%%%%%%%%%%%%%%%%%%%%%%%%%%%%%%%%%%%%%%%%%%%%%%%%%%%%%%%

\subsection{Trapezia} \

Let $\Delta$ be a reduced diagram over the canonical presentation of $M(\textbf{S})$ whose contour is of the form $\textbf{p}_1^{-1}\textbf{q}_1\textbf{p}_2\textbf{q}_2^{-1}$, where $\textbf{p}_1$ and $\textbf{p}_2$ are sides of $q$-bands and $\textbf{q}_1$ and $\textbf{q}_2$ are maximal parts of the sides of $\theta$-bands whose labels start and end with $q$-letters. Then $\Delta$ is called a \textit{trapezium}.

In this case, $\textbf{q}_1$ and $\textbf{q}_2$ are called the \textit{bottom} and \textit{top} of the trapezium, respectively, while $\textbf{p}_1$ and $\textbf{p}_2$ are the \textit{left} and \textit{right} sides. Further, $\textbf{p}_1^{-1}\textbf{q}_1\textbf{p}_2\textbf{q}_2^{-1}$ is called the \textit{standard factorization} of the contour.

\begin{figure}[H]
\centering
\includegraphics[scale=1.25]{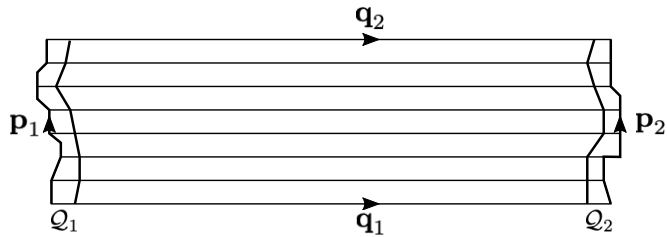}
\caption{Trapezium with side $q$-bands $\pazocal{Q}_1$ and $\pazocal{Q}_2$}
\end{figure}

The \textit{(step) history} of the trapezium is the (step) history of the rim $q$-band with $\textbf{p}_2$ as one of its sides and the length of this history is the trapezium's \textit{height}. The base of $\text{Lab}(\textbf{q}_1)$ is called the \textit{base} of the trapezium.

It's easy to see from this definition that a $\theta$-band $\pazocal{T}$ whose first and last cells are $(\theta,q)$-cells can be viewed as a trapezium of height 1 as long as its top and bottom start and end with $q$-edges. We extend this to all such $\theta$-bands by merely disregarding any $a$-edges of the top and bottom that precede the first $q$-edge or follow the final $q$-edge. The paths formed by disregarding these edges are called the \textit{trimmed} top and bottom of the band and are denoted $\textbf{ttop}(\pazocal{T})$ and $\textbf{tbot}(\pazocal{T})$.

\begin{figure}[H]
\centering
\includegraphics[scale=1.75]{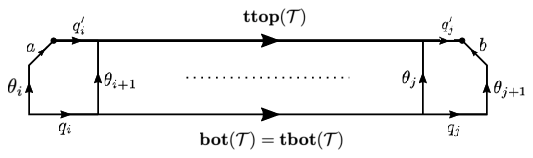}
\caption{$\theta$-band $\pazocal{T}$ with trimmed top}
\end{figure}

\begin{lemma} \label{theta-bands are one-rule computations}

Let $\pazocal{T}$ be a $\theta$-band in a reduced diagram $\Delta$ over the canonical presentation of $M(\textbf{S})$ whose first and last cells are $(\theta,q)$-cells. Then $\lab(\textbf{tbot}(\pazocal{T}))$ and $\lab(\textbf{ttop}(\pazocal{T}))$ are admissible words. Moreover, for $\theta$ the rule corresponding to the band $\pazocal{T}$, $\lab(\textbf{tbot}(\pazocal{T}))$ is $\theta$-admissible and $\lab(\textbf{tbot}(\pazocal{T}))\cdot\theta\equiv\lab(\textbf{ttop}(\pazocal{T}))$.

\end{lemma}

\begin{proof}

Suppose $\theta\in\Theta^+$. 

Further, suppose $\pazocal{T}$ consists of one $(\theta,q)$-cell $\pi$. Then $\textbf{top}(\pazocal{T})$ and $\textbf{bot}(\pazocal{T})$ contain just one $q$-edge, which is a part of $\partial\pi$ (or its inverse). So, $\textbf{ttop}(\pazocal{T})$ and $\textbf{tbot}(\pazocal{T})$ each consist of this one $q$-edge. It follows from the definition of $(\theta,q)$-relations that $\lab(\textbf{tbot}(\pazocal{T}))\cdot\theta\equiv\lab(\textbf{ttop}(\pazocal{T}))$.

Now suppose $\pazocal{T}$ contains at least two $(\theta,q)$-cells and let $\textbf{e}_1,\textbf{e}_2$ be the first two $q$-edges of $\textbf{bot}(\pazocal{T})$ with $q_1=\lab(\textbf{e}_1)$ and $q_2=\lab(\textbf{e}_2)$. So, $\lab(\textbf{tbot}(\pazocal{T}))$ has prefix $q_1wq_2$ for some $w\in F(Y)$. For $j=1,2$, let $\pi_j$ be the $(\theta,q)$-cell of $\pazocal{T}$ so that $\textbf{e}_j$ is an edge of $\partial\pi_j$.

For $0\leq i\leq s$, suppose $q_1\in Q_i$. Then the $i$-th part of $\theta$ must be $q_1\to u_iq_1'v_{i+1}$ for some $q_1'\in Q_i$, $u_i\in F(Y_i(\theta))$, and $v_{i+1}\in F(Y_{i+1}(\theta))$ with $\|u_i\|,\|v_{i+1}\|\leq1$. So, we have $\lab(\partial\pi_1)\equiv\theta_i^{-1}q_1\theta_{i+1}v_{i+1}^{-1}(q_1')^{-1}u_i^{-1}$. If there exists any cell of $\pazocal{T}$ between $\pi_1$ and $\pi_2$, it must be a $(\theta,a)$-cell with an edge labelled by $\theta_{i+1}$ on its contour. Hence, $w\in F(Y_{i+1}(\theta))$. 

What's more, the label of $\partial\pi_2$ must have a subword $\theta_{i+1}^{-1}q_2$. By the definition of the $(\theta,q)$-relations, this means one of two things:

\begin{enumerate}

\item $q_2\in Q_{i+1}$ and the $(i+1)$-th part of $\theta$ is $q_2\to u_{i+1}q_2'v_{i+2}$ for some $q_2'\in Q_{i+1}$, $u_{i+1}\in F(Y_{i+1}(\theta))$, and $v_{i+2}\in F(Y_{i+2}(\theta))$ with $\|u_{i+1}\|,\|v_{i+2}\|\leq1$; or

\item $q_2=q_1^{-1}$

\end{enumerate}

In case (a), the subword $q_1wq_2$ of $\lab(\textbf{tbot}(\pazocal{T}))$ satisfies condition (1) in the requirements for subwords of admissible words (see Section 3.1).

In case (b), it satisfies condition (2) as long as there is some $(\theta,a)$-cell between them; but this is required in the band, as otherwise $\pi_1$ and $\pi_2$ would be a pair of cancellable cells.

Let $\pazocal{T}_1=(\pi_1,\dots,\pi_2)$ be the corresponding subband of $\pazocal{T}$. Then $\lab(\textbf{tbot}(\pazocal{T}_1))\equiv q_1wq_2$. The above arguments make it clear that $q_1wq_2$ is $\theta$-admissible. Further, it is easy to see that $\lab(\textbf{ttop}(\pazocal{T}_1))\equiv (q_1wq_2)\cdot\theta$.

If $q_1\in Q_i^{-1}$, then an analogous argument yields the same conclusion.

If $\textbf{tbot}(\pazocal{T})$ has more than two $q$-edges, then the argument above can be iterated to apply to the whole band, implying the statement.

Conversely, if $\theta\in\Theta^-$, then the analogous arguments apply to $\textbf{ttop}(\pazocal{T})$ to show that $\lab(\textbf{ttop}(\pazocal{T}))$ is $\theta^{-1}$-admissible with $\lab(\textbf{ttop}(\pazocal{T}))\cdot \theta^{-1}\equiv \lab(\textbf{tbot}(\pazocal{T}))$. 

But then $\lab(\textbf{tbot}(\pazocal{T}))$ is $\theta$-admissible with $\lab(\textbf{tbot}(\pazocal{T}))\cdot\theta\equiv\lab(\textbf{ttop}(\pazocal{T}))$.

\end{proof}

\begin{lemma} \label{one-rule computations are theta-bands}

Let $U\to V$ be a computation of $\textbf{S}$ with history $H$ of length $1$, so that $H=\theta\in\Theta$. Then there exists a $\theta$-band $\pazocal{T}$ corresponding to the rule $\theta$ whose first and last cells are $(\theta,q)$-cells such that $\lab(\textbf{tbot}(\pazocal{T}))\equiv U$ and $\lab(\textbf{ttop}(\pazocal{T}))\equiv V$.

\end{lemma}

\begin{proof}

Suppose $\theta\in\Theta^+$ and set $U\equiv q_0^{\eps_0}w_1q_1^{\eps_1}\dots w_\ell q_\ell^{\eps_\ell}$ so that for each $0\leq i\leq \ell$, $q_i\in Q_{j(i)}$ for some $0\leq j(i)\leq s$ and $\eps_i\in\{\pm1\}$.

Then $q_i\in Q(\theta)$ for each $0\leq i\leq\ell$, so that the $j(i)$-th part of $\theta$ takes the form $q_i\to u_{j(i)}q_i'v_{j(i)+1}$ for some $q_i'\in Q_{j(i)}$, $u_{j(i)}\in F(Y_{j(i)}(\theta))$, and $v_{j(i)+1}\in F(Y_{j(i)+1}(\theta))$. So, there are relations of $M(\textbf{S})$ of the form $R_i=\theta_{j(i)}^{-1}q_i\theta_{j(i)+1}(u_{j(i)}q_i'v_{j(i)+1})^{-1}$ for all $i$.

If $\eps_i=1$, then each letter of $w_{i+1}$ is an element of $Y_{j(i)+1}(\theta)$ since $U$ is $\theta$-admissible, so that there are $(\theta,a)$-relations of the form $\theta_{j(i)+1}^{-1}a\theta_{j(i)+1}a^{-1}$ for each letter $a$ of $w_{i+1}$. Gluing along the edges labelled by $\theta_{j(i)+1}^{\pm1}$, one can construct a $\theta$-band $\pazocal{T}_{i+1}$ with contour label $\theta_{j(i)+1}^{-1}w_{i+1}\theta_{j(i)+1}w_{i+1}^{-1}$.

If $\eps_i=-1$, then each letter of of $w_{i+1}$ is in $Y_{j(i)}(\theta)$ since $U$ is $\theta$-admissible. So, there are relations of $M(\textbf{S})$ of the form $\theta_{j(i)}^{-1}a\theta_{j(i)}a^{-1}$ for each letter $a$ of $w_{i+1}$. Then, gluing along the edges labelled by $\theta_{j(i)}^{\pm1}$, one can construct a $\theta$-band $\pazocal{T}_{i+1}$ with contour label $\theta_{j(i)}^{-1}w_{i+1}\theta_{j(i)}w_{i+1}^{-1}$.

Now, let $\pi_i$ be a cell with boundary labelled by $R_i^{\eps_i}$. For either possibility of $\eps_i$, one can glue $\pazocal{T}_i$ and $\pazocal{T}_{i+1}$ to the left and right of $\pi_i$, respectively.

After $0$-refinement (or gluing) to cancel any adjacent edges with mutually inverse labels, this process produces a $\theta$-band $\pazocal{T}$ corresponding to the rule $\theta$ with $\lab(\textbf{tbot}(\pazocal{T}))\equiv U$. By the makeup of the band, it is easy to see that $\lab(\textbf{ttop}(\pazocal{T}))\equiv V$.

If $\theta\in\Theta^-$, then the same construction forms a $\theta$-band $\pazocal{T}$ corresponding to the rule $\theta^{-1}$ with $\lab(\textbf{tbot}(\pazocal{T}))\equiv V$ and $\lab(\textbf{ttop}(\pazocal{T}))\equiv U$. Taking the `inverse' of this band (i.e inverting the label of each cell) produces a $\theta$-band corresponding to $\theta$ as in the statement.

\end{proof}

By Lemma \ref{M(S) annuli}, any trapezium $\Delta$ of height $h\geq1$ can be decomposed into $\theta$-bands $\pazocal{T}_1,\dots,\pazocal{T}_h$ connecting the left and right sides of the trapezium, with $\textbf{bot}(\pazocal{T}_1)$ and $\textbf{top}(\pazocal{T}_h)$ making up the bottom and top of $\Delta$, respectively. Moreover, the first and last cells of each $\pazocal{T}_i$ are $(\theta,q)$-cells and $\textbf{ttop}(\pazocal{T}_i)=\textbf{tbot}(\pazocal{T}_{i+1})$ for all $1\leq i\leq h-1$.

The following two statements are clear from the previous two and exemplify how the group $M(\textbf{S})$ simulates the work of the $S$-machine:

\begin{lemma} \label{trapezia are computations}

Let $\Delta$ be a trapezium with history $H\equiv\theta_1\dots\theta_h$ for $h\geq1$ and maximal $\theta$-bands $\pazocal{T}_1,\dots,\pazocal{T}_h$ enumerated from bottom to top. If $U_j\equiv\lab(\textbf{tbot}(\pazocal{T}_j))$ and $V_j\equiv\lab(\textbf{ttop}(\pazocal{T}_j))$ for all $j$, then $H$ is a reduced word, $U_j$ and $V_j$ are admissible words, and $V_j\equiv U_j\cdot\theta_j$ for all $j$.

\end{lemma}

\begin{lemma} \label{computations are trapezia}

For any reduced computation $U\to\dots\to U\cdot H\equiv V$ of the $S$-machine $\textbf{S}$ with $\|H\|\geq1$, there exists a trapezium $\Delta$ with trimmed bottom label $U$, trimmed top label $V$, and history $H$.

\end{lemma}

\newpage

%%%%%%%%%%%%%%%%%%%%%%%%%%%%%%%%%%%%%%%%%%%%%%%%%%%%%%%%%%%%%%%%%

\section{Modified length and area functions}

\subsection{Modified length function} \

To assist with the proofs to come, we now modify the length function on words over the groups associated to an $S$-machine and paths in diagrams over their presentations. This is done in the same way as in [16] and [23]. The standard length of a word/path will henceforth be referred to as its \textit{combinatorial length} and the modified length simply as its \textit{length}.

Define a word consisting of no $q$-letters, one $\theta$-letter, and one $a$-letter as a \textit{$(\theta,a)$-syllable}. Then, define the length of:

\begin{itemize}

\item any $q$-letter as 1

\item any $\theta$-letter as 1

\item any $a$-letter as the parameter $\delta$ (as indicated in Section 3.3, this should be thought of as a very small positive number)

\item any $(\theta,a)$-syllable as 1

\end{itemize}

For a word $w$ over the generators of the canonical presentation of $G_\Omega(\textbf{S})$ (or any group associated to $\textbf{S}$), define a \textit{decomposition} of $w$ as a factorization of $w$ into a product of letters and $(\theta,a)$-syllables. The length of a decomposition of $w$ is then defined to be the sum of the lengths of the factors. 

Finally, the length of $w$, denoted $|w|$, is defined to be the minimum of the lengths of its decompositions.

The length of a path in a diagram over the presentations of the groups associated to $\textbf{S}$ is defined to be the length of its label.

The following gives some basic properties of the length function. Its proof is an immediate consequence of Lemma \ref{simplify rules}.

\begin{lemma} \label{lengths}

\textit{(Lemma 6.2 of [23])} Let \textbf{s} be a path in a diagram $\Delta$ over the canonical presentation of $G_\Omega(\textbf{S})$ (or any of the groups associated to $\textbf{S}$) consisting of $c$ $\theta$-edges and $d$ $a$-edges. Then:

\begin{enumerate}[label=({\alph*})]

\item $|\textbf{s}|\geq\max(c,c+(d-c)\delta)$

\item $|\textbf{s}|=c$ if $\textbf{s}$ is the top or a bottom of a $q$-band

\item For any product $\textbf{s}=\textbf{s}_1\textbf{s}_2$ of two paths in a diagram,
$$|\textbf{s}_1|+|\textbf{s}_2|-\delta\leq|\textbf{s}|\leq|\textbf{s}_1|+|\textbf{s}_2|$$

\item Let $\pazocal{T}$ be a $\theta$-band with base of length $l_b$. If $\textbf{top}(\pazocal{T})$ (or $\textbf{bot}(\pazocal{T})$) has $l_a$ $a$-edges, then the number of cells in $\pazocal{T}$ is between $l_a-l_b$ and $l_a+3l_b$.

\end{enumerate}

\end{lemma}

\smallskip

%%%%%%%%%%%%%%%%%%%%%%%%%%%%%%%%%%%%%%%%%%%%%%%%%%%%%%%%%%%%%%%%%

\subsection{Disks and weights} \

Next, we add extra relations to the groups $G(\textbf{S})$ and $G_\Omega(\textbf{S})$ that will aid with later estimates. This is done in the same way as in [16] and [23] (though no group $G_\Omega(\textbf{S})$ was present in those sources).

These relations, called \textit{disk relations}, are of the form $W=1$ for any configuration $W$ accepted by the machine $\textbf{S}$.

\begin{lemma} \label{disks are relations}

If the configuration $W$ is accepted by the machine $\textbf{S}$ and $Y_{s+1}=\emptyset$, then the word $W$ is trivial over the groups $G(\textbf{S})$ and $G_\Omega(\textbf{S})$.

\end{lemma}

\begin{proof}

Let $\pazocal{C}$ be an accepting computation of $W$ and $H$ be its history. By Lemma \ref{computations are trapezia}, there exists a trapezium $\Delta$ corresponding to $\pazocal{C}$ with trimmed bottom label $W$ and trimmed top label $W_{ac}$. 

As this is a computation of the standard base and every rule locks the $Q_sQ_0$-sector, one can further assume that no trimming was necessary in $\Delta$, i.e the labels of the bottom and top of $\Delta$ are $W$ and $W_{ac}$, respectively. Finally, it follows that the sides of the trapezium are labelled identically; specifically, they are labelled by the copy of $H$ obtained by adding the index $0$ to each letter. 

So, $W$ and $W_{ac}$ are conjugate in $M(\textbf{S})$. Taking into account the hub relation in both $G(\textbf{S})$ and $G_\Omega(\textbf{S})$ then implies the relation $W=1$.

\end{proof}

As a result of Lemma \ref{disks are relations}, the presentation obtained by adding the disk relations to the group $G(\textbf{S})$ (respectively $G_\Omega(\textbf{S})$) defines a group isomorphic to $G(\textbf{S})$ (respectively $G_\Omega(\textbf{S})$). The presentation containing disk relations will be referred to in what follows as the \textit{disk presentation} of the group $G(\textbf{S})$ (respectively $G_\Omega(\textbf{S})$). A cell of a diagram over the disk presentation corresponding to a disk relation (or its inverse) is referred to simply as a \textit{disk}.

One should note the following when considering diagrams over a disk presentation rather than diagrams over a canonical presentation:

\begin{itemize}

\item The disk presentation of $G(\textbf{S})$ or of $G_\Omega(\textbf{S})$ need not be finite. In particular, there may be infinitely many disk relations in this presentation. In particular, the disk presentations of $G(\textbf{M})$ and of $G_\Omega(\textbf{M})$ are not finitely presented.

\item For a word $w\in F(\pazocal{X})$ that represents the trivial element of $G(\textbf{S})$, the minimal area of diagrams over the disk presentation with contour label $w$ can be drastically different than that of diagrams over the canonical presentation of $G(\textbf{S})$.

\item As in Section 6.2, we insist that an $a$-band in a diagram over the disk presentation of $G_\Omega(\textbf{S})$ consist only of $(\theta,a)$-cells. As a consequence, a maximal $a$-band may end on a disk in addition to the other possibilities outlined in Section 6.2.

\end{itemize}

%\smallskip
%
%%%%%%%%%%%%%%%%%%%%%%%%%%%%%%%%%%%%%%%%%%%%%%%%%%%%%%%%%%%%%%%%%%
%
%\subsection{Weight of a diagram} \

Similar to how we modified the length function in Section 7.1, we now alter the definition of the area of a diagram over the disk presentations of $G(\textbf{S})$ and $G_\Omega(\textbf{S})$. 

We do this first by introducing a weight function on the cells of such diagrams, $\text{wt}$, defined by:

$
\begin{array}{ll}
      \bullet \ \text{wt}(\Pi)=1 & \ \text{if $\Pi$ is a $(\theta,q)$-cell or a $(\theta,a)$-cell} \\
      \bullet \ \text{wt}(\Pi)=C_1|\partial\Pi|^2 & \ \text{if $\Pi$ is a disk} \\
      \bullet \ \text{wt}(\Pi)=C_1\|\partial\Pi\|^2 & \ \text{if $\Pi$ is an $a$-cell}
   \end{array}
$

Naturally, we extend this to define the weight of a reduced diagram $\Delta$, $\text{wt}(\Delta)$, as the sum of the weights of its cells.

\smallskip

%%%%%%%%%%%%%%%%%%%%%%%%%%%%%%%%%%%%%%%%%%%%%%%%%%%%%%%%%%%%%%%%%

\subsection{Mixtures} \

We now recall an invariant of reduced diagrams over the relevant presentations, first introduced in [22], that will prove useful in future numerical estimates.

Let $O$ be a circle containing a finite two-colored set of points, with the two colors taken to be black and white. The circle $O$ is called a \textit{necklace} while the corresponding points are called \textit{white beads} and \textit{black beads}.

Let $P_j$ be the set of ordered pairs of distinct white beads, $(o_1,o_2)$, such that the counterclockwise simple arc on $O$ from $o_1$ to $o_2$ contains at least $j$ black beads.

Define $\mu_J(O)=\sum\limits_{j=1}^J \#P_j$ as the \textit{$J$-mixture} of $O$, where $J$ is the parameter specified in Section 3.3.

\begin{lemma} \label{mixtures}

\textit{(Lemma 6.1 of [22])} Let $O$ be a necklace with $x$ white beads and $y$ black beads.

\begin{enumerate}

\item $\mu_J(O)\leq J(x^2-x)$

\item If $O'$ is a necklace obtained from $O$ through the removal of one white bead, then for every $j$, $\# P_j-2x<\#P_j'\leq\#P_j$, and so $\mu_J(O)-2Jx<\mu_J(O')\leq\mu_J(O)$

\item If $O'$ is a necklace obtained from $O$ through the removal of one black bead, then for every $j$, $\#P_j'\leq\#P_j$, and so $\mu_J(O')\leq\mu_J(O)$

\item Suppose $v_1,v_2,v_3$ are three black beads on $O$ such that the counterclockwise arc from $v_1$ to $v_3$, $v_1 - v_3$, has at most $J$ black beads (excluding $v_1$ and $v_3$). Let $y_1$ and $y_2$ be the number of white beads on the counterclockwise arcs $v_1 - v_2$ and $v_2 - v_3$, respectively. If $O'$ is the necklace obtained from $O$ through the removal of $v_2$, then $\mu_J(O')\leq\mu_J(O)-y_1y_2$.

\end{enumerate}

\end{lemma}

%\begin{proof}
%
%$(a)$ The number of pairs of white beads that are distinct is $x(x-1)=x^2-x$. As each pair contributes at most $J$ to $\mu_J(O)$, the inequality is clear.
%
%$(b)$ is clear since any white bead can be in at most $2(x-1)$ ordered pairs in $P_j$.
%
%$(c)$ is clear since removing a black bead can only decrease the number of ordered pairs with at least $j$ black beads between them
%
%$(d)$ Each of the $y_1y_2$ pairs $(o_1,o_2)$ for $o_1$ on the arc $v_1-v_2$ and $o_2$ on the arc $v_2-v_3$ will be in $P_j$ but not $P_{j+1}$ for some $1\leq j\leq J$. So, after the removal of $v_2$, each of these pairs will no longer be in $P_j$, so that they will contribute one less to $\mu_J(O')$.
%
%\end{proof}

Let $\Delta$ be a reduced diagram over a group associated to an $S$-machine $\textbf{S}$. Let $O$ be a circle partitioned by subarcs labeled by the edges of $\partial\Delta$. At the midpoint of a subarc labeled by a $\theta$-edge (respectively a $q$-edge), place a white bead (respectively a black bead). Then, define the \textit{mixture on $\Delta$} $\mu(\Delta)$ as the $J$-mixture of the corresponding necklace, i.e $\mu(\Delta)=\mu_J(O)$.

\medskip

%%%%%%%%%%%%%%%%%%%%%%%%%%%%%%%%%%%%%%%%%%%%%%%%%%%%%%%%%%%%%%%%%

\section{Diagrams without disks}

\subsection{$M$-minimal diagrams} \

In this section, we study diagrams over $M_\Omega(\textbf{M})$, with the ultimate goal of bounding the `size' of such a diagram in terms of its perimeter. To do this, we first define a special class of diagrams for which this bound will hold.

A reduced diagram $\Delta$ over the canonical presentation of $M_\Omega(\textbf{M})$ is called \textit{$M$-minimal} if:

\begin{addmargin}[1em]{0em}

\begin{enumerate}[label=(MM{\arabic*})]

%\item it contains no $\theta$-annulus $S$ whose sides are labelled by letters of the tape alphabet of the `special' input sector,

\item for any $a$-cell $\pi$ and any $\theta$-band $\pazocal{T}$, at most half of the edges of $\partial\pi$ mark the start of an $a$-band that crosses $\pazocal{T}$, and

\item no maximal $a$-band ends on two different $a$-cells.

\end{enumerate}

\end{addmargin}

It follows immediately from this definition that a subdiagram of an $M$-minimal diagram is $M$-minimal.

\smallskip

%%%%%%%%%%%%%%%%%%%%%%%%%%%%%%%%%%%%%%%%%%%%%%%%%%%%%%%%%%%%%%%%%

\subsection{Annuli} \

Our first step is to rule out the possible existence of certain types of subdiagrams in an $M$-minimal diagram.

\begin{lemma} \label{M_a no annuli 1}

A reduced diagram $\Delta$ over $G_\Omega(\textbf{M})$ contains no:

\begin{addmargin}[1em]{0em}

(1) $(\theta,q)$-annuli

(2) $(\theta,a)$-annuli

(3) $a$-annuli

(4) $q$-annuli

\end{addmargin}

\end{lemma}

\begin{proof}

(1) Suppose $\Delta$ contains a $(\theta,q)$-annulus $S$. Let $\Delta_S$ be the subdiagram bounded by the outer component of the contour of $S$ and $\pazocal{Q}$ be the defining $q$-band (see Figure 8.1(a)).

By the definition of the annulus, the history $H$ of $\pazocal{Q}$ must be of the form $\theta w\theta^{-1}$ for some rule $\theta\in\Theta$ and some word $w\in F(\Theta^+)$. Since $H$ must be reduced, $w$ cannot be trivial, and so $\pazocal{Q}$ must contain a $(\theta,q)$-cell $\pi$ with neither $q$-edge on $\partial\Delta_S$.

Note that each cell of $\pazocal{Q}$ has a $\theta$-edge on its boundary that is shared with $\partial\Delta_S$. Indeed, all $\theta$-edges of $\partial\Delta_S$ arise in this way. 

\renewcommand{\thesubfigure}{\alph{subfigure}}
\begin{figure}[H]
\centering
\captionsetup[subfigure]{labelformat=parens}
\begin{subfigure}[b]{0.48\textwidth}
\centering
\raisebox{0.4in}{\includegraphics[scale=0.85]{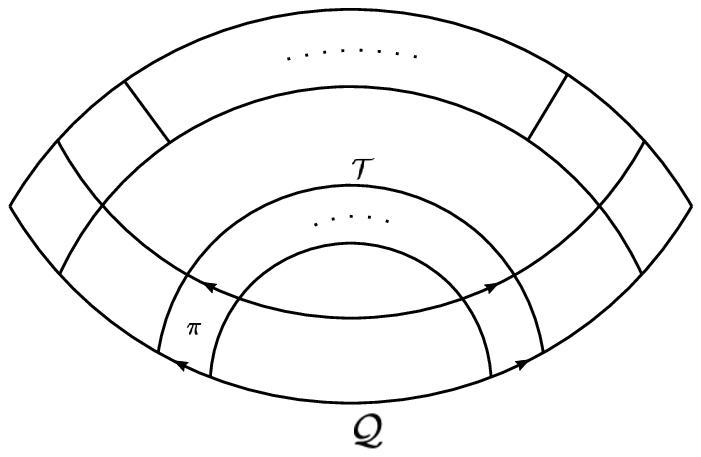}}
\caption{$\Delta_S$ for $S$ a $(\theta,q)$-annulus}
\end{subfigure}\hfill
\begin{subfigure}[b]{0.48\textwidth}
\centering
\includegraphics[scale=0.75]{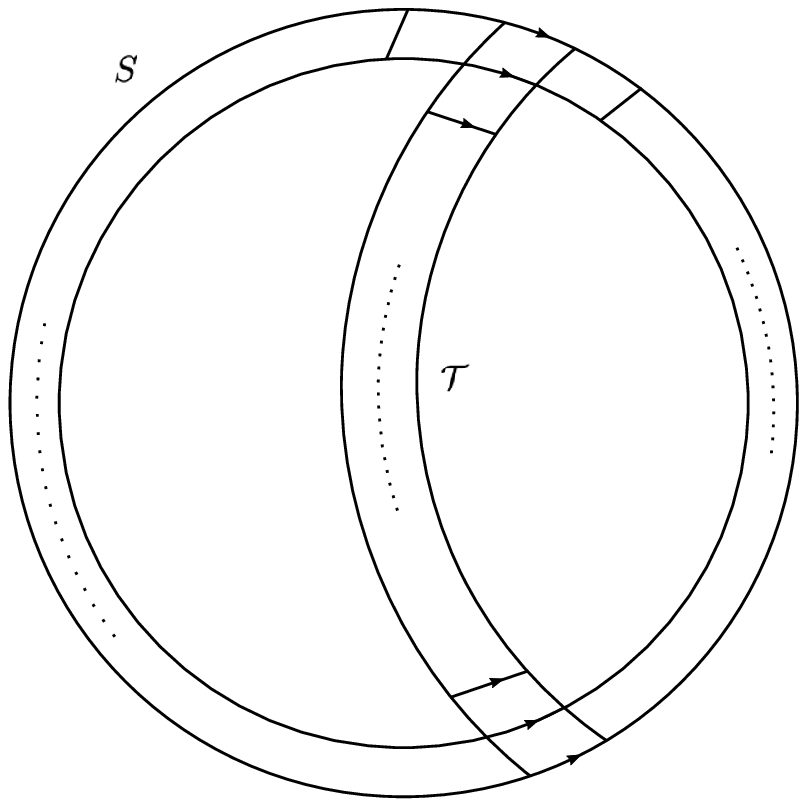}
\caption{$\Delta_S$ for $S$ an $a$-annulus}
\end{subfigure}
\caption{ \ }
\end{figure}

Let $\pazocal{T}$ be the maximal $\theta$-band in $\Delta_S$ containing $\pi$, so that $\pazocal{T}$ starts at the $\theta$-edge of $\partial\pi$ shared with $\partial\Delta_S$. Then $\pazocal{T}$ must also end on a $\theta$-edge of $\partial\Delta_S$, and so defines a $(\theta,q)$-annulus $S'$ with some subband of $\pazocal{Q}$. Note that the history of the $q$-band defining $S'$ is a proper subword of $w$.

Iterating, there exists a $\theta$-band that starts and ends on the boundary of adjacent cells of $\pazocal{Q}$. But then these two cells are cancellable, contradicting the assumption that $\Delta$ is reduced.

(2) is proved by an identical argument to (1).

(3) Suppose $\Delta$ contains an $a$-annulus $S$ and let $\Delta_S$ be the subdiagram bounded by the outer component of the contour of $S$ (see Figure 8.1(b)).

Recall that each cell comprising $S$ is a $(\theta,a)$-cell. By the definition of $(\theta,a)$-relations, each of these cells must have two $\theta$-edges on its boundary, one of which is shared with $\partial\Delta_S$. The maximal $\theta$-band $\pazocal{T}$ starting at such an edge must also end on $\partial\Delta_S$, i.e at a $\theta$-edge on the boundary of another cell of $S$.

But then $\pazocal{T}$ and a subband of $S$ form a $(\theta,a)$-annulus, contradicting (2).

%Suppose $\pi_1$ and $\pi_2$ are adjacent. Then the $a$-letters corresponding to the $(\theta,a)$-relations of $\pi_1$ and $\pi_2$ are the same while the $\theta$-letters must be mutually inverse. So, $\pi_1$ and $\pi_2$ are cancellable, contradicting the assumption that $\Delta$ is reduced. 
%
%Otherwise, let $S'$ be a $(\theta,a)$-annulus formed by $\pazocal{T}$ and a subband $\pazocal{B}$ of $S$. Further, let $\Delta_{S'}$ be the subdiagram of $\Delta_S$ bounded by the outer component of the contour of $S'$. As the side of $\pazocal{T}$ has no $\theta$-edges, any $\theta$-band of $\Delta_{S'}$ starts and ends on the side of $\pazocal{B}$, forming another $(\theta,a)$-annulus with a subband of $\pazocal{B}$. Continuing in this way, we find a $\theta$-band in $\Delta_S$ whose first and last edges are on the boundary of adjacent cells. As above, these adjacent cells must be cancellable, yielding a contradiction.

(4) As each cell comprising a $q$-annulus must be a $(\theta,q)$-cell, an identical argument to (3) produces a $(\theta,q)$-annulus which contradicts (1).

\end{proof}

\begin{figure}[H]
\centering
\includegraphics[scale=1.25]{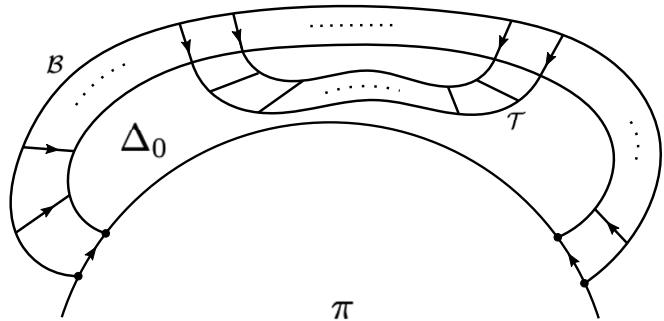}
\caption{$a$-band ending twice on an $a$-cell}
\end{figure}

\begin{lemma} \label{a-band on same a-cell}

For any $a$-cell $\pi$ in a reduced diagram $\Delta$ over $G_\Omega(\textbf{M})$, no $a$-band can have two ends on $\pi$.

\end{lemma}

\begin{proof}

Suppose $\pazocal{B}$ is an $a$-band ending twice on $\pi$. As $\lab(\partial\pi)$ is a reduced word, $\pazocal{B}$ must have nonzero length.

Consider the subdiagram $\Delta_0$ of $\Delta$ bounded by a side of $\pazocal{B}$ and the subpath of $\partial\pi$ whose initial and terminal edges correspond to the two ends of $\pazocal{B}$ (see Figure 8.2).

Since each cell of $\pazocal{B}$ is a $(\theta,a)$-cell, the portion of $\partial\Delta_0$ coinciding with a side of $\pazocal{B}$ is comprised entirely of $\theta$-edges. Moreover, as $\partial\pi$ is comprised entirely of $a$-edges, any $\theta$-edge of $\partial\Delta_0$ must lie on this side of $\pazocal{B}$.

So, a maximal $\theta$-band $\pazocal{T}$ of $\Delta_0$ starting on the side of $\pazocal{B}$ must also end on this side. But then $\pazocal{T}$ and a subband of $\pazocal{B}$ form a $(\theta,a)$-annulus in $\Delta$, contradicting Lemma \ref{M_a no annuli 1}(2).

\end{proof}

\begin{lemma} \label{M_a no annuli 2}

Let $\Delta$ be a reduced diagram over $M_\Omega(\textbf{M})$. 

\begin{enumerate}[label=({\arabic*})]

\item Suppose $\Delta$ contains a $\theta$-annulus $S$ and let $\Delta_S$ be the subdiagram of $\Delta$ bounded by the outer component of the contour of $S$. Then $\Delta_S$ contains no $(\theta,q)$-cells and $\lab(\partial\Delta_S)$ is a word over the tape alphabet of the `special' input sector.

\item If $\Delta$ is $M$-minimal, then it contains no $\theta$-annuli.

\end{enumerate}

\end{lemma}

\begin{proof}

(1) Suppose $\Delta_S$ contains a $(\theta,q)$-cell and let $\pazocal{Q}$ be the maximal $q$-band containing this cell. Lemma \ref{M_a no annuli 1}(4) then implies that $\pazocal{Q}$ must have two ends on $\partial\Delta_S$. But then $\pazocal{Q}$ and a subband of $S$ define a $(\theta,q)$-annulus in $\Delta$, contradicting Lemma \ref{M_a no annuli 1}(1). Hence, every cell of $S$ is a $(\theta,a)$-cell, so that $\partial\Delta_S$ consists entirely of $a$-edges.

Similarly, if a maximal $a$-band $\pazocal{B}$ of $\Delta_S$ has both ends on $\partial\Delta_S$, then $\pazocal{B}$ and a subband of $S$ define a $(\theta,a)$-annulus in $\Delta$, contradicting Lemma \ref{M_a no annuli 1}(2). So, every edge of $\partial\Delta_S$ is an $a$-edge marking the start of a maximal $a$-band in $\Delta_S$ which must end on an $a$-cell. 

Thus, as the boundary of an $a$-cell is labelled by tape letters from the `special' input sector and the $a$-edges of an $a$-band are labelled identically, the proof is complete.

\smallskip

(2) Suppose $\Delta$ contains a $\theta$-annulus. As $\theta$-bands cannot cross, the $\theta$-annuli of $\Delta$ are partially ordered by the relation: $$S'\leq S \text{ if } S' \text{ is contained in }\Delta_S$$ where $S$ and $S'$ are $\theta$-annuli in $\Delta$ and $\Delta_S$ is the subdiagram of $\Delta$ bounded by the outer contour of $S$. Since $\Delta$ is finite, it is clear that there exists a minimal $\theta$-annulus $T$ with respect to this partial order. 

Let $\Delta_T$ be the subdiagram of $\Delta$ bounded by the outer contour of $T$. If $\Delta_T\setminus T$ is empty, then $T$ must contain a pair of cancellable cells, contradicting the assumption that $\Delta$ is reduced. 

Suppose $\Delta_T\setminus T$ contains a $(\theta,a)$-cell $\pi$. Then, letting $T'$ be the maximal $\theta$-band of $\Delta_T$ containing $\pi$, $T'$ cannot cross $T$ and so must be a $\theta$-annulus. But then $T'<T$, contradicting the minimality of $T$.

Hence, by (1), any cell $\pi$ of $\Delta_T\setminus T$ must be an $a$-cell. Then, (MM2) and Lemma \ref{a-band on same a-cell} imply that every maximal $a$-band starting at an edge of $\partial\pi$ must either end on a cell of $T$ or cross $T$. But since (1) implies that $T$ consists entirely of $(\theta,a)$-cells, $\pi$ and $T$ form a counterexample to (MM1), contradicting the assumption that $\Delta$ is $M$-minimal.

\end{proof}

%%%%%%%%%%%%%%%%%%%%%%%%%%%%%%%%%%%%%%%%%%%%%%%%%%%%%%%%%%%%%%%%%

\subsection{Transpositions of a $\theta$-band with an $a$-cell} \

Let $\Delta$ be a reduced diagram over $G_\Omega(\textbf{M})$ containing an $a$-cell $\pi$ and a $\theta$-band $\pazocal{T}$ subsequently crossing some of the $a$-bands starting at $\pi$. As the cells shared by these bands and $\pazocal{T}$ are $(\theta,a)$-cells, the domain of the rule $\theta$ corresponding to $\pazocal{T}$ must be nonempty in the `special' input sector. So, by the definition of the rules of $\textbf{M}$, the domain of $\theta$ in this sector is the entire alphabet.

Suppose there are no other cells between $\pi$ and the bottom of $\pazocal{T}$, i.e there is a subdiagram formed by $\pi$ and $\pazocal{T}$.

Let $\textbf{s}_1$ be the maximal subpath of $\partial\pi$ so that each edge is on the boundary of a $(\theta,a)$-cell of $\pazocal{T}$. Further, let $\textbf{s}_2$ be the complement of $\textbf{s}_1$ in $\partial\pi$ so that $\partial\pi=\textbf{s}_1\textbf{s}_2$ and let $\pazocal{T}'$ be the subband of $\pazocal{T}$ satisfying $\textbf{bot}(\pazocal{T}')=\textbf{s}_1$.

\begin{figure}[H]
\centering
\begin{subfigure}[b]{0.48\textwidth}
\centering
\raisebox{0.3825in}{\includegraphics[scale=0.7]{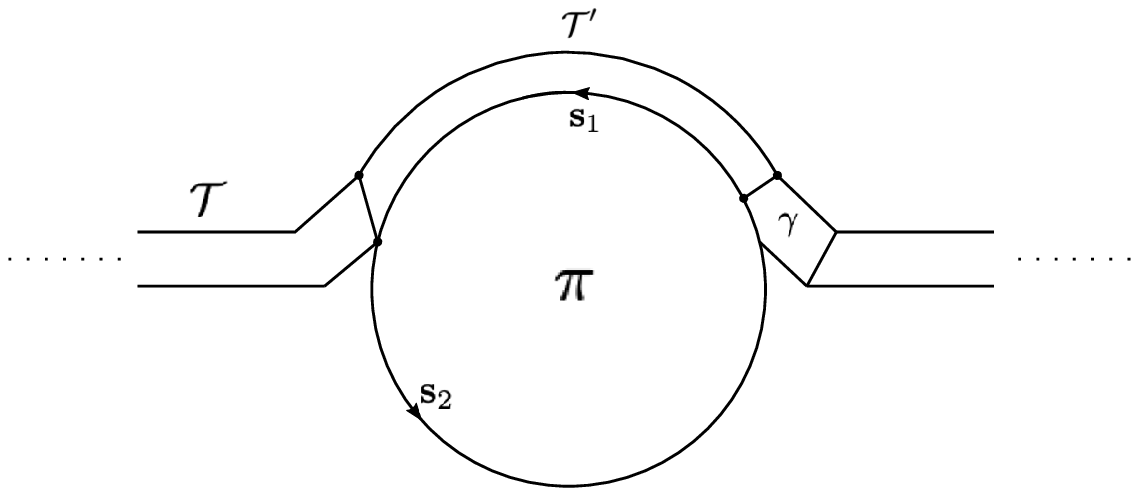}}
\caption{The subdiagram $\Gamma$}
\end{subfigure}\hfill
\begin{subfigure}[b]{0.48\textwidth}
\centering
\includegraphics[scale=0.7]{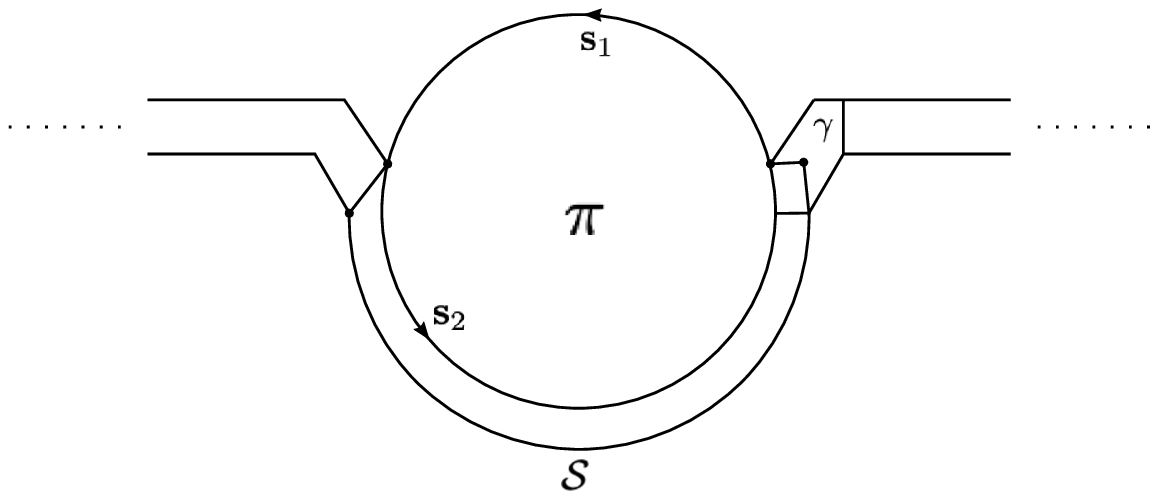}
\caption{The resulting subdiagram $\Gamma'$}
\end{subfigure}
\caption{The transposition of a $\theta$-band with an $a$-cell, $\gamma$ a $(\theta,q)$-cell}
\end{figure}

Let $V_1\equiv\text{Lab}(\textbf{s}_1)$ and $V_2\equiv\text{Lab}(\textbf{s}_2)$. Finally, let $\Gamma$ be the subdiagram formed by $\pi$ and $\pazocal{T}'$.

Then, we can construct the $\theta$-band $\pazocal{S}$ corresponding to $\theta$ consisting only of $(\theta,a)$-cells and with top label $V_2$. Let $\Gamma'$ be the subdiagram obtained by gluing a copy of $\pi$ to $\textbf{top}(\pazocal{S})$ in the clear way.

As $\text{Lab}(\textbf{top}(\pazocal{T}'))\equiv V_1^{-1}$, $\lab(\partial\Gamma)\equiv\lab(\partial\Gamma')$. So, we may replace the $\Gamma$ with $\Gamma'$, attaching the first and last cells of $\pazocal{S}$ to the complement of $\pazocal{T}'$ in $\pazocal{T}$ and making any necessary cancellations in the resulting band.

This process is called the \textit{transposition} of the $\theta$-band with the $a$-cell. 

Note that the diagram $\tilde{\Delta}$ resulting from the transposition has the same contour label as $\Delta$. Further, if a maximal $a$-band of $\Delta$ has one end on the $a$-cell $\pi$, then the other end is not changed by the transposition. 

Hence, if $\Delta$ is $M$-minimal, then $\tilde{\Delta}$ satisfies (MM2). However, $\tilde{\Delta}$ may not be $M$-minimal, as the transposed $\theta$-band may cross the maximal $a$-bands emanating from more than half of the $a$-edges on the boundary of the transposed $a$-cell.

Further, since the number of $(\theta,a)$-cells is altered by the transposition, the weight of the diagrams $\Delta$ and $\tilde{\Delta}$ may differ considerably. 

Despite these disadvantages, this process will prove valuable in forthcoming arguments.

\smallskip

%%%%%%%%%%%%%%%%%%%%%%%%%%%%%%%%%%%%%%%%%%%%%%%%%%%%%%%%%%%%%%%%%

\subsection{$a$-trapezia} \

We now generalize the concept of trapezium defined in Section 6.3 to the setting of $M$-minimal diagrams, allowing the existence of $a$-cells within the diagram.

To be specific, an \textit{$a$-trapezium} $\Delta$ is an $M$-minimal diagram with contour of the form $\textbf{p}_1^{-1}\textbf{q}_1\textbf{p}_2\textbf{q}_2^{-1}$, where each $\textbf{p}_i$ is the side of a $q$-band and each $\textbf{q}_i$ is the maximal subpath of the side of a $\theta$-band that starts and ends with $q$-edges. As with trapezia, the factorization $\textbf{p}_1^{-1}\textbf{q}_1\textbf{p}_2\textbf{q}_2^{-1}$ of the boundary is called the \textit{standard factorization} of $\partial\Delta$.

The \textit{history}, \textit{step history}, \textit{height}, and \textit{base} of an $a$-trapezium are defined in the same way they are defined for a trapezium. 

Note that the history of an $a$-trapezium must be reduced. Further, by Lemma \ref{theta-bands are one-rule computations}, the base of an $a$-trapezium must be the base of an admissible word. So, in an $a$-trapezium $\Delta$, the subdiagram $\Gamma$ bounded by two consecutive $q$-bands is an $a$-trapezium with base $UV$ corresponding to these $q$-bands' makeups. In this case, $\Gamma$ is called a \textit{$UV$-sector} in $\Delta$. As with admissible words, an $a$-trapezium may contain sectors of the same name.

\begin{lemma} \label{a-cells sector}

Suppose $\Delta$ is an $a$-trapezium containing an $a$-cell $\pi$. Then $\pi$ is contained in a $(P_0(1)Q_0(1))^{\pm1}$-, $P_0(1)P_0(1)^{-1}$-, or $Q_0(1)^{-1}Q_0(1)$-sector. Moreover, the step history of $\Delta$ must contain the letter $(1)_1$.

\end{lemma}

\begin{proof}

Let $\Delta'$ be the sector of $\Delta$ containing $\pi$. Enumerate the maximal $\theta$-bands of $\Delta'$ as $\pazocal{T}_1',\dots,\pazocal{T}_h'$. Then, there exists $j\in\{1,\dots,h-1\}$ such that $\pi$ sits between $\pazocal{T}_j'$ and $\pazocal{T}_{j+1}'$.

By (MM2) and Lemma \ref{a-band on same a-cell}, each edge of $\partial\pi$ marks the start of a maximal $a$-band that must end on $\partial\Delta$ or on a $(\theta,q)$-cell of one of the $q$-bands bounding $\Delta'$. So, such a band must cross $\pazocal{T}_j'$, cross $\pazocal{T}_{j+1}'$, or end on a $(\theta,q)$-cell in one of these bands.

Suppose an $a$-band ends on a $(\theta,q)$-cell of $\pazocal{T}_j'$. Then by Lemma \ref{theta-bands are one-rule computations}, the rule corresponding to $\pazocal{T}_j'$ must be of step history $(1)_1$ and the tape alphabet corresponding to $\Delta'$ must be the same as that of the `special' input sector. So, the base of $\Delta'$ must be of the form $(P_0(1)Q_0(1))^{\pm1}$ or $Q_0(1)^{-1}Q_0(1)$.

If an $a$-band ends on a $(\theta,q)$-cell of $\pazocal{T}_{j+1}'$, then the same conclusion may be reached.

So, we may assume that all maximal $a$-bands with one end on $\pi$ must cross either $\pazocal{T}_j'$ or $\pazocal{T}_{j+1}'$. Taking $n\geq3$, property (MM1) implies that there must exist such $a$-bands crossing each of these $\theta$-bands.

Lemma \ref{theta-bands are one-rule computations} then implies that the base of $\Delta'$ is of one of the forms in the statement. Moreover, the rules corresponding to $\pazocal{T}_j'$ and $\pazocal{T}_{j+1}'$ cannot lock the `special' input sector, so that the step history corresponding to each rule is either of the form $(s)_1^{\pm1}$ or $(1)_1$. As the rules cannot be mutually inverse, at least one contributes to an occurrence of $(1)_1$ in the step history of $\Delta'$.

\end{proof}

As a result of Lemmas \ref{a-cells sector}, \ref{theta-bands are one-rule computations}, and \ref{locked sectors}, if $\Delta$ is an $a$-trapezium with base $B$ containing at least one $a$-cell, then every unreduced two-letter subword of $B$ must be of the form:

\begin{enumerate}

\item $P_0(i)P_0(i)^{-1}$ or $Q_0(i)^{-1}Q_0(i)$, or

\item $Q_0(i)Q_0(i)^{-1}$ or $P_1(i)^{-1}P_1(i)$.

\end{enumerate}

Note the following immediate consequences of Lemmas \ref{theta-bands are one-rule computations} and \ref{locked sectors}:

\renewcommand{\labelenumi}{(\roman{enumi})}

\begin{enumerate}

\item if $B$ contains a subword of form (a), then $H$ cannot contain $\theta(12)_j^{\pm1}$ for $j=1,2$

\item if $B$ contains a subword of form (b), then $H$ cannot contain $\theta(s)_1^{\pm1}$ or a copy of the connecting rule of $\textbf{M}_4(3^-)$

\end{enumerate}

\smallskip

An $a$-trapezium is called \textit{standard} if its base is pararevolving and its history contains a controlled subword. Note that the subdiagram of a standard $a$-trapezium bounded by the $\theta$-bands corresponding to the controlled subword of the history is a trapezium. 

Further, an $a$-trapezium is called \textit{big} if its base is revolving, it contains $a$-cells, and it contains a subdiagram that is a standard trapezium. Note that the base of a big $a$-trapezium is necessarily reduced.

An $a$-trapezium is called \textit{exceptional} if it contains $a$-cells and its base $B$ is a cyclic shift of either:

\begin{itemize}

\item $Q_0(1)Q_0(1)^{-1}Q_0(1)$, or

\item $P_0(1)Q_0(1)Q_0(1)^{-1}P_0(1)^{-1}\{t(1)\}^{-1}Q_4(L)^{-1}\dots P_1(L)^{-1}P_1(L)\dots Q_4(L)\{t(1)\}P_0(1)$

\end{itemize}

where gaps correspond to strings of letters that follow the order of the standard base (or its inverse).

Note that the base of an exceptional $a$-trapezium is hyperfaulty. 
%Further, (i) and (ii) imply that the step history of an exceptional $a$-trapezium is $(1)_1$.

A \textit{partition} of an $a$-trapezium $\Delta$ is a (finite) collection of subdiagrams $\{\Delta_i\}_{i=1}^m$ such that each $\Delta_i$ consists of a number of sectors of $\Delta$, $\Delta_i\cap\Delta_j$ is either empty or a $q$-band for $i\neq j$, and each sector is a subdiagram of some $\Delta_i$. Note that $\|\textbf{tbot}(\Delta)\|=\sum_i\|\textbf{tbot}(\Delta_i)\|-m$ and, similarly, $\|\textbf{ttop}(\Delta)\|=\sum_i\|\textbf{ttop}(\Delta_i)\|-m$. Moreover, as $\textbf{tbot}(\Delta)$ and $\textbf{ttop}(\Delta)$ each have at least $m$ $q$-edges, $\sum_i\|\textbf{tbot}(\Delta_i)\|\leq2\|\textbf{tbot}(\Delta)\|$ and $\sum_i\|\textbf{ttop}(\Delta_i)\|\leq2\|\textbf{ttop}(\Delta)\|$.

Clearly, given a partition $\{\Delta_i\}$ of an $a$-trapezium $\Delta$, $\text{wt}(\Delta)\leq\sum_i\text{wt}(\Delta_i)$.

Let $\Delta$ be an $a$-trapezium with revolving base $B$ and let $B'$ be a cyclic permutation of $B$. Then, there exists an $a$-trapezium $\Delta'$ with revolving base $B'$ such that $\text{wt}(\Delta')=\text{wt}(\Delta)$. This diagram is constructed by cutting along a maximal $q$-band $\pazocal{Q}$ of $\Delta$, pasting together the left and right $q$-bands of $\Delta$, and pasting a copy of $\pazocal{Q}$ onto the side of the diagram. As with reduced computations, $\Delta'$ is called a cyclic permutation of $\Delta$.

Note that by Lemma \ref{lengths}(d), for any maximal $\theta$-band $\pazocal{T}$ in an $a$-trapezium $\Delta$, the length of $\pazocal{T}$ is at most $|\textbf{tbot}(\pazocal{T})|_q+3|\textbf{tbot}(\pazocal{T})|_a\leq 3\|\textbf{tbot}(\pazocal{T})\|$.

\begin{lemma} \label{revolving trapezium}

Let $\Delta$ be a trapezium with history $H$ and revolving base $B$. Then for $h=\|H\|$,
$$\text{wt}(\Delta)\leq 3c_4h\max(\|\textbf{tbot}(\Delta)\|,\|\textbf{ttop}(\Delta)\|)$$

\end{lemma}

\begin{proof}

Enumerate the maximal $\theta$-bands of $\Delta$ as $\pazocal{T}_1,\dots,\pazocal{T}_h$. Then, letting $\ell_i$ be the length of $\pazocal{T}_i$, the definition of weight implies $\text{wt}(\Delta)=\sum_i\ell_i\leq\sum_i 3\|\textbf{tbot}(\pazocal{T}_i)\|$.

By Lemmas \ref{computations are trapezia} and \ref{M revolving}, $\|\textbf{tbot}(\pazocal{T}_i)\|\leq c_4\max(\|\textbf{tbot}(\Delta)\|,\|\textbf{ttop}(\Delta)\|)$ for all $i$.

Hence, the statement follows.

\end{proof}

\begin{lemma} \label{a-trapezia hyperfaulty}

Let $\Delta$ be an $a$-trapezium with history $H$ and hyperfaulty base $B$. Then either $\Delta$ is exceptional or for $h=\|H\|$, $$\text{wt}(\Delta)\leq3c_0h\max(\|\textbf{tbot}(\Delta)\|,\|\textbf{ttop}(\Delta)\|)+C_1(|\textbf{tbot}(\Delta)|_a+|\textbf{ttop}(\Delta)|_a)^2$$

\end{lemma}

\begin{proof}

We proceed by cases:

\ \textbf{1.} Suppose $\Delta$ contains no $a$-cells. 

Then $\Delta$ is a trapezium, so that Lemma \ref{trapezia are computations} yields a corresponding reduced computation with a hyperfaulty base. For any maximal $\theta$-band $\pazocal{T}$ of $\Delta$, Lemma \ref{hyperfaulty} then implies $|\textbf{tbot}(\pazocal{T})|_a\leq c_0\max(|\textbf{ttop}(\Delta)|_a,|\textbf{tbot}(\Delta)|_a)$. Hence,
$$\text{wt}(\Delta)\leq3c_0h\max(\|\textbf{tbot}(\Delta)\|,\|\textbf{ttop}(\Delta)\|)$$

\smallskip

So, it suffices to assume that $\Delta$ contains at least one $a$-cell.

\ \textbf{2.} Suppose $B$ contains a subword $B'\equiv P_0(1)P_0(1)^{-1}$. 

The definition of hyperfaulty and Lemma \ref{a-cells sector} then imply that the $B'$-sector $\Delta'$ is the only sector containing $a$-cells.

By (i), the step history of $\Delta$ cannot contain $(12)_1^{\pm}$. Further, Lemmas \ref{theta-bands are one-rule computations} and \ref{locked sectors} imply that the step history cannot contain the letter $(s)_2^{\pm1}$. Hence, the step history must be a subword of $(s)_1(1)_1(s)_1^{-1}$.

Let $\Delta''$ be any sector of $\Delta$ other than $\Delta'$ and $\pazocal{T}''$ be a maximal $\theta$-band in $\Delta''$. Lemma \ref{trapezia are computations} yields a reduced computation $\pazocal{C}''$ corresponding to $\Delta''$ with the same history as $\Delta$. The maximal computation of $\pazocal{C}''$ with step history $(1)_1$ has fixed tape word, satisfies the hypotheses of Lemma \ref{multiply one letter} or satisfies the hypotheses of Lemma \ref{unreduced base}. In each case, $\|\textbf{tbot}(\pazocal{T}'')\|\leq\max(\|\textbf{tbot}(\Delta'')\|,\|\textbf{ttop}(\Delta'')\|)$, so that $\text{wt}(\Delta'')\leq3h\max(\|\textbf{tbot}(\Delta'')\|,\|\textbf{ttop}(\Delta'')\|)$.

By Lemma \ref{a-band on same a-cell} and (MM2), every maximal $a$-band starting on the boundary of an $a$-cell must have another end on $\textbf{tbot}(\Delta')$ or $\textbf{ttop}(\Delta')$. So, the sum of the combinatorial perimeters of all $a$-cells in $\Delta'$ is at most $|\textbf{tbot}(\Delta')|_a+|\textbf{ttop}(\Delta')|_a\leq|\textbf{tbot}(\Delta)|_a+|\textbf{ttop}(\Delta)|_a$.

Further, for $\pazocal{T}'$ a maximal $\theta$-band in $\Delta'$, every $a$-edge of $\textbf{bot}(\pazocal{T}')$ must be part of a maximal $a$-band which has at least one end on $\textbf{tbot}(\Delta')$ or $\textbf{ttop}(\Delta')$. So, $\|\textbf{tbot}(\pazocal{T}')\|\leq\|\textbf{tbot}(\Delta')\|+\|\textbf{ttop}(\Delta')\|$.

Combining these gives $\text{wt}(\Delta')\leq3h(\|\textbf{tbot}(\Delta')\|+\|\textbf{ttop}(\Delta')\|)+C_1(|\textbf{tbot}(\Delta)|_a+|\textbf{ttop}(\Delta)|_a)^2$.

The set of sectors $\{\Delta_i\}$ then form a partition of $\Delta$, so that
\begin{align*}
\text{wt}(\Delta)&\leq3h\sum(\|\textbf{tbot}(\Delta_i)\|+\|\textbf{ttop}(\Delta_i)\|)+C_1(|\textbf{tbot}(\Delta)|_a+|\textbf{ttop}(\Delta)|_a)^2 \\
&\leq6h(\|\textbf{tbot}(\Delta)\|+\|\textbf{ttop}(\Delta)\|)+C_1(|\textbf{tbot}(\Delta)|_a+|\textbf{ttop}(\Delta)|_a)^2 \\
&\leq12h\max(\|\textbf{tbot}(\Delta)\|,\|\textbf{ttop}(\Delta)\|)+C_1(|\textbf{tbot}(\Delta)|_a+|\textbf{ttop}(\Delta)|_a)^2
\end{align*}
The statement is then satisfied by the parameter choice $c_0\geq4$.

\smallskip

\ \textbf{3.} Suppose $B$ contains a subword $B'\equiv Q_0(1)^{-1}Q_0(1)$. 

Again, the $B'$-sector $\Delta'$ must be the only sector containing $a$-cells and the step history of $\Delta$ must be a subword of $(s)_1(1)_1(s)_1^{-1}$.

As $\Delta$ is not exceptional, $B$ must be a cyclic permutation of
$$P_1(1)^{-1}Q_0(1)^{-1}Q_0(1)P_1(1)\dots Q_4(1)\{t(2)\}P_0(2)P_0(2)^{-1}\{t(2)\}^{-1}Q_4(1)^{-1}\dots P_1(1)^{-1}$$
As in Step 2, we then have that $\text{wt}(\Delta'')\leq3h\max(\|\textbf{tbot}(\Delta'')\|,\|\textbf{ttop}(\Delta'')\|)$ for any sector $\Delta''$ other than $\Delta'$.

However, a maximal $a$-band in $\Delta'$ need not have one end on $\textbf{tbot}(\Delta')$ or $\textbf{ttop}(\Delta')$, as it may end on a $(\theta,q)$-cell of one of the bounding $q$-bands. Note that the rule corresponding to such a $(\theta,q)$-cell must be of step history $(1)_1$. Let $H_1$ be the maximal subword of $H$ comprising all letters with step history $(1)_1$ and $h_1=\|H_1\|$.

Then, Lemma \ref{a-band on same a-cell} and (MM2) imply that the sum of the combinatorial perimeters of all $a$-cells in $\Delta'$ is at most $|\textbf{tbot}(\Delta')|_a+|\textbf{ttop}(\Delta')|_a+2h_1$. Further, for $\pazocal{T}'$ a maximal $\theta$-band in $\Delta'$, $\|\textbf{tbot}(\pazocal{T}')\|\leq\|\textbf{tbot}(\Delta')\|+\|\textbf{ttop}(\Delta')\|+2h_1$.

Let $\Delta_1''$ be the $Q_0(1)P_1(1)$-sector of $\Delta$. By Lemma \ref{trapezia are computations}, there exists a reduced computation $\pazocal{C}_1''$ corresponding to $\Delta_1''$. The maximal subcomputation of $\pazocal{C}_1''$ with step history $(1)_1$ satisfies the hypotheses of Lemma \ref{multiply one letter}, so that $h_1\leq|\textbf{tbot}(\Delta_1'')|_a+|\textbf{ttop}(\Delta_1'')|_a$.

Similarly, letting $\Delta_2''$ be the $P_1(1)^{-1}Q_0(1)^{-1}$-sector of $\Delta$, $h_1\leq|\textbf{tbot}(\Delta_2'')|_a+|\textbf{ttop}(\Delta_2'')|_a$

So, the sum of the combinatorial perimeters of the $a$-cells in $\Delta'$ is at most 
$$|\textbf{tbot}(\Delta')|_a+|\textbf{ttop}(\Delta')|_a+\sum|\textbf{tbot}(\Delta_i'')|_a+|\textbf{ttop}(\Delta_i'')|_a\leq|\textbf{tbot}(\Delta)|_a+|\textbf{ttop}(\Delta)|_a$$
This implies that $\text{wt}(\Delta')\leq3h(\sum\|\textbf{tbot}(\pazocal{T}')\|)+C_1(|\textbf{tbot}(\Delta)|_a+|\textbf{ttop}(\Delta)|_a)^2$, where this sum is taken over all maximal $\theta$-bands $\pazocal{T}'$ in $\Delta'$.

Meanwhile, for $\pazocal{T}'$ a maximal $\theta$-band in $\Delta'$, $$\|\textbf{tbot}(\pazocal{T}')\|\leq\|\textbf{tbot}(\Delta')\|+\|\textbf{ttop}(\Delta')\|+\sum_{i=1}^2\|\textbf{tbot}(\Delta_i'')\|+\|\textbf{ttop}(\Delta_i'')\|$$
So, taking the set of sectors $\{\Delta_i\}$ as a partition of $\Delta$, we have
\begin{align*}
\text{wt}(\Delta)&\leq6h\sum(\|\textbf{tbot}(\Delta_i)\|+\|\textbf{ttop}(\Delta_i)\|)+C_1(|\textbf{tbot}(\Delta)|_a+|\textbf{ttop}(\Delta)|_a)^2 \\
&\leq12h(\|\textbf{tbot}(\Delta)\|+\|\textbf{ttop}(\Delta)\|)+C_1(|\textbf{tbot}(\Delta)|_a+|\textbf{ttop}(\Delta)|_a)^2 \\
&\leq24h\max(\|\textbf{tbot}(\Delta)\|,\|\textbf{ttop}(\Delta)\|)+C_1(|\textbf{tbot}(\Delta)|_a+|\textbf{ttop}(\Delta)|_a)^2
\end{align*}
so that the statement follows for $c_0\geq8$.

\smallskip

\ \textbf{4.} By Steps 2 and 3, $B$ must contain a subword of the form $(P_0(1)Q_0(1))^{\pm1}$. So, $B$ must be a cyclic permutation of
$$P_0(1)Q_0(1)Q_0(1)^{-1}P_0(1)^{-1}\{t(1)\}^{-1}Q_s(L)^{-1}\dots P_1(L)^{-1}P_1(L)\dots Q_s(L)\{t(1)\}P_0(1)$$
But we may assume that $\Delta$ contains an $a$-cell by Step 1, so that $\Delta$ is exceptional.

\end{proof}

\begin{lemma} \label{a-trapezia revolving}

Let $\Delta$ be an $a$-trapezium with history $H$ and revolving base $B$. If $\Delta$ is neither big nor exceptional, then for $h=\|H\|$,
$$\text{wt}(\Delta)\leq 3C_1h\max(\|\textbf{tbot}(\Delta)\|,\|\textbf{ttop}(\Delta)\|)+C_2(\|\textbf{tbot}(\Delta)\|+\|\textbf{ttop}(\Delta)\|)^2$$

\end{lemma}

\begin{proof}

By Lemma \ref{a-trapezia hyperfaulty}, we may assume that $B$ is not hyperfaulty. As a result, $B$ must contain a reduced pararevolving subword.

Further, Lemma \ref{revolving trapezium} allows us to assume that $\Delta$ contains an $a$-cell.

\ \textbf{1.} Suppose $B$ contains a reduced pararevolving subword $B'$ such that $B'$ has a subword of the form $(P_0(i)Q_0(i))^{\pm1}$ for some $i\geq2$. 

Let $\Delta'$ be the maximal subdiagram of $\Delta$ which is an $a$-trapezium with base $B'$. By Lemma \ref{a-cells sector}, $\Delta'$ is a trapezium.

Then, Lemma \ref{trapezia are computations} yields a reduced computation $\pazocal{C}':V_0'\to\dots\to V_h'$ with base $B'$ corresponding to $\Delta'$. By the parallel nature of the rules, we may assume that the base of this computation is $\{t(i)\}B_3(i)\{t(i+1)\}$.

If $h>c_3\max(\|V_0'\|,\|V_h'\|)$, then we may apply Lemma \ref{M projected long history} to $\pazocal{C}'$, so that its history must contain a controlled subword. But then Lemma \ref{M controlled} implies that $\Delta$ is a big $a$-trapezium.

So, $h\leq c_3\max(\|V_0'\|,\|V_h'\|)\leq c_3\max(\|\textbf{tbot}(\Delta)\|,\|\textbf{ttop}(\Delta)\|)$.

Let $\pazocal{T}$ be a maximal $\theta$-band in $\Delta$. By Lemma \ref{a-band on same a-cell} and (MM2), any $a$-edge of $\textbf{tbot}(\pazocal{T})$ is part of a maximal $a$-band which must have at least one end on $\textbf{tbot}(\Delta)$, on $\textbf{ttop}(\Delta)$, or on a maximal $q$-band of $\Delta$. 

By the definition of revolving, there are at most $22L+1$ maximal $q$-bands in $\Delta$, each of which consists of $h$ $(\theta,q)$-cells. Further, each such $(\theta,q)$-cell has at most two $a$-edges on its contour. So, $|\textbf{tbot}(\pazocal{T})|_a\leq|\textbf{tbot}(\Delta)|_a+|\textbf{ttop}(\Delta)|_a+(44L+2)h$.

As $C_1$ is chosen after $L$ and $c_3$, this implies $\|\textbf{tbot}(\pazocal{T})\|\leq C_1\max(\|\textbf{tbot}(\Delta)\|,\|\textbf{ttop}(\Delta)\|)$.

Similarly, any maximal $a$-band of $\Delta$ with one end on an $a$-cell has its other end on $\textbf{tbot}(\Delta)$, on $\textbf{ttop}(\Delta)$, or on a maximal $q$-band corresponding to a letter $Q_0(1)^{\pm1}$. As there are at most two such $q$-bands bounding a `special' input sector, the sum of the combinatorial perimeters of $a$-cells must be at most $|\textbf{tbot}(\Delta)|_a+|\textbf{ttop}(\Delta)|_a+2h\leq(2c_3+1)(\|\textbf{tbot}(\Delta)\|+\|\textbf{ttop}(\Delta)\|)$.

Hence, $\text{wt}(\Delta)\leq3C_1h\max(\|\textbf{tbot}(\Delta)\|,\|\textbf{ttop}(\Delta)\|)+C_1(2c_3+1)^2(\|\textbf{tbot}(\Delta)\|+\|\textbf{ttop}(\Delta)\|)^2$. The parameter choices $C_2>>C_1>>c_3$ then imply the statement.

Thus, we may assume that $B$ is faulty but not hyperfaulty and that every reduced pararevolving subword of $B$ contains a subword of the form $(P_0(1)Q_0(1))^{\pm1}$. As a result, $B$ has exactly two sectors corresponding to the `special' input sector, each of which is of this form.

As $\Delta$ must contain $a$-cells, Lemma \ref{a-cells sector} implies that its step history must contain the letter $(1)_1$.

\ \textbf{2.} Suppose the step history of $\Delta$ is $(1)_1$. 

For any sector $\Delta''$ not containing $a$-cells, the corresponding computation must have fixed tape word, satisfy the hypotheses of Lemma \ref{multiply one letter}, or satisfy the hypotheses of Lemma \ref{unreduced base}. So, for any maximal $\theta$-band $\pazocal{T}''$ of $\Delta''$, $\|\textbf{tbot}(\pazocal{T}'')\|\leq\max(\|\textbf{tbot}(\Delta'')\|,\|\textbf{ttop}(\Delta'')\|)$. Consequently, $\text{wt}(\Delta'')\leq3h\max(\|\textbf{tbot}(\Delta'')\|,\|\textbf{ttop}(\Delta'')\|)$.

Let $\Delta'$ be a sector containing $a$-cells and $\pazocal{T}'$ be a maximal $\theta$-band of $\Delta'$. 

Then any $a$-edge of $\textbf{tbot}(\pazocal{T}')$ is part of a maximal $a$-band which must end on $\textbf{tbot}(\Delta')$, on $\textbf{ttop}(\Delta')$, or on the maximal $q$-band corresponding to $Q_0(1)^{\pm1}$. This then implies that $|\textbf{tbot}(\pazocal{T}')|_a\leq|\textbf{tbot}(\Delta')|_a+|\textbf{ttop}(\Delta')|_a+h$. Similarly, the sum of the combinatorial perimeters of the $a$-cells in $\Delta'$ is at most $|\textbf{tbot}(\Delta')|_a+|\textbf{ttop}(\Delta')|_a+h$.

So, $\text{wt}(\Delta')\leq6h\max(\|\textbf{tbot}(\Delta')\|,\|\textbf{ttop}(\Delta')\|)+3h^2+C_1(|\textbf{tbot}(\Delta')|_a+|\textbf{ttop}(\Delta')|_a+h)^2$.

As $B$ must contain a reduced pararevolving subword, it must have a subword $B_0''$ of the form $(Q_0(i)P_1(i))^{\pm1}$. Let $\Delta_0''$ be the $B_0''$-sector of $\Delta$. Then the corresponding computaton satisfies the hypotheses of Lemma \ref{multiply one letter}, so that $h\leq|\textbf{tbot}(\Delta_0'')|_a+|\textbf{ttop}(\Delta_0'')|_a$.

Thus, letting $\{\Delta_i\}$ be the partition of $\Delta$ given by its sectors, we have
\begin{align*}
\text{wt}(\Delta)&\leq6h\sum(\|\textbf{tbot}(\Delta_i)\|+\|\textbf{ttop}(\Delta_i)\|)+C_1(|\textbf{tbot}(\Delta)|_a+|\textbf{ttop}(\Delta)|_a)^2 \\
&\leq12h(\|\textbf{tbot}(\Delta)\|+\|\textbf{ttop}(\Delta)\|)+C_1(|\textbf{tbot}(\Delta)|_a+|\textbf{ttop}(\Delta)|_a)^2 \\
&\leq24h\max(\|\textbf{tbot}(\Delta)\|,\|\textbf{ttop}(\Delta)\|)+C_1(|\textbf{tbot}(\Delta)|_a+|\textbf{ttop}(\Delta)|_a)^2
\end{align*}

\ \textbf{3.} Suppose the step history of $\Delta$ contains a letter $(12)_j$.

By (i), any unreduced two-letter subword of $B$ must be of form (b).

Suppose $B$ contains a subword of the form $(Q_0(1)P_1(1))^{\pm1}$. Then, by the definition of faulty, it must contain a reduced pararevolving subword of the form $(Q_0(1)\dots Q_0(2))^{\pm1}$. But then $B$ satisfies the hypothesis of Step 1, so that the statement follows.

Suppose $B$ contains a subword of the form $(Q_0(L)P_1(L))^{\pm1}$. Then by the definition of faulty, it must also contain a sector of the form $(P_0(L)Q_0(L))^{\pm1}$, so that it contains a reduced pararevolving subword of the form $(P_0(L)\dots P_0(1))^{\pm1}$. Again, $B$ then satisfies the hypothesis of Step 1.

But assuming these to be false, since $B$ must contain a subword of the form $(P_0(1)Q_0(1))^{\pm1}$, $B$ must be hyperfaulty and $\Delta$ exceptional.

So, the step history of $\Delta$ contains no letter of the form $(12)_j$ or, by symmetry, $(21)_j$.

\ \textbf{4.} By Steps 2 and 3, we assume that the step history of $\Delta$ contains a letter of the form $(s)_j^{\pm1}$.

By (ii), every unreduced two-letter subword must be of form (a).

So, $B$ must be a cyclic permutation of
$$P_0(2)P_0(2)^{-1}\dots P_1(L)^{-1}Q_0(L)^{-1}Q_0(L)P_1(L)\dots P_0(2)$$
where gaps correspond to strings of letters following the order of the standard base or its inverse written on a circle.

As $B$ contains a subword of the form $(Q_0(1)P_1(1))^{\pm1}$, then applying Lemma \ref{M step history 2}(a) to the corresponding reduced computation with such a base allows us to assume that the step history is a subword of $(1)_1(s)_1^{-1}(s)_2(1)_2$.

Let $\Delta''$ be the maximal subtrapezium of $\Delta$ with base $P_1(L)^{-1}Q_0(L)^{-1}Q_0(L)P_1(L)$. Further, let $W_0''\to\dots\to W_h''$ be the corresponding reduced computation and $W_0''\to\dots\to W_r''$ be the maximal subcomputation with step history $(1)_1$. Then, Lemma \ref{primitive unreduced} applies to the maximal subcomputations with step history $(1)_j$, so that $|W_r''|_a\leq\dots\leq|W_0''|_a$ and $|W_r''|_a\leq\dots\leq|W_h''|_a$. So, for any maximal $\theta$-band $\pazocal{T}''$ of $\Delta''$, $\|\textbf{tbot}(\pazocal{T}'')\|\leq\max(\|\textbf{tbot}(\Delta'')\|,\|\textbf{ttop}(\Delta'')\|)$, yielding $\text{wt}(\Delta'')\leq3h\max(\|\textbf{tbot}(\Delta'')\|,\|\textbf{ttop}(\Delta'')\|)$.

Let $\Delta'''$ be a sector of $\Delta$ with base of the form $(Q_0(1)P_1(1))^{\pm1}$. Then as above, Lemma \ref{multiply one letter} implies $\text{wt}(\Delta''')\leq3h\max(\|\textbf{tbot}(\Delta''')\|,\|\textbf{ttop}(\Delta''')\|)$ and $r=|\textbf{tbot}(\Delta''')|_a$.

Let $\Delta'$ be a sector of $\Delta$ with base of the form $(P_0(1)Q_0(1))^{\pm1}$. Then for any maximal $\theta$-band $\pazocal{T}'$, an $a$-edge of $\textbf{tbot}(\pazocal{T}')$ is part of a maximal $a$-band with one end on $\textbf{tbot}(\Delta')$, on $\textbf{ttop}(\Delta')$, or on a $(\theta,q)$-cell corresponding to the base letter $Q_0(1)^{\pm1}$ and a rule of step history $(1)_1$. So, $\|\textbf{tbot}(\pazocal{T}')\|\leq\|\textbf{tbot}(\Delta')\|+\|\textbf{ttop}(\Delta')\|+r$. Similarly, the sum of the combinatorial perimeters of the $a$-cells of $\Delta'$ is at most $|\textbf{tbot}(\Delta')|_a+|\textbf{ttop}(\Delta')|_a+r$. Hence, $$\text{wt}(\Delta')\leq3h(\|\textbf{tbot}(\Delta')\|+\|\textbf{ttop}(\Delta')\|+r)+C_1(|\textbf{tbot}(\Delta')|_a+|\textbf{ttop}(\Delta')|_a+r)^2$$
For $\Delta_0'$ any other sector of $\Delta$ not already accounted for, its tape word is fixed, so that $\text{wt}(\Delta_0')\leq3h\max(\|\textbf{tbot}(\Delta_0')\|,\|\textbf{ttop}(\Delta_0')\|)$.

Let $\{\Delta_i\}$ be the partition of $\Delta$ given by the subdiagrams detailed above. Then,
\begin{align*}
\text{wt}(\Delta)&\leq6h\sum(\|\textbf{tbot}(\Delta_i)\|+\|\textbf{ttop}(\Delta_i)\|)+C_1(|\textbf{tbot}(\Delta)|_a+|\textbf{ttop}(\Delta)|_a)^2 \\
&\leq12h(\|\textbf{tbot}(\Delta)\|+\|\textbf{ttop}(\Delta)\|)+C_1(|\textbf{tbot}(\Delta)|_a+|\textbf{ttop}(\Delta)|_a)^2 \\
&\leq24h\max(\|\textbf{tbot}(\Delta)\|+\|\textbf{ttop}(\Delta)\|)+C_1(|\textbf{tbot}(\Delta)|_a+|\textbf{ttop}(\Delta)|_a)^2
\end{align*}

Thus, the statement follows from the parameter choices $C_2>>C_1\geq8$.

\end{proof}

\smallskip

%%%%%%%%%%%%%%%%%%%%%%%%%%%%%%%%%%%%%%%%%%%%%%%%%%%%%%%%%%%%%%%%%

\subsection{Combs and Subcombs} \

Let $\Gamma$ be an $M$-minimal diagram containing a maximal $q$-band $\pazocal{Q}$ such that $\textbf{bot}(\pazocal{Q})$ is a subpath of $\partial\Delta$ and every maximal $\theta$-band of $\Delta$ ends at an edge of $\textbf{bot}(\pazocal{Q})$. Then $\Gamma$ is called a \textit{comb} and $\pazocal{Q}$ its \textit{handle}.

The number of cells in the handle of $\pazocal{Q}$ is the comb's \textit{height} and the maximal length of the bases of the $\theta$-bands its \textit{basic width}.

Note that every $a$-trapezium (or trapezium) may be viewed as a comb with either maximal side $q$-band its handle.

\begin{figure}[H]
\centering
\includegraphics[scale=1]{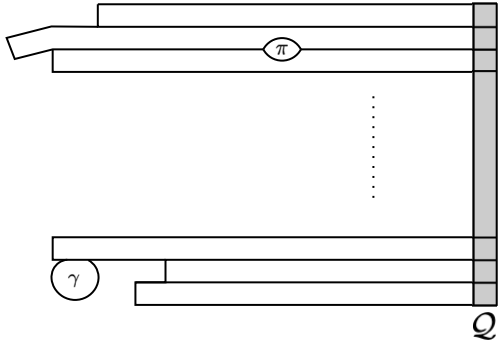}
\caption{Comb with handle $\pazocal{Q}$ containing $a$-cells $\pi$ and $\gamma$}
\end{figure}

\begin{lemma} \label{comb weights}

Let $\Gamma$ be a comb with height $h$, basic width $b$, and $|\partial\Gamma|_a=\a$. Let $\pazocal{T}_1,\dots,\pazocal{T}_h$ be the consecutive maximal $\theta$-bands of $\Gamma$ enumerated from bottom to top. Factor $\partial\Gamma=\textbf{y}\textbf{x}\textbf{z}$, where $\textbf{z}$ is the bottom of the handle of $\Gamma$ and $\textbf{x}$ is the maximal subpath below $\pazocal{T}_1$. Then: 

\begin{enumerate}[label=({\arabic*})]

\item $\text{wt}(\Gamma)\leq c_0bh^2+2\a h+C_1(bh+\a)^2$

\item $|\textbf{bot}(\pazocal{T}_1)|_a\leq|\textbf{y}|_a+4bh$

\end{enumerate}

\end{lemma}

\begin{proof}

(1) Let $n_i$ be the length of $\pazocal{T}_i$.

By Lemma \ref{a-band on same a-cell} and (MM2), every $a$-band starting on an $a$-cell must either end on a $(\theta,q)$-cell or on $\partial\Gamma$. Since every $(\theta,q)$-relation contains at most one $a$-letter from the `special' input sector, the sum of the combinatorial perimeters of all $a$-cells in $\Gamma$ is at most $bh+\a$.

So, $\text{wt}(\Gamma)\leq C_1(bh+\a)^2+\sum\limits_{i=1}^h n_i$.

Let $\a_i$ (respectively $\a_i'$) be the number of (unoriented) $a$-edges of $\textbf{bot}(\pazocal{T}_i)$ (respectively $\textbf{top}(\pazocal{T}_i)$) that are shared  with $\partial\Gamma$. Similarly, let $\b_i$ (respectively $\b_i'$) be the number of $a$-edges of $\textbf{bot}(\pazocal{T}_i)$ (respectively $\textbf{top}(\pazocal{T}_i)$) that are on the boundary of an $a$-cell. Note that $\sum_{i=1}^h(\a_i+\a_i')\leq\a$ and $\sum_{i=1}^h(\b_i+\b_i')\leq bh+\a$.

By the definition of a comb, any cell of $\Gamma$ below $\pazocal{T}_1$ must be an $a$-cell. So, any $a$-edge of $\textbf{bot}(\pazocal{T}_1)$ that is not shared with $\partial\Gamma$ is on the boundary of an $a$-cell below $\pazocal{T}_1$. Hence, $|\textbf{bot}(\pazocal{T}_1)|_a=\a_1+\b_1$. 

Similarly, $|\textbf{top}(\pazocal{T}_h)|_a=\a_h'+\b_h'$.

Lemma \ref{lengths}(d) implies $\a_1+\b_1-b\leq n_1\leq\a_1+\b_1+3b$ and $\a_h'+\b_h'-b\leq n_h\leq \a_h'+\b_h'+3b$.

Suppose an $a$-edge of $\textbf{top}(\pazocal{T}_i)$ is not shared with $\textbf{bot}(\pazocal{T}_{i+1})$. Then either this edge is counted in $\a_i'$ or is on the boundary of an $a$-cell between $\pazocal{T}_i$ and $\pazocal{T}_{i+1}$, so that it is counted in $\b_i'$. Similarly, an $a$-edge of $\textbf{bot}(\pazocal{T}_{i+1})$ not shared with $\textbf{top}(\pazocal{T}_i)$ is either counted in $\a_{i+1}$ or in $\b_{i+1}$.

So, the difference in the number of $a$-edges of $\textbf{top}(\pazocal{T}_i)$ and $\textbf{bot}(\pazocal{T}_{i+1})$ is at most $\a_i'+\a_{i+1}+\b_i'+\b_{i+1}$. Lemma \ref{lengths}(d) then implies that $|n_{i+1}-n_i|\leq4b+\a_i'+\a_{i+1}+\b_i'+\b_{i+1}$ for $1\leq i\leq h-1$.

Hence, for all $1\leq i\leq h$, we have:
\begin{align*}
n_i&\leq3b+4b(i-1)+\a_1+\a_1'+\dots+\a_{i-1}+\a_{i-1}'+\a_i+\b_1+\b_1'+\dots+\b_{i-1}+\b_{i-1}'+\b_i \\
&\leq5bh+2\a
\end{align*}
Thus, $\text{wt}(\Gamma)\leq C_1(bh+\a)^2+5bh^2+2\a h$.

%Finally, we have:
%\begin{align*}
%\a-|\textbf{x}|_a&=\a-(\a_1+\b_1)\geq\sum_{i=1}^{h-1}\a_i'+\sum_{i=2}^h\a_i+\sum_{i=1}^{h-1}\b_i'+\sum_{i=2}^h\b_i+\a_h'+\b_h' \\
%&\geq\left(\sum_{i=1}^{h-1}|n_{i+1}-n_i|\right)-4(h-1)b+n_h-3b \\
%&\geq\left(\sum_{i=1}^{h-1}(n_i-n_{i+1})\right)+n_h-4bh\geq n_1-4bh \\
%&\geq|\textbf{x}|_a-5bh
%\end{align*}

\bigskip

(2) For $i=1,2,3$, define $\textbf{A}_i$ as the subset of the set of (unoriented) $a$-edges of $\textbf{y}$ as follows:

\begin{itemize}

\item $\textbf{A}_1$ is the subset of edges that are on the boundary of an $a$-cell

\item $\textbf{A}_2$ is the subset of edges marking the start of a maximal $a$-band which ends on an edge of $\partial\Gamma$ shared with $\textbf{bot}(\pazocal{T}_1)$

\item $\textbf{A}_3$ is the subset of edges that mark the start of a maximal $a$-band of nonzero length which ends on an $a$-cell

\end{itemize}

Note that these three sets are disjoint, so that $\gamma_1+\gamma_2+\gamma_3\leq|\textbf{y}|_a$ for $\gamma_i=\#\textbf{A}_i$.

Let $\pi$ be an $a$-cell in $\Gamma$ such that some edge of $\partial\pi$ contributes to $\b_i'$. Let $\b_\pi'$ be the number of edges of $\partial\pi$ contributing to $\b_i'$ and $\b_\pi''$ be the number of such edges on the contour of a $(\theta,q)$-cell of $\pazocal{T}_i$. Property (MM1) implies $\b_\pi'\leq\frac{1}{2}\|\partial\pi\|+\b_\pi''$. So,
$$\sum_{i=1}^h\b_i'\leq\sum_\pi\b_\pi'\leq\sum_\pi\left(\frac{1}{2}\|\partial\pi\|+\b_\pi''\right)$$
Note that any edge of $\partial\pi$ not contributing to $\b_\pi'$ is either part of $\partial\Gamma$ or on $\textbf{bot}(\pazocal{T}_{i+1})$, and so contributes to $\gamma_1$ or $\b_{i+1}$, respectively. So, since there are at least $\frac{1}{2}\|\partial\pi\|-\b_\pi''$ such edges, 
$$\gamma_1+\sum_{i=2}^h\b_i\geq\sum_\pi\left(\frac{1}{2}\|\partial\pi\|-\b_\pi''\right)$$
As the contour of any $(\theta,q)$-cell contains at most one $a$-edge corresponding to the `special' input sector, we then have
$$\sum_{i=1}^h\b_i'\leq\gamma_1+\sum_{i=2}^h\b_i+2\sum_\pi\b_\pi''\leq\gamma_1+\sum_{i=2}^h\b_i+2bh$$
Next, let $\textbf{e}$ be an $a$-edge of $\textbf{bot}(\pazocal{T}_1)$ contributing to $\a_1$. Then the maximal $a$-band starting at $\textbf{e}$ ends on $\textbf{y}$, ends on a $(\theta,q)$-cell, or ends on an $a$-cell. Those that end on $\textbf{y}$ correspond to edges of $\textbf{A}_2$ while those that end on  an $a$-cell correspond to $a$-edges contributing to $\b_i'$ for some $i$. So,
$$\a_1\leq\gamma_2+bh+\sum_{i=1}^h\b_i'\leq\gamma_1+\gamma_2+\sum_{i=2}^h\b_i+3bh$$
Finally, let $\textbf{e}$ be an $a$-edge of $\textbf{bot}(\pazocal{T}_i)$ contributing to $\b_i$. By Lemma \ref{a-band on same a-cell} and (MM2), the maximal $a$-band starting on $\textbf{e}$ must end on $\textbf{y}$ or on a $(\theta,q)$-cell. As those that end on $\textbf{y}$ correspond to edges of $\textbf{A}_3$, we have $\sum_{i=1}^h\b_i\leq\gamma_3+bh$. Thus,
$$
|\textbf{bot}(\pazocal{T}_1)|_a=\a_1+\b_1\leq\gamma_1+\gamma_2+\sum_{i=1}^h\b_i+3bh\leq\gamma_1+\gamma_2+\gamma_3+4bh\leq|\textbf{y}|_a+4bh
$$

\end{proof}

A base word $B$ is \textit{tight} if it is of the form $uxvx$ for some letter $x$ and words $u$ and $v$, where: 

\begin{enumerate}[label=({\arabic*})]

\item $xvx$ is revolving, and

\item no letter from $u$ occurs in $xvx$.

\end{enumerate}

Note that any tight base has length at most $K_0=22L+1$, while any base with length at least $K_0$ must have a tight prefix.

A comb $\Delta$ is called \textit{tight} if:

\begin{enumerate}[label=(C{\arabic*})]

\item one of its maximal $\theta$-bands $\pazocal{T}$ has a tight base when read toward the handle, and

\item all maximal $\theta$-bands have tight bases or bases without tight prefixes

\end{enumerate}

\smallskip

If $\Delta$ is an $M$-minimal diagram over $M_\Omega(\textbf{M})$, then a subdiagram $\Gamma$ is a \textit{subcomb} of $\Delta$ if $\Gamma$ is a comb and its handle divides $\Delta$ into two parts, one of which is $\Gamma$.

Let $\Gamma$ be a comb with handle $\pazocal{C}$ and $\pazocal{B}$ be another maximal $q$-band in $\Gamma$. Then $\pazocal{B}$ cuts $\Gamma$ into two parts, where the part not containing $\pazocal{C}$ is a subcomb $\Gamma'$ with handle $\pazocal{B}$. Note that each maximal $\theta$-band $\pazocal{T}$ of $\Gamma$ crossing $\pazocal{B}$ has a subband $\pazocal{T}_0$ connecting $\pazocal{B}$ with $\pazocal{C}$. If $\pazocal{T}_0$ has no $(\theta,q)$-cells, then $\Gamma'$ is called a \textit{derivative subcomb} of $\Gamma$. 

Note that no maximal $\theta$-band of a comb can cross the handles of more than one derivative subcomb.

\begin{lemma} \label{tight subcomb}

\textit{(Compare with Lemma 6.6 of [16] and Lemma 6.10 of [23])} 

Let $\Delta$ be an $M$-minimal diagram such that $|\partial\Delta|_\theta>0$ and every quasi-rim $\theta$-band has base of length at least $K$. %Assume that:
%\begin{addmargin}[1em]{0em}
%
%(1) $\Delta$ is a diagram over $M_a(\textbf{M})$, or
%
%(2) $\Delta$ has a subcomb of basic width at least $K_0$.
%
%\end{addmargin}
Then $\Delta$ contains a tight subcomb.

\end{lemma}

\begin{proof}

As maximal $\theta$-bands cannot cross, there exists a quasi-rim $\theta$-band $\pazocal{T}_0$ in $\Delta$. Taking $K>2K_0$, the base of $\pazocal{T}_0$ has disjoint prefix and suffix, $B_1$ and $B_2$, of lengths $K_0$. As a result, $B_1$ has a prefix $B_1'$ which is tight, while $B_2$ has a suffix $B_2'$ such that $(B_2')^{-1}$ is tight.

Let $\pi$ be the $(\theta,q)$-cell of $\pazocal{T}_0$ corresponding to the last base letter of $B_1'$ and $\pazocal{Q}'$ be the maximal $q$-band of $\Delta$ containing $\pi$. Let $\Gamma'$ be the subdiagram of $\Delta$ bounded by $\pazocal{Q}'$ containing the subband of $\pazocal{T}_0$ with base $B_1'$. 

Note that we may do the same with $B_2'$ to construct a subdiagram $\Gamma''$.

Hence, there exists a maximal $q$-band $\pazocal{Q}$ such that for one of the subdiagrams $\Gamma$ of $\Delta$ bounded by $\pazocal{Q}$, there exists a maximal $\theta$-band $\pazocal{T}$ whose base is tight when read toward $\pazocal{Q}$. Choose such a $\pazocal{Q}$ and $\Gamma$ such that $\text{wt}(\Gamma)$ is minimal.

\begin{figure}[H]
\centering
\includegraphics[scale=1.4]{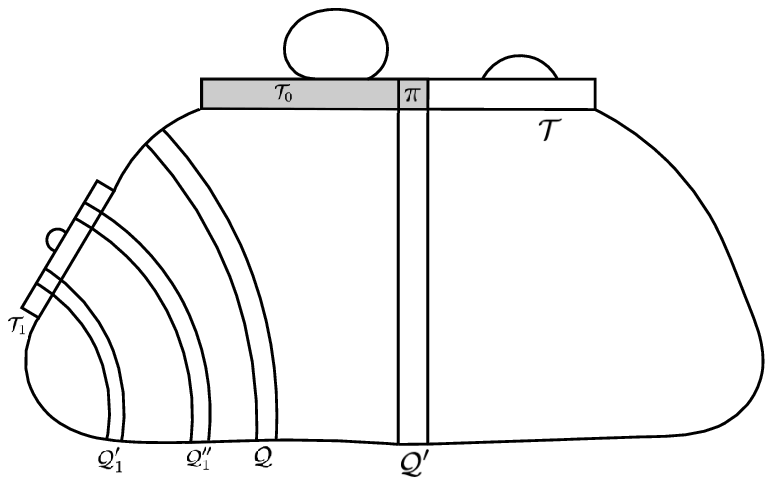}
\caption{Lemma \ref{tight subcomb}}
\end{figure}

Suppose there exists a $\theta$-band in $\Gamma$ which does not cross $\pazocal{Q}$. Then, there exists a quasi-rim $\theta$-band $\pazocal{T}_1$ not crossing $\pazocal{Q}$. As the base of $\pazocal{T}_1$ has length at least $K$, we may repeat the argument above. This produces disjoint subdiagrams $\Gamma_1'$ and $\Gamma_1''$ of $\Gamma$ bounded by the maximal $q$-bands $\pazocal{Q}_1'$ and $\pazocal{Q}_1''$, respectively, such that the subband of $\pazocal{T}_1$ which is a maximal $\theta$-band of $\Gamma_1'$ (resp $\Gamma_1''$) has tight base when read toward $\pazocal{Q}_1'$ (resp $\pazocal{Q}_1''$). One of these subdiagrams, say $\Gamma_1'$, does not contain $\pazocal{Q}$, and so is a subcomb of $\Delta$. But then $\text{wt}(\Gamma_1')<\text{wt}(\Gamma)$, so that $\pazocal{Q}_1'$ and $\Gamma_1'$ contradict the choice of $\pazocal{Q}$ and $\Gamma$.

Thus, $\Gamma$ is a comb with handle $\pazocal{Q}$ satisfying condition (C1).

Now suppose there exists a maximal $\theta$-band $\pazocal{T}'$ in $\Gamma$ with a tight proper prefix $B'$. Let $\pi'$ be the $(\theta,q)$-cell of $\pazocal{T}'$ corresponding to the last letter of $B'$ and $\pazocal{Q}'$ be the maximal $q$-band in $\Gamma$ containing $\pi'$. Then for $\Gamma'$ the subcomb of $\Gamma$ with handle $\pazocal{Q}'$, $\pazocal{Q}'$ and $\Gamma'$ contradict the choice of $\pazocal{Q}$ and $\Gamma$. Hence, $\Gamma$ must be a tight subcomb.

\end{proof}

\smallskip

%%%%%%%%%%%%%%%%%%%%%%%%%%%%%%%%%%%%%%%%%%%%%%%%%%%%%%%%%%%%%%%%%

\subsection{G-weight} \

The goal of this section is to bound the size of an $M$-minimal diagram over $M_\Omega(\textbf{M})$ in terms of its perimeter. However, this bound will not be given in terms of the area or weight of the diagram. Instead, we give the bound in terms of the artifical concept of \textit{$G$-weight} (adapted from the concept of $G$-area in [16] and [23]).

Let $\Gamma$ be an $a$-trapezium with base $B$ and history $H$. Suppose $B$ is of the form $(P_0(1)Q_0(1))^{\pm1}$ or $Q_0(1)^{-1}Q_0(1)$, the step history of $\Gamma$ is $(1)_1$, and $H$ has a factorization $H_1H_2^\ell H_3$ for some $\ell\geq0$. Then $\Gamma$ is called an \textit{impeding} $a$-trapezium.

In this case, let $\eta=\|H_1\|+n\|H_2\|+\|H_3\|$ and $h=\|H\|$. Then we define the $G$-weight of $\Gamma$, denoted $\text{wt}_G(\Gamma)$, to be the minimum of half its weight and:
$$3h\max(\|\textbf{tbot}(\Gamma)\|,\|\textbf{ttop}(\Gamma)\|)+3C_1h\eta+C_1(|\textbf{tbot}(\Gamma)|_a+|\textbf{ttop}(\Gamma)|_a+2\eta)^2$$
Similarly, if $\Gamma$ is a big $a$-trapezium with height $h$ then its $G$-weight is defined to be the minimum of half its weight and:
$$c_5\max(\|\textbf{ttop}(\Gamma)\|,\|\textbf{tbot}(\Gamma)\|)+4C_1(\|\textbf{tbot}(\Gamma)\|+\|\textbf{ttop}(\Gamma)\|)^2$$

Finally, any single cell in $\Gamma$ is assigned $G$-weight equal to its weight.

For a reduced diagram $\Delta$ over $G_\Omega(\textbf{M})$, consider a family of subdiagrams $\textbf{P}$ such that:

\begin{itemize}

\item if $P\in\textbf{P}$, then $P$ is a single cell, a big $a$-trapezium, or an impeding trapezium,

\item every cell of $\Delta$ belongs to an element of $\textbf{P}$, and

\item if there exist $P_1,P_2\in\textbf{P}$ with nonempty intersection, then both $P_1$ and $P_2$ are $a$-trapezia and this intersection is a $q$-band.

\end{itemize}

In this case, $\textbf{P}$ is called a \textit{covering} of $\Delta$. The $G$-weight of $\textbf{P}$, $\text{wt}_G(\textbf{P})$, is defined to be the sum of the $G$-weights of its elements. 

Note that any reduced diagram over $G_\Omega(\textbf{M})$ has a covering, namely the one given by its cells. So, we may define the $G$-weight of $\Delta$, $\text{wt}_G(\Delta)$, as the minimum of the $G$-weights of its coverings.

Further, since the $G$-weight of a big or impeding $a$-trapezium does not exceed half of its weight and any cell belongs to at most two elements of a covering, the inequality $\text{wt}_G(\Delta)\leq\text{wt}(\Delta)$ holds for all $\Delta$.

\begin{lemma} \label{G-weight subdiagrams}

Let $\Delta$ be a reduced diagram over $G_\Omega(\textbf{M})$ and suppose every cell $\pi$ of $\Delta$ belongs in one of the subdiagrams $\Delta_1,\dots,\Delta_m$, where any nonempty intersection $\Delta_i\cap\Delta_j$ is a $q$-band. Then $\text{wt}_G(\Delta)\leq\sum_{i=1}^m\text{wt}_G(\Delta_i)$.

\end{lemma}

\begin{proof}

Let $\textbf{P}_1,\dots,\textbf{P}_m$ be coverings of $\Delta_1,\dots,\Delta_m$, respectively, so that the $G$-weight of $\textbf{P}_i$ is equal to that of $\Delta_i$. Then $\textbf{P}=\textbf{P}_1\cup\dots\cup\textbf{P}_m$ is a covering of $\Delta$ with $\text{wt}_G(\textbf{P})\leq\sum_{i=1}^m\text{wt}_G(\textbf{P}_i)$, implying the statement.

\end{proof}

In particular, note that Lemma \ref{G-weight subdiagrams} implies that if $\{\Delta_i\}$ is a partition of the $a$-trapezium $\Delta$, then $\text{wt}_G(\Delta)\leq\sum\text{wt}_G(\Delta_i)$.

\begin{lemma} \label{revolving G-weight}

Suppose $\Delta$ is an $a$-trapezium with revolving base $B$ and history $H$. Then for $h=\|H\|$,
$$\text{wt}_G(\Delta)\leq C_2h\max(\|\textbf{tbot}(\Delta)\|,\|\textbf{ttop}(\Delta)\|)+C_2(\|\textbf{tbot}(\Delta)\|+\|\textbf{ttop}(\Delta)\|)^2$$

\end{lemma}

\begin{proof}

By Lemma \ref{a-trapezia revolving} and the assignment of $G$-weight to big trapezia, we may assume that $\Delta$ is exceptional.

\ \textbf{1.} Suppose the step history of $\Delta$ is $(1)_1$.

Let $\Delta'$ be the $Q_0(1)Q_0(1)^{-1}$-sector of $\Delta$ and let $\pazocal{C}':W_0'\to\dots\to W_h'$ be the reduced computation corresponding to $\Delta'$. Lemma \ref{unreduced base} implies that $|W_j'|_a\leq\max(|W_0'|_a,|W_h'|_a)$ for all $0\leq j\leq h$, so that $\text{wt}_G(\Delta')\leq\text{wt}(\Delta')\leq3h\max(\|\textbf{tbot}(\Delta')\|,\|\textbf{ttop}(\Delta')\|)$.

What's more, Lemma \ref{unreduced base} implies that the history $H$ of $\pazocal{C}'$ can be factored as $H_1H_2^\ell H_3$ for some $\ell\geq0$ with $\|H_1\|\leq\frac{1}{2}|W_0'|_a$, $\|H_3\|\leq\frac{1}{2}|W_h'|_a$, and $\|H_2\|\leq\min(|W_0'|_a,|W_h'|_a)$.

So, for $\Delta''$ any sector of $\Delta$ corresponding to the `special' input sector, $\Delta''$ is an impeding $a$-trapezium with $\eta\leq\frac{n+1}{2}(|W_0'|_a+|W_h'|_a)\leq c_0(|\textbf{tbot}(\Delta)|_a+|\textbf{ttop}(\Delta)|_a)$. 

Setting $S=|\textbf{tbot}(\Delta)|_a+|\textbf{ttop}(\Delta)|_a$, this implies:
$$\text{wt}_G(\Delta'')\leq 3h\max(\|\textbf{tbot}(\Delta'')\|,\|\textbf{ttop}(\Delta'')\|)+3c_0C_1hS+C_1(2c_0+1)^2S^2$$
Every sector $\Delta'''$ not of these forms is a trapezium whose corresponding computation has fixed $a$-length, so that $\text{wt}(\Delta''')\leq3h\max(\|\textbf{tbot}(\Delta''')\|,\|\textbf{ttop}(\Delta''')\|)$.

Let $\{\Delta_i\}$ be the partition of $\Delta$ given by its sectors. Then $\text{wt}_G(\Delta)\leq\sum\text{wt}_G(\Delta_i)$ by Lemma \ref{G-weight subdiagrams}. Recall that $\sum\max(\|\textbf{tbot}(\Delta_i)\|,\|\textbf{ttop}(\Delta_i)\|)\leq4\max(\|\textbf{tbot}(\Delta)\|,\|\textbf{ttop}(\Delta)\|)$.

Since at most two sectors correspond to the `special' input sector, the parameter choices $C_2>>C_1>>c_0$ yield:
\begin{align*}
\text{wt}_G(\Delta)&\leq12h\max(\|\textbf{tbot}(\Delta)\|,\|\textbf{ttop}(\Delta)\|)+6c_0C_1hS+2C_1(3c_0)^2S^2 \\
&\leq12(c_0C_1+1)h\max(\|\textbf{tbot}(\Delta)\|,\|\textbf{ttop}(\Delta)\|)+18c_0^2C_1(|\textbf{tbot}(\Delta)|_a+|\textbf{ttop}(\Delta)|_a)^2  \\
&\leq C_2h\max(\|\textbf{tbot}(\Delta)\|,\|\textbf{ttop}(\Delta)\|)+C_2(|\textbf{tbot}(\Delta)|_a+|\textbf{ttop}(\Delta)|_a)^2
\end{align*}
\ \textbf{2.} Thus, we may assume that the step history of $\Delta$ is not $(1)_1$. 

By (i) and (ii), we may then assume that the base $B$ of $\Delta$ is a cyclic permutation of
$$P_0(1)Q_0(1)Q_0(1)^{-1}P_0(1)^{-1}\{t(1)\}^{-1}Q_4(L)^{-1}\dots P_1(L)^{-1}P_1(L)\dots Q_4(L)\{t(1)\}P_0(1)$$
As a result, $H$ cannot contain a letter of the form $\theta(s)_j^{\pm1}$ or the copy of a connecting rule of $\textbf{M}_4(3^-)$. So, the step history of $\Delta$ must contain the letter $(12)_1$ or $(21)_1$.

Let $\Delta'$ be the $R_2(L)Q_3(L)$-sector of $\Delta$. Lemma \ref{a-cells sector} then implies that $\Delta'$ is a trapezium, so that Lemma \ref{trapezia are computations} gives a corresponding reduced computation. As a result, Lemma \ref{M_3 step history}(b) implies that the step history of $\Delta$ has no subword of the form $(12)_1(2)_1(21)_1$.

Similarly, as $B$ contains a subword $P_1(L)Q_1(L)$, Lemma \ref{multiply one letter} implies that the step history of $\Delta$ has no subword of the form $(23)_1(3)_1(32)_1$.

Hence, the step history of $\Delta$ is a subword of
$$(3)_1(32)_1(2)_1(21)_1(1)_1(12)_1(2)_1(23)_1(3)_1$$
containing the letter $(1)_1$.

Let $\Delta_1$ be the maximal subdiagram of $\Delta$ which is an $a$-trapezium with step history $(1)_1$.

Suppose $\textbf{top}(\Delta_1)$ does not coincide with $\textbf{top}(\Delta)$.

Let $\pazocal{T}_1$ be the maximal $\theta$-band of $\Delta_1$ such that $\textbf{top}(\pazocal{T}_1)=\textbf{top}(\Delta_1)$. Then, there exists a maximal $\theta$-band $\pazocal{T}_1'$ of $\Delta$ corresponding to the rule $\theta(12)_1$ and such that some edges of $\textbf{bot}(\pazocal{T}_1')$ coincide with those of $\textbf{top}(\pazocal{T}_1)$.

Suppose there is an $a$-cell $\pi$ in $\Delta$ between $\pazocal{T}_1$ and $\pazocal{T}_1'$. As $\theta(12)_1$ locks the `special' input sector, no $a$-band starting on $\partial\pi$ can cross $\pazocal{T}_1'$ or end on a $(\theta,q)$-cell of $\pazocal{T}_1'$. Lemma \ref{a-cells sector} implies that $\pi$ must belong to a $(P_0(1)Q_0(1))^{\pm1}$ sector of $\Delta$, so that at most one $a$-band starting on $\partial\pi$ can end on a $(\theta,q)$-cell of $\pazocal{T}_1$. Taking $n\geq3$, more than half of the $a$-bands starting on $\partial\pi$ must cross $\pazocal{T}_1$. But then $\pi$ and $\pazocal{T}_1$ contradict property (MM1). So, $\textbf{bot}(\pazocal{T}_1')=\textbf{top}(\pazocal{T}_1)$.

Let $\Delta_1'$ be the maximal subdiagram of $\Delta$ which is an $a$-trapezium with $\textbf{bot}(\Delta_1')=\textbf{bot}(\pazocal{T}_1')$. Lemma \ref{a-cells sector} then implies that $\Delta_1'$ is a trapezium, so that Lemma \ref{trapezia are computations} yields a corresponding reduced computation $\pazocal{C}:W_0\to\dots\to W_t$ with base $B$ and step history a prefix of $(12)_1(2)_1(23)_1(3)_1$. 

We now prove that $|W_0|_a\leq\dots\leq|W_t|_a$. Assuming toward contradiction, the step history of $\pazocal{C}$ cannot be $(12)_1$. Let $\pazocal{C}_2:W_0\to\dots\to W_r$ be the maximal subcomputation with step history $(12)_1(2)_1$ and $\pazocal{C}_2':W_0'\to\dots\to W_r'$ be the restriction to the subword 
$$Q_3(L)^{-1}R_2(L)^{-1}\dots Q_1(L)^{-1}P_1(L)^{-1}P_1(L)Q_1(L)\dots R_2(L)Q_3(L)$$
of $B$. As every rule of $\pazocal{C}_2'$ locks the $Q_1(L)R_1(L)$-, $R_1(L)Q_2(L)$-, and $Q_2(L)R_2(L)$-sectors, we may view the subwords of the form $(Q_1(L)R_1(L)Q_2(L)R_2(L))^{\pm1}$ as a single state letter. With this view, we may apply Lemma \ref{primitive unreduced} to $\pazocal{C}_2'$, so that $|W_0'|_a\leq\dots\leq|W_r'|_a$. As all other sectors have fixed tape word throughout $\pazocal{C}_2$, this implies $|W_0|_a\leq\dots\leq|W_r|_a$.

So, we may assume that $t>r$. As a result, there exists a subcomputation $\pazocal{C}_3:W_r\to\dots\to W_t$ with step history $(23)_1(3)_1$. Letting $\pazocal{C}_3':W_r'\to\dots\to W_t'$ be the restriction of $\pazocal{C}_3$ to the subword $Q_1(L)^{-1}P_1(L)^{-1}P_1(L)Q_1(L)$, Lemma \ref{primitive unreduced} implies $|W_r'|_a\leq\dots\leq|W_t'|_a$. As above, all other sectors have fixed tape word throughout $\pazocal{C}_3$, so that $|W_r|_a\leq\dots\leq|W_t|_a$.

As a result, for any maximal $\theta$-band $\pazocal{T}$ of $\Delta_1'$, $\|\textbf{tbot}(\pazocal{T})\|\leq\|W_t\|=\|\textbf{ttop}(\Delta_1')\|=\|\textbf{ttop}(\Delta)\|$. In particular, $\|\textbf{ttop}(\Delta_1)\|\leq\|\textbf{ttop}(\Delta)\|$. Hence, for $h_1'$ the height of $\Delta_1'$, $\text{wt}(\Delta_1')\leq 3h_1'\|\textbf{ttop}(\Delta)\|$.

Similarly, if $\textbf{bot}(\Delta_1)$ does not coincide with $\textbf{bot}(\Delta)$, then there exists a subdiagram $\Delta_1''$ of $\Delta$ which is a trapezium satisfying $\textbf{top}(\Delta_1'')=\textbf{bot}(\Delta_1)$ and $\textbf{bot}(\Delta_1'')=\textbf{bot}(\Delta)$. By analogous arguments, $\|\textbf{tbot}(\Delta_1)\|\leq\|\textbf{tbot}(\Delta)\|$ and $\text{wt}(\Delta_1'')\leq3h_1''\|\textbf{tbot}(\Delta)\|$ for $h_1''$ the height of $\Delta_1''$.

By Step 1, letting $h_1$ be the height of $\Delta_1$, we have
\begin{align*}
\text{wt}_G(\Delta_1)&\leq C_2h_1\max(\|\textbf{tbot}(\Delta_1)\|,\|\textbf{ttop}(\Delta_1)\|)+C_2(|\textbf{tbot}(\Delta_1)|_a+|\textbf{ttop}(\Delta_1)|_a)^2 \\
&\leq C_2h_1\max(\|\textbf{tbot}(\Delta)\|,\|\textbf{ttop}(\Delta)\|)+C_2(|\textbf{tbot}(\Delta)|_a+|\textbf{ttop}(\Delta)|_a)^2
\end{align*}
Thus, Lemma \ref{G-weight subdiagrams} yields
\begin{align*}
\text{wt}_G(\Delta)&\leq (C_2h_1+3h_1'+3h_1'')\max(\|\textbf{tbot}(\Delta)\|,\|\textbf{ttop}(\Delta)\|)+C_2(|\textbf{tbot}(\Delta)|_a+|\textbf{ttop}(\Delta)|_a)^2 \\
&\leq C_2h\max(\|\textbf{tbot}(\Delta)\|,\|\textbf{ttop}(\Delta)\|)+C_2(|\textbf{tbot}(\Delta)|_a+|\textbf{ttop}(\Delta)|_a)^2
\end{align*}

\end{proof}

\medskip

%%%%%%%%%%%%%%%%%%%%%%%%%%%%%%%%%%%%%%%%%%%%%%%%%%%%%%%%%%%%%%%%%

\subsection{Quadratic upper bound} \

Our goal throughout the rest of this section is to prove that for any $M$-minimal diagram $\Delta$, 
\reqnomode
\begin{equation} \label{counterexample}
\text{wt}_G(\Delta)\leq N_2|\partial\Delta|^2+N_1\mu(\Delta)
\end{equation}
for the parameters $N_1$ and $N_2$.

We do this by arguing toward contradiction, considering a `minimal counterexample' diagram $\Delta$. In other words, $\Delta$ is an $M$-minimal diagram over $M_\Omega(\textbf{M})$ satisfying $\text{wt}_G(\Delta)>N_2|\partial\Delta|^2+N_1\mu(\Delta)$, while (\ref{counterexample}) holds for all $M$-minimal diagrams $\Gamma$ over $M_\Omega(\textbf{M})$ satisfying $|\partial\Gamma|<|\partial\Delta|$. 

\bigskip

\begin{lemma} \label{no q-edge quadratic}

If $\Gamma$ is an $M$-minimal diagram over $M_\Omega(\textbf{M})$, with no $q$-edges on its boundary, then $\text{wt}_G(\Gamma)\leq C_2|\partial\Gamma|^2$.

\end{lemma}

\begin{proof}

Since any $q$-edge in $\Gamma$ would give rise to a maximal $q$-band which, by Lemma \ref{M_a no annuli 1}, can only end on the boundary of the diagram, $\Gamma$ cannot have any $q$-edges. So, $\Gamma$ is comprised entirely of $(\theta,a)$-cells and $a$-cells. 

In particular, $\Gamma$ contains no $a$-trapezia (or trapezia), so that the only covering of $\Gamma$ is by single cells. Hence, $\text{wt}_G(\Gamma)=\text{wt}(\Gamma)$.

Lemma \ref{a-band on same a-cell} and (MM2) then imply that any maximal $a$-band with one end on an $a$-cell must have its other end on the boundary, so that the sum of the (combinatorial) perimeters of the $a$-cells is at most $\|\partial\Gamma\|$. It follows that the sum of the weights of the $a$-cells is at most $C_1\|\partial\Gamma\|^2$. 

Further, Lemma \ref{M_a no annuli 2} implies that any maximal $\theta$-band must start and end on $\partial\Gamma$, so that there are at most $\frac{1}{2}\|\partial\Gamma\|$ maximal $\theta$-bands in $\Gamma$. As Lemma \ref{M_a no annuli 1} implies that each maximal $a$-band must have at least one end on $\partial\Gamma$ and each $\theta$-band intersects each $a$-band in at most one cell, the length of each $\theta$-band is at most $\|\partial\Gamma\|$. So, the sum of the lengths of all maximal $\theta$-bands, and so the number of $(\theta,a)$-cells, is at most $\frac{1}{2}\|\partial\Gamma\|^2$.

Taking into account the modified definition of perimeter, the statement follows from an appropriate choice of $C_2$ in terms of $C_1$ and $\delta$.

\end{proof}

The parameter choice $N_2>>C_2$ and Lemma \ref{no q-edge quadratic} allow us to assume that $\partial\Delta$ consists of at least two $q$-edges, i.e $|\partial\Delta|\geq2$.

\begin{lemma} \label{a-cell in counterexample}

Let $\pi$ be an $a$-cell contained in $\Delta$. Suppose $\partial\pi$ has a subpath $\textbf{s}$ shared with $\partial\Delta$. Then $\|\textbf{s}\|\leq\frac{2}{3}\|\partial\pi\|$.

\end{lemma}

\begin{proof}

Let $\partial\pi=\textbf{s}\textbf{t}$ and $\partial\Delta=\textbf{s}\textbf{s}_0$.

Assuming toward contradiction that $\|\textbf{s}\|>\frac{2}{3}\|\partial\pi\|$, we have $\|\textbf{s}\|>2\|\textbf{t}\|$ and $\|\textbf{s}\|>\frac{2}{3}n\geq8$ by a parameter choice.

Let $\Delta_0$ be the subdiagram of $\Delta$ obtained by removing $\pi$. So, $\partial\Delta_0=\textbf{t}^{-1}\textbf{s}_0$.

By Lemma \ref{lengths}(c), $|\partial\Delta_0|\leq|\textbf{s}_0|+|\textbf{t}|=|\textbf{s}_0|+\delta\|\textbf{t}\|$ and $|\partial\Delta|\geq|\textbf{s}|+|\textbf{s}_0|-2\delta\geq|\textbf{s}_0|+\delta(\|\textbf{s}\|-2)$.

So, $|\partial\Delta|-|\partial\Delta_0|\geq\delta(\|\textbf{s}\|-\|\textbf{t}\|-2)\geq\frac{1}{2}\delta(\|\textbf{s}\|-4)\geq\frac{1}{4}\delta\|\textbf{s}\|>0$.

The inductive hypothesis then applies to $\Delta_0$, yielding 
$$\text{wt}_G(\Delta_0)\leq N_2|\partial\Delta_0|^2+N_1\mu(\Delta_0)\leq N_2(|\partial\Delta|-\delta\|\textbf{s}\|/4)^2+N_1\mu(\Delta_0)$$
As $\delta\|s\|/4\leq|\partial\Delta|$, $(|\partial\Delta|-\delta\|\textbf{s}\|/4)^2\leq|\partial\Delta|^2-\frac{1}{4}\delta\|\textbf{s}\||\partial\Delta|$.

By Lemma \ref{G-weight subdiagrams}, we have $\text{wt}_G(\Delta)\leq\text{wt}_G(\Delta_0)+\text{wt}(\pi)$. Further, the necklaces corresponding to $\partial\Delta$ and $\partial\Delta_0$ are identical, so that $\mu(\Delta)=\mu(\Delta_0)$. So, since the combinatorial perimeter of $\pi$ is $\|\textbf{s}\|+\|\textbf{t}\|$, Lemma \ref{G-weight subdiagrams} then implies:
$$\text{wt}_G(\Delta)\leq N_2|\partial\Delta|^2-\frac{1}{4}N_2\delta\|\textbf{s}\||\partial\Delta|+N_1\mu(\Delta)+C_1(\|\textbf{s}\|+\|\textbf{t}\|)^2$$
So, we reach the contradiction $\text{wt}_G(\Delta)\leq N_2|\partial\Delta|^2+N_1\mu(\Delta)$ if 
$$\frac{1}{4}N_2\delta\|\textbf{s}\||\partial\Delta|\geq C_1(\|\textbf{s}\|+\|\textbf{t}\|)^2$$
As $\|\textbf{t}\|<\frac{1}{2}\|\textbf{s}\|$, we have $C_1(\|\textbf{s}\|+\|\textbf{t}\|)^2\leq\frac{9}{4}C_1\|\textbf{s}\|^2$. So, since $|\partial\Delta|\geq\frac{1}{4}\delta\|\textbf{s}\|$, it suffices to show $N_2\delta^2\geq36C_1$. But this follows from the parameter choices $N_2>>C_1>>\delta^{-1}$.

\end{proof}

The following is the direct analogue of Lemma 6.12 of [16] and Lemma 6.16 of [23]. The method of proof is identical to the ones presented in those sources, though many of the estimates differ.

\bigskip

\begin{lemma} \label{6.16} \

\begin{enumerate}[label=({\arabic*})]

\item $\Delta$ has no two disjoint subcombs $\Gamma_1$ and $\Gamma_2$ of basic widths at most $K$ with handles $\pazocal{B}_1$ and $\pazocal{B}_2$ such that some ends of these handles are connected by a subpath $\textbf{x}$ of $\partial\Delta$ with $|\textbf{x}|_q\leq c_0$.

\item If $\Gamma$ is a subcomb of $\Delta$ with basic width $s\leq K$, $|\partial\Gamma|_q=2s$.

\end{enumerate}

\end{lemma}

\renewcommand\thesubfigure{\arabic{subfigure}}
\begin{figure}[H]
\centering
\begin{subfigure}[b]{0.48\textwidth}
\centering
\includegraphics[scale=1]{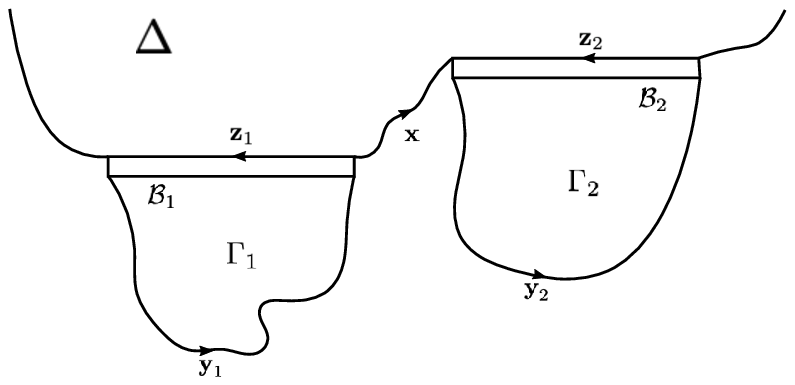}
\caption{ \ }
\end{subfigure}\hfill
\begin{subfigure}[b]{0.48\textwidth}
\centering
\includegraphics[scale=1.225]{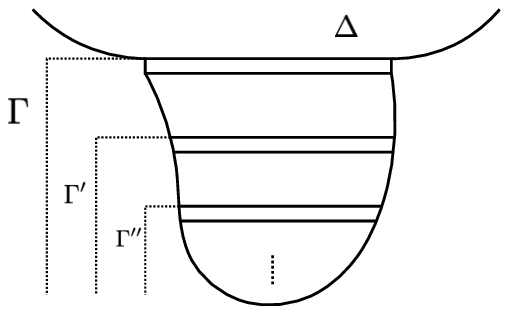}
\caption{ \ }
\end{subfigure}
\caption{Lemma \ref{6.16}}
\end{figure}

\begin{proof}

We prove (1) and (2) simultaneously, inducting on $W=\text{wt}(\Gamma_1)+\text{wt}(\Gamma_2)$ for (1) and $W=\text{wt}(\Gamma)$ for (2). In other words, we consider a counterexample to one of these two with minimal value of $W$.

Suppose the minimal counterexample is of the form (1). 

As $\text{wt}(\Gamma_i)<W$ for $i=1,2$, the inductive hypothesis implies that (2) holds for each. So, $\partial\Gamma_i$ has at most $2K$ $q$-edges.

Let $h_i$ be the height of $\Gamma_i$ and assume without loss of generality that $h_1\leq h_2$. For $i=1,2$, let $\partial\Gamma_i=\textbf{y}_i\textbf{z}_i$ where $\textbf{y}_i$ is a subpath of $\partial\Delta$ and $\textbf{z}_i=\textbf{bot}(\pazocal{B}_i)$. Without loss of generality, assume $\textbf{y}_1\textbf{x}\textbf{y}_2$ is a subpath of $\partial\Delta$.

Then each $\theta$-edge of $\textbf{y}_1$ is separated in $\partial\Delta$ from each $\theta$-edge of $\textbf{y}_2$ by at most $4K+c_0$ $q$-edges, and so by at most $J$ $q$-edges by the choice of parameters. Hence, each (correctly ordered) pair of such edges (or the white beads corresponding to these edges) makes a contribution to $\mu(\Delta)$.

Let $\Delta'$ be the diagram obtained by removing the subdiagram $\Gamma_1$ from $\Delta$. When passing from $\partial\Delta$ to $\partial\Delta'$, one replaces each $\theta$-edge of $\textbf{y}_1$ with the corresponding $\theta$-edge of $\textbf{z}_1$ belonging to the same $\theta$-band. But since $\pazocal{B}_1$ is removed, there is at least one less $q$-edge separating any of the $h_1h_2$ (correctly ordered) pairs of $\theta$-edges described above. So, $\mu(\Delta)-\mu(\Delta')\geq h_1h_2$ by Lemma \ref{mixtures}(d).

Letting $|\partial\Gamma_1|_a=\a$, Lemma \ref{comb weights} yields $\text{wt}_G(\Gamma_1)\leq\text{wt}(\Gamma_1)\leq c_0Kh_1^2+2\a h_1+C_1(Kh_1+\a)^2$.

By Lemma \ref{lengths}(b), we have $|\textbf{z}_1|=h_1$. Moreover, each of the $h_1$ $(\theta,q)$-cells of $\pazocal{B}_1$ contributes at most one $a$-edge to $\textbf{z}_1$.

So, $\textbf{y}_1$ consists of $h_1$ $\theta$-edges, at least two $q$-edges, and at least $\max(0,\a-h_1)$ $a$-edges. Lemma \ref{lengths}(a) then implies $|\textbf{y}_1|\geq\max(h_1+2,h_1+2+(\a-2h_1)\delta)$.

Letting $\textbf{s}$ be the complement of $\textbf{y}_1$ in $\partial\Delta$, $\textbf{s}$ is also the complement of $\textbf{z}_1^{-1}$ in $\partial\Delta'$. So, Lemma \ref{lengths}(c) implies that $|\partial\Delta'|\leq|\textbf{z}_1|+|\textbf{s}|=h_1+|\textbf{s}|$ and 
$$|\partial\Delta|\geq|\textbf{y}_1|+|\textbf{s}|-2\delta\geq h_1+|\textbf{s}|+2-2\delta+\max(0,(\a-2h_1)\delta)$$ 
Hence, taking $\delta^{-1}>2$, we have
\begin{equation} \label{6.16 difference in perimeter}
|\partial\Delta|-|\partial\Delta'|\geq\gamma=\max(1,(\a-2h_1)\delta)
\end{equation}
In particular, $|\partial\Delta'|<|\partial\Delta|$, so that the inductive hypothesis implies
$$\text{wt}_G(\Delta')\leq N_2|\partial\Delta'|+N_1\mu(\Delta')\leq N_2(|\partial\Delta|-\gamma)^2+N_1(\mu(\Delta)-h_1h_2)$$
Noting that $\gamma\leq|\partial\Delta|$, we have $(|\partial\Delta|-\gamma)^2\leq|\partial\Delta|^2-\gamma|\partial\Delta|$, so that
$$\text{wt}_G(\Delta')\leq N_2|\partial\Delta|^2-N_2\gamma|\partial\Delta|+N_1\mu(\Delta)-N_1h_1h_2$$
Combining this with the $G$-weight of $\Gamma_1$, Lemma \ref{G-weight subdiagrams} then implies:
$$\text{wt}_G(\Delta)\leq N_2|\partial\Delta|^2-N_2\gamma|\partial\Delta|+N_1\mu(\Delta)-N_1h_1h_2+c_0Kh_1^2+2\a h_1+C_1(Kh_1+\a)^2$$
So, in order to reach the contradiction $\text{wt}_G(\Delta)\leq N_2|\partial\Delta|^2+N_1\mu(\Delta)$, it suffices to show:
$$-N_2\gamma|\partial\Delta|-N_1h_1h_2+c_0Kh_1^2+2\a h_1+C_1(Kh_1+\a)^2\leq0$$
As $h_1\leq h_2$, this amounts to proving:
\begin{equation}\label{6.16 suffices 1}
(c_0K+C_1K^2)h_1^2+2(C_1K+1)\a h_1+C_1\a^2\leq N_2\gamma|\partial\Delta|+N_1h_1^2
\end{equation}
If $\a\leq4h_1$, then the inequality (\ref{6.16 suffices 1}) follows from the parameter choice of $N_1$, as it is chosen after $c_0$, $K$, and $C_1$.

Otherwise, we have $\a>4h_1$. The parameter choice $N_1\geq c_0K+C_1K^2$ means that it suffices only to show that:
\begin{equation} \label{6.16 suffices 2}
\left(\frac{C_1K+1}{2}+C_1\right)\a^2\leq N_2\gamma|\partial\Delta|
\end{equation}
But then $\a-2h_1\geq\a/2$, so that $\gamma\geq\frac{1}{2}\delta\a$. Hence, $N_2\gamma|\partial\Delta|\geq\frac{1}{4}N_2\delta^2\a^2$, so that (\ref{6.16 suffices 2}) follows from the parameter choices $N_2>>C_1>>\delta^{-1}>>K$.

Now suppose we have a minimal counterexample of the form (2).

As each derivative subcomb of $\Gamma$ is connected with the handle $\pazocal{B}$ of $\Gamma$ by $\theta$-bands, they can be ordered in a natural way. 

Consider two neighbor derivative subcombs, $\Gamma_1$ and $\Gamma_2$. The handle of $\Gamma_i$ is intersected by two disjoint collections of $\theta$-bands which connect them with $\pazocal{B}$. If there is any $\theta$-band between these two collections, then it cannot intersect any $q$-bands except for $\pazocal{B}$, as otherwise it intersects a derivative subcomb between $\Gamma_1$ and $\Gamma_2$. So, the subpath $\textbf{x}$ of $\partial\Delta$ between the handles of $\Gamma_1$ and $\Gamma_2$ satisfies $|\textbf{x}|_q=0$. 

Hence, $\Gamma_1$ and $\Gamma_2$ form a contradiction to (1). However, $\text{wt}_G(\Gamma_1)+\text{wt}_G(\Gamma_2)<\text{wt}_G(\Gamma)=W$ since they contain no cells of $\pazocal{B}$, contradicting the minimality of the counterexample.

Thus, $\Gamma$ contains at most one derivative subcomb $\Gamma'$. In turn, $\Gamma'$ contains at most one derivative subcomb $\Gamma''$, and so on. Thus, there are $s$ maximal $q$-bands in $\Gamma$, so that Lemma \ref{M_a no annuli 1} implies that $|\partial\Gamma|_q=2s$.

\end{proof}

Similarly, the next statement is a direct analogue of Lemma 6.14 in [16] and Lemma 6.17 in [23] with altered estimates.

\begin{lemma} \label{6.17}

Suppose $\Gamma$ is a subcomb of $\Delta$ whose basic width is at most $K_0$ and whose handle $\pazocal{B}$ has length $\ell$. If $\Gamma'$ is a subcomb of $\Gamma$ with handle $\pazocal{B}'$ of length $\ell'$, then $\ell'>\ell/2$.

\end{lemma}

\begin{proof}

Assume toward contradiction that $\Gamma'$ is a subcomb of $\Gamma$ whose handle $\pazocal{B}'$ has length $\ell'\leq\ell/2$. Then, we can choose $\Gamma'$ so that $\ell'$ is minimal for all subcombs in $\Gamma$ and so that $\Gamma'$ has no proper subcombs, i.e the basic width of $\Gamma'$ is 1. Then, letting $\a=|\partial\Gamma'|_a$, Lemma \ref{comb weights} implies
$$\text{wt}_G(\Gamma')\leq\text{wt}(\Gamma')\leq c_0(\ell')^2+2\a\ell'+C_1(\ell'+\a)^2$$
Let $\Delta'$ be the diagram obtained from $\Delta$ by removing $\Gamma'$. Then the following inequality arises as the analogue of (\ref{6.16 difference in perimeter}):
\begin{equation} \label{6.17 difference in perimeter}
|\partial\Delta|-|\partial\Delta'|\geq\gamma=\max(1,(\a-2\ell')\delta)
\end{equation}
In particular, $|\partial\Delta'|<|\partial\Delta|$, so that
\begin{equation} \label{6.17 G-weight 1}
\text{wt}_G(\Delta')\leq N_2|\partial\Delta'|^2+N_1\mu(\Delta')\leq N_2(|\partial\Delta|-\gamma)^2+N_1\mu(\Delta')
\end{equation}

Every maximal $\theta$-band of $\Gamma$ passing through $\pazocal{B}'$ connects it to $\pazocal{B}$. The cells of $\pazocal{B}$ that such bands end on form a subband $\pazocal{C}$ of $\pazocal{B}$ with length $\ell'$. 

Then, the maximal $\theta$-bands of $\Gamma$ starting from $\pazocal{C}$ bound a comb with handle $\pazocal{C}$. So, there exists a maximal subdiagram $\Gamma''$ of $\Gamma$ which is a comb with handle $\pazocal{C}$. Note that $\Gamma''$ contains $\Gamma'$.

The components of $\pazocal{B}\setminus\pazocal{C}$ are handles of combs $E_1$ and $E_2$, respectively, which comprise the complement of $\Gamma''$ in $\Gamma$. Letting $\ell_i$ be the height of $E_i$, we then have $\ell_1+\ell_2=\ell-\ell'\geq\ell'$. 
%Moreover, the maximal $\theta$-bands of $E_1$ and $E_2$ do not cross $\pazocal{B}'$.

Let $\partial\Gamma=\textbf{yz}$ be the factorization such that $\textbf{z}=\textbf{bot}(\pazocal{B})$ and $\textbf{y}$ is a subpath of $\partial\Delta$. So, there are $\ell_i$ $\theta$-edges on the common subpath $\textbf{x}_i$ of $\textbf{y}$ and $\partial E_i$ and $\ell'$ $\theta$-edges on the common subpath $\textbf{x}$ of $\textbf{y}$ and $\partial\Gamma''$. Further, as the basic width of $\Gamma$ is at most $K_0$, Lemma \ref{6.16}(2) implies that $\textbf{y}$ contains at most $K$ $q$-edges.

So, for any edge from $\textbf{x}$ and any edge from $\textbf{x}_i$, there are at most $K$ $q$-edges between the pair in $\textbf{y}$, and so at most $J$ such edges by the choice of parameters. Hence, each of these $\ell'(\ell_1+\ell_2)$ (appropriately ordered) pairs of edges contributes to $\mu(\Delta)$. 

When passing from $\Delta$ to $\Delta'$, the $\theta$-edges of $\textbf{y}$ are replaced with the corresponding edges of $\textbf{bot}(\pazocal{B}')$. However, the $q$-edges of $\pazocal{B}'$ are removed, so that there is at least one less $q$-edge between a $\theta$-edge corresponding to an edge of $\textbf{x}$ and one corresponding to an edge of $\textbf{x}_i$. So, Lemma \ref{mixtures}(d) implies $\mu(\Delta)-\mu(\Delta')\geq\ell'(\ell_1+\ell_2)$. Substituting this into (\ref{6.17 G-weight 1}) then implies
\begin{equation} \label{6.17 G-weight 2}
\text{wt}_G(\Delta')\leq N_2(|\partial\Delta|-\gamma)^2+N_1\mu(\Delta)-N_1\ell'(\ell_1+\ell_2)
\end{equation}

\begin{figure}[H]
\centering
\includegraphics[scale=1]{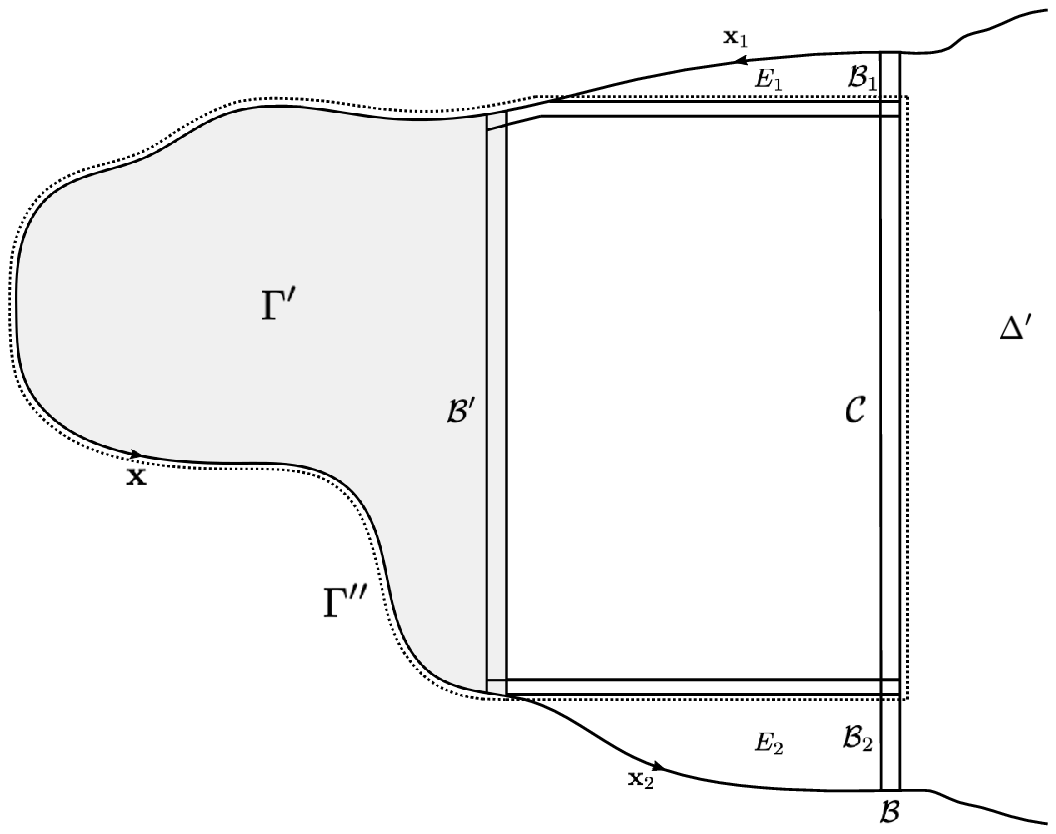}
\caption{Lemma \ref{6.17}}
\end{figure}

Note that $|\partial\Delta|\geq\gamma$, so that $(|\partial\Delta|-\gamma)^2\leq|\partial\Delta|^2-\gamma|\partial\Delta|$. Factoring in $\Gamma'$ and applying Lemma \ref{G-weight subdiagrams} then yields
$$
\text{wt}_G(\Delta)\leq N_2|\partial\Delta|^2-N_2\gamma|\partial\Delta|+N_1\mu(\Delta)-N_1\ell'(\ell_1+\ell_2)+c_0(\ell')^2+2\a\ell'+C_1(\ell'+\a)^2
$$
So, it suffices to show
\begin{equation} \label{suffices 1}
-N_2\gamma|\partial\Delta|-N_1\ell'(\ell_1+\ell_2)+c_0(\ell')^2+2\a\ell'+C_1(\ell'+\a)^2\leq0
\end{equation}

Suppose $\a\leq4\ell'$. Then $c_0(\ell')^2+2\a\ell'+C_1(\ell'+\a)^2\leq c_0(\ell')^2+8(\ell')^2+C_1(5\ell')^2$. As $\ell_1+\ell_2\geq\ell'$, (\ref{suffices 1}) then follows from the parameter choices $N_1>>C_1>>c_0$.

Otherwise, $\a>4\ell'$, so that $\gamma\geq\frac{1}{2}\delta\a$. Hence, $|\partial\Delta|\geq\gamma$ and (\ref{suffices 1}) imply that it suffices to show
\begin{equation} \label{suffices 2}
c_0(\ell')^2+2\a\ell'+C_1(\ell'+\a)^2\leq\frac{1}{4}N_2\delta^2\a^2+N_1(\ell')^2
\end{equation}

Note that $2\a\ell'+C_1(\ell'+\a)^2\leq\left(\frac{25C_1+8}{16}\right)\a^2$, so that (\ref{suffices 2}) follows by the parameter choices $N_1>>c_0$ and $N_2>>C_1>>\delta^{-1}$.

\end{proof}

\begin{lemma} \label{6.18}

If $\pazocal{T}$ is a quasi-rim $\theta$-band in $\Delta$, then the base of $\pazocal{T}$ has length $s>K$.

\end{lemma}

\begin{proof}

Suppose $\pazocal{T}$ is a quasi-rim $\theta$-band in $\Delta$ with base of length $s\leq K$. Without loss of generality, say that any cell between $\textbf{top}(\pazocal{T})$ and $\partial\Delta$ is an $a$-cell. Let $\textbf{P}_1$ be the set of such $a$-cells.

Let $\textbf{u}$ be the subpath of $\partial\Delta$ bounded by the two end $\theta$-edges of $\pazocal{T}$ and $\textbf{v}$ be its complement in $\partial\Delta$. For $\pi\in\textbf{P}_1$, factor $\partial\pi=\textbf{p}_\pi\textbf{p}_\pi'$ where $\textbf{p}_\pi$ is a subpath of $\textbf{u}$ and $\textbf{p}_\pi'$ is a subpath of $\textbf{top}(\pazocal{T})$. 

Let $b_\pi$ be the number of edges of $\textbf{p}_\pi'$ that are on the boundary of a $(\theta,q)$-cell of $\pazocal{T}$. 

By Lemma \ref{a-cell in counterexample}, $\|\textbf{p}_\pi'\|\geq\frac{1}{3}\|\partial\pi\|$. Further, by Lemma \ref{Lemma 6.2}, $\|\partial\pi\|\geq(1-\b)n\geq n/2$ by the parameter choice for $\b$ (see Section 2.8). As a result, $\|\textbf{p}_\pi'\|\geq3$ and $b_\pi\leq2$, so that $\textbf{p}_\pi'$ has a maximal subpath $\textbf{p}_\pi''$ consisting of edges on the boundary of $(\theta,a)$-cells of $\pazocal{T}$.

Consider the diagram $\Delta'$ obtained from $\Delta$ by cutting along $\textbf{bot}(\pazocal{T})$, removing $\pazocal{T}$ and the $a$-cells of $\textbf{P}_1$. For $\pi\in\textbf{P}_1$, the subpath $\textbf{p}_\pi''$ can be identified with a subpath of $\textbf{bot}(\pazocal{T})$, so that we may paste $\pi$ to $\Delta'$ along this subpath. 

Let $\Delta''$ be the diagram obtained by pasting all cells of $\textbf{P}_1$ to $\Delta'$. Note that $\textbf{v}$ can be identified with a subpath of $\partial\Delta''$. Let $\textbf{u}''$ be the complement of $\textbf{v}$ in $\partial\Delta''$.

\renewcommand\thesubfigure{\alph{subfigure}}
\begin{figure}[H]
\centering
\begin{subfigure}[b]{0.48\textwidth}
\centering
\includegraphics[scale=1.25]{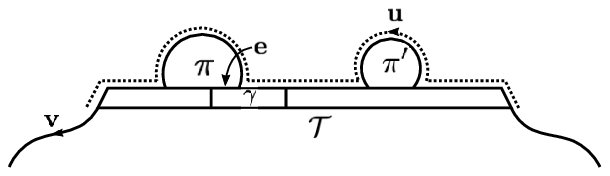}
\caption{$\Delta$, $\gamma$ a $(\theta,q)$-cell}
\end{subfigure}\hfill
\begin{subfigure}[b]{0.48\textwidth}
\centering
\includegraphics[scale=1.25]{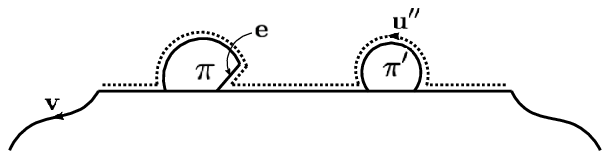}
\caption{$\Delta''$}
\end{subfigure}
\caption{Lemma \ref{6.18}}
\end{figure}

For any $\pi\in\textbf{P}_1$, the edges of $\partial\pi$ contributing to $b_\pi$ belong to $\textbf{u}''$ after this pasting. So, at least $\|\textbf{p}_\pi\|+b_\pi\geq\frac{1}{2}\|\partial\pi\|$ edges of $\partial\pi$ are shared with $\partial\Delta''$. It is thus clear from construction that $\Delta''$ is $M$-minimal.

Meanwhile, by Lemma \ref{simplify rules}, each $(\theta,q)$-cell of $\pazocal{T}$ contributes at most two $a$-edges to $\textbf{bot}(\pazocal{T})$. Any other edge of $\textbf{u}''$ corresponds to an edge of $\textbf{u}$.

As each $a$-edge contributing to $b_\pi$ for some $\pi\in\textbf{P}$ is labelled by a letter from the alphabet of the `special' input sector and each $(\theta,q)$-relation has at most one such letter, $\sum b_\pi\leq s$. So, since two $\theta$-edges are removed from $\textbf{u}$, Lemma \ref{lengths} implies
$$|\textbf{u}|-|\textbf{u}''|\geq2-(2s+2)\delta-\delta\sum b_\pi\geq2-(3s+2)\delta\geq2-(3K+2)\delta\geq1$$
The parameter choice $\delta^{-1}>>K$ and Lemma \ref{lengths} then imply
$$|\partial\Delta|-|\partial\Delta''|\geq(|\textbf{u}|+|\textbf{v}|-\delta)-(|\textbf{u}''|+|\textbf{v}|)\geq2-(3K+3)\delta\geq1$$
Hence, the inductive hypothesis may be applied to $\Delta''$, so that
$$\text{wt}_G(\Delta'')\leq N_2|\partial\Delta''|^2+N_1\mu(\Delta'')\leq N_2(|\partial\Delta|-1)^2+N_1\mu(\Delta'')$$
Note that the necklace corresponding to $\Delta''$ is obtained from that corresponding to $\Delta$ by the removal of two white beads. Lemma \ref{mixtures}(a) then yields $\mu(\Delta'')\leq\mu(\Delta)$.

Let $\textbf{P}''$ be a minimal covering of $\Delta''$. As each $a$-cell of $\textbf{P}_1$ has a boundary edge shared with $\partial\Delta''$, it cannot be contained in a trapezium in $\Delta''$. So, $\textbf{P}_1\subset\textbf{P}''$.

Let $\textbf{P}$ be the covering of $\Delta$ given by $\textbf{P}''$ and the cells of $\pazocal{T}$. Then for $\ell$ the length of $\pazocal{T}$,
$$\text{wt}_G(\Delta)\leq\text{wt}_G(\textbf{P})=\textbf{wt}_G(\textbf{P}'')+\ell\leq N_2|\partial\Delta|^2-N_2|\partial\Delta|+N_1\mu(\Delta)+\ell$$
Hence, it suffices to show that $N_2|\partial\Delta|\geq\ell$. 

For $\pi\in\textbf{P}_1$, (MM1) implies $\|\textbf{p}_\pi'\|\leq\|\textbf{p}_\pi\|+2b_\pi$. So, since $\sum b_\pi\leq s$, $|\textbf{top}(\pazocal{T})|_a\leq|\textbf{u}|_a+2s$.

By Lemma \ref{lengths}(d), $\ell\leq|\textbf{top}(\pazocal{T})|_a+3|\textbf{top}(\pazocal{T})|_q\leq|\textbf{u}|_a+5s$. As each $a$-edge of $\textbf{u}$ contributes at least $\delta$ to $|\Delta|$ and there are $s$ $q$-edges of $\textbf{u}$, $\ell\leq\delta^{-1}(|\partial\Delta|-s)+5s\leq\delta^{-1}|\partial\Delta|$.

But then the statement follows from the parameter choice $N_2>>\delta^{-1}$.

\end{proof}

Thus, Lemmas \ref{tight subcomb} and \ref{6.18} imply that there exists a tight subcomb $\Gamma$ in $\Delta$. By the definition of tight combs, the basic width of $\Gamma$ is at most $K_0$ (see Figure 8.9).

\begin{figure}[H]
\centering
\includegraphics[height=3.5in]{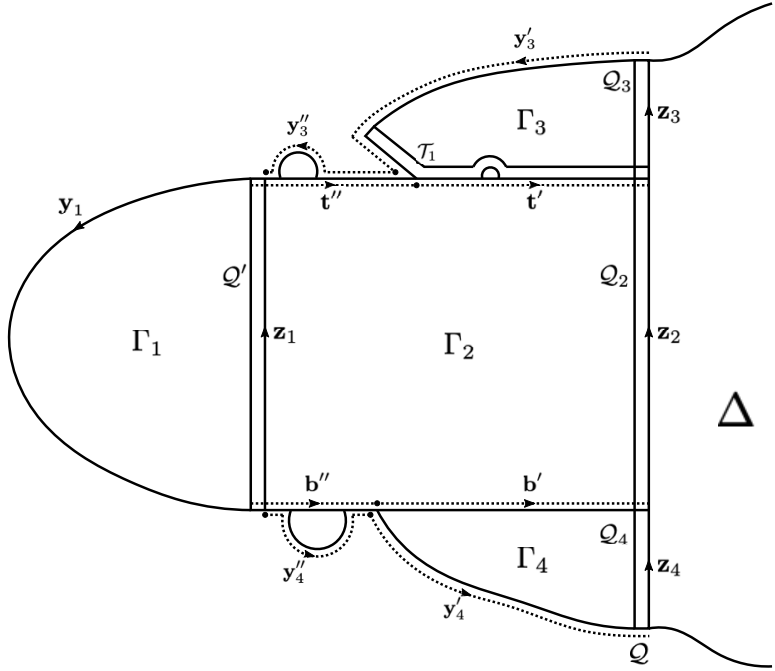}
\caption{Tight subcomb $\Gamma$}
\end{figure}

Let $\pazocal{T}$ be a maximal $\theta$-band in $\Gamma$ with tight base $B$. Then $B$ has the form $uxvx$, where $x$ does not occur in $u$ or $v$ and the final letter corresponds to the handle $\pazocal{Q}$ of $\Gamma$. Let $\pazocal{Q}'$ be the $q$-band corresponding to the first occurrence of $x$ in $B$.

Every maximal $\theta$-band in $\Gamma$ crossing $\pazocal{Q}'$ has a subband connecting $\pazocal{Q}$ and $\pazocal{Q}'$. The $(\theta,q)$-cells of $\pazocal{Q}$ on which these $\theta$-bands end form a subband $\pazocal{Q}_2$ of $\pazocal{Q}$ with length $\ell'$.

Let $\Gamma_2$ be the $a$-trapezium with side $q$-bands $\pazocal{Q}'$ and $\pazocal{Q}_2$ and bounded by the $\theta$-bands connecting these two. By the definition of tight, the base of $\Gamma_2$ is revolving.

Cutting along $\textbf{bot}(\pazocal{Q}')$ separates $\Delta$ into two components, one of which and is a subcomb $\Gamma_1$ with handle $\pazocal{Q}'$. Further, $\Gamma'=\Gamma_1\cup\Gamma_2$ is a comb contained in $\Gamma$ with handle $\pazocal{Q}_2$.

Let $\pazocal{Q}_3$ and $\pazocal{Q}_4$ be the components of $\pazocal{Q}\setminus\pazocal{Q}_2$. Then there exist maximal subdiagrams $\Gamma_3$ and $\Gamma_4$ of $\Gamma$ that are combs with handles $\pazocal{Q}_3$ and $\pazocal{Q}_4$, respectively.

Let $\ell$, $\ell'$, $\ell_3$, and $\ell_4$ be the heights of $\Gamma$, $\Gamma_1$, $\Gamma_3$, and $\Gamma_4$, respectively. By Lemma \ref{6.17}, $\ell'>\ell/2$.

Let $\partial\Gamma=\textbf{yz}$ be the factorization given by $\textbf{z}=\textbf{bot}(\pazocal{Q})$. Similarly, let $\partial\Gamma_1=\textbf{y}_1\textbf{z}_1$ be the factorization given by $\textbf{z}_1=\textbf{bot}(\pazocal{Q}')$.

Note that $\textbf{t}\defeq\textbf{ttop}(\Gamma_2)$ can be factored as $\textbf{t}=\textbf{t}''\textbf{t}'$ such that $\textbf{t}'$ is a maximal subpath shared with $\partial\Gamma_3$. Similarly, $\textbf{b}\defeq\textbf{tbot}(\Gamma_2)$ has a factorization $\textbf{b}=\textbf{b}''\textbf{b}'$ such that $(\textbf{b}')^{-1}$ is a maximal subpath shared with $\partial\Gamma_4$. Note that $\Gamma\setminus(\Gamma'\sqcup\Gamma_3\sqcup\Gamma_4)$ consists of $a$-cells that are attached to $\textbf{t}''$ or $\textbf{b}''$.

Factor $\textbf{y}=\textbf{y}_3\textbf{y}_1\textbf{y}_4$. Then, $\textbf{y}_3$ can be factored as $\textbf{y}_3'\textbf{y}_3''$ where $\textbf{y}_3'$ is a maximal subpath shared with $\partial\Gamma_3$. Note that every edge of $\textbf{y}_3''$ is either shared with $(\textbf{t}'')^{-1}$ or is on the boundary of an $a$-cell attached to $\textbf{t}''$. Similarly, we may factor $\textbf{y}_4=\textbf{y}_4''\textbf{y}_4'$.

Finally, factor $\textbf{z}=\textbf{z}_4\textbf{z}_2\textbf{z}_3$ where $\textbf{z}_i=\textbf{bot}(\pazocal{Q}_i)$.

Let $\pazocal{T}_1$ be the bottom $\theta$-band of $\Gamma_3$. Then, let $\Gamma_3'$ be the comb contained in $\Gamma_3$ obtained by removing any $a$-cells below $\pazocal{T}_1$. So, $\textbf{bot}(\pazocal{T}_1)$ is a subpath of $\partial\Gamma_3'$. 

Similarly, define $\Gamma_4'$ by removing any $a$-cells above the top $\theta$-band of $\Gamma_4$.

\smallskip

\begin{lemma} \label{counterexample combs} \

\begin{enumerate}[label=({\arabic*})] 

\item $|\textbf{t}''|_a\leq|\textbf{y}_3''|_a+4$ and $|\textbf{b}''|_a\leq|\textbf{y}_4''|_a+4$

%\item $|\textbf{y}_3''|_a+|\partial\Gamma_3|_a\geq|\textbf{ttop}(\Gamma_2)|_a-K$ and $|\textbf{y}_4''|_a+|\partial\Gamma_4|_a\geq|\textbf{tbot}(\Gamma_2)|_a-K$

\item $|\textbf{y}_3|_a\geq|\textbf{t}|_a-2\ell_3K-4$ and $|\textbf{y}_4|_a\geq|\textbf{b}|_a-2\ell_4K-4$.

\end{enumerate}

\end{lemma}

\begin{proof}

Let $\textbf{E}$ be the set of (unoriented) $a$-edges of $\textbf{t}$. Then $\textbf{E}=\sqcup_{i=1}^4\textbf{E}_i$, where:

\begin{itemize}

\item $\textbf{E}_1$ is the set of such edges shared with $\partial\Delta$,

\item $\textbf{E}_2$ is the set of such edges shared with the boundary of an $a$-cell not contained in $\Gamma_3$,

\item $\textbf{E}_3$ is the set of such edges shared with the boundary of an $a$-cell contained in $\Gamma_3$, and

\item $\textbf{E}_4$ is the set of such edges shared with $\textbf{bot}(\pazocal{T}_1)$.

\end{itemize}

Further, let $\textbf{E}_2'$ (resp $\textbf{E}_3'$) be the subset of $\textbf{E}_2$ (resp $\textbf{E}_3$) consisting of the edges which are on the boundary of a $(\theta,q)$-cell in $\Gamma_2$.

Let $\a_i=\#\textbf{E}_i$ for all $1\leq i\leq 4$. Similarly, let $\a_i'=\#\textbf{E}_i'$ for $i=2,3$.

Note that $\a_1+\a_2=|\textbf{t}''|_a$ and $\textbf{E}_1$ is a subset of the $a$-edges of $\textbf{y}_3''$.

Further, every edge of $\textbf{E}_2'\sqcup\textbf{E}_3'$ is labelled by a letter from the `special' input sector and is on the boundary of a $(\theta,q)$-cell of the same $\theta$-band (the top $\theta$-band of $\Gamma_2$). Any $(\theta,q)$-cell with such an edge on its boundary must correspond to the base letter $Q_0(1)^{\pm1}$, in which case it has exactly one such letter on its boundary. Hence, by the definition of revolving, $\a_2'+\a_3'\leq2$.

For any $a$-edge $\textbf{e}$ of $\textbf{t}''$, either $\textbf{e}\in\textbf{E}_1$ or $\textbf{e}$ is on the boundary of some $a$-cell $\pi$. In the latter case, $\partial\pi$ can be factored as $\textbf{s}_1\textbf{s}_2$, where $\textbf{s}_1$ is a subpath of $\textbf{y}_3''$ and $\textbf{s}_2$ is a subpath of $\textbf{t}''$. Let $b$ be the number of edges of $\textbf{s}_2$ that are on the boundary of a $(\theta,q)$-cell contained in $\Gamma_2$. Then by condition (MM1), $\|\textbf{s}_2\|-b\leq\frac{1}{2}\|\partial\pi\|$. So, $\|\textbf{s}_1\|\geq\|\textbf{s}_2\|-2b$.

Applying the same reasoning to all such $a$-cells, we have $\a_1+\max(0,\a_2-2\a_2')\leq|\textbf{y}_3''|_a$.

Hence, $|\textbf{t}''|_a=\a_1+\a_2\leq|\textbf{y}_3''|_a+2\a_2'\leq|\textbf{y}_3''|_a+4$.

%As the edges of $\textbf{E}_3$ and $\textbf{E}_4$ are on $\partial\Gamma_3$ below $\pazocal{T}_1$, we then have:
%$$|\textbf{t}|_a=\a_1+\a_2+\a_3+\a_4\leq|\textbf{y}_3''|_a+2\a_2'+\a_3+\a_4\leq|\textbf{y}_3''|_a+|\partial\Gamma_3|_a+K$$
Next, let $\textbf{F}_3$ be the edges of $\textbf{y}_3'$ on the boundary of an $a$-cell below $\textbf{bot}(\pazocal{T}_1)$ and set $\b=\#\textbf{F}_3$. 

Let $\pi$ be an $a$-cell with a boundary edge contributing to $\a_3$. Then as above, for $b$ the number of edges of $\partial\pi$ on the boundary of a $(\theta,q)$-cell of $\Gamma_2$, at most $\frac{1}{2}\|\partial\pi\|+b$ of the edges of $\partial\pi$ are shared with $\textbf{t}'$. Note that the other edges of $\partial\pi$ are either part of $\textbf{bot}(\pazocal{T}_1)$ or contribute to $\b$.

So, $\a_4+\max(0,\a_3-2\a_3')\leq\b+|\textbf{bot}(\pazocal{T}_1)|_a$.

Hence, $|\textbf{t}|_a=\a_1+\a_2+\a_3+\a_4\leq|\textbf{y}_3''|_a+\b+|\textbf{bot}(\pazocal{T}_1)|_a+2(\a_2'+\a_3')\leq|\textbf{y}_3''|_a+\b+|\textbf{bot}(\pazocal{T}_1)|_a+4$.

Let $\partial\Gamma_3'=\textbf{x}'\textbf{x}\textbf{z}_3$ be the factorization given by $\textbf{x}=\textbf{bot}(\pazocal{T}_1)$. Note that $\textbf{x}'$ is a subpath of $\textbf{y}_3'$ not containing any edges of $\textbf{F}_3$, so that $\b+|\textbf{x}'|_a\leq|\textbf{y}_3'|_a$.

Then, applying Lemma \ref{comb weights} to $\Gamma_3'$, we have $|\textbf{x}|_a\leq|\textbf{x}'|_a+4K_0\ell_3$.

Thus, $|\textbf{t}|_a\leq|\textbf{y}_3''|_a+\b+|\textbf{x}'|_a+4K_0\ell_3+4\leq|\textbf{y}_3|_a+2\ell_3K+4$.

Applying the analogous argument to $\Gamma_4$ yields the inequalities $|\textbf{b}''|_a\leq|\textbf{y}_3''|_a+4$ and $|\textbf{b}|_a\leq|\textbf{y}_4|_a+2\ell_4K+4$.

\end{proof}

\newpage

\begin{lemma} \label{H>M}

Set $M=\max(|\textbf{b}|_a,|\textbf{t}|_a)$. Then $2K\ell>M$.

\end{lemma}

\begin{proof}

As $|\textbf{y}_i|_a\leq|\textbf{y}|_a$ and $\ell_i\leq\ell/2$ for $i=3,4$, Lemma \ref{counterexample combs}(2) implies
$$M\leq|\textbf{y}|_a+K\ell+4\leq|\textbf{y}|_a+3K\ell/2$$
Assuming that $2K\ell\leq M$, we then have $|\textbf{y}|_a\geq \frac{1}{2}K\ell\geq K_0\ell$. 

By Lemma \ref{comb weights}, we have 
\begin{align*}
\text{wt}(\Gamma)&\leq c_0K_0\ell^2+2(|\textbf{y}|_a+\ell)\ell+C_1(K_0\ell+|\textbf{y}|_a+\ell)^2 \\
&\leq(c_0K_0+2+C_1(K_0+1)^2)\ell^2+(2+2C_1(K_0+1))|\textbf{y}|_a\ell+C_1|\textbf{y}|_a^2
\end{align*}
So, as $C_2$ is chosen after $C_1$, $K_0$, and $c_0$, we have:
$$\text{wt}(\Gamma)\leq C_2|\textbf{y}|_a^2$$
Since $|\textbf{y}|_\theta=\ell$ and $|\textbf{y}|_q\geq2$, Lemma \ref{lengths}(a) implies that $|\textbf{y}|\geq2+\ell+(|\textbf{y}|_a-\ell)\delta\geq\ell+2+\frac{1}{2}\delta|\textbf{y}|_a$.

Let $\Delta'$ be the $M$-minimal diagram formed from $\Delta$ by removing $\Gamma$. Then in $\partial\Delta'$, $\textbf{y}$ is replaced with $\textbf{z}$. Lemma \ref{lengths}(b) implies that $|\textbf{z}|=\ell$.

Letting $\textbf{s}$ be the complement of $\textbf{z}$ in $\partial\Delta'$, Lemma \ref{lengths}(c) implies $|\partial\Delta'|\leq|\textbf{s}|+\ell$ and $|\partial\Delta|\geq|\textbf{s}|+|\textbf{y}|-2\delta$. So, $|\partial\Delta|-|\partial\Delta'|\geq\gamma=\max(1,\frac{1}{2}\delta|\textbf{y}|_a)$.

Hence, we may apply the inductive hypothesis to $\Delta'$, yielding
$$\text{wt}_G(\Delta')\leq N_2|\partial\Delta'|^2+N_1\mu(\Delta')\leq N_2(|\partial\Delta|-\gamma)^2+N_1\mu(\Delta')$$
As $\gamma\leq|\partial\Delta|$, $(|\partial\Delta|-\gamma)^2\leq|\partial\Delta|^2-\gamma|\partial\Delta|$. Lemma \ref{mixtures} further implies that $\mu(\Delta')\leq\mu(\Delta)$. So, adding in the weight of $\Gamma$, Lemma \ref{G-weight subdiagrams} implies:
$$\text{wt}_G(\Delta)\leq N_2|\partial\Delta|^2-N_2\gamma|\partial\Delta|+N_1\mu(\Delta)+C_2|\textbf{y}|_a^2$$
So, it suffices to show that $N_2\gamma|\partial\Delta|\geq C_2|\textbf{y}|_a^2$. 

But $\frac{1}{2}\delta|\textbf{y}|_a\leq\gamma\leq|\partial\Delta|$, so that $N_2\gamma|\partial\Delta|\geq\frac{1}{4}N_2\delta^2|\textbf{y}|_a^2$. So, the desired inequality follows from the parameter choices $N_2>>C_2>>\delta^{-1}$.

\end{proof}

\bigskip

\begin{lemma} \label{diskless}

The counterexample diagram $\Delta$ does not exist.

\end{lemma}

\begin{proof}

Let $\Delta_1$ be the diagram obtained from $\Delta$ by removing $\Gamma\setminus\pazocal{Q}$. 

As the base of $\Gamma_2$ is revolving, the bands $\pazocal{Q}'$ and $\pazocal{Q}_2$ are labelled identically. So, we may construct a diagram $\Delta_0$ by pasting $\Gamma_1$ to $\Delta_1$ along $\pazocal{Q}'$ and $\pazocal{Q}_2$ (see Figure 8.10, compare with Figure 8.9).

Since an $a$-band cannot cross a $q$-band, any counterexample to (MM1) or (MM2) in $\Delta_0$ is contained in one of $\Delta_1$ or $\Gamma_1$. But $\Delta_1$ and $\Gamma_1$ are $M$-minimal as subdiagrams of $\Delta$. Hence, $\Delta_0$ must be $M$-minimal.

%Similarly, any $\theta$-annulus $\pazocal{T}$ in $\Delta_0$ cannot be contained in one of $\Delta_1$ or $\Gamma_1$, and so must cross the $q$-band corresponding to the pasting of $\pazocal{Q}'$ and $\pazocal{Q}_2$. Consider the $\theta$-band $\pazocal{T}'$ of $\Delta$ obtained by replacing any cell of $\pazocal{Q}'$ that $\pazocal{T}$ crosses with the corresponding maximal $\theta$-band of $\Gamma_2$. But then $\pazocal{T}'$ is a $\theta$-annulus, contradicting the $M$-minimality of $\Delta$. Hence, $\Delta_0$ satisfies (MM1), and so is $M$-minimal.

Let $\textbf{P}$ be a covering of $\Delta_0$. Suppose there exists $P\in\textbf{P}$ that is not completely contained in $\Gamma_1$ or $\Delta_1$. Then $P$ is a big $a$-trapezium containing a maximal $q$-band $\pazocal{B}$ that is a subband of $\pazocal{Q}_2$ (and $\pazocal{Q}'$) and which is not a side $q$-band of $P$. Then, the history of $P$ is a subword of the history of $\Gamma_2$, so that $\Gamma_2$ is itself a big $a$-trapezium by Lemma \ref{M controlled}.

Let $\pazocal{Q}''$ be the maximal $q$-band of $\Gamma$ corresponding to the first letter of the base of $P$. Further, let $\pazocal{T}$ be a $\theta$-band in $\Gamma$ connecting $\pazocal{Q}$ to $\pazocal{Q}''$. Then the base of $\pazocal{T}$ (read toward $\pazocal{Q}$) is $B_1B_2$, where $B_1$ is a prefix of the base of $P$ and $B_2$ is the base of $\Gamma_2$. By Lemma \ref{M controlled}, both the base of $P$ and $B_2$ are reduced. Moreover, since the first letter of $B_2$ appears in the base of $P$, $B_1B_2$ must be reduced.

\begin{figure}[H]
\centering
\includegraphics[height=3.5in]{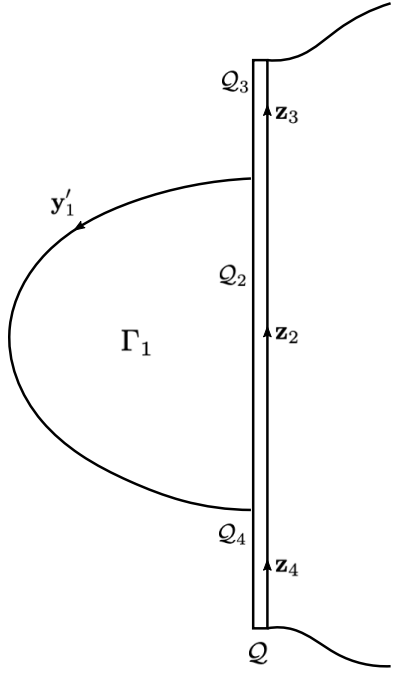}
\caption{The construction of $\Delta_0$}
\end{figure}

As $B_2$ is revolving, the first letter of $B_1$ appears in $B_2$. But then the base of the maximal $\theta$-band of $\Gamma$ containing $\pazocal{T}$ has a tight prefix, contradicting the assumption that $\Gamma$ is a tight comb.

So, for any covering of $\Delta_0$, each element is either contained completely in $\Gamma_1$ or completely in $\Delta_1$. Hence, given a minimal covering $\textbf{P}$ of $\Delta_0$, we may construct coverings $\textbf{P}'$ and $\textbf{P}''$ of $\Gamma_1$ and $\Delta_1$, respectively, by including only the elements belonging to these subdiagrams and perhaps adding in the cells of $\pazocal{Q}'$ or $\pazocal{Q}_2$. As at most the $\ell'$ cells of $\pazocal{Q}_2$ are counted twice in these coverings, we have $\text{wt}_G(\Delta_0)\geq\text{wt}_G(\Gamma_1)+\text{wt}_G(\Delta_1)-\ell'$.

Lemma \ref{G-weight subdiagrams} then implies:
\begin{equation} \label{G-weight}
\text{wt}_G(\Delta)\leq\text{wt}_G(\Delta_1)+\text{wt}_G(\Gamma)\leq\text{wt}_G(\Delta_0)+\text{wt}_G(\Gamma_2)+\text{wt}_G(\Gamma_3)+\text{wt}_G(\Gamma_4)+A+\ell'
\end{equation}
where $A$ is the sum of the weights of the $a$-cells attached to $\textbf{t}''$ or $\textbf{b}''$.

For $i=3,4$, note that the subpath $\textbf{y}_i''$ has no $\theta$-edges, while $\textbf{y}_i'$ consists of $\ell_i$ $\theta$-edges and at least one $q$-edge. So, Lemma \ref{lengths}(a) implies $|\textbf{y}_i|\geq1+\ell_i+\delta\max(0,|\textbf{y}_i|_a-\ell_i,|\textbf{y}_i''|_a-1)$.

Letting $\textbf{s}$ be the complement of $\textbf{y}$ in $\partial\Delta$, Lemma \ref{lengths}(c) then yields
$$|\partial\Delta|\geq|\textbf{s}|+|\textbf{y}|-2\delta\geq|\textbf{s}|+|\textbf{y}_3|+|\textbf{y}_1|+|\textbf{y}_4|-4\delta$$
Next, let $\textbf{y}_1'$ be the subpath of $\textbf{y}_1$ not containing the first or last edge. Note that both of these edges are $q$-edges corresponding to $\pazocal{Q}'$, so that $|\textbf{y}_1|=|\textbf{y}_1'|+2$.

Then, Lemma \ref{lengths}(c) implies $|\partial\Delta_0|\leq|\textbf{s}|+1+|\textbf{top}(\pazocal{Q}_3)|+|\textbf{y}_1'|+|\textbf{top}(\pazocal{Q}_4)|+1$. Further, Lemma \ref{lengths}(b) implies $|\textbf{top}(\pazocal{Q}_i)|=\ell_i$ for $i=3,4$. So, $|\partial\Delta_0|\leq|\textbf{s}|+|\textbf{y}_1|+\ell_3+\ell_4$.

Hence,
\begin{align*}
|\partial\Delta|-|\partial\Delta_0|\geq\gamma&=(|\textbf{y}_3|-\ell_3)+(|\textbf{y}_4|-\ell_4)-4\delta \\
&\geq2-4\delta+\delta\max(0,|\textbf{y}_3|_a-\ell_3,|\textbf{y}_3''|_a-1)+\delta\max(0,|\textbf{y}_4|_a-\ell_4,|\textbf{y}_4''|_a-1)
\end{align*}
So, taking $\delta^{-1}\geq4$, $|\partial\Delta|-|\partial\Delta_0|\geq\gamma\geq2-4\delta\geq1$.

%2-4\delta+\delta(\max(0,|\textbf{y}_3|_a-\ell_3)+\max(0,|\textbf{y}_4|_a-\ell_4))\geq1$$

Hence, we may apply the inductive hypothesis to $\Delta_0$, so that
$$\text{wt}_G(\Delta_0)\leq N_2|\partial\Delta_0|^2+N_1\mu(\Delta_0)\leq N_2(|\partial\Delta|-\gamma)^2+N_1\mu(\Delta_0)$$
In $\textbf{y}$, any $\theta$-edge of $\textbf{y}_1$ is separated from a $\theta$-edge of $\textbf{y}_3$ or $\textbf{y}_4$ by a $q$-edge at the end of $\pazocal{Q}'$. Moreover, since the basic width of $\Gamma$ is at most $K_0$, the parameter choice $J>>K_0$ implies that each of these (correctly ordered) pairs contribute to $\mu(\Delta)$. But the black beads corresponding to $\pazocal{Q}'$ are removed in the formation of the necklace for $\Delta_0$, so that Lemma \ref{mixtures}(d) implies
$$\mu(\Delta)-\mu(\Delta_0)\geq\ell'(\ell_3+\ell_4)$$
Noting that $\gamma\leq|\partial\Delta|$, we then have:
$$\text{wt}_G(\Delta_0)\leq N_2|\partial\Delta|^2-N_2\gamma|\partial\Delta|+N_1\mu(\Delta)-N_1\ell'(\ell_3+\ell_4)$$
Hence, by (\ref{G-weight}), it suffices to show that:
\begin{equation} \label{diskless suffices 1}
N_2\gamma|\partial\Delta|+N_1\ell'(\ell_3+\ell_4)\geq\text{wt}_G(\Gamma_2)+\text{wt}_G(\Gamma_3)+\text{wt}_G(\Gamma_4)+A+\ell'
\end{equation}
Setting $\nu_i=|\partial\Gamma_i|_a$ for $i=3,4$, Lemma \ref{comb weights} implies:
$$\text{wt}_G(\Gamma_i)\leq\text{wt}(\Gamma_i)\leq c_0K_0\ell_i^2+2\nu_i\ell_i+C_1(K_0\ell_i+\nu_i)^2\leq C_2(\ell_i+\nu_i)^2$$
By Lemmas \ref{revolving G-weight},
\begin{align*}
\text{wt}_G(\Gamma_2)&\leq C_2\ell'\max(\|\textbf{t}\|,\|\textbf{b}\|)+C_2(\|\textbf{t}\|+\|\textbf{b}\|)^2 \\
&\leq C_2\ell'(M+K_0)+4C_2(M+K_0)^2
\end{align*}
For any $a$-cell $\pi$ whose weight contributes to $A$, $\pi$ is attached to either $\textbf{t}''$ or $\textbf{b}''$. Lemma \ref{a-cell in counterexample} then implies that at least a third of the edges of $\partial\pi$ are shared with $\textbf{t}''$ or $\textbf{b}''$. 

So, $A\leq C_1(3|\textbf{t}''|_a+3|\textbf{b}''|_a)^2\leq C_1(6M)^2\leq36C_1M^2$.

By Lemma \ref{H>M} and \ref{6.17}, $M\leq2K\ell\leq4K\ell'$. So, since $C_2>>C_1>>K$, $A\leq C_2\ell'M$.

Hence, the parameter choices $C_3>>C_2>>K>>K_0$ imply
\begin{align*}
\text{wt}_G(\Gamma_2)+A+\ell'&\leq C_2\ell'M+C_2K_0\ell'+4C_2M^2+8C_2K_0M+4C_2K_0^2+C_2\ell'M+\ell' \\
&\leq 2C_2\ell'M+16C_2K\ell'M+C_2K_0\ell'+32C_2K_0K\ell'+\ell'+4C_2K_0^2 \\
&\leq C_3\ell'M+C_3\ell'+C_3
\end{align*}
So, by (\ref{diskless suffices 1}), it suffices to show that:
\begin{equation} \label{diskless suffices 2}
N_2\gamma|\partial\Delta|+N_1\ell'(\ell_3+\ell_4)\geq C_3\ell'M+C_3\ell'+C_3+C_2(\ell_3+\nu_3)^2+C_2(\ell_4+\nu_4)^2
\end{equation}
Without loss of generality, assume $\nu_4\leq\nu_3$.

Note $M=\max(|\textbf{t}|_a,|\textbf{b}|_a)=\max(|\textbf{t}''|_a+|\textbf{t}'|_a,|\textbf{b}''|_a+|\textbf{b}'|_a)\leq\max(|\textbf{t}''|_a,|\textbf{b}''|_a)+\max(|\textbf{t}'|_a,|\textbf{b}'|_a)$. Since $\textbf{t}'$ and $\textbf{b}'$ are subpaths of $\partial\Gamma_3$ and $\partial\Gamma_4$, respectively, we then have $M\leq\max(|\textbf{t}''|_a,|\textbf{b}''|_a)+\nu_3$.

Lemma \ref{counterexample combs}(1) then yields $M\leq\max(|\textbf{y}_3''|_a,|\textbf{y}_4''|_a)+\nu_3+4$.

So, $C_3\ell'M\leq C_3\ell'\max(|\textbf{y}_3''|_a,|\textbf{y}_4''|_a)+C_3\ell'\nu_3+4C_3\ell'$.

As $\gamma\geq1$, $|\partial\Delta|\geq\ell'$, and $|\partial\Delta|\geq2$, the parameter choice $N_2>>C_3$ allows us to assume that $$\frac{1}{3}N_2\gamma|\partial\Delta|\geq5C_3\ell'+C_3$$
Hence, it suffices to show that:
\begin{equation} \label{diskless suffices 3}
\frac{2}{3}N_2\gamma|\partial\Delta|+N_1\ell'(\ell_3+\ell_4)\geq C_3\ell'\max(|\textbf{y}_3''|_a,|\textbf{y}_4''|_a)+C_3\ell'\nu_3+C_2(\ell_3+\nu_3)^2+C_2(\ell_4+\nu_4)^2
\end{equation}
Suppose $\max(|\textbf{y}_3''|_a,|\textbf{y}_4''|_a)\leq1$. Then since $\gamma\geq1$, $|\partial\Delta|\geq\ell'$, and $N_2>>C_3$, we may take 
\begin{equation} \label{diskless inequality}
\frac{1}{3}N_2\gamma|\partial\Delta|\geq C_3\ell'\max(|\textbf{y}_3''|_a,|\textbf{y}_4''|_a)
\end{equation}
Otherwise, recall that $\gamma\geq\delta(|\textbf{y}_3''|_a-1)+\delta(|\textbf{y}_4''|_a-1)\geq\delta(\max(|\textbf{y}_3''|_a,|\textbf{y}_4''|_a)-1)\geq\frac{1}{2}\delta\max(|\textbf{y}_3''|_a,|\textbf{y}_4''|_a)$. So, since $|\partial\Delta|\geq\ell'$, the parameter choices $N_2>>C_3>>\delta^{-1}$ allow us to again assume (\ref{diskless inequality}) holds.

Thus, by (\ref{diskless suffices 3}), it suffices to show that
\begin{equation} \label{diskless suffices 4}
\frac{1}{3}N_2\gamma|\partial\Delta|+N_1\ell'(\ell_3+\ell_4)\geq C_3\ell'\nu_3+C_2(\ell_3+\nu_3)^2+C_2(\ell_4+\nu_4)^2
\end{equation}

\smallskip

\ \textbf{1.} Suppose $\nu_3\leq 3J(\ell_3+\ell_4)$.

Then, for $i=3,4$, $\ell_i+\nu_i\leq4J(\ell_3+\ell_4)$. As Lemma \ref{6.17} implies $\ell_3+\ell_4\leq\ell'$, this implies $(\ell_i+\nu_i)^2\leq16J^2\ell'(\ell_3+\ell_4)$.

So, the parameter choices $C_3>>C_2>>J$ imply $C_2(\ell_3+\nu_3)^2+C_2(\ell_4+\nu_4)^2\leq C_3\ell'(\ell_3+\ell_4)$.

Hence, as $N_1>>C_3>>J$, we have $$N_1\ell'(\ell_3+\ell_4)\geq C_3\ell'\nu_3+C_2(\ell_3+\nu_3)^2+C_2(\ell_4+\nu_4)^2$$

\smallskip

Thus, we may assume that $\nu_3>3J(\ell_3+\ell_4)$.

As a result, for $i=3,4$, $\ell_i+\nu_i\leq\ell_3+\ell_4+\nu_3\leq2\nu_3$. So, $C_2(\ell_i+\nu_i)^2\leq4C_2\nu_3^2$.

It then follows from (\ref{diskless suffices 4}) that it suffices to show that
\begin{equation} \label{diskless suffices 5}
\frac{1}{3}N_2\gamma|\partial\Delta|+N_1\ell'(\ell_3+\ell_4)\geq C_3\ell'\nu_3+8C_2\nu_3^2
\end{equation}

\smallskip

\ \textbf{2.} Suppose $\nu_3\leq16$.

So, $C_3\ell'\nu_3+8C_2\nu_3^2\leq8C_3\ell'+C_3$ by the parameter choice $C_3>>C_2$.

As $|\partial\Delta|\geq\max(2,\ell')$ and $\gamma\geq1$, the parameter choices $N_2>>C_3>>C_2$ allow us to assume that
$$\frac{1}{3}N_2\gamma|\partial\Delta|\geq C_3\ell'\nu_3+8C_2\nu_3^2$$

\smallskip

\ \textbf{3.} Thus, it suffices to assume that $\nu_3>\max(3J(\ell_3+\ell_4),16)$ and show that (\ref{diskless suffices 5}) holds.

As any $a$-edge of $\partial\Gamma_3$ is part of $\textbf{y}_3'$, $\textbf{t}'$, or $\textbf{bot}(\pazocal{Q}_3)$, we have 
$$\nu_3\leq|\textbf{y}_3|_a+|\textbf{t}|_a+\ell_3$$
By Lemma \ref{counterexample combs}(2), this implies $\nu_3\leq2|\textbf{y}_3|_a+2K\ell_3+\ell_3+4\leq2|\textbf{y}_3|_a+J\ell_3+4$.

Note that $J\ell_3+4<\frac{1}{3}\nu_3+\frac{1}{4}\nu_3=\frac{5}{12}\nu_3$, so that $\frac{7}{12}\nu_3\leq2|\textbf{y}_3|_a$.

Recall that $\gamma\geq\delta(|\textbf{y}_3|_a-\ell_3)$. So, since $\ell_3<\frac{1}{3J}\nu_3\leq\frac{1}{24}\nu_3$ by taking $J\geq8$, we have $\gamma\geq\frac{1}{4}\delta\nu_3$.

As $|\partial\Delta|\geq\gamma$, we then have $\gamma|\partial\Delta|\geq\frac{1}{16}\delta^2\nu_3^2$. So, the parameter choices $N_2>>C_2>>\delta^{-1}$ allow us to assume
$$\frac{1}{6}N_2\gamma|\partial\Delta|\geq8C_2\nu_3^2$$
By (\ref{diskless suffices 5}), it then suffices to show:
\begin{equation} \label{diskless suffices 6}
\frac{1}{6}N_2\gamma|\partial\Delta|+N_1\ell'(\ell_3+\ell_4)\geq C_3\ell'\nu_3
\end{equation}
But $|\partial\Delta|\geq\ell'$, so that the parameter choices $N_2>>C_3>>\delta^{-1}$ give us:
$$\frac{1}{6}N_2\gamma|\partial\Delta|\geq\frac{1}{24}N_2\delta\ell'\nu_3\geq C_3\ell'\nu_3$$

Thus, (\ref{diskless suffices 6}) is satisfied, and so the statement is proved.

\end{proof}

\medskip

%%%%%%%%%%%%%%%%%%%%%%%%%%%%%%%%%%%%%%%%%%%%%%%%%%%%%%%%%%%%%%%%%

\section{Diagrams with disks}

\subsection{Diminished, Minimial, and $D$-minimal diagrams} \

A $q$-letter of the form $t(i)$ for $2\leq i\leq L$ is called a \textit{$t$-letter}. Accordingly, a $(\theta,q)$-relation corresponding to a $t$-letter is called a \textit{$(\theta,t)$-relation}. Note that for each rule $\theta$ and each $t$-letter, the corresponding $(\theta,t)$-relation is of the simple form $\theta_jt(i)=t(i)\theta_{j+1}$.

Now, we modify the definition of a reduced diagram over the canonical presentation of $M_\Omega(\textbf{M})$ or over the disk presentation of $G_\Omega(\textbf{M})$. To this end, we introduce the \textit{signature} of such a diagram $\Delta$ as the four-tuple $s(\Delta)=(\a_1,\a_2,\a_3,\a_4)$ where: 

\begin{addmargin}[1em]{0em}

$\bullet$ $\a_1$ is the number of disks in $\Delta$ (of course, this is zero if $\Delta$ is a diagram over $M_\Omega(\textbf{M})$), 

$\bullet$ $\a_2$ is the number of $(\theta,t)$-cells, 

$\bullet$ $\a_3$ is the total number of $(\theta,q)$-cells, and

$\bullet$ $\a_4$ is the number of $a$-cells.

\end{addmargin}

The signatures of reduced diagrams over the disk presentation of $G_\Omega(\textbf{M})$ are ordered lexicographically. In particular, if $\Delta$ and $\Gamma$ are such diagrams with $s(\Delta)=(\a_1,\a_2,\a_3,\a_4)$ and $s(\Gamma)=(\b_1,\b_2,\b_3,\b_4)$, then $s(\Delta)\leq s(\Gamma)$ if:

\begin{addmargin}[1em]{0em}

$\bullet$ $\a_1\leq\b_1$

$\bullet$ for $i\in\{2,3,4\}$, if $\a_j=\b_j$ for all $j<i$, then $\a_i\leq\b_i$

\end{addmargin}

A reduced diagram $\Delta$ over the disk presentation of $G_\Omega(\textbf{M})$ is called \textit{diminished} if for any reduced diagram $\Gamma$ with $\lab(\partial\Delta)\equiv\lab(\partial\Gamma)$, we have $s(\Delta)\leq s(\Gamma)$.

Given a reduced diagram $\Delta$ over the disk presentation of $G_\Omega(\textbf{M})$ with $s(\Delta)=(\a_1,\a_2,\a_3,\a_4)$, the \textit{2-signature} of $\Delta$ is the ordered pair $s_2(\Delta)=(\a_1,\a_2)$. The \textit{1-signature} $s_1(\Delta)$ is defined similarly and can be interpreted simply as the number of disks in $\Delta$ with the natural order on the natural numbers.

A reduced diagram $\Delta$ over the disk presentation of $G_\Omega(\textbf{M})$ is called \textit{$D$-minimal} if for any reduced diagram $\Gamma$ with $\lab(\partial\Gamma)\equiv\lab(\partial\Delta)$, $s_1(\Delta)\leq s_1(\Gamma)$. By the definition of the lexicographic order, a diminished diagram is necessarily $D$-minimal.

Finally, reduced diagram $\Delta$ over the disk presentation of $G_\Omega(\textbf{M})$ is called \textit{minimal} if:

\begin{addmargin}[1em]{0em}

\begin{enumerate}[label=(M{\arabic*})]

%\item it contains no $\theta$-annulus $S$ whose sides are labelled by letters of the tape alphabet of the `special' input sector,

\item for any $a$-cell $\pi$ and any $\theta$-band $\pazocal{T}$, at most half of the edges of $\partial\pi$ mark the start of an $a$-band that crosses $\pazocal{T}$,

\item no maximal $a$-band ends on two different $a$-cells, and

\item for any reduced diagram $\Gamma$ with $\text{Lab}(\partial\Delta)\equiv\text{Lab}(\partial\Gamma)$, $s_2(\Delta)\leq s_2(\Gamma)$.

\end{enumerate}

\end{addmargin}

Note that conditions (M1) and (M2) are equivalent to the conditions (MM1) and (MM2) in the definition of $M$-minimal. As a result, a minimal diagram containing no disks is necessarily $M$-minimal. Further, a diminished diagram necessarily satisfies (M3).

As with $M$-minimal diagrams, a subdiagram of a diminished (resp minimal, $D$-minimal) diagram is necessarily diminished (resp minimal, $D$-minimal).

In what follows, it is taken implicitly that any diminished, minimal, or $D$-minimal diagram over $G_\Omega(\textbf{M})$ is formed over its disk presentation (rather than its canonical presentation).

\begin{lemma} \label{diminished exist}

A word $W$ over $\pazocal{X}$ represents the trivial element of $M_\Omega(\textbf{M})$ if and only if there exists a diminished diagram $\Delta$ over $M_\Omega(\textbf{M})$ such that $\lab(\partial\Delta)\equiv W$ and $\Delta$ contains no $\theta$-annuli.

\end{lemma}

\begin{proof}

The reverse direction follows immediately from van Kampen's Lemma (see Section 2.1). 

Let $\pazocal{S}_1$ be the set of relators defining the $(\theta,a)$-relations of $M_\Omega(\textbf{M})$, i.e the words $[\theta_i,a]$ for $\theta\in\Theta^+$ and $a\in Y_i(\theta)$. Similarly, let $\pazocal{S}_2$ be the set of relators defining the $a$-relations, $\pazocal{S}_3$ be the set of relators defining the $(\theta,q)$-relations for the $q$-letters that are not $t$-letters, and $\pazocal{S}_4$ be the set of relators defining the $(\theta,t)$-relations. Note that any cyclic permutation of an element of $\pazocal{S}_i^{\pm1}$ is not an element of $\pazocal{S}_j$ for $j\neq i$. So, the partition of the relations given by $\pazocal{S}_1\sqcup\dots\sqcup\pazocal{S}_4$ defines a grading on the presentation of $M_\Omega(\textbf{M})$ (see Section 2.7).

A reduced graded diagram with respect to this grading has minimal signature. Hence, for $W$ a word over $\pazocal{X}$ representing the trivial element of $G_\Omega(\textbf{M})$, the strengthened version of van Kampen's Lemma (Section 2.7) yields a diminished diagram $\Delta$ over $M_\Omega(\textbf{M})$ with $\lab(\partial\Delta)\equiv W$.

Now suppose $\Delta$ contains a $\theta$-annulus. As $\theta$-bands cannot cross, the $\theta$-annuli of $\Delta$ are partially ordered as in the proof of Lemma \ref{M_a no annuli 2}(2). Since $\Delta$ is finite, there exists a minimal $\theta$-annulus $S$ with respect to this partial order. 

Let $\Delta_S$ be the subdiagram bounded by the outer contour of $S$ and suppose $\Delta_S\setminus S$ contains a $(\theta,a)$-cell $\pi$. Then, let $S'$ be the maximal $\theta$-band containing $\pi$. Since $\theta$-bands cannot cross, $S'$ must be a $\theta$-annulus contained in $\Delta_S\setminus S$. But this contradicts the minimality of $S$.

So, since Lemma \ref{M_a no annuli 2}(1) implies that $\Delta_S$ contains no $(\theta,q)$-cells, $\Delta_S\setminus S$ consists entirely of $a$-cells. Hence, $\lab(\partial\Delta_S)\equiv\lab(\partial(\Delta_S\setminus S))$ is trivial in the group $B(\pazocal{A},n)$.

As a result, we may form the reduced diagram $\Gamma$ by excising $\Delta_S$ from $\Delta$, pasting a single $a$-cell in its place, and making any necessary cancellations. Note that $\lab(\partial\Gamma)\equiv\lab(\partial\Delta)$ and $s(\Gamma)\leq s(\Delta)$, so that $\Gamma$ must again be diminished. However, the number of $\theta$-annuli in $\Gamma$ is one less than the number in $\Delta$.

Iterating this process, we remove all $\theta$-annuli in $\Delta$, producing a diminished diagram satisfying the statement.

\end{proof}

\begin{lemma} \label{a-bands between a-cells}

Every diminished diagram satisfies (M2).

\end{lemma}

\begin{proof}

Arguing toward a contradiction, let $\pi_1$ and $\pi_2$ be two $a$-cells in a diminished diagram $\Delta$ connected by an $a$-band. Let $\Delta_0$ be the subdiagram consisting of $\pi_1$, $\pi_2$, and this $a$-band (see Figure 9.1).

As an $a$-band consists only of $(\theta,a)$-cells, the top and bottom of the $a$-band have equivalent labels visually equal to a word $H\in F(R)$.

So, $\lab(\partial\Delta_0)\equiv uHvH^{-1}$ for some words $u,v\in F(\pazocal{A})$. Note that for any rule $\theta$ corresponding to a letter of $H$, the makeup of the $a$-band implies the existence of a $(\theta,a)$-relation corresponding to $\theta$ and an $a$-letter from the `special' input sector. This then implies that the domain of $\theta$ in the `special' input sector is nonempty, which in turn implies that the domain of $\theta$ in this sector is the entire alphabet.

As a result, we can build a reduced annular diagram $\Gamma'$ over the canonical presentation of $M(\textbf{M})$ with outer label $uHvH^{-1}$, inner label $uv$, and made up entirely of $(\theta,a)$-cells. Then, since $\Delta_0$ is a diagram over $M_\Omega(\textbf{M})$, we have $uv=1$ in $M_\Omega(\textbf{M})$.

\begin{figure}[H]
\centering
\includegraphics[scale=1]{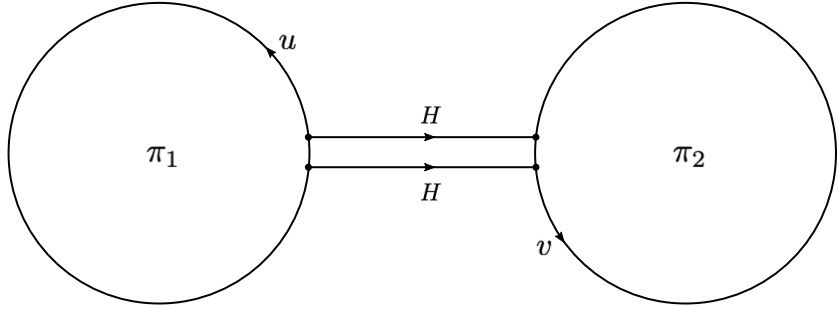}
\caption{The subdiagram $\Delta_0$}
\end{figure}

Let $\Psi$ be the diminished diagram over $M_\Omega(\textbf{M})$ with $\lab(\partial\Psi)\equiv uv$ given by Lemma \ref{diminished exist}. Since $\partial\Psi$ has no $\theta$-edges and $\Psi$ has no $\theta$-annuli, $\Psi$ must consist only of $a$-cells.

By van Kampen's Lemma, we then have $uv=1$ over $B(\pazocal{A},n)$, so that $\Psi$ must consist of exactly one $a$-cell by the minimality of its signature. Pasting $\Psi$ into the middle of $\Gamma'$ then yields a reduced diagram $\Gamma_0$ over $M_\Omega(\textbf{M})$ with contour label $uHvH^{-1}$.

Let $\Gamma$ be the reduced diagram obtained from $\Delta$ by excising $\Delta_0$, pasting $\Gamma_0$ in its place, and making any necessary cancellations. Then $\Gamma$ has the same contour label as $\Delta$, one less $a$-cell, and the same number of disks, $(\theta,t)$-cells, and $(\theta,q)$-cells. Hence, $s(\Gamma)<s(\Delta)$, contradicting the assumption that $\Delta$ is diminished.

\end{proof}

Note that Lemma \ref{a-bands between a-cells} implies that a diminished diagram satisfying (M1) is minimal.

\smallskip

%%%%%%%%%%%%%%%%%%%%%%%%%%%%%%%%%%%%%%%%%%%%%%%%%%%%%%%%%%%%%%%%%

\subsection{$t$-spokes} \

When considering diminished, minimal, or $D$-minimal diagrams in what follows, many arguments rely on the $q$-bands corresponding to $t$-letters. To distinguish these from bands corresponding to other parts of the base, we adopt the convention of [16] and [23] and refer to them as \textit{$t$-bands}. Note that the top and the bottom of such a band are each labelled by a copy of the band's history.

In a diminished, minimal, or $D$-minimal diagram, a maximal $q$-band with one end on a disk $\Pi$ is called a \textit{spoke} of $\Pi$. A \textit{$t$-spoke} is then defined in the natural way.

The pairs $\{t(2),t(3)\},\dots,\{t(L-1),t(L)\},\{t(L),t(2)\}$ are called \textit{adjacent} $t$-letters. Two $t$-spokes of the same disk are called \textit{consecutive} if they correspond to adjacent $t$-letters.

\begin{lemma} \label{extend}

For $i\in\{2,\dots,L\}$, let $\pazocal{C}:A(i)\to\dots\to A(i)$ be a reduced computation of $\textbf{M}$ with history $H$. Then there exists a reduced diagram $\Delta$ over $M_\Omega(\textbf{M})$ with contour label $H(0)^{-1}W_{ac}H(0)W_{ac}^{-1}$, where $H(0)$ is the copy of $H$ in $F(R)$ obtained by adding the subscript 0 to each letter.

\end{lemma}

\begin{proof}

Consider the factorization $H\equiv H_1\cdots H_\ell$ for $\ell\geq2$ given by Lemma \ref{extending computations}.

Define $H_i(0)$ as the word in $F(R)$ obtained from $H_i$ by adding a subscript $0$ to each letter. By Lemma \ref{computations are trapezia}, for each $1\leq j\leq \ell$, there exists a trapezium $\Delta_j$ with contour label $$H_j(0)^{-1}W_{j-1}^{(z_j)}H_j(0)(W_j^{(z_j)})^{-1}$$ where $W_j^{(z_j)}$ is defined as in Lemma \ref{extending computations}.

Recall that for $1\leq j\leq\ell-1$, $W_j^{(z_j)}$ differs from $W_j^{(z_{j+1})}$ only by the insertion/deletion of words in $\pazocal{L}$ in the `special' input sector, while $W_0^{(z_1)}\equiv W_\ell^{(z_\ell)}\equiv W_{ac}$. Note that every word of $\pazocal{L}$ represents the trivial element of $B(\pazocal{A},n)$, so that $\pazocal{L}\subset\Omega$. For $1\leq j\leq\ell-1$, let $\tilde{\Delta}_j$ be the diagram obtained from pasting the $a$-cell corresponding to this element of $\pazocal{L}$ to the top of $\Delta_j$, so that the `top' label of $\tilde{\Delta}_j$ is $W_j^{(z_{j+1})}$.

Then, letting $\tilde{\Delta}_\ell=\Delta_\ell$, we may glue the top of $\tilde{\Delta}_j$ to the bottom of $\tilde{\Delta}_{j+1}$. Letting $\Delta$ be the reduced diagram that results from these pastings, $\text{Lab}(\partial\Delta)\equiv H(0)^{-1}W_{ac}H(0)W_{ac}^{-1}$.

\end{proof}

\begin{lemma} \label{extend 2}

Let $\pazocal{C}:V_0\to\dots\to V_t$ be a reduced computation of $\textbf{M}$ with history $H$ and base $\{t(i)\}B_3(i)$ for some $i\in\{2,\dots,L\}$. Suppose there exists an accepted configuration $W_0$ such that $W_0(i)\equiv V_0$. Then there exists an accepted configuration $W_t$ with $W_t(i)\equiv V_t$ and a reduced diagram $\Delta$ over $M_\Omega(\textbf{M})$ with contour label $H(0)^{-1}W_0H(0)W_t^{-1}$, where $H(0)$ is the copy of $H$ in $F(R)$ obtained by adding the subscript 0 to each letter.

\end{lemma}

\begin{proof}

Let $H\equiv H_1\dots H_\ell$ be the factorization such that each $H_j$ is the history of a maximal one-machine subcomputation. 

For $1\leq j\leq\ell$, let $\pazocal{C}_j:V_{y(j)}\to\dots\to V_{z(j)}$ be the subcomputation with history $H_j$. Then, let $\pazocal{C}_j':W_{y(j)}'\to\dots\to W_{z(j)}'$ be the reduced computation in the standard base given by Lemma \ref{extending one-machine}. As in the proof of Lemma \ref{M projected long history}, note that $W_{z(j)}'$ may differ from $W_{y(j+1)}'$.

\smallskip

\ \textbf{1.} Suppose $W_0$ is $H_1$-admissible and $\ell=1$. 

Then, there exists a reduced computation $\pazocal{C}'$ with initial configuration $W_0$ and history $H\equiv H_1$. So, we may let $W_t\equiv W_0\cdot H$ and $\Delta$ be the trapezium corresponding to $\pazocal{C}'$ from Lemma \ref{computations are trapezia}.

\smallskip

\ \textbf{2.} Suppose $W_0$ is $H_1$-admissible and $\ell\geq2$.

Without loss of generality, we may assume that $W_{y(1)}'\equiv W_0$, i.e $\pazocal{C}_1'$ is a reduced computation with initial configuration $W_0$ and history $H_1$.

Then $W_{z(1)}'$ is an accepted start or end configuration. By the construction of the computation given in the proof of Lemma \ref{extending one-machine}, $W_{y(2)}'$ is also an accepted configuration. Continuing in this way, $W_{z(\ell)}'$ is an accepted configuration with $W_{z(\ell)}'(i)\equiv V_t$, so that we may let $W_t\equiv W_{z(\ell)}'$.

For $1\leq j\leq\ell$, let $\Delta_j$ be the trapezium corresponding to $\pazocal{C}_j'$ given by Lemma \ref{computations are trapezia}. Then, for $1\leq j\leq\ell-1$, $\lab(\textbf{top}(\Delta_j))\equiv W_{z(j)}'$ and $\lab(\textbf{bot}(\Delta_{j+1}))\equiv W_{y(j+1)}'$ differ only by the insertion of a word from $\pazocal{L}$ in the `special' input sector. So, we may paste $\Delta_j$ to $\Delta_{j+1}$ along an $a$-cell corresponding to this difference.

The reduced diagram $\Delta$ arising from these pastings then satisfies the statement.

\smallskip

\ \textbf{3.} Suppose $W_0$ is not $H_1$-admissible.

Let $H\equiv H_1'H_1''$ such that $H_1'$ is the maximal (perhaps empty) prefix for which $W_0$ is $H_1'$-admissible. Then for $\theta$ the first letter of $H_1''$, $(W_0\cdot H_1')(i)$ is $\theta$-admissible while $W_0\cdot H_1'$ is not. Lemma \ref{projection admissible configuration not} then implies the following:

\begin{enumerate}[label=({\arabic*})]

\item If $\pazocal{C}_1$ is a one-machine computation of the first machine, then there exists $w\in\pazocal{L}$ such that $W_0\cdot H_1'\equiv J(w)\cdot H'$, where $H'$ is the natural copy of $w$ read right to left in the language of positive rules with step history $(1)_1$

\item If $\pazocal{C}_1$ is a one-machine computation of the second machine, then there exists $w\in\pazocal{L}$ such that $W_0\cdot H_1'\equiv I(w)$.

\end{enumerate}

Let $\|H_1'\|=r$. Then by the construction given in Lemma \ref{extending one-machine}, $W_r'\equiv I(w)\cdot H'$ in case (1) or $W_r'\equiv J(w)$ in case (2). Either way, Lemma \ref{M language} implies $W_r'$ is an accepted configuration. As a result, every configuration of in $\pazocal{C}_1'$ is accepted.

As in Step 2, as $W_{z(1)}'$ is accepted, $W_{y(2)}'$ (and so each configuration of $\pazocal{C}_2'$) is accepted. Continuing, this implies $W_{z(\ell)}'$ is accepted with $W_{z(\ell)}'(i)\equiv V_t$, so that we may let $W_t\equiv W_{z(\ell)}'$.

Let $\Delta_1'$ be the trapezium corresponding to the reduced computation with initial configuration $W_0$ and history $H_1'$. Further, let $\Delta_1''$ be the trapezium corresponding to the reduced computation with initial configuration $W_r'$ and history $H_1''$.

Then, $\lab(\textbf{top}(\Delta_1'))\equiv W_0\cdot H_1'$ and $\lab(\textbf{bot}(\Delta_1''))\equiv W_r'$ differ by the insertion/deleetion of an element of $\pazocal{L}$ in the `special' input sector. Hence, we may construct a reduced diagram $\Delta_1$ by pasting $\Delta_1'$ to $\Delta_1''$ along an $a$-cell corresponding to this difference.

As in previous steps, we may then construct the reduced diagram $\Delta$ satisfying the statement.

\end{proof}

\begin{lemma} \label{t-spokes between disks}

Let $\Delta$ be a $D$-minimal diagram over the disk presentation of $G_\Omega(\textbf{M})$. Suppose there exist two disks $\Pi_1$ and $\Pi_2$ in $\Delta$ so that $\pazocal{Q}_1$ and $\pazocal{Q}_2$ are consecutive $t$-spokes of both. Let $\Psi$ be the subdiagram bounded by the sides of $\pazocal{Q}_i$ and the subpaths of $\partial\Pi_i$ such that $\Psi$ does not contain $\Pi_1$ or $\Pi_2$. Then $\Psi$ contains a disk.

\end{lemma}

\begin{proof}

Assume that $\Pi_1$ and $\Pi_2$ are hubs. Note that if either of these two bands has zero length, then the two hubs are cancellable, contradicting the assumption that $\Delta$ is reduced.

Arguing toward contradiction, suppose $\Psi$ contains no disk. 
%Factor $\partial\Psi=\textbf{p}_1^{-1}\textbf{q}_1\textbf{p}_2\textbf{q}_2^{-1}$ such that $\textbf{q}_i$ is a side of $\pazocal{Q}_i$ and $\textbf{p}_i^{-1}$ is a subpath of $\partial\Pi_i$. 

First, suppose the pair of adjacent $t$-letters corresponding to $\pazocal{Q}_1$ and $\pazocal{Q}_2$ is $\{t(i),t(i+1)\}$ for some $2\leq i\leq L-1$ (see Figure 9.2(a)).

By Lemma \ref{diminished exist}, there exists a diminished diagram $\Lambda$ over $M_\Omega(\textbf{M})$ with $\lab(\partial\Lambda)\equiv\lab(\partial\Psi)$. 
%Factor $\partial\Lambda=(\textbf{p}_1')^{-1}\textbf{q}_1'\textbf{p}_2'(\textbf{q}_2')^{-1}$ such that $\lab(\textbf{p}_i')\equiv\lab(\textbf{p}_i)$ and $\lab(\textbf{q}_i')\equiv\lab(\textbf{q}_i)$.

Suppose there exists an $a$-cell $\pi$ in $\Lambda$. Note that no edge of $\partial\Lambda$ is labelled by an $a$-letter from the `special' input sector. So, by Lemmas \ref{a-band on same a-cell} and \ref{a-bands between a-cells}, any $a$-band starting on $\partial\pi$ must end on a $(\theta,q)$-cell in $\Lambda$. Further, by Lemma \ref{M_a no annuli 1}, the maximal $q$-band $\pazocal{Q}$ containing this $(\theta,q)$-cell must have two ends on $\partial\Lambda$. But the definition of the rules implies that $\pazocal{Q}$ corresponds to $Q_0(1)^{\pm1}$ while no $q$-edge of $\partial\Lambda$ corresponds to such a base letter.

Hence, $\Lambda$ is a reduced diagram over $M(\textbf{M})$, so that Lemma \ref{M(S) annuli} implies that $\Lambda$ is a trapezium with top and bottom labels $A(i)\{t(i+1)\}$ (up to inversion). By Lemma \ref{trapezia are computations}, there exists a corresponding computation $\pazocal{C}:A(i)\to\dots\to A(i)$ with history $H$. Thus, Lemma \ref{extend} yields a reduced diagram $\Gamma_1$ over $M_\Omega(\textbf{M})$ with contour label $H(0)^{-1}W_{ac}H(0)W_{ac}^{-1}$.

The subdiagram $\Gamma_0$ of $\Gamma_1$ bounded by the two $t$-bands corresponding to $t(i)$ and $t(i+1)$ has the same contour label as $\Lambda$, and so as $\Psi$. By cyclic permutation, we may assume that $\Gamma_0$ can be cut from $\Gamma_1$ to produce a reduced diagram $\Gamma$ over $M_\Omega(\textbf{M})$.

Let $\Psi'$ be the smallest subdiagram of $\Delta$ containing $\Pi_1$, $\Pi_2$, $\pazocal{Q}_1$, and $\pazocal{Q}_2$. Then $\lab(\partial\Psi')\equiv\lab(\partial\Gamma)$. Hence, excising $\Psi'$ from $\Delta$ and pasting $\Gamma$ in its place reduces the number of hubs (and so disks) by two, contradicting the assumption that $\Delta$ is $D$-minimal.

\renewcommand\thesubfigure{\alph{subfigure}}
\begin{figure}[H]
\centering
\begin{subfigure}[b]{\textwidth}
\centering
\includegraphics[scale=1]{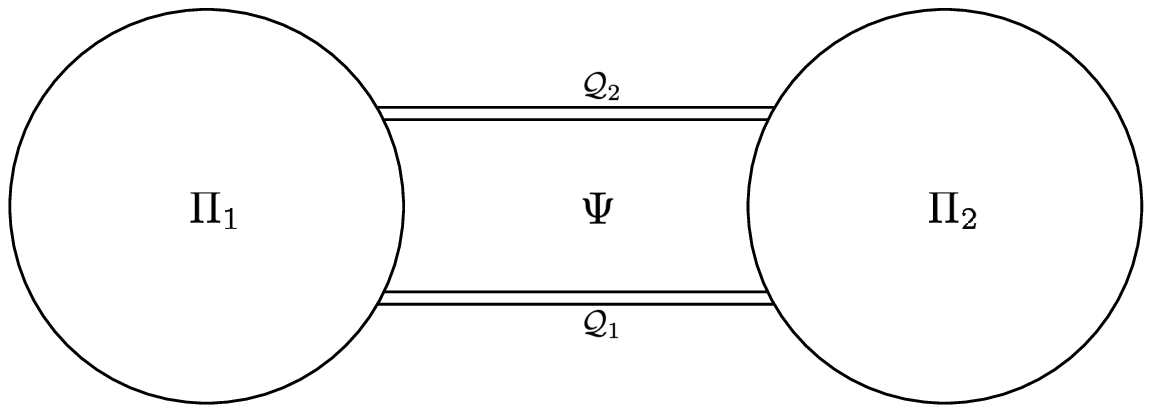}
\caption{Adjacent $t$-letters are $\{t(i),t(i+1)\}$ for $2\leq i\leq L-1$}
\end{subfigure} \\ \vspace{0.2in}
\begin{subfigure}[b]{\textwidth}
\centering
\includegraphics[scale=1]{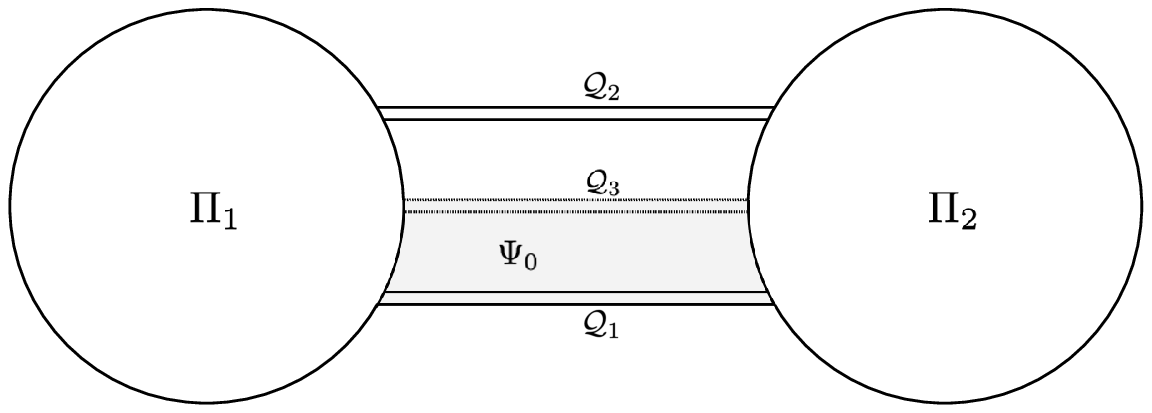}
\caption{Adjacent $t$-letters are $\{t(L),t(2)\}$           }
\end{subfigure}
\caption{Lemma \ref{t-spokes between disks}}
\end{figure}

Now suppose the adjacent $t$-letters corresponding to $\pazocal{Q}_1$ and $\pazocal{Q}_2$ are $t(L)$ and $t(2)$, respectively. Then, the $q$-band $\pazocal{Q}_3$ starting on $\Pi_1$ corresponding to the part $\{t(1)\}$ must end on $\Pi_2$. Let $\Psi_0$ be the subdiagram of $\Psi$ bounded by $\pazocal{Q}_1$ and $\pazocal{Q}_3$ (see Figure 9.2(b)).

Let $\Lambda_0$ be a reduced diagram over $M_\Omega(\textbf{M})$ given by Lemma \ref{diminished exist}, so that $\Lambda_0$ satisfies (M2) and $\lab(\partial\Lambda_0)\equiv\lab(\partial\Psi_0)$. Then, as above, $\Lambda_0$ must be a trapezium with top and bottom labels $A(L)\{t(1)\}$. Hence, Lemma \ref{trapezia are computations} gives a corresponding computation $\pazocal{C}:A(L)\to\dots\to A(L)$, so that we may repeat the argument above to reduce the number of hubs.

Finally, suppose $\Pi_1$ and $\Pi_2$ are not necessarily hubs. Then we can replace these disks with reduced diagrams over $G(\textbf{M})$ consisting of a hub and a trapezium (as formed in Lemma \ref{disks are relations}). Let $\tilde{\Delta}$ be the resulting reduced diagram, $\tilde{\Pi}_1$ and $\tilde{\Pi}_2$ be the two hubs, $\tilde{\pazocal{Q}}_1$ and $\tilde{\pazocal{Q}}_2$ be the consecutive $t$-spokes at these hubs, and $\tilde{\Psi}$ be the subdiagram bounded by the sides of $\tilde{\pazocal{Q}}_i$ and the contours of $\tilde{\Pi}_i$.

As $\tilde{\Delta}$ and $\Delta$ have the same number of disks and the same contour labels, $\tilde{\Delta}$ is $D$-minimal. Thus, the same arguments as outlined above can be applied to remove $\tilde{\Pi}_1$ and $\tilde{\Pi}_2$, yielding a contradiction.

\end{proof}

For each reduced diagram $\Delta$ over the disk presentation of $G_\Omega(\textbf{M})$, there is a corresponding planar graph $\Gamma\equiv\Gamma(\Delta)$ defined by:

\begin{enumerate}[label=({\arabic*})]

\item $V(\Gamma)=\{v_0,v_1,\dots,v_\ell\}$ where each $v_i$ for $i\geq1$ corresponds to one of the $\ell$ disks of $\Delta$ and $v_0$ is one exterior vertex

\item For $i,j\geq1$, each shared $t$-spoke of the disks corresponding to $v_i$ and $v_j$ corresponds to an edge $(v_i,v_j)\in E(\Gamma)$

\item For $i\geq1$, each $t$-spoke of the disk corresponding to $v_i$ which ends on $\partial\Delta$ corresponds to an edge $(v_0,v_i)\in E(\Gamma)$

\end{enumerate}

Note that the degree of each interior vertex of $\Gamma$ is $L-1$. The following statement is a consequence of this fact and Lemma \ref{t-spokes between disks}.

\begin{lemma} \label{graph}

\textit{(Lemma 3.2 of [15])} Suppose $\Delta$ is a $D$-minimal diagram containing at least one disk. Then $\Delta$ contains a disk $\Pi$ such that $L-4$ consecutive $t$-spokes $\pazocal{Q}_1,\dots,\pazocal{Q}_{L-4}$ of $\Pi$ end on $\partial\Delta$ and such that every subdiagram $\Gamma_i$ bounded by $\pazocal{Q}_i$, $\pazocal{Q}_{i+1}$, $\partial\Pi$, and $\partial\Delta$ ($i=1,\dots,L-5$) contains no disks.

\end{lemma}

\begin{figure}[H]
\centering
\includegraphics[scale=1]{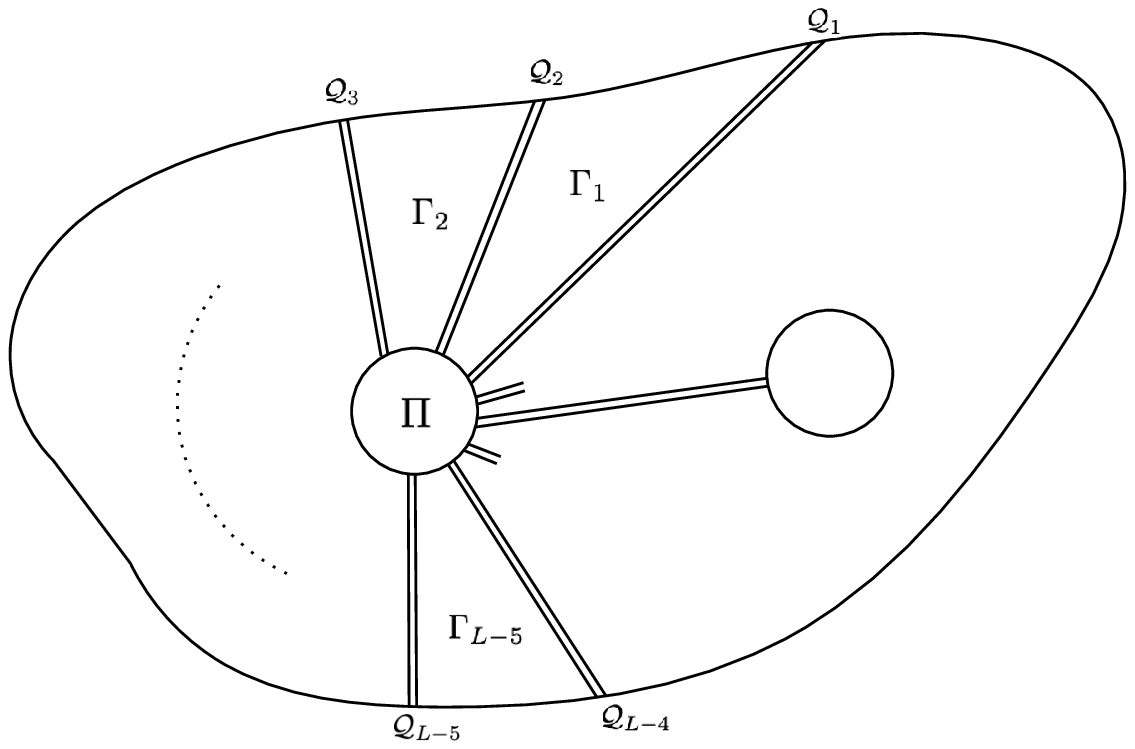}
\caption{Lemma \ref{graph}}
\end{figure}

%Applying induction on the number of hubs then implies:
%
%\begin{lemma} \label{number of q-edges}
%
%\textit{(Lemma 5.19 of [24])} Let $\Delta$ be a reduced diagram over the disk presentation of $G_a(\textbf{M})$ satisfying (M3). If $\Delta$ contains $\ell\geq1$ disks, then the number of spokes ending on $\partial\Delta$, and therefore the number of $q$-edges of $\partial\Delta$, is greater than $\ell LN/2$.
%
%\end{lemma}

\smallskip

%%%%%%%%%%%%%%%%%%%%%%%%%%%%%%%%%%%%%%%%%%%%%%%%%%%%%%%%%%%%%%%%%

\subsection{Transposition of a $\theta$-band and a disk} \

We now describe a procedure, similar to the construction in Section 8.3, for moving a $\theta$-band about a disk.

Let $\Delta$ be a $D$-minimal diagram containing a disk $\Pi$ and a $\theta$-band $\pazocal{T}$ subsequently crossing the $t$-spokes $\pazocal{Q}_1,\dots,\pazocal{Q}_\ell$ of $\Pi$. Assume $\ell\geq2$ is maximal for $\Pi$ and $\pazocal{T}$.

First, suppose there are no other cells between $\Pi$ and the bottom of $\pazocal{T}$, i.e there is a subdiagram formed by $\Pi$ and $\pazocal{T}$.

Let $\pazocal{T}'$ be the subband of $\pazocal{T}$ whose bottom path, $\textbf{s}_1^{-1}$, starts with the $t$-edge corresponding to the start of $\pazocal{Q}_1$ and ends with that of $\pazocal{Q}_\ell$. Further, let $\textbf{s}_2$ be the complement of $\textbf{s}_1$ in $\partial\Pi$ so that $\partial\Pi=\textbf{s}_1\textbf{s}_2$. Then as any sector of the standard base containing a $t$-letter has empty tape alphabet, $\lab(\textbf{s}_2)$ is an admissible word.

Let $W\equiv\text{Lab}(\partial\Pi)^{\pm1}$, $V\equiv\text{Lab}(\textbf{s}_1)$, and $\theta$ be the rule corresponding to $\pazocal{T}$. Further, let $\Gamma$ be the subdiagram formed by $\Pi$ and $\pazocal{T}'$. Then, by Lemma \ref{theta-bands are one-rule computations}, $V^{-1}$ is $\theta$-admissible with $V^{-1}\cdot\theta\equiv\lab(\textbf{ttop}(\pazocal{T}'))=\lab(\textbf{top}(\pazocal{T}'))$ 

Suppose $\lab(\textbf{s}_2)$ is $\theta$-admissible. Then $W$ is $\theta$-admissible, so that $W\cdot\theta$ is a disk relator. Let $\bar{\Pi}$ be a disk with contour labelled by $W\cdot\theta$. Let $\pazocal{T}''$ be the auxiliary $\theta$-band corresponding to $\theta$ whose top is labelled by $\lab(\textbf{s}_2)\cdot\theta$. Then, let $\bar{\Gamma}$ be the diagram obtained from attaching $\pazocal{T}''$ to $\bar{\Pi}$. Finally, let $\bar{\Delta}$ be the reduced diagram obtained from excising $\Gamma$ from $\Delta$ and pasting $\bar{\Gamma}$ in its place, attaching the first and last cells of $\pazocal{T}''$ to the complement of $\pazocal{T}'$ in $\pazocal{T}$ and perhaps making cancellations in the resulting $\theta$-band. Note that $\bar{\Delta}$ has the same contour label as that of $\Delta$.

Conversely, suppose $\text{Lab}(\textbf{s}_2)$ is not $\theta$-admissible. Then Lemma \ref{projection admissible configuration not} applies to $W$, so that $\lab(\textbf{s}_2)$ contains the `special' input sector and would be $\theta$-admissible with the insertion/deletion of some $u^n\in\pazocal{L}$. So, after attaching to $\Pi$ an $a$-cell corresponding to $u^n$, we may construct the disk $\bar{\Pi}$ and the auxiliary $\theta$-band $\pazocal{T}''$ as above. Attaching the mirror $a$-cell on the other side of $\pazocal{T}''$ then produces a diagram $\bar{\Delta}$ with the same contour label as $\Delta$.

The procedure of excising $\Gamma$ from $\Delta$ to create $\bar{\Delta}$ is called the \textit{transposition} of the disk $\Pi$ and the $\theta$-band $\pazocal{T}$ in $\Delta$.

\renewcommand\thesubfigure{\alph{subfigure}}
\begin{figure}[H]
\centering
\begin{subfigure}[b]{0.48\textwidth}
\centering
\raisebox{0.2in}{\includegraphics[width=3in]{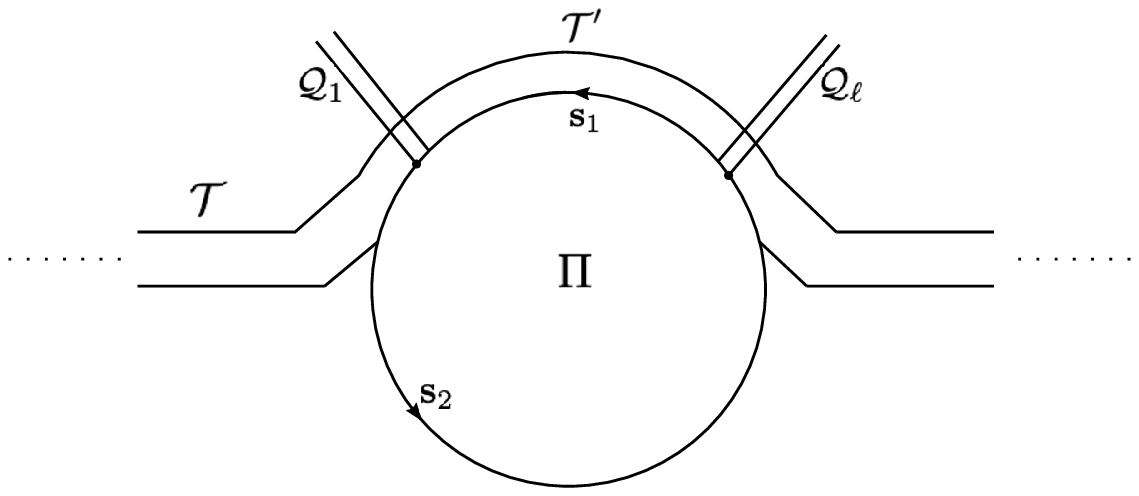}}
\caption{The subdiagram $\Gamma$}
\end{subfigure}\hfill
\begin{subfigure}[b]{0.48\textwidth}
\centering
\includegraphics[width=3in]{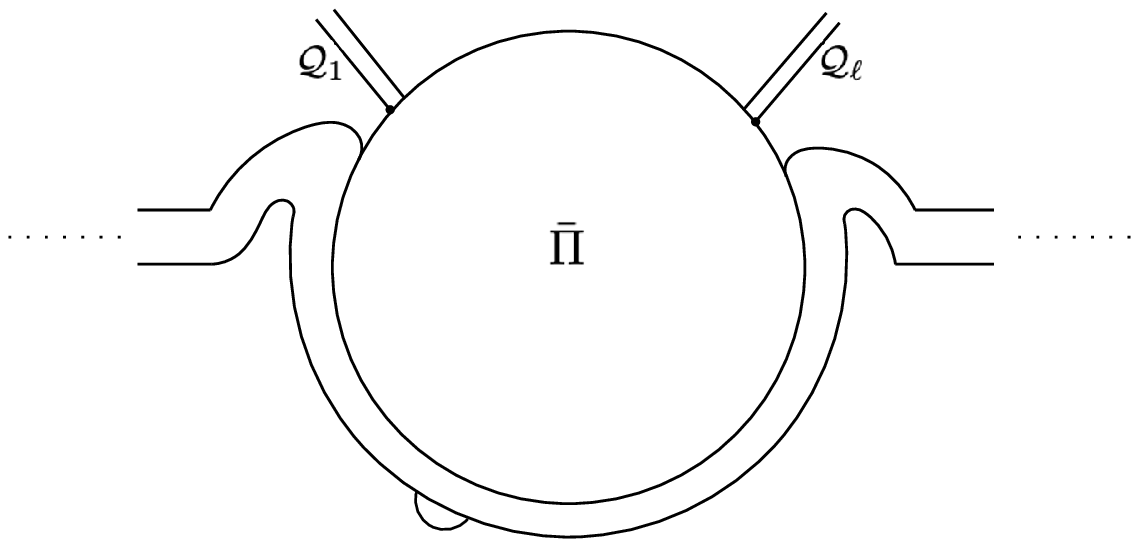}
\caption{The resulting subdiagram $\bar{\Gamma}$}
\end{subfigure}
\caption{The transposition of a $\theta$-band with a disk}
\end{figure}

Now, consider the situation where there are cells between the $\theta$-band and the disk, each of which is an $a$-cell.

Suppose the pair of adjacent $t$-letters corresponding to $\pazocal{Q}_1$ and $\pazocal{Q}_2$ is $\{t(i),t(i+1)\}$ for some $2\leq i\leq L-1$. Let $\pazocal{T}_1'$ be the subband of $\pazocal{T}'$ between $\pazocal{Q}_1$ and $\pazocal{Q}_2$. Then, let $\Psi$ be the subdiagram of $\Delta$ bounded by $\pazocal{T}_1'$ and $\partial\Pi$. By Lemma \ref{diminished exist}, there exists a diminished diagram $\Lambda$ over $M_\Omega(\textbf{M})$ with $\lab(\partial\Psi)\equiv\lab(\partial\Lambda)$. Lemmas \ref{a-band on same a-cell} and \ref{a-bands between a-cells} then imply that $\Lambda$ contains no $a$-cell, so that Lemma \ref{M(S) annuli} implies that $\Lambda$ consists of a single $\theta$-band. Hence, by Lemma \ref{theta-bands are one-rule computations}, $W(i)$ is $\theta$-admissible.

Otherwise, if the pair of adjacent $t$-letters is $\{t(L),t(2)\}$, then the same argument applies to the subdiagram bounded by the $t$-band corresponding to $t(L)$, the $q$-band corresponding to $t(1)$, and $\pazocal{T}'$. As a result, $W(L)$ is $\theta$-admissible.

As above, Lemma \ref{projection admissible configuration not} then implies that, perhaps after attaching an $a$-cell, we may construct a new disk and auxiliary band that, perhaps after attaching another $a$-cell, functions as the transposition of $\Pi$ with $\pazocal{T}$.

The reduced diagram $\Delta'$ arising from the transposition has the number of disks and contour label as $\Delta$, and so is $D$-minimal.

However, the minimality of the 2-signature (and so the signature) need not be preserved by a transposition. This is because many $(\theta,t)$-cells may be added through transposition.

Note that the definition of transposition above differs from that in [16] and [23] only by the presence of $a$-cells.

\newpage

\begin{lemma} \label{G_a theta-annuli}

\textit{(Compare with Lemma 7.5 of [16] and 7.7 of [23])} 

Let $\Delta$ be a reduced diagram over the disk presentation of $G_\Omega(\textbf{M})$ satisfying (M3).

\begin{enumerate}[label=({\arabic*})]

\item Suppose a $\theta$-band $\pazocal{T}$ crosses $\ell$ $t$-spokes of a disk $\Pi$ and there are no disks in the subdiagram bounded by these spokes, $\pazocal{T}$, and $\partial\Pi$. Then $\ell\leq(L-1)/2$.

\item Suppose $\pazocal{T}$ and $\pazocal{T}'$ are disjoint $\theta$-bands crossing $\ell$ and $\ell'$ $t$-spokes, respectively, of a disk $\Pi$. Suppose further that every cell between the bottom of $\pazocal{T}$ (of $\pazocal{T}'$) and $\Pi$ is an $a$-cell. Further, suppose these bands correspond to the same rule $\theta$ if the history is read toward the disk. Then $\ell+\ell'\leq(L-1)/2$.

\item If $S$ is a $\theta$-annulus in $\Delta$ and $\Delta_S$ is the subdiagram bounded by the outer contour of $S$, then $\Delta_S$ is a diagram over $M_\Omega(\textbf{M})$.

\end{enumerate}

\end{lemma}

\begin{figure}[H]
\centering
\includegraphics[scale=1.5]{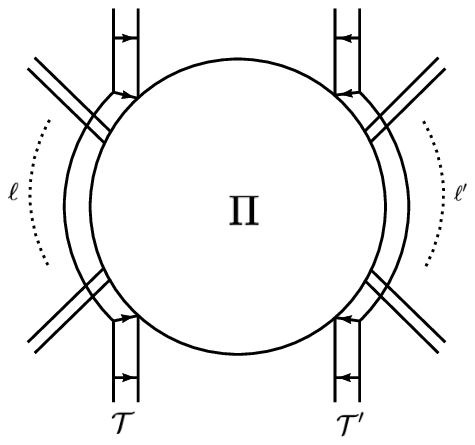}
\caption{Lemma \ref{G_a theta-annuli}(2)}
\end{figure}

\begin{proof}

(1) Lemma \ref{M_a no annuli 1} implies that there exists a $\theta$-band $\pazocal{T}_0$ crossing all $\ell$ spokes such that the only cells between it and $\Pi$ are $a$-cells. If $\ell>(L-1)/2$, then the transposition of $\Pi$ and $\pazocal{T}_0$ in $\Delta$ then yields a diagram with the same contour label, the same number of disks, and strictly less $(\theta,t)$-cells. This contradicts the minimality of $s_2(\Delta)$.

(2) The transposition of $\pazocal{T}$ and $\Pi$ removes $\ell$ $(\theta,t)$-cells and adds $(L-1)-\ell$ new $(\theta,t)$-cells in the resulting band. However, $\ell'$ of these cells form cancellable pairs with cells of $\pazocal{T}'$, so that it is possible to cancel $2\ell'$ cells. Hence, the change in the number of $(\theta,t)$-cells is $(L-1)-2\ell-2\ell'$, so that the relation $\ell+\ell'>(L-1)/2$ would contradict the minimality of $s_2(\Delta)$.

(3) Suppose $\Delta_S$ contains a disk. Then, since $\Delta$ is $D$-minimal, Lemma \ref{graph} gives a disk $\Pi$ in $\Delta_S$ with $L-4$ consecutive $t$-spokes that end on $\partial\Delta_S$ and such that the subdiagram of $\Delta_S$ bounded by these spokes contains no disks. But then taking $L>7$, $S$ and $\Pi$ contradict (1).

\end{proof}

The following is an immediate consequence of Lemmas \ref{M_a no annuli 2}(2) and \ref{G_a theta-annuli}(3).

\begin{lemma} \label{minimal theta-annuli}

A minimal diagram $\Delta$ contains no $\theta$-annuli.

\end{lemma}

The following statement gives a strengthened version of van Kampen's Lemma for $G_\Omega(\textbf{M})$, specifically for minimal diagrams.

\begin{lemma} \label{minimal exist}

A word $W$ over $\pazocal{X}$ represents the trivial element of $G_\Omega(\textbf{M})$ if and only if there exists a minimal diagram $\Delta$ such that $\lab(\partial\Delta)\equiv W$.

\end{lemma}

\begin{figure}[H]
\centering
\includegraphics[scale=1]{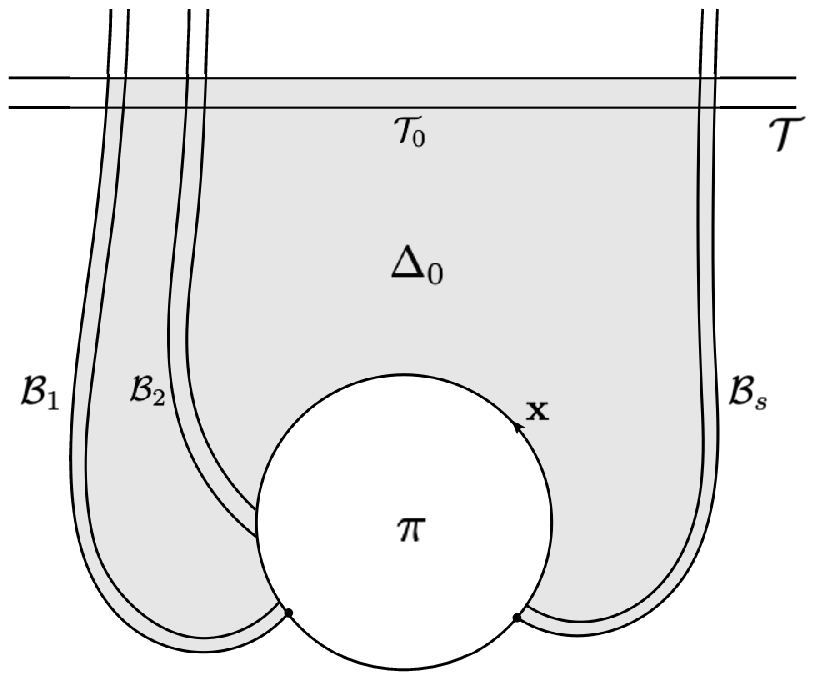}
\caption{Lemma \ref{minimal exist}}
\end{figure}

\begin{proof}

As in the proof of Lemma \ref{diminished exist}, the reverse direction is an immediate consequence of van Kampen's Lemma.

Let $\pazocal{S}_5$ be the set of words defining the disk relations of $G_\Omega(\textbf{M})$. Then, letting $\pazocal{S}_1,\dots,\pazocal{S}_4$ be as defined in the proof of Lemma \ref{diminished exist}, the partition of the relations $\pazocal{S}_1\sqcup\dots\sqcup\pazocal{S}_5$ defines a grading on the disk presentation of $G_\Omega(\textbf{M})$.

By the definition of the grading, a reduced graded diagram has minimal signature. So, the strengthened version of van Kampen's Lemma implies that for any word $W$ representing the trivial element of $G_\Omega(\textbf{M})$, there exists a diminished diagram $\Delta$ with $\lab(\partial\Delta)\equiv W$.

Suppose $\Delta$ is not minimal. As $\Delta$ is diminished, Lemma \ref{a-bands between a-cells} implies that it satisfies (M2). So, $\Delta$ must not satisfy (M1), i.e it contains an $a$-cell $\pi$ and a $\theta$-band $\pazocal{T}$ such that for some $s>\frac{1}{2}\|\partial\pi\|$, $s$ maximal $a$-bands start on $\partial\pi$ and cross $\pazocal{T}$. Without loss of generality, we assume that $\pi$ and $\pazocal{T}$ are chosen so that $s/\|\partial\pi\|$ is maximal amongst such pairs.

Enumerate these $a$-bands $\pazocal{B}_1,\dots,\pazocal{B}_s$ based on where they cross $\pazocal{T}$ and let $\Delta_0$ be the subdiagram containing each of these $s$ bands which is bounded by a side of $\pazocal{B}_1$, a side of $\pazocal{B}_s$, a subpath $\textbf{x}$ of $\partial\pi$, and the top of a subband $\pazocal{T}_0$ of $\pazocal{T}$ (see Figure 9.6).

As $a$-bands consist only of $(\theta,a)$-cells, the sides of $\pazocal{B}_1$ and $\pazocal{B}_s$ consist only of $\theta$-edges. So, any $q$-edge of $\partial\Delta_0$ must be part of the top of $\pazocal{T}_0$.

Suppose $\Delta_0$ contains a disk. By Lemma \ref{graph}, there exists a disk $\Pi$ in $\Delta_0$ with at least $L-4$ $t$-spokes ending on $\textbf{top}(\pazocal{T}_0)$ such that there are no disks in the subdiagram bounded by these spokes. But then the parameter choice $L>7$ means that $\pazocal{T}_0$ and $\Pi$ form a counterexample to Lemma \ref{G_a theta-annuli}(1).

So, as $\Delta_0$ contains no disks, any $q$-edge of $\partial\Delta_0$ must mark the start of a maximal $q$-band which has two ends on $\textbf{top}(\pazocal{T}_1)$. But then this $q$-band bounds a $(\theta,q)$-annulus with some subband of $\pazocal{T}_0$, contradicting Lemma \ref{M_a no annuli 1}(1). So, Lemma \ref{M_a no annuli 1}(4) implies $\Delta_0$ contains no $(\theta,q)$-cells.

As $\Delta_0$ satisfies (M2), each edge of $\textbf{x}$ is the start of an $a$-band which ends on the top of $\pazocal{T}_0$, so that this $a$-band crosses $\pazocal{T}$. Hence, $\textbf{x}$ consists entirely of the $s$ edges of $\partial\pi$ marking the start of $\pazocal{B}_1,\dots,\pazocal{B}_s$.

By Lemmas \ref{M_a no annuli 1} and \ref{M_a no annuli 2}, any maximal $\theta$-band $\pazocal{T}_0'$ of $\Delta_0$ connects the side of $\pazocal{B}_1$ to the side of $\pazocal{B}_s$, so that all $s$ $a$-bands must cross this $\theta$-band. Letting $\pazocal{T}'$ be the maximal $\theta$-band of $\Delta$ containing $\pazocal{T}_0'$, the maximality of $s/\|\partial\pi\|$ implies that $\pazocal{B}_1,\dots,\pazocal{B}_s$ comprise all maximal $a$-bands starting on $\partial\pi$ and crossing $\pazocal{T}'$. So, we may pass to $\pazocal{T}'$, assuming that $\pazocal{T}'=\pazocal{T}$ is the $\theta$-band chosen above. As a result, $\pazocal{T}_0$ is the only maximal $\theta$-band of $\Delta_0$.

So, any cell between $\textbf{bot}(\pazocal{T}_0)$ and $\textbf{x}$ must be an $a$-cell. Supposing such an $a$-cell exists, property (M2) implies that each of the edges on its boundary marks the start of an $a$-band that must cross $\pazocal{T}_0$, forming another counterexample to (M1). The maximality of $s/\|\partial\pi\|$ then implies $s=\|\partial\pi\|$. Since $a$-bands cannot cross, we may find a `minimal' counterexample, i.e an $a$-cell with no cells between it and $\textbf{bot}(\pazocal{T}_0)$. Passing to this cell, we may assume without loss of generality that $\textbf{x}^{-1}=\textbf{bot}(\pazocal{T}_0)$.

As a result, $\pi$ may be transposed with $\pazocal{T}$ to produce a reduced diagram $\tilde{\Delta}$ in which at most $\|\partial\pi\|-s<\frac{1}{2}\|\partial\pi\|$ maximal $a$-bands start on $\partial\pi$ and cross the maximal $\theta$-band arising from $\pazocal{T}$. Since the rest of the diagram remains unchanged throughout this process, $\tilde{\Delta}$ is diminished and contains one less counterexample to property (M1). Hence, iterating the process eliminates any $a$-cell and $\theta$-band violating property (M1), producing the desired minimal diagram.

\end{proof}

\begin{remark}

The proof of Lemma \ref{minimal exist} corresponds to the following sharper statement: A word $W$ over $\pazocal{X}$ represents the trivial element of $G_\Omega(\textbf{M})$ if and only if there exists a diminished diagram $\Delta$ satisfying (M1) such that $\lab(\partial\Delta)\equiv W$. However, the statement above suffices for our purposes.

\end{remark}

Lemma \ref{minimal exist} immedately implies the following strengthened version of van Kampen's Lemma for $M$-minimal diagrams.

\begin{lemma} \label{M-minimal exist}

A word $W$ over $\pazocal{X}$ represents the trivial element of $M_\Omega(\textbf{M})$ if and only if there exists an $M$-minimal diagram $\Delta$ such that $\lab(\partial\Delta)\equiv W$.

\end{lemma}

\smallskip

%%%%%%%%%%%%%%%%%%%%%%%%%%%%%%%%%%%%%%%%%%%%%%%%%%%%%%%%%%%%%%%%%

\subsection{Quasi-trapezia} \

Next, the concept of trapezium is generalized to the setting of minimal diagrams over $G_\Omega(\textbf{M})$. 

A \textit{quasi-trapezium} is a minimal diagram defined in much the same way as an $a$-trapezium (see section 8.4) except that it is permitted to contain disks. In other words, a quasi-trapezium is a minimal diagram whose boundary can be factored as $\textbf{p}_1^{-1}\textbf{q}_1\textbf{p}_2\textbf{q}_2^{-1}$, where each $\textbf{p}_i$ is the side of a $q$-band and each $\textbf{q}_i$ is the maximal subpath of the side of a $\theta$-band where the subpath starts and ends with a $q$-edge.

The \textit{(step) history} of a quasi-trapezium is defined in the same way as for an $a$-trapezium, as are the \textit{base}, the \textit{height}, and the \textit{standard factorization}.

Note that a quasi-trapezium containing no disks is an $a$-trapezium, while one without any disks or $a$-cells is a trapezium. 

Indeed, an $a$-trapezium is necessarily a quasi-trapezium. To see that an $a$-trapezium satisfies (M3), note that Lemmas \ref{M_a no annuli 1} and \ref{M_a no annuli 2} imply that in any minimal diagram with the same contour label, any maximal $\theta$-band must cross each maximal $q$-band exactly once.

\begin{lemma} \label{M_a reduced a-trapezia}

Suppose $\Gamma$ is a reduced diagram over $M_\Omega(\textbf{M})$ with contour $\textbf{p}_1^{-1}\textbf{q}_1\textbf{p}_2\textbf{q}_2^{-1}$ where each $\textbf{p}_j$ is the side of a $q$-band and each $\textbf{q}_j$ is the maximal subpath of the side of a $\theta$-band that starts and ends with a $q$-letter. Then there exists a minimal diagram $\Gamma'$ over $M_\Omega(\textbf{M})$ such that:

\begin{enumerate}[label=({\arabic*})]

\item $\partial\Gamma'=(\textbf{p}_1')^{-1}\textbf{q}_1'\textbf{p}_2'(\textbf{q}_2')^{-1}$, where $\text{Lab}(\textbf{p}_j')\equiv\text{Lab}(\textbf{p}_j)$ and $\text{Lab}(\textbf{q}_j')\equiv\text{Lab}(\textbf{q}_j)$ for $j=1,2$

\item there exists a simple path $\textbf{s}_1$ (respectively $\textbf{s}_2$) connecting the vertices $(\textbf{p}_1')_-$ and $(\textbf{p}_2')_-$ (respectively the vertices $(\textbf{p}_1')_+$ and $(\textbf{p}_2')_+$) such that

\begin{enumerate}

\item $(\textbf{p}_1')^{-1}\textbf{s}_1\textbf{p}_2'\textbf{s}_2^{-1}$ is the standard factorization of the boundary of an $a$-trapezium  $\Gamma_2$ and

\item any cell above $\textbf{s}_2$ or below $\textbf{s}_1$ is an $a$-cell.

\end{enumerate}

\end{enumerate}

\end{lemma}

\begin{proof}

By Lemma \ref{minimal exist}, there exists a minimal diagram $\Gamma'$ with $\lab(\partial\Gamma')\equiv\lab(\partial\Gamma)$. Then $\partial\Gamma'$ can be factored as in (1).

Every $q$-edge of $\textbf{q}_j'$ gives rise to a maximal $q$-band of $\Gamma'$. Suppose such a band $\pazocal{Q}$ starts and ends on $\textbf{q}_j'$ and consider the subdiagram $\Delta$ bounded by a side of $\pazocal{Q}$ and $\textbf{q}_j'$. Since $q$-bands are comprised entirely of $(\theta,q)$-cells, the side of $\pazocal{Q}$ contains $\theta$-edges that give rise to maximal $\theta$-bands in $\Delta$. Lemma \ref{M_a no annuli 1} then implies that no such $\theta$-band can have both ends on the side of $\pazocal{Q}$, so that it must end on $\textbf{q}_j'$. But $\textbf{q}_j'$ contains no $\theta$-edge since $\lab(\textbf{q}_j')\equiv\lab(\textbf{q}_j)$. 

Hence, by Lemma \ref{M_a no annuli 1}, every maximal $q$-band in $\Gamma'$ connects an edge of $\textbf{q}_1'$ with an edge of $\textbf{q}_2'$. 

Now suppose a maximal $\theta$-band of $\Gamma'$ has two ends on $\textbf{p}_j'$. Then, as no two $\theta$-bands can cross, there exists a $\theta$-band connecting adjacent $\theta$-letters of $\textbf{p}_j'$ (with perhaps $a$-letters between them). Then, the corresponding $\theta$-edges of $\textbf{p}_j$ in $\Gamma$ are mutually inverse adjacent $\theta$-edges, so that the corresponding cells of the $q$-band with side $\textbf{p}_j$ are cancellable. But this contradicts the assumption that $\Gamma$ is reduced. 

Hence, by Lemma \ref{G_a theta-annuli}(3), every maximal $\theta$-band in $\Gamma'$ connects an edge of $\textbf{p}_1'$ with an edge of $\textbf{p}_2'$, and so we can enumerate them from bottom to top $\pazocal{T}_1,\dots,\pazocal{T}_h$ for $h=|\textbf{p}_j|$.

Let $\pazocal{Q}_1$ be the maximal $q$-band of $\Gamma$ such that $\textbf{p}_1=\textbf{top}(\pazocal{Q}_1)$ and let $\pazocal{Q}_1'$ be the maximal $q$-band of $\Gamma'$ starting at the first letter of $\textbf{q}_1'$. Then $\pazocal{Q}_1$ and $\pazocal{Q}_1'$ must correspond to the same base letter. Moreover, since every maximal $\theta$-band must cross $\pazocal{Q}_1'$ exactly once, $\pazocal{Q}_1$ and $\pazocal{Q}_1'$ must have the same history. So, $\lab(\textbf{top}(\pazocal{Q}_1'))\equiv\lab(\textbf{p}_1')$.

As $\Gamma'$ is minimal, any cell between $\textbf{top}(\pazocal{Q}_1')$ and $\textbf{p}_1'$ must be a $(\theta,a)$-cell. But removing any such cell from $\Gamma'$ does not affect the minimality of $\Gamma'$. Hence, we may assume that $\textbf{top}(\pazocal{Q}_1')=\textbf{p}_1'$.

By an analogous argument, letting $\pazocal{Q}_2'$ be the maximal $q$-band of $\Gamma'$ starting at the final letter of $\textbf{q}_1'$, we may assume that $\textbf{bot}(\pazocal{Q}_2')=\textbf{p}_2'$.

Now let $\textbf{s}_1=\textbf{bot}(\pazocal{T}_1)$ and $\textbf{s}_2=\textbf{top}(\pazocal{T}_h)$. By definition, (2a) is satisfied. 

Further, as there is no maximal $\theta$-band above $\pazocal{T}_h$ or below $\pazocal{T}_1$, there can be no $(\theta,q)$- or $(\theta,a)$-cells above $\textbf{s}_2$ or below $\textbf{s}_1$. Thus, (2b) is satisfied.

\end{proof}

\begin{lemma} \label{quasi-trapezia}

Let $\Gamma$ be a quasi-trapezium with standard factorization of its contour $\textbf{p}_1^{-1}\textbf{q}_1\textbf{p}_2\textbf{q}_2^{-1}$. Then there exists a reduced diagram $\Gamma'$ such that:

\begin{enumerate}[label=({\arabic*})]

\item $\partial\Gamma'=(\textbf{p}_1')^{-1}\textbf{q}_1'\textbf{p}_2'(\textbf{q}_2')^{-1}$, where $\lab(\textbf{p}_j')\equiv\lab(\textbf{p}_j)$ and $\lab(\textbf{q}_j')\equiv\lab(\textbf{q}_j)$ for $j=1,2$

\item the number of disks 
%and $(\theta,q)$-cells 
in $\Gamma'$ is the same as the number of disks in $\Gamma$

\item there exists a simple path $\textbf{s}_1$ (respectively $\textbf{s}_2$) connecting the vertices $(\textbf{p}_1')_-$ and $(\textbf{p}_2')_-$ (respectively $(\textbf{p}_1')_+$ and $(\textbf{p}_2')_+$) such that

\begin{enumerate}

\item $(\textbf{p}_1')^{-1}\textbf{s}_1\textbf{p}_2'\textbf{s}_2^{-1}$ is the standard factorization of the boundary of an $a$-trapezium $\Gamma_2$ and

\item any cell above $\textbf{s}_2$ or below $\textbf{s}_1$ is a disk or an $a$-cell

\end{enumerate}

\item there exists $m\in\N$ such that any maximal $\theta$-band of $\Gamma$ contains $m$ $(\theta,t)$-cells and any maximal $\theta$-band of $\Gamma_2$ contains $m$ $(\theta,t)$-cells.

\end{enumerate}

\end{lemma}

\begin{proof}

By Lemmas \ref{M_a no annuli 1} and \ref{G_a theta-annuli}(3), every maximal $\theta$-band of $\Gamma$ must connect an edge of $\textbf{p}_1$ with an edge of $\textbf{p}_2$. So, we can enumerate these bands from bottom to top as $\pazocal{T}_1,\dots,\pazocal{T}_h$ for $h=|\textbf{p}_1|=|\textbf{p}_2|$.

Choose $i$ such that the number of $(\theta,t)$-cells in $\pazocal{T}_i$, $m$, is minimal. Note that $\Gamma$ has at least $hm$ $(\theta,t)$-cells.

If $\Gamma$ contains a disk, then by Lemma \ref{graph} there exists a disk $\Pi_1$ such that at least $L-4$ of its $t$-spokes end on $\textbf{q}_1$ or on $\textbf{q}_2$. By Lemma \ref{G_a theta-annuli}(1), at least $L-4-(L-1)/2\geq2$ of these spokes must end on $\textbf{q}_1$ (on $\textbf{q}_2$). So, for any $j\in\{1,\dots,h\}$, the number of $t$-spokes of $\Pi_1$ crossing $\pazocal{T}_j$ is at least 2.

Fix $j_1\in\{1,\dots,h-1\}$ such that $\Pi_1$ lies between $\pazocal{T}_{j_1}$ and $\pazocal{T}_{j_1+1}$. 

If $j_1\geq i$ (i.e $\Pi_1$ lies above $\pazocal{T}_i$), then move $\Pi_1$ upwards by transposing it with $\pazocal{T}_{j_1+1}$. Then iterate this process, moving the resulting disk upward until it is transposed with $\pazocal{T}_h$. 

If $j_1<i$, then move $\Pi_1$ down in the same way until the corresponding disk is transposed with $\pazocal{T}_1$. 

Let $\Lambda_1'$ be the reduced diagram resulting from this process and $\Pi_1'$ be the disk arising from $\Pi_1$. As $\Lambda_1'$ is formed by a sequence of transpositions, $\Lambda_1'$ has the same contour label and number of disks as does $\Gamma$, and so must be $D$-minimal. Factor $\partial\Lambda_1'=\textbf{u}_1^{-1}\textbf{b}_1'\textbf{v}_1(\textbf{t}_1')^{-1}$, where $\lab(\textbf{u}_1)\equiv\lab(\textbf{p}_1)$, $\lab(\textbf{v}_1)\equiv\lab(\textbf{p}_2)$,  $\lab(\textbf{b}_1')\equiv\lab(\textbf{q}_1)$, and $\lab(\textbf{t}_1')\equiv\lab(\textbf{q}_2)$.

Enumerate the maximal $\theta$-bands of $\Lambda_1'$ as $\pazocal{S}_1',\dots,\pazocal{S}_h'$ from bottom to top. Then, letting $\Lambda_1$ be the subdiagram of $\Lambda_1'$ given by removing $\Pi_1'$, we may factor $\partial\Lambda_1=\textbf{u}_1^{-1}\textbf{b}_1\textbf{v}_1\textbf{t}_1^{-1}$ such that $\textbf{b}_1=\textbf{tbot}(\pazocal{S}_1')$ and $\textbf{t}_1=\textbf{ttop}(\pazocal{S}_h')$.

If $\Lambda_1$ contains a disk, then Lemma \ref{graph} may be applied to yield a disk $\Pi_2$ such that $L-4$ of its $t$-spokes end on $\textbf{b}_1$ or on $\textbf{t}_1$. Fix $j_2\in\{1,\dots,h-1\}$ such that $\Pi_2$ lies between $\pazocal{S}_{j_2}'$ and $\pazocal{S}_{j_2+1}'$.

Suppose at least two $t$-spokes of $\Pi_2$ end on each of $\textbf{b}_1$ and $\textbf{t}_1$. Then we repeat the argument above, moving $\Pi_2$ above $\pazocal{S}_h'$ if $j_2\geq i$ or below $\pazocal{S}_1'$ if $j_2<i$. 

Next, suppose that at most one $t$-spoke of $\Pi_2$ in $\Lambda_1$ ends on $\textbf{b}_1$. Then, there is a set $T_2'$ of at least $L-5$ $t$-spokes of $\Pi_2$ in $\Lambda_1$ such that each ends on $\textbf{t}_1$. Note that there is a natural bijection between $T_2'$ and a subset $T_2$ of the $t$-spokes of $\Pi_2$ in $\Gamma$. As each $t$-spoke of $T_2'$ ends on $\textbf{t}_1$, each $t$-spoke in $T_2$ either ends on $\textbf{q}_2$ or on $\Pi_1$. Since $\Gamma$ satisfies (M3), at most $(L-1)/2$ $t$-spokes of $\Pi_2$ cross $\pazocal{T}_{j_2+1}$, so that at least one of the $t$-spokes of $T_2$ does not cross $\pazocal{T}_{j_2+1}$. This spoke must end on $\Pi_1$, so that $\Pi_1$ must lie below $\pazocal{T}_{j_2+1}$. Hence, $j_2\geq j_1\geq i$. In this case, move $\Pi_2$ up by transpositions until it is above $\pazocal{S}_h'$.

Finally, if at most one $t$-spoke of $\Pi_2$ in $\Lambda_1$ ends on $\textbf{t}_1$, then the symmetric argument to the one above yields $j_2\leq j_1<i$. In this case, move $\Pi_2$ down by transpositions until it is below $\pazocal{S}_1'$.

In each case, let $\Lambda_2'$ be the diagram resulting from applying the corresponding transpositions to $\Lambda_1'$. Further, let $\Pi_2'$ be the disk arising from $\Pi_2$. Then $\Lambda_2'$ must be $D$-minimal. Factor $\partial\Lambda_2'=\textbf{u}_2^{-1}\textbf{b}_2'\textbf{v}_2(\textbf{t}_2')^{-1}$ such that $\lab(\textbf{u}_2)\equiv\lab(\textbf{u}_1)$, $\lab(\textbf{v}_2)\equiv\lab(\textbf{v}_1)$, $\lab(\textbf{b}_2')\equiv\lab(\textbf{b}_1')$, and $\lab(\textbf{t}_2')\equiv\lab(\textbf{t}_1')$. Let $\Lambda_2$ be the subdiagram of $\Lambda_2'$ given by removing $\Pi_2'$ and enumerate the maximal $\theta$-bands in $\Lambda_2'$ as $\pazocal{S}_1'',\dots,\pazocal{S}_h''$. Then, we may factor $\partial\Lambda_2=\textbf{u}_2^{-1}\textbf{b}_2\textbf{v}_2\textbf{t}_2^{-1}$ such that $\textbf{b}_2\equiv\textbf{tbot}(\pazocal{S}_1'')$ and $\textbf{t}_2=\textbf{ttop}(\pazocal{S}_h'')$.

This process can then be iterated moving every disk above the top $\theta$-band or below the bottom $\theta$-band. 

The resulting reduced diagram $\Gamma''$ satisfies $\lab(\partial\Gamma'')\equiv\lab(\partial\Gamma)$. Enumerating the maximal $\theta$-bands of $\Gamma''$ as $\pazocal{T}_1'',\dots,\pazocal{T}_h''$, these $\theta$-bands bound a subdiagram $\Gamma_2''$ of $\Gamma''$ containing no disks and such that every cell of $\Gamma''\setminus\Gamma_2''$ is a disk.

Note that the transpositions performed to obtain $\Gamma''$ do not alter the side $q$-bands. So, identifying these $q$-bands with those in $\Gamma$, $\partial\Gamma_2''=\textbf{p}_1^{-1}\textbf{bot}(\pazocal{T}_1'')\textbf{p}_2\textbf{top}(\pazocal{T}_h'')^{-1}$. Hence, we may apply Lemma \ref{M_a reduced a-trapezia} to $\Gamma_2''$, yielding a minimal diagram $\Gamma_2'$. 

Factor $\partial\Gamma_2'$ as $(\textbf{p}_1')^{-1}\textbf{q}_1'\textbf{p}_2'(\textbf{q}_2')^{-1}$ such that $\lab(\textbf{p}_j)\equiv\lab(\textbf{p}_j')$. Further, let $\textbf{s}_1$ and $\textbf{s}_2$ be the simple paths such that the subdiagram $\Gamma_2$ of $\Gamma_2'$ with contour $(\textbf{p}_1')^{-1}\textbf{s}_1\textbf{p}_2'\textbf{s}_2^{-1}$ is a minimal $a$-trapezium.

Pasting $\Gamma_2'$ in place of $\Gamma_2''$ in $\Gamma''$ and making any necessary cancellations then produces a reduced diagram $\Gamma'$ satisfying (1) and (3).

In passing from $\Gamma$ to $\Gamma'$, no disks are added. So, since $\Gamma$ is minimal, (2) must be satisfied.

By Lemmas \ref{M_a no annuli 1}(1) and \ref{minimal theta-annuli}, $\Gamma$ contains no $(\theta,q)$-annuli or $\theta$-annuli. As no such annulus can be created through a transposition, it follows that every maximal $q$-band of $\Gamma_2''$ crosses every maximal $\theta$-band exactly once.

Since the $\theta$-band $\pazocal{T}_i$ did not participate in any of the transpositions in the construction of $\Gamma''$, the resulting maximal $\theta$-band $\pazocal{T}_i''$ in $\Gamma_2''$ also contains $m$ $(\theta,t)$-cells.  Hence, there are exactly $hm$ $(\theta,t)$-cells in $\Gamma_2''$. 

The minimality of $\Gamma_2$ then implies that it contains at most $hm$ $(\theta,t)$-cells. But $\Gamma$ is a minimal diagram containing at least $hm$ $(\theta,t)$-cells, so that both $\Gamma$ and $\Gamma_2$ must contain exactly $hm$ $(\theta,t)$-cells. 

As $\Gamma$ contains $h$ maximal $\theta$-bands and each contains at least $m$ $(\theta,t)$-cells, each of these $\theta$-bands must contain exactly $m$ $(\theta,t)$-cells. Conversely, by Lemmas \ref{M_a no annuli 1} and \ref{minimal theta-annuli}, each maximal $\theta$-band of $\Gamma_2$ must contain the same number of $(\theta,t)$-cells, which again must be $m$.

\end{proof}

\begin{remark}

The concept of $D$-minimal diagram is introduced in this paper specifically to aid in the iterative step in the proof of Lemma \ref{quasi-trapezia}. It is necessary for this goal as it is both sufficient as a hypothesis for Lemma \ref{graph} and preserved under transposition (whereas, for example, (M3) satisfies the first condition but not the second). 

\end{remark}

\bigskip

%%%%%%%%%%%%%%%%%%%%%%%%%%%%%%%%%%%%%%%%%%%%%%%%%%%%%%%%%%%%%%%%%

\subsection{Shafts} \

%\begin{lemma} \label{controlled has reduced}
%
%Suppose $\Delta$ is a quasi-trapezium such that its history $H$ is controlled. Then the base of $\Delta$ is a reduced word.
%
%\end{lemma}
%
%\begin{proof}
%
%Suppose there exists a disk or $a$-cell in $\Delta$. Then, Lemma \ref{quasi-trapezia} implies that there exists a quasi-trapezium $\Gamma$ over $M_a(\textbf{M})$ with the same base as $\Delta$. If there exists an $a$-cell in $\Gamma$, then Lemma \ref{a-bands between a-cells} implies that there exists an $a$-edge corresponding to the `special' input sector on either a $(\theta,q)$-cell of $\Gamma$ or on $\partial\Gamma$. However, since every rule of a controlled history locks the `special' input sector, this is impossible.
%
%So, we can assume $\Gamma$ is a trapezium, so that it suffices to consider trapezia with a controlled history. The statement then follows from Lemmas \ref{M controlled} and \ref{computations are trapezia}. 
%
%\end{proof}

We now introduce a concept that, as it was in [16] and [23], will be used to define a valuable measure on minimal diagrams.

%A trapezium $\Delta$ over the canonical presentation of $M(\textbf{M})$ is called \textit{standard} if its base is the standard base (or its inverse) and its history $H$ contains a subword $H_0^{\pm1}$ for some controlled history $H_0$. Naturally, a history $H$ (i.e $H\in F(\Theta^+)$) is called \textit{standard} if there exists a standard trapezium with history $H$. By Lemmas \ref{M controlled} and \ref{computations are trapezia}, a standard trapezium is uniquely determined by its history.

Let $\Pi$ be a disk contained in a minimal diagram and $\pazocal{B}$ be a $t$-spoke of $\Pi$. Suppose there is a subband $\pazocal{C}$ of $\pazocal{B}$ starting on $\Pi$ whose history $H$ contains a controlled subword. For $W$ the configuration corresponding to $\lab(\partial\Pi)$, suppose $W(i)$ is $H$-admissible for $i\geq2$. Then the $t$-band $\pazocal{C}$ is called a \textit{shaft} of $\Pi$. 

Note that this definition differs from that used in previous sources (for example, [16] and [23]), where it was required that $W$ be $H$-admissible. The change here is to allow for `flexibility' in the `special' input sector, as $W(1)$ need not be $H$-admissible.

For a disk $\Pi$, a shaft $\pazocal{C}$ of $\Pi$ is called a \textit{$\lambda$-shaft} of $\Pi$ if for every factorization $H\equiv H_1H_2H_3$ satisfying $\|H_1\|+\|H_3\|\leq\lambda\|H\|$, $H_2$ contains a controlled subword. Note that a shaft is a 0-shaft.

The following is an adaptation of Lemma 7.8 of [16] and Lemma 7.11 of [23] to this setting.

\begin{lemma} \label{shafts}

Let $\Pi$ be a disk in a minimal diagram $\Delta$ and $\pazocal{C}$ be a $\lambda$-shaft at $\Pi$ with history $H$. Then $\pazocal{C}$ has no factorization $\pazocal{C}=\pazocal{C}_1\pazocal{C}_2\pazocal{C}_3$ such that

\begin{enumerate}[label=({\arabic*})]

\item the sum of the lengths of $\pazocal{C}_1$ and $\pazocal{C}_3$ do not exceed $\lambda\|H\|$ and

\item $\Delta$ contains a quasi-trapezium $\Gamma$ such that the bottom (or top) of $\Gamma$ has $L$ $t$-edges and $\pazocal{C}_2$ starts on the bottom and ends on the top of $\Gamma$.

\end{enumerate}

\end{lemma}

\begin{proof}

Assuming toward contradiction, let $H_i$ be the history of the subband $\pazocal{C}_i$ for $i\in\{1,2,3\}$. Then, let $\tilde{\Delta}$ be the reduced diagram obtained from $\Delta$ by replacing $\Gamma$ with the reduced diagram $\Gamma'$ given in Lemma \ref{quasi-trapezia} and let $\Gamma_2$ be the minimal $a$-trapezium contained in $\Gamma'$. 

Lemma \ref{quasi-trapezia}(4) implies that the base of $\Gamma_2$ also contains $L$ $t$-letters. Moreover, as the side labels of $\Gamma_2$ and $\Gamma$ are the same, $\Gamma_2$ has history $H_2$.

By the definition of $\lambda$-shaft, $H_2$ must contain a controlled subword $H'$. So, Lemmas \ref{M controlled} and \ref{theta-bands are one-rule computations} imply that the base of $\Gamma_2$ must be reduced. Hence, assuming without loss of generality that the bottom (or top) label of $\Gamma_2$ starts and ends with one of its $L$ $t$-letters, $\Gamma_2$ is a big $a$-trapezium.

Let $\Lambda$ be the minimal diagram obtained from $\Gamma_2$ by removing one of the side $t$-bands. So, $\Lambda$ is an $a$-trapezium whose base $B$ is a cyclic permutation of the standard base (or its inverse). Let $\Lambda'$ be the subdiagram of $\Lambda$ that is an $a$-trapezium with base $B$ and history $H'$. Then, let $\Lambda_1$ and $\Lambda_2$ be the two subdiagrams of $\Lambda$ obtained by cutting along $\textbf{bot}(\Lambda')$.

Let $W$ be the configuration corresponding to $\partial\Pi$. By the definition of shaft, $W(i)$ is $H$-admissible for $i\geq2$. So, by Lemma \ref{extend 2}, there exists an accepted configuration $V$ and a reduced diagram $\Psi$ over $M_\Omega(\textbf{M})$ with $\lab(\partial\Psi)\equiv H(0)^{-1}WH(0)V^{-1}$, where $H(0)\in F(R)$ is the word obtained by adding the subscript $0$ to every letter of $H$.

Recall that $\Psi$ is constructed by pasting together trapezia corresponding to one-machine computations in the standard base along $a$-cells in the `special' input sector. So, any subdiagram bounded by two consecutive maximal $q$-bands not corresponding to the `special' input sector is a trapezium. As such, we may view $\Psi$ as an $a$-trapezium (though it may not be $M$-minimal), referring to its base, history, etc.

Since every rule locks the $Q_4(L)\{t(1)\}$-sector, cutting $\Psi$ along the appropriate $q$-band and pasting the sides together produces such a diagram so that the base of any maximal $\theta$-band is $B^{\pm1}$. Perhaps taking the mirror then produces a reduced diagram $\Psi_0$ with base $B$.

Let $\Psi''$ be the subdiagram of $\Psi_0$ bounded by the maximal $\theta$-bands corresponding to the history $H_2$. Similarly, let $\Psi'$ be the subdiagram of $\Psi''$ corresponding to the history $H'$. 

As computations with controlled history are one-machine, $\Psi'$ is a trapezium. Let $\Psi_1$, $\Psi_2$ be the two subdiagrams of $\Psi''$ obtained by cutting along $\textbf{bot}(\Psi')$, with the `bottom' of $\Psi_1$ corresponding to the `bottom' of $\Psi''$.

By Lemma \ref{M controlled}, $\lab(\textbf{bot}(\Lambda'))\equiv\lab(\textbf{bot}(\Psi'))$, so that $\lab(\textbf{top}(\Psi_1))\equiv\lab(\textbf{top}(\Lambda_1))$ and $\lab(\textbf{bot}(\Psi_2))\equiv\lab(\textbf{bot}(\Lambda_2))$. Further, as the histories of the side $q$-bands of $\Psi_j$ are the same as those of the side $q$-bands of $\Lambda_j$, $\Psi_j$ and $\Lambda_j$ have the same side labels.

Lemma \ref{a-cells sector} then implies that $\lab(\textbf{bot}(\Lambda))$ and $\lab(\textbf{bot}(\Psi''))$ (or $\lab(\textbf{top}(\Lambda))$ and $\lab(\textbf{top}(\Psi''))$) differ only by their projection to the `special' input sector. Gluing together the common contours of $\Psi_j$ and $\Lambda_j$, it then follows that these differences correspond to $a$-relations.

\newpage

\ \textbf{1.} Suppose $\|H_1\|+\|H_3\|=0$. 

Then $\Lambda$ and $\Pi$ have a common edge, so that they form a subdiagram $\Delta'$ of $\tilde{\Delta}$. Perhaps adding two pairs of cancellable $a$-cells to $\Delta'$, we obtain a (perhaps unreduced) diagram $\Delta_0'$ with the same contour label as $\Delta'$ and containing a subdiagram $\Delta_0''$ such that $\lab(\partial\Delta_0'')=V^{\pm1}$ in $F(\pazocal{X})$ and the complement of $\Delta_0''$ in $\Delta_0'$ consists of at most two $a$-cells (see Figure 9.7(a)).

Since $V$ is accepted, there exists a disk relation corresponding to $V$. So, we can replace $\Delta_0''$ in $\Delta_0'$ with one disk, producing the diagram $\Delta_1'$. 

Let $\tilde{\Delta}_1$ be the reduced diagram obtained by excising $\Delta'$ from $\tilde{\Delta}$, pasting $\Delta_1'$ in its place, and making any necessary cancellations. Then the number of disks in $\tilde{\Delta}_1$ is at most the number in $\tilde{\Delta}$, while the number of $(\theta,t)$-cells is strictly smaller. 

By Lemma \ref{quasi-trapezia}, $\tilde{\Delta}$ and $\Delta$ have the same number of disks and $(\theta,t)$-cells. But then $s_2(\tilde{\Delta}_1)< s_2(\Delta)$, contradicting the minimality of $\Delta$.

\renewcommand\thesubfigure{\alph{subfigure}}
\begin{figure}[H]
\centering
\begin{subfigure}[b]{0.48\textwidth}
\centering
\includegraphics[scale=1]{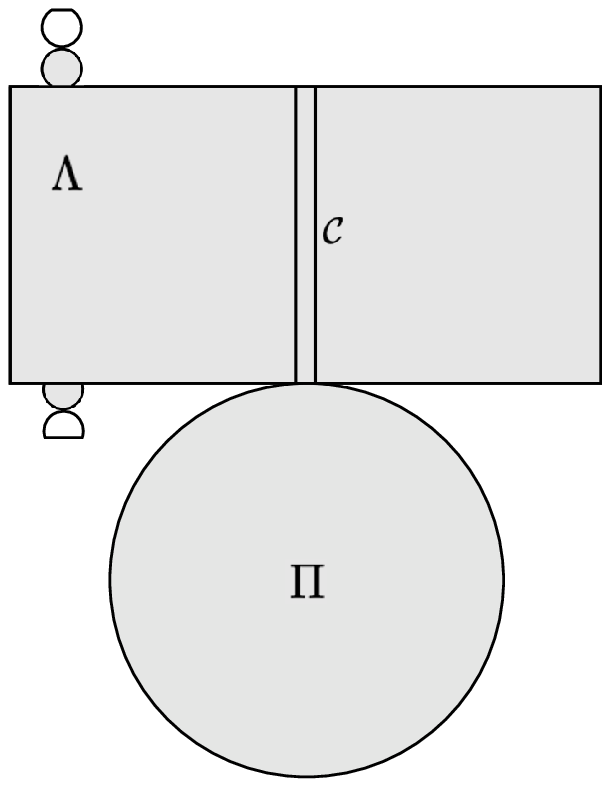}
\caption{The subdiagram $\Delta_0''$ if $\|H_1\|+\|H_3\|=0$}
\end{subfigure}\hfill
\begin{subfigure}[b]{0.48\textwidth}
\centering
\includegraphics[scale=1]{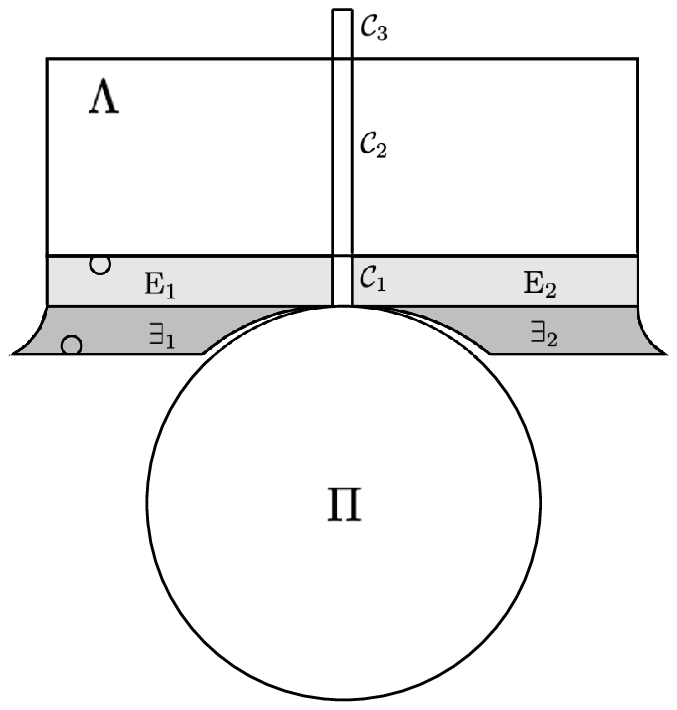}
\caption{The construction of $\tilde{\Delta}''$}
\end{subfigure}
\caption{Lemma \ref{shafts}}
\end{figure}

\ \textbf{2.} Suppose $\|H_1\|+\|H_3\|>0$. 

%Consider the subdiagram of $\Delta$ given by $\Pi\cup \pazocal{C}_1\pazocal{C}_2\cup\Lambda$.

Let $\Psi_1''$ be the subdiagram of $\Psi_0$ with base $B$ and history $H_1$, so that $\textbf{top}(\Psi_1'')=\textbf{bot}(\Psi_1)$.

Let E be the diagram obtained by attaching the appropriate $a$-cell to the top of $\Psi_1''$ so that the top label is the same as that of $\textbf{bot}(\Lambda_1)$. Further, let $\exists$ be the mirror image of E and $\exists$E be the diagram fomed by gluing $\exists$ to E along the bottom of $\Psi_1''$. Note that there are at most $\lambda\|H\|(L-1)$ $(\theta,t)$-cells in $\exists$. Then let $\tilde{\Delta}''$ be the (unreduced) diagram obtained from $\tilde{\Delta}$ by gluing the proper components of $\exists$E to the bottom of $\Lambda$ and along $\pazocal{C}_1$ (see Figure 9.7(b)).

Let $\Delta'$ be the subdiagram of $\tilde{\Delta}''$ formed by $\Pi$, $\pazocal{C}_1$, $\Lambda$, and the components of E. As in the previous case, we may replace $\Delta'$ with a diagram made of one disk and perhaps some new $a$-cells. 

After necessary cancellations, the resulting reduced diagram $\tilde{\Delta}_1$ at most as many disks as $\tilde{\Delta}''$, and so the same number as $\tilde{\Delta}$. In passing to this diagram, we added at most $\lambda\|H\|(L-1)$ $(\theta,t)$-cells from $\exists$, while removing at least $\|H_2\|(L-1)\geq(1-\lambda)\|H\|(L-1)$ $(\theta,t)$-cells of $\Lambda$. 

Taking $\lambda<1/2$, it follows that $\tilde{\Delta}_1$ has less $(\theta,t)$-cells than $\tilde{\Delta}$. Thus, as in the previous case, Lemma \ref{quasi-trapezia} implies that $s_2(\tilde{\Delta}_1)<s_2(\Delta)$, contradicting the minimality of $\Delta$.

\end{proof}

\smallskip

%%%%%%%%%%%%%%%%%%%%%%%%%%%%%%%%%%%%%%%%%%%%%%%%%%%%%%%%%%%%%%%%%

\subsection{Designs on a Disk} \

In this section, we recall the measure on minimal diagrams, first introduced in [16], that was alluded to in Section 9.5.

Let $\pazocal{D}$ be a disk in the Euclidean plane, $\textbf{T}$ be a finite set of disjoint chords, and $\textbf{Q}$ be a finite set of disjoint simple curves in $\pazocal{D}$, called \textit{arcs} (as to differentiate them from the chords). 

Assume that arcs belong to the open disk $\pazocal{D}^\circ$ and that each chord crosses any arc transversely and at most one, with the intersection not coming at either of the arc's endpoints.

With these assumptions, the pair $(\textbf{T},\textbf{Q})$ is called a \textit{design} on the disk.

The length of an arc $C\in\textbf{Q}$, denoted $|C|$, is the number of chords crossing it. \textit{Subarcs} are defined in the natural way, so that the inequality $|D|\leq|C|$ is clear for $D$ a subarc of $C$.

An arc $C_1$ is \textit{parallel} to an arc $C_2$, denoted $C_1 \ \| \ C_2$, if every chord crossing $C_1$ also crosses $C_2$. Note that this relation is reflexive and transitive, but not symmetric.

For $\lambda\in(0,1/2)$ the parameter listed in Section 3.3 and $m$ a positive integer, a design $(\textbf{T},\textbf{Q})$ is said to satisfy property $P(\lambda,m)$ if for any collection of $m$ distinct arcs $C_1,\dots,C_m\in\textbf{Q}$, there are no subarcs $D_1,\dots,D_m$, respectively, such that $|D_i|>(1-\lambda)|C_i|$ for all $i$ and $D_1 \ \| \ D_2 \ \| \dots \| \ D_m$.

For a design $(\textbf{T},\textbf{Q})$, define the length of $\textbf{Q}$, $\ell(\textbf{Q})$, to be $\ell(\textbf{Q})=\sum\limits_{C\in\textbf{Q}}|C|$.

%Every chord $T\in\textbf{T}$ divides $\pazocal{D}$ into two half-disks. If one of these half-disks contains no other chords, then the chord $T$ is called \textit{peripheral} and the half-disk $O_T$ with no chords a \textit{peripheral half-disk}.
%
%For $C\in\textbf{Q}$, an arc $D$ is called an \textit{extension} of $C$ if $C$ is a subarc of $D$. Note that extensions need not be elements of $\textbf{Q}$; however, we only consider extensions of elements of $\textbf{Q}$ that, upon replacing the arcs with that extensions, yield a new design $(\textbf{T},\textbf{Q}')$. An arc $C\in\textbf{Q}$ is \textit{maximal} if there is no extension $D$ of $C$ such that $|D|>|C|$.

\medskip

\begin{lemma} \label{design}

\textit{(Lemma 8.2 of [16])} There is a constant $c$ dependant on $\lambda$ and $m$ such that for any design $(\textbf{T},\textbf{Q})$ satisfying property $P(\lambda,m)$, $\ell(\textbf{Q})\leq c(\#\textbf{T})$.

\end{lemma}

Let $\Delta$ be a minimal diagram and $\pazocal{Q}$ be a $t$-spoke of a disk $\Pi$ in $\Delta$. Let $\pazocal{Q}_\Pi$ be the subband of $\pazocal{Q}$ which is a $\lambda$-shaft at $\Pi$ of maximal length. Then, define $\sigma_\lambda(\Delta)$ as the sum of the lengths of the $\lambda$-shafts $\pazocal{Q}_\Pi$ for all disks $\Pi$ and $t$-spokes $\pazocal{Q}$.

If $\Delta$ is a minimal diagram, then identify $\Delta$ with a disk and construct the design $(\textbf{T},\textbf{Q})$ as follows: Let the middle lines of maximal $\theta$-bands be the chords and the middle lines of maximal $\lambda$-shafts be the arcs. 

Note that there is a subtle hindrance to this construction: If a maximal $t$-spoke connects two disks, then it may contain a $\lambda$-shaft at each disk, and these $\lambda$-shafts may overlap. However, this issue can be remedied simply by `making room' in the spoke for both arcs to fit and be disjoint. 

Note that the length $|C|$ of an arc with respect to this design is the number of cells in the $\lambda$-shaft and $\#\textbf{T}=\frac{1}{2}|\partial\Delta|_\theta\leq\frac{1}{2}|\partial\Delta|$ since every maximal $\theta$-band ends twice on $\partial\Delta$.

\begin{lemma} \label{G_a design}

\textit{(Lemma 8.5 of [16])} If $\Delta$ is a minimal diagram, then $\sigma_\lambda(\Delta)\leq C_1|\partial\Delta|_\theta\leq C_1|\partial\Delta|$.

\end{lemma}

\begin{proof}

By Lemma \ref{design} and the parameter choices $C_1>>L>>\lambda^{-1}$, it suffices to prove that the design $(\textbf{T},\textbf{Q})$ satisfies Property $P(\lambda,2L-1)$.

Arguing toward contradiction, there are $2L-1$ maximal $\lambda$-shafts $\pazocal{C}_1,\dots,\pazocal{C}_{2L-1}$ such that for some subband $\pazocal{D}$ of $\pazocal{C}_1$, $|\pazocal{D}|>(1-\lambda)|\pazocal{C}_1|$ and every maximal $\theta$-band crossing $\pazocal{D}$ also crosses each of $\pazocal{C}_2,\dots,\pazocal{C}_{2L-1}$. So, since at most two of these $\lambda$-shafts correspond to any particular $t$-spoke, each of the $\theta$-bands crossing $\pazocal{D}$ crosses at least $L$ $t$-bands. 

But then the $\lambda$-shaft $\pazocal{C}_1$ crosses a quasi-trapezium of height $|\pazocal{D}|>(1-\lambda)|\pazocal{C}_1|$ whose base has at least $L$ $t$-letters, contradicting Lemma \ref{shafts}.

\end{proof}

\medskip

%%%%%%%%%%%%%%%%%%%%%%%%%%%%%%%%%%%%%%%%%%%%%%%%%%%%%%%%%%%%%%%%%

\section{Upper bound on the weight of minimal diagrams}

\subsection{Weakly minimal diagrams} \

\begin{figure}[H]
\centering
\includegraphics[scale=0.75]{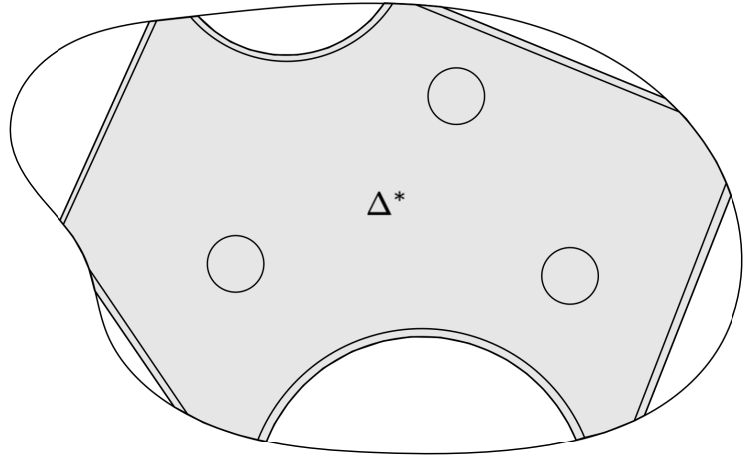}
\caption{The stem of a reduced diagram $\Delta$ containing disks}
\end{figure}

The goal in this section is to bound the $G$-weight of all minimal diagrams $\Delta$ in terms of $|\partial\Delta|^2$. In light of Lemma \ref{diskless}, it suffices to restrict our attention to minimal diagrams containing disks. However, it proves necessary to consider a larger class of diagrams over the disk presentation of $G_\Omega(\textbf{M})$, called weakly minimal.

Let $\Delta$ be a reduced diagram over the disk presentation of $G_\Omega(\textbf{M})$ which contains a disk. Then, let $\pazocal{C}$ be a cutting $q$-band of $\Delta$, i.e $\pazocal{C}$ ends twice on the boundary of $\Delta$. Then $\pazocal{C}$ is called a \textit{stem band} if it is either a rim band of $\Delta$ or both components of $\Delta\setminus\pazocal{C}$ contain disks. The unique maximal subdiagram of $\Delta$ satisfying the property that every cutting $q$-band is a stem band is called the \textit{stem} of $\Delta$ and denoted $\Delta^*$.

If $\pazocal{C}$ is a cutting $q$-band that is not a stem band, then exactly one component $\Gamma$ of $\Delta\setminus\pazocal{C}$ contains no disks. In this situation, the cells of $\Gamma$ are called \textit{crown} cells. Note that one can construct $\Delta^*$ from $\Delta$ simply by cutting off all of the crown cells.

Finally, a reduced diagram $\Delta$ over the disk presentation of $G_\Omega(\textbf{M})$ which contains a disk is called \textit{weakly minimal} if:

\begin{addmargin}[1em]{0em}

\begin{enumerate}[label=(WM{\arabic*})]

%\item it contains no $\theta$-annulus $S$ whose sides are labelled by letters of the tape alphabet of the `special' input sector,

\item for any $a$-cell $\pi$ and any $\theta$-band $\pazocal{T}$, at most half of the edges of $\partial\pi$ mark the start of an $a$-band that crosses $\pazocal{T}$,

\item no maximal $a$-band ends on two different $a$-cells, and 

\item its stem $\Delta^*$ is a minimal diagram. 

\end{enumerate}

\end{addmargin}

Note that conditions (WM1) and (WM2) are identical to conditions (MM1) and (MM2) in the definition of $M$-minimal (see Section 8.1). As a result, any subdiagram of a weakly minimal diagram which contains no disks is $M$-minimal. 

Conversely, any minimal diagram containing a disk is weakly minimal.

\begin{lemma} \label{weakly minimal} \textit{(Compare to Lemma 9.3 of [16] and Lemma 7.17 of [23])}

\begin{enumerate}[label=({\alph*})]

\item If $\Delta_1$ is a subdiagram of a weakly minimal diagram $\Delta$ and contains a disk, then $\Delta_1$ is weakly minimal, $\Delta_1^*\subset\Delta^*$, and $\sigma_\lambda(\Delta_1^*)\leq\sigma_\lambda(\Delta^*)$.

\item For every weakly minimal diagram $\Delta$, $\sigma_\lambda(\Delta^*)\leq C_1|\partial\Delta|$.

\item A weakly minimal diagram $\Delta$ contains no $\theta$-annuli.

\item Let $\pazocal{C}$ be a cutting $q$-band of a reduced diagram $\Delta$ over the disk presentation of $G_\Omega(\textbf{M})$ and let $\Delta_1$, $\Delta_2$ be the components of $\Delta\setminus\pazocal{C}$. Suppose $\Delta_1\cup\pazocal{C}$ is $M$-minimal (over $M_\Omega(\textbf{M})$) and $\Delta_2\cup\pazocal{C}$ is weakly minimal. Then $\Delta$ is weakly minimal.

\end{enumerate}

\end{lemma}

\begin{figure}[H]
\centering
\includegraphics[scale=0.75]{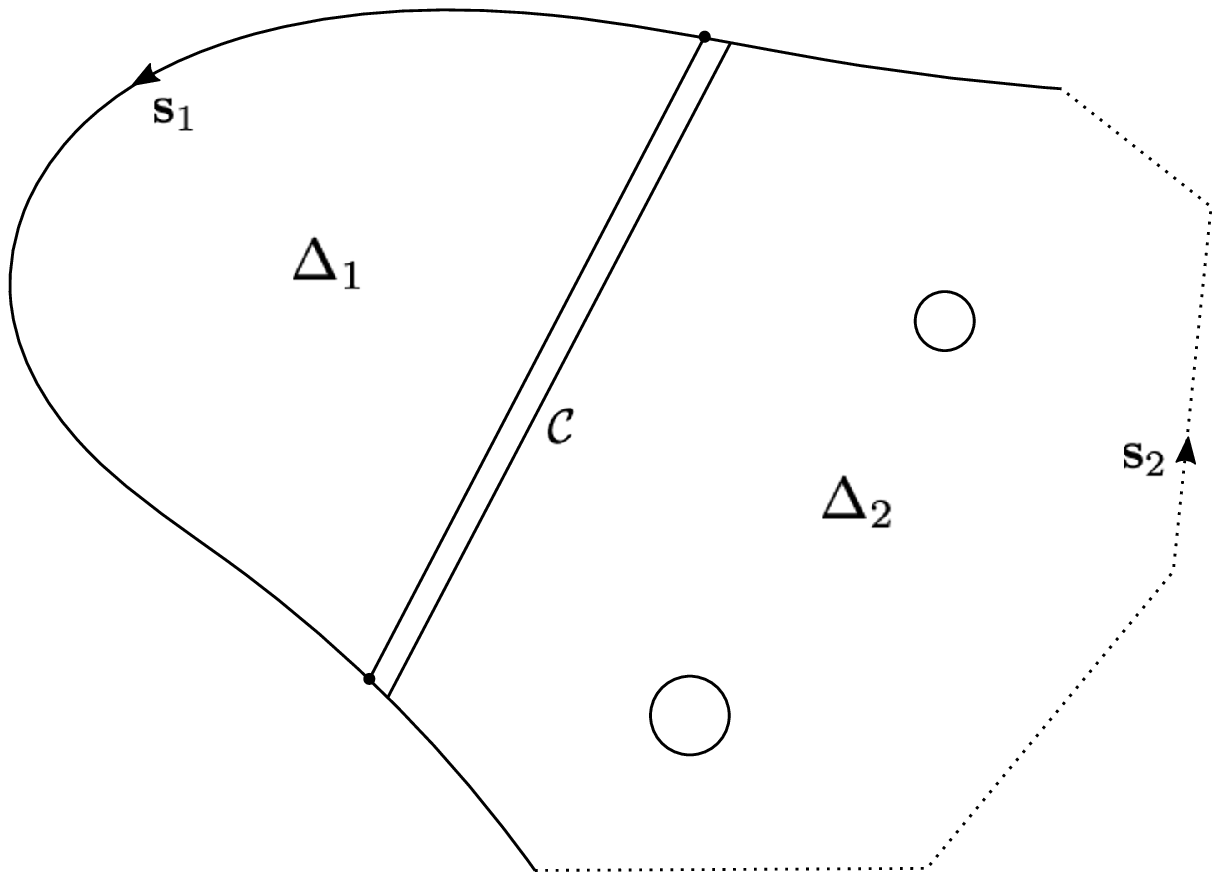}
\caption{Lemma \ref{weakly minimal}(b)}
\end{figure}

\begin{proof}

(a) Let $\pi$ be a crown cell of $\Delta$ contained in $\Delta_1$. Then there exists a cutting $q$-band $\pazocal{Q}$ separating $\pi$ from all disks of $\Delta$. The intersection of $\pazocal{Q}$ with $\Delta_1$ is a cutting $q$-band separating $\pi$ from all disks of $\Delta_1$, so that $\pi$ is a crown cell of $\Delta_1$. Consequently, $\Delta_1^*\subset\Delta^*$, and hence $\Delta_1^*$ is minimal being a subdiagram of a minimal diagram.

By definition, every maximal $\lambda$-shaft $\pazocal{C}_1$ of $\Delta_1^*$ is contained in a maximal $\lambda$-shaft $\pazocal{C}$ of $\Delta^*$. The length of $\pazocal{C}_1$ is then at most as large as the length of $\pazocal{C}$, so that $\sigma_\lambda(\Delta_1^*)\leq\sigma_\lambda(\Delta^*)$.

(b) Suppose $\pazocal{C}$ is a cutting $q$-band that is not a stem band and set $\Delta_1$, $\Delta_2$ as the components of $\Delta\setminus\pazocal{C}$. Then, one of $\Delta_1$ or $\Delta_2$, say $\Delta_1$, is diskless. 

Let $\textbf{s}_1$ be the portion of $\partial\Delta_1$ shared with $\partial\Delta$ and set $\partial\Delta=\textbf{s}_1\textbf{s}_2$. By Lemma \ref{M_a no annuli 1}, every maximal $\theta$-band of $\Delta_1\cup\pazocal{C}$ intersects $\pazocal{C}$ at most once. So, for every $\theta$-edge on the side of $\pazocal{C}$, there is a maximal $\theta$-band of $\Delta_1\cup\pazocal{C}$ with one end on this edge and one end on $\textbf{s}_1$.

So, letting $\ell$ be the number of $\theta$-edges of $\textbf{s}_1$, Lemma \ref{lengths}(b) implies $|\textbf{s}_1|\geq\ell=|\textbf{bot}(\pazocal{C})|=|\textbf{top}(\pazocal{C})|$. Hence, for $\Delta'=\Delta_2\cup\pazocal{C}=\Delta\setminus\Delta_1$, Lemma \ref{lengths}(c) implies $|\partial\Delta'|\leq|\textbf{s}_2|+|\textbf{top}(\pazocal{C})|\leq|\textbf{s}_2|+|\textbf{s}_1|$. 

But $\textbf{s}_2$ starts and ends with $q$-letters, so that $|\textbf{s}_2|+|\textbf{s}_1|=|\partial\Delta|$.

Iterating this process, $|\partial\Delta^*|\leq|\partial\Delta|$. 

Thus. the statement is a consequence of Lemma \ref{G_a design}.

(c) Lemma \ref{M_a no annuli 2}(2) implies that no $\theta$-annulus can be contained in a crown of $\Delta$. Since $\Delta^*$ is minimal, Lemma \ref{minimal theta-annuli} implies that no $\theta$-annulus can be contained in $\Delta^*$. 

Thus, the statement follows from Lemma \ref{M_a no annuli 1}, as no $\theta$ band can cross a rim $q$-band of $\Delta^*$ twice. 

(d) Suppose there exists a counterexample to (WM1) in $\Delta$ and fix $j\in\{1,2\}$ such that the $a$-cell $\pi$ is contained in $\Delta_j$. Let $\pazocal{T}_j$ be the maximal subband of $\pazocal{T}$ contained in $\Delta_j$. Then, since $a$-bands cannot cross $q$-bands, $\pi$ and $\pazocal{T}_j$ form a counterexample to (WM1) in $\Delta_j$. But this contradicts the $M$-minimality of $\Delta_1$ or the weak minimality of $\Delta_2$.

Similarly, any counterexample to (WM2) in $\Delta$ must be contained entirely in $\Delta_1$ or $\Delta_2$, leading to a contradiction.

Finally, it is clear from the definition that $\Delta^*=(\Delta_2\cup\pazocal{C})^*$, and so is minimal.

\end{proof}

\smallskip

%%%%%%%%%%%%%%%%%%%%%%%%%%%%%%%%%%%%%%%%%%%%%%%%%%%%%%%%%%%%%%%%%

\subsection{Definition of the minimal counterexample and cloves} \

The objective of the rest of this section is to exhibit an upper bound for the $G$-weight of a weakly minimal diagram in terms of its perimeter. In particular, we will prove that for any weakly minimal diagram $\Gamma$, the inequality $$\text{wt}_G(\Gamma)\leq N_4(|\partial\Gamma|+\sigma_\lambda(\Gamma^*))^2+N_3\mu(\Gamma)$$ holds for large enough choices of the parameters $N_4$ and $N_3$. The proof of this follows a similar path as that presented in Section 7 of [23] and Section 9 of [16] (taking $F(x)=x^2$ and $g(x)=x$ in that setting). 
%As such, many of the proofs of the statements that follow will either be omitted or amount to a remark on how it differs from the corresponding proof in [23].

Let $\Delta$ be a `minimal counterexample' diagram with respect to $|\partial\Delta|+\sigma_\lambda(\Delta^*)$, i.e a weakly minimal diagram satisfying
$$\text{wt}_G(\Delta)>N_4(|\partial\Delta|+\sigma_\lambda(\Delta^*))^2+N_3\mu(\Delta)$$ 
while for any weakly minimal diagram $\Gamma$ such that $|\partial\Gamma|+\sigma_\lambda(\Gamma^*)<|\partial\Delta|+\sigma_\lambda(\Delta^*)$, we have
$$\text{wt}_G(\Gamma)\leq N_4(|\partial\Gamma|+\sigma_\lambda(\Gamma^*))^2+N_3\mu(\Gamma)$$ 
As with Lemma \ref{a-cell in counterexample}, the following statement is an immediate consequence of the inductive hypothesis.

\begin{lemma} \label{a-cell in counterexample 2}

Let $\pi$ be an $a$-cell contained in $\Delta$. Suppose $\partial\pi$ has a subpath $\textbf{s}$ shared with $\partial\Delta$. Then $\|\textbf{s}\|\leq\frac{2}{3}\|\partial\pi\|$.

\end{lemma}

Since $\Delta^*$ contains every disk of $\Delta$ and is minimal, $\Delta$ is a $D$-minimal diagram. So, Lemma \ref{graph} guarantees that it contains a disk $\Pi$ with $L-4$ consecutive $t$-spokes $\pazocal{Q}_1,\dots,\pazocal{Q}_{L-4}$ ending on $\partial\Delta$ and bounding $L-5$ diskless subdiagrams (see Figure 9.3).

For $1\leq i<j\leq L-4$, the subdiagram of $\Delta$ bounded by $\partial\Pi$, $\pazocal{Q}_i$, and $\pazocal{Q}_j$ (and not containing $\Pi$) is called a \textit{clove} and is denoted $\Psi_{ij}$. The maximal clove $\Psi_{1,L-4}$ is simply denoted $\Psi$.

\begin{lemma} \label{rim theta-bands}

\textit{(Compare to Lemma 9.5 of [16] and Lemma 7.19 of [23])} 

Let $\pazocal{T}$ be a quasi-rim $\theta$-band in $\Delta$. Then the base of $\pazocal{T}$ has length $s>K$.

\end{lemma}

\begin{proof}

Assume toward contradiction that $\pazocal{T}$ is a quasi-rim $\theta$-band with base of length $s\leq K$. Then, define $\Delta'$ and $\Delta''$ as in the proof of Lemma \ref{6.18}. As in that setting, $\Delta''$ satisfies (WM1) and (WM2) and $|\partial\Delta''|\leq|\partial\Delta|-1$.

As $\Delta'$ is a subdiagram of $\Delta$, Lemma \ref{weakly minimal}(a) implies that it is weakly minimal with $\sigma_\lambda((\Delta')^*)\leq\sigma_\lambda(\Delta^*)$. 

Since the diagram $\Delta''$ is formed from $\Delta'$ through the addition of $a$-cells, the 2-signatures of $(\Delta')^*$ and $(\Delta'')^*$ are equal. Hence, $\Delta''$ is a weakly minimal diagram.

Further, every $\lambda$-shaft of $(\Delta'')^*$ is at most as long as the corresponding $\lambda$-shaft of $\Delta^*$, so that $\sigma_\lambda((\Delta'')^*)\leq\sigma_\lambda(\Delta^*)$. Consequently, $|\partial\Delta''|+\sigma_\lambda((\Delta'')^*)\leq|\partial\Delta|+\sigma_\lambda(\Delta^*)-1$, and so the inductive hypothesis may be applied to $\Delta''$.

Thus, the proof of Lemma \ref{6.18} adapts naturally to this setting, providing a contradiction.

\end{proof}

\smallskip

%%%%%%%%%%%%%%%%%%%%%%%%%%%%%%%%%%%%%%%%%%%%%%%%%%%%%%%%%%%%%%%%%

\subsection{Properties of the cloves of $\Delta$} \

The following statement is an adaptation of Lemma \ref{6.16} to this setting and is proved in exactly the same way.

\begin{lemma} \label{7.20} \

\begin{enumerate}[label=({\arabic*})]

\item $\Delta$ has no two disjoint subcombs $\Gamma_1$ and $\Gamma_2$ contained in $\Psi$ with basic widths at most $K$ and handles $\pazocal{B}_1$ and $\pazocal{B}_2$ such that some ends of these handles are connected by a subpath $\textbf{x}$ of $\partial\Delta$ with $|\textbf{x}|_q\leq c_0$.

\item If $\Gamma$ is a subcomb of $\Delta$ contained in $\Psi$ with basic width $s\leq K$, $|\partial\Gamma|_q=2s$.

\end{enumerate}

\end{lemma}

\smallskip

\begin{lemma} \label{7.21}

Any subcomb of $\Delta$ contained in $\Psi$ has basic width at most $K_0$.

\end{lemma}

\begin{proof}

Assume toward contradiction that there exists a subcomb of $\Delta$ contained in $\Psi$ with basic width $s>K_0$. Then, using Lemma \ref{rim theta-bands}, an identical proof to the one presented in Lemma \ref{tight subcomb} implies that there exists a tight subcomb $\Gamma$ of $\Delta$ contained in $\Psi$.

Further, an analogous proof to that presented in Lemma \ref{6.17} implies that any subcomb of $\Gamma$ has height greater than $\ell'/2$. Indeed, other than switching the parameters and using Lemma \ref{7.20} in place of Lemma \ref{6.18}, the only necessary alteration to the proof of Lemma \ref{6.17} is in the application of the inductive hypothesis, where we must use the inequality $\sigma_\lambda((\Delta')^*)\leq\sigma_\lambda(\Delta^*)$ arising from Lemma \ref{weakly minimal}(a).

But then similar analogues of Lemmas \ref{counterexample combs}-\ref{diskless} yield a contradiction in the same way. Only one major alteration is needed: In the adaptation of Lemma \ref{diskless}, the diagram $\Delta_0$ is weakly minimal by Lemma \ref{weakly minimal}(d) and satisfies $\sigma_\lambda(\Delta_0^*)=\sigma_\lambda(\Delta^*)$ since the handle of the tight subcomb $\Gamma$ is a non-stem cutting $q$-band.

\end{proof}

\begin{remark}

The reason for our consideration of weakly minimal diagrams in this section is revealed in the proof of Lemma \ref{7.21}: The adaptation of Lemma \ref{diskless} relies on Lemma \ref{weakly minimal}(d), whose statement fails if one replaces `weakly minimal' with `minimal'.

\end{remark}

\begin{lemma} \label{7.22} \textit{(Compare with Lemma 9.8 of [16] and Lemma 7.22 of [23])}

\begin{enumerate}[label=({\arabic*})]

\item Every maximal $\theta$-band of $\Psi$ crosses either $\pazocal{Q}_1$ or $\pazocal{Q}_{L-4}$

\item There exists an $r$ satisfying $(L-1)/2-3\leq r\leq (L-1)/2$ such that the $\theta$-bands of $\Psi$ crossing $\pazocal{Q}_{L-4}$ do not cross $\pazocal{Q}_r$ and the $\theta$-bands of $\Psi$ crossing $\pazocal{Q}_1$ do not cross $\pazocal{Q}_{r+1}$

\end{enumerate}

\end{lemma}

\begin{proof}

(1) Suppose there exists a maximal $\theta$-band $\pazocal{T}$ of $\Psi$ crossing neither $\pazocal{Q}_1$ nor $\pazocal{Q}_{L-4}$. As $\theta$-bands cannot cross, we may assume that $\pazocal{T}$ is a quasi-rim $\theta$-band. By Lemma \ref{rim theta-bands}, $\pazocal{T}$ must cross more than $K$ maximal $q$-bands of $\Psi$. 

Taking $K>11L+2K_0$, there exists a non-stem cutting $q$-band $\pazocal{C}'$ crossing $\pazocal{T}$ such that for $\Gamma'$ the subdiagram of $\Psi$ consisting of $\pazocal{C}'$ and the corresponding crown, $\Gamma'$ contains no cells of the spokes of $\Pi$ and at least $K_0$ maximal $q$-bands crossing $\pazocal{T}$ (see Figure 10.3).

By Lemma \ref{7.21}, $\Gamma'$ cannot be a comb with handle $\pazocal{C}'$, and so must contain a maximal $\theta$-band $\pazocal{T}'$ not crossing $\pazocal{C}'$. As above, we may assume $\pazocal{T}'$ is a quasi-rim $\theta$-band containing greater than $K$ $(\theta,q)$-cells, yielding a non-stem cutting $q$-band $\pazocal{C}''$ crossing $\pazocal{T}'$ such that the corresponding subdiagram $\Gamma''$ does not contain any cell of $\pazocal{C}'$ but contains at least $K_0$ $q$-bands crossing $\pazocal{T}'$.

\begin{figure}[H]
\centering
\includegraphics[scale=1]{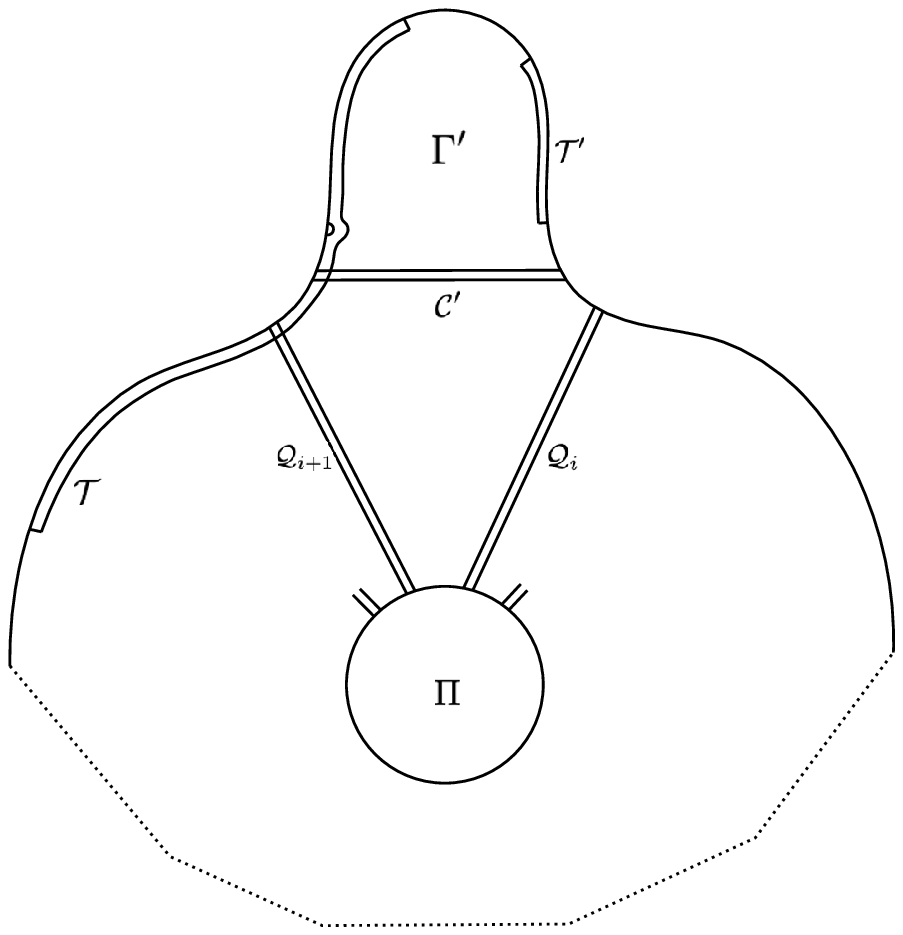}
\caption{Lemma \ref{7.22}(1)}
\end{figure}

Iterating this process, we obtain a series of subdiagrams $\Gamma',\Gamma'',\dots$ with $\text{wt}(\Gamma')>\text{wt}(\Gamma'')>\dots$. Since these diagrams are finite, this process must terminate. But then the resulting subdiagram is a subcomb of $\Delta$ contained in $\Psi$ with basic width at least $K_0$, contradicting Lemma \ref{7.21}.

(2) Let $\pazocal{T}$ be the maximal $\theta$-band of $\Psi$ crossing the $t$-spoke $\pazocal{Q}_1$ closest to $\Pi$, i.e the intersection of $\pazocal{T}$ and $\pazocal{Q}_1$ is the first cell of $\pazocal{Q}_1$. 

Note that all spokes of a disk in $\Delta$ must lie in the minimal diagram $\Delta^*$. So, if $\pazocal{Q}_1,\dots,\pazocal{Q}_\ell$ are the $t$-spokes crossed by $\pazocal{T}$, then $\ell\leq(L-1)/2$ by Lemma \ref{G_a theta-annuli}(1). Since $\pazocal{T}$ does not cross $\pazocal{Q}_{\ell+1}$, no other maximal $\theta$-band of $\Psi$ crossing $\pazocal{Q}_1$ can either. Similarly, no maximal $\theta$-band crossing $\pazocal{Q}_{L-4}$ can cross $\pazocal{Q}_\ell$.

By the symmetric argument, if $\pazocal{Q}_{s+1},\dots,\pazocal{Q}_{L-4}$ are the spokes crossed by the maximal $\theta$-band crossing $\pazocal{Q}_{L-4}$ closest to $\Pi$, then no $\theta$-band crossing $\pazocal{Q}_1$ can cross $\pazocal{Q}_{s+1}$ and $(L-4)-s\leq(L-1)/2$, i.e $s\geq(L-1)/2-3$.

Thus, the statement follows for $r=\max(\ell,(L-1)/2-3)$.

\end{proof}

\smallskip

\subsection{Paths in the cloves} \

For $1\leq i<j\leq L-4$, denote $\textbf{p}_{ij}$ as the shared subpath of $\partial\Psi_{ij}$ and $\partial\Delta$. For simplicity, denote the path $\textbf{p}_{1,L-4}$ associated to the maximal clove simply as $\textbf{p}$.

Let $\bar{\Delta}$ be the subdiagram of $\Delta$ consisting of $\Pi$ and $\Psi$. Then, let $\bar{\textbf{p}}=\textbf{bot}(\pazocal{Q}_1)^{-1}\textbf{u}^{-1}\textbf{top}(\pazocal{Q}_{L-4})$ where $\textbf{u}$ is a subpath of $\partial\Pi$ and such that cutting along $\bar{\textbf{p}}$ separates $\Delta$ into two components, one of which is $\bar{\Delta}$. Denote the other component $\Psi'$.

Similarly, for $1\leq i<j\leq L-4$, define the the path $\bar{\textbf{p}}_{ij}=\textbf{bot}(\pazocal{Q}_i)^{-1}\textbf{u}_{ij}^{-1}\textbf{top}(\pazocal{Q}_j)$ and the subdiagrams $\bar{\Delta}_{ij}$ and $\Psi_{ij}'$ (see Figure 10.4).

\begin{figure}[H]
\centering
\includegraphics[scale=0.75]{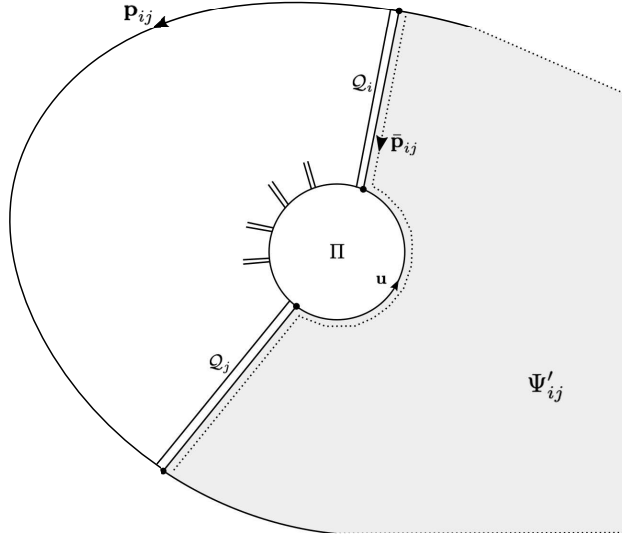}
\caption{Subdiagrams and paths in $\Delta$}
\end{figure}

Let $H_1,\dots,H_{L-4}$ be the histories of the spokes $\pazocal{Q}_1,\dots,\pazocal{Q}_{L-4}$, respectively, read starting from the disk $\Pi$. Further, let $h_i=\|H_i\|$ for all $i$. Lemma \ref{7.22} then implies the inequalities
$$h_1\geq h_2\geq\dots\geq h_r; \ \ h_{r+1}\leq\dots\leq h_{L-4}$$ 

where $(L-1)/2-3\leq r\leq(L-1)/2$. It then follows that $H_{i+1}$ is a prefix of $H_i$ for $i=1,\dots,r-1$ while $H_j$ is a prefix of $H_{j+1}$ for $j=r+1,\dots,L-5$.

Let $W$ be the accepted configuration corresponding to $\lab(\partial\Pi)$. Then, using the notation of Section 5.3, $W\equiv W(1)W(2)\dots W(L)$, where $W(2),\dots,W(L)$ are all copies of the same configuration $V$ of $\textbf{M}_4$. Further, by Lemma \ref{accepted configuration a-length}, $|W(1)|_a\leq2|V|_a$.

\begin{lemma} \label{7.23} 

\textit{(Compare with Lemma 9.9 of [16] and Lemma 7.23 of [23])} For $1\leq i\leq L-5$, $|\textbf{p}_{i,i+1}|_q<3K_0$.

\end{lemma}

\begin{proof}

Suppose there exists a maximal $q$-band $\pazocal{B}$ which is not a spoke of $\Pi$ and has one end on $\textbf{p}_{i,i+1}$. Then, since $q$-bands cannot cross, $\pazocal{B}$ must have two ends on $\textbf{p}_{i,i+1}$. So, $\pazocal{B}$ is a non-stem cutting $q$-band. Let $\Gamma$ be the subdiagram of $\Psi$ consisting of $\pazocal{B}$ and the corresponding crown. By Lemma \ref{7.22}(1), any maximal $\theta$-band in $\Gamma$ must cross $\pazocal{B}$, so that $\Gamma$ is a comb with handle $\pazocal{B}$. 

Note that any subcomb whose handle ends twice on $\textbf{p}_{i,i+1}$ lies in a maximal subcomb with this condition. 

Suppose $\Psi_{i,i+1}$ contains two such maximal subcombs. Then, let $\Gamma_1$ and $\Gamma_2$ be two adjacent such subcombs with handles $\pazocal{B}_1$ and $\pazocal{B}_2$, respectively. By Lemma \ref{7.21}(1), $\Gamma_1$ and $\Gamma_2$ are disjoint subcombs of $\Delta$ contained in $\Psi$ with basic widths at most $K_0$. Moreover, as we assume these subcombs are adjacent, there exists a subpath $\textbf{x}$ of $\textbf{p}_{i,i+1}$ connecting $\pazocal{B}_1$ and $\pazocal{B}_2$ such that any $q$-edge of $\textbf{x}$ is the end a spoke of $\Pi$. As at most $23$ spokes of $\Pi$ can end on $\textbf{p}_{i,i+1}$, $|\textbf{x}|_q\leq c_0$. But this contradicts Lemma \ref{7.20}(1).

As a result, $\Psi_{i,i+1}$ contains at most one maximal subcomb whose handle ends twice on $\textbf{p}_{i,i+1}$. By Lemmas \ref{7.20}(2) and \ref{7.21}(1), such a subcomb contributes at most $2K_0$ $q$-edges to $\textbf{p}_{i,i+1}$.

Thus, $|\textbf{p}_{i,i+1}|_q\leq 2K_0+23<3K_0$ by a parameter choice for $K_0$.

\end{proof}

\begin{lemma} \label{7.24} \textit{(Compare with Lemma 9.10 of [16] and Lemma 7.24 of [23])}

\begin{enumerate}[label=({\arabic*})]

\item If $i\leq r$ and $j\geq r+1$, then $|\textbf{p}_{ij}|\geq|\textbf{p}_{ij}|_\theta+|\textbf{p}_{ij}|_q\geq h_i+h_j+11(j-i)+1$

\item $|\bar{\textbf{p}}_{ij}|\leq h_i+h_j+11(L-j+i)+(L-j+i+1)\delta|V|_a-1$

\end{enumerate}

\end{lemma}

\begin{proof}

(1) Lemma \ref{7.22}(2) implies that $\textbf{p}_{ij}$ contains $h_i+h_j$ $\theta$-edges. Further, as $q$-bands cannot cross, every spoke starting on the complement $\bar{\textbf{u}}_{ij}$ of $\textbf{u}_{ij}$ in $\partial\Pi$ must end on $\textbf{p}_{ij}$, so that $\textbf{p}_{ij}$ contains at least $11(j-i)+1$ $q$-edges. The inequality thus follows.

(2) By parts (b) and (c) of Lemma \ref{lengths}, it suffices to show that $$|\textbf{u}_{ij}|\leq11(L-j+i)+(L-j+i+1)\delta|V|_a-1$$
As $\partial\Pi$ consists only of $q$-edges and $a$-edges, $|\textbf{u}_{ij}|=|\partial\Pi|-|\bar{\textbf{u}}_{ij}|$. 

By Lemma \ref{accepted configuration a-length}, $|\partial\Pi|=11L+\delta\sum_{i=1}^L|W(i)|_a\leq11L+\delta(L+1)|V|_a$. Further, $\lab(\bar{\textbf{u}}_{ij})$ consists of at least $j-i$ copies of $V^{\pm1}$ and one more $t$-letter, so that $|\bar{\textbf{u}}_{ij}|\geq11(j-i)+(j-i)\delta|V|_a+1$. Thus, the inequality follows.

\end{proof}

\begin{lemma} \label{7.25}

\textit{(Compare with Lemma 9.11 of [16] and Lemma 7.25 of [23])} 

If $1\leq i<j\leq L-4$ such that $j-i\geq L/2$, then $\mu(\Delta)-\mu(\Psi_{ij}')>-2J|\partial\Delta|(h_i+h_j)\geq-2J|\partial\Delta||\textbf{p}_{ij}|$.

\end{lemma}

\begin{proof}

As $j-i\geq L/2$, the path $\bar{\textbf{u}}_{ij}$ contains at least $11L/2+1$ $q$-edges. So, since every spoke of $\Pi$ starting on $\bar{\textbf{u}}_{ij}$ must end on $\textbf{p}_{ij}$, we have $|\textbf{p}_{ij}|_q\geq|\bar{\textbf{p}}_{ij}|_q$.

Let $\textbf{e}_1$ and $\textbf{e}_2$ be a pair of $\theta$-edges of $\partial\Delta$ such that neither is part of $\textbf{p}_{ij}$. Further, let $\textbf{x}$ be the subarc of $\partial\Delta$ connecting $\textbf{e}_1$ and $\textbf{e}_2$ and containing $\textbf{p}_{ij}$. Then, replacing the subpath $\textbf{p}_{ij}$ with $\bar{\textbf{p}}_{ij}$ produces a subarc $\bar{\textbf{x}}$ of $\partial\Psi_{ij}'$ connecting $\textbf{e}_1$ and $\textbf{e}_2$. Note that there are at least as many $q$-edges in $\textbf{x}$ as in $\bar{\textbf{x}}$. Hence, since the complement of $\textbf{x}$ in $\partial\Delta$ is a subpath of $\partial\Psi_{ij}'$, both ordered pairs of white edges corresponding to $\textbf{e}_1$ and $\textbf{e}_2$ contributes at least as much to $\mu(\Delta)$ as it does to $\mu(\Psi_{ij}')$.

So, we need only consider the contribution to $\mu(\Psi_{ij}')$ from pairs of $\theta$-edges where at least one is part of $\bar{\textbf{p}}_{ij}$. As $\bar{\textbf{p}}_{ij}$ consists of $h_i+h_j$ $\theta$-edges from the sides of $\pazocal{Q}_i$ and $\pazocal{Q}_j$, there are at most $|\partial\Delta|(h_i+h_j)$ such unordered pairs. By definition, each ordered such pair contributes at most $J$ to $\mu(\Psi_{ij}')$. Hence, $\mu(\Delta)-\mu(\Psi_{ij}')\geq-2J|\partial\Delta|(h_i+h_j)$. 

Since $j-i\geq L/2$, Lemma \ref{7.22}(2) implies that $i\leq r<r+1\leq j$ and so every $\theta$ band crossing $\pazocal{Q}_i$ or $\pazocal{Q}_j$ ends on $\textbf{p}_{ij}$. So, Lemma \ref{lengths}(a) yields $|\textbf{p}_{ij}|\geq h_i+h_j$, implying the statement.

\end{proof}

\begin{lemma} \label{7.26}

\textit{(Compare with Lemma 9.12 of [16] and Lemma 7.26 of [23])} 

If $1\leq i<j\leq L-4$ such that $j-i\geq L/2$, then $$|\textbf{p}_{ij}|+\sigma_\lambda(\bar{\Delta}_{ij}^*)\leq|\textbf{p}_{ij}|+\sigma_\lambda(\Delta^*)-\sigma_\lambda((\Psi_{ij}')^*)<(1+\eps)|\bar{\textbf{p}}_{ij}|$$ for $\eps=1/\sqrt{N_4}$.

\end{lemma}

\begin{proof}

Set $y=|\textbf{p}_{ij}|+\sigma_\lambda(\Delta^*)-\sigma_\lambda((\Psi_{ij}')^*)$ and $d=y-|\bar{\textbf{p}}_{ij}|$. Suppose $d\geq\eps|\bar{\textbf{p}}_{ij}|>0$.

Then $d\geq y-\eps^{-1}d$, so that $d\geq(1+\eps^{-1})^{-1}y\geq\frac{\eps y}{2}$ as $N_4\geq1$. 

As $\Psi_{ij}'$ and $\bar{\Delta}_{ij}$ are disjoint, the definition of the design on a minimal diagram implies 
$$\sigma_\lambda(\bar{\Delta}_{ij}^*)+\sigma_\lambda((\Psi_{ij}')^*)\leq\sigma_\lambda(\Delta^*)$$
Let $\textbf{s}$ be the complement of $\textbf{p}_{ij}$ in $\partial\Delta$. As $\textbf{p}_{ij}$ starts and ends with $q$-edges, $|\partial\Delta|=|\textbf{p}_{ij}|+|\textbf{s}|$. Further, by Lemma \ref{lengths}(c), $|\partial\Psi_{ij}'|\leq|\bar{\textbf{p}}_{ij}|+|\textbf{s}|$. 

So, these relations imply
\begin{align*}
(|\partial\Delta|+\sigma_\lambda(\Delta^*))-(|\partial\Psi_{ij}'|+\sigma_\lambda((\Psi_{ij}')^*))&\geq|\partial\Delta|-|\partial\Psi_{ij}'|+\sigma_\lambda(\Delta^*)-\sigma_\lambda((\Psi_{ij}')^*)\ \\
&\geq|\textbf{p}_{ij}|-|\bar{\textbf{p}}_{ij}|+\sigma_\lambda(\Delta^*)-\sigma_\lambda((\Psi_{ij}')^*) \\
&=d>0
\end{align*}
Hence, if $\Psi_{ij}'$ contains a disk, then we may apply the inductive hypothesis to it. Conversely, if $\Psi_{ij}'$ contains no disks, then we may apply Lemma \ref{diskless} to it. So, setting $x=|\partial\Delta|+\sigma_\lambda(\Delta^*)$, we have
$$\text{wt}_G(\Psi_{ij}')\leq N_4(x-d)^2+N_3\mu(\Psi_{ij}')$$
Noting that $d\leq x$, Lemma \ref{7.25} then implies
\begin{equation} \label{7.26 G-weight of Psi'}
\text{wt}_G(\Psi_{ij}')\leq N_4x^2-N_4xd+N_3\mu(\Delta)+2N_3J|\partial\Delta||\textbf{p}_{ij}|
\end{equation}
Note that $|\partial\Pi|\leq(L+1)|\bar{\textbf{p}}_{ij}|\leq(L+1)y$, so that
\begin{equation} \label{7.26 weight of Pi}
\text{wt}(\Pi)\leq C_1(L+1)^2y^2
\end{equation}
Further, as $j-i\geq L/2$, we must have $i\leq r<r+1\leq j$ by Lemma \ref{7.22}(2). So, Lemma \ref{7.24} implies 
$$|\bar{\textbf{p}}_{ij}|<|\textbf{p}_{ij}|+11L+(L-1)\delta|V|_a\leq|\textbf{p}_{ij}|+|\partial\Pi|$$ 
and hence $|\partial\Psi_{ij}|<2|\textbf{p}_{ij}|+|\partial\Pi|\leq(L+3)y$ by Lemma \ref{lengths}(c). Thus, by Lemma \ref{diskless},
\begin{equation} \label{7.26 G-weight of Psi}
\text{wt}_G(\Psi_{ij})\leq N_2(L+3)^2y^2+N_1\mu(\Psi_{ij})
\end{equation}

By (\ref{7.26 G-weight of Psi'}), (\ref{7.26 weight of Pi}), and (\ref{7.26 G-weight of Psi}), Lemma \ref{G-weight subdiagrams} implies
$$
\text{wt}_G(\Delta)\leq N_4x^2-N_4xd+N_3\mu(\Delta)+2N_3Jy|\partial\Delta|+N_2(L+3)^2y^2+N_1\mu(\Psi_{ij})+C_1(L+1)^2y^2 
$$
Hence, in order to reach a contradiction, it suffices to show that
\begin{equation} \label{7.26 suffices 1}
N_4xd\geq 2N_3Jy|\partial\Delta|+(N_2+C_1)(L+3)^2y^2+N_1\mu(\Psi_{ij})
\end{equation}
Note that $x=|\partial\Delta|+\sigma_\lambda(\Delta^*)\geq|\textbf{p}_{ij}|+\sigma_\lambda(\Delta^*)$, so that $x\geq\max(|\partial\Delta|,y)$. Hence, $$N_4xd\geq\frac{\eps}{2}N_4y\max(|\partial\Delta|,y)=\frac{1}{2}\sqrt{N_4}\max(y|\partial\Delta|,y^2)$$
The parameter choices $N_4>>N_3>>N_2>>C_1>>J>>L$ then allow us to assume
$$\frac{1}{2}N_4xd\geq2N_3Jy|\partial\Delta|+(N_2+C_1)(L+3)^2y^2$$
Hence, by (\ref{7.26 suffices 1}), it suffices to show that
\begin{equation} \label{7.26 suffices 2}
N_4xd\geq2N_1\mu(\Psi_{ij})
\end{equation}
As each $\theta$-edge of $\partial\Psi_{ij}$ must be in its own factor of any decomposition of $\lab(\Psi_{ij})$, the number of white beads on the necklace corresponding to $\partial\Psi_{ij}$ is at most $|\partial\Psi_{ij}|\leq(L+3)y$. So, Lemma \ref{mixtures}(a) implies
$$\mu(\Psi_{ij})\leq J(L+3)^2y^2$$
\newpage
But as above, $N_4xd\geq\frac{1}{2}\sqrt{N_4}y^2$, so that (\ref{7.26 suffices 2}) follows from the parameter choices $N_4>>J>>L$.

\end{proof}

For $i=1,\dots,L-5$, if the pair of adjacent $t$-letters associated to $\pazocal{Q}_i$ and to $\pazocal{Q}_{i+1}$ are $t(L)$ and $t(2)$ (or vice versa), then $\Psi_{i,i+1}$ is called the \textit{distinguished clove}. As $q$-bands cannot cross, the distinguished clove contains a cutting $q$-band $\pazocal{Q}_i'$ formed by the $q$-spoke of $\Pi$ corresponding to the base letter $t(1)$. Let $\Lambda_{i,i+1}'$ (respectively $\Lambda_{i,i+1}''$) be the subdiagram bounded by $\pazocal{Q}_i'$ and the $t$-spoke corresponding to $t(L)$ (respectively $t(2)$). Define $\textbf{p}_{i,i+1}'$ (respectively $\textbf{p}_{i,i+1}''$) as the subpath of $\partial\Lambda_{i,i+1}'$ (respectively $\partial\Lambda_{i,i+1}''$) shared with $\partial\Delta$, so that $\textbf{p}_{i,i+1}$ is the concatenation of these two paths along a shared $q$-edge.

Suppose $\Psi_{i,i+1}$ is not the distinguished clove. Then, let $\textbf{q}_{i,i+1}$ be the shortest path in $\Psi_{i,i+1}$ homotopic to $\textbf{p}_{i,i+1}$ and having the same first and last edges. 

If $\Psi_{i,i+1}$ is the distinguished clove, then define $\textbf{q}_{i,i+1}'$ and $\textbf{q}_{i,i+1}''$ as the analgous shortest paths in $\Lambda_{i,i+1}'$ and $\Lambda_{i,i+1}''$. Then, let $\textbf{q}_{i,i+1}$ be the concatenation of $\textbf{q}_{i,i+1}'$ and $\textbf{q}_{i,i+1}''$ along their shared $q$-edge.

For $1\leq i<j\leq L-4$, let $\textbf{q}_{ij}$ be the concatenation of the paths $\textbf{q}_{i,i+1},\dots,\textbf{q}_{j-1,j}$ along their shared $q$-edges.

Then, let $\Psi_{ij}^0$ be the diagram obtained from $\Psi_{ij}$ by replacing $\textbf{p}_{ij}$ in the contour with $\textbf{q}_{ij}$, i.e by removing any cells between $\textbf{q}_{ij}$ and $\textbf{p}_{ij}$. Similarly define $\Psi^0$, $(\Lambda_{i,i+1}')^0$,  and $(\Lambda_{i,i+1}'')^0$.

The following is the direct analogue of Lemma \ref{7.24}(1) and is proved in exactly the same way.

\begin{lemma} \label{7.27} \textit{(Compare with Lemma 9.13 of [16] and Lemma 7.27 of [23])} 
\newline
If $i\leq r$ and $j\geq r+1$, then $|\textbf{q}_{ij}|\geq h_i+h_j+11(j-i)+1$.

\end{lemma}

\begin{lemma} \label{7.28} \textit{(Compare with Lemma 9.14 of [16] and Lemma 7.28 of [23])}

\begin{enumerate}[label=({\arabic*})]

\item Every maximal $q$-band of $\Psi^0$ corresponds to a spoke of $\Pi$.

\item No two $\theta$-edges of $\textbf{q}_{i,i+1}$ are part of the same $\theta$-band of $\Psi_{i,i+1}$.

\end{enumerate}

\end{lemma}

\begin{proof}

(1) Assuming there exists a maximal $q$-band $\pazocal{Q}$ in $\Psi^0$ not corresponding to a spoke of $\Pi$, $\pazocal{Q}$ must end twice on $\textbf{q}$. In fact, as $q$-bands cannot cross, there exists $1\leq i\leq L-4$ such that $\pazocal{Q}$ ends twice on $\textbf{q}_{i,i+1}$.

Let $\textbf{x}$ be the subpath of $\textbf{q}_{i,i+1}$ starting and ending with the ends of $\pazocal{Q}$. By Lemmas \ref{M_a no annuli 1} and \ref{G_a theta-annuli}, any maximal $\theta$-band crossing $\pazocal{Q}$ must have one end on $\textbf{x}$. So, for $\ell$ the length of $\pazocal{Q}$, $|\textbf{x}|\geq\ell+2$.

By Lemma \ref{lengths}(b), $|\textbf{bot}(\pazocal{Q})|=|\textbf{top}(\pazocal{Q})|=\ell$. But then replacing $\textbf{x}$ in $\textbf{q}_{i,i+1}$ with a side of $\pazocal{Q}$ produces a homotopic path with shorter length, contradicting the definition of $\textbf{q}_{i,i+1}$.

(2) Assuming the statement is false, there exists a $\theta$-band $\pazocal{T}$ in $\Psi_{i,i+1}$ connecting $\theta$-edges $\textbf{e}$ and $\textbf{f}$ of $\textbf{q}_{i,i+1}$. Perhaps passing to a subband, we may assume that no other $\theta$-edge comprising $\pazocal{T}$ is part of $\textbf{q}_{i,i+1}$.

Let $\textbf{y}$ be the subpath of $\textbf{q}_{i,i+1}$ bounded by $\textbf{e}$ and $\textbf{f}$. As $\theta$-bands cannot cross, we may assume that $\textbf{e}$ and $\textbf{f}$ are the only $\theta$-edges of $\textbf{y}$. So, every cell between a side of $\pazocal{T}$, say $\textbf{top}(\pazocal{T})$, and $\textbf{y}$ is an $a$-cell.

Suppose one of the following holds:

\begin{enumerate}[label=({\roman*})]

\item $\Psi_{i,i+1}$ is not the distinguished clove (see Figure 10.5(i)),

\item $\Psi_{i,i+1}$ is the distinguished clove and $\textbf{y}$ is a subpath of $\textbf{q}_{i,i+1}'$ (see Figure 10.5(ii)), or

\item $\Psi_{i,i+1}$ is the distinguished clove and $\textbf{y}$ is a subpath of $\textbf{q}_{i,i+1}''$ (see Figure 10.5(iii)).

\end{enumerate}

\renewcommand\thesubfigure{\roman{subfigure}}
\begin{figure}[H]
\centering
\begin{subfigure}[b]{0.48\textwidth}
\centering
\includegraphics[scale=1]{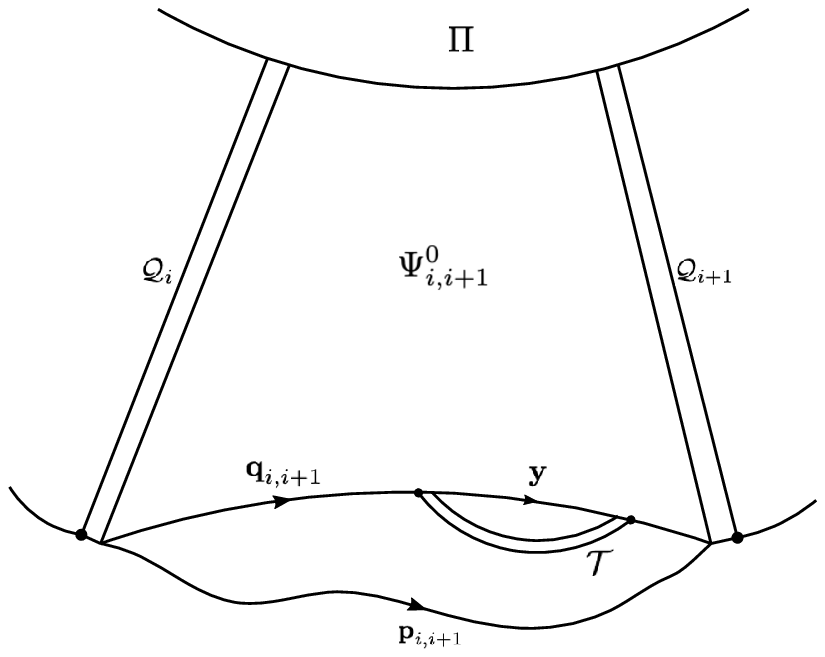}
\caption{}
\end{subfigure}
\\
\vspace{0.4in}
\begin{subfigure}[b]{0.48\textwidth}
\centering
\includegraphics[scale=1]{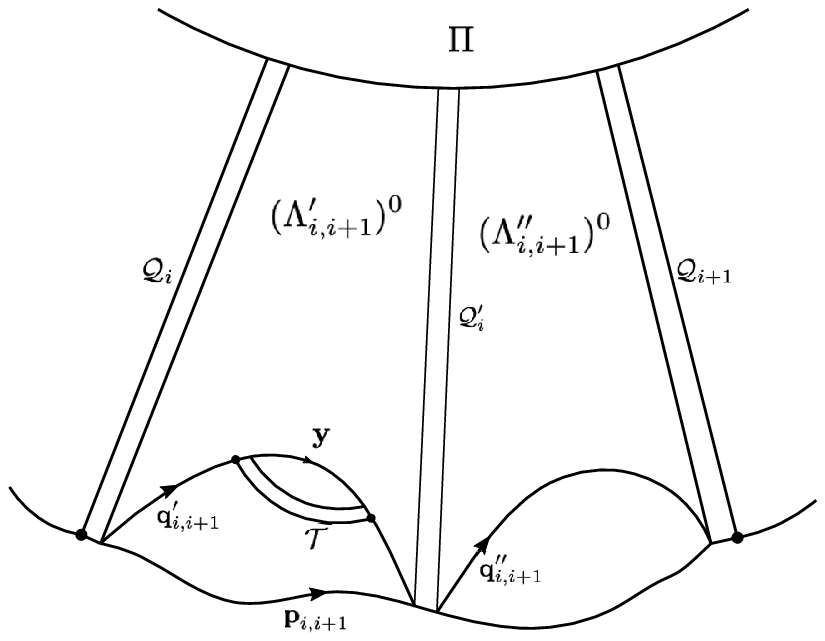}
\caption{}
\end{subfigure}\hfill
\begin{subfigure}[b]{0.48\textwidth}
\centering
\includegraphics[scale=1]{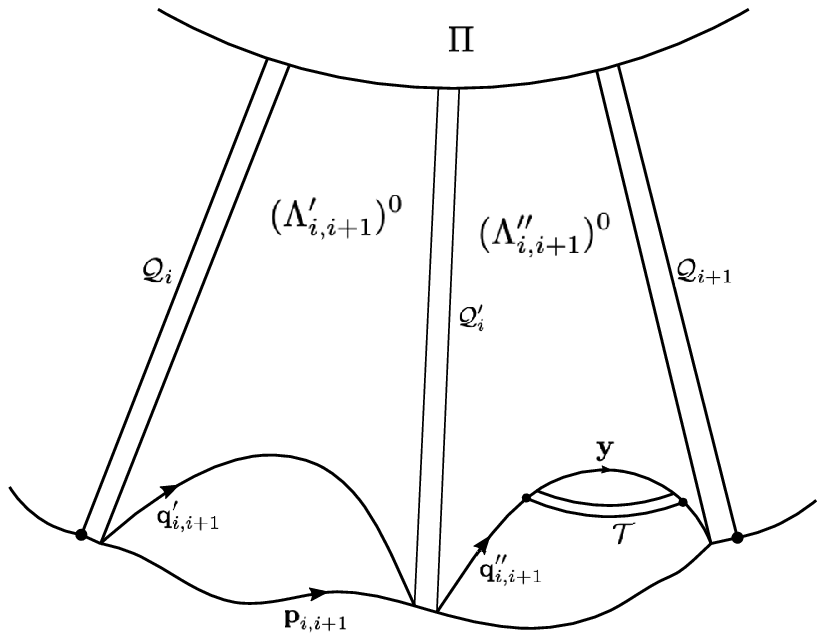}
\caption{}
\end{subfigure}
\caption{Lemma \ref{7.28}(2)}
\end{figure}

Note that every $q$-edge of $\textbf{top}(\pazocal{T})$ must be shared with $\textbf{y}$. So, by (1), every $q$-band crossing $\pazocal{T}$ must be a spoke of $\Pi$ in $\Psi_{i,i+1}$ (respectively $\Lambda_{i,i+1}'$, $\Lambda_{i,i+1}''$) in case (i) (respectively (ii), (iii)). As a result, the base of $\pazocal{T}$ is a subword of $B_3(j)^{\pm1}$ for some $j$.

If there exists an $a$-cell $\pi$ between $\textbf{top}(\pazocal{T})$ and $\textbf{y}$, then let $b_\pi$ be the number of edges of $\partial\pi$ which are on the boundary of a $(\theta,q)$-cell of $\pazocal{T}$. By (WM1), at most $\frac{1}{2}\|\partial\pi\|+b_\pi$ of the edges of $\partial\pi$ are shared with $\textbf{top}(\pazocal{T})$ while all other edges are shared with $\textbf{y}$.

By the definition of the rules of $\textbf{M}$, at most one edge on the boundary of a $(\theta,q)$-cell of $\pazocal{T}$ is labelled with a letter from the alphabet of the `special' input sector (on a cell corresponding to the base letter $Q_0(1)^{\pm1}$). As a result, $\sum b_\pi\leq1$, so that Lemma \ref{lengths}(c) implies $|\textbf{y}|\geq2+|\textbf{top}(\pazocal{T})|-4\delta$.

As the base of $\pazocal{T}$ has length at most $11$, Lemma \ref{simplify rules} implies $|\textbf{bot}(\pazocal{T})|-|\textbf{top}(\pazocal{T})|\leq22\delta$. So, $|\textbf{y}|-|\textbf{bot}(\pazocal{T})|\geq2-26\delta\geq1$ by a parameter choice for $\delta^{-1}$.

But then replacing $\textbf{y}$ in $\textbf{q}_{i,i+1}$ with $\textbf{bot}(\pazocal{T})$ contradicts the definition of $\textbf{q}_{i,i+1}$.

Hence, it suffices to assume that (i), (ii), and (iii) all do not hold. 

So, $\Psi_{i,i+1}$ is the distinguished clove, $\textbf{e}$ is an edge of $\textbf{q}_{i,i+1}'$, and $\textbf{f}$ is an edge of $\textbf{q}_{i,i+1}''$. Further, since $\textbf{q}_{i,i+1}$ contains the $q$-edge of $\pazocal{Q}'$ shared with $\partial\Delta$, $\pazocal{T}$ must be contained in $\Psi_{i,i+1}^0$.

By Lemma \ref{7.22}(1), the maximal $\theta$-band containing $\pazocal{T}$ must cross $\pazocal{Q}_i$ or $\pazocal{Q}_{i+1}$, so that it must contain another $\theta$-edge of $\textbf{q}_{i,i+1}'$ or $\textbf{q}_{i,i+1}''$. 

But then there exists a $\theta$-band satisfying (ii) or (iii), so that a similar contradiction can be reached.

\end{proof}

\begin{lemma} \label{non-distinguished a-cells} \

\begin{enumerate}[label=({\arabic*})]

\item If $\Psi_{i,i+1}$ is not the distinguished clove, then $\Psi_{i,i+1}^0$ contains no $a$-cells. 

\item If $\Psi_{i,i+1}$ is the distinguished clove, then $(\Lambda_{i,i+1}')^0$ contains no $a$-cells.

\end{enumerate}

\end{lemma}

\begin{proof}

(1) Suppose $\pi$ is an $a$-cell contained in $\Psi_{i,i+1}^0$.

By Lemma \ref{7.28}(1), no maximal $q$-band of $\Psi_{i,i+1}^0$ corresponds to a base letter with coordinate 1. So, the contour of any $(\theta,a)$- or $(\theta,q)$-cell has no $a$-edge labelled by a letter from the alphabet of the `special' input sector.

As a result, any edge of $\partial\pi$ must be shared with $\partial\Psi_{i,i+1}^0$. In particular, $\partial\pi$ must be a subpath of $\textbf{q}_{i,i+1}$. 

But then removing this subpath produces a path homotopic to $\textbf{q}_{i,i+1}$ that contradicts its definition.

(2) is proved analogously, as the only base letter with coordinate 1 present in $(\Lambda_{i,i+1}')^0$ is $\{t(1)\}$.

\end{proof}

\subsection{Trapezia and combs in the cloves} \

For $1\leq i\leq r-1$, suppose $\Psi_{i,i+1}$ is not the distinguished clove. Then Lemma \ref{7.22} implies that all maximal $\theta$-bands of $\Psi_{i,i+1}$ crossing $\pazocal{Q}_{i+1}$ must also cross $\pazocal{Q}_i$. So, these $\theta$-bands bound an $a$-trapezium $\Gamma_i$ in $\Psi_{i,i+1}^0$ with height $h_{i+1}$. The base of $\Gamma_i$ (or its inverse) is $\{t(\ell)\}B_3(\ell)\{t(\ell+1)\}$ for some $2\leq\ell\leq L-1$. Lemma \ref{a-cells sector} then implies that $\Gamma_i$ is a trapezium. Set $\textbf{y}_i=\textbf{bot}(\Gamma_i)$ and $\textbf{z}_i=\textbf{top}(\Gamma_i)$. Note that $\textbf{y}_i^{-1}$ is shared with $\partial\Pi$. 

For $2\leq i\leq r-1$, suppose neither $\Psi_{i-1,i}$ nor $\Psi_{i,i+1}$ is the distinguished clove. Then $\lab(\textbf{y}_{i-1})$ and $\lab(\textbf{y}_i)$ are coordinate shifts of one another while $H_{i+1}$ is a prefix of $H_i$. So, $h_{i+1}$ $\theta$-bands of $\Gamma_{i-1}$ form a copy of $\Gamma_i$, $\Gamma_i'$, contained in $\Gamma_{i-1}$. Set $\textbf{y}_i'=\textbf{bot}(\Gamma_i')$ and $\textbf{z}_i'=\textbf{top}(\Gamma_i')$. Note that $\textbf{y}_i'=\textbf{y}_{i-1}$.

For $1\leq i\leq r-1$, if $\Psi_{i,i+1}$ is not the distinguished clove, then denote by $E_i$ (respectively $E_i^0$) the maximal comb in $\Psi_{i,i+1}$ (respectively $\Psi_{i,i+1}^0$) containing the maximal $\theta$-bands that cross the $t$-spoke $\pazocal{Q}_i$ but not the $t$-spoke $\pazocal{Q}_{i+1}$. The handle $\pazocal{C}_i$ of these combs has height $h_i-h_{i+1}$ and is contained in $\pazocal{Q}_i$. Any cell of $\Psi_{i,i+1}$ (respectively $\Psi_{i,i+1}^0$) not contained in $\Gamma_i$ or $E_i$ (respectively $E_i^0$) must be an $a$-cell attached to either $\textbf{z}_i$ or $\pazocal{Q}_{i+1}$. By the structure of the relations, such an $a$-cell must share every boundary edge with $\partial\Delta$. But this contradicts Lemma \ref{a-cell in counterexample 2}. Hence, $E_i$ (respectively $E_i^0$) is the complement of $\Gamma_i$ in $\Psi_{i,i+1}$ (respectively $\Psi_{i,i+1}^0$). 

Now suppose $\Psi_{i,i+1}$ is the distinguished clove for $1\leq i\leq r-1$. 

First, suppose $\pazocal{Q}_i$ corresponds to the base letter $\{t(L)\}$, so that the subdiagram $\Lambda_{i,i+1}'$ is bounded by $\pazocal{Q}_i$ and $\pazocal{Q}_i'$ (see Figure 10.7(a)). By Lemma \ref{7.22}, every maximal $\theta$-band of $\Lambda_{i,i+1}'$ crossing $\pazocal{Q}_i'$ must also cross $\pazocal{Q}_i$. So, these $\theta$-bands bound an $a$-trapezia $\Gamma_i$ contained in $(\Lambda_{i,i+1}')^0$. As above, Lemma \ref{a-cells sector} implies that $\Gamma_i$ must be a trapezium. The base of $\Gamma_i$ (or its inverse) is $\{t(L)\}B_3(L)\{t(1)\}$, while the height is the length $h_i'$ of the band $\pazocal{Q}_i'$.

\begin{figure}[H]
\centering
\includegraphics[scale=1.25]{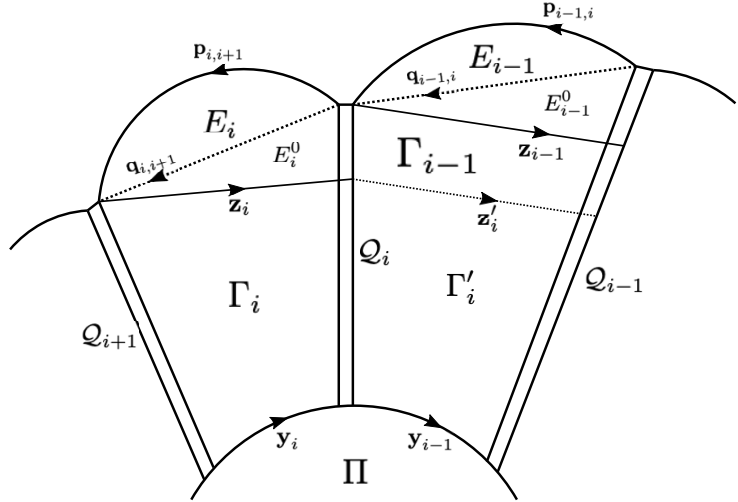}
\caption{Trapezia and combs if neither $\Psi_{i-1,i}$ and $\Psi_{i,i+1}$ are distinguished}
\end{figure}

Otherwise, $\pazocal{Q}_{i+1}$ corresponds to the base letter $\{t(L)\}$, so that the subdiagram $\Lambda_{i,i+1}'$ is bounded by $\pazocal{Q}_i'$ and $\pazocal{Q}_{i+1}$ (see Figure 10.7(2)). Lemma \ref{7.22} then implies that every maximal $\theta$-band of $\Lambda_{i,i+1}'$ crossing $\pazocal{Q}_{i+1}$ must also cross $\pazocal{Q}_i'$, so that these $\theta$-bands bound an $a$-trapezium $\Gamma_i$ contained in $(\Lambda_{i,i+1}')^0$. Again, $\Gamma_i$ must be a trapezium whose base (or its inverse) is $\{t(L)\}B_3(L)\{t(1)\}$. In this case, the height of $\Gamma_i$ is $h_{i+1}$.

In either case, we define $\textbf{y}_i=\textbf{bot}(\Gamma_i)$ and $\textbf{z}_i=\textbf{top}(\Gamma_i)$. If $i\geq2$, then again $\lab(\textbf{y}_i)$ is a coordinate shift of $\lab(\textbf{y}_{i-1})$ and there exists a copy $\Gamma_i'$ of $\Gamma_i$ in $\Gamma_{i-1}$ with $\textbf{bot}(\Gamma_i')=\textbf{y}_i'=\textbf{y}_{i-1}$. Similarly, if $i\leq r-2$, then $\lab(\textbf{y}_i)$ is a coordinate shift of $\lab(\textbf{y}_{i+1})$ and there exists a copy $\Gamma_{i+1}'$ of $\Gamma_{i+1}$ in $\Gamma_i$.

Suppose $\Lambda_{i,i+1}'$ is bounded by $\pazocal{Q}_i$ and $\pazocal{Q}_i'$. Then denote by $E_i$ (respectively $E_i^0$) the maximal comb in $\Lambda_{i,i+1}'$ (respectively $(\Lambda_{i,i+1}')^0$) containing the maximal $\theta$-bands that cross the $t$-spoke $\pazocal{Q}_i$ but not the $q$-spoke $\pazocal{Q}_i'$. The handle $\pazocal{C}_i$ of these combs has height $h_i-h_i'$ and is contained in $\pazocal{Q}_i$. As above, $E_i$ (respectively $E_i^0$) is the complement of $\Gamma_i$ in $\Lambda_{i,i+1}'$ (respectively $(\Lambda_{i,i+1}')^0$).

Otherwise, $\Lambda_{i,i+1}'$ is bounded by $\pazocal{Q}_i'$ and $\pazocal{Q}_{i+1}$. In this case denote by $E_i$ (respectively $E_i^0$) the maximal comb in $\Lambda_{i,i+1}'$ (respectively $(\Lambda_{i,i+1}')^0$) containing the maximal $\theta$-bands that cross the $q$-spoke $\pazocal{Q}_i'$ but not the $t$-spoke $\pazocal{Q}_{i+1}$. The handle $\pazocal{C}_i$ of these combs has height $h_i'-h_{i+1}$ and is contained in $\pazocal{Q}_i'$. Again, $E_i$ (respectively $E_i^0$) is the complement of $\Gamma_i$ in $\Lambda_{i,i+1}'$ (respectively $(\Lambda_{i,i+1}')^0$).

Note that no $a$-trapezium or comb has been defined in the subdiagram $\Lambda_{i,i+1}''$. Though such subdiagrams exist, their consideration is not necessary for the rest of the proof. As a result, one may view the indexing as `skipping over' the portion of the clove between the base letters $\{t(1)\}$ and $\{t(2)\}$.

For $r+1\leq i\leq L-5$, the trapezium $\Gamma_i$, the combs $E_i$ and $E_i^0$, and the paths $\textbf{y}_i$ and $\textbf{z}_i$ are defined symmetrically.

\renewcommand\thesubfigure{\alph{subfigure}}
\begin{figure}[H]
\centering
\begin{subfigure}[b]{0.48\textwidth}
\centering
\includegraphics[scale=1]{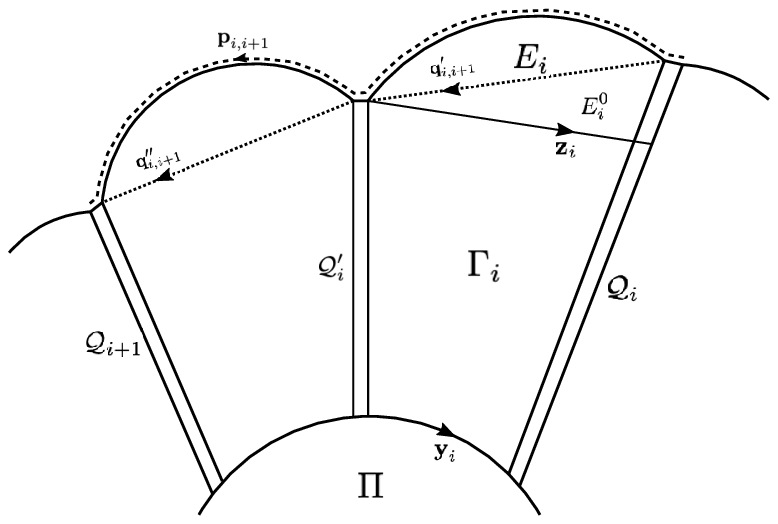}
\caption{$\pazocal{Q}_i$ and $\pazocal{Q}_i'$ bound $\Lambda_{i,i+1}'$}
\end{subfigure}\hfill
\begin{subfigure}[b]{0.48\textwidth}
\centering
\includegraphics[scale=1]{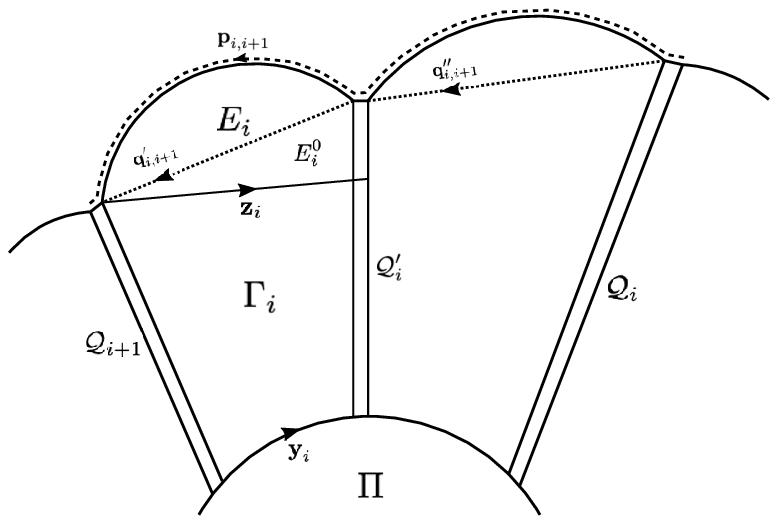}
\caption{$\pazocal{Q}_i$ and $\pazocal{Q}_i'$ bound $\Lambda_{i,i+1}''$}
\end{subfigure}
\caption{Trapezia in the distinguished clove}
\end{figure}

\begin{lemma} \label{7.29}

\textit{(Compare with Lemma 9.15 of [16] and Lemma 7.29 of [23])}
\newline
For $i\in\{2,\dots,r-1\}$, suppose a maximal $a$-band $\pazocal{B}$ of $E_i^0$ starts on $\textbf{z}_i$ and ends on a side of a maximal $q$-band $\pazocal{C}$. Let $\nabla$ be the comb bounded by $\pazocal{B}$, a part of $\pazocal{C}$, and a subpath $\textbf{x}$ of $\textbf{z}_i$. Then there is a copy of the comb $\nabla$ in the trapezium $\Gamma=\Gamma_{i-1}\setminus\Gamma_i'$.

\end{lemma}

\begin{proof}

By Lemma \ref{non-distinguished a-cells}, $\nabla$ contains no $a$-cells.

Let the $a$-edge $\textbf{e}$ and the $q$-edge $\textbf{f}$ be the first and last edge of $\textbf{x}$, respectively. Since $\textbf{z}_i'$ is a copy of $\textbf{z}_i$ in the trapezium $\Gamma_{i-1}$, it contains a subpath $\textbf{x}'$ that is a copy of $\textbf{x}$ and starts with an $a$-edge $\textbf{e}'$ and ends with a $q$-edge $\textbf{f}'$. If $\pi$ is the $(\theta,q)$-cell attached to $\textbf{f}$ in $\nabla$, then the $(\theta,q)$-cell $\pi'$ attached to $\textbf{f}'$ is a copy since it corresponds to the same letter of the history. Moving from $\textbf{f}$ to $\textbf{e}$, the whole maximal $\theta$-band of $\nabla$ containing $\pi$ has a copy in $\Gamma_{i-1}$. Moving up, we find a copy of every maximal $\theta$-band of $\nabla$ in $\Gamma_{i-1}$, forming a copy of $\nabla$ in $\Gamma_{i-1}$.

\end{proof}

\begin{lemma} \label{7.30}

\textit{(Compare with Lemma 9.16 of [16] and Lemma 7.30 of [23])} 
\newline
At most $6$ $a$-bands starting on the path $\textbf{y}_i$ (or $\textbf{z}_i$) can end on $(\theta,q)$-cells of the same $\theta$-band.

\end{lemma}

\begin{proof}

Assume each of the $a$-bands $\textbf{A}_1,\dots,\textbf{A}_m$ starts from an edge of $\textbf{y}_i$ and ends on some $(\theta,q)$-cell of a $\theta$-band $\pazocal{T}$.  Let $\pazocal{T}_0$ be the minimal subband of $\pazocal{T}$ such that the $a$-bands $\textbf{A}_2,\dots,\textbf{A}_{m-1}$ end on $\pazocal{T}_0$. Then, let $\bar{\textbf{y}}_i$ the minimal subpath of $\textbf{y}_i$ where the $a$-bands $\textbf{A}_1,\dots,\textbf{A}_m$ start (see Figure 10.8).

By Lemma \ref{M_a no annuli 1}, each $q$-band starting on $\bar{\textbf{y}}_i$ has to cross $\pazocal{T}_0$ and vice versa. So, the base of $\pazocal{T}_0$ is a subword of a reduced pararevolving base not containing the `special' input sector. 

As a result, we can identify this base with a subword of the standard base of $\textbf{M}_4$ (or its inverse). By the structure of the rules of $\textbf{M}_4$, an application of any rule inserts/deletes at most $4$ $a$-letters in a configuration. Thus, $m-2\leq4$, so that the statement follows.

An analogous argument applies for $a$-bands starting from $\textbf{z}_i$.

\end{proof}

\begin{figure}[H]
\centering
\includegraphics[scale=2]{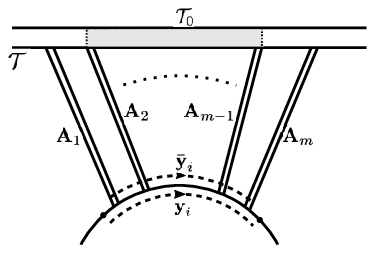}
\caption{}
\end{figure}

By the parameter choice $L>>L_0$ and Lemma \ref{7.22}, we may assume that $L_0+1\leq r$ and $L-L_0-4\geq r+1$. Then, suppose without loss of generality that $h\defeq h_{L_0+1}\geq h_{L-L_0-4}$.

\begin{lemma} \label{7.31}

\textit{(Compare with Lemma 9.17 of [16] and Lemma 7.31 of [23])} 
\newline
Let $I$ be the subset of the set of indices $i\in[L_0+1,r-1]\cup[r+1,L-L_0-5]$ such that $|\textbf{z}_i|_a\geq|V|_a/8c_3$. If $h\leq L_0^2|V|_a$, then $\# I\leq L/5$.

\end{lemma}

\begin{proof}

%By Lemma \ref{7.30}, there are at most $4$ maximal $a$-bands starting on $\textbf{z}_i$ and ending on each of the $\theta$-bands of $E_i^0$. Let $S$ be the subset of integers in $[L_0+1,r-1]\cup[r+1,L-L_0-4]$ not including the index $j$ such that $\Psi_{j,j+1}$ is the distinguished clove (if such an index exists). 

For any $i\in [L_0+1,r-1]\cup[r+1,L-L_0-5]$, denote the set of maximal $a$-bands of $E_i^0$ starting at $\textbf{z}_i$ by $\textbf{A}_i$. Then set $\textbf{A}=\cup\textbf{A}_i$. 

As no base letter of $\Gamma_i$ is of the form $Q_0(1)^{\pm1}$ or $P_0(1)^{\pm1}$, $\textbf{z}_i$ has no $a$-letters from the `special' input sector. So, any $a$-band of $\textbf{A}$ either ends on a $(\theta,q)$-cell or on $\textbf{q}_{i,i+1}$.

Letting $g_i$ be the length of the handle of $E_i^0$, then $\sum g_i\leq2h$, where the sum is taken over the integers in $[L_0+1,r-1]\cup[r+1,L-L_0-5]$. So, by Lemma \ref{7.30}, at most $12h$ maximal $a$-bands of $\textbf{A}$ end on $(\theta,q)$-cells. 

Assuming the statement is false, $\textbf{A}$ contains at least $L|V|_a/40c_3$ $a$-bands. As a result, at least $\max(0,L|V|_a/40c_3-12h)$ bands from $\textbf{A}$ must end on the subpaths $\textbf{q}_{i,i+1}$. Since $\textbf{q}_{i,i+1}$ has at most $2h$ $\theta$-edges by Lemma \ref{7.28}(2), at least $\max(0,L|V|_a/40c_3-14h)$ $a$-edges contribute $\delta$ to $|\textbf{q}_{i,i+1}|$.

By assumption, $14h\leq14L_0^2|V|_a$, so that the parameter choices $L>>L_0>>c_3$ imply that $14h\leq L|V|_a/80c_3$.

It follows from Lemma \ref{7.27} that 
\begin{align*}
|\textbf{p}_{L_0+1,L-L_0-4}|\geq|\textbf{q}_{L_0+1,L-L_0-4}|&\geq h_{L_0+1}+h_{L-L_0-4}+11L/2+\delta(L|V|_a/40c_3-14h) \\
&\geq h_{L_0+1}+h_{L-L_0-4}+11L/2+\delta L|V|_a/80c_3
\end{align*}
Also, by Lemma \ref{7.24}, we have
\begin{align*}
|\bar{\textbf{p}}_{L_0+1,L-L_0-4}|&\leq h_{L_0+1}+h_{L-L_0-4}+11(3L_0)+3L_0\delta|V|_a \\
&\leq h_{L_0+1}+h_{L-L_0-4}+11(3L_0)+\delta L|V|_a/160c_3
\end{align*}
as $L>>L_0>>c_3$. These inequalities imply
\begin{equation} \label{7.31 later}
|\textbf{p}_{L_0+1,L-L_0-4}|-|\bar{\textbf{p}}_{L_0+1,L-L_0-4}|\geq 11L/3+\delta L|V|_a/160c_3
\end{equation}
Since $h_{L_0+1}+h_{L-L_0-4}\leq2h\leq2L_0^2|V|_a<\frac{1}{2}L|V|_a$, it follows that
$$|\bar{\textbf{p}}_{L_0+1,L-L_0-4}|<\frac{1}{2}L|V|_a+11(3L_0)+\delta L|V|_a/160c_3\leq11(3L_0)+L|V|_a$$
which implies that
$$\frac{|\textbf{p}_{L_0+1,L-L_0-4}|-|\bar{\textbf{p}}_{L_0+1,L-L_0-4}|}{|\bar{\textbf{p}}_{L_0+1,L-L_0-4}|}\geq\min\bigg(\frac{11L/3}{11(3L_0)},\frac{\delta L|V|_a/160c_3}{L|V|_a}\bigg)=\delta/160c_3$$
since we have $L>>L_0$. Finally, $\delta/160c_3>\eps=1/\sqrt{N_4}$ for sufficiently large $N_4$, and so
$$\frac{|\textbf{p}_{L_0+1,L-L_0-4}|}{|\bar{\textbf{p}}_{L_0+1,L-L_0-4}|}>1+\eps$$
But $L-L_0-4-(L_0+1)=L-2L_0-5\geq L-3L_0>L/2$ since $L>>L_0$, so that the above inequality contradicts Lemma \ref{7.26}.

\end{proof}

\begin{lemma} \label{7.32}

\textit{(Compare with Lemma 9.18 of [16] and Lemma 7.32 of [23])} 
\newline
If $h\leq L_0^2|V|_a$, then the histories $H_1$ and $H_{L-4}$ have different first letters.

\end{lemma}

\begin{proof}

Let $\pazocal{T}$ and $\pazocal{T}'$ be the maximal $\theta$-bands of $\Psi$ crossing $\pazocal{Q}_1$ and $\pazocal{Q}_{L-4}$, respectively, closest to the disk $\Pi$.

Set $\ell,\ell'$ as the maximal integers such that $\pazocal{T}$ crosses the $t$-spokes $\pazocal{Q}_1,\dots,\pazocal{Q}_\ell$ and $\pazocal{T}'$ crosses the $t$-spokes $\pazocal{Q}_{L-\ell'-3},\dots,\pazocal{Q}_{L-4}$. Note that $\ell\leq r\leq L-\ell'-4$.

For any $\ell+1\leq i\leq L-\ell'-5$, $\textbf{z}_i$ is a subpath of $\partial\Pi$. As a result, $|\textbf{z}_i|=|V|_a\geq|V|_a/8c_3$. Hence, if also $i\in[L_0+1,r-1]\cup[r+1,L-L_0-5]$, then $i\in I$ (for $I$ as defined in the statement of Lemma \ref{7.31}).

If $\ell\leq L_0$, then $\# I\geq(r-1)-(L_0+1)\geq (L-1)/2-L_0-5\geq L/3$. Similarly, if $\ell'\leq L_0$, then $\# I\geq L/3$. But these inequalities contradict Lemma \ref{7.31}, so that $\ell,\ell'>L_0$.

This implies $\# I\geq (L-\ell'-5)-(\ell+1)-2\geq L-(\ell+\ell')-8$, so that Lemma \ref{7.31} yields $\ell+\ell'\geq L-L/5-8\geq 3L/4$.

Thus, if the rules corresponding to $\pazocal{T}$ and $\pazocal{T}'$ are same, then the minimality of $\Delta^*$ contradicts Lemma \ref{G_a theta-annuli}(2).

\end{proof}

\begin{lemma} \label{7.33} \textit{(Compare with Lemma 7.33 of [23])} If $h\leq L_0^2|V|_a$, then $\displaystyle|V|_a>\frac{11L}{4\delta L_0}$.

\end{lemma}

\begin{proof}

Assume that $|V|_a\leq 11L/4\delta L_0$. Then Lemma \ref{7.24} implies the inequalities $$|\textbf{p}_{L_0+1,L-L_0-4}|\geq h_{L_0+1}+h_{L-L_0-4}+11(L-3L_0)$$ $$|\bar{\textbf{p}}_{L_0+1,L-L_0-4}|\leq h_{L_0+1}+h_{L-L_0-4}+3L_0(11+\delta|V|_a)$$
Hence, as $L>>L_0$, $$|\textbf{p}_{L_0+1,L-L_0-4}|-|\bar{\textbf{p}}_{L_0+1,L-L_0-4}|\geq11(L-6L_0)-3L_0\delta|V|_a>11(L-6L_0)-33L/4>11L/5$$ 
The inequality $h_{L_0+1}+h_{L-L_0-4}\leq 2h$ then implies $$|\bar{\textbf{p}}_{L_0+1,L-L_0-4}|\leq2h+3L_0(11+11L/4L_0)\leq2L_0^2\frac{11L}{4\delta L_0}+11L<11L_0L/\delta$$
So, since $N_4>>\delta^{-1}>>L_0$, we have $$\frac{|\textbf{p}_{L_0+1,L-L_0-4}|-|\bar{\textbf{p}}_{L_0+1,L-L_0-4}|}{|\bar{\textbf{p}}_{L_0+1,L-L_0-4}|}>\frac{\delta}{5L_0}>\eps$$
But $L-L_0-4-(L_0+1)\geq L/2$, so that the above inequality contradicts Lemma \ref{7.26}.

\end{proof}

\begin{lemma} \label{7.34}

\textit{(Compare with Lemma 9.19 of [16] and Lemma 7.34 of [23])} 
\newline
The inequality $h>L_0^2|V|_a$ must be true.

\end{lemma}

\begin{proof}

Assuming the statement is false, Lemma \ref{7.31} implies that for at least $L-5-L/5-2L_0>3L/4$ indices $j\in\{1,\dots,L-5\}$,  $|\textbf{z}_j|_a<|V|_a/8c_3$. So, we can choose two such indices, $i$ and $j$, such that $L_0+1\leq i\leq r<r+1\leq j\leq L-L_0-5$, $j-i\geq 3L/5$, and neither $\Psi_{i,i+1}$ nor $\Psi_{j,j+1}$ is the distinguished clove. 

Since $H_{i+1}$ (respectively $H_j$) is a prefix of $H_1$ (respectively $H_{L-4}$), it follows from Lemma \ref{7.32} that the first letters of $H_{i+1}$ and $H_j$ are different.

Since $\lab(\textbf{y}_i)$ and $\lab(\textbf{y}_j)$ are coordinate shifts of one another (and are copies of $V$), we can construct an auxiliary trapezium $E$ by pasting the mirror of a coordinate shift of $\Gamma_j$ to $\Gamma_i$ along $\textbf{y}_i$. The history of $E$ is $H_j^{-1}H_{i+1}$, which is a reduced word since the first letter of $H_{i+1}$ is different from the first letter of $H_j$.

The top and the bottom of $E$ are copies of $\textbf{z}_i$ and $\textbf{z}_j$, respectively, and so have $a$-lengths less than $|V|_a/8c_3$. Without loss of generality, assume $h_{i+1}\geq h_j$, and so $h_{i+1}\geq t/2$ for $t$ the height of $E$.

Note that $|V|_a-|V|_a/8c_3>|V|_a/2$, and so $h_{i+1},h_j>|V|_a/8$ since any rule of $\textbf{M}_4$ alters the $a$-length of a configuration by at most four.

By Lemma \ref{7.33}, $|V|_a/8>\frac{11L}{32\delta L_0}\geq 12c_3$ since $\delta^{-1}>>L>>L_0>>c_3$. Further, letting $W_0$ and $W_t$ be the bottom and top labels of $E$, $|V|_a/8> c_3\max(|W_0|_a,|W_t|_a)$.

As a result,
$$t=h_{i+1}+h_j>|V|_a/4>c_3\max(|W_0|_a,|W_t|_a)+12c_3\geq c_3\max(\|W_0\|,\|W_t\|)$$
Let $\pazocal{C}$ be the computation associated to $E$ through Lemma \ref{trapezia are computations}. Then the restriction of $\pazocal{C}$ (or its inverse) to $\{t(\ell)\}B_3(\ell)$ for the appropriate $\ell\geq2$ satisfies the hypotheses of Lemma \ref{M projected long history controlled}. 

Setting $\lambda<1/10$, every factorization $H'H''H'''$ of $H_{i+1}$ with $\|H'\|+\|H'''\|\leq\lambda h_{i+1}$ satisfies $\|H''\|>0.4t$. So, applying Lemma \ref{M projected long history controlled}, $H''$ contains a controlled subword. Further, since all $\theta$-bands crossing $\pazocal{Q}_{i+1}$ must cross $\pazocal{Q}_i$, $W(\ell)$ is $H_{i+1}$-admissible. Hence, $\pazocal{Q}_{i+1}$ is a $\lambda$-shaft. 

Lemma \ref{7.24}(1) then implies that $|\textbf{p}_{i+1,j}|+\sigma_\lambda(\bar{\Delta}_{i+1,j}^*)\geq 2h_{i+1}+h_j$.

As $h_{i+1}>|V|_a/8$, it follows that $\delta(L+1)|V|_a\leq8\delta(L+1)h_{i+1}<\frac{1}{4}h_{i+1}$ by the parameter choice $\delta^{-1}>>L$. Similarly, by Lemma \ref{7.33} and $\delta^{-1}>>L_0$, $11L<4\delta L_0|V|_a<32\delta L_0h_{i+1}\leq\frac{1}{4}h_{i+1}$.

So, Lemma \ref{7.24}(2) yields $|\bar{\textbf{p}}_{i+1,j}|\leq \frac{3}{2}h_{i+1}+h_j$.

Hence, $$\frac{|\textbf{p}_{i+1,j}|+\sigma_\lambda(\bar{\Delta}_{i+1,j}^*)}{|\bar{\textbf{p}}_{i+1,j}|}\geq\frac{2h_{i+1}+h_j}{\frac{3}{2}h_{i+1}+h_j}\geq\frac{6}{5}$$ since $h_j\leq h_{i+1}$.

Taking $N_4$ sufficiently large, $\eps=1/\sqrt{N_4}<0.2$. However, as $j-(i+1)\geq3L/5-1\geq L/2$, the above inequality contradicts Lemma \ref{7.26}.

\end{proof}

%Now, define the positive integers $\omega$ and $\tau$ satisfying $1\leq\omega<\tau\leq L_0$ satisfying the property that $\omega,\dots,\tau$ is the maximal string of consecutive such integers such that for all $\omega\leq i\leq\tau$, $\Psi_{i,i+1}$ is not the distinguished clove. Note the following three properties:
%
%\begin{itemize}
%
%\item $\tau-\omega\geq L_0/2$,
%
%\item $1\leq\omega\leq L_0/2\leq\tau\leq L_0$, and
%
%\item $h_\tau\geq h$
%
%\end{itemize}
%
%As $\Psi_{i,i+1}$ is not distinguished for each $\omega\leq i\leq\tau$, the following statements are direct analogues of statements of [23] and can be proved in exactly the same ways.

\begin{lemma} \label{7.35}

\textit{(Compare with Lemma 9.20 of [16] and Lemma 7.35 of [23])} 
\newline
For $i=1,\dots,L_0$, we have $h_i>\delta^{-1}$.

\end{lemma}

\begin{proof}

For such $i$, note that $h_i\geq h\geq h_{L-L_0-4}$. Assuming toward contradiction that $h_i\leq\delta^{-1}$, Lemma \ref{7.34} implies that $\delta^{-1}>L_0^2|V|_a$, and so $\delta|V|_a<1/L_0^2$.

Note that $L-L_0-4-i\geq L-L_0-4-L_0\geq L-(2L_0+4)$. Taking $L>>L_0$, Lemma \ref{7.24} then yields the inequalities
$$|\bar{\textbf{p}}_{i,L-L_0-4}|\leq h_i+h_{L-L_0-4}+3L_0(11+\delta|V_a|)\leq h_i+h_{L-L_0-4}+11(4L_0)$$
$$|\textbf{p}_{i,L-L_0-4}|\geq h_i+h_{L-L_0-4}+11(L-2L_0-4)\geq h_i+h_{L-L_0-4}+11L/2$$
But then $h_i+h_{L-L_0-4}\leq 2h_i\leq 2\delta^{-1}$ and $4L_0<L/4$, so that
$$\frac{|\textbf{p}_{i,L-L_0-4}|}{|\bar{\textbf{p}}_{i,L-L_0-4}|}\geq\frac{h_i+h_{L-L_0-4}+11L/2}{h_i+h_{L-L_0-4}+11L/4}\geq\frac{8\delta^{-1}+22L}{8\delta^{-1}+11L}=1+\delta\frac{L}{\delta L+8/11}>1+\delta$$
As $N_4>>\delta^{-1}$, we may take $1+\delta>1+\eps$. But then noting that $L-L_0-4-i\geq L/2$, the above inequality contradicts Lemma \ref{7.26}.

\end{proof}

\begin{lemma} \label{7.36}

\textit{(Compare with Lemma 9.21 of [16] and Lemma 7.36 of [23])} 
\newline
For $i=1,\dots,L_0$, the spoke $\pazocal{Q}_i$ does not contain a $\lambda$-shaft of $\Pi$ of length at least $\delta h$.

\end{lemma}

\begin{proof}

Let $j=L_0+1$ and $\ell=L-L_0-4$.

Since $\Pi$ is removed when passing from $\Delta$ to $\Psi_{j,\ell}'$, $\pazocal{Q}_i$ is a cutting $q$-band of $\Psi_{j,\ell}'$. So, $\pazocal{Q}_i$ contains no $\lambda$-shaft in $\Psi_{j,\ell}'$. 

As Lemma \ref{weakly minimal}(1) implies $(\Psi_{j,\ell}')^*\subset\Delta^*$, we then have
$$\sigma_\lambda(\Delta^*)-\sigma_\lambda((\Psi_{j,\ell}')^*)\geq\delta h$$
Lemma \ref{7.24} then yields the inequalities
$$|\textbf{p}_{j,\ell}|\geq h_j+h_\ell+11(L-3L_0)$$
$$|\bar{\textbf{p}}_{j,\ell}|\leq h_j+h_\ell+3L_0(11+\delta|V|_a)$$
By Lemma \ref{7.34}, $\delta|V|_a<\delta h/L_0^2$, so that
$$|\bar{\textbf{p}}_{j,\ell}|<h_j+h_\ell+11(3L_0)+3\delta h/L_0$$
So, taking $L>>L_0\geq6$,
$$|\textbf{p}_{j,\ell}|+\sigma_\lambda(\Delta^*)-\sigma_\lambda((\Psi_{j,\ell}')^*)-|\bar{\textbf{p}}_{j,\ell}|\geq11(L-6L_0)+\delta h(1-3/L_0)\geq\frac{11L+\delta h}{2}$$
Hence, noting that $h_j\leq h=h_\ell$, we have:
\begin{align*}
\frac{|\textbf{p}_{j,\ell}|+\sigma_\lambda(\Delta^*)-\sigma_\lambda((\Psi_{j,\ell}')^*)-|\bar{\textbf{p}}_{j,\ell}|}{|\bar{\textbf{p}}_{j,\ell}|}&\geq\frac{11L+\delta h}{2(2h+11(3L_0)+3\delta h/L_0)} \\
&\geq\frac{11L+\delta h}{11(6L_0)+5h} \\
&\geq\min(L/6L_0,\delta/5)=\delta/5
\end{align*}
since $L>>L_0$. Taking $N_4>>\delta^{-1}$ implies $\eps<\delta/5$. But $\ell-j\geq L-3L_0\geq L/2$, so that the above inequality contradicts Lemma \ref{7.26}.

\end{proof}

\newpage

\begin{lemma} \label{7.37}

\textit{(Compare with Lemma 9.22 of [16] and Lemma 7.37 of [23])} 
\newline
For $i=1,\dots,L_0-1$, $|\textbf{z}_i|_a>h_{i+1}/2c_3$.

\end{lemma}

\begin{proof}

Suppose to the contrary that $|\textbf{z}_i|_a\leq h_{i+1}/2c_3$. 

Then $\|\textbf{z}_i\|=|\textbf{z}_i|_a+12\leq (h_{i+1}/2c_3)+12$.

Taking $\delta^{-1}>>c_3$, Lemma \ref{7.35} yields $h_{i+1}/2c_3>\delta^{-1}/2c_3>12$. So, $\|\textbf{z}_i\|< h_{i+1}/c_3$.

Further, taking $\delta^{-1}>>L_0>>c_3$, Lemma \ref{7.34} yields 
$$\|\textbf{y}_i\|=12+|V|_a<12+h/L_0^2\leq h_{i+1}/2c_3+h_{i+1}/L_0^2\leq h_{i+1}/c_3$$
Recall that by Lemma \ref{non-distinguished a-cells}, $\Gamma_i$ contains no $a$-cells, and so is a trapezium. By Lemma \ref{trapezia are computations}, there exists a reduced computation $\pazocal{C}'$ corresponding to $\Gamma_i$ with base $(\{t(\ell)\}B_3(\ell)\{t(\ell+1)\})^{\pm1}$ for some $2\leq\ell\leq L$, where $L+1$ is taken to be 1. Let $\pazocal{C}$ be the restriction of $\pazocal{C}^{\pm1}$ to the base $\{t(\ell)\}B_3(\ell)$.

Then, the history of $\pazocal{C}$ has length at least $h_{i+1}>c_3\max(\|\textbf{y}_i\|,\|\textbf{z}_i\|)$. As a result, $\pazocal{C}$ satisfies the hypotheses of Lemma \ref{M projected long history controlled}. Further, since every $\theta$-band crossing $\pazocal{Q}_{i+1}$ also crosses $\pazocal{Q}_i$, $W(\ell)$ is $H_{i+1}$-admissible for some $\ell\geq2$. So, $\pazocal{Q}_i$ contains a $\lambda$-shaft of length at least $h_{i+1}$.

But then $h_{i+1}\geq h>\delta h$, so that this contradicts Lemma \ref{7.36}.

\end{proof}

\begin{lemma} \label{7.38}

\textit{(Compare with Lemma 9.23 of [16] and Lemma 7.38 of [23])} 
\newline
For $i=1,\dots,L_0-1$, $h_{i+1}<(1-\frac{1}{30c_3})h_i$.

\end{lemma}

\begin{proof}

Assuming $h_{i+1}\geq(1-\frac{1}{30c_3})h_i$, the handle of $E_i$ has height at most $h_i-h_{i+1}\leq h_i/30c_3$. So, by Lemma \ref{7.30}, at most $h_i/5c_3$ maximal $a$-bands of $E_i$ starting on $\textbf{z}_i$ can end on $(\theta,q)$-cells of $E_i$. Hence, at least $\max(0,|\textbf{z}_i|_a-h_i/5c_3)$ of these bands end on $\textbf{p}_{i,i+1}$.

Lemma \ref{7.37} implies that $|\textbf{z}_i|_a>h_{i+1}/2c_3$, so that 
$$|\textbf{z}_i|_a-h_i/5c_3\geq h_{i+1}/2c_3-h_i/5c_3\geq \bigg(1-\frac{1}{30c_3}\bigg)h_i/2c_3-h_i/5c_3>h_i/15c_3$$
By Lemma \ref{7.22}(2), $\textbf{p}_{i,i+1}$ also has $h_i-h_{i+1}\leq h_i/30c_3$ $\theta$-edges. So, Lemma \ref{7.24}(1) implies the inequalities
$$|\textbf{p}_{i,i+1}|\geq h_i-h_{i+1}+\delta h_i/30c_3$$
$$|\textbf{p}_{i+1,L-L_0-4}|\geq h_{i+1}+h_{L-L_0-4}+22L/3$$
As these paths have an overlap of one $q$-edge, this implies
$$|\textbf{p}_{i,L-L_0-4}|=|\textbf{p}_{i,i+1}|+|\textbf{p}_{i+1,L-L_0-4}|-1>h_i+h_{L-L_0-4}+11L/2+\delta h_i/30c_3$$
Meanwhile, Lemma \ref{7.24}(2) gives us
$$|\bar{\textbf{p}}_{i,L-L_0-4}|\leq h_i+h_{L-L_0-4}+11(3L_0)+3L_0\delta|V|_a$$
As Lemma \ref{7.34} implies $|V|_a<h/L_0^2\leq h_i/L_0^2$, we then have
$$|\bar{\textbf{p}}_{i,L-L_0-4}|\leq h_i+h_{L-L_0-4}+11(3L_0)+3\delta h_i/L_0$$
Hence, since $L>>L_0>>c_3$ and $h_{L-L_0-4}\leq h\leq h_i$,
\begin{align*}
\frac{|\textbf{p}_{i,L-L_0-4}|-|\bar{\textbf{p}}_{i,L-L_0-4}|}{|\bar{\textbf{p}}_{i,L-L_0-4}|}&\geq\frac{11(L/2-3L_0)+\delta h_i(1/30c_3-3/L_0)}{h_i+h_{L-L_0-4}+11(3L_0)+3\delta h_i/L_0} \\
&\geq\frac{11L/3+\delta h_i/60c_3}{11(3L_0)+3h_i} \\
&\geq\min(L/9L_0,\delta/180c_3)=\delta/180c_3
\end{align*}
However, taking $N_4>>\delta^{-1}>>c_3$ yields $\eps<\delta/180c_3$, so that the above inequality contradicts Lemma \ref{7.26}.

\end{proof}

\begin{lemma} \label{upper bound on z_i}

\textit{(Compare with Lemma 9.24 of [16] and Lemma 7.39 of [23])} 
\newline
For $i=1,\dots,L_0-1$, $|\textbf{z}_i|_a\leq 8h_i$.

\end{lemma}

\begin{proof}

Assume $|\textbf{z}_i|_a>8h_i$. By Lemma \ref{7.30}, at most $6h_i$ maximal $a$-bands of $E_i^0$ starting on $\textbf{z}_i$ can end on the $(\theta,q)$-cells of $E_i$. So, since $a$-bands cannot cross $q$-bands, Lemma \ref{non-distinguished a-cells} implies that at least $|\textbf{z}_i|_a-6h_i>2h_i$ maximal $a$-bands of $E_i^0$ starting on $\textbf{z}_i$ must end on the path $\textbf{q}_{i,i+1}$. Hence, $|\textbf{q}_{i,i+1}|_a>2h_i$.

By Lemma \ref{7.22}(2), $\textbf{q}_{i,i+1}$ has at most $h_i$ $\theta$-edges. As a result, Lemma \ref{lengths} implies that at least $h_i$ $a$-edges of $\textbf{q}_{i,i+1}$ contribute $\delta$ to $|\textbf{q}_{i,i+1}|$, and so also to $|\textbf{q}_{i,L-L_0-4}|$. So, Lemmas \ref{7.24} and \ref{7.27} give the inequalities
$$|\textbf{p}_{i,L-L_0-4}|\geq|\textbf{q}_{i,L-L_0-4}|\geq h_i+h_{L-L_0-4}+11L/2+\delta h_i$$
$$|\bar{\textbf{p}}_{i,L-L_0-4}|\leq h_i+h_{L-L_0-4}+11(3L_0)+3L_0\delta|V|_a$$
Taking $|V|_a<h/L_0^2\leq h_i/L_0^2$ by Lemma \ref{7.34} then gives
$$|\bar{\textbf{p}}_{i,L-L_0-4}|\leq h_i+h_{L-L_0-4}+11(3L_0)+3\delta h_i/L_0$$
so that
$$|\textbf{p}_{i,L-L_0-4}|-|\bar{\textbf{p}}_{i,L-L_0-4}|\geq 11(L/2-3L_0)+\delta h_i(1-3/L_0)\geq 11L/3+\delta h_i/2$$
Then, since $h_{L-L_0-4}\leq h\leq h_i$,
\begin{align*}
\frac{|\textbf{p}_{i,L-L_0-4}|-|\bar{\textbf{p}}_{i,L-L_0-4}|}{|\bar{\textbf{p}}_{i,L-L_0-4}|}&\geq\frac{11L/3+\delta h_i/2}{11(3L_0)+3h_i} \\
&\geq\min(L/9L_0,\delta /6)=\delta /6
\end{align*}
However, again taking $N_4>>\delta^{-1}$, $\eps<\delta/6$ so that the above inequality contradicts Lemma \ref{7.26}.

\end{proof}

Note that if $\Psi_{i,i+1}$ is the distinguished clove for $i\leq r-1$, then $H_{i+1}$ need not be the history of $\Gamma_i$. To account for this, let $H_{i+1}'$ be the history of $\Gamma_i$. Note that $H_{i+1}$ is always a prefix of $H_{i+1}'$.

The following is the analogue of Lemma 9.25 of [16] and Lemma 7.40 of [23].

\begin{lemma} \label{no one-step}

For $2\leq i\leq L_0-2$, let $H_i'=H_{i+1}'H'=H_{i+2}'H''H'$ and $\pazocal{C}$ be the computation corresponding to the trapezium $\Gamma_{i-1}$. Suppose the subcomputation $\pazocal{D}$ of $\pazocal{C}$ with history $H''H'$ has step history of length 1. Then there is no two-letter subword $Q'Q$ of the base of $\Gamma_{i-1}$ such that every rule of $\pazocal{D}$ inserts one letter to the left of $Q$.

%Then the subcomputation $\pazocal{D}$ of $\pazocal{C}$ with history $H''H'$ cannot have step history of length 1 so that there exists a sector $QQ'$ such that one of either $Q$ or $Q'$ has a letter inserted next to it to increase the length of this sector for each rule of $\pazocal{D}$.

\end{lemma}

\begin{proof}

%Supposing the statement is false, Lemma \ref{M one-step} allows us to assume without loss of generality that the rules of $\pazocal{D}$ write letters to the right of $Q$.

Let $\pazocal{Q}$ be the maximal $q$-band of $E_i^0$ that is a subband of the $q$-spoke of $\Pi$ corresponding to a coordinate shift the state letter $Q$. Similarly, let $\pazocal{Q}'$ be the maximal $q$-band corresponding to a coordinate shift of $Q'$, so that $\pazocal{Q}'$ and $\pazocal{Q}$ are neighbor $q$-bands. Let $\textbf{x}$ be the subpath of $\textbf{z}_i$ between $\pazocal{Q}'$ and $\pazocal{Q}$.

Since $\Gamma_i$ contains a copy $\Gamma_{i+1}'$ of the trapezium $\Gamma_{i+1}$, the bottom of the trapezium $\Gamma_i\setminus\Gamma_{i+1}'$ is a copy $\textbf{z}_{i+1}'$ of $\textbf{z}_{i+1}$, while the top is $\textbf{z}_i$. This trapezium has history $H''$, so that the corresponding computation inserts one $a$-letter to the left of the state letter corresponding to $\pazocal{Q}$ at each transition. As a result, $|\textbf{x}|_a\geq\|H''\|\geq h_{i+1}-h_{i+2}$. 

By Lemma \ref{7.38}, $h_{i+1}-h_{i+2}>\frac{1}{30c_3}h_{i+1}$. As $h_{i+1}\geq h$, Lemma \ref{7.34} and the parameter choice $L_0>>c_3$ imply
$$|\textbf{x}|_a\geq\frac{h}{30c_3}>\frac{L_0^2|V|_a}{30c_3}>10L_0|V|_a$$
If an $a$-band starting on $\textbf{x}$ ended on a $(\theta,q)$-cell of $\pazocal{Q}$, then Lemma \ref{7.26} implies that there is a copy of this in the trapezium $\Gamma_{i-1}\setminus\Gamma_i'$. By Lemma \ref{trapezia are computations}, though, this would contradict the assumption that rules of $\pazocal{D}$ only write letters in the sector.

\begin{figure}[H]
\centering
\includegraphics[scale=1]{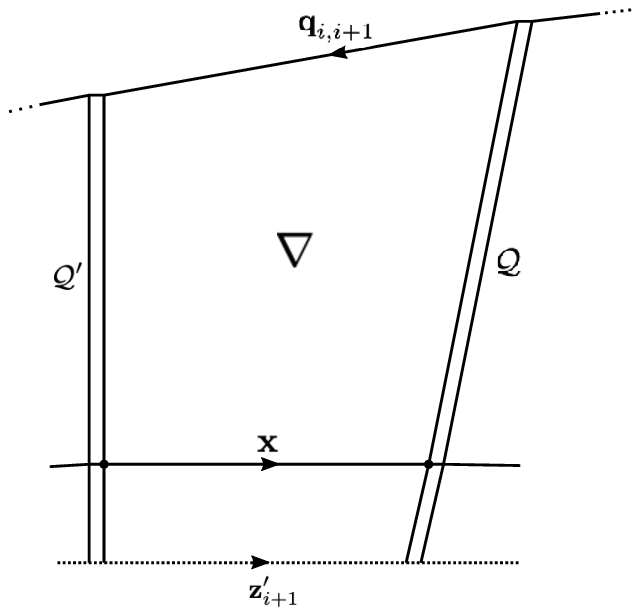}
\caption{}
\end{figure}

Now, consider the comb $\nabla$ contained in $E_i^0$ bounded by $\pazocal{Q}'$, $\pazocal{Q}$, $\textbf{x}$, and $\textbf{q}_{i,i+1}$ (see Figure 10.9). Set $s$ and $s'$ as the lengths of $\pazocal{Q}$ and $\pazocal{Q}'$, respectively. Lemmas \ref{7.22}(2) and \ref{7.28}(2) imply $s'\leq s$. So, by Lemma \ref{7.28}(1), there are $|\textbf{x}|_a+s$ maximal $a$-bands starting on $\textbf{x}$ or $\pazocal{Q}$ and ending on $\pazocal{Q}'$ or on $\textbf{q}_{i,i+1}$. Since only $s'$ $a$-bands can end on $\pazocal{Q}'$, at least $|\textbf{x}|_a+s-s'$ of them end on the segment of $\textbf{q}_{i,i+1}$ between $\pazocal{Q}$ and $\pazocal{Q}'$. By Lemmas \ref{7.22}(2) and \ref{7.28}(2), the same segment contains $s-s'$ $\theta$-edges, meaning at least $|\textbf{x}|_a$ of them contribute $\delta$ to its length. So, by Lemma \ref{7.24}(1),
\begin{align*}
|\textbf{p}_{i,L-L_0-4}|&\geq h_i+h_{L-L_0-4}+11L/2+\delta\frac{h_{i+1}}{30c_3} \\
&\geq h_i+h_{L-L_0-4}+11L/2+10\delta L_0|V|_a
\end{align*}
Also by Lemma \ref{7.24}(2) and \ref{7.34},
$$|\bar{\textbf{p}}_{i,L-L_0-4}|\leq h_i+h_{L-L_0-4}+11(3L_0)+3\delta L_0|V|_a\leq h_i+h_{L-L_0-4}+11(3L_0)+3\delta h/L_0$$
So,
\begin{align*}
|\textbf{p}_{i,L-L_0-4}|-|\bar{\textbf{p}}_{i,L-L_0-4}|&\geq11(L/2-3L_0)+\delta(h_{i+1}/30c_3-3h/L_0) \\
&>\delta h_{i+1}(1/30c_3-3/L_0) \\
&>\delta h_{i+1}/50c_3
\end{align*}
by again taking $L_0>>c_3$.

Let $\textbf{s}$ be the complement of the $\textbf{p}_{i,L-L_0-4}$ in $\partial\Delta$. Then, since $\textbf{p}_{i,L-L_0-4}$ starts and ends with $q$-edges, Lemma \ref{lengths}(c) implies
\begin{align*}
|\partial\Psi'_{i,L-L_0-4}|&\leq|\textbf{s}|+|\bar{\textbf{p}}_{i,L-L_0-4}| \\
&<|\textbf{s}|+|\textbf{p}_{i,L-L_0-4}|-\delta h_{i+1}/50c_3 \\
&=|\partial\Delta|-\delta h_{i+1}/50c_3 \numberthis
\end{align*}
Lemma \ref{weakly minimal}(1) implies that $\Psi'_{i,L-L_0-4}$ is weakly minimal with $\sigma_\lambda((\Psi'_{i,L-L_0-4})^*)\leq\sigma_\lambda(\Delta^*)$.

Hence, if $\Psi_{i,L-L_0-4}'$ contains a disk, then we may apply the inductive hypothesis to it. Otherwise, we may apply Lemma \ref{diskless} to $\Psi_{i,L-L_0-4}'$. In either case, this implies
$$
\text{wt}_G(\Psi'_{i,L-L_0-4})\leq N_4(|\Psi'_{i,L-L_0-4}|+\sigma_\lambda((\Psi'_{i,L-L_0-4})^*))^2+N_3\mu(\Psi_{i,L-L_0-4}')
$$
Taking $x=|\partial\Delta|+\sigma_\lambda(\Delta^*)$, we have
$$0\leq|\Psi_{i,L-L_0-4}'|+\sigma_\lambda((\Psi_{i,L-L_0-4}')^*)\leq x-\delta h_{i+1}/50c_3$$
So, since $\delta h_{i+1}/50c_3\leq x$,
\begin{equation} \label{G-weight Psi'}
\text{wt}_G(\Psi_{i,L-L_0-4}')<N_4x^2-N_4\delta xh_{i+1}/50c_3+N_3\mu(\Psi_{i,L-L_0-4}')
\end{equation}
Next, note that $|V|_a\leq h_i/L_0^2$ by Lemma \ref{7.34}, $h_i>\delta^{-1}>11(100L_0)$ by Lemma \ref{7.35}, and $h_{L-L_0-4}\leq h\leq h_i$. So, for sufficiently large $L_0$, we have
$$|\bar{\textbf{p}}_{i,L-L_0-4}|\leq 2h_i+3h_i/100+3\delta h_i/L_0\leq 2.1h_i$$
As Lemma \ref{7.26} implies $|\textbf{p}_{i,L-L_0-4}|\leq(1+\eps)|\bar{\textbf{p}}_{i,L-L_0-4}|$, taking $N_4$ sufficiently large yields 
$$|\partial\Psi_{i,L-L_0-4}|\leq|\textbf{p}_{i,L-L_0-4}|+|\bar{\textbf{p}}_{i,L-L_0-4}|+|\partial\Pi|\leq4.5h_i+|\partial\Pi|$$
Taking $\delta^{-1}>>L$, note that
$$|\partial\Pi|\leq 11L+(L+1)\delta|V|_a\leq\delta^{-1}/4+h_i/L_0^2\leq h_i/2$$
Since $\Psi_{i,L-L_0-4}$ contains no disks, Lemma \ref{diskless} implies
$$\text{wt}_G(\Psi_{i,L-L_0-4})\leq 25N_2h_i^2+N_1\mu(\Psi_{i,L-L_0-4})$$
while the assignment of weight implies
$$\text{wt}(\Pi)=C_1|\partial\Pi|^2\leq C_1h_i^2$$
Combining these two inequalities, Lemma \ref{G-weight subdiagrams} implies
$$\text{wt}_G(\bar{\Delta}_{i,L-L_0-4})\leq 25N_2h_i^2+N_1\mu(\Psi_{i,L-L_0-4})+C_1h_i^2<26N_2h_i^2+N_1\mu(\Psi_{i,L-L_0-4})$$
By Lemma \ref{7.22}, $|\partial\Psi_{i,L-L_0-4}|_\theta=2(h_i+h_{L-L_0-4})\leq4h_i$. So, by Lemma \ref{mixtures}(1), 
$$\mu(\Psi_{i,L-L_0-4})\leq16Jh_i^2$$
Hence, taking $N_2>>N_1>>J$,
\begin{equation} \label{G-weight bar Delta}
\text{wt}_G(\bar{\Delta}_{i,L-L_0-4})\leq26N_2h_i^2+16N_1Jh_i^2\leq30N_2h_i^2
\end{equation}
Thus, combining (\ref{G-weight Psi'}) and (\ref{G-weight bar Delta}), Lemma \ref{G-weight subdiagrams} implies
$$
\text{wt}_G(\Delta)\leq N_4x^2-N_4\delta xh_{i+1}/50c_3+N_3\mu(\Psi_{i,L-L_0-4}')+30N_2h_i^2
$$
Hence, to reach a contradiction, it suffices to show that
\begin{equation} \label{7.40 suffices 1}
N_4\delta xh_{i+1}/50c_3\geq N_3\mu(\Psi_{i,L-L_0-4}')-N_3\mu(\Delta)+30N_2h_i^2
\end{equation}
Now consider the diagram $\Psi_{i+1,L-L_0-4}'$. When passing from $\partial\Psi_{i+1,L-L_0-4}'$ to $\partial\Psi_{i,L-L_0-4}'$, a subpath $\textbf{t}$ is replaced with $\textbf{bot}(\pazocal{Q}_i)^{-1}$ (see Figure 10.10). The subpath $\textbf{t}$ consists of:

\begin{itemize}

\item the subpath $\textbf{p}_{i,i+1}'$ of $\partial\Delta$ obtained from $\textbf{p}_{i,i+1}$ by removing the end of $\pazocal{Q}_{i+1}$,

\item $\textbf{bot}(\pazocal{Q}_{i+1})^{-1}$, and

\item a subpath of the inverse of $\partial\Pi$

\end{itemize}

\begin{figure}[H]
\centering
\includegraphics[scale=1.4]{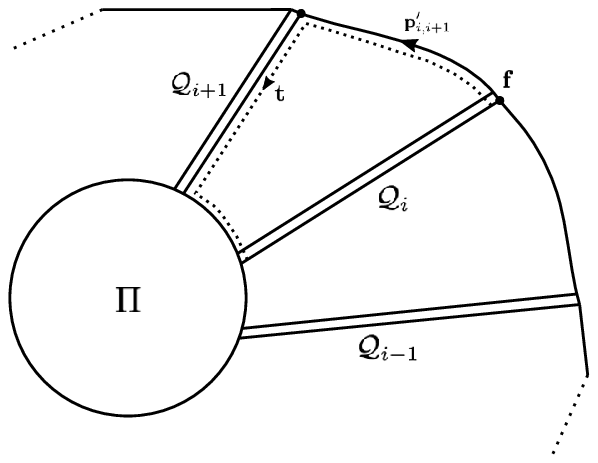}
\caption{}
\end{figure}

By Lemma \ref{7.22}, there is a correspondence between the $\theta$-edges of $\textbf{t}$ and those of $\textbf{bot}(\pazocal{Q}_i)$. So, since $\textbf{bot}(\pazocal{Q}_i)$ contains no $q$-edges, the necklace corresponding to $\Psi_{i,L-L_0-4}'$ may be obtained from that of $\Psi_{i+1,L-L_0-4}'$ by the removal of the black beads corresponding to the $q$-edges of $\textbf{t}$.

Consider the $q$-edge $\textbf{f}$ on the end of $\pazocal{Q}_i$. In $\partial\Psi_{i+1,L-L_0-4}'$, $\textbf{f}$ separates the $h_{i-1}-h_i$ $\theta$-edges of $\textbf{p}_{i-1,i}$ from the $h_i+h_{L-L_0-4}$ $\theta$-edges of the path $\textbf{p}_{i,i+1}'\bar{\textbf{p}}_{i+1,L-L_0-4}$.

By Lemma \ref{7.23}, $|\textbf{p}_{i,i+1}'|_q<3K_0$, while $|\bar{\textbf{p}}_{i+1,L-L_0-4}|_q\leq 11L$. So, as $J>>K>>L$, there are at most $J$ $q$-edges of $\partial\Psi_{i+1,L-L_0-4}'$ between any pair of $\theta$-edges separated by $\textbf{f}$ mentioned above. As such, Lemma \ref{mixtures}(d) implies 
$$\mu(\Psi_{i+1,L-L_0-4}')-\mu(\Psi_{i,L-L_0-4}')\geq (h_{i-1}-h_i)(h_i+h_{L-L_0-4})$$
Meanwhile, Lemma \ref{7.25} implies
$$\mu(\Delta)-\mu(\Psi_{i+1,L-L_0-4}')\geq-2J|\partial\Delta|(h_{i+1}+h_{L-L_0-4})$$
Combining these and noting that $h_{L-L_0-4}\leq h\leq h_{i+1}\leq h_i$, we have
$$\mu(\Delta)-\mu(\Psi_{i,L-L_0-4}')\geq 2h_i(h_{i-1}-h_i)-4J|\partial\Delta|h_{i+1}$$
Hence by (\ref{7.40 suffices 1}), it suffices to show that
\begin{equation} \label{7.40 suffices 2}
N_4\delta xh_{i+1}/50c_3+2N_3h_i(h_{i-1}-h_i)\geq 4N_3J|\partial\Delta|h_{i+1}+30N_2h_i^2
\end{equation}
As $x=|\partial\Delta|+\sigma_\lambda(\Delta^*)\geq|\partial\Delta|$, the parameter choices $N_4>>N_3>>\delta^{-1}>>J>>c_3$ imply
$$N_4\delta xh_{i+1}/50c_3\geq\frac{N_4\delta}{50c_3}|\partial\Delta|h_{i+1}\geq 4N_3J|\partial\Delta|h_{i+1}$$
Moreover, by Lemma \ref{7.38}, $h_{i-1}-h_i>h_{i-1}/30c_3\geq h_i/30c_3$. Hence, the parameter choices $N_3>>N_2>>c_3$ imply
$$2N_3h_i(h_{i-1}-h_i)>\frac{N_3}{15c_3}h_i^2\geq30N_2h_i^2$$
Thus, the statement is proved.

\end{proof}

\begin{remark}

Recall that we have assumed without loss of generality that $h_{L-L_0-4}\leq h_{L_0+1}$. If $h_{L-L_0-4}>h_{L_0+1}$, then the symmetric statement to Lemma \ref{no one-step} will be needed for $L-L_0-2\leq i\leq L-1$. This statement can be proved analogously.

\end{remark}

Finally, we reach the final contradiction of this section, the analogue of Lemma 9.26 of [16] and Lemma 7.41 of [23].

\begin{lemma} \label{contradiction}

The counterexample diagram $\Delta$ cannot exist.

\end{lemma}

\begin{proof}

First, fix an integer $\eta\geq2$ dependant on $c_3$ such that $(1-\frac{1}{30c_3})^\eta<\frac{1}{64c_3}$. Note that, although $\eta$ is not listed as one of the parameters of Section 3.3, we may take $L_0>>\eta$ since $L_0$ is chosen after $c_3$.

For $i=1,\dots,L_0-1$, Lemma \ref{7.38} implies $h_{i+1}<(1-\frac{1}{30c_3})h_i$. So, if $1\leq i<j\leq L_0-1$ with $j-i-1\geq\eta$, then $h_j<(1-\frac{1}{30c_3})^\eta h_{i+1}<\frac{1}{64c_3}h_{i+1}$.

For such $i,j$, Lemma \ref{7.37} then implies that $|\textbf{z}_i|_a\geq h_{i+1}/2c_3>32h_j$. As Lemma \ref{upper bound on z_i} implies $8h_j\geq|\textbf{z}_j|_a$, we then have $|\textbf{z}_i|_a>4|\textbf{z}_j|_a$.

Now, as $L_0>>\eta$ and $L_0>>c_0$, there exist indices $2\leq j_1<j_2<\dots<j_m\leq L_0-1$ such that $m\geq c_0$ and $j_{i+1}-j_i-1\geq\eta$. So, $|\textbf{z}_{j_i}|_a>4|\textbf{z}_{j_{i+1}}|_a$ and $h_{j_i+1}\geq64c_3h_{j_{i+1}}$.

Let $\pazocal{C}:W_0\to\dots\to W_t$ be the computation corresponding to the trapezium $\Gamma_{j_2}$ by Lemma \ref{trapezia are computations}. As $\Gamma_{j_2}$ contains a copy of $\Gamma_{j_2+1}$, which in turn contains a copy of $\Gamma_{j_2+2}$ and so on, there exist words $V_i$ in $\pazocal{C}$ for $i=1,\dots,m$ that are coordinate shifts of the labels of $\textbf{z}_{j_i}$. By the inequalities above, $|V_{i+1}|_a>4|V_i|_a$.

If for some $i$ the subcomputation $V_{i+2}\to\dots\to V_i$ is a one-step computation, then by Lemma \ref{M one-step} there exists a right-active (or left-active if $h=h_{L-L_0-4}$) sector $Q'Q$ such that the sector's length increases with each transition of the subcomputation. But since $\eta\geq2$, there must exist a subcomputation contradicting Lemma \ref{no one-step}. 

Hence, the subcomputation $\pazocal{C}':V_m\to\dots\to W_t$ of $\pazocal{C}$ must contain at least $c_0/2\geq8n$ distinct one-step computations. Lemma \ref{long step history} then implies that the step history of $\pazocal{C}'$ contains a subword of the form $(4n-2,4n-1)_j(4n-1)_j(4n-1,4n)_j$ or $(4n,4n-1)_j(4n-1)_j(4n-1,4n-2)_j$. Let $\pazocal{C}''$ be the subcomputation of $\pazocal{C}'$ with this step history.

Then, we may factor $H_{j_2+1}'\equiv H'H''H'''$ where $H''$ is a controlled history. Further, since the subcomputation $\pazocal{C}''$ repeats $k$ copies of a controlled history, taking $k\geq3$ allows us to assume $\|H''\|\leq\|H'\|$.

Since $h_{j_1+1}>64c_3h_{j_2}$, $H_{j_1+1}'$ has prefix $K\equiv H'H''H'''_1$ where $\|H'''_1\|=\|H'\|\geq\|H''\|$. Set $\pazocal{B}$ as the subband of the spoke $\pazocal{Q}_{j_1}$ with history $K$. Then, for any factorization $\pazocal{B}_1\pazocal{B}_2\pazocal{B}_3$ such that the sum of the lengths of $\pazocal{B}_1$ and $\pazocal{B}_3$ is at most $\frac{1}{3}\|K\|$, the history of $\pazocal{B}_2$ must contain $H''$. So, since all $\theta$-bands crossing $\pazocal{Q}_{j_1}$ must cross $\pazocal{Q}_{j_1-1}$, taking $\lambda<1/3$ implies $\pazocal{B}$ is a $\lambda$-shaft with length $\|K\|$.

However, note that the subcomputation $W_0\to\dots\to V_m$ has length at least $h_{L_0-1}\geq h$, so that $\|K\|\geq\|H'\|\geq h>\delta h$. Thus, the existence of $\pazocal{B}$ in $\pazocal{Q}_{j_1}$ contradicts Lemma \ref{7.36}.

\end{proof}

\medskip

%%%%%%%%%%%%%%%%%%%%%%%%%%%%%%%%%%%%%%%%%%%%%%%%%%%%%%%%%%%%%%%%%

\section{Proof of Theorem \ref{main theorem}}

We now complete the proof of the main theorem.

The first step toward this is to justify the assignments made throughout the construction.

\subsection{Assignment of $a$-relations and weights} \

As mentioned in the introduction to the groups of interest in Section 6, the set of $a$-relators of interest in this section, $\pazocal{S}$, is the set of words in the letters $\pazocal{A}\cup\pazocal{A}^{-1}$ whose value in the free Burnside group $B(\pazocal{A},n)$ is 1.

The following Lemma sheds some light on why these particular relations are adjoined to the group presentation.

\begin{lemma} \label{a-relations are relations}

For any word $u\in F(\pazocal{A})$, the relation $u^n=1$ holds in the group $G(\textbf{M})$.

\end{lemma}

\begin{proof}

Lemmas \ref{M language} and \ref{disks are relations} imply that the words corresponding to the configurations $I(u^n)$ and $J(u^n)$ are trivial over the group $G(\textbf{M})$. These two words differ only by the insertion of the word $u^n$ in the `special' input sector, so that $u^n=1$ in $G(\textbf{M})$.

\end{proof}

\begin{lemma} \label{G isomorphic to G_a}

The groups $G(\textbf{M})$ and $G_\pazocal{S}(\textbf{M})$ are isomorphic.

\end{lemma}

\begin{proof}

Identify $B(\pazocal{A},n)$ with the presentation $\gen{\pazocal{A}\mid w=1, w\in\pazocal{L}}$.

Then let $\varphi:\pazocal{A}\to G(\textbf{M})$ be the map sending each letter to its natural copy in the tape alphabet of the `special' input sector. By the theorem of von Dyck (Theorem 4.5 in [14]), Lemma \ref{a-relations are relations} implies that $\varphi$ extends to a homomorphism $B(\pazocal{A},n)\to G(\textbf{M})$. So, for any word $w$ corresponding to an $a$-relation $w=1$, the relation $w=1$ holds in $G(\textbf{M})$. 

The theorem of von Dyck then implies that the map sending each generator of the canonical presentation of $G(\textbf{M})$ to the corresponding generator of the disk presentation of $G_\pazocal{S}(\textbf{M})$ extends to an isomorphism between the two groups.

\end{proof}

\begin{lemma} \label{embedding}

The group $B(\pazocal{A},n)$ embeds in the group $G(\textbf{M})$.

\end{lemma}

\begin{proof}

Consider the natural map $\varphi:\pazocal{A}\to G_\pazocal{S}(\textbf{M})$ sending the elements of $\pazocal{A}$ to their copies in the tape alphabet of the `special' input sector. The theorem of von Dyck implies that this extends to a homomorphism $\varphi:B(\pazocal{A},n)\to G_\pazocal{S}(\textbf{M})$.

Now suppose the reduced word $w$ over $\pazocal{A}$ satisfies $\varphi(w)=1$. Then by Lemma \ref{minimal exist}, there exists a minimal diagram $\Delta$ over $G_\pazocal{S}(\textbf{M})$ satisfying $\text{Lab}(\partial\Delta)\equiv w$. By Lemmas \ref{graph} and \ref{G_a theta-annuli}, every cell of $\Delta$ must be an $a$-cell. But then this is a diagram over $B(\pazocal{A},n)$, so that $w=1$ in $B(\pazocal{A},n)$.

So, $\varphi:B(\pazocal{A},n)\to G_\pazocal{S}(\textbf{M})$ is an embedding. Lemma \ref{G isomorphic to G_a} then implies the statement.

\end{proof}

Now we wish to justify our assignment of weights to $a$-cells and disks over the disk presentation of $G_\pazocal{S}(\textbf{M})$. To do so, we first study areas of a diagram over the canonical presentation of $G(\textbf{M})$ with contour label corresponding to a disk relation.

\begin{lemma} \label{disks are quadratic}

(1) For any configuration $W$ accepted by $\textbf{M}$, there exists a reduced diagram $\Delta$ over the canonical presentation of $G(\textbf{M})$ such that $\lab(\partial\Delta)\equiv W$ and $\text{Area}(\Delta)\leq C_1|W|^2$.

(2) For any $u^n\in\pazocal{L}$, there exists a reduced diagram $\Delta$ over the canonical presentation of $G(\textbf{M})$ with $\lab(\partial\Delta)\equiv u^n$ and $\text{Area}(\Delta)\leq C_1\|u\|^2$.

\end{lemma}

\begin{proof}

(1) By Lemma \ref{M accepted configurations}, there exists a computation $\pazocal{C}:W\equiv W_0\to\dots\to W_t$ accepting $W$ such that $t\leq c_2\|W(i)\|$ for all $i\geq2$. Further, by Lemma \ref{M width}, $\|W_j\|\leq c_2\|W\|$ for all $j$.

By Lemma \ref{computations are trapezia}, we can then build a trapezium $\Gamma$ over $M(\textbf{M})$ corresponding to $\pazocal{C}$, so that $\lab(\textbf{tbot}(\Gamma))\equiv W$ and $\lab(\textbf{ttop}(\Gamma))\equiv W_{ac}$. 

Given a maximal $\theta$-band $\pazocal{T}$ of $\Gamma$, $\|\textbf{tbot}(\pazocal{T})\|\leq c_2\|W\|$. So, $\text{Area}(\Gamma)\leq 3c_2^2\|W\|^2$.

As the $Q_s(L)\{t(1)\}$-sector is locked by every rule, the sides of $\Gamma$ are labelled identically and no trimming was necessary. So, we may glue these sides together and paste a hub into the middle of the diagram. This produces a reduced diagram $\Delta$ over the canonical presentation of $G(\textbf{M})$ with $\lab(\partial\Delta)\equiv W$ and satisfying $\text{Area}(\Delta)\leq3c_2^2\|W\|^2+1$. 

The statement then follows as we choose the parameter $C_1$ after $c_2$ and $\delta$.

(2) Clearly, we may assume that $u^n$ is nontrivial in $F(\pazocal{A})$.

As in the previous case, we can build diagrams $\Delta_1$ and $\Delta_2$ over the canonical presentations of $G(\textbf{M})$ where $\Delta_j$ is made of a hub and a trapezium satisfying: 
\begin{itemize}
\item $\text{Lab}(\Delta_1)\equiv I(u^n)$ and $\text{Area}(\Delta_1)\leq3c_2^2\|I(u^n)\|^2+1$

\item $\text{Lab}(\Delta_2)\equiv J(u^n)$ and $\text{Area}(\Delta_2)\leq3c_2^2\|J(u^n)\|^2+1$
\end{itemize}

Note that $\|I(u^n)\|,\|J(u^n)\|\leq L(11+n\|u\|)$. So, since $C_1$ is chosen after $c_2$, $L$, and $n$, we can assume that $\text{Area}(\Delta_j)\leq \frac{1}{2}C_1\|u\|^2$ for $j=1,2$.

Gluing $\Delta_1$ and $\Delta_2$ along their common contours (and making any possible cancellations) then yields a diagram $\Delta$ satisfying the statement.

\end{proof}

\begin{lemma} \label{a-cells are quadratic} If $w$ is a reduced word over the alphabet $\pazocal{A}$ such that $w=1$ in $B(\pazocal{A},n)$, then there exists a reduced diagram $\Delta$ over the canonical presentation of $G(\textbf{M})$ with $\lab(\partial\Delta)\equiv w$ and satisfying $\text{Area}(\Delta)\leq C_1\|w\|^2$.

\end{lemma}

\begin{proof}

Let $\Delta_0$ be a van Kampen diagram over the presentation $\gen{\pazocal{A}\mid\pazocal{R}}$ of $B(\pazocal{A},n)$ (see Section 2.8) with $\lab(\partial\Delta_0)\equiv w$. For each cell $\Pi_0$ in $\Delta_0$, $\lab(\partial\Pi_0)\in\pazocal{R}\subset\pazocal{L}$. Setting $\lab(\partial\Pi_0)\equiv (u(\Pi_0))^n$, Lemma \ref{disks are quadratic}(2) then implies that there exists a diagram $\Pi$ over the canonical presentation of $G(\textbf{M})$ satisfying $\lab(\partial\Pi)\equiv (u(\Pi_0))^n$ and $\text{Area}(\Pi)\leq C_1\|u(\Pi_0)\|^2$.

Pasting $\Pi$ in place of $\Pi_0$ for each cell of $\Delta_0$ then produces a diagram $\Delta$ over the canonical presentation of $G(\textbf{M})$ satsifying $\lab(\partial\Delta)\equiv w$ and $$\text{Area}(\Delta)=\sum\text{Area}(\Pi)\leq\sum\limits_{\Pi_0\in\Delta_0} C_1\|u(\Pi_0)\|^2$$

But defining $\rho(\Pi_0)=\|u(\Pi_0)\|^2$ as in the definition of mass in Section 2.9, Lemma \ref{a-cells are quadratic B(m,n)} implies $$\sum\limits_{\Pi_0\in\Delta_0}\|u(\Pi_0)\|^2=\sum\limits_{\Pi_0\in\Delta_0}\rho(\Pi_0)\defeq\rho(\Delta_0)\leq\|\partial\Delta_0\|^2$$
Hence, $\text{Area}(\Delta)\leq C_1\|\partial\Delta_0\|^2=C_1\|w\|^2$.

\end{proof}

Note that the proof of Lemma \ref{a-cells are quadratic} relies on the assumption that $n$ satisfies $(*)$.

\subsection{Assignment of $G$-weight} \

\begin{lemma} \label{impeding G-weight}

Let $\Delta$ be an impeding $a$-trapezium. Then there exists a reduced diagram $\tilde{\Delta}$ over $G(\textbf{M})$ such that $\lab(\partial\tilde{\Delta})\equiv\lab(\partial\Delta)$ and $\text{Area}(\tilde{\Delta})\leq2\text{wt}_G(\Delta)$.

\end{lemma}

\begin{proof}

Suppose $\text{wt}_G(\Delta)=\frac{1}{2}\text{wt}(\Delta)$. Then, let $\tilde{\Delta}$ be the diagram constructed from $\Delta$ by replacing each $a$-cell with the corresponding reduced diagram over $G(\textbf{M})$ constructed in Lemma \ref{disks are quadratic}. Then $\text{Area}(\tilde{\Delta})\leq\text{wt}(\Delta)=2\text{wt}_G(\Delta)$.

So, it suffices to assume that:
$$\text{wt}_G(\Delta)=3h\max(\|\textbf{tbot}(\Delta)\|,\|\textbf{ttop}(\Delta)\|)+3C_1h\eta+C_1(|\textbf{tbot}(\Delta)|_a+|\textbf{ttop}(\Delta)|_a+2\eta)^2$$ 
for $\eta=\|H_1\|+n\|H_2\|+\|H_3\|$.

Let $\pazocal{T}_0$ be the maximal $\theta$-band of $\Delta$ such that $\textbf{bot}(\pazocal{T}_0)=\textbf{bot}(\Delta)$. Letting $\theta_0$ be the rule corresponding to $\pazocal{T}_0$, Lemma \ref{theta-bands are one-rule computations} implies that the admissible word $V_0\equiv\lab(\textbf{tbot}(\Delta))$ is $\theta_0$-admissible.

If the base of $\Delta$ is $(P_0(1)Q_0(1))^{\pm1}$, then $V_0$ is $H$-admissible by the definition of the rules. Otherwise, the base of $\Delta$ is $Q_0(1)^{-1}Q_0(1)$. As $V_0$ is $\theta_0$-admissible, its tape word must be nonempty. The application of each rule conjugates the tape word of this sector, so that $V_0$ must again be $H$-admissible.

Suppose $\ell\leq n$. Let $\pazocal{C}$ be the reduced computation starting with $V_0$ and with history $H$. Then, let $\Gamma$ be the trapezium corresponding to $\pazocal{C}$ by Lemma \ref{computations are trapezia}. By Lemmas \ref{multiply one letter} and \ref{unreduced base}, for any maximal $\theta$-band $\pazocal{T}$ of $\Gamma$, $\|\textbf{tbot}(\pazocal{T})\|\leq\max(\|\textbf{tbot}(\Gamma)\|,\|\textbf{ttop}(\Gamma)\|)$, so that $\text{wt}(\Gamma)\leq3h\max(\|\textbf{tbot}(\Gamma)\|,\|\textbf{ttop}(\Gamma)\|)$. 

Note that $|\textbf{ttop}(\Gamma)|_a\leq|V_0|_a+2h$ and $h\leq\|H_1\|+n\|H_2\|+\|H_3\|=\eta$, so that
$$\text{wt}(\Gamma)\leq3h(\|\textbf{tbot}(\Delta)\|+2\eta)$$
Further, the bottom and side labels of $\Gamma$ are the same as those of $\Delta$, while the top labels differ only by a word $w$ from the `special' input sector with $\|w\|\leq|\textbf{ttop}(\Gamma)|_a+|\textbf{ttop}(\Delta)|_a$. Pasting $\Gamma$ and $\Delta$ along their shared contour then yields a diagram over $M_a(\textbf{M})$ with contour label $w$. By Lemma \ref{embedding}, $w$ must be an $a$-relation. So, we may paste an $a$-cell corresponding to $w$ to the top of $\Gamma$ to produce a diagram $\tilde{\Gamma}$ with the same contour label as $\Delta$ and
$$\text{wt}(\tilde{\Gamma})\leq3h(\|\textbf{tbot}(\Delta)\|+2\eta)+C_1(|\textbf{tbot}(\Delta)|_a+2\eta+|\textbf{ttop}(\Delta)|_a)^2\leq\text{wt}_G(\Delta)$$
Now suppose $\ell>n$.

Let $\pazocal{C}_1$ be the reduced computation starting with $V_0$ and having history $H_1$ and $\Gamma_1$ be the trapezium corresponding to $\pazocal{C}_1$ by Lemma \ref{computations are trapezia}. Set $V_1\equiv V_0\cdot H_1\equiv\lab(\textbf{ttop}(\Gamma_1))$.

As $V_0$ is $H$-admissible, there exists a reduced computation $\pazocal{C}_2$ starting with $V_1$ and having history $H_2^n$. For $q\in\N$ such that $\ell=qn+r$ with $0\leq r<n$, let $\Gamma_2(1),\dots,\Gamma_2(q)$ be $q$ copies of the trapezium corresponding to $\pazocal{C}_2$ by Lemma \ref{computations are trapezia}.

Let $v_1$ be the tape word of $V_1$. If the base of $\Delta$ is $P_0(1)Q_0(1)$ (or $Q_0(1)^{-1}P_0(1)^{-1}$), then the tape word written on $\textbf{ttop}(\Gamma_2(i))$ is equal to $v_1u^{-n}$ (or $u^nv_1$) in $F(\pazocal{A})$, where $u$ is the natural copy of $H_2$ over the alphabet $\pazocal{A}$. Otherwise, the base of $\Delta$ is $Q_0(1)^{-1}Q_0(1)$, so that the tape word written on $\textbf{ttop}(\Gamma_2(i))$ is equal in $F(\pazocal{A})$ to $u^nv_1u^{-n}$.

In each case, the projection of $\lab(\textbf{ttop}(\Gamma_2(i)))$ onto $F(\pazocal{A})$ is equivalent to $v_1$ in $B(\pazocal{A},n)$. In particular, for $1\leq i\leq q$ we may attach $a$-cells corresponding to $u^{\pm n}$ to $\textbf{top}(\Gamma_2(i))$ so that the top of the resulting diagram $\tilde{\Gamma}_2(i)$ is equivalent to $V_1$. By the assignment of weights, each of these $a$-cells has weight at most $C_1n^2\|H_2\|^2$.

Finally, let $\tilde{\Gamma}_2(q+1)$ be the trapezium corresponding to the reduced computation starting with $V_1$ and having history $H_2^r$.

Then, we paste the top of $\tilde{\Gamma}_2(i)$ to the bottom of $\tilde{\Gamma}_2(i+1)$ for each $1\leq i\leq q$ to form the diagram $\Gamma_2$.

Let $V_2\equiv\lab(\textbf{ttop}(\tilde{\Gamma}_2(q+1)))$. Note that $V_2$ is $\theta_2$-admissible, where $\theta_2$ is the first rule of $H_2$. So as above, $V_2$ is $H_3$-admissible, i.e there exists a reduced computation $\pazocal{C}_3$ starting with $V_2$ and having history $H_3$. Let $\Gamma_3$ be the corresponding trapezium.

Finally, we form the the diagram $\Gamma$ by pasting together $\Gamma_1$, $\Gamma_2$, and $\Gamma_3$.

Note that $\|V_1\|\leq\|V_0\|+2\|H_1\|$, while $\|V_2\|\leq\|V_1\|+2r\|H_2\|\leq\|V_0\|+2\|H_1\|+2r\|H_2\|$. So, for any maximal $\theta$-band $\pazocal{T}$ of $\Gamma$, $|\textbf{tbot}(\pazocal{T})|_a\leq|V_0|_a+2\eta$.

In particular, $|\textbf{ttop}(\Gamma)|_a\leq|V_0|_a+2\eta$ and
$$\text{wt}(\Gamma)\leq3h(\|\textbf{tbot}(\Delta)\|+2\eta)+2C_1qn^2\|H_2\|^2$$
As above, there exists an $a$-relation $w$ with $\|w\|\leq|\textbf{ttop}(\Gamma)|_a+|\textbf{ttop}(\Delta)|_a$ such that if we paste the $a$-cell corresponding to $w$ to the top of $\Gamma$, we obtain a diagram $\tilde{\Gamma}$ with $\lab(\partial\tilde{\Gamma})\equiv\lab(\partial\Delta)$ and
$$\text{wt}(\tilde{\Gamma})\leq3h(\|\textbf{tbot}(\Delta)\|+2\eta)+2C_1qn^2\|H_2\|^2+C_1(|\textbf{tbot}(\Delta)|_a+2\eta+|\textbf{ttop}(\Delta)|_a)^2$$
As $qn\|H_2\|\leq\ell\|H_2\|\leq h$, we also have $2C_1qn^2\|H_2\|^2\leq 2C_1hn\|H_2\|\leq2C_1h\eta$. Hence, taking $C_1\geq6$,
$$\text{wt}(\tilde{\Gamma})\leq3h\|\textbf{tbot}(\Delta)\|+(2C_1+6)h\eta+C_1(|\textbf{tbot}(\Delta)|_a+|\textbf{ttop}(\Delta)|_a+2\eta)^2\leq\text{wt}_G(\Delta)$$
Thus, the reduced diagram $\tilde{\Delta}$ obtained from $\tilde{\Gamma}$ by replacing any $a$-cell with the diagram constructed in Lemma \ref{disks are quadratic} (and making any necessary cancellations) satisfies the statement.

\end{proof}

\begin{lemma} \label{big G-weight}

For every big $a$-trapezium $\Delta$, there is a reduced diagram $\tilde{\Delta}$ over the finite presentation of $G(\textbf{M})$ such that $\lab(\partial\tilde{\Delta})\equiv\lab(\partial\Delta)$ and $\text{Area}(\tilde{\Delta})\leq2\text{wt}_G(\Delta)$.

\end{lemma}

\begin{proof}

As in the proof of Lemma \ref{impeding G-weight}, if $\text{wt}_G(\Delta)=\frac{1}{2}\text{wt}(\Delta)$, then we may construct $\tilde{\Delta}$ simply by replacing all $a$-cells with the corresponding diagram constructed in Lemma \ref{disks are quadratic}. Hence, it suffices to assume that $$\text{wt}_G(\Delta)=c_5h\max(\|\textbf{tbot}(\Delta)\|,\|\textbf{ttop}(\Delta)\|)+4C_1(|\textbf{tbot}(\Delta)|_a+|\textbf{ttop}(\Delta)|_a)^2$$

Without loss of generality, suppose the base of $\Delta$ begins and ends with $\{t(1)\}$. Then, let $\Delta_1,\dots,\Delta_L$ be the maximal subdiagrams of $\Delta$ bounded by the maximal $t$-bands, so that each is an $a$-trapezium with pararevolving base.

By Lemma \ref{a-cells sector}, only $\Delta_1$ may contain $a$-cells, so that $\Delta_2,\dots,\Delta_L$ are trapezia.

For $\pazocal{T}$ a maximal $\theta$-band of $\Delta$, let $\pazocal{T}_i$ be the subband which is a maximal $\theta$-band of $\Delta_i$. 

Suppose $h\leq c_3\max(\|\textbf{tbot}(\Delta_2)\|,\|\textbf{ttop}(\Delta_2)\|)$. Then, the parameter choice $c_4>>c_3$ implies $\|\textbf{tbot}(\pazocal{T}_2)\|\leq c_4\max(\|\textbf{tbot}(\Delta_2)\|,\|\textbf{ttop}(\Delta_2)\|)$.

Since $\Delta$ is big, its history contains a controlled subword. So, by Lemma \ref{M controlled}, each $\Delta_i$ must be a coordinate shift of $\Delta_2$ for $i\geq2$, so that $h\leq c_3\max(\|\textbf{tbot}(\Delta_i)\|,\|\textbf{ttop}(\Delta_i)\|)$ and $\|\textbf{tbot}(\pazocal{T}_i)\|\leq c_4\max(\|\textbf{tbot}(\Delta_i)\|,\|\textbf{ttop}(\Delta_i)\|)$. As $L>>c_3$, this implies $h\leq\max(\|\textbf{tbot}(\Delta)\|,\|\textbf{ttop}(\Delta)\|)$.

Moreover, the only sector of $\Delta_1$ that may not be a coordinate shift of the corresponding sector of $\Delta_2$ is the `special' input sector. 

For any $a$-edge $\textbf{e}$ of $\textbf{tbot}(\pazocal{T}_1)$ in the `special' input sector, let $\pazocal{B}$ be the maximal $a$-band containing $\textbf{e}$. Then Lemma \ref{a-band on same a-cell} and (MM3) imply that $\pazocal{B}$ must have one end on $\textbf{ttop}(\Delta_1)$, on $\textbf{tbot}(\Delta_1)$, or on the $q$-band corresponding to $Q_0(1)$.

So, $\|\textbf{tbot}(\pazocal{T}_1)\|\leq\|\textbf{tbot}(\pazocal{T}_2)\|+\|\textbf{tbot}(\Delta_1)\|+\|\textbf{ttop}(\Delta_1)\|+h$. It then follows that: 
\begin{align*}
\|\textbf{tbot}(\pazocal{T})\|&\leq\sum_{i=1}^L\|\textbf{tbot}(\pazocal{T}_i)\|\leq\left(\sum_{i=2}^L 2\|\textbf{tbot}(\pazocal{T}_i)\|\right)+\|\textbf{tbot}(\Delta_1)\|+\|\textbf{ttop}(\Delta_1)\|+h \\
&\leq\left(\sum_{i=2}^L 2c_4\max(\|\textbf{tbot}(\Delta_i)\|,\|\textbf{ttop}(\Delta_i)\|)\right)+\|\textbf{tbot}(\Delta_1)\|+\|\textbf{ttop}(\Delta_1)\|+h \\
&\leq 2c_4\left(\sum_{i=1}^L\|\textbf{tbot}(\Delta_i)\|+\|\textbf{ttop}(\Delta_i)\|\right)+h \\
&\leq 4c_4(\|\textbf{tbot}(\Delta)\|+\|\textbf{ttop}(\Delta)\|)+h\leq(8c_4+1)\max(\|\textbf{tbot}(\Delta)\|,\|\textbf{ttop}(\Delta)\|)
\end{align*}
So, by the parameter choice $c_5>>c_3$, the sum of the lengths, and so the weights, of the maximal $\theta$-bands of $\Delta$ is at most $c_5h\max(\|\textbf{tbot}(\Delta)\|,\|\textbf{ttop}(\Delta)\|)$.

Any other cell of $\Delta$ is an $a$-cell. As above, Lemma \ref{a-band on same a-cell} and (MM2) then imply that each maximal $a$-band starting from an $a$-cell must end on the $q$-band corresponding to $Q_0(1)$, on $\textbf{ttop}(\Delta)$, or on $\textbf{tbot}(\Delta)$. So, the sum of the combinatorial perimeters of the $a$-cells of $\Delta$ is at most $h+|\textbf{ttop}(\Delta)|_a+|\textbf{tbot}(\Delta)|_a$, i.e at most $2(\|\textbf{tbot}(\Delta)\|+\|\textbf{ttop}(\Delta)\|)$. As a result,
$$\text{wt}(\Delta)\leq c_5h\max(\|\textbf{tbot}(\Delta)\|,\|\textbf{ttop}(\Delta)\|)+4C_1(\|\textbf{tbot}(\Delta)\|+\|\textbf{ttop}(\Delta)\|)^2=\text{wt}_G(\Delta)$$
The reduced diagram $\tilde{\Delta}$ constructed from $\Delta$ by replacing any $a$-cell with the corresponding reduced diagram from Lemma \ref{disks are quadratic} then satisfies the statement.

Hence, we may assume that $h>c_3\max(\|\textbf{tbot}(\Delta_2)\|,\|\textbf{ttop}(\Delta_2)\|)$.

As $\Delta$ is big, its history must contain a controlled subword $H'$. Let $\Delta'$ be the subtrapezium whose history is $H'$.

Let $\Delta_-$ be the subdiagram of $\Delta$ obtained by removing the maximal $q$-band $\pazocal{Q}$ corresponding to the final letter of the base of $\Delta$. So, $\Delta_-$ is an $a$-trapezium with the standard base. Similarly define $\Delta_-'$ as the corresponding subdiagram of $\Delta'$. Lemma \ref{M controlled} then implies that $\lab(\textbf{tbot}(\Delta_-'))=\lab(\textbf{bot}(\Delta_-'))$ is an accepted configuration.

As $\Delta_2$ is a trapezium, Lemma \ref{trapezia are computations} yields a corresponding computation $\pazocal{C}_2:V_0\to\dots\to V_h$ with base $\{t(2)\}B_3(2)$ satisfying $h>c_3\max(\|V_0\|,\|V_h\|)$. Hence, $\pazocal{C}_2$ satisfies the hypotheses of Lemma \ref{M projected long history}, so that there exist accepted configurations $W_0'$ and $W_h'$ with $W_0'(2)\equiv V_0$ and $W_h'(2)\equiv V_h$.

Let $\Delta_0$ be the subdiagram of $\Delta_-$ which is an $a$-trapezium with $\textbf{bot}(\Delta_0)=\textbf{bot}(\Delta_-)$ and $\textbf{top}(\Delta_0)=\textbf{bot}(\Delta_-')$. Then $\Delta_0$ is an $a$-trapezium with sides labelled identically and top labelled by an accepted configuration. So, $W_0\equiv\lab(\textbf{bot}(\Delta_-))$ must be a configuration which is trivial in $G_\pazocal{S}(\textbf{M})$.

Similarly, $W_h\equiv\lab(\textbf{top}(\Delta_-))$ is a configuration which is trivial in $G_\pazocal{S}(\textbf{M})$.

By Lemma \ref{a-cells sector}, any sector of $\Delta_0$ other than the `special' input sector is a trapezium. So, since $\lab(\textbf{top}(\Delta_0))$ is an accepted configuration, the parallel nature of the rules implies that $W_0(i)$ and $W_0(j)$ are coordinate shifts of one another for $i,j\geq2$ while the corresponding coordinate shift of $W_0(i)$ differs from $W_0(1)$ only in the `special' input sector. Hence, since $W_0'$ is an accepted configuration with $W_0(2)\equiv W_0'(2)$, $W_0$ and $W_0'$ can differ only in the `special' input sector. 

Let $w_0$ and $w_0'$ be the tape words of $W_0$ and $W_0'$, respectively, in this sector. Then, since $W_0$ and $W_0'$ are each trivial over $G_\pazocal{S}(\textbf{M})$, it follows that $w_0^{-1}w_0'$ is trivial over $G_\pazocal{S}(\textbf{M})$. Lemma \ref{embedding} then implies that $w_0^{-1}w_0'$ is trivial over $B(\pazocal{A},n)$, so that it corresponds to an $a$-relation.

Let $H_0$ be the history of an accepting computation $\pazocal{C}_0$ of $W_0'$ with $\ell(\pazocal{C}_0)=\ell(W_0')$. Then using Lemmas \ref{M accepted configurations} and \ref{computations are trapezia}, we may construct a reduced diagram $\Psi_0'$ over $G(\textbf{M})$ with $\lab(\partial\Psi_0')\equiv W_0'$ consisting of one hub and a trapezium with $\text{Area}(\Psi_0')\leq c_3\|H_0\|\|W_0'\|$ ($c_3>>c_2$).

By Lemma \ref{a-cells are quadratic}, we also construct a reduced diagram $\Psi_0''$ over $G(\textbf{M})$ with $\lab(\partial\Psi_0'')\equiv w_0^{-1}w_0'$ with $\text{Area}(\Psi_0'')\leq C_1(\|w_0\|+\|w_0'\|)^2$.

Lemma \ref{accepted configuration a-length} implies that $\|w_0'\|\leq|W_0'(1)|_a\leq2|W_0'(2)|_a=2|W_0(2)|_a$. So, $$\|W_0'\|\leq\|W_0\|+\|w_0'\|\leq\|W_0\|+|W_0(2)|_a+|W_0(3)|_a\leq2\|W_0\|$$
Further, $\|w_0\|\leq|\textbf{tbot}(\Delta_1)|_a$ and, since $L\geq3$, $\|w_0'\|\leq|\textbf{tbot}(\Delta_2)|_a+|\textbf{tbot}(\Delta_3)|_a$.

Let $\Psi_0$ be the diagram obtained from pasting $\Psi_0'$ and $\Psi_0''$ along their common boundary labels. Then, $\lab(\partial\Psi_0)\equiv W_0$ and $\text{Area}(\Psi_0)\leq 2c_3\|H_0\|\|W_0\|+C_1(|W_0|_a)^2$.

Similarly, we may construct a reduced diagram $\Psi_h$ over the finite presentation of $G(\textbf{M})$ satisfying $\lab(\partial\Psi_h)\equiv W_h$ and $\text{Area}(\Psi_h)\leq2c_3\|H_h\|\|W_h\|+C_1(|W_h|_a)^2$, where $H_h$ is the history of an accepting computation $\pazocal{C}_h$ of $W_h'$ with $\ell(\pazocal{C}_h)=\ell(W_h')$.

Attaching the corresponding ends of $\pazocal{Q}$ to $\Psi_0$ and $\Psi_h$, we then obtain a reduced diagram $\tilde{\Delta}$ with $\lab(\partial\tilde{\Delta})\equiv\lab(\partial\Delta)$ and
\begin{align*}
\text{Area}(\tilde{\Delta})&\leq 2c_3(\|H_0\|\|W_0\|+\|H_h\|\|W_h\|)+C_1(|W_0|_a+|W_h|_a)^2+h \\
&\leq 2c_3(\|H_0\|+\|H_h\|)(\|\textbf{tbot}(\Delta)\|+\|\textbf{ttop}(\Delta)\|)+C_1(|\textbf{tbot}(\Delta)|_a+|\textbf{ttop}(\Delta)|_a)^2+h
\end{align*}
By Lemma \ref{M projected long history}, $\|H_0\|+\|H_h\|\leq h$. 

Thus, the parameter choice $c_5>>c_3$ implies $\text{Area}(\tilde{\Delta})\leq \text{wt}_G(\Delta)$.

\end{proof}

\subsection{Quadratic upper bound} \

Finally, we complete the proof of Theorem \ref{main theorem}.

As Lemma \ref{embedding} implies that $G(\textbf{M})$ contains an infinite torsion subgroup, the Dehn function of $G(\textbf{M})$ is at least quadratic. Thus, it suffices to prove a quadratic upper bound bound.

Let $w\in F(\pazocal{X})$ such that $w=1$ in $G(\textbf{M})$. By Lemma \ref{G isomorphic to G_a}, $w$ is also trivial over the group $G_\pazocal{S}(\textbf{M})$, so that Lemma \ref{minimal exist} yields a minimal diagram $\Delta_a$ over $G_\pazocal{S}(\textbf{M})$ with $\lab(\partial\Delta_a)\equiv w$. By Lemma \ref{contradiction}, we have
$$\text{wt}_G(\Delta_a)\leq N_4(|w|+\sigma_\lambda(\Delta_a^*))^2+N_3\mu(\Delta_a)$$
Lemma \ref{G_a design} implies that $\sigma_\lambda(\Delta_a^*)\leq C_1|w|$, while Lemma \ref{mixtures}(a) implies $\mu(\Delta_a)\leq J|w|^2$. So, as $|w|\leq\|w\|$, we can choose $N_5$ large enough so that
$$\text{wt}_G(\Delta_a)\leq\frac{1}{2}N_5\|w\|^2$$
Now, let $\textbf{P}$ be a minimal covering of $\Delta_a$ and construct the reduced diagram $\Delta$ over the canonical presentation of $G(\textbf{M})$ by:

\begin{itemize}

\item excising any impeding $a$-trapezium $P\in\textbf{P}$ and pasting in its place the reduced diagram given in Lemma \ref{impeding G-weight} with the same contour label and area at most $2\text{wt}_G(P)$

\item excising any big $a$-trapezium $P\in\textbf{P}$ and pasting in its place the reduced diagram given in Lemma \ref{big G-weight} with the same contour label and area at most $2\text{wt}_G(P)$

\item excising any disk $\Pi\in\textbf{P}$ and pasting in its place the reduced diagram given in Lemma \ref{disks are quadratic} with the same contour label and area at most $C_1|\partial\Pi|^2$

\item excising any $a$-cell $\pi\in\textbf{P}$ and pasting in its place the reduced diagram given in Lemma \ref{a-cells are quadratic} with the same contour label and area at most $C_1\|\partial\pi\|^2$

\end{itemize}

By the definition of $G$-weight, it follows that $\text{Area}(\Delta)\leq2\text{wt}_G(\Delta_a)\leq N_5\|w\|^2$.

Therefore, the Dehn function of $G(\textbf{M})$ is at most quadratic, and so the proof of Theorem \ref{main theorem} is complete.

\medskip

%%%%%%%%%%%%%%%%%%%%%%%%%%%%%%%%%%%%%%%%%%%%%%%%%%%%%%%%%%%%%%%%%

\section{Proof of Theorem \ref{distortion}}

\subsection{$g$-diagrams and $g$-minimal diagrams} \

By Lemma \ref{embedding}, every $g\in B(\pazocal{A},n)$ can be identified with an element of $G_\pazocal{S}(\textbf{M})$, namely $\varphi(g)$. For $g\in B(\pazocal{A},n)$, define $|g|_\pazocal{A}$ as the smallest number of letters comprising a word over $\pazocal{A}$ whose value in $B(\pazocal{A},n)$ is $g$.

For $g\in B(\pazocal{A},n)$, a minimal diagram $\Delta$ is called a \textit{$g$-diagram} if $\partial\Delta=\textbf{st}$, $\lab(\textbf{t})$ is a word over $\pazocal{A}$ whose value in $B(\pazocal{A},n)$ is $g^{-1}$, and $\|\textbf{t}\|=|g|_\pazocal{A}$. 

A $g$-diagram $\Delta$ is called \textit{$g$-minimal} if $|\partial\Delta|+\sigma_\lambda(\Delta)$ is minimal amongst all $g$-diagrams.

\begin{lemma} \label{distortion bound}

For $g\in B(\pazocal{A},n)$, if $\Delta$ is a $g$-minimal diagram, then $|\partial\Delta|+\sigma_\lambda(\Delta)\leq2\delta|g|_\pazocal{A}$.

\end{lemma}

\begin{proof}

Let $v$ be a word over $\pazocal{A}$ whose value in $B(\pazocal{A},n)$ is $g$ and such that $\|v\|=|g|_\pazocal{A}$.

By van Kampen's Lemma, there exists a diagram $\Phi$ over $F(\pazocal{A})$ (in which every cell is a 0-cell) with $\lab(\partial\Phi)\equiv vv^{-1}$. Viewing $\Phi$ as a diagram over $G_\pazocal{S}(\textbf{M})$, it is clear that $\Phi$ is a $g$-diagram containing no disks, so that $\sigma_\lambda(\Phi)=0$.

Hence, for any $g$-minimal diagram $\Delta$,
$$|\partial\Delta|+\sigma_\lambda(\Delta)\leq|\partial\Phi|=2\delta\|v\|=2\delta|g|_\pazocal{A}$$

\end{proof}

\begin{lemma} \label{distortion q-bands}

If $\Delta$ is a $g$-minimal diagram for some $g\in B(\pazocal{A},n)$, then no $q$-band of $\Delta$ has two ends on $\partial\Delta$.

\end{lemma}

\begin{proof}

Decompose $\partial\Delta=\textbf{st}$ as in the definition of $g$-diagram. 

Suppose $\pazocal{Q}$ is a $q$-band with two ends on $\partial\Delta$. Then, since $\textbf{t}$ consists entirely of $a$-edges, both ends of $\pazocal{Q}$ must be edges of $\textbf{s}$.

Let $\textbf{s}_0$ be the subpath of $\textbf{s}$ bounded by the two ends of $\pazocal{Q}$. So, $\textbf{s}_0$ and a side of $\pazocal{Q}$, say $\textbf{top}(\pazocal{Q})$, bound a subdiagram $\Delta_0$ of $\Delta$ containing $\pazocal{Q}$.

By Lemma \ref{lengths}(b), $|\textbf{top}(\pazocal{Q})|=\ell$, where $\ell$ is the length of $\pazocal{Q}$. 

Further, by Lemma \ref{minimal theta-annuli}, every maximal $\theta$-band of $\Delta_0$ must have two ends on $\partial\Delta_0$. Lemma \ref{M_a no annuli 1} implies that no $\theta$-band can end twice on $\textbf{top}(\pazocal{Q})$. As a result, each of the $\ell$ $\theta$-edges of $\textbf{top}(\pazocal{Q})$ correspond to a $\theta$-edge of $\textbf{s}_0$, so that Lemma \ref{lengths}(a) implies $|\textbf{s}_0|\geq\ell+2$.

Consider the diagram $\Delta'$ obtained from $\Delta$ by cutting off $\Delta_0$. As $\Delta'$ is a subdiagram of $\Delta$, it is minimal. Moreover, $\partial\Delta'$ can be decomposed as $\textbf{s}_1(\textbf{top}(\pazocal{Q}))\textbf{s}_2\textbf{t}$, where $\textbf{s}=\textbf{s}_1\textbf{s}_0\textbf{s}_2$. As a result, $\Delta'$ is a $g$-diagram.

By Lemma \ref{lengths}(c), 
\begin{align*}
|\partial\Delta'|&\leq|\textbf{s}_1|+|\textbf{top}(\pazocal{Q})|+|\textbf{s}_2|+|\textbf{t}|=|\textbf{s}_1|+\ell+|\textbf{s}_2|+|\textbf{t}| \\
&\leq|\textbf{s}_1|+|\textbf{s}_0|-2+|\textbf{s}_2|+|\textbf{t}|\leq|\partial\Delta|-(2-4\delta)
\end{align*}
A parameter choice for $\delta$ then implies $|\partial\Delta'|\leq|\partial\Delta|-1$.

Finally, as $\Delta_0$ and $\Delta'$ are disjoint, $\sigma_\lambda(\Delta)\geq\sigma_\lambda(\Delta_0)+\sigma_\lambda(\Delta')$. In particular, this implies $$|\partial\Delta'|+\sigma_\lambda(\Delta')\leq|\partial\Delta|+\sigma_\lambda(\Delta)-1$$
But this contradicts the assumption that $\Delta$ is a $g$-minimal diagram.

\end{proof}

Let $\Delta$ be a $g$-minimal diagram for some $g\in B(\pazocal{A},n)$ and decompose $\partial\Delta=\textbf{st}$ as in the definition of $g$-diagram. Suppose $\Delta$ contains a quasi-rim $\theta$-band $\pazocal{T}$. Since $\textbf{t}$ is comprised entirely of $a$-edges, $\pazocal{T}$ must end twice on $\textbf{s}$. Let $\textbf{s}_0$ be the subpath of $\partial\Delta$ bounded by the two ends of $\pazocal{T}$ such that, per the definition of quasi-rim $\theta$-band, any cell between $\textbf{bot}(\pazocal{T})$ (or $\textbf{top}(\pazocal{T})$) and $\textbf{s}_0$ is an $a$-cell. If $\textbf{s}_0$ is a subpath of $\textbf{s}$, then $\pazocal{T}$ is called a \textit{$g$-rim $\theta$-band}.

\begin{lemma} \label{distortion theta-bands}

Let $\Delta$ be a $g$-minimal diagram for some $g\in B(\pazocal{A},n)$. If $\pazocal{T}$ is an $g$-rim $\theta$-band in $\Delta$, then the base of $\pazocal{T}$ has length $s>K$.

\end{lemma}

\begin{proof}

Suppose to the contrary that the base of $\pazocal{T}$ has length $s\leq K$.

Decompose $\partial\Delta=\textbf{st}$ as in the definition of $g$-diagram and let $\textbf{s}_0$ be the subpath of $\textbf{s}$ as in the definition of $g$-rim $\theta$-band.

Suppose every cell between $\textbf{bot}(\pazocal{T})$ and $\textbf{s}_0$ is an $a$-cell. As in the proof of Lemma \ref{distortion q-bands}, let $\Delta_0$ be the subdiagram bounded by $\textbf{top}(\pazocal{T})$ and $\textbf{s}_0$ and let $\Delta'$ be the diagram obtained from $\Delta$ by cutting off $\Delta_0$.

Then, any cell of $\Delta_0$ not comprising $\pazocal{T}$ is an $a$-cell. Letting $\pi$ be such an $a$-cell, for any edge $\textbf{e}$ of $\partial\pi$, either $\textbf{e}$ is shared with $\textbf{s}_0$ or $\textbf{e}^{-1}$ is shared with $\textbf{bot}(\pazocal{T})$. Similarly, any edge of $\textbf{bot}(\pazocal{T})$ is either shared with $\textbf{s}_0$ or its inverse is on the boundary of an $a$-cell in $\Delta_0$.

For an $a$-cell $\pi$ in $\Delta_0$, let $\partial\pi=\textbf{p}_\pi\textbf{q}_\pi$, where $\textbf{p}_\pi$ is a maximal subpath shared with $\partial\Delta$. By (M1), at most $\frac{1}{2}\|\partial\pi\|$ edges of $\partial\pi$ are shared with the boundary of a $(\theta,a)$-cell of $\pazocal{T}$. So, $\|\textbf{q}_\pi\|\leq\frac{1}{2}\|\partial\pi\|+b_\pi$, where $b_\pi$ is the number of edges of $\partial\pi$ shared with the boundary of a $(\theta,q)$-cell of $\pazocal{T}$. As a result, $\|\textbf{p}_\pi\|\geq\frac{1}{2}\|\partial\pi\|-b_\pi\geq\|\textbf{q}_\pi\|-2b_\pi$.

Since $\pazocal{T}$ contains $s$ $(\theta,q)$-cells, the boundary of any of which contains at most one $a$-edge labelled by a letter from the `special' input sector, $\sum b_\pi\leq s$. Hence, $|\textbf{s}_0|_a\geq|\textbf{bot}(\pazocal{T})|_a-2s\geq|\textbf{bot}(\pazocal{T})|_a-2K$.

As every $q$-edge of $\textbf{bot}(\pazocal{T})$ is shared with $\textbf{s}_0$, it then follows from Lemma \ref{lengths} that 
\begin{align*}
|\textbf{s}_0|&\geq2+|\textbf{s}_0|_q+\delta(|\textbf{s}_0|_a-2)\geq2-2\delta+|\textbf{bot}(\pazocal{T})|_q+\delta(|\textbf{bot}(\pazocal{T})|_a-2K) \\
&\geq2+|\textbf{bot}(\pazocal{T})|-(2K+2)\delta
\end{align*}

Further, by Lemma \ref{simplify rules}, $|\textbf{top}(\pazocal{T})|_a\leq|\textbf{bot}(\pazocal{T})|_a+2s\leq|\textbf{bot}(\pazocal{T})|_a+2K$.

Thus, $|\textbf{s}_0|\geq2-(2K+2)\delta+|\textbf{top}(\pazocal{T})|-2K\delta\geq|\textbf{top}(\pazocal{T})|+2-(4K+2)\delta\geq|\textbf{top}(\pazocal{T})|+1$ by the parameter choice $\delta^{-1}>>K$.

But then Lemma \ref{lengths} and a parameter choice for $\delta$ implies that $\Delta'$ is a $g$-diagram satisfying $|\partial\Delta'|+\sigma_\lambda(\Delta')<|\partial\Delta|+\sigma_\lambda(\Delta)$, contradicting the assumption that $\Delta$ is $g$-minimal.

\end{proof}

\begin{lemma} \label{diskless distortion}

Let $\Delta$ be a $g$-minimal diagram for some $g\in B(\pazocal{A},n)$ and decompose $\partial\Delta=\textbf{st}$ as in the definition of $g$-diagram. If $\Delta$ contains no disks, then $|\textbf{s}|=|\textbf{t}|=\delta|g|_\pazocal{A}$.

\end{lemma}

\begin{proof}

By the definition of the design, $\sigma_\lambda(\Delta)=0$. Further, Lemma \ref{distortion q-bands} implies that $\Delta$ contains no $q$-bands, so that Lemma \ref{M_a no annuli 1} implies that $\partial\Delta$ contains no $q$-edges. 

If $\Delta$ contains a $\theta$-band, then Lemma \ref{M_a no annuli 2} implies that $\Delta$ contains a $g$-rim $\theta$-band $\pazocal{T}$. But then $\pazocal{T}$ has base of length zero, contradicting Lemma \ref{distortion theta-bands}.

Hence, any cell of $\Delta$ is an $a$-cell and every edge of $\partial\Delta$ is an $a$-edge.

Suppose $\textbf{e}$ is an edge of $\textbf{s}$ that is not on the boundary of an $a$-cell. Then the maximal $a$-band of $\Delta$ starting at $\textbf{e}$ must be of length zero, i.e $\textbf{e}^{-1}$ is an edge of $\partial\Delta$. If $\textbf{e}^{-1}$ is part of $\textbf{s}$, then deleting the subpath of $\textbf{s}$ bounded by $\textbf{e}$ and $\textbf{e}^{-1}$ results in a $g$-diagram with smaller perimeter, contradicting the assumption that $\Delta$ is $g$-minimal. So, $\textbf{e}^{-1}$ must be an edge of $\textbf{t}$. 

As a result, $\lab(\partial\Delta)$ must be a word over $\pazocal{A}$.

Now, as in the proof of Lemma \ref{a-cells are quadratic}, excise any $a$-cell of $\Delta$ and paste in its place an appropriate reduced diagram over the presentation $\gen{\pazocal{A}\mid\pazocal{R}}$ of $B(\pazocal{A},n)$. This produces a reduced diagram $\Psi$ over $B(\pazocal{A},n)$ with $\lab(\partial\Psi)\equiv\lab(\partial\Delta)$. Hence, $\lab(\textbf{s})=g$ in $B(\pazocal{A},n)$.

By the definition of the word norm, this implies that $|\textbf{s}|=\delta\|\textbf{s}\|\geq\delta|g|_\pazocal{A}$.

Hence, $|\partial\Delta|=|\textbf{s}|+|\textbf{t}|\geq2\delta|g|_\pazocal{A}$.

But then Lemma \ref{distortion bound} implies $|\partial\Delta|=2\delta|g|_\pazocal{A}$, so that $|\textbf{s}|=\delta|g|_\pazocal{A}$.

\end{proof}

\subsection{$g$-minimal diagrams containing disks} \

\begin{lemma} \label{distortion graph}

Let $\Delta$ be a $g$-minimal diagram for some $g\in B(\pazocal{A},n)$ containing at least one disk. Decompose $\partial\Delta=\textbf{st}$ as in the definition of $g$-diagram. Then $\Delta$ contains a disk $\Pi$ such that:

\begin{enumerate}[label=({\alph*})]

\item $L-6$ consecutive $t$-spokes $\pazocal{Q}_1,\dots,\pazocal{Q}_{L-6}$ of $\Pi$ end on $\partial\Delta$

\item for $i=1,\dots,L-7$, the subdiagram $\Psi_{i,i+1}$ of $\Delta$ bounded by $\pazocal{Q}_i$, $\pazocal{Q}_{i+1}$, $\partial\Pi$, and $\partial\Delta$ contains no disks, and

\item $\textbf{t}$ is not a subpath of $\partial\Psi_{i,i+1}$ for any $i=1,\dots,L-7$.

\end{enumerate}

\end{lemma}

\begin{proof}

Let $\Pi_1$ be a disk in $\Delta$ guaranteed by Lemma \ref{graph}. As $\textbf{t}$ consists entirely of $a$-edges, every $t$-spoke of $\Pi$ ending on $\partial\Delta$ must end on $\textbf{s}$.

Let $\Psi_1$ be the subdiagram of $\Delta$ bounded by $\partial\Pi_1$, $\partial\Delta$, and the $L-4$ consecutive $t$-spokes of $\Pi_1$ ending on $\partial\Delta$. Then, one may assume that $\partial\Psi_1$ contains $\textbf{t}$ as a subpath, as otherwise choosing $\Pi=\Pi_1$ satisfies the statement. Similarly, one may assume that the complement of $\Psi_1\cup\Pi_1$ in $\Delta$, $\Psi_1'$, contains a disk, as otherwise $\Pi_1$ is the only disk in $\Delta$.

Then, decompose $\partial\Psi_1'=\textbf{s}_1\textbf{t}_1$ where $\textbf{s}_1$ is the maximal subpath of $\textbf{s}$.

Next, apply Lemma \ref{graph} to $\Psi_1'$, yielding a disk $\Pi_2$. Let $\Psi_2$ be the subdiagram of $\Delta$ bounded by $\partial\Pi_2$, $\partial\Psi_2$, and the $L-4$ consecutive $t$-spokes of $\Pi_2$ ending on $\partial\Psi_1'$.

Suppose $\partial\Psi_2$ does not contain $\textbf{t}_1$. Then, by Lemma \ref{t-spokes between disks}, at most two of the $L-4$ $t$-spokes of $\Pi_2$ ending on $\partial\Psi_1'$ end on $\textbf{t}_1$, in which case such a spoke is the first or last in the sequence. As such, choosing $\Pi=\Pi_2$ satisfies the statement. So, one may assume that $\partial\Psi_2$ contains $\textbf{t}_1$. 

Similarly, one may assume that the complement of $\Psi_2\cup\Pi_2$ in $\Psi_1'$, $\Psi_2'$, contains a disk, as otherwise $\Pi_2$ is the only disk in $\Psi_1'$.

Now decompose $\partial\Psi_2'=\textbf{s}_2\textbf{t}_2$ where $\textbf{s}_2$ is a maximal subpath of $\textbf{s}$ and apply Lemma \ref{graph} to $\Psi_2'$, yielding a disk $\Pi_3$.

Continuing in this way, the finiteness of $\Delta$ implies that the process must terminate. Hence, there exists $\ell$ such that $\Pi_\ell$ satisfies the statement.

\end{proof}

The goal throughout the rest of this subsection is to prove that for any $g\in B(\pazocal{A},n)$, a $g$-minimal diagram must be diskless. This is done by arguing toward contradiction in much the same way as proof presented in Section 10.

As such, we adopt much of the same notation of Section 10 for a $g$-minimal diagram $\Delta$ containing a disk $\Pi$ satisfying Lemma \ref{distortion graph}. So, for $1\leq i<j\leq L-6$, we define the the subdiagrams $\Psi_{ij}$, $\Psi_{ij}'$, and $\bar{\Delta}_{ij}$ as well as the paths $\textbf{p}_{ij}$ and $\bar{\textbf{p}}_{ij}$. As in Section 10, the subscripts are suppressed in the case that $j-i$ is maximal; in other words, $\Psi=\Psi_{1,L-6}$, $\textbf{p}=\textbf{p}_{1,L-6}$, etc.

Further, for each $i\in\{1,\dots,L-6\}$, define $H_i$ as the history of the $t$-spoke $\pazocal{Q}_i$ and define $h_i=\|H_i\|$.

Finally, let $W$ be the accepted configuration of $\textbf{M}$ corresponding to $\lab(\partial\Pi)$ and let $V$ be the accepted configuration of $\textbf{M}_4$ such that $W(i)$ is a copy of $V$ for each $i\geq2$.

Note that Lemma \ref{distortion q-bands} can function as an analogue of Lemma \ref{7.21} in this setting. The following statement can similarly be viewed as an analogue of Lemma \ref{7.22}.

\begin{lemma} \label{distortion r} \

\begin{enumerate}[label=({\arabic*})]

\item Every maximal $\theta$-band of $\Psi$ crosses either $\pazocal{Q}_1$ or $\pazocal{Q}_{L-6}$.

\item There exists an $r$ satisfying $(L-1)/2-5\leq r\leq (L-1)/2$ such that the $\theta$-bands of $\Psi$ crossing $\pazocal{Q}_{L-6}$ do not cross $\pazocal{Q}_r$ and the $\theta$-bands of $\Psi$ crossing $\pazocal{Q}_1$ do not cross $\pazocal{Q}_{r+1}$.

\end{enumerate}

\end{lemma}

\begin{proof}

(1) Suppose there exists a $\theta$-band $\pazocal{T}$ of $\Psi$ not crossing either $\pazocal{Q}_1$ or $\pazocal{Q}_{L-6}$. Then $\pazocal{T}$ must have both ends on $\textbf{p}$ which is a subpath of $\textbf{s}$. So, $\pazocal{T}$ is a maximal $\theta$-band of $\Delta$.

As $\Psi$ is diskless, perhaps passing to $\theta$-bands contained in the subdiagram bounded by $\pazocal{T}$ and $\textbf{p}$, it is no loss of generality to assume that $\pazocal{T}$ is a quasi-rim $\theta$-band. By definition, this means that $\pazocal{T}$ is a $g$-rim $\theta$-band.

By Lemma \ref{distortion q-bands}, every $q$-band crossing $\pazocal{T}$ must be a spoke of $\Pi$. But then the base of $\pazocal{T}$ has length at most $11(L-7)+1<K$ by the parameter choice $K>>L$, contradicting Lemma \ref{distortion theta-bands}.

(2) This is proved in much the same way as Lemma \ref{7.22}(2). Note that the difference in the bounds is due to the difference in the number of relevant $t$-spokes.

\end{proof}

%The next statement is the analogue of Lemma \ref{7.23} in this setting.
%
%\begin{lemma} \label{distortion q-length}
%
%For $1\leq i\leq L-7$, $|\textbf{p}_{i,i+1}|_q<3K_0$.
%
%\end{lemma}
%
%\begin{proof}
%
%By Lemma \ref{distortion q-bands}, any maximal $q$-band in $\Psi_{i,i+1}$ must connect $\textbf{p}_{i,i+1}$ to $\Pi$. But then the statement follows from a parameter choice for $K_0$.
%
%\end{proof}

%As in Section 10, for $W$ the accepted configuration corresponding to $\lab(\partial\Pi)$, each component $W(i)$ of $W$ for $i\geq2$ is a copy of the same configuration $V$ of $\textbf{M}_4$. Lemma \ref{accepted configuration a-length} then implies $|W(1)|_a\leq2|V|_a$.

By Lemma \ref{distortion r}(2), it follows that, as in Section 10,
$$h_1\geq h_2\geq\dots\geq h_{r-1}\geq h_r \ ;$$ 
$$h_{r+1}\leq h_{r+2}\leq\dots\leq h_{L-7}\leq h_{L-6}$$

The next statement is the analogue of Lemma \ref{7.24} in this setting. Its proof is an exact copy of the one provided in Section 10.

\begin{lemma} \label{distortion path lengths} \

\begin{enumerate}[label=({\arabic*})]

\item If $i\leq r$ and $j\geq r+1$, then 
\begin{align*}
|\textbf{p}_{ij}|&\geq|\textbf{p}_{ij}|_\theta+|\textbf{p}_{ij}|_q+\delta(|\textbf{p}_{ij}|_a-|\textbf{p}_{ij}|_\theta) \\
&\geq h_i+h_j+11(j-i)+\delta(|\textbf{p}_{ij}|_a-h_i-h_j)+1
\end{align*}

\item $|\bar{\textbf{p}}_{ij}|\leq h_i+h_j+11(L-j+i)+(L-j+i+1)\delta|V|_a-1$.

\end{enumerate}

\end{lemma}

The following statement is the analogue of Lemma \ref{7.26} in this setting.

\newpage

\begin{lemma} \label{distortion eps}

For any $1\leq i<j\leq L-6$, $|\textbf{p}_{ij}|+\sigma_\lambda(\bar{\Delta}_{ij})\leq|\textbf{p}_{ij}|+\sigma_\lambda(\Delta)-\sigma_\lambda(\Psi_{ij}')\leq
|\bar{\textbf{p}}_{ij}|$.

\end{lemma}

\begin{proof}

As $\bar{\Delta}_{ij}$ is the complement of $\Psi_{ij}'$ in $\Delta$, $\sigma_\lambda(\Psi_{ij}')+\sigma_\lambda(\bar{\Delta}_{ij})\leq\sigma_\lambda(\Delta)$. Hence, it suffices to show that $|\textbf{p}_{ij}|+\sigma_\lambda(\Delta)\leq|\bar{\textbf{p}}_{ij}|+\sigma_\lambda(\Psi_{ij}')$.

%So, since $\Delta$ is $h$-minimal, $|\partial\Delta|+\sigma_\lambda(\Delta)\leq|\partial\Psi_{ij}'|+\sigma_\lambda(\Psi_{ij}')\leq|\partial\Psi_{ij}'|+\sigma_\lambda(\Delta)-\sigma_\lambda(\bar{\Delta}_{ij})$.

Let $\textbf{p}_{ij}'$ be the complement of $\textbf{p}_{ij}$ in $\partial\Delta$. As $\textbf{p}_{ij}$ starts and ends with $t$-edges, $|\partial\Delta|=|\textbf{p}_{ij}|+|\textbf{p}_{ij}'|$.

Further, since $\partial\Psi_{ij}'=\bar{\textbf{p}}_{ij}\textbf{p}_{ij}'$, Lemma \ref{lengths}(c) implies $|\partial\Psi_{ij}'|\leq|\bar{\textbf{p}}_{ij}|+|\textbf{p}_{ij}'|$.

Since $\partial\Psi_{ij}'$ contains $\textbf{t}$ as a subpath, $\Psi_{ij}'$ is a $g$-diagram. But then $\Delta$ is $g$-minimal, so that $|\partial\Delta|+\sigma_\lambda(\Delta)\leq|\partial\Psi_{ij}'|+\sigma_\lambda(\Psi_{ij}')$. Combining these inequalities then yields
$$|\textbf{p}_{ij}|+|\textbf{p}_{ij}'|+\sigma_\lambda(\Delta)\leq|\bar{\textbf{p}}_{ij}|+|\textbf{p}_{ij}'|+\sigma_\lambda(\Psi_{ij}')$$
yielding the desired inequality.

\end{proof}

\begin{lemma} \label{distortion |V|}

The inequality $\displaystyle|V|_a>\frac{11L}{16\delta}$ must be true.

\end{lemma}

\begin{proof}

As $\sigma_\lambda(\bar{\Delta})\geq0$, Lemma \ref{distortion eps} implies that $|\textbf{p}|\leq|\bar{\textbf{p}}|$. 

So, by Lemma \ref{distortion path lengths},
$$h_1+h_{L-6}+11(L-7)\leq h_1+h_{L-6}+11(7)+8\delta|V|_a$$
As a result, $\displaystyle|V|_a\geq\frac{11(L-14)}{8\delta}>\frac{11L}{16\delta}$ by a parameter choice for $L$.

\end{proof}

Note that the previous statement serves as an analogue (though a strengthened version) of Lemma \ref{7.33}.

\begin{lemma} \label{distortion rules}

Let $\pazocal{T}$ (resp $\pazocal{T}'$) be the maximal $\theta$-band of $\Psi$ crossing $\pazocal{Q}_1$ (resp $\pazocal{Q}_{L-6}$) closest to $\Pi$. Let $\ell$ and $\ell'$ be the maximal integers for which $\pazocal{T}$ crosses $\pazocal{Q}_1,\dots,\pazocal{Q}_\ell$ and $\pazocal{T}'$ crosses $\pazocal{Q}_{L-5-\ell'},\dots,\pazocal{Q}_{L-6}$. Then:

\begin{enumerate}[label=({\arabic*})]

\item $\ell+\ell'>L-L_0$

\item $H_1$ and $H_{L-6}$ have different first letters.

\end{enumerate}

\end{lemma}

\begin{proof}

(1) By Lemma \ref{distortion theta-bands}, for $\ell\leq i\leq L-6-\ell'$, $\Psi_{i,i+1}$ contains no cells. In particular, $\textbf{p}_{i,i+1}$ is shared with $\partial\Pi$, and so contains $|V|_a$ $a$-edges and no $\theta$-edges. Applying Lemma \ref{lengths}, this implies
$$|\textbf{p}|\geq h_1+h_{L-6}+11(L-7)+(L-6-\ell-\ell')\delta|V|_a$$
By Lemma \ref{distortion path lengths}(2), $|\bar{\textbf{p}}|\leq h_1+h_{L-6}+11(7)+8\delta|V|_a$.

By Lemma \ref{distortion eps}, these inequalities yield $$11(L-14)+(L-14)\delta|V|_a\leq(\ell+\ell')\delta|V|_a$$
As $L>0$, this necessitates $|V|_a>0$. But then $\ell+\ell'>L-14$, so that a parameter choice for $L_0$ implies the statement.

(2) The parameter choice $L>>L_0$ allows one to assume that $L-L_0>(L-1)/2$. But then if the statement were false, then (1) would contradict Lemma \ref{G_a theta-annuli}(2).

\end{proof}

Note that Lemma \ref{distortion rules}(2) is the natural analogue of Lemma \ref{7.32} in this setting.

Since Lemma \ref{G_a theta-annuli}(1) implies that $\ell,\ell'\leq (L-1)/2$, it follows from Lemma \ref{distortion rules}(1) and the parameter choice $L>>L_0$ that $\ell,\ell'>L_0$.

The distinguished clove $\Psi_{i,i+1}$ is defined here in the analogous way as it is in Section 10, as are its subdiagrams $\Lambda_{i,i+1}'$ and $\Lambda_{i,i+1}''$.

Further, for $i\in[1,r-1]\cup[r+1,L-7]$, define the trapezium $\Gamma_i$, the comb $E_i$, and the paths $\textbf{y}_i$ and $\textbf{z}_i$ in the same way they were defined in Section 10.5.

\begin{lemma} \label{distortion a-cells}

For each $i$, the comb $E_i$ contains no $a$-cell.

\end{lemma}

\begin{proof}

Suppose $\Psi_{i,i+1}$ is not the distinguished clove. 

By Lemma \ref{distortion q-bands}, every $q$-band of $E_i$ corresponds to a spoke of $\Pi$ contained in $\Psi_{i,i+1}$. So, by the structure of the relations, every edge of any $a$-cell of $E_i$ must be shared with $\textbf{p}_{i,i+1}$. But then cutting this $a$-cell off of $\Delta$ produces a $g$-diagram with smaller inductive parameter, contradicting the assumption that $\Delta$ is $g$-minimal.

If $\Psi_{i,i+1}$ is the distinguished clove, then the same argument applies, as the spokes of $\Pi$ contained in $\Lambda_{i,i+1}'$ cannot correspond to the `special' input sector.

\end{proof}

The following two statements are the analogues of Lemmas \ref{7.29} and \ref{7.30} and are proved in exactly the same way.

\begin{lemma} \label{distortion combs}

For $i\in\{2,\dots,r-1\}$, suppose a maximal $a$-band $\pazocal{B}$ of $E_i$ starts on $\textbf{z}_i$ and ends on a side of a maximal $q$-band $\pazocal{C}$. Let $\nabla$ be the comb bounded by $\pazocal{B}$, a part of $\pazocal{C}$, and a subpath $\textbf{x}$ of $\textbf{z}_i$. Then there is a copy of the comb $\nabla$ in the trapezium $\Gamma=\Gamma_{i-1}\setminus\Gamma_i'$, where $\Gamma_i'$ is the natural copy of $\Gamma_i$ in $\Gamma_{i-1}$.

\end{lemma}

\begin{lemma} \label{distortion a-bands}

At most 6 $a$-bands starting on the path $\textbf{y}_i$ (or $\textbf{z}_i$) can end on $(\theta,q)$-cells of the same $\theta$-band.

\end{lemma}

\begin{lemma} \label{distortion big z_i}

There exist no two indices $i\in\{1,\dots,r-1\}$ and $j\in\{r+1,\dots,L-7\}$ such that $|\textbf{z}_i|_a,|\textbf{z}_j|_a<|V|_a/8c_3$.

\end{lemma}

\begin{proof}

Suppose neither $\Psi_{i,i+1}$ nor $\Psi_{j,j+1}$ is the distinguished clove.

As in the proof of Lemma \ref{7.34}, we may use Lemma \ref{distortion rules}(2) to construct a trapezium $E$ with history $H_j^{-1}H_{i+1}$ by pasting the mirror of a coordinate shift of $\Gamma_j$ to $\Gamma_i$. Without loss of generality, suppose $h_{i+1}\geq h_j$.

Note that $|V|_a-|V|_a/8c_3>|V|_a/2$, so that $h_{i+1},h_j>|V|_a/8$ since any rule of $\textbf{M}_4$ alters the $a$-length of a configuration by at most four. By Lemma \ref{distortion |V|}, $|V|_a/8>\frac{11L}{128\delta}\geq12c_3$ since $\delta^{-1}>>L>>c_3$. Further, for $t$ the height of $E$, $W_0\equiv\textbf{tbot}(E)$, and $W_t\equiv\textbf{ttop}(E)$, $|V|_a/8>c_3\max(|W_0|_a,|W_t|_a)$. So, $t>c_3\max(\|W_0\|,\|W_t\|)$.

Taking $\lambda<1/10$, Lemmas \ref{trapezia are computations} and \ref{M projected long history controlled} imply that $\pazocal{Q}_i$ contains a $\lambda$-shaft of $\Pi$ of length at least $h_{i+1}$. So, Lemma \ref{distortion path lengths}(1) yields the inequality $|\textbf{p}_{i,j}|+\sigma_\lambda(\bar{\Delta}_{i,j})\geq h_i+h_j+h_{i+1}$.

By Lemma \ref{distortion |V|} and a parameter choice for $\delta$, $11L<16\delta|V|_a<128\delta h_{i+1}\leq\frac{1}{4}h_{i+1}$. Similarly, $(L+1)\delta|V|_a\leq8\delta h_{i+1}<\frac{1}{4}h_{i+1}$. Hence, $|\bar{\textbf{p}}_{i,j}|<h_i+h_j+\frac{1}{2}h_{i+1}$.

But then $|\bar{\textbf{p}}_{i,j}|<|\textbf{p}_{i,j}|+\sigma_\lambda(\bar{\Delta}_{i,j})$, contradicting Lemma \ref{distortion eps}.

Now suppose either $\Psi_{i,i+1}$ or $\Psi_{j,j+1}$ is the distinguished clove. Then letting $h_{i+1}'$ and $h_j'$ be the heights of $\Gamma_i$ and $\Gamma_j$, respectively, we may construct the trapezium $E$ as above. Assuming $h_{i+1}'\geq h_j'$, then the same arguments as above imply that $\pazocal{Q}_i$ contains a $\lambda$-shaft of $\Pi$ of length at least $h_{i+1}'\geq h_{i+1}$. But then a contradiction is reached as above.

\end{proof}

As a result of Lemma \ref{distortion big z_i}, we may assume without loss of generality that for all $i=1,\dots,r-1$, $|\textbf{z}_i|_a\geq|V|_a/8c_3$.

\begin{lemma} \label{distortion number of cloves}

Let $I$ be the subset of the indices $\{1,\dots,r-1\}$ defined by the property that for any $i\in I$, $h_i-h_{i+1}\leq|\textbf{z}_i|_a/8$. Then $\#I\leq c_4$.

\end{lemma}

\begin{proof}

Note that the length of the handle $E_i$ is at most $h_i-h_{i+1}$. So, by Lemma \ref{distortion a-bands}, at most $6(h_i-h_{i+1})$ $a$-bands beginning on $\textbf{z}_i$ can end on a $(\theta,q)$-cell of $E_i$. 

Since $\textbf{p}_{i,i+1}$ consists of exactly $h_i-h_{i+1}$ $\theta$-edges, for any $i\in I$, at least $|\textbf{z}_i|_a-7(h_i-h_{i+1})\geq|\textbf{z}_i|_a/8$ $a$-edges contribute $\delta$ to $|\textbf{p}_{i,i+1}|$, and so to $|\textbf{p}|$.

As $|\textbf{z}_i|_a\geq|V|_a/8c_3$ for all $i\in I$, it follows that at least $|V|_a/64c_3$ $a$-edges contribute $\delta$ to $|\textbf{p}|$.

Suppose $\#I>c_4$. Then, by Lemma \ref{distortion path lengths} and the parameter choice $c_4>>c_3$,
$$|\textbf{p}|\geq h_1+h_{L-6}+11(L-7)+8\delta|V|_a$$
$$|\bar{\textbf{p}}|\leq h_1+h_{L-6}+11(7)+8\delta|V|_a$$
But then $|\bar{\textbf{p}}|<|\textbf{p}|$, contradicting Lemma \ref{distortion eps}.

\end{proof}

\begin{lemma} \label{distortion first L_0}

For $i=1,\dots,L_0$, $h_i> L_0^2|V|_a$.

\end{lemma}

\begin{proof}

As $h_i\geq h_{i+1}\geq\dots\geq h_{r-1}\geq h_r$, we have $h_i\geq h_i-h_{r}=\sum\limits_{j=i}^{r-1}(h_j-h_{j+1})$.

By Lemma \ref{distortion number of cloves}, the number of $j\in\{i,\dots,r-1\}$ such that $j\notin I$ is at least $r-1+i-c_4$. Since $r\geq(L-1)/2-5$ by Lemma \ref{distortion r}, the parameter choices $L>>L_0>>c_4$ imply that $r-1+i-c_4>L_0^3$.

As a result, $h_i>L_0^3|\textbf{z}_i|_a/8c_3>L_0^2|V|_a$ by the parameter choice $L_0>>c_3$.

\end{proof}

\begin{lemma} \label{distortion no shafts}

For $i=1,\dots,L_0$, $\pazocal{Q}_i$ contains no $\lambda$-shaft of length at least $h_{L_0}$.

\end{lemma}

\begin{proof}

Assuming to the contrary that $\pazocal{Q}_i$ contains such a $\lambda$-shaft, Lemma \ref{distortion first L_0} implies $$\sigma_\lambda(\Delta)-\sigma_\lambda(\Psi_{i,L-6}')\geq h_{L_0}\geq L_0|V|_a$$
Then, Lemma \ref{distortion path lengths} and the parameter choices $L>>L_0$ yield the inequalities 
\begin{align*}
|\textbf{p}_{i,L-6}|+\sigma_\lambda(\Delta)-\sigma_\lambda(\Psi_{i,L-6}')&>h_i+h_{L-6}+11(L-i-6)+L_0|V|_a \\
&>h_i+h_{L-6}+11L/2+L_0|V|_a
\end{align*}
\begin{align*}
|\bar{\textbf{p}}_{i,L-6}|&\leq h_i+h_{L-6}+11(i+6)+(i+7)\delta|V|_a \\
&<h_i+h_{L-6}+11(2L_0)+2L_0\delta|V|_a
\end{align*}
But then a parameter choice for $\delta$ implies that $|\textbf{p}_{i+1,L-6}|+\sigma_\lambda(\Delta)-\sigma_\lambda(\Psi_{i+1,L-6}')>|\bar{\textbf{p}}_{i+1,L-6}|$, contradicting Lemma \ref{distortion eps}.

\end{proof}

\newpage

\begin{lemma} \label{distortion big h_i+1}

For all $i=1,\dots,L_0-1$, $h_{i+1}\leq 2c_3|\textbf{z}_i|_a$.

\end{lemma}

\begin{proof}

Assume toward contradiction that $h_{i+1}>2c_3|\textbf{z}_i|_a$.

Then, for the computation $\pazocal{C}:W_0\to\dots\to W_t$ corresponding to the trapezium $\Gamma_i$ through Lemma \ref{trapezia are computations}, $t>2c_3|W_t|_a$. Further, Lemma \ref{distortion first L_0} implies $t>L_0|V|_a$. The parameter choice $L_0>>c_3$ then implies $t>2c_3\max(|W_0|_a,|W_t|_a)$.

Lemma \ref{distortion |V|} then implies that $t\geq|V|_a\geq11L/16\delta$, so that the parameter choices $\delta^{-1}>>L>>c_3$ implies $t>c_3\max(\|W_0\|,\|W_t\|)$.

As in the proof of Lemma \ref{distortion big z_i}, this implies that $\pazocal{Q}_i$ contains a $\lambda$-shaft of length at least $h_{i+1}$. But this contradicts Lemma \ref{distortion no shafts}.

%, so that Lemma \ref{distortion path lengths} yields $|\textbf{p}_{i+1,j}|+\sigma_\lambda(\bar{\Delta}_{i+1,j})\geq2h_{i+1}+h_j$.

%Lemma \ref{distortion |V|} again implies $11L<16\delta|V|_a<8c_3\delta|\textbf{z}_i|_a<\frac{8c_3\delta}{L_0}h_{i+1}$, so that the parameter choices $\delta^{-1}>>L_0>>c_3$ imply $11L\leq\frac{1}{4}h_{i+1}$.
%
%Similarly, $(L+1)\delta|V|_a\leq8c_3(L+1)\delta|\textbf{z}_i|_a<\frac{8c_3(L+1)\delta}{L_0}h_{i+1}\leq\frac{1}{4}h_{i+1}$ by the parameter choices $\delta^{-1}>>L>>L_0>>c_3$.
%
%Hence, Lemma \ref{distortion path lengths} implies $|\bar{\textbf{p}}_{i+1,j}|<\frac{3}{2}h_{i+1}+h_j$.
%
%But then again $|\bar{\textbf{p}}_{i+1,j}|<|\textbf{p}_{i+1,j}|+\sigma_\lambda(\bar{\Delta}_{i+1,j})$, contradicting Lemma \ref{distortion eps}.

\end{proof}

\begin{lemma} \label{distortion h_i+1}

For $i=1,\dots,L_0-1$, $h_{i+1}<\left(1-\frac{1}{30c_3}\right)h_i$.

\end{lemma}

\begin{proof}

Assuming the statement is false, $h_i-h_{i+1}\leq h_i/30c_3$. By Lemma \ref{distortion a-bands}, at most $h_i/5c_3$ maximal $a$-bands of $E_i$ starting on $\textbf{z}_i$ can end on $(\theta,q)$-cells. So, at least $\max(0,|\textbf{z}_i|_a-h_i/5c_3)$ of these bands end on $\textbf{p}_{i,i+1}$.

Lemma \ref{distortion big h_i+1} implies $|\textbf{z}_i|_a\geq h_{i+1}/2c_3$, so that $|\textbf{z}_i|_a-h_i/5c_3\geq h_i/15c_3$.

By Lemma \ref{distortion r}, $\textbf{p}_{i,i+1}$ has $h_i-h_{i+1}\leq h_i/30c_3$ $\theta$-edges. So, Lemma \ref{distortion path lengths}(1) implies
$$|\textbf{p}_{i,i+1}|\geq h_i-h_{i+1}+12+\delta h_i/30c_3$$
$$|\textbf{p}_{i+1,L-6}|\geq h_{i+1}+h_{L-6}+11(L-i-7)$$
As in the proof of Lemma \ref{7.38}, this implies
\begin{align*}
|\textbf{p}_{i,L-6}|&=|\textbf{p}_{i,i+1}|+|\textbf{p}_{i+1,L-6}|-1 \\
&\geq h_i+h_{L-6}+11(L-i-6)+\delta h_i/30c_3
\end{align*}
Meanwhile, Lemma \ref{distortion path lengths} implies
$$|\bar{\textbf{p}}_{i,L-6}|\leq h_i+h_{L-6}+11(i+6)+(i+7)\delta|V|_a$$
Lemma \ref{distortion first L_0} yields $|V|_a<h_i/L_0^2$, so that
$$|\bar{\textbf{p}}_{i,L-6}|\leq h_i+h_{L-6}+11(2L_0)+2\delta h_i/L_0$$
But then the parameter choices $L>>L_0>>c_3$ imply $|\bar{\textbf{p}}_{i,L-L_0-6}|<|\textbf{p}_{i,L-L_0-6}|$, contradicting Lemma \ref{distortion eps}.

\end{proof}

%The following statement is the analogue of Lemma \ref{upper bound on z_i} in this setting.
%
%\begin{lemma} \label{distortion z_i}
%
%For $i=1,\dots,L_0-1$, $|\textbf{z}_i|_a\leq8h_i$.
%
%\end{lemma}
%
%\begin{proof}
%
%By Lemma \ref{distortion a-bands}, at most $6h_i$ maximal $a$-bands of $E_i$ starting on $\textbf{z}_i$ can end on $(\theta,q)$-cells of $E_i$. So, assuming the statement is false, Lemma \ref{distortion a-cells} implies that more than $2h_i$ maximal $a$-bands of $E_i$ starting on $\textbf{z}_i$ must end on $\textbf{p}_{i,i+1}$.
%
%As Lemma \ref{distortion r}(2) says that $\textbf{p}_{i,i+1}$ contains at most $h_i$ $\theta$-edges, Lemma \ref{lengths} implies that more than $h_i$ $a$-edges of $\textbf{p}_{i,i+1}$ contribute $\delta$ to $|\textbf{p}_{i,i+1}|$. So, Lemma \ref{distortion path lengths} gives the inequalities
%$$|\textbf{p}_{i,L-L_0-6}|\geq h_i+h_{L-L_0-6}+11L/2+\delta h_i$$
%$$|\bar{\textbf{p}}_{i,L-L_0-6}|\leq h_i+h_{L-L_0-6}+11(3L_0)+3L_0\delta|V|_a$$
%But $|V|_a<h/L_0^2\leq h_i/L_0^2$ by Lemma \ref{distortion h}, so that the parameter choice $L>>L_0$ yields the inequality $|\bar{\textbf{p}}_{i,L-L_0-6}|<|\textbf{p}_{i,L-L_0-6}|$, contradicting Lemma \ref{distortion eps}.
%
%\end{proof}

Recall that if $\Psi_{i,i+1}$ is the distinguished clove for $i\leq r-1$, then the history of $\Gamma_i$ need not be $H_{i+1}$, but rather could be a proper prefix. As in Section 10, set $H_{i+1}'$ as the history of $\Gamma_i$.

%The next statement is the analogue of Lemma \ref{no one-step} in this setting.

\begin{lemma} \label{distortion no one-step}

For $2\leq i\leq L_0-2$, let $H_i'=H_{i+1}'H'=H_{i+2}'H''H'$ and $\pazocal{C}$ be the computation corresponding to the trapezium $\Gamma_{i-1}$. Suppose the subcomputation $\pazocal{D}$ of $\pazocal{C}$ with history $H''H'$ has step history of length 1. Then there is no two-letter subword $Q'Q$ of the base of $\Gamma_{i-1}$ such that every rule of $\pazocal{D}$ inserts one letter to the left of $Q$.

\end{lemma}

\begin{proof}

Let $\pazocal{Q}$ be the maximal $q$-band of $E_i$ that is a subband of the $q$-spoke of $\Pi$ corresponding to a coordinate shift of the state letter $Q$. Similarly, let $\pazocal{Q}'$ be the $q$-band corresponding to a coordinate shift of $Q'$, so that $\pazocal{Q}'$ and $\pazocal{Q}$ are neighbor $q$-bands. Let $\textbf{x}$ be the subpath of $\textbf{z}_i$ between $\pazocal{Q}'$ and $\pazocal{Q}$.

As in the proof of Lemma \ref{no one-step}, $|\textbf{x}|_a\geq\|H''\|\geq h_{i+1}-h_{i+2}$.

Further, Lemma \ref{distortion h_i+1} implies that $h_{i+1}-h_{i+2}>\frac{1}{30c_3}h_{i+1}$, so that Lemma \ref{distortion first L_0} and the parameter choice $L_0>>c_3$ yield $|\textbf{x}|_a\geq10L_0|V|_a$.

Consider the comb $\nabla$ contained in $E_i$ bounded by $\pazocal{Q}'$, $\pazocal{Q}$, $\textbf{x}$, and $\textbf{p}_{i,i+1}$. By Lemma \ref{distortion a-cells}, at least $|\textbf{x}|_a$ $a$-edges contribute $\delta$ to $\textbf{p}_{i,i+1}$.

Lemma \ref{distortion path lengths} then gives the inequalities:
$$|\textbf{p}_{i,L-6}|\geq h_i+h_{L-6}+11L/2+10\delta L_0|V|_a$$
$$|\bar{\textbf{p}}_{i,L-6}|\leq h_i+h_{L-6}+11(2L_0)+2\delta L_0|V|_a$$
But then the parameter choice $L>>L_0$ implies $|\bar{\textbf{p}}_{i,L-6}|<|\textbf{p}_{i,L-6}|$, contradicting Lemma \ref{distortion eps}.

\end{proof}

Finally, the following statement yields the desired contradiction.

\begin{lemma} \label{distortion contradiction}

For any $g\in B(\pazocal{A},n)$, a $g$-minimal diagram contains no disks.

\end{lemma}

\begin{proof}

The proof follows the same outline as that of Lemma \ref{contradiction}.

For $\eta\geq2$ an integer such that $\left(1-\frac{1}{30c_3}\right)^\eta<\frac{1}{64c_3}$, Lemma \ref{distortion h_i+1} implies that if $1\leq i<j\leq L_0-1$ with $j-i-1\geq\eta$, then $h_j<\frac{1}{64c_3}h_{i+1}$.

For such $i,j,$ Lemma \ref{distortion big h_i+1} implies that $|\textbf{z}_i|_a>32h_j$, while Lemma \ref{distortion number of cloves} implies $8h_j\geq|\textbf{z}_j|_a$. As a result, $|\textbf{z}_i|_a>4|\textbf{z}_j|_a$.

Taking $L_0>>\eta$ and $L_0>>c_0$, there exist indices $2\leq j_1<j_2<\dots<j_m\leq L_0-1$ such that $m\geq c_0$ and $j_{i+1}-j_i-1\geq\eta$. So, $|\textbf{z}_{j_i}|_a>4|\textbf{z}_{j_{i+1}}|_a$ and $h_{j_i+1}\geq64c_3h_{j_i+1}$.

Let $\pazocal{C}:W_0\to\dots\to W_t$ be the computation corresponding to the trapezium $\Gamma_{j_2}$. As $\Gamma_{j_2}$ contains a copy of $\Gamma_{j_2+1}$, which in turn contains a copy of $\Gamma_{j_2+2}$ and so on, there exist words $V_i$ in $\pazocal{C}$ for $i=1,\dots,m$ that are coordinate shifts of $\lab(\textbf{z}_{j_i})$. Note that $|V_{i+1}|_a>4|V_i|_a$.

As in the proof of Lemma \ref{contradiction}, the subcomputations $V_{i+2}\to\dots\to V_i$ cannot be one-step, as an application of Lemma \ref{M one-step} would lead to a contradiction of Lemma \ref{distortion no one-step}.

But then this implies that $\pazocal{Q}_{j_1}$ contains a $\lambda$-shaft of length at least $h_{L_0}$, contradicting Lemma \ref{distortion no shafts}.

\end{proof}

\smallskip

\subsection{Upper bound} \

As referenced in the Introduction, to prove Theorem \ref{distortion}, it suffices to find a constant $M>0$ such that for any $g\in B(\pazocal{A},n)$, $|g|_\pazocal{A}\leq M|g|_\pazocal{X}$.

For $g\in B(\pazocal{A},n)$, let $\Delta$ be an $g$-minimal diagram. 

Fix $w\in F(\pazocal{X})$ so that the value of $w$ in $G_\pazocal{S}(\textbf{M})$ is $g$ and $|w|$ is minimal for all such words. Further, let $\Gamma$ be a minimal diagram over $G_\pazocal{S}(\textbf{M})$ such that $\lab(\partial\Gamma)\equiv wv^{-1}$ for $v\in F(\pazocal{A})$ such that the value of $v$ in $B(\pazocal{A},n)$ is $g$ and $\|v\|=|g|_\pazocal{A}$.

Then, $\Gamma$ is a $g$-diagram, so that $|\partial\Gamma|+\sigma_\lambda(\Gamma)\geq|\partial\Delta|+\sigma_\lambda(\Delta)$.

By Lemmas \ref{distortion contradiction} and \ref{diskless distortion}, $|\partial\Delta|+\sigma_\lambda(\Delta)=|\partial\Delta|=2\delta|g|_\pazocal{A}$. Further, Lemma \ref{lengths}(c) implies $|\partial\Gamma|\leq|w|+\delta|g|_\pazocal{A}$. Hence, $|w|+\sigma_\lambda(\Gamma)\geq\delta|g|_\pazocal{A}$.

Further, by Lemma \ref{G_a design}, $\sigma_\lambda(\Gamma)\leq C_1|\partial\Gamma|_\theta$. Noting that $v$ consists entirely of $a$-letters, it follows that $|\partial\Gamma|_\theta\leq|w|_\theta\leq|w|$.

As a result, $|w|\geq\delta(C_1+1)^{-1}|g|_\pazocal{A}$.

But for any word $u$ in the alphabet $\pazocal{X}\cup\pazocal{X}^{-1}$, $|u|_\pazocal{X}=\|u\|\geq\delta^{-1}|u|$. So, since $|w|$ is minimal for all words over $\pazocal{X}$ whose value in $G_\pazocal{S}(\textbf{M})$ is $g$, $|g|_\pazocal{X}\geq(C_1+1)^{-1}|g|_\pazocal{A}$.

Thus, taking $M=C_1+1$ completes the proof of Theorem \ref{distortion}.

\bigskip

%%%%%%%%%%%%%%%%%%%%%%%%%%%%%%%%%%%%%%%%%%%%%%%%%%%%%%%%%%%%%%%%%

\section{Proof of Theorem \ref{CEP}}

Let $T$ be some subset of $B(\pazocal{A},n)$. Clearly, $\gen{\gen{T}}^{B(\pazocal{A},n)}\subseteq B(\pazocal{A},n)\cap\gen{\gen{T}}^{G_\pazocal{S}(\textbf{M})}$, so that it suffices just to show the opposite inclusion.

Let $\Sigma_1$ be a set of words over $\pazocal{A}$ so that for each $g\in T$, there exists a word $w\in\Sigma_1$ such that the value of $w$ in $B(\pazocal{A},n)$ is $g$. Then, set $\Sigma=\Sigma_1\cup\pazocal{S}$.

Then, $G_\Sigma(\textbf{M})=G_\pazocal{S}(\textbf{M})/\gen{\gen{\Sigma_1}}^{G_\pazocal{S}(\textbf{M})}\cong G_\pazocal{S}(\textbf{M})/\gen{\gen{T}}^{G_\pazocal{S}(\textbf{M})}$.

Fix $g\in B(\pazocal{A},n)\cap\gen{\gen{T}}^{G_\pazocal{S}(\textbf{M})}$ and let $w$ be a word over $\pazocal{A}$ whose value in $G_\pazocal{S}(\textbf{M})$ is $g$. Then $w$ represents the trivial element of $G_\Sigma(\textbf{M})$, so that Lemma \ref{minimal exist} implies that there exists a minimal diagram $\Delta$ over $G_\Sigma(\textbf{M})$ such that $\lab(\partial\Delta)\equiv w$.

If $\Delta$ were to contain a disk, then Lemma \ref{graph} implies that at least $L-4$ $t$-spokes end on $\partial\Delta$. But $\partial\Delta$ contains no $t$-edges, so that this is impossible. Hence, $\Delta$ must be an $M$-minimal diagram.

By Lemmas \ref{M_a no annuli 1} and \ref{M_a no annuli 2}, each maximal $q$-band and each maximal $\theta$-band of $\Delta$ end twice on $\partial\Delta$. But again, $\partial\Delta$ contains no $q$-edge or $\theta$-edge, so that $\Delta$ can contain no $q$-band or $\theta$-band.

As a result, each cell of $\Delta$ must be an $a$-cell, i.e $\Delta$ is a reduced diagram over the group with presentation $\gen{\pazocal{A}\mid\Sigma}\cong B(\pazocal{A},n)/\gen{\gen{T}}^{B(\pazocal{A},n)}$. So, the value of $w$ in $B(\pazocal{A},n)$ is an element of $\gen{\gen{T}}^{B(\pazocal{A},n)}$. But by Lemma \ref{embedding}, the value of $w$ in $B(\pazocal{A},n)$ is $g$, so that $g\in\gen{\gen{T}}^{B(\pazocal{A},n)}$. 

Thus, $B(\pazocal{A},n)\cap\gen{\gen{T}}^{G_\pazocal{S}(\textbf{M})}\subseteq\gen{\gen{T}}^{B(\pazocal{A},n)}$, and so $B(\pazocal{A},n)\leq_{CEP}G_\pazocal{S}(\textbf{M})$.

\bigskip

%%%%%%%%%%%%%%%%%%%%%%%%%%%%%%%%%%%%%%%%%%%%%%%%%%%%%%%%%%%%%%%%%

\section{References}

[1] S. I. Adian, \textit{The Burnside Problem and Identities in Groups}, Springer-Verlag, (1979).

[2] J.-C. Birget, A. Yu. Ol'shanskii, E. Rips, M. Sapir, \textit{Isoperimetric functions of groups and combinatorial complexity of the word problem}, Annals of Mathematics , 156 (2002), no. 2, 467–518.

[3] B. H. Bowditch, \textit{Notes on Gromov's hyperbolicity criterion for path-metric spaces}, "Group theory from a geometrical viewpoint (Trieste, 1990)", (E Ghys, A Haefliger, A Verjovsky, editors), World Sci. Publ., River Edge, NJ (1991)

[4] M. Bridson, A. Haefliger, \textit{Metric Spaces of Non- Positive Curvature}, Grundlehren der mathematischen Wis- senschaften, Volume 319, Springer (1999).

%[4] S. G. Brick, \textit{Dehn functions of groups and products of groups}, Transactions of the American Mathematical Society. 335. 369-384, (1993).

%[5] A. Darbinyan, \textit{Word and conjugacy problems in lacunary hyperbolic groups}, (2017).

%[6] D. B. A. Epstein, J. W. Cannon, S. V. F. Levy, M. S. Paterson, W. P. Thurston, \textit{Word Processing in Groups}, Jones and Bartlett, Boston, (1992).

[5] E. Ghys, P. de la Harpe, \textit{Sur les Groupes Hyperboliques d'apr\`{e}s Mikhael Gromov}, Springer, (1990).

[6] E. S. Golod, I. R. Shafarevich, \textit{On the class field tower}, Izv. Akad. Nauk SSSR Ser. Mat., 28:2 (1964), 261–272 

[7] M. Gromov, \textit{Hyperbolic groups}, Essays in Group Theory (S.M.Gersten, ed.), M.S.R.I. Pub. 8, Springer, (1987), 75–263.

%[11] M. Gromov, \textit{Asymptotic invariants of infinite groups}, in: Geometric Group Theory. Vol. 2 (G.A.Niblo and M.A.Roller, eds.), London Math. Soc. Lecture Notes Ser., 182 (1993), 1–295.

[8] S. V. Ivanov, \textit{The free Burnside groups of sufficiently large exponents}, International Journal of Algebra and Computation, 4(1-2), (1994) ii+308.

[9] S. V. Ivanov, A. Yu. Ol'shanskii, \textit{On finite and locally finite subgroups of free burnside groups of large even exponents}. Journal of Algebra, 195(1), (1997) 241-284. 
%https://doi.org/10.1006/jabr.1996.6941

%[13] S. V. Ivanov, \textit{On subgroups of free Burnside groups of large odd exponent}. Illinois J. Math. 47 (2003), no. 1-2, 299--304. 

%[14] S. V. Ivanov, \textit{Embedding free Burnside groups in finitely presented groups}, Geometriae Dedicata, (2005), vol. 111, no. 1, pp. 87-105.

[10] R. C. Lyndon and P. E. Schupp, \textit{Combinatorial group theory}, Springer-Verlag, 1977.

[11] P. S. Novikov, S. I. Adian, \textit{Defining relations and the word problem for free periodic groups of odd order}, Izv. Akad. Nauk SSSR Ser. Mat., 32:4 (1968)

[12] A. Yu. Ol'shanskii, \textit{Groups of bounded period with subgroups of prime order}, Algebra and Logic 21 (1983), 369–418; translation of Algebra i Logika 21 (1982)

[13] A. Yu. Ol'shanskii, \textit{Hyperbolicity of groups with subquadratic isoperimetric inequality} Internat. J. Algebra Comput. 1 (1991), no. 3, 281–289.

[14] A. Yu. Ol'shanskii, \textit{Geometry of Defining Relations in Groups}, Springer Netherlands, (1991)

[15] A. Yu. Ol'shanskii, \textit{On subgroup distortion in finitely presented groups} Mat. Sb., 188:11 (1997), 51–98; Sb. Math., 188:11 (1997), 1617–1664 

[16] A. Yu. Ol'shanskii, \textit{Polynomially-bounded Dehn functions of groups}, Journal of Combinatorial Algebra, 2. (2018) 311-433

[17] A. Yu. Ol'shanskii, M. V. Sapir, \textit{Embeddings of relatively free groups into finitely presented groups}, (2000)

[18] A. Yu. Ol'shanskii, M. V. Sapir, \textit{Length and area functions in groups and quasiisometric Higman embeddings}, Intern. J. Algebra and Comput., 11 (2001), no. 2, 137–170.

[19] A. Yu. Ol'shanskii, M. V. Sapir, \textit{The Conjugacy Problem and Higman Embeddings}. Memoirs of the American Mathematical Society. 170. (2003). 

[20] A. Yu. Ol’shanskii, M. V. Sapir, \textit{Non-Amenable Finitely Presented Torsion-by-Cyclic Groups}, Publ. math., Inst. Hautes Étud. Sci. (2003)

[21] A. Yu. Ol’shanskii, M. V. Sapir, \textit{Groups with Small Dehn functions and Bipartite Chord Diagrams} GAFA, Geom. funct. anal. 16 (2006), 1324

[22] A. Yu. Ol'shanskii, M. V. Sapir, \textit{Groups with undecidable word problem and almost quadratic Dehn function}, Journal of Topology. 5. (2012) 785-886. 10.1112/jtopol/jts020

[23] A. Yu. Ol'shanskii, M. V. Sapir, \textit{Conjugacy problem in groups with quadratic Dehn function}, (2018). 

[24] M. V. Sapir, J. C. Birget, E. Rips, \textit{Isoperimetric and Isodiametric Functions of Groups}, Annals of Mathematics, 156(2), second series, (1998), 345-466 

[25] M. V. Sapir. \textit{Algorithmic and asymptotic properties of groups}, International Congress of Mathematicians, ICM (2006). 

[26] M. V. Sapir, \textit{Combinatorial algebra: Syntax and Semantics. With contributions by Victor S. Guba and Mikhail V. Volkov}, Springer Monographs in Mathematics, Springer, Cham, (2014).

[27] V. L. Shirvanyan, \textit{Embedding the group $B(\infty,n)$ in the group $B(2,n)$}, Izv. Akad. Nauk SSR Ser. Mat. 40 (1976), 190–208.

[28] D. Sonkin, \textit{CEP-Subgroups of Free Burnside Groups of Large Odd Exponents}, Communications in Algebra Vol. 31. No. 10. (2003) 4687-4695. 10.1081/AGB-120023127. 

[29] E. van Kampen, \textit{On Some Lemmas in the Theory of Groups}, American Journal of Mathematics
Vol. 55, No. 1 (1933), pp. 268-273.

\end{document}